\documentclass[pdflatex,sn-mathphys-num,iicol]{sn-jnl}

\usepackage{anyfontsize}

\usepackage{graphicx}%
\usepackage{multirow}%
\usepackage{amsmath,amssymb,amsfonts}%
\usepackage{amsthm}%
\usepackage{mathrsfs}%
\usepackage[title]{appendix}%
\usepackage{xcolor}%
\usepackage{textcomp}%
\usepackage{manyfoot}%
\usepackage{algorithm}%
\usepackage{algpseudocodex}%
\usepackage{listings}%
\usepackage{enumitem}%
\usepackage{booktabs}%

\usepackage{etoolbox}
\apptocmd{\sloppy}{\hbadness 10000\relax}{}{}

\usepackage{mathtools, bm}
\usepackage{array,multirow}

\usepackage{tikz}
\usetikzlibrary{decorations.pathreplacing}

\usepackage{balance}
\usepackage{nicefrac}

\usepackage{hyperref}

\usepackage{pifont}
\usetikzlibrary{shapes.geometric}
\newcommand{\Circle}{%
  \tikz[baseline=-0.75ex] \node[circle, draw, inner sep=1pt, scale=3.0] {};%
}
\newcommand{\Star}{%
  \tikz[baseline=-0.65ex] \node[star, star points=5, star point ratio=3, draw, inner sep=1pt, scale=1.1] {};%
}
\newcommand{\Triangle}{%
  \tikz[baseline=-0.65ex] \node[regular polygon, regular polygon sides=3, inner sep=1pt, draw, scale=1.8] {};%
}
\newcommand{\Ubox}{%
  \tikz[baseline=-0.7ex] \node[regular polygon, regular polygon sides=4, inner sep=1pt, draw, scale=2.5, rotate=45] {};%
}
\newcommand{\Utriangle}{%
  \tikz[baseline=-0.73ex] \node[regular polygon, regular polygon sides=3, inner sep=1pt, draw, scale=1.8, rotate=180] {};%
}
\newcommand{\Pentagon}{%
  \tikz[baseline=-0.65ex] \node[regular polygon, regular polygon sides=5, inner sep=1pt, draw, scale=2.5] {};%
}
\newcommand{\Ltriangle}{%
  \tikz[baseline=-0.65ex] \node[regular polygon, regular polygon sides=3, inner sep=1pt, draw, scale=1.8, rotate=90] {};%
}
\newcommand{\Rtriangle}{%
  \tikz[baseline=-0.65ex] \node[regular polygon, regular polygon sides=3, inner sep=1pt, draw, scale=1.8, rotate=-90] {};%
}
\newcommand{\Hexagon}{%
  \tikz[baseline=-0.65ex] \node[regular polygon, regular polygon sides=6, inner sep=1pt, draw, scale=2.5] {};%
}
\newcommand{\Heptagon}{%
  \tikz[baseline=-0.65ex] \node[regular polygon, regular polygon sides=7, inner sep=1pt, draw, scale=2.5] {};%
}

\newcommand{\Boxblue}{%
  \tikz[baseline=-0.75ex] \node[rectangle, draw=black, fill=blue, inner sep=1pt, scale=4, line width=1pt] {};%
}
\newcommand{\Boxred}{%
  \tikz[baseline=-0.75ex] \node[rectangle, draw=black, fill=red, inner sep=1pt, scale=4, line width=1pt] {};%
}
\newcommand{\Boxgreen}{%
  \tikz[baseline=-0.75ex] \node[rectangle, draw=black, fill=green!60!black, inner sep=1pt, scale=4, line width=1pt] {};%
}
\newcommand{\Boxyellow}{%
  \tikz[baseline=-0.75ex] \node[rectangle, draw=black, fill=yellow, inner sep=1pt, scale=4, line width=1pt] {};%
}
\newcommand{\Boxpurple}{%
  \tikz[baseline=-0.75ex] \node[rectangle, draw=black, fill=purple!80!black, inner sep=1pt, scale=4, line width=1pt] {};%
}
\newcommand{\Boxbrown}{%
  \tikz[baseline=-0.75ex] \node[rectangle, draw=black, fill=brown!80!black, inner sep=1pt, scale=4, line width=1pt] {};%
}
\newcommand{\Boxorange}{%
  \tikz[baseline=-0.75ex] \node[rectangle, draw=black, fill=orange, inner sep=1pt, scale=4, line width=1pt] {};%
}
\newcommand{\Boxpink}{%
  \tikz[baseline=-0.75ex] \node[rectangle, draw=black, fill=pink, inner sep=1pt, scale=4, line width=1pt] {};%
}
\newcommand{\Boxcyan}{%
  \tikz[baseline=-0.75ex] \node[rectangle, draw=black, fill=cyan, inner sep=1pt, scale=4, line width=1pt] {};%
}
\newcommand{\Boxolive}{%
  \tikz[baseline=-0.75ex] \node[rectangle, draw=black, fill=olive, inner sep=1pt, scale=4, line width=1pt] {};%
}

\newcommand{\UnBoxblue}{%
  \tikz[baseline=-0.75ex] \node[rectangle, fill=blue, inner sep=1pt, scale=4] {};%
}
\newcommand{\UnBoxred}{%
  \tikz[baseline=-0.75ex] \node[rectangle, fill=red, inner sep=1pt, scale=4] {};%
}
\newcommand{\UnBoxgreen}{%
  \tikz[baseline=-0.75ex] \node[rectangle, fill=green!60!black, inner sep=1pt, scale=4] {};%
}
\newcommand{\UnBoxyellow}{%
  \tikz[baseline=-0.75ex] \node[rectangle, fill=yellow, inner sep=1pt, scale=4] {};%
}
\newcommand{\UnBoxpurple}{%
  \tikz[baseline=-0.75ex] \node[rectangle, fill=purple!80!black, inner sep=1pt, scale=4] {};%
}
\newcommand{\UnBoxbrown}{%
  \tikz[baseline=-0.75ex] \node[rectangle, fill=brown!80!black, inner sep=1pt, scale=4] {};%
}
\newcommand{\UnBoxorange}{%
  \tikz[baseline=-0.75ex] \node[rectangle, fill=orange, inner sep=1pt, scale=4] {};%
}
\newcommand{\UnBoxpink}{%
  \tikz[baseline=-0.75ex] \node[rectangle, fill=pink, inner sep=1pt, scale=4] {};%
}
\newcommand{\UnBoxcyan}{%
  \tikz[baseline=-0.75ex] \node[rectangle, fill=cyan, inner sep=1pt, scale=4] {};%
}
\newcommand{\UnBoxolive}{%
  \tikz[baseline=-0.75ex] \node[rectangle, fill=olive, inner sep=1pt, scale=4] {};%
}

\newcommand{\UnsBoxblue}{%
  \tikz[baseline=-0.75ex] \node[rectangle, fill=blue, inner sep=1pt, scale=4, opacity=0.5] {};%
}
\newcommand{\UnsBoxred}{%
  \tikz[baseline=-0.75ex] \node[rectangle, fill=red, inner sep=1pt, scale=4, opacity=0.5] {};%
}
\newcommand{\UnsBoxgreen}{%
  \tikz[baseline=-0.75ex] \node[rectangle, fill=green!60!black, inner sep=1pt, scale=4, opacity=0.5] {};%
}
\newcommand{\UnsBoxyellow}{%
  \tikz[baseline=-0.75ex] \node[rectangle, fill=yellow, inner sep=1pt, scale=4, opacity=0.5] {};%
}
\newcommand{\UnsBoxpurple}{%
  \tikz[baseline=-0.75ex] \node[rectangle, fill=purple!80!black, inner sep=1pt, scale=4, opacity=0.5] {};%
}
\newcommand{\UnsBoxbrown}{%
  \tikz[baseline=-0.75ex] \node[rectangle, fill=brown!80!black, inner sep=1pt, scale=4, opacity=0.5] {};%
}
\newcommand{\UnsBoxorange}{%
  \tikz[baseline=-0.75ex] \node[rectangle, fill=orange, inner sep=1pt, scale=4, opacity=0.5] {};%
}
\newcommand{\UnsBoxpink}{%
  \tikz[baseline=-0.75ex] \node[rectangle, fill=pink, inner sep=1pt, scale=4, opacity=0.5] {};%
}
\newcommand{\UnsBoxcyan}{%
  \tikz[baseline=-0.75ex] \node[rectangle, fill=cyan, inner sep=1pt, scale=4, opacity=0.5] {};%
}
\newcommand{\UnsBoxolive}{%
  \tikz[baseline=-0.75ex] \node[rectangle, fill=olive, inner sep=1pt, scale=4, opacity=0.5] {};%
}

\newcommand{\LineClassOne}{%
  \tikz[baseline=-0.75ex] \node[rectangle, fill={rgb,255:red,126; green,0; blue,0}, inner sep=1pt, scale=0.5, minimum width=10ex] {};%
}
\newcommand{\LineClassTwo}{%
  \tikz[baseline=-0.75ex] \node[rectangle, fill={rgb,255:red,0; green,54; blue,0}, inner sep=1pt, scale=0.5, minimum width=10ex] {};%
}
\newcommand{\LineClassThree}{%
  \tikz[baseline=-0.75ex] \node[rectangle, fill={rgb,255:red,0; green,0; blue,126}, inner sep=1pt, scale=0.5, minimum width=10ex] {};%
}

\newcommand{\ALineClassblue}{%
  \tikz[baseline=-0.75ex] \node[rectangle, fill=blue, inner sep=1pt, scale=0.5, minimum width=10ex] {};%
}
\newcommand{\ALineClassred}{%
  \tikz[baseline=-0.75ex] \node[rectangle, fill=red, inner sep=1pt, scale=0.5, minimum width=10ex] {};%
}
\newcommand{\ALineClassgreen}{%
  \tikz[baseline=-0.75ex] \node[rectangle, fill=green!70!black, inner sep=1pt, scale=0.5, minimum width=10ex] {};%
}
\newcommand{\ALineClassyellow}{%
  \tikz[baseline=-0.75ex] \node[rectangle, fill=orange, inner sep=1pt, scale=0.5, minimum width=10ex] {};%
}
\newcommand{\ALineClasspurple}{%
  \tikz[baseline=-0.75ex] \node[rectangle, fill=purple!80!red, inner sep=1pt, scale=0.5, minimum width=10ex] {};%
}
\newcommand{\ALineClassbrown}{%
  \tikz[baseline=-0.75ex] \node[rectangle, fill=brown!80!black, inner sep=1pt, scale=0.5, minimum width=10ex] {};%
}

\theoremstyle{thmstyleone}%
\newtheorem{theorem}{Theorem}%
\newtheorem{proposition}{Proposition}%
\newtheorem{lemma}{Lemma}%

\theoremstyle{thmstyletwo}%
\newtheorem{example}{Example}%

\theoremstyle{thmstylethree}%

\def\N			{\mathbb N}

\def\R			{\mathbb R}

\def\Sphere		{\mathbb{S}}
\def\SO			{\mathrm{SO}}
\def\Ball		{\mathbb{B}_2}
\def\I		    {\mathbb{I}}
\def\XX			{\mathbb X}

\def\PP			{\mathbb P}

\def\M	        {\mathcal M}
\def\Lebesgue	{\mathrm L}

\def\Radon		{\mathcal R}
\def\NRCDT		{\mathcal N}

\def\maxNRCDT	{\NRCDT_{\mathrm{m}}}
\def\varNRCDT	{\NRCDT_{\mathrm{tv}}}

\def\GL			{\mathrm{GL}}
\def\O			{\mathrm{O}}

\def\d			{\mathop{}\!\mathrm{d}}

\def\I			{\mathrm i}
\def\c			{\mathrm c}

\def\bfq		{\mathbf{q}}
\def\bfx		{\mathbf{x}}
\def\bfy		{\mathbf{y}}

\def\bfzero		{\mathbf{0}}
\def\bfA		{\mathbf{A}}
\def\bfQ		{\mathbf{Q}}
\def\bfR		{\mathbf{R}}
\def\bfI		{\mathbf{I}}
\def\bftheta	{{\boldsymbol{\theta}}}

\def\B			{\mathcal B}
\def\P			{\mathcal P}
\def\U			{\mathcal U}
\def\V			{\mathcal V}

\def\Glue		{\mathcal I}

\def\F			{\mathbb F}
\def\G			{\mathbb G}

\def\RI			{\mathrm{RI}}
\def\VI			{\mathrm{VI}}

\DeclareMathOperator{\diam}{diam}
\DeclareMathOperator{\supp}{supp}

\DeclareMathOperator{\mean}{mean}

\DeclareMathOperator{\std}{std}

\DeclareMathOperator*{\argmin}{arg\,min}
\DeclareMathOperator{\trace}{trace}

\newcommand{\subsubset}{\subset\joinrel\subset}

\DeclareFontFamily{U}{mathx}{\hyphenchar\font45}
\DeclareFontShape{U}{mathx}{m}{n}{<-> mathx10}{}
\DeclareSymbolFont{mathx}{U}{mathx}{m}{n}
\DeclareMathAccent{\widebar}{0}{mathx}{"73}

\def\mNRCDT		{\textsubscript{$\mathrm{m}$}NR-CDT}
\def\tvNRCDT	{\textsubscript{$\mathrm{tv}$}NR-CDT}

\newcommand{\hNRCDT}[1]{\textsubscript{$h_{\mathrm{#1}}$}NR-CDT}
\newcommand{\genNRCDT}[1]{\NRCDT_{h_{\mathrm{#1}}}}

\def\miNRCDT	{\textsubscript{$h_{\mathrm{d}}$}NR-CDT}
\def\maNRCDT	{\textsubscript{$h_{\mathrm{a}}$}NR-CDT}
\def\iaNRCDT	{\textsubscript{$h_{\mathrm{b}}$}NR-CDT}
\def\miaNRCDT{\textsubscript{$h_{\mathrm{c}}$}NR-CDT}

\begin{document}

\title[Generalized NR-CDTs]{Generalizations of the Normalized Radon Cumulative Distribution Transform for Limited Data Recognition}

\author*[1]{\fnm{Matthias} \sur{Beckmann}}\email{research@mbeckmann.de}

\author*[2]{\fnm{Robert} \sur{Beinert}}\email{beinert@math.tu-berlin.de}

\author*[2]{\fnm{Jonas} \sur{Bresch}}\email{bresch@math.tu-berlin.de}

\affil[1]{\orgdiv{Fachbereich Mathematik}, \orgname{Universit\"at Hamburg}, \orgaddress{\street{Bundesstra{\ss}e 55}, \city{Hamburg}, \postcode{20146}, \state{Hamburg}, \country{Germany}}}

\affil[2]{\orgdiv{Institut f\"ur Mathematik}, \orgname{Technische Universit\"at Berlin}, \orgaddress{\street{Stra{\ss}e des 17.\ Juni 136}, \city{Berlin}, \postcode{10623}, \state{Berlin}, \country{Germany}}}

\abstract{
The Radon cumulative distribution transform (R-CDT) exploits 
one-dimensional Wasserstein transport
and the Radon transform
to represent prominent features in images.
It is closely related to the sliced Wasserstein distance
and facilitates classification tasks,
especially in the small data regime, 
like the recognition of watermarks in filigranology.
Here, a typical issue is 
that the given data may be subject to affine transformations
caused by the measuring process.
To make the R-CDT
invariant under arbitrary affine transformations,
a two-step normalization of the R-CDT has been proposed
in our earlier works.
The aim of this paper is twofold.
First,
we propose a family of generalized normalizations
to enhance flexibility for applications.
Second,
we study multi-dimensional and non-Euclidean settings
by making use of generalized Radon transforms.
We prove that our novel feature representations are
invariant under certain transformations
and allow for linear separation in feature space.
Our theoretical results are supported by numerical experiments
based on 2d images, 3d shapes and 3d rotation matrices,
showing near perfect classification accuracies
and clustering results.
}

\keywords{Radon-CDT,
    feature representation,
    image and shape classification,
    pattern recognition,
    small data regime}

\maketitle

\section{Introduction}

Automated pattern recognition and classification
play a central role in numerous applications and disciplines,
be it in medical imaging, biometrics, or document analysis.
Nowadays,
in the big data regime,
end-to-end deep neural networks provide the latest state of the art. 
In the small data regime,
however,
hand-crafted feature extractors and similarity measures still stand their ground.

In recent years,
optimal transport-based techniques 
like the Wasserstein distance and variations of this 
have become popular tools
for image comparison.
At their core,
the Wasserstein distances define metrics 
between probability measures 
on a common Polish space 
and find an optimal coupling \cite{Kantorovich2006,Villani2003,Santambrogio2015,Bogachev2012}.
Although being based on a linear program,
the numerical calculation is challenging,
initiating the development of entropic optimal transport \cite{Cuturi2013,Bresch2025},
linear optimal transport \cite{Wang2013,Park2018,Moosmueller2023},
so-called sliced Wasserstein distance \cite{Bonneel2015,Shifat-E-Rabbi2023,Kolouri2019,Piening2025,Shi2025},
and the subspace robust Wasserstein distance \cite{Paty2019}
to lower the computational burden and to increase stability.
The definition of the Wasserstein distance can be extended
to ensure invariance under orthogonal transformations
of the input measures,
leading to the so-called
Procrustes--Wasserstein distance \cite{Grave2019,Adamo2025}
and the rotation-invariant sliced Wasserstein distance \cite{Lai2017,Shi2025}.

The optimal transport technique
behind the Wasserstein distance 
can be further generalized to the so-called Gromov--Wasserstein \cite{Memoli2011}
and embedded Wasserstein distance \cite{Beier2025, Salmona2024},
which both allow for the comparison of measures on different Polish spaces.
These distances play a major role in shape and graph analysis
since they are invariant under isometric transformations.
In practice,
they are calculated by solving costly quadratic programmes,
limiting their applicability.
As remedy, 
regularized \cite{Sejourne2021}, 
linear \cite{Beier2022}, 
and sliced \cite{Beinert2023,Vayer2019,Piening2025a} versions have been proposed.

Beyond similarity measures between images and shapes,
optimal transport-based methods may be used 
to design feature extractors,
being the focus of this paper.
Ideally, 
a feature extractor transforms different classes
to linearly separable subsets.
This may,
for instance,
be achieved by 
the so-called Radon cumulative distribution transform (R-CDT)
introduced in~\cite{Kolouri2016}.
The R-CDT is based on one-dimensional optimal transport,
which is generalized to two-dimensional data
by applying the Radon transform,
known from tomography~\cite{Ramm1996,Natterer2001}.
This approach shows great potential in many applications
\cite{Kolouri2017,DiazMartin2024,Shifat-E-Rabbi2021}
and is closely related to the sliced Wasserstein distance.
A similar approach for data on the sphere is studied in \cite{Quellmalz2023,Quellmalz2024},
for multi-dimensional optimal transport maps in \cite{Moosmueller2023},
and for optimal Gromov--Wasserstein transport maps in \cite{Beier2022}.

A central inspiration for this paper is the 
application of pattern recognition techniques in filigranology%
---the study of analogue watermarks.
These play a central role in dating historical manuscripts
as well as identifying scribes and papermills.
For automatic classification,
the main issue is the enormous number of classes
with only few members per class.
An end-to-end processing pipeline for thermograms of watermarks 
including an R-CDT-based classification is proposed
in \cite{Hauser2024}, where the authors report classification invariance
with respect to translation and dilation of the watermark.
In order to include
other affine transformations caused,
e.g., 
by unstandardized recording methods,
in our previous works \cite{Beckmann2024a,Beckmann2025},
normalized R-CDT versions are proposed
based on a two-step normalization scheme.

\paragraph*{Contribution}

This paper extends 
the max-normalized R-CDT (\mNRCDT)
introduced at the SSVM'25 conference \cite{Beckmann2024a}
by employing a more flexible normalization scheme
and by generalizing the underlying Radon transform.
Similar to \cite{Beckmann2024a},
the unaltered goal is 
to design easy-to-implement feature extractors
that are invariant under common transformations
in the specific field of application.
For instance,
the watermark recognition task in filigranology requires feature extractors 
that are invariant under affine image transformations
like rotation, shifting, shearing, scaling.
Since this paper is an extension of \cite{Beckmann2024a}
and thus substantially based on this,
the first two contributions remain:
\begin{itemize}
    \item the proposal of a new two-step normalization procedure for the R-CDT
        yielding the \mNRCDT\ feature extractor,
        which is invariant under arbitrary affine transformations,
    \item a rigorous study of the separability properties of the \mNRCDT,
        especially,
        affine classes become linear separable in \mNRCDT\ space,
\end{itemize}
The robustness of the original \mNRCDT\ 
under non-affine deformations
is studied in \cite{Beckmann2025},
which is not the focus of this paper.
In difference to \cite{Beckmann2024a,Beckmann2025},
this extended version contains the following additional contributions:
\begin{itemize}
    \item a flexible final normalization for the NR-CDT
        yielding the new \hNRCDT{}
        and 
        enabling the usage of further angular informations
        for instance based on the total variation,
    \item the replacement of the underlying 2d Radon transform
        by the generalized Radon transform
        extending the \hNRCDT{} to the multidimensional setting,
        the circular Radon transform 
        or the Radon transform on the 3d rotation group SO(3),
    \item a rigorous discussion about 
        the linear separability properties,
        which provably carry over to all new extended NR-CDT variants,
    \item an extension of the image classification experiments in \cite{Beckmann2024a},
        by considering image clustering,
        3d shape recognition,
        and 
        classification of SO(3) point clouds.
\end{itemize}

\paragraph*{Outline}

The first part of this paper revisit the SSVM proceeding \cite{Beckmann2024a}
where the \mNRCDT\ is originally proposed.
To this end,
we first generalize the classical Radon transform to measures
in §~\ref{sec:radon} and, 
thereon, 
restate the definition of the \mNRCDT\ in §~\ref{sec:ot-dist}.
The extension of the final normalization step is, especially, presented in §~\ref{sec:hNRCDT}
yielding the new \hNRCDT{}.
The linear separability of affinely transformed measure classes
in \mNRCDT\ and \hNRCDT{} space is proven 
in Theorem~\ref{thm:sep-max-nrcdt} and \ref{thm:sep-h-nrcdt}.
The novel generalization of the \hNRCDT{} to arbitrary domains 
beyond the two-dimensional setting is introduced in §~\ref{sec:gen-NRCDT}.
The focus here lies on the extension to $\R^d$
related to the multidimensional and circular Radon transform
as well as to the 3d rotation group.
The separability properties for these cases are reported
in Theorem~\ref{thm:sep-multi-h-nrcdt}--\ref{thm:sep-so3-h-nrcdt}.
Our theoretical findings are supported 
by proof-of-concept experiments 
in §~\ref{sec:num-ex}
showing significant improvements of the classification accuracy 
by the proposed normalizations,
especially,
in the small data regime.

\section{Radon Transform}
\label{sec:radon}

The main idea behind the classical Radon transform \cite{Natterer2001} is
to integrate a given bivariate function along all parallel lines
pointing in a certain direction.
This integral transform can also be interpreted
as projection of the given function
onto the line with orthogonal orientation.
In the following,
we briefly review the classical Radon transform for functions
and generalize the concept to measures.
Finally,
we study the effect of affine transformations 
on the Radon transform,
which is crucial
to solve the classification task at hand.

\subsection{Radon Transform of Functions}

Depending on $\bftheta \in \Sphere_1 \coloneqq \{\bfx \in \R^2 \mid \lVert \bfx \rVert = 1\}$,
we introduce the \textit{slicing operator} $S_\bftheta\colon \R^2 \to \R$ by
\begin{equation*}
    S_\bftheta(\bfx) \coloneqq \langle \bfx, \bftheta \rangle,
    \quad
    \bfx \in \R^2.
\end{equation*}
Its preimages $S_\bftheta^{-1}(t)$, 
$t \in \R$,
are the lines $\ell_{t,\bftheta}$ 
in direction 
$\bftheta^\perp \coloneqq (\theta_2, - \theta_1)^\top \in \Sphere_1$ 
with distance $t$ to the origin.
More precisely,
we have 
\begin{equation*}
    \ell_{t,\bftheta} 
    \coloneqq 
    S_\bftheta^{-1}(t) 
    = 
    \{ t \bftheta + \tau \bftheta^\perp 
        \mid 
        \tau \in \R\}
    \subset \R^2.
\end{equation*}
Using the bijection $\varphi_\bftheta \colon \R^2 \to \R^2$
defined as
$\varphi_\bftheta(t, \tau) 
\coloneqq 
t \, \bftheta + \tau \, \bftheta^\perp$,
whose inverse is given by 
$\varphi_\bftheta(\bfx)^{-1} 
= (\langle \bfx, \bftheta \rangle, \langle \bfx, \bftheta^\perp \rangle)$,
we parameterize $\ell_{t, \bftheta}$ via
$\tau \mapsto \varphi_\bftheta(t, \tau)$.

For $f \in \Lebesgue^1(\R^2)$, 
we define its \textit{Radon transform} $\Radon [f] \colon \R \times \Sphere_1 \to \R$ as the line integral 
\begin{equation*}
    \Radon [f] (t, \bftheta) 
    \coloneqq 
    \int_{\ell_{t,\bftheta}} 
    f(s) 
    \d s,
    \quad 
    (t,\bftheta) \in \R \times \Sphere_1,
\end{equation*}
where $\d s$ denotes the arc length element of $\ell_{t,\bftheta}$.
This defines the \textit{Radon operator} $\Radon \colon \Lebesgue^1(\R^2) \to \Lebesgue^1(\R \times \Sphere_1)$.
For fixed $\bftheta \in \Sphere_1$, 
we set $\Radon_\bftheta \coloneqq \Radon(\cdot, \bftheta)$, 
which is referred to as the \textit{restricted Radon operator} 
$\Radon_\bftheta \colon \Lebesgue^1(\R^2) \to \Lebesgue^1(\R)$.
The action of the Radon operator is illustrated in Figure~\ref{fig:radon}.

\begin{figure}[t]
    \centering
    \scalebox{0.7}{
    \begin{tikzpicture}
        \draw[->] (-5,0) -- (5,0) node[right] {$\R$};
        \draw[->] (0,-2.5) -- (0,2.5) node[above] {$\R$};
        \draw[thin, dashed] (-3*1.2,-2*1.2) -- (3*1.3,2*1.3);  
        \draw[gray] (3,1) -- (3.0507*0.832, 3.0507*0.5547);
        \draw[gray] (3,1) -- (3 + 3*0.5547,1 - 3*0.832);
        \draw[gray] (3.0507*0.832, 3.0507*0.5547) -- (3.0507*0.832 - 1*0.5547, 3.0507*0.5547 + 1*0.832);
        \node at (3.0507*0.832 - 1*0.5547, 3.0507*0.5547 + 1*0.832) [right] {\color{gray}$\ell_{t_1,\bftheta}$};
        \draw[gray] (-4,2) -- (-2.2186*0.832, -2.2186*0.5547);
        \draw[gray] (-2.2186*0.832, -2.2186*0.5547) -- (-2.2186*0.832 + 1*0.5547, -2.2186*0.5547 - 1*0.832);
        \node at (-2.2186*0.832 + 1*0.5547, -2.2186*0.5547 - 1*0.832) [right] {\color{gray}$\ell_{t_2,\bftheta}$};
        \draw[gray] (-0.5547*3,0.832*3) -- (0.5547*3,-0.832*3);
        \node at (0.5547*3,-0.832*3) [above right] {\color{gray}$\ell_{0,\bftheta}$};
        \draw[very thick,->] (0,0) -- (0.5547,-0.832) node[below] {$\bftheta^\perp$};
        \draw[very thick,->] (0,0) -- (0.832,0.5547) node[above] {$\bftheta$};
        \draw[thin, dotted] (0,0) circle (1);
        \node at (3.0507*0.832, 3.0507*0.5547)[circle,fill,inner sep=0.5pt]{};
        \node at (-2.2186*0.832, -2.2186*0.5547)[circle,fill,inner sep=0.5pt]{};
        \draw[decorate,decoration={brace,amplitude=7.5pt,mirror,raise=0pt}] (0,0) -- (3.0507*0.832, 3.0507*0.5547) node[midway,below right,yshift=-0.15cm] {$|t_1|$};
        \draw[decorate,decoration={brace,amplitude=7.5pt,mirror,raise=0pt}] (0,0) -- (-2.2186*0.832, -2.2186*0.5547) node[midway,above left,yshift=1pt,xshift=-2pt] {$|t_2|$};
        \node at (4,0) [anchor=north] {};
        \node at (0,3) [anchor=east] {};
    \end{tikzpicture}
    }
    \caption{Illustration of the bivariate Radon transform
    with distance $t \in \R$
    and normal direction $\bftheta \in \Sphere_1$.
    The given function is integrated along the lines $\ell_{t, \bftheta}$.
    }
    \label{fig:radon}
\end{figure}
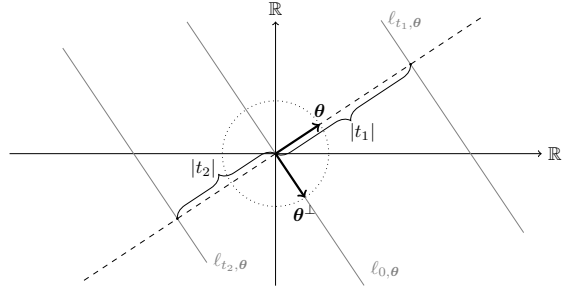

The Radon transform is also well-defined 
for all $f \in \Lebesgue^p(\R^2)$ with $p \geq 1$ 
and $\supp(f) \subseteq \Ball \coloneqq \{ \bfx \in \R^2 \mid \lVert \bfx \rVert \le 1\}$, 
in which case $\Radon [f] \in \Lebesgue^p(\R \times \Sphere_1)$
with $\supp(\Radon [f]) \subseteq \I \times \Sphere_1$,
where $\I \coloneqq [-1,1]$.
According to~\cite{Natterer2001}, 
the adjoint operator
$\Radon^* \colon \Lebesgue^\infty(\R \times \Sphere_1) \to \Lebesgue^\infty(\R^2)$ 
of the Radon transform 
$\Radon \colon \Lebesgue^1(\R^2) \to \Lebesgue^1(\R \times \Sphere_1)$ 
is given by the \emph{back projection}
\begin{equation*}
    \Radon^* [g](\bfx) \coloneqq \frac{1}{2\pi} \int_{\Sphere_1} g(S_\bftheta(\bfx), \bftheta) \d \sigma_{\Sphere_1}(\bftheta),
    \quad \bfx \in \R^2,
\end{equation*}
where $\sigma_{\Sphere_1}$ denotes the surface measure on $\Sphere_1$.

\subsection{Radon Transform of Measures}

The concept of the Radon transform is now translated to signed, 
regular, finite measures $\mu \in \M(\R^2)$.
For a fixed direction $\bftheta \in \Sphere_1$, 
we generalize the \textit{restricted Radon transform} $\Radon_\bftheta$ to measures by setting
\begin{equation*}
    \Radon_\bftheta \colon \M(\R^2) \to \M(\R), 
    \quad
    \mu \mapsto (S_\bftheta)_\# \mu = \mu \circ S_\bftheta^{-1},
\end{equation*}
which corresponds to the integration along $\ell_{t,\bftheta}$.
Note that $\Radon_\bftheta[\mu](\R) = \mu(\R^2)$
for all  $\bftheta \in \Sphere_1$ and,
thus, 
the mass of $\mu$ is preserved by $\Radon_\bftheta$.
In measure theory, $\Radon_\bftheta$ can be considered as a disintegration family.
Heuristically, 
we may generalize the Radon transform 
by integrating $\Radon_\bftheta$ along $\bftheta \in \Sphere_1$.
Therefore, we define the \textit{Radon transform} $\Radon \colon \M(\R^2) \to \M(\R \times \Sphere_1)$ via
\begin{equation}
    \label{eq:radon-meas}
    \Radon [\mu] \coloneqq \Glue_\#[\mu \times u_{\Sphere_1}]
\end{equation}
with $\Glue(\bfx, \bftheta) \coloneqq (S_\bftheta(\bfx), \bftheta)$
for $(\bfx, \bftheta) \in \R^2 \times \Sphere_1$.
Here $\mu \times u_{\Sphere_1}$ denotes the product measure
between the given $\mu$ 
and the uniform measure 
$u_{\Sphere_1} \coloneqq \sigma_{\Sphere_1}/2\pi$
on $\Sphere_1$.

\begin{proposition}
    Let $\mu \in \M(\R^2)$.
    Then, 
    $\Radon [\mu]$ can be disintegrated 
    into the family $\Radon_\bftheta [\mu]$ 
    with respect to $u_{\Sphere_1}$,
    i.e., 
    for all continuous $g \in C_0(\R \times \Sphere_1)$ vanishing at infinity, 
    we have
    \begin{equation*}
        \langle \Radon [\mu], g\rangle 
        = 
        \int_{\Sphere_1} \langle \Radon_{\bftheta} [\mu], g(\cdot, \bftheta)\rangle \d u_{\Sphere_1}(\bftheta).
    \end{equation*}
\end{proposition}

\begin{proof}
    By definition in \eqref{eq:radon-meas}, 
    we obtain 
    \begin{align*}
        \langle \Radon [\mu], g\rangle
        & = 
        \int_{\R \times \Sphere_1} g(t, \bftheta) \d \Glue_\#[\mu \times u_{\Sphere_1}] (t, \bftheta)
        \\
        & = 
        \int_{\Sphere_1} \int_{\R^2} g(S_\bftheta(\bfx), \bftheta) \d \mu(\bfx) \d u_{\Sphere_1}(\bftheta) 
        \\
        & = 
        \int_{\Sphere_1} \int_{\R^2} g(t, \bftheta) \d [(S_\bftheta)_{\#}\mu](t) \d u_{\Sphere_1}(\bftheta)
        \\
        & = 
        \int_{\Sphere_1} \langle \Radon_{\bftheta} [\mu], g(\cdot, \bftheta)\rangle \d u_{\Sphere_1}(\bftheta)
    \end{align*}
    using Fubini's theorem.
\end{proof}

One can find the measure-valued Radon transform 
$\Radon \colon \M(\R^2) \to \M(\R \times \Sphere_1)$ 
as the adjoint of the function-valued adjoint 
$\Radon^* \colon \Lebesgue^\infty (\R \times \Sphere_1) \to \Lebesgue^\infty(\R^2)$,
similar to the case of distributions with compact support,
cf.~\cite{Ramm1996}.

\begin{proposition} \label{prop:adjRT_measure}
    The Radon transform of $\mu \in \M(\R^2)$ satisfies
    \begin{equation*}
        \langle \Radon [\mu], g\rangle = \langle \mu, \Radon^* [g]\rangle
        \quad \forall \,
        g \in \Lebesgue^\infty(\R \times \Sphere_1).
    \end{equation*}
\end{proposition}

\begin{proof}
    For all $\mu \in \M(\R^2)$ and $g \in \Lebesgue^\infty(\R \times \Sphere_1)$, applying Fubini's theorem gives
    \begin{align*}
        \langle \Radon [\mu], g\rangle
        & = \int_{\R \times \Sphere_1} g(t, \bftheta) \d \Glue_\#[\mu \times u_{\Sphere_1}] (t, \bftheta)
        \\
        & =
        \int_{\R^2} \int_{\Sphere_1} g(S_\bftheta(\bfx), \bftheta) \d u_{\Sphere_1}(\bftheta) \d \mu(\bfx).
    \end{align*}
\end{proof}

Note that, for $f \in \Lebesgue^1(\R^2)$ and the Lebesgue measure $\lambda_{\R^2}$ on $\R^2$,
the Radon transform satisfies 
\begin{equation*}
\Radon[f \, \lambda_{\R^2}]  =  \Radon [f] \, \sigma_{\R \times \Sphere_1},
\end{equation*}
where $\sigma_{\R \times \Sphere_1}$ denotes the surface measure on $\R \times \Sphere_1$.
In particular, the Radon transform of an absolutely continuous measure is again absolutely continuous.

\subsection{Radon Transform of Affine Transformations}
\label{sec:Radon_affine_transform}

We now consider the Radon transform of an affinely transformed finite measure $\mu \in \M(\R^2)$.
To this end, let $\bfA \in \GL(2)$ and $\bfy \in \R^2$,
i.e., $\bfA$ is contained in the general linear group $\GL$ of regular matrices.
We define $\mu_{\bfA,\bfy} \in \M(\R^2)$ via
\begin{equation}
    \label{eq:aff-meas}
    \mu_{\bfA,\bfy} 
    \coloneqq 
    (\bfA \cdot + \bfy)_{\#} \mu = \mu \circ (\bfA^{-1}(\cdot-\bfy)).
\end{equation}

\begin{proposition} \label{prop:RT_transformation}
    For any $\bftheta \in \Sphere_1$, the restricted Radon transform satisfies
    \begin{align*}
        \Radon_{\bftheta} [\mu_{\bfA,\bfy}] 
        & = 
        (\lVert \bfA^\top \bftheta\rVert \cdot + \langle \bfy, \bftheta\rangle)_{\#} 
        \Radon_{\frac{\bfA^\top \bftheta}{\lVert \bfA^\top \bftheta \rVert}}[\mu]
        \\
        & = 
        \Radon_{\frac{\bfA^\top \bftheta}{\lVert \bfA^\top \bftheta \rVert}}[\mu] 
        \circ
        \Bigl(\tfrac{\cdot - \langle \bfy, \bftheta\rangle}{\lVert \bfA^\top \bftheta \rVert}\Bigr).
    \end{align*}
\end{proposition}

\begin{proof}
    Direct calculations yield
    \begin{align*}
        \Radon_{\bftheta} [\mu_{\bfA,\bfy}]
        & = 
        (S_\bftheta)_{\#}[(\bfA \cdot + \bfy)_{\#} \mu]
        \\
        & = 
        (\langle \bfA \cdot + \bfy, \bftheta\rangle)_{\#} \mu \\
        & =
        (\langle \cdot, \bfA^\top \bftheta\rangle + \langle \bfy, \bftheta\rangle)_{\#} \mu 
        \\
        &=
        (\lVert \bfA^\top \bftheta\rVert 
        \bigl\langle \cdot, \tfrac{\bfA^\top \bftheta}{\lVert \bfA^\top \bftheta\rVert}\bigr\rangle 
        +
        \langle \bfy, \bftheta\rangle)_{\#} \mu
        \\
        & =
        (\lVert \bfA^\top \bftheta\rVert \cdot + \langle \bfy, \bftheta\rangle)_{\#} 
        \Radon_{\frac{\bfA^\top \bftheta}{\lVert \bfA^\top \bftheta \rVert}}[\mu],
    \end{align*}
    and the proof is complete.
\end{proof}

\begin{table*}[t]
    \resizebox{\linewidth}{!}{
        \begin{tabular}{l@{\enspace}l@{\enspace}l@{\enspace}l}
            \toprule
            transformation 
            & 
            $\bfA$
            &
            $\bfy$
            &
            $\Radon_{\bftheta(\vartheta)} [\mu_{\bfA, \bfy}]$,
            $\vartheta \in (-\tfrac{\pi}{2}, \tfrac{\pi}{2})$
            \\
            \midrule
            translation 
            & 
            $\bfI$ 
            &
            $\R^2$ 
            &
            $\Radon_{\bftheta(\vartheta)} [\mu] \circ (\cdot - \langle \bfy, \bftheta(\vartheta) \rangle)$ 
            \\[1ex]
            rotation
            & 
            $\bigl(\begin{smallmatrix} \cos(\varphi) & -\sin(\varphi) \\ \sin(\varphi) & \hphantom{-}\cos(\varphi) \end{smallmatrix}\bigr)$
            &
            $\bfzero$ 
            & 
            $\Radon_{\bftheta(\vartheta - \varphi)} [\mu]$ 
            \\[1ex]
            reflection 
            &
            $\bigl(\begin{smallmatrix} \cos(\varphi) & \hphantom{-}\sin(\varphi) \\ \sin(\varphi) & -\cos(\varphi) \end{smallmatrix}\bigr)$
            &
            $\bfzero$  
            &
            $\Radon_{\bftheta(\varphi - \vartheta)} [\mu]$ 
            \\[1ex]
            anisotropic scaling 
            & 
            $\bigl(\begin{smallmatrix} a & 0 \\ 0 & b \end{smallmatrix}\bigr)$
            &
            $\bfzero$  
            &
            $\Radon_{\bftheta(\arctan(\frac{b}{a} \tan(\vartheta)))}[\mu] \circ ([a^2 \cos^2(\vartheta) + b^2 \sin^2(\vartheta)]^{-1/2} \cdot)$ 
            \\[1ex]
            vertical shear
            & 
            $\bigl(\begin{smallmatrix} 1 & 0 \\ c & 1 \end{smallmatrix}\bigr)$
            &
            $\bfzero$ 
            & 
            $\Radon_{\bftheta(\arctan(c + \tan(\vartheta)))}[\mu] \circ ([1 + c^2 \cos^2(\vartheta) + c \sin(2\vartheta)]^{-1/2} \cdot)$
            \\
            \bottomrule
        \end{tabular}
    }
    \caption{Summary of common transformations for $\mu \in \M(\R^2)$ 
        with $a,b > 0$ and $c, \varphi \in \R$.
        The unit circle is parameterized by
        $\bftheta(\vartheta) \coloneqq (\cos(\vartheta), \sin(\vartheta))^\top$.
        The Radon transform for the left half of $\Sphere_1$ follows by symmetry.}
    \label{tab:radon-aff-trans}
\end{table*}

The effect of common affine transformations on the Radon transform
is given in Table~\ref{tab:radon-aff-trans}.
In order to describe the deformation with respect to $\bftheta$,
we over-parameterize the unit circle $\Sphere_1$ via
$\bftheta(\vartheta) \coloneqq (\cos(\vartheta), \sin(\vartheta))^\top$, $\vartheta \in \R$.
As by Proposition~\ref{prop:RT_transformation},
an affine transformation essentially causes a translation and dilation
of the transformed measure 
together with a non-affine remapping in $\bftheta$.

\section{Optimal Transport-Based Transforms}
\label{sec:ot-dist}

The aim of the following is to introduce an image distance
that is unaware of affine transformations.
Methodologically,
we rely on the \emph{Radon cumulative distribution transform} (R-CDT)
introduced in \cite{Kolouri2016},
which allows to utilize the fast-to-compute, one-dimensional Wasserstein distance
in the context of image processing 
due to a Radon-based slicing technique.
As the R-CDT is not invariant under affine transformation by itself,
we propose a two-step normalization scheme,
which is essentially grounded on our observations 
regarding the Radon transform under affine transformations
in §~\ref{sec:Radon_affine_transform}.
Finally,
we study the linear separability
of affinely transformed image classes
by our novel normalized R-CDT.

\subsection{R-CDT for Measures}

The R-CDT traces back to Kolouri et al.\ \cite{Kolouri2016}
and transforms smooth, bivariate density functions.
In contrast to \cite{Kolouri2016},
we introduce the concept
for arbitrary probability measures,
similar to~\cite{DiazMartin2024}.
In a first step,
we consider probability measures $\P(\R) \subset \M(\R)$ 
defined on the real line.
For $\mu \in \P(\R)$,
the \emph{cumulative distribution function}
$F_\mu \colon \R \to [0,1]$
is given by $F_\mu(t) \coloneqq \mu((-\infty,t])$, $t \in \R$.
Its generalized inverse,
known as \emph{quantile function},
reads as
\begin{equation*}
    F_\mu^{[-1]}(t) \coloneqq \inf \{s \in \R \mid F_\mu(s) > t\},
    \quad t \in \R.
\end{equation*}
Based on a reference measure $\rho \in \P(\R)$
that does not give mass to atoms,
e.g.,
the uniform distribution
$u_{[0,1]}$ on $[0,1]$,
we define
the {\em cumulative distribution transform}
$\widehat{\mu} \colon \R \to \R$, in short CDT, via
\begin{equation}
    \label{eq:cdt}
    \widehat{\mu} \coloneqq F_\mu^{[-1]} \circ F_\rho.
\end{equation}
Throughout the paper we use a fixed $\rho$,
but all our results are independent from it's choice.
For $\rho = u_{[0,1]}$, the CDT reduces to \smash{$\widehat{\mu} = F_\mu^{[-1]}$}.
In applications, however, a different choice of $\rho$ might be advantageous.

For any convex cost function $c \colon \R \to [0,\infty)$,
the CDT
(with respect to $\rho$) 
solves the Monge--Kantorovich transportation problem
\cite{Villani2003},
i.e.,
\begin{equation}
    \label{eq:monge}
    \widehat{\mu} = \argmin_{T_\# \rho = \mu} \int_\R c(s-T(s))  \d \rho(s),
\end{equation}
where the minimum is taken over all measurable functions $T \colon \R \to \R$.
In other words,
$\widehat{\mu} \colon \R \to \R$ is an optimal Monge map
transporting $\rho$ to $\mu$
while minimizing the cost.
If $\mu \in \P_2(\R)$, i.e., $\mu$ has finite 2nd moment,
then $\widehat{\mu}$ is square integrable
with respect to $\rho$,
i.e.,
$\widehat{\mu} \in \Lebesgue^2_\rho(\R)$.
Moreover,
for $\mu, \nu \in \P_2(\R)$,
the norm distance
\begin{equation*}
    \lVert \widehat{\mu} - \widehat{\nu} \rVert_\rho
    \coloneqq
    \Bigl(
        \int_\R
        \lvert \widehat{\mu}(t) - \widehat{\nu}(t) \rvert^2
        \d \rho(t)
    \Bigr)^{\frac{1}{2}}
\end{equation*}
equals the well-established Wasserstein-2 distance 
\cite{Villani2003}.

To deal with a probability measure $\mu \in \P(\R^2)$
defined on the plane,
we first determine the Radon transform $\Radon[\mu] \in \M(\R \times \Sphere_1)$ 
with its disintegration family $\{\Radon_\bftheta[\mu] \in \P(\R) \mid \bftheta \in \Sphere_1 \}$.
Then,
for each fixed $\bftheta \in \Sphere_1$, 
we consider the CDT $\widehat{\Radon}_\bftheta [\mu]$
(with respect to the same reference measure $\rho \in \P(\R)$
for all $\bftheta \in \Sphere_1$)
of the Radon projection $\Radon_\bftheta [\mu]$, 
yielding the \emph{R-CDT}\, 
$\widehat{\Radon}[\mu] \colon \R \times \Sphere_1 \to \R$ of $\mu$ via
\begin{equation*}
    \widehat{\Radon} [\mu](t,\bftheta) 
    \coloneqq \widehat{\Radon}_\bftheta [\mu](t),
    \quad  (t,\bftheta) \in \R \times \Sphere_1.
\end{equation*}
If $\mu \in \P_2(\R^2)$,
then the Radon projection $\Radon_\bftheta [\mu] \in \P_2(\R)$
has finite 2nd moment as well.
Consequently,
$\widehat{\Radon}[\mu] \in \Lebesgue_{\rho \times u_{\Sphere_1}}^2(\R \times \Sphere_1)$.
For $\mu, \nu \in \P_2(\R^2)$,
the norm distance
\begin{align*}
    &\lVert 
        \widehat{\Radon}[\mu] - \widehat{\Radon}[\nu]
    \rVert_{\rho \times u_{\Sphere_1}}
    \coloneqq
    \\
    &\Bigl(
        \int_{\Sphere_1} \int_\R
        \lvert 
            \widehat{\Radon}[\mu](t, \bftheta)
            -
            \widehat{\Radon}[\nu](t, \bftheta)
        \rvert^2
        \d \rho(t) \d u_{\Sphere_1}(\bftheta)
    \Bigr)^{\frac{1}{2}}
\end{align*}
is also called sliced Wasserstein-2 distance in the literature \cite{Bonneel2015}.

Another generalization of the CDT to multi-dimensional domains
is the so-called linear optimal transport (LOT) \cite{Moosmueller2023}.
This approach is based on the fact that
the CDT computes an optimal Monge map,
i.e.,
a solution to \eqref{eq:monge}.
Similarly,
LOT maps a given measure to an optimal
multi-dimensional Monge map $T \colon \R^d \to \R^d$.
The key feature of both transforms%
---R-CDT and LOT---%
is that they allow for linear separation
of certain measure classes.
The theoretical advantage of LOT is that 
this mapping can be interpreted as lifting
to the Wasserstein tangent space,
which is not the case for R-CDT.
On the downside,
the computation of LOT is numerical more challenging
and requires approximations.
In contrast to this,
the R-CDT can be computed analytically.

\subsection{Normalized R-CDT}

The R-CDT is by itself not invariant under affine transformations,
which emerge in various applications.
More precisely,
the R-CDT inherits the behavior of the Radon transform
observed in §~\ref{sec:Radon_affine_transform}.
Notice that
the translation and dilation of $\Radon_\bftheta[\mu]$
causes a horizontal shift (addition of a constant)
and a scaling (multiplication with a constant)
of $\widehat{\Radon}_\bftheta [\mu]$,
respectively.
In the first normalization step,
we revert these effects by
ensuring zero mean and unit standard deviation
of the R-CDT projection.
More precisely,
we define the {\em normalized R-CDT} (NR-CDT)
$\NRCDT [\mu] \colon \R \times \Sphere_1 \to \R$ 
of $\mu \in \P_2(\R^2)$ via
\begin{equation*}
    \NRCDT [\mu](t,\bftheta) 
    \coloneqq
    \NRCDT_\bftheta[\mu](t),
    \quad (t,\bftheta) \in \R \times \Sphere_1,
\end{equation*}
with
\begin{equation*}
    \NRCDT_\bftheta[\mu](t)
    \coloneqq
    \frac
    {\widehat{\Radon}_\bftheta[\mu](t) 
    - 
    \mean(\widehat{\Radon}_\bftheta[\mu])}
    {\std(\widehat{\Radon}_\bftheta[\mu])},
\end{equation*}
where, for $g \in \Lebesgue^2_\rho(\R)$,
\begin{equation*}
    \mean(g) 
    \coloneqq 
    \int_\R g(s) \d \rho(s)
\end{equation*}
and
\begin{equation*}
    \std(g) 
    \coloneqq 
    \Bigr(
        \int_\R 
        \lvert g(s) - \mean(g) \rvert^2 
        \d \rho(s)
    \Bigr)^{\frac{1}{2}}.
\end{equation*}

To ensure
that the NR-CDT is well defined,
we have to guarantee
that the standard deviation of the R-CDT projection does not vanish.
For this,
we restrict ourselves to measures
whose supports are not contained in a straight line.
More precisely,
we consider the class
\begin{align*}
    \P_\c^*(\R^2) 
    \coloneqq
    \{\mu \in \P(\R^2)
    \mid
    & \supp(\mu) \subsubset \R^2
    \land \\
    & \dim(\supp(\mu)) > 1\}
\end{align*}
satisfying
\begin{equation*}
    \P_\c^*(\R^2) 
    \subset
    \P_2(\R^2).
\end{equation*}
Here,
$\subsubset$ denotes a compact subset, 
and $\dim$ the dimension 
of the affine hull.
For these,
the standard deviation of the restricted Radon transform
is bounded away from zero and
cannot vanish.
If $\dim(\supp(\mu)) = 1$,
$\mu$ is supported on a line and
$\NRCDT_\bftheta[\mu]$ is well-defined
for all $\bftheta \in \Sphere^1$
except the two being
orthogonal to the line.
If $\mu$ is a single point measure,
$\NRCDT_\bftheta[\mu]$ is not well-defined
for any $\bftheta \in \Sphere^1$.

\begin{proposition} 
    \label{prop:sigma_bounded}
    Let $\mu \in \P_\c^*(\R^2)$.
    Then, there exists a constant $c > 0$ such that
    \begin{equation*}
        \std(\widehat{\Radon}_\bftheta [\mu]) \geq c
        \quad \forall \, \bftheta \in \Sphere_1.
    \end{equation*}
\end{proposition}

For the proof,
we first show the following continuity.

\begin{lemma}
    For fixed $\mu \in \P_\c^*(\R^2)$,
    the functions
    $\bftheta \in \Sphere_1 \mapsto \mean(\widehat{\Radon}_\bftheta [\mu]) \in \R$
    and
    $\bftheta \in \Sphere_1 \mapsto \std(\widehat{\Radon}_\bftheta [\mu]) \in \R_{\geq 0}$
    are continuous.
\end{lemma}

\begin{proof}
    We rewrite the mean as
    \begin{align*}
        \mean(\widehat{\Radon}_\bftheta[\mu])
        & =
        \int_\R \widehat{\Radon}_\bftheta [\mu] (t) \d \rho(t) \\
        & =
        \int_\R t  \d \Radon_\bftheta [\mu] (t) \\
        & =
        \int_{\R^2} \langle \bfx, \bftheta \rangle \d \mu(\bfx).
    \end{align*}
    Since the integrand is continuous in $\bftheta$
    and uniformly bounded by 
    $\lvert \langle \cdot, \bftheta \rangle \rvert \le \lVert \cdot \rVert$,
    the dominated convergence yields the assertion.
    Analogously,
    we have
    \begin{align*}
        \std(\widehat{\Radon}_\bftheta [\mu])
        & =
        \Bigl(
        \int_\R 
        \lvert
            \widehat{\Radon}_\bftheta [\mu] (t)
            -
            \mean(\widehat{\Radon}_\bftheta [\mu])
        \rvert^2
        \d \rho(t)
        \Bigr)^{\frac{1}{2}}
        \\
        & =
        \Bigl(
        \int_{\R^2} 
        \lvert
            \langle \bfx, \bftheta \rangle
            -
            \mean(\widehat{\Radon}_\bftheta [\mu])
        \rvert^2
        \d \mu(\bfx)
        \Bigr)^{\frac{1}{2}}.
    \end{align*}
    The integrand is again continuous in $\bftheta$
    and uniformly bounded by
    \begin{align*}
        \lvert
            \langle \cdot, \bftheta \rangle
            -
            \mean(\widehat{\Radon}_\bftheta [\mu])
        \rvert^2
        \le 
        & 2 \lVert \cdot \rVert^2 
        + \\
        & 2 \max_{\bftheta \in \Sphere_1} \;
        (\mean(\widehat{\Radon}_\bftheta [\mu]))^2;
    \end{align*}
    thus,
    the standard deviation is continuous by dominated convergence.
\end{proof}

\begin{proof}[Proof of Proposition~\ref{prop:sigma_bounded}]
    Assume the contrary,
    i.e.,
    $c = 0$.
    Then,
    due to the continuity of $\bftheta \mapsto \std(\widehat{\Radon}_\bftheta [\mu])$,
    there exists a minimizing and convergent sequence in $\Sphere_1$
    whose limit $\bftheta$ is attained and satisfies
    $\std(\widehat{\Radon}_\bftheta [\mu]) = 0$,
    i.e.,
    \begin{equation*}
        \int_{\R^2} \lvert \langle \bfx, \bftheta \rangle - \mean(\widehat{\Radon}_\bftheta [\mu]) \rvert^2 \d \mu(\bfx) = 0.
    \end{equation*}
    Hence, 
    the support of $\mu$ is contained in the line 
    $\{\bfx \in \R^2 \mid \langle \bfx, \bftheta \rangle = \mean(\widehat{\Radon}_\bftheta [\mu])\}$
    in contradiction to $\mu \in \P_\c^*(\R^2)$.
\end{proof}

The NR-CDT is nearly invariant under affine transformations
up to bijective remappings of the directions,
i.e.,
up to a resorting of the family 
$\{\NRCDT_\bftheta[\mu] \mid \bftheta \in \Sphere_1\}$.

\begin{proposition}
    Let $\mu \in \P_\c^*(\R^2)$,
    $\bfA \in \GL(2)$,
    $\bfy \in \R^2$,
    and $\mu_{\bfA, \bfy}$ as in \eqref{eq:aff-meas}.
    Then, for any $\bftheta \in \Sphere_1$,
    the NR-CDT satisfies
    \begin{equation*}
        \NRCDT_\bftheta[\mu_{\bfA, \bfy}] 
        =
        \NRCDT_{\frac{\bfA^\top \bftheta}{\lVert\bfA^\top \bftheta\rVert}} [\mu].
    \end{equation*}
\end{proposition}

\begin{proof}
    Transferring Proposition~\ref{prop:RT_transformation} to the CDT space,
    we have
    \begin{equation*}
        \widehat{\Radon}_\bftheta [\mu_{\bfA, \bfy}](t)
        =
        \lVert \bfA^\top \bftheta \rVert \,
        \widehat{\Radon}_{\phi_\bfA(\bftheta)} [\mu](t)
        + 
        \langle \bfy, \bftheta \rangle
    \end{equation*}
    with the bijection $\phi_\bfA (\bftheta) \coloneqq (\bfA^\top \bftheta) / \lVert \bfA^\top \bftheta \rVert$,
    $\bftheta \in \Sphere_1$;
    so that
    \begin{equation*}
        \mean(\widehat{\Radon}_\bftheta [\mu_{\bfA,\bfy}]) 
        =
        \lVert \bfA^\top \bftheta \rVert  
        \mean (\widehat{\Radon}_{\phi_\bfA(\bftheta)} [\mu] ) 
        +
        \langle \bfy, \bftheta \rangle
    \end{equation*}
    and
    \begin{equation*}
        \std(\widehat{\Radon}_\bftheta [\mu_{\bfA, \bfy}]) 
        =
        \lVert \bfA^\top \bftheta \rVert 
        \std(\widehat{\Radon}_{\phi_\bfA(\bftheta)} [\mu]).
    \end{equation*}
    Consequently,
    \begin{align*}
        \NRCDT_\bftheta [\mu_{\bfA, \bfy}](t) 
        & =
        \frac
        {
            \widehat{\Radon}_{\phi_\bfA(\bftheta)} [\mu](t)
            - 
            \mean(\widehat{\Radon}_{\phi_\bfA(\bftheta)} [\mu])
        }{
            \std(\widehat{\Radon}_{\phi_\bfA(\bftheta)} [\mu])
        }
        \\
        & =
        \NRCDT_{\phi_\bfA(\bftheta)} [\mu](t).
    \end{align*}
\end{proof}

\subsection{Max-Normalized R-CDT}
\label{sec:mnrcdt}

In the final normalization step,
we treat the resorting of 
$\{\NRCDT_\bftheta[\mu] \mid \bftheta \in \Sphere_1\}$.
Since the underlying mapping is unknown in general,
we propose to take the supremum over all directions.
More precisely,
for $\mu \in \P_\c^*(\R^2)$, 
we define its \emph{max-normalized R-CDT} (\mNRCDT) $\maxNRCDT [\mu] \colon \R \to \R$ 
via
\begin{equation*}
    \maxNRCDT [\mu](t) 
    \coloneqq  
    \sup_{\bftheta \in \Sphere_1} \NRCDT_\bftheta [\mu](t),
    \quad  t \in \R.
\end{equation*}
We show that
$\maxNRCDT$ maps a given measure
to a bounded function
so that
the \mNRCDT{} space $\maxNRCDT[\P_\c^*(\R^2)]$
is contained in $\Lebesgue_\rho^\infty(\R)$ for the underlying reference measure $\rho \in \P(\R)$.

\begin{proposition}
	\label{prop:max-bounded}
    Let $\mu \in \P_\c^*(\R^2)$.
    Then, $\maxNRCDT [\mu] \in \Lebesgue^\infty_\rho(\R)$.
\end{proposition}

\begin{proof}
    The restricted Radon operator cannot enlarge the size of the support
    $\diam(\mu) \coloneqq \sup_{\bfx, \bfy \in \supp(\mu)} \, \lVert \bfx - \bfy \rVert$,
    i.e.,
    $\diam(\Radon_\bftheta [\mu]) \le \diam(\mu)$.
    Moreover, the range of $\widehat{\Radon}_\bftheta [\mu]$ coincides
    with the support of $\Radon_\bftheta [\mu]$.
    Using that the mean lies in the convex hull of the support,
    we thus have
    \begin{equation*}
        \lvert 
        \widehat{\Radon}_\bftheta [\mu] (t) 
        - 
        \mean(\widehat{\Radon}_\bftheta [\mu])
        \rvert
        \le 
        \diam(\mu)
        \quad
        \forall \, \bftheta \in \Sphere_1.
    \end{equation*}
    Since $\mu \in \P_\c^*(\R^2)$,
    according to Proposition~\ref{prop:sigma_bounded}
    we have
    $c \coloneqq \min_{\bftheta \in \Sphere_1} \std(\widehat{\Radon}_\bftheta [\mu]) > 0$.
    Thus,
    the \mNRCDT{} is bounded by
    $\lvert \maxNRCDT [\mu](t) \rvert \le \diam(\mu) / c$
    for all $t \in \R$.
\end{proof}

With the \mNRCDT{},
we accomplish our objective
to define a transport-based transform
that is invariant under affine transformations.

\begin{proposition}
    \label{prop:inv-max-nrcdt}
    Let $\mu \in \P_\c^*(\R^2)$,
    $\bfA \in \GL(2)$,
    $\bfy \in \R^2$,
    and $\mu_{\bfA, \bfy}$ as in~\eqref{eq:aff-meas}.
    Then,
    the \mNRCDT{} satisfies
    $\maxNRCDT[\mu_{\bfA, \bfy}] = \maxNRCDT[\mu]$.
\end{proposition}

\begin{proof}
    Since the mapping 
    $\phi_\bfA(\bftheta) \coloneqq (\bfA^\top \bftheta) / \lVert \bfA^\top \bftheta \rVert$
    is a bijection on $\Sphere_1$,
    we obtain
    \begin{align*}
        \maxNRCDT[\mu_{\bfA,\bfy}](t)
        & =
        \sup_{\bftheta \in \Sphere_1} \NRCDT_\bftheta [\mu_{\bfA, \bfy}](t)
        \\
        & =
        \sup_{\bftheta \in \Sphere_1} \NRCDT_{\phi_\bfA(\bftheta)} [\mu](t)
        \\
        & =
        \maxNRCDT [\mu](t).
    \end{align*}        
\end{proof}

The invariance 
under affine transformations 
immediately yields the linear separability
of affine measure classes,
which originate from a single template.

\begin{theorem}
    \label{thm:sep-max-nrcdt}
    For template measures $\mu_0, \nu_0 \in \P_\c^*(\R^2)$ with
    \begin{equation*}
        \maxNRCDT [\mu_0] \neq \maxNRCDT [\nu_0]
    \end{equation*}
    consider the classes
    \begin{subequations}
    \label{eq:aff-class}
    \begin{align}
        \F 
        &=
        \bigl\{(\bfA \cdot + \bfy)_\# \mu_0 \mid \bfA \in \GL(2), \, \bfy \in \R^2\bigr\},
        \\
        \G 
        &= 
        \bigl\{(\bfA \cdot + \bfy)_\# \nu_0 \mid \bfA \in \GL(2), \, \bfy \in \R^2\bigr\}.
    \end{align}
    \end{subequations}
    Then, $\F \subset \P_\c^*(\R^2)$ and $\G \subset \P_\c^*(\R^2)$ are linearly separable in \mNRCDT{}~space.
\end{theorem}

\begin{proof}
    Due to the affine construction of $\F$ and $\G$,
    Proposition~\ref{prop:inv-max-nrcdt} yields
    $\maxNRCDT [\F] = \bigl\{\maxNRCDT [\mu_0]\bigr\}$
    and $\maxNRCDT [\G] = \bigl\{\maxNRCDT [\nu_0]\bigr\}$.
    Hence, 
    the assumption
    $\maxNRCDT [\mu_0] \neq \maxNRCDT [\nu_0]$
    implies the linear separability of 
    $\maxNRCDT [\F]$ and $\maxNRCDT [\G]$
    in $\Lebesgue^\infty_\rho(\R)$.
\end{proof}

The proof of Theorem~\ref{thm:sep-max-nrcdt} reveals that 
the classes $\F$ and $\G$ collapse to single points in \mNRCDT{}~space, 
directly allowing for linear separability.
This theorem and the following generalizations are stated in this form
to highlight the analogy to
R-CDT \cite{Kolouri2016} and LOT \cite{Moosmueller2023, Park2018},
where linear separability is observed as well
as one of the key achievements.
The transformation of affine classes to singletons, however,
is in contrast to the results in the literature,
which can only tackle a subset of affine transforms,
and where the considered classes are mapped to convex sets.

\subsection{\texorpdfstring{$h$}{h}-Normalized R-CDT}
\label{sec:hNRCDT}

We now describe a more general
final normalization step 
than proposed in §~\ref{sec:mnrcdt}
for treating the resorting of 
$\{\NRCDT_\bftheta[\mu] \mid \bftheta \in \Sphere_1\}$.
To this end,
let the \emph{aggregation function}
$h \colon \Lebesgue^\infty(\Sphere_1) \to \R$ be boundedness-preserving,
i.e.,
bounded sets are mapped to bounded sets,
such that
$h(g \circ \phi_\bfA) = h(g)$
for all $g \in \Lebesgue^\infty(\Sphere_1)$
and $\phi_\bfA(\bftheta) = (\bfA^\top \bftheta) / \lVert \bfA^\top \bftheta \rVert$ with $\bfA \in \GL(2)$.
With this,
for $\mu \in \P_\c^*(\R^2)$,
we define its \emph{$h$-normalized R-CDT} (\hNRCDT{}) $\genNRCDT{} [\mu] \colon \R \to \R$ 
via
\begin{equation*}
    \genNRCDT{} [\mu](t) 
    \coloneqq  
    h(\NRCDT[\mu](t,\cdot)),
    \quad  t \in \R.
\end{equation*}
We show that
$\genNRCDT{}$ maps a given measure
to a bounded function
so that
the \hNRCDT{} space $\genNRCDT{} [\P_\c^*(\R^2)]$
is contained in $\Lebesgue_\rho^\infty(\R)$ for the underlying reference measure $\rho \in \P(\R)$.

\begin{proposition}
    $\genNRCDT{} [\mu] \in \Lebesgue^\infty_\rho(\R)$
    for all
    $\mu \in \P_\c^*(\R^2)$.
\end{proposition}

\begin{proof}
    As in the proof of Proposition~\ref{prop:max-bounded},
    we have
    \begin{equation*}
        \lvert 
        \widehat{\Radon}_\bftheta [\mu] (t) 
        - 
        \mean(\widehat{\Radon}_\bftheta [\mu])
        \rvert
        \le 
        \diam(\mu)
        \quad
        \forall \, \bftheta \in \Sphere_1
    \end{equation*}
    and
    $c \coloneqq \min_{\bftheta \in \Sphere_1} \std(\widehat{\Radon}_\bftheta [\mu]) > 0$.
    Thus,
    the NR-CDT is uniformly bounded by
    $\lvert \NRCDT[\mu](t,\bftheta) \rvert \le \diam(\mu) / c$
    for all $t \in \R$ and $\bftheta \in \Sphere_1$
    so that
    $\NRCDT[\mu](t,\cdot) \in \Lebesgue^\infty(\Sphere_1)$.
    Since $h \colon \Lebesgue^\infty(\Sphere_1) \to \R$ is bounded,
    there exists a constant $C > 0$
    such that
    $\lvert \genNRCDT{} [\mu](t) \rvert \le C$
    for all $t \in \R$.
\end{proof}

With the \hNRCDT{},
we have defined a whole family of
transport-based transforms
that are invariant under affine transformations.

\begin{proposition}
    \label{prop:inv-h-nrcdt}
    Let $\mu \in \P_\c^*(\R^2)$,
    and $\mu_{\bfA, \bfy}$ as in~\eqref{eq:aff-meas}
    with
    $\bfA \in \GL(2)$,
    $\bfy \in \R^2$.
    Then,
    the \hNRCDT{} satisfies
    $\genNRCDT{} [\mu_{\bfA, \bfy}] = \genNRCDT{} [\mu]$.
\end{proposition}

\begin{proof}
    By assumption on $h$,
    we directly obtain
    \begin{align*}
        \genNRCDT{} [\mu_{\bfA,\bfy}](t)
        & =
        h(\NRCDT[\mu_{\bfA, \bfy}](t,\cdot))
        \\
        & =
        h(\NRCDT[\mu](t,\cdot) \circ \phi_\bfA)
        \\
        & =
        h(\NRCDT[\mu](t,\cdot))
        =
        \genNRCDT{} [\mu](t)
    \end{align*}
    for all $t \in \R$.
\end{proof}

The invariance 
under affine transformations 
immediately yields linear separability
of affine measure classes
originating from single templates.

\begin{theorem}
    \label{thm:sep-h-nrcdt}
    For template measures $\mu_0, \nu_0 \in \P_\c^*(\R^2)$ with
    \begin{equation*}
        \genNRCDT{} [\mu_0] \neq \genNRCDT{} [\nu_0]
    \end{equation*}
    consider the classes
    \begin{subequations}
    \label{eq:aff-class-hNRCDT}
    \begin{align}
        \F 
        &=
        \bigl\{(\bfA \cdot + \bfy)_\# \mu_0 \mid \bfA \in \GL(2), \, \bfy \in \R^2\bigr\},
        \\
        \G 
        &= 
        \bigl\{(\bfA \cdot + \bfy)_\# \nu_0 \mid \bfA \in \GL(2), \, \bfy \in \R^2\bigr\}.
    \end{align}
    \end{subequations}
    Then, $\F \subset \P_\c^*(\R^2)$ and $\G \subset \P_\c^*(\R^2)$ are linearly separable in \hNRCDT{}~space.
\end{theorem}

\begin{proof}
    By construction of $\F$ and $\G$,
    Proposition~\ref{prop:inv-h-nrcdt} yields
    $\genNRCDT{} [\F] = \bigl\{\genNRCDT{} [\mu_0]\bigr\}$
    and $\genNRCDT{} [\G] = \bigl\{\genNRCDT{} [\nu_0]\bigr\}$.
    Hence, 
    the assumption
    $\genNRCDT{} [\mu_0] \neq \genNRCDT{} [\nu_0]$
    implies the linear separability of 
    $\genNRCDT{} [\F]$ and $\genNRCDT{} [\G]$
    in $\Lebesgue^\infty_\rho(\R)$.
\end{proof}

If $\eta \colon \R \to \R$
is a boundedness-preserving function
and $H \colon \B(\R) \to \R$ is bounded,
where $\B(\R)$ denotes the set of bounded subsets of $\R$,
we set
\begin{equation*}
h(g) = H(\{\eta(g(\bftheta)) \mid \bftheta \in \Sphere_1\}),
\quad g \in \Lebesgue^\infty(\Sphere_1),
\end{equation*}
so that
$h(g \circ \phi) = h(g)$
is automatically satisfied
for all $g \in \Lebesgue^\infty(\Sphere_1)$
and all bijections $\phi \colon \Sphere_1 \to \Sphere_1$.

Note that
$\eta = \mathrm{Id}$ and $H = \sup$
recovers the \mNRCDT{}
from Section~\ref{sec:mnrcdt}.
The first obvious variation
would be choosing
$\eta = \mathrm{Id}$ and $H = \inf$.
The induced operator, however,
does not contain
additional information
as compared to the \mNRCDT{}.
This is shown in
the next proposition,
where, for simplicity,
we assume that
the reference measure $\rho$
is given by the
uniform measure $u_{[0,1]}$.

\begin{proposition}
    Let $\rho = u_{[0,1]}$
    and $\mu \in \P_\c^*(\R^2)$.
    Then,
    \begin{equation*}
        \inf_{\bftheta \in \Sphere_1} \NRCDT_\bftheta [\mu](t)
        =
        -\maxNRCDT[\mu](1-t)
        \quad
        \forall \, t \in \R.
    \end{equation*}
\end{proposition}

\begin{proof}
    As $\rho = u_{[0,1]}$
    by assumption,
    for all $t \in \R$,
    we have
    \begin{equation*}
        F_{u_{[0,1]}}(1-t)
        =
        1 - F_{u_{[0,1]}}(t).
    \end{equation*}
    Moreover,
    for $\tau \in [0,1]$,
    \begin{equation*}
        F_{\Radon_\bftheta[\mu]}^{[-1]}(1-\tau)
        =
        - F_{\Radon_{-\bftheta}[\mu]}^{[-1]}(\tau),
    \end{equation*}
    since the Radon transform
    satisfies the eveness property
    \begin{equation*}
        \Radon_\bftheta[\mu]((s,\infty))
        =
        \Radon_{-\bftheta}[\mu]((-\infty,-s))
        \quad
        \forall \, s \in \R.
    \end{equation*}
    This implies that
    \begin{equation*}
        \widehat{\Radon}_\bftheta [\mu](1-t)
        =
        -\widehat{\Radon}_{-\bftheta} [\mu](t)
    \end{equation*}
    and,
    consequently,
    \begin{align*}
        & \inf_{\bftheta \in \Sphere_1} \NRCDT_\bftheta [\mu](t)
        =
        - \sup_{\bftheta \in \Sphere_1} \frac{-\widehat{\Radon}_\bftheta [\mu](t) - \mean(-\widehat{\Radon}_\bftheta [\mu])}{\std(-\widehat{\Radon}_\bftheta [\mu])}
        \\
        & \quad =
        - \sup_{\bftheta \in \Sphere_1} \frac{\widehat{\Radon}_{-\bftheta} [\mu](1-t) - \mean(\widehat{\Radon}_{-\bftheta} [\mu])}{\std(\widehat{\Radon}_{-\bftheta} [\mu])}
        \\
        & \quad =
        -\maxNRCDT[\mu](1-t),
    \end{align*}
    as stated.
\end{proof}

The following example
lists further variations
of the $h$-normalization
based on $\eta$ and $H$.    

\begin{example} \label{ex:h_normalization}
    Suitable choices of $\eta$ and $H$ include
    \begin{enumerate}[label=\alph*)]
        \item
            $\eta = |\cdot|$ and $H = \sup$,
            i.e.,
            \begin{equation*}
                \genNRCDT{a}[\mu](t) 
                =
                \sup_{\bftheta \in \Sphere_1} |\NRCDT_\bftheta [\mu](t)|,
            \end{equation*}
        \item
            $\eta = |\cdot|$ and $H = \inf$,
            i.e.,
            \begin{equation*}
                \genNRCDT{b}[\mu](t) 
                =
                \inf_{\bftheta \in \Sphere_1} |\NRCDT_\bftheta [\mu](t)|,
            \end{equation*}
        \item
            $\eta = |\cdot|$ and $H = \sup - \inf$,
            i.e.,
            \begin{equation*}
                \genNRCDT{c}[\mu](t) 
                =
                \sup_{\bftheta \in \Sphere_1} |\NRCDT_\bftheta [\mu](t)| - 
                \inf_{\bftheta \in \Sphere_1} |\NRCDT_\bftheta [\mu](t)|.
            \end{equation*}
        \item
            $\eta = \mathrm{Id}$ and $H = \sup - \inf$,
            i.e.,
            \begin{equation*}
                \genNRCDT{d}[\mu](t) 
                =
                \sup_{\bftheta \in \Sphere_1} \NRCDT_\bftheta [\mu](t) - 
                \inf_{\bftheta \in \Sphere_1} \NRCDT_\bftheta [\mu](t),
            \end{equation*}
    \end{enumerate}
\end{example}

Note that $\genNRCDT{a}[\mu]$
and $\genNRCDT{b}[\mu]$
capture only
a single angular
piece of information,
the maximum or minimum, respectively,
whereas $\genNRCDT{c}[\mu]$
and $\genNRCDT{d}[\mu]$
encode the angular range.
To include even further knowledge,
we now construct an example
for a normalization
not only operating on the image
$\{\NRCDT_\bftheta [\mu](t) \mid \bftheta \in \Sphere_1\}$.
To this end,
we introduce the {\em TV-normalized NR-CDT} (\tvNRCDT)
$\varNRCDT[\mu] \colon \R \to \R$ via
\begin{equation*}
    \varNRCDT[\mu](t) 
    \coloneqq
    \sup_{P \in \PP}
    \sum_{i=1}^{n_P}
    |\NRCDT_{\bftheta_{i+1}}[\mu](t)
    -
    \NRCDT_{\bftheta_i}[\mu](t)|,
\end{equation*}
where the supremum
runs over the set of
partitions
$\PP = \{P = (\bftheta_1,\ldots,\bftheta_{n_P+1})
\mid P \text{ is partition of } \Sphere_1\}$,
which means that there exist
$0 \leq \vartheta_1 < \vartheta_{n_P} < 2\pi$
so that $\bftheta_i = (\cos(\vartheta_i),\sin(\vartheta_i))^\top$
and $\bftheta_{n_P+1} = \bftheta_1$.
To ensure well-definedness,
we restrict ourselves to
the class
\begin{equation}
    \label{eq:P_tv}
    \P_{\mathrm{tv}}^*(\R^2) 
    \coloneqq
    \{\mu \in \P_\c^*(\R^2)
    \mid
    \varNRCDT[\mu](t) < \infty ~
    \forall \, t \in \R\},
\end{equation}    
which is a suitable setting
for all our numerical experiments below.
This normalization variant corresponds to
\begin{equation*}
    h(g)
    =
    \sup_{P \in \PP}
    \sum_{i=1}^{n_P}
    |g(\bftheta_{i+1})
    -
    g(\bftheta_i)|,
\end{equation*}
which satisfies $h(g \circ \phi_\bfA) = h(g)$
for all $\phi_\bfA(\bftheta) = (\bfA^\top \bftheta) / \lVert \bfA^\top \bftheta \rVert$ with $\bfA \in \GL(2)$
since $\phi_A(\PP) = \PP$.
Therefore,
Theorem~\ref{thm:sep-h-nrcdt}
also holds for \tvNRCDT{},
when assuming
$\mu_0, \nu_0 \in \P_{\mathrm{tv}}^*(\R^2)$.

\section{Generalized NR-CDT}
\label{sec:gen-NRCDT}

The model behind \hNRCDT{} is clearly tailored 
to 2d pattern recognition tasks
under affine transformations.
Theoretically,
this idea may be transferred to 
the multi-dimensional 
and non-Euclidean setting.
For the underlying R-CDT
and absolutely continuous probabilities,
this extension is studied in \cite{Kolouri2019},
where the Radon transform is replaced 
by the so-called generalized Radon transform.

To extend this approach beyond functions,
and to allow more flexibility,
we consider arbitrary (probability) measures 
on a Polish space $\XX$,
i.e.,
a separable completely metrizable topological space.
On the basis of a \emph{direction set} $\Theta$
and a \emph{defining function}
$\phi \colon \XX \times \Theta \to \R$
such that 
$\phi(\cdot, \bftheta) \colon \XX \to \R$
is measurable for all $\bftheta \in \Theta$,
we define the \emph{generalized slicing operator}
$S_{\phi,\bftheta} \colon \XX \to \R$ by
\begin{equation}
    \label{eq:gen-slicing-operator}
    S_{\phi, \bftheta}(\bfx)
    \coloneqq \phi(\bfx, \bftheta),
    \quad
    \bfx \in \XX, ~
    \bftheta \in \Theta,
\end{equation}
and the \emph{generalized restricted Radon transform} via
\begin{equation}
    \label{eq:restrict-gen-radon-meas}
    \Radon_{\phi,\bftheta} 
    \colon
    \M(\XX) \to \M(\R),
    \quad
    \mu \mapsto (S_{\phi, \bftheta})_\# \mu.
\end{equation}

If $\Theta$ is a Polish space too,
then the restricted Radon transforms may be glued together
to form a measure on $\R \times \Theta$.
More precisely,
for a \emph{gluing measure} $\gamma \in \P(\Theta)$,
we define the \emph{generalized Radon transform}
$\Radon_{\phi,\gamma} \colon \M(\XX) \to \M(\R \times \Theta)$
as
\begin{equation*}
    \Radon_{\phi,\gamma} [\mu] 
    \coloneqq 
    (\Glue_\phi)_\#[\mu \times \gamma]
\end{equation*}
with
$\Glue_\phi(\bfx, \bftheta) \coloneqq (S_{\phi, \bftheta}(\bfx), \bftheta)$
for $(\bfx, \bftheta) \in \XX \times \Theta$.
Similarly to before,
the generalized (restricted) Radon transform maps
probability measures to probability measures.
Furthermore,
the disintegration of the Radon transform carries over.

\begin{proposition}
    Let $\mu \in \M(\XX)$.
    Then, 
    $\Radon_{\phi,\gamma} [\mu]$ can be disintegrated 
    into the family $\Radon_{\phi,\bftheta} [\mu]$ 
    with respect to $\gamma$,
    i.e., 
    for all continuous $g \in C_0(\R \times \Theta)$ vanishing at infinity, 
    we have
    \begin{equation*}
        \langle \Radon_{\phi,\gamma} [\mu], g\rangle 
        = 
        \int_{\Theta} \langle \Radon_{\phi,\bftheta} [\mu], g(\cdot, \bftheta)\rangle \d \gamma(\bftheta).
    \end{equation*}
\end{proposition}

\begin{proof}
    Using Fubini's theorem, 
    we directly obtain 
    \begin{align*}
        \langle \Radon_{\phi,\gamma} [\mu], g\rangle
        & = 
        \int_{\R \times \Theta} g(t, \bftheta) \d \Glue_\#[\mu \times \gamma] (t, \bftheta)
        \\
        & = 
        \int_{\Theta} \int_{\XX} g(S_{\phi,\bftheta}(\bfx), \bftheta) \d \mu(\bfx) \d \gamma(\bftheta) 
        \\
        & = 
        \int_{\Theta} \int_{\XX} g(t, \bftheta) \d [(S_{\phi,\bftheta})_{\#}\mu](t) \d \gamma(\bftheta)
        \\
        & = 
        \int_{\Theta} \langle \Radon_{\phi,\bftheta} [\mu], g(\cdot, \bftheta)\rangle \d \gamma(\bftheta).
        \tag*{\qedhere}
    \end{align*}
\end{proof}

Notice that
the original generalized Radon transform for smooth functions
is based on a so-called double fibering \cite{Gelfand1969}.
If $\phi$ satisfies certain regularity assumptions,
this generalized Radon transform becomes invertible
\cite{Beylkin1984,Quinto1980,Gelfand1969,Homan2017}.
Our generalized Radon transform for measures
is closely related to
the \emph{back projection}
\begin{equation*}
    \Radon_{\phi,\gamma}^*[g](\bfx) 
    \coloneqq 
    \int_{\Theta} 
    g(S_{\phi,\bftheta}(\bfx), \theta) 
    \d \gamma(\bftheta),
    \quad \bfx \in \XX.
\end{equation*}

\begin{proposition} 
    The generalized Radon transform of $\mu \in \M(\XX)$ satisfies
    \begin{equation*}
        \langle \Radon_{\phi,\gamma} [\mu], g\rangle 
        = 
        \langle \mu, \Radon_{\phi,\gamma}^* [g]\rangle
        \quad \forall \,
        g \in \Lebesgue_{\lambda_\R\times\gamma}^\infty(\R \times \Theta).
    \end{equation*}
\end{proposition}

\begin{proof}
    For all $\mu \in \M(\XX)$
    and $g \in \Lebesgue^\infty(\R \times \Theta)$,
    applying Fubini's theorem gives
    \begin{align*}
        \langle \Radon_{\phi, \gamma} [\mu], g\rangle
        &=
        \int_{\R \times \Theta} 
        g(t, \bftheta)
        \d (\Glue_\phi)_\#[\mu \times \gamma] (t, \bftheta)
        \\
        &=
        \int_{\XX} \int_{\Theta} 
        g(S_{\phi,\bftheta}(\bfx), \bftheta) 
        \d \gamma(\bftheta) \d \mu(\bfx).
        \tag*{\qedhere}
    \end{align*}
\end{proof}

To introduce a generalized \hNRCDT{},
we have to ensure that
each normalization step is well-defined.
For this,
we restrict ourselves to
\emph{uniformly bounded} defining functions $\phi$
meaning that, 
for every $L>0$,
there exits $M_L >0$ satisfying
\begin{equation*}
    \lvert \phi( \bfx, \bftheta) - \phi(\bfy, \bftheta) \rvert \le M_L
    \enspace
    \forall \, d_\XX( \bfx , \bfy) \le L, ~
    \forall \, \bftheta \in \Theta,
\end{equation*}
where $d_\XX$ denotes a topology-compatible metric on $\XX$.
Furthermore,
we only consider measures from
\begin{align*}
    \P_{\phi, \mathrm c}^+(\XX)
    &\coloneqq
    \{ \mu \in \P(\XX)
    \mid
    \supp(\mu) \subsubset \XX 
    \, \wedge
    \\
    &\qquad
    \exists \, c > 0:
    \std(\Radon_{\phi,\bftheta}[\mu]) \ge c
    ~ \forall \, \theta \in \Theta \}.
\end{align*}
For uniformly bounded $\phi$
and $\mu \in \P_\c^+(\XX)$,
the support of $\Radon_{\phi, \bftheta}[\mu]$ is bounded;
so their means and standard derivations exist.

The {\em generalized NR-CDT}
$\NRCDT_{\phi} [\mu] \colon \R \times \Theta \to \R$ 
of $\mu \in \P_{\phi, \mathrm c}^+(\XX)$ 
is defined as
\begin{equation*}
    \NRCDT_\phi [\mu](t,\bftheta) 
    \coloneqq
    \NRCDT_{\phi,\bftheta} [\mu](t)
\end{equation*}
with
\begin{equation*}
    \NRCDT_{\phi,\bftheta} [\mu](t)
    \coloneqq
    \frac
    {\widehat{\Radon}_{\phi,\bftheta} [\mu](t) 
    - 
    \mean(\widehat{\Radon}_{\phi,\bftheta} [\mu])}
    {\std(\widehat{\Radon}_{\phi,\bftheta} [\mu])}
\end{equation*}
By construction,
$\NRCDT_{\phi}[\mu]$ is bounded on $\supp (\rho) \times \Theta$.
In the style of §~\ref{sec:hNRCDT},
let $h \colon \Lebesgue^\infty(\Theta) \to \R$ be bounded.
Unless otherwise stated,
we assume that
$h(g \circ \psi) = h(g)$
for all $g \in \Lebesgue^\infty(\Theta)$
and any bijection $\psi \colon \Theta \to \Theta$,
which is a stricter setting as in §~\ref{sec:hNRCDT}
to allow for a general discussion.
For $\mu \in \P_{\phi,\mathrm c}^+(\XX)$,
the \emph{generalized \hNRCDT{}} $\NRCDT_{h,\phi} [\mu] \colon \R \to \R$ 
may be defined via
\begin{equation*}
    \NRCDT_{h,\phi} [\mu](t) 
    \coloneqq  
    h(\NRCDT_\phi[\mu](t,\cdot)),
    \quad  t \in \R.
\end{equation*}
In the following,
we briefly discuss the situation for specific $\XX$ and $\phi$,
especially with respect to the possible invariant transformations.

\subsection{Multi-dimensional NR-CDT}
\label{sec:multi-NRCDT}

Choosing $\XX \coloneqq \R^d$ 
and $\Theta \coloneqq \Sphere_{d-1} \coloneqq \{ \bfx \in \R^d \mid \lvert \bfx \rvert = 1\}$
together with Euclidean inner product
\begin{equation*}
    \phi(\bfx, \bftheta)
    \coloneqq
    \langle\bfx, \bftheta \rangle,
    \quad
    \bfx \in \R^d,
    \bftheta \in \Sphere_{d-1},
\end{equation*}
we obtain the straightforward generalization
of our NR-CDT variants 
to the multi-dimensional setting. 
A brief inspection of the two-dimensional setting shows that
all results line-by-line generalize to
the multi-dimensional setting. 
In particular,
for measures in the class
\begin{align*}
    \P_\c^*(\R^d)
    \coloneqq
    \{ \mu \in \P(\R^d) \mid \, 
    &
    \supp(\mu) \subsubset \R^d \land
    \\
    &
    \dim(\supp (\mu)) > d - 1 \},
\end{align*}
whose supports are not concentrated on hyperplanes,
the standard deviation \smash{$\std(\widehat\Radon_{\bftheta}[\mu])$}
is uniformly bounded away from zero,
i.e.,
$\P_\c^*(\R^d) = \P_{\phi,\c}^+(\R^d)$.
The multi-dimensional \hNRCDT{} again promotes 
the linear separability of affine classes.

\begin{theorem}
    \label{thm:sep-multi-h-nrcdt}
    For template measures $\mu_0, \nu_0 \in \P_\c^*(\R^d)$ with
    \begin{equation*}
        \genNRCDT{} [\mu_0] \neq \genNRCDT{} [\nu_0]
    \end{equation*}
    consider the classes
    \begin{align*}
        \F 
        &=
        \bigl\{(\bfA \cdot + \bfy)_\# \mu_0 \mid \bfA \in \GL(d), \, \bfy \in \R^d\bigr\},
        \\
        \G 
        &= 
        \bigl\{(\bfA \cdot + \bfy)_\# \nu_0 \mid \bfA \in \GL(d), \, \bfy \in \R^d\bigr\}.
    \end{align*}
    Then, $\F \subset \P_\c^*(\R^d)$ and $\G \subset \P_\c^*(\R^d)$ are linearly separable in \hNRCDT{}~space.
\end{theorem}

\subsection{Circular NR-CDT}

Another example
for a generalized Radon transform
on the multi-dimensional Euclidean space is
the circular Radon transform \cite{Kuchment2006}.
Domain and direction set of this transform are given by 
$\XX \coloneqq \R^d$ and $\Theta \coloneqq \R^d$.
Furthermore,
the slicing operator reads
\begin{equation}
    \label{eq:slice-circ}
    \phi(\bfx, \bftheta) 
    \coloneqq
    \lVert \bfx - \bftheta \rVert,
    \quad
    \bfx, \bftheta \in \R^d.
\end{equation}
Note that, due to the reversed triangle inequality,
this function satisfies the above
uniform boundedness assumption.
Figuratively,
the circular Radon transform integrates
along circles with center $\bftheta \in \Theta$.
The resulting \hNRCDT{} is especially useful
for classification tasks under
isotropic affine transformations,
i.e.,
for probability classes build by
\begin{equation*}
    \mu_{s\bfQ, \bfy} 
    \coloneqq
    (s \bfQ \, \cdot + \bfy)_\# \mu
\end{equation*}
with $s > 0$,
$\bfQ \in \O(d)$ in the orthogonal group,
and $\bfy \in \R^d$.

\begin{proposition}
    \label{prop:iso-aff}
    For any $\bftheta \in \R^d$,
    the restricted circular Radon transform satisfies
    \begin{equation*}
        \Radon_{\phi,\bftheta}[\mu_{s \bfQ, \bfy}]
        = 
        (s \, \cdot)_\#
        \Radon_{\phi, \bfQ^\top(\bftheta - \bfy)} [\mu].
    \end{equation*}
\end{proposition}

\begin{proof}
    The circular Radon transform is induced by
    the slicing operator in~\eqref{eq:slice-circ},
    which here yields
    \begin{align*}
        \Radon_{\phi,\bftheta}[\mu_{s\bfQ,\bfy}]
        &=
        (S_{\phi,\bftheta})_\#
        [(s\bfQ \, \cdot + \bfy)_\# \mu]
        \\
        &=
        \lVert s \bfQ \, \cdot + \bfy - \bftheta \rVert_\# \mu
        \\
        &=
        \lVert s \, \cdot - \bfQ^\top (\bftheta - \bfy) \rVert_\# \mu
        \\
        &=
        (s \, \cdot )_\# \Radon_{\phi, \bfQ^\top(\bftheta - \bfy)} [\mu].
        \tag*{\qedhere}
    \end{align*}
\end{proof}

Due to the normalizations behind \hNRCDT{},
Proposition~\ref{prop:iso-aff} immediately implies
\begin{equation}
    \label{eq:iso-aff-inv}
    \NRCDT_{\phi,h}[\mu_{s\bfQ, \bfy}]
    =
    \NRCDT_{\phi,h}[h]
\end{equation}
for all $\mu \in \P_{\phi, \mathrm c}^+(\R^d)$,
$s > 0$,
$\bfQ \in \O(d)$,
and $\bfy \in \R^d$.
Therefore,
our circular \hNRCDT{} is invariant
under isotropic affine transformations
and promotes separability 
of the corresponding classes.

\begin{theorem}
    \label{thm:sep-circ-h-nrcdt}
    For template measures $\mu_0, \nu_0 \in \P_{\phi, \mathrm c}^+(\R^d)$
    whose circular \hNRCDT{}s satisfy
    \begin{equation*}
        \NRCDT_{\phi,h}[\mu_0] \neq \NRCDT_{\phi,h}[\nu_0],
    \end{equation*}
    consider the classes
    \begin{align*}
        \F 
        &=
        \bigl\{(s\bfQ \cdot + \bfy)_\# \mu_0 \mid s>0, \, \bfQ \in \O(d), \, \bfy \in \R^d\bigr\},
        \\
        \G 
        &= 
        \bigl\{(s\bfQ \cdot + \bfy)_\# \nu_0 \mid s>0, \, \bfQ \in \O(d), \, \bfy \in \R^d\bigr\}.
    \end{align*}
    Then, $\F \subset \P_{\phi,\mathrm c}^+(\R^d)$ and $\G \subset \P_{\phi,\mathrm c}^+(\R^d)$ are linearly separable 
    in circular \hNRCDT{}~space.
\end{theorem}

\begin{proof}
    Because of \eqref{eq:iso-aff-inv},
    we have
    $\NRCDT_{\phi,h} [\F] = \bigl\{\NRCDT_{\phi,h} [\mu_0]\bigr\}$
    and $\NRCDT_{\phi,h} [\G] = \bigl\{\NRCDT_{\phi,h} [\nu_0]\bigr\}$.
    Then $\NRCDT_{\phi,h} [\mu_0] \neq \NRCDT_{\phi,h} [\nu_0]$
    implies their linear separability 
    in $\Lebesgue^\infty_\rho(\R)$.
\end{proof}

\subsection{NR-CDT on the Rotation Group}
\label{sec:SO3-NRCDT}

The generalized \hNRCDT{} is not restricted
to the Euclidean setting.
As an instance,
we consider the Radon transform 
on the 3d rotation group
introduced in \cite{Quellmalz2024},
which corresponds to 
$\XX,\Theta \coloneqq \mathrm{SO}(3)$
and the slicing operator
\begin{equation*}
    \phi(\bfx, \bftheta)
    \coloneqq
    \arccos \biggl(
    \frac{\trace(\bftheta^\top \bfx) -1}{2}
    \biggr).
\end{equation*}
In the following,
we consider the rotation of measures
given by
$\mu_\bfR \coloneqq (\bfR \, \cdot)_\# \mu$
for $\mu \in \P_{\phi, \mathrm c}^+(\SO(3))$
and $\bfR \in \SO(3)$.

\begin{proposition}
    \label{prop:rot}
    For any $\bftheta \in \SO(3)$,
    the restricted Radon transform on $\SO(3)$ satisfies
    \begin{equation*}
        \Radon_{\phi,\bftheta}[\mu_{\bfR}]
        = 
        \Radon_{\phi, \bfR^\top \bftheta} [\mu].
    \end{equation*}
\end{proposition}

\begin{proof}
    Plugging in the definition of the generalized restricted Radon transform,
    we obtain
    \begin{align*}
        \Radon_{\phi,\bftheta}[\mu_{\bfR}]
        &=
        (S_{\phi,\bftheta})_\#
        [(\bfR \, \cdot )_\# \mu]
        \\
        &=
        \Bigl( \arccos \Bigl( \tfrac{\trace(\bftheta^\top \bfR \, \cdot) - 1}{2} \Bigr) \Bigr)_\# \mu
        \\
        &=
        \Radon_{\phi, \bfR^\top \bftheta}[\mu].
        \tag*{\qedhere}
    \end{align*}
\end{proof}

Since $h$ has to be invariant under the rotation of the direction set $\Theta = \SO(3)$,
the \hNRCDT{} inherit this property,
promoting the separability under rotations.

\begin{theorem}
    \label{thm:sep-so3-h-nrcdt}
    For $\mu_0, \nu_0 \in \P_{\phi, \mathrm c}^+(\SO(3))$
    with
    \begin{equation*}
        \genNRCDT{} [\mu_0] \neq \genNRCDT{} [\nu_0],
    \end{equation*}
    consider the classes
    \begin{align*}
        \F 
        &=
        \bigl\{(\bfR \cdot )_\# \mu_0 \mid \bfR \in \SO(3) \bigr\},
        \\
        \G 
        &= 
        \bigl\{(\bfR \cdot )_\# \nu_0 \mid \bfR \in \SO(3)\bigr\}.
    \end{align*}
    Then, $\F \subset \P_{\phi,\mathrm c}^+(\R^d)$ and $\G \subset \P_{\phi,\mathrm c}^+(\R^d)$ are linearly separable 
    in \hNRCDT{}~space.
\end{theorem}

The proof is literally the same as of Theorem~\ref{thm:sep-circ-h-nrcdt}.
Notice that the Radon transform
on $\SO(3)$
of a rotated measure
causes only a reparametrization of the direction set;
therefore,
the first normalization step of the \hNRCDT{} is not necessary
in the strict sense. 

\subsection{Further NR-CDT Variants}

In addition to the above examples,
the \hNRCDT{} can be adapted
to further generalized Radon transforms.
E.g.~for studying measures on spheres,
the vertically sliced transform \cite{Quellmalz2023}
and parallelly sliced transform \cite{Quellmalz2024}
may be of interest,
where our normalization procedure 
yields invariant feature extractor
under rotations.
For measures on hyperbolic spaces,
the hyperbolic Radon transform \cite{Casadio2021} may be used. 
Finally,
for measures on arbitrary manifolds,
the eigenfunctions of the Laplace--Beltrami operator 
may be employed to define generalized Radon transforms
regarding our framework
in the style of the intrinsic sliced Wasserstein distance \cite{Rustamov2023}.
Actual studies 
about the usefulness of the generalized \hNRCDT{}
to obtain invariant feature extractors
under certain transformations
are left for future research.

\section{Numerical Experiments}
\label{sec:num-ex}

We now present numerical experiments
to support our linear separability results
in Theorems~\ref{thm:sep-max-nrcdt}--\ref{thm:sep-so3-h-nrcdt}.
Keeping the application in computational filigranology in mind,
we start with the two-dimensional setting and
focus on image classification and clustering
to show the effectiveness of our new feature representations.
Thereon, we provide proof-of-concept experiments
in the multi-dimensional setting, where we focus on
classification in $\R^3$ and $\SO(3)$.
We compare our approach with
the original R-CDT~\cite{Kolouri2016} and
the Euclidean representation as baseline.
Since R-CDT already outperforms 
other state-of-the-art approaches
in the small data regime \cite{Shifat-E-Rabbi2021},
we omit comparisons with neural networks.
All methods are implemented in Julia
and the code is publicly available at
\url{https://github.com/DrBeckmann/NR-CDT}.
Our experiments are performed on a off-the-shelf
MacBookPro 2020 with 1.4~GHz quad‑core Intel Core i5 CPU
and 8~GB RAM.

Before reporting our numerical results,
we wish to comment on some details
regarding the numerical implementation.
In the two-dimensional setting,
each grayscale image is modeled as a non-negative
piecewise constant function $f \in \Lebesgue^1(\R^2)$ with support in
$[-1/\sqrt{2},1/\sqrt{2}]^2 \subset \Ball$.
Moreover,
we normalize $f$
such that $\int_{\Ball} f(\bfx) \d \bfx = 1$
to be interpretable as a probability density function.
This allows us to compute
its Radon transform in closed form.
For numerical stability,
we replace the line integrals in~\eqref{eq:radon-meas}
by integrals over disjoint stripes of width $\frac{2}{R}$,
where $R$ is the number of radial samples.
Thereon,
the CDTs are computed based on
linear interpolation on an equispaced grid 
of interpolation points in $(0,1)$.
In the multi-dimensional setting,
we model the objects as empirical measures on $\XX$.
Then,
for a discretized direction set $\Theta$,
the restricted Radon transforms in~\eqref{eq:restrict-gen-radon-meas}
are themselves empirical measures that
can be explicitly computed by evaluating~\eqref{eq:gen-slicing-operator}.
Consequently, the CDTs are piecewise constant functions
and can be calculated exactly.

\begin{algorithm}[t]
    \caption{Generalized NR-CDT}
    \label{alg:gen-nrcdt}
    \begin{algorithmic}[1]
    \Require 
        input probability measure $\mu \in \P_{\phi, \mathrm c}^+(\XX)$,
        \\
        \Comment{image / empirical measure}
        \\
        Radon directions $\bftheta_1, \dots, \bftheta_n \in \Theta$,
        \\
        number of sampling points $m \in \N$,
        \\
        aggregation map $h \colon \R^m \to \R$
    \Ensure $\mathbf a = \NRCDT_{\phi,h}[\mu]$
    \For{$k = 1, \dots, n$}
        \State $\nu \gets \Radon_{\phi, \bftheta_k}[\mu]$
        \\
        \Comment{sampled density / empirical measure}
        \\[-7pt]
        \State $\bfq \gets 
            \bigl[ F^{[-1]}_\nu \bigl(\tfrac{1}{m+1} \bigr), 
                \dots,
                F^{[-1]}_\nu\bigl(\tfrac{m}{m+1}\bigr) \bigr]$
        \\[2pt]
        \Comment{linear interpolation / explicit inversion}
        \\[-7pt]
        \State $\eta \gets \tfrac{1}{m} \sum_{\ell = 1}^m  \bfq_\ell$
        \\[-5pt]
        \State $\sigma \gets \sqrt{\tfrac{1}{m} \sum_{\ell = 1}^m (\bfq_\ell - \eta)^2 }$
        \\[-7pt]
        \State $\mathbf N_{\bullet, k} \gets (\bfq \ominus \eta) \oslash \sigma$
        \\
        \Comment{pointwise subtraction \& division}
    \EndFor
    \For{$\ell = 1, \dots, m $}
        \State $\mathbf a_\ell = h(\mathbf N_{\ell,\bullet})$
    \EndFor
\end{algorithmic}
\end{algorithm}

The pseudocode
for the generalized NR-CDT
is given in Algorithm~\ref{alg:gen-nrcdt},
where the notation $\bullet,k$ is used 
to indicate the $k$th column
and $\ell, \bullet$
to indicate the $\ell$th row
of a matrix.
The computational complexity is dominated by
the employed (generalized) Radon transform.
For a two-dimensional image with in total $N$ pixels,
a na\"{\i}ve implementation
of the restricted Radon transform $\Radon_\bftheta$
essentially requires $\mathcal O(M N)$ operations,
where $M$ denotes the number of Radon radii.
As opposed to this,
for a $d$-dimensional point cloud with $N$ elements,
the generalized restricted Radon transform $\Radon_{\phi,\bftheta}$
can be implemented with $\mathcal O(d N)$ operations.
Since $\Radon_{\phi,\bftheta}$ yields a one-dimensional
point cloud with $N$ elements,
we choose the corresponding nodes
as the $M = N$ Radon radii.

For the uniform reference measure $\rho = u_{[0,1]}$,
the computation of the CDT basically consists of
calculating the cumulative sums of the $M$ Radon samples,
which can be combined with the interpolation step.
Hence, for $m$ CDT sampling points,
we can calculate the CDT with $\mathcal O(M + m)$ operations.
Computation of mean and standard deviation
as well as
the normalization step 
require $\mathcal O(m)$ operations.
Finally, if $n$ is the number of Radon directions,
all our considered aggregation functions
have a complexity of $\mathcal O(n)$.
Hence, the total computational complexity
of the generalized NR-CDT
is of order $\mathcal O(n(MN + m))$
for two-dimensional images
and of order $\mathcal O(n(dN + m))$
for $d$-dimensional point clouds.
Especially for images, 
where $N$ is usually large,
the complexity can be lowered by
exploiting the Fourier slice theorem
and the non-equispaced fast Fourier transform.

\subsection{Image Classification}
\label{sec:image_classification}

\begin{figure*}[t]
    \centering%
    \footnotesize%
    \begin{tabular}{c @{\hspace{3pt}} c @{\hspace{3pt}} c @{\hspace{3pt}} c @{\hspace{3pt}} c @{\hspace{3pt}} c @{\hspace{3pt}} c @{\hspace{3pt}} c @{\hspace{3pt}} c}
        class~1 
        & class~2
        & class~3
        & class~4
        & class~5
        & class~6
        & class~7
        & class~8
        & class~9 \\%
        \includegraphics[width=0.1\linewidth, clip=true, trim=132pt 27pt 113pt 12pt]{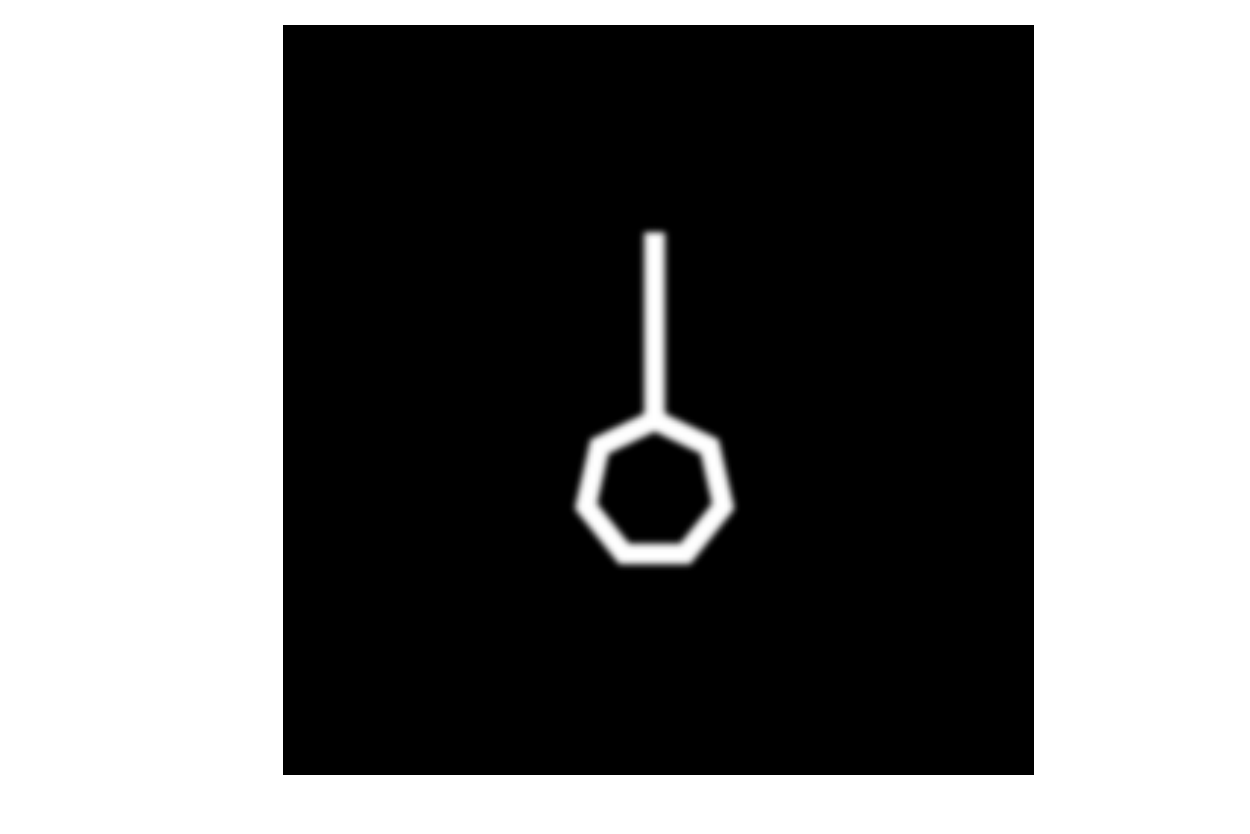}%
        & \includegraphics[width=0.1\linewidth, clip=true, trim=132pt 27pt 113pt 12pt]{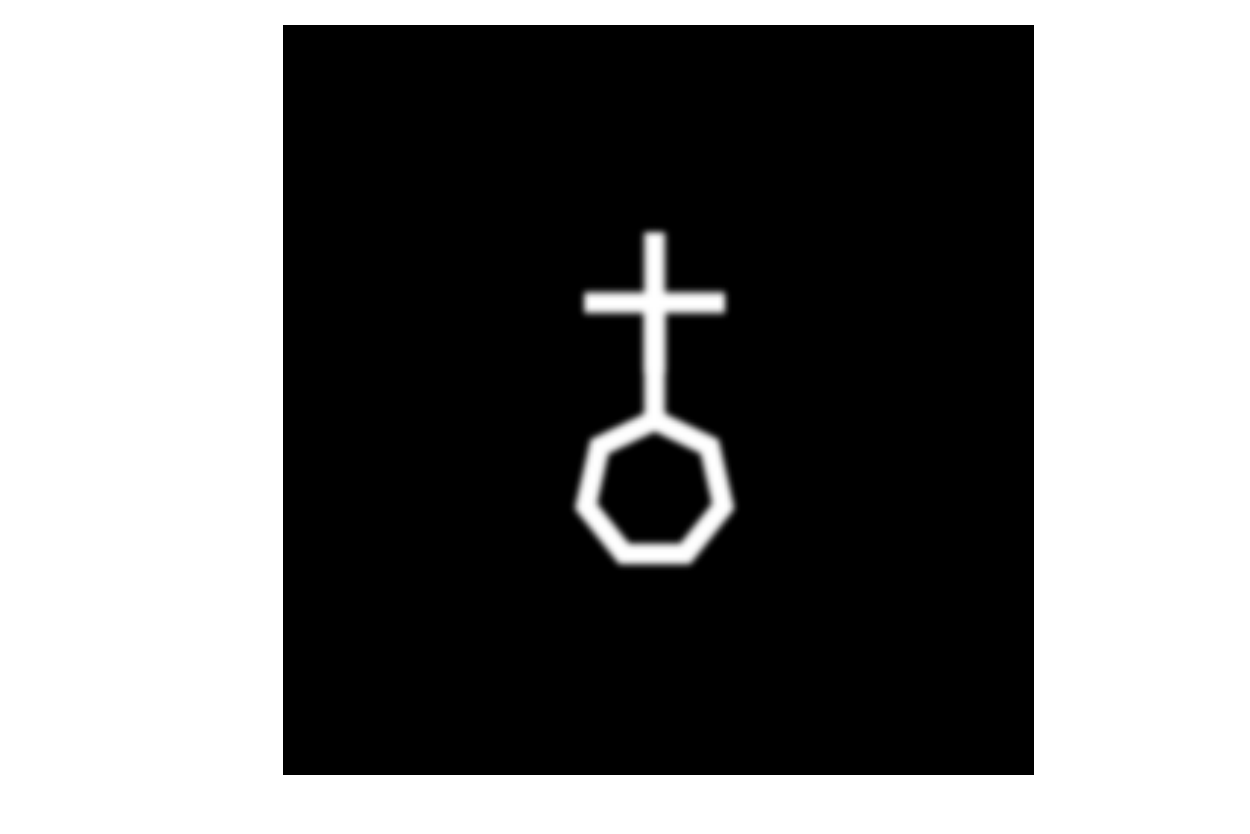}%
        & \includegraphics[width=0.1\linewidth, clip=true, trim=132pt 27pt 113pt 12pt]{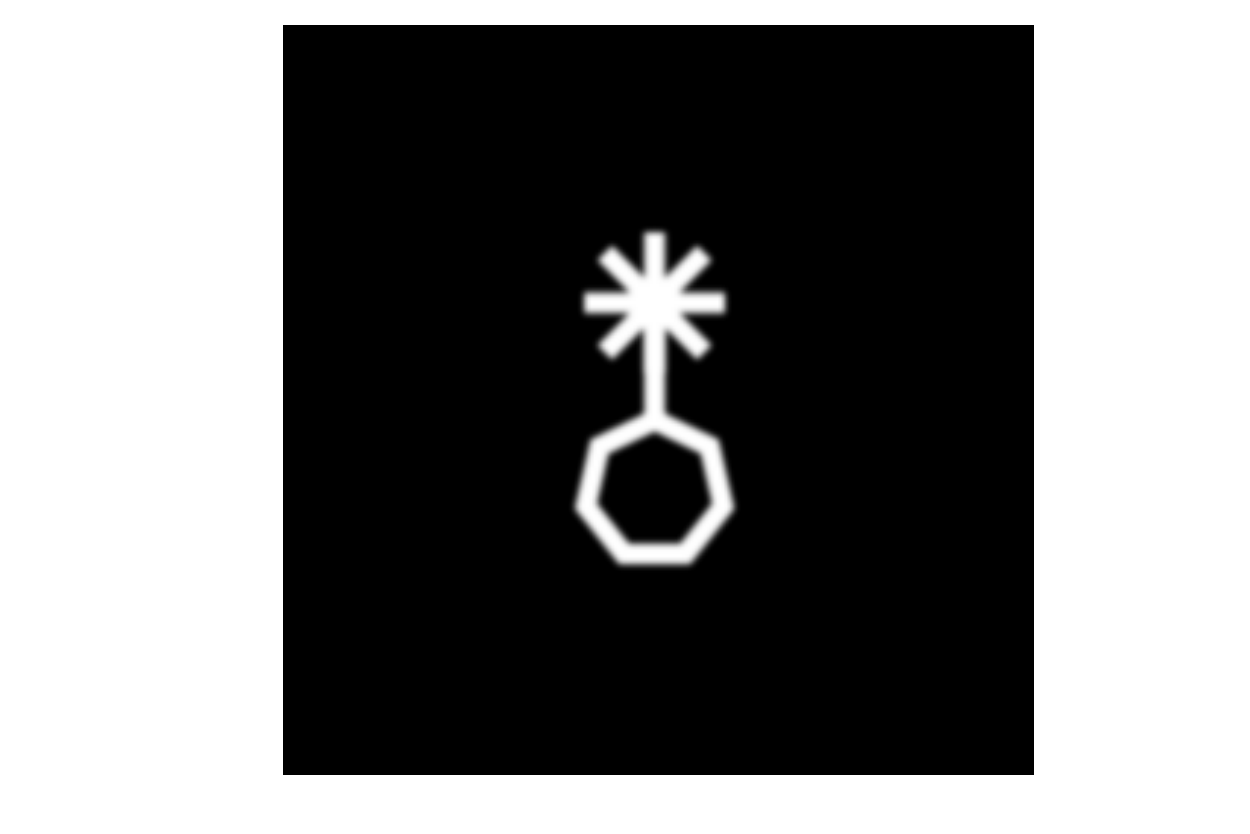}%
        & \includegraphics[width=0.1\linewidth, clip=true, trim=132pt 27pt 113pt 12pt]{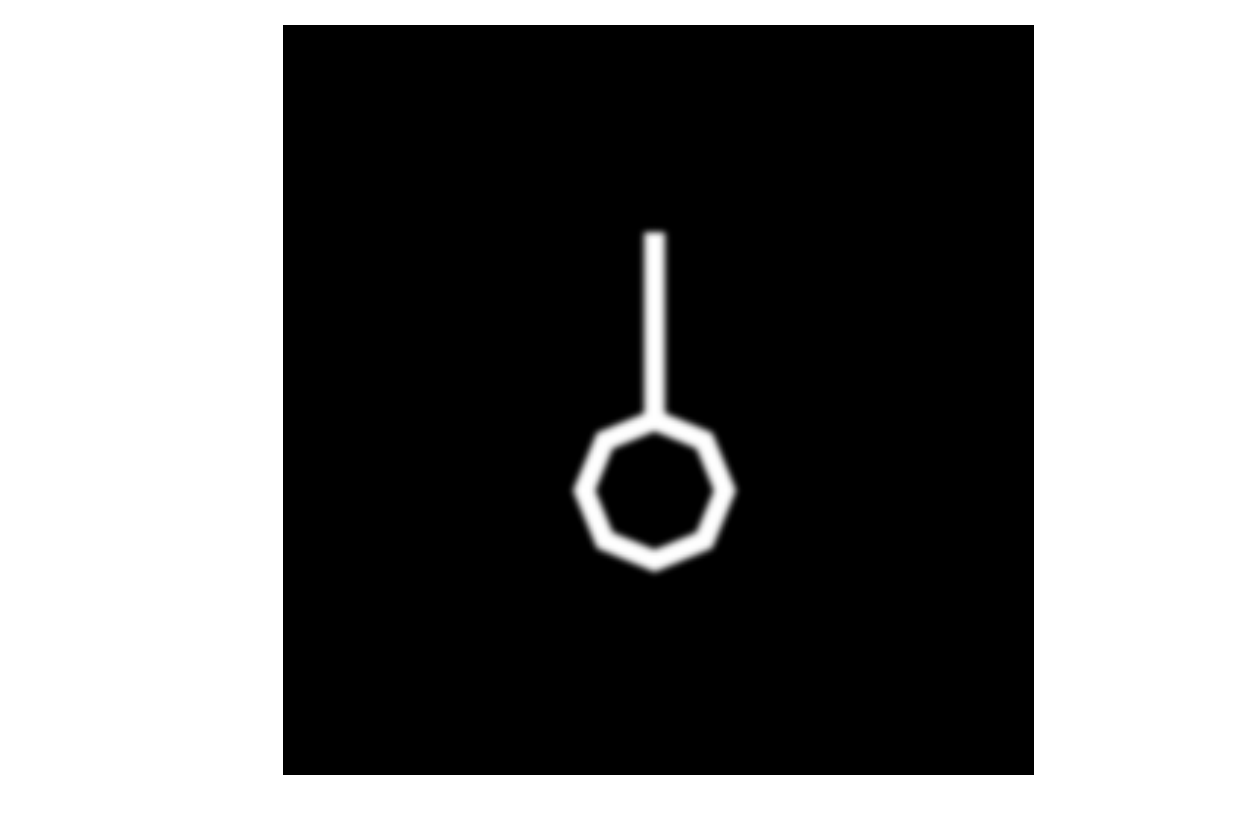}%
        & \includegraphics[width=0.1\linewidth, clip=true, trim=132pt 27pt 113pt 12pt]{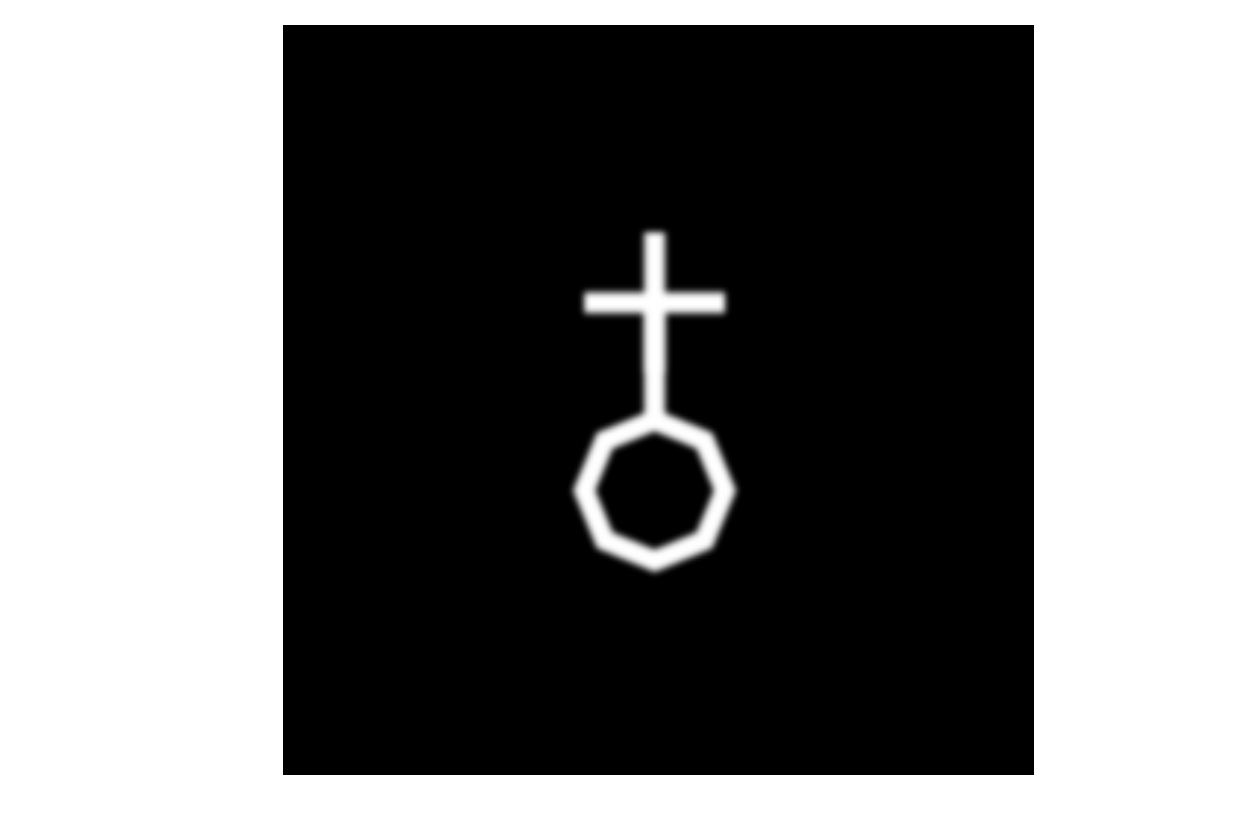}%
        & \includegraphics[width=0.1\linewidth, clip=true, trim=132pt 27pt 113pt 12pt]{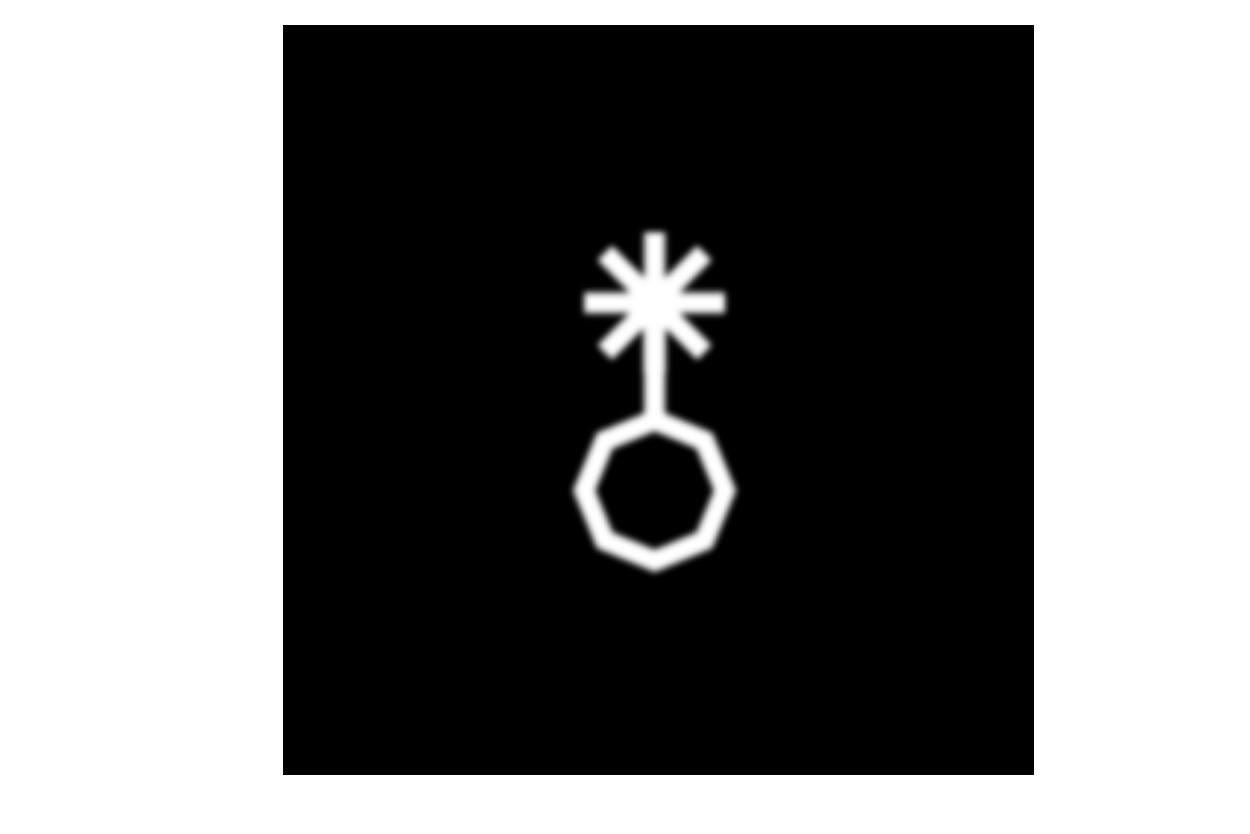}%
        & \includegraphics[width=0.1\linewidth, clip=true, trim=132pt 27pt 113pt 12pt]{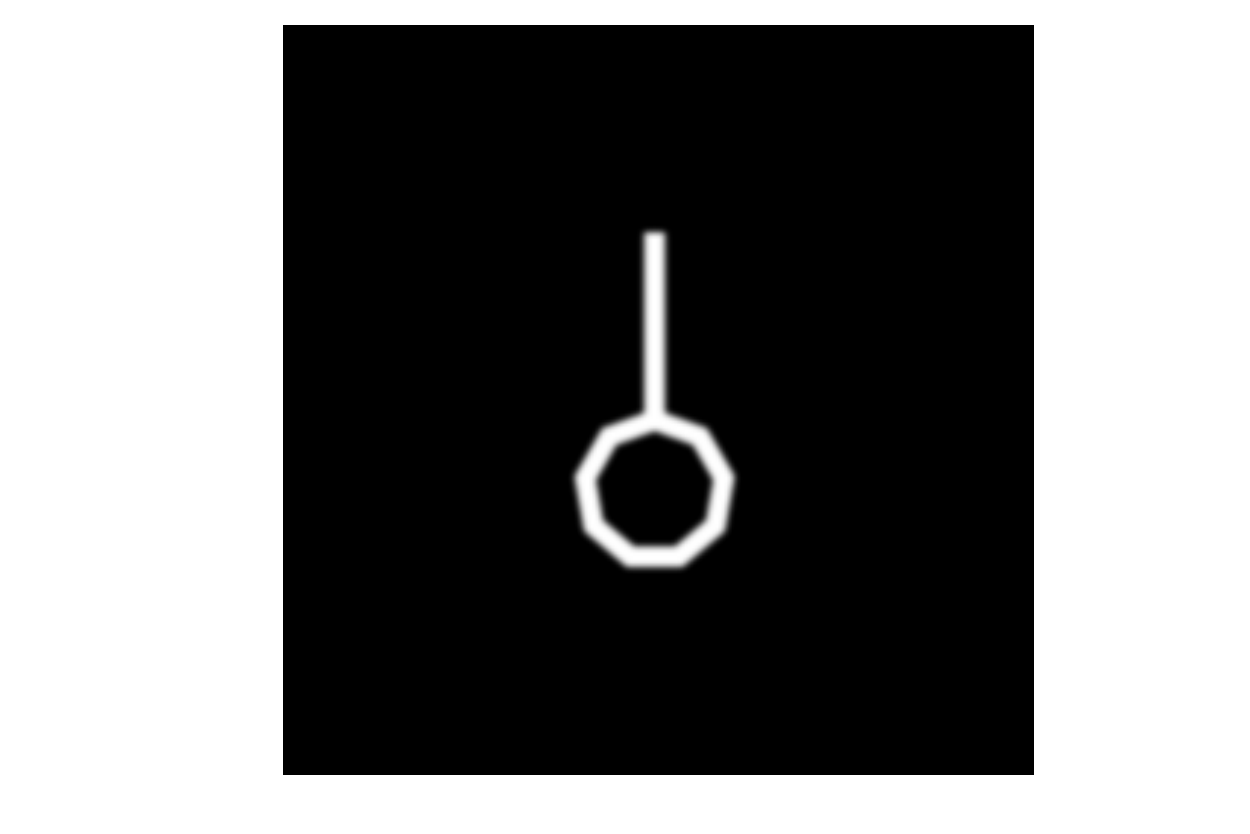}%
        & \includegraphics[width=0.1\linewidth, clip=true, trim=132pt 27pt 113pt 12pt]{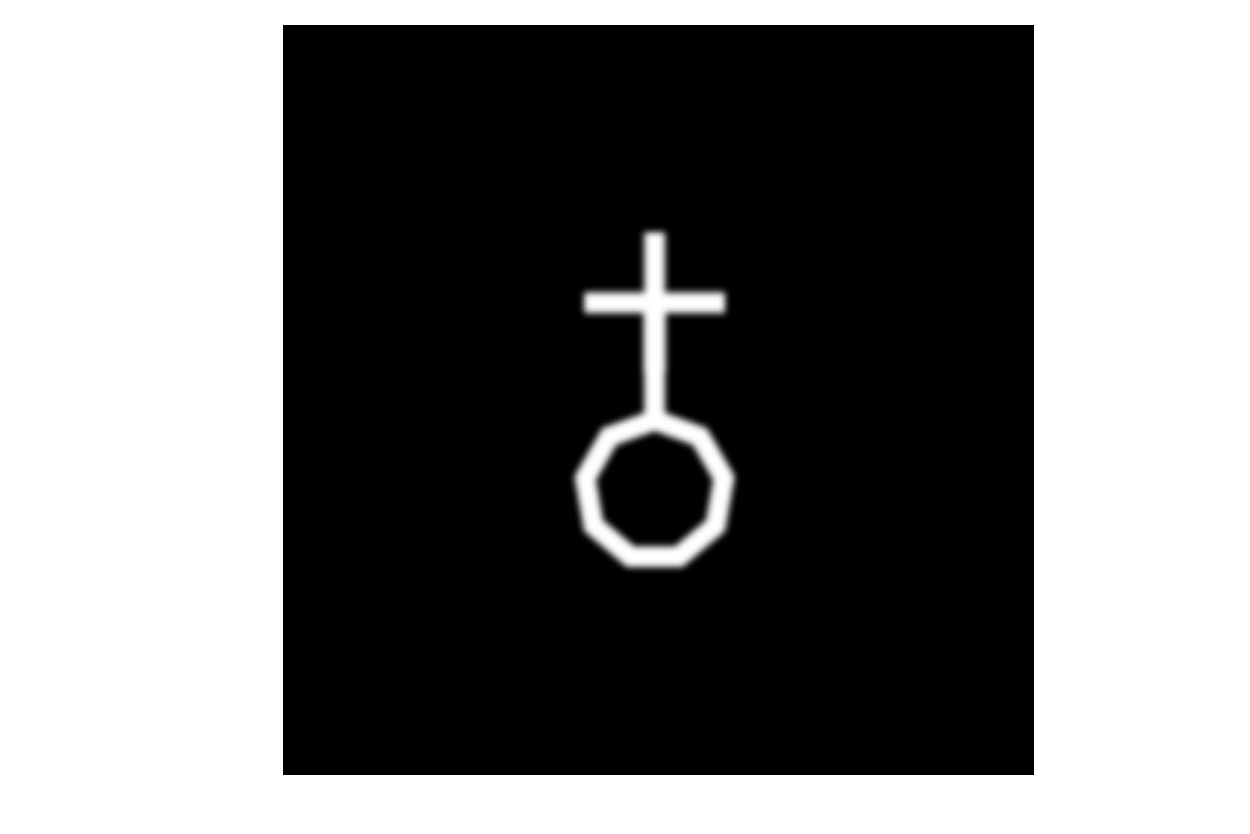}%
        & \includegraphics[width=0.1\linewidth, clip=true, trim=132pt 27pt 113pt 12pt]{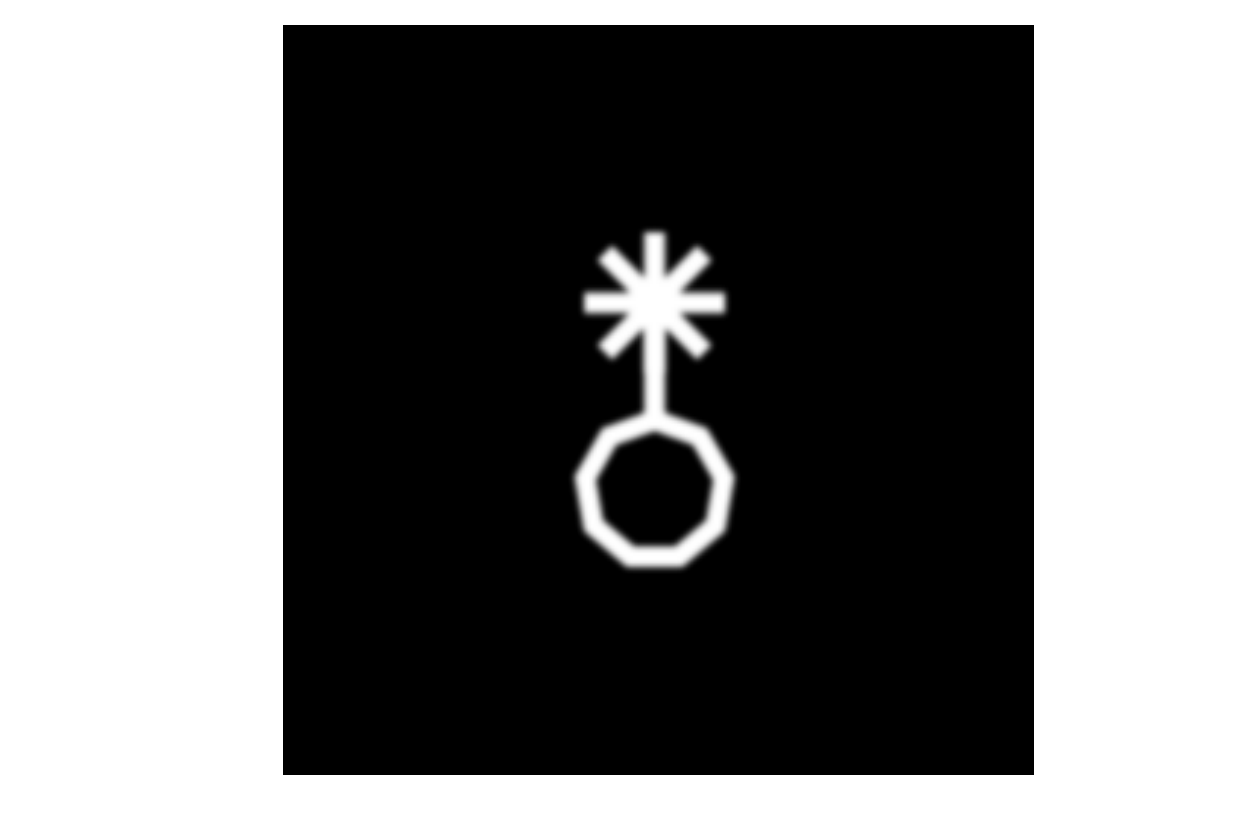}%
        \\%
        \includegraphics[width=0.05\linewidth, clip=true, trim=132pt 27pt 113pt 12pt]{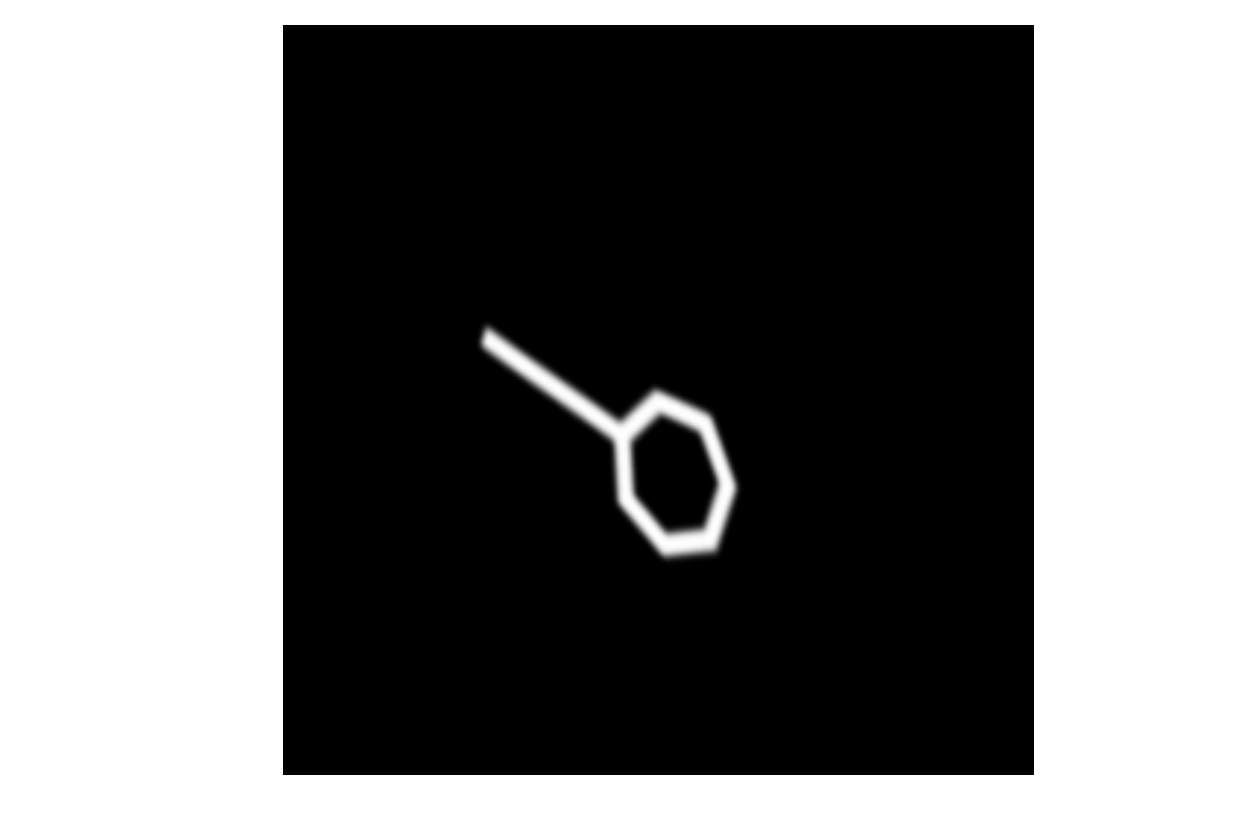}%
        \hspace{1pt}%
        \includegraphics[width=0.05\linewidth, clip=true, trim=132pt 27pt 113pt 12pt]{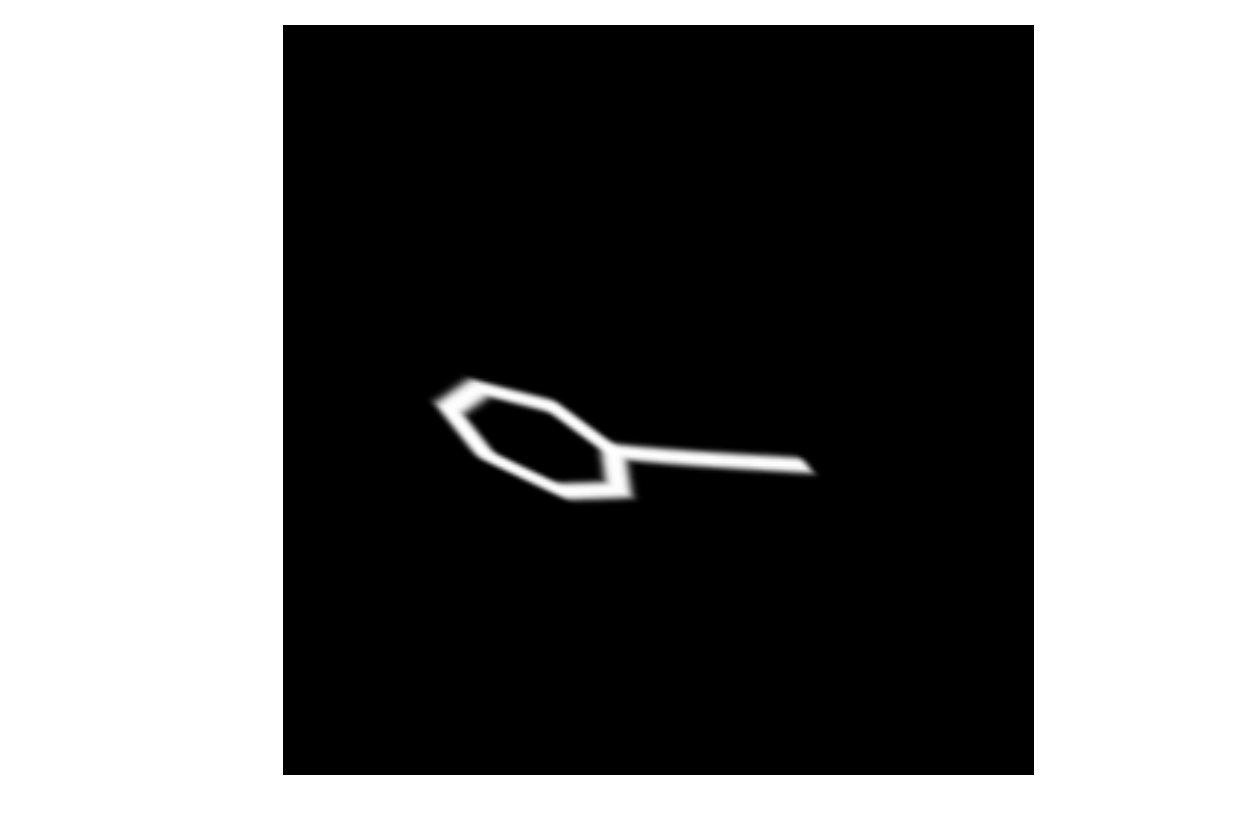}%
        & \includegraphics[width=0.05\linewidth, clip=true, trim=132pt 27pt 113pt 12pt]{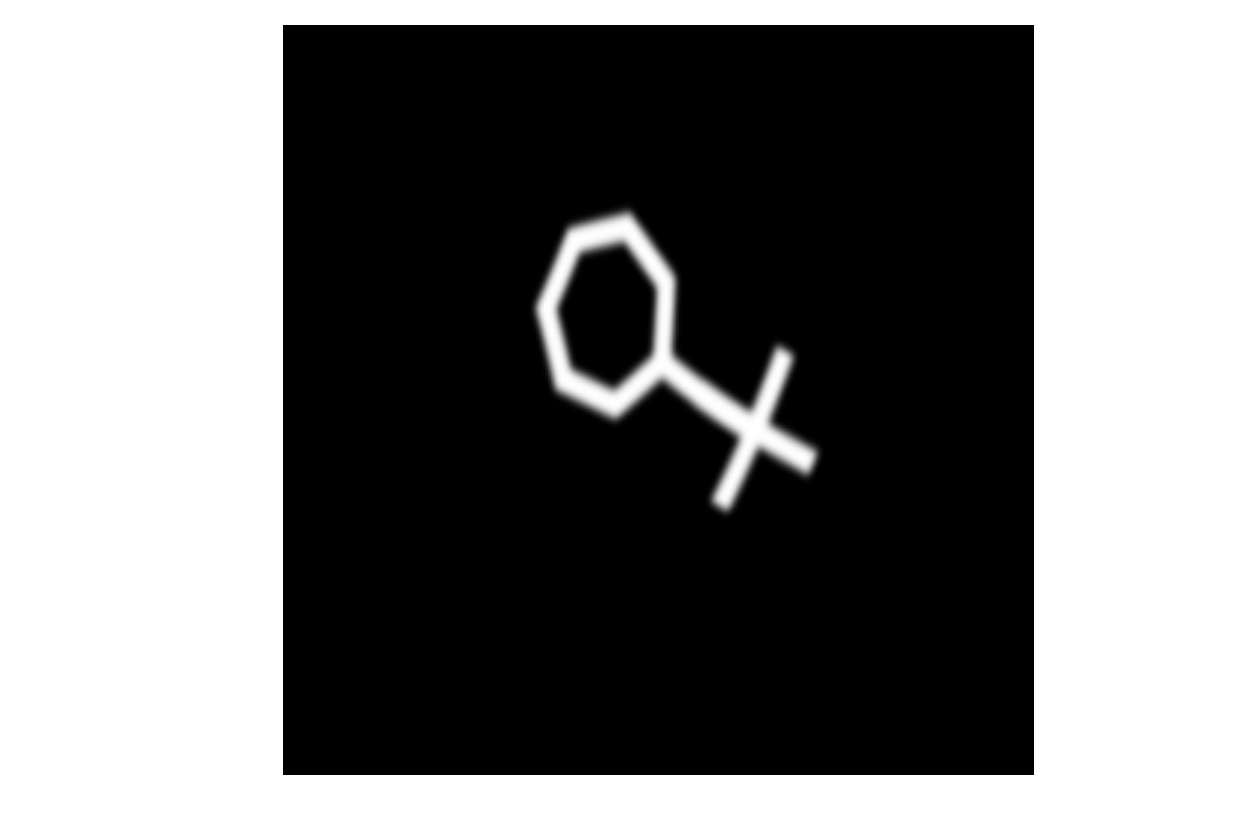}%
        \hspace{1pt}%
        \includegraphics[width=0.05\linewidth, clip=true, trim=132pt 27pt 113pt 12pt]{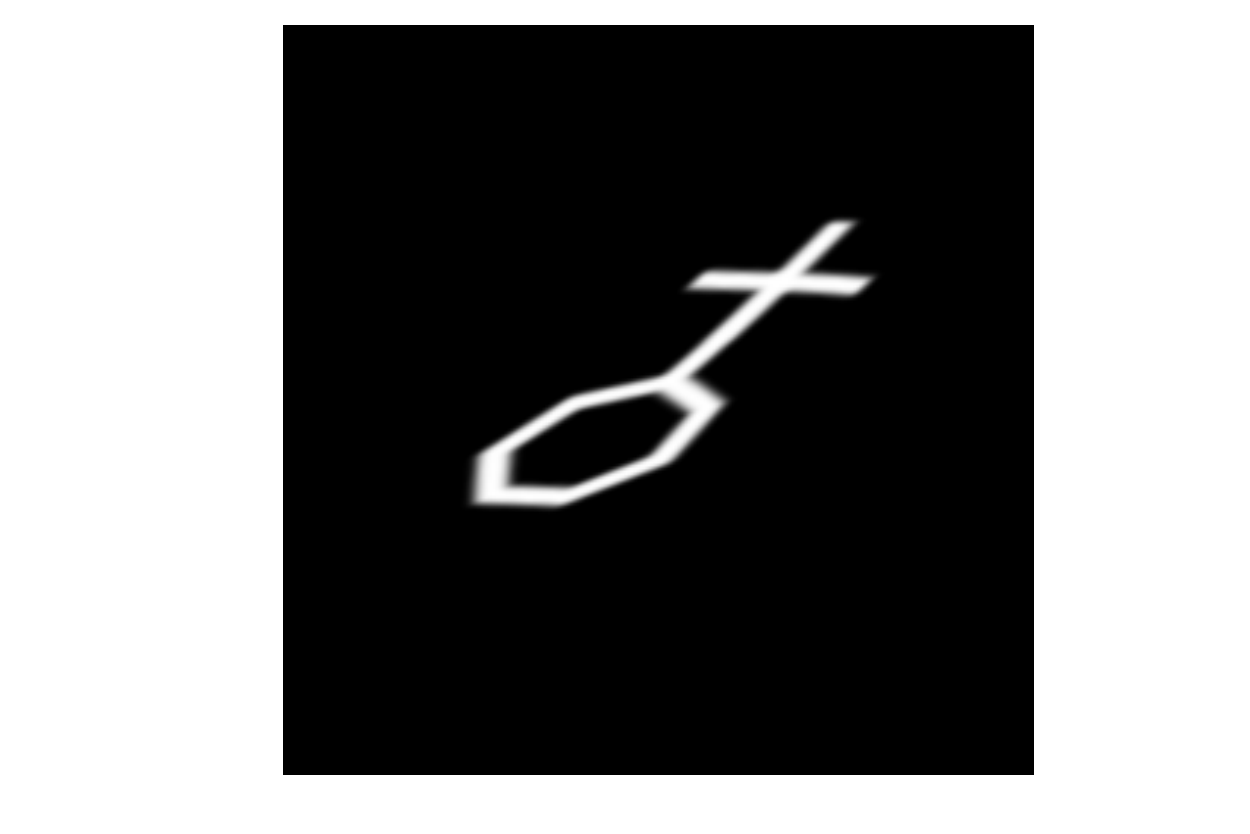}%
        & \includegraphics[width=0.05\linewidth, clip=true, trim=132pt 27pt 113pt 12pt]{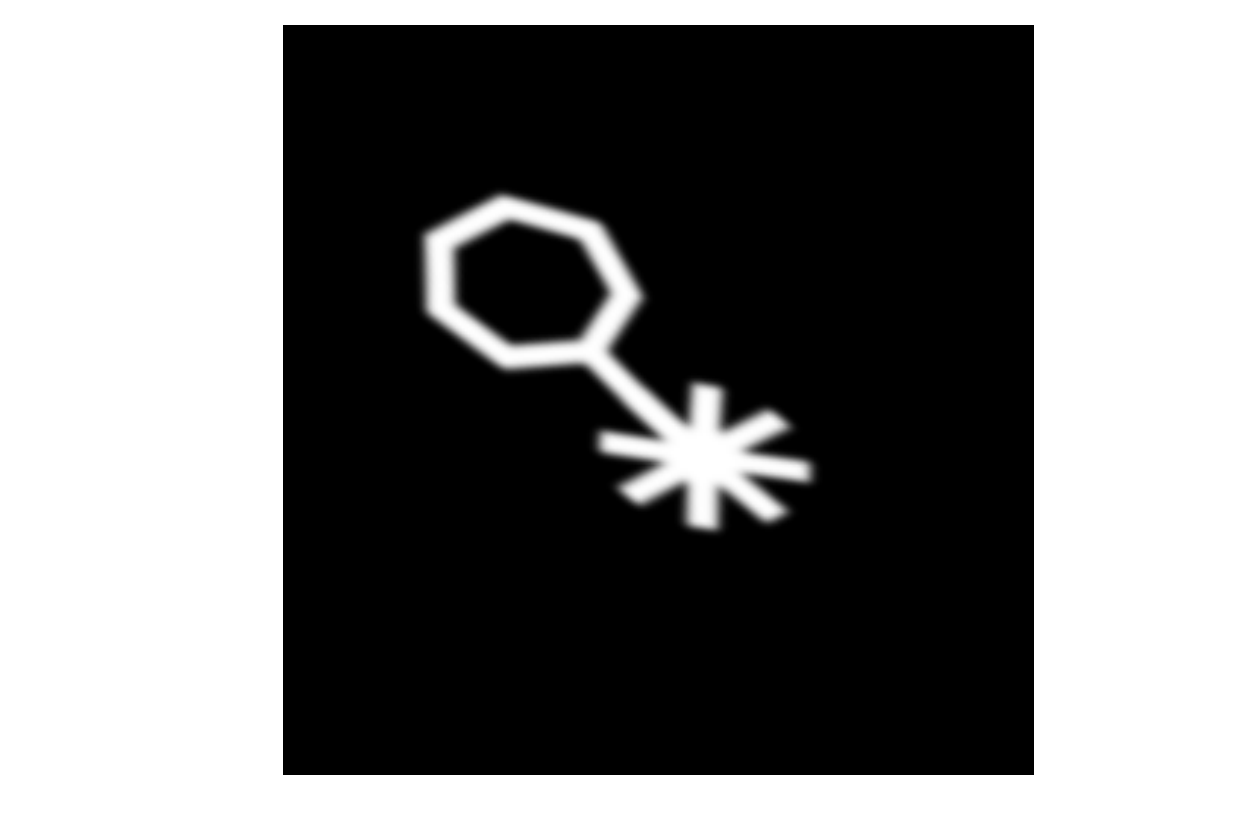}%
        \hspace{1pt}%
        \includegraphics[width=0.05\linewidth, clip=true, trim=132pt 27pt 113pt 12pt]{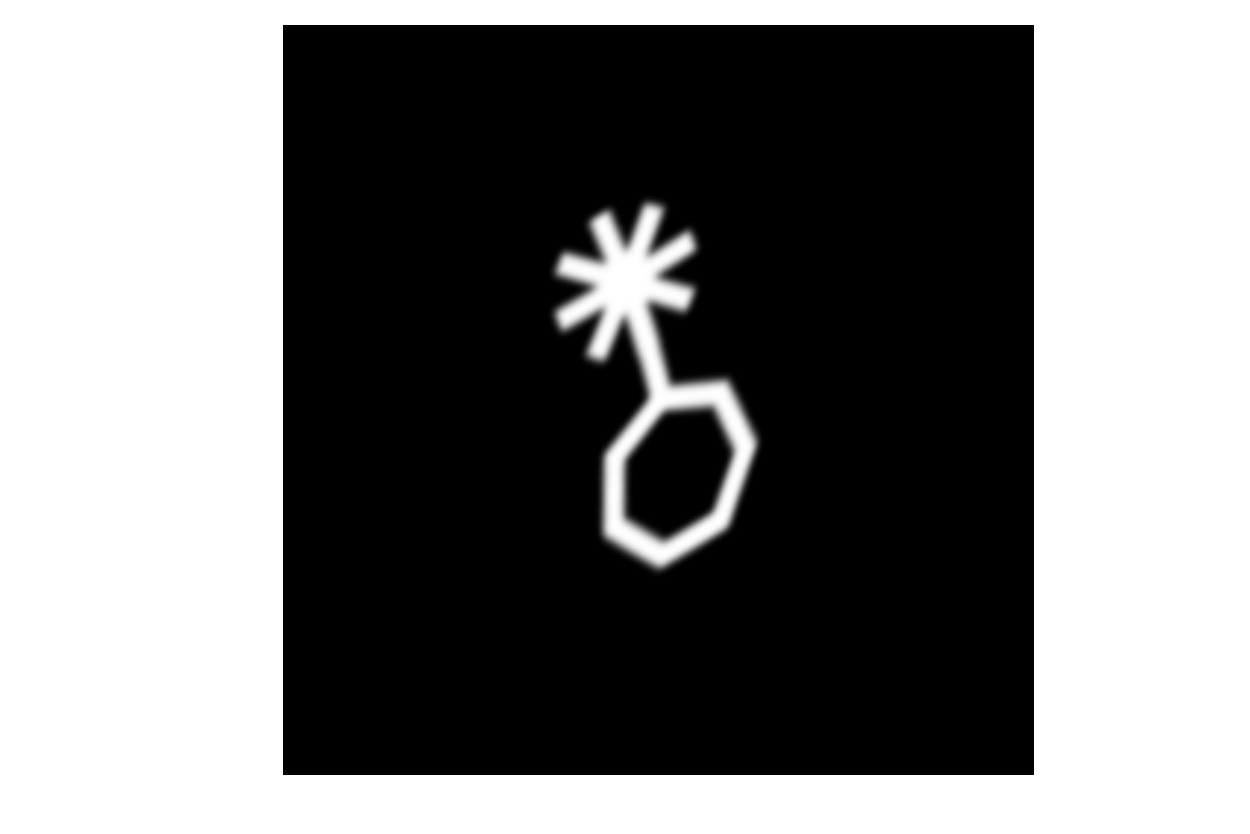}%
        & \includegraphics[width=0.05\linewidth, clip=true, trim=132pt 27pt 113pt 12pt]{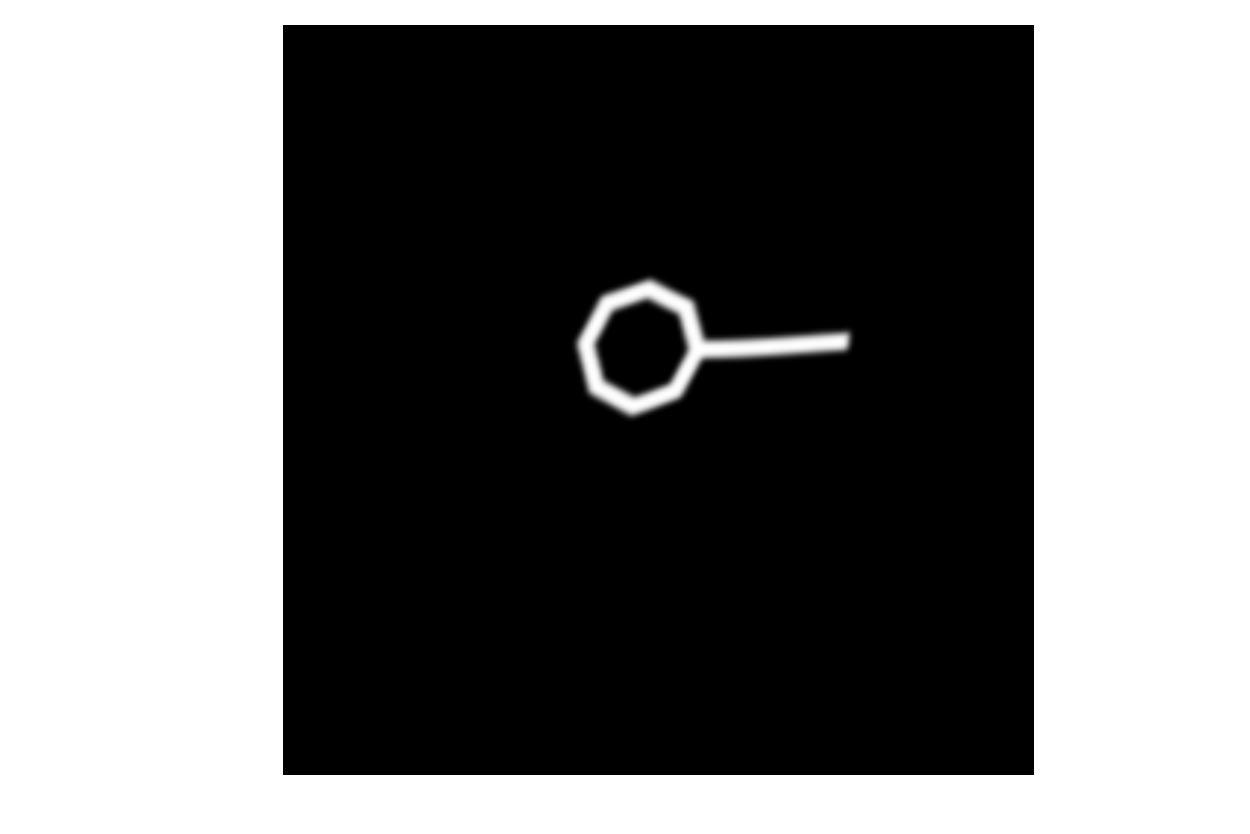}%
        \hspace{1pt}%
        \includegraphics[width=0.05\linewidth, clip=true, trim=132pt 27pt 113pt 12pt]{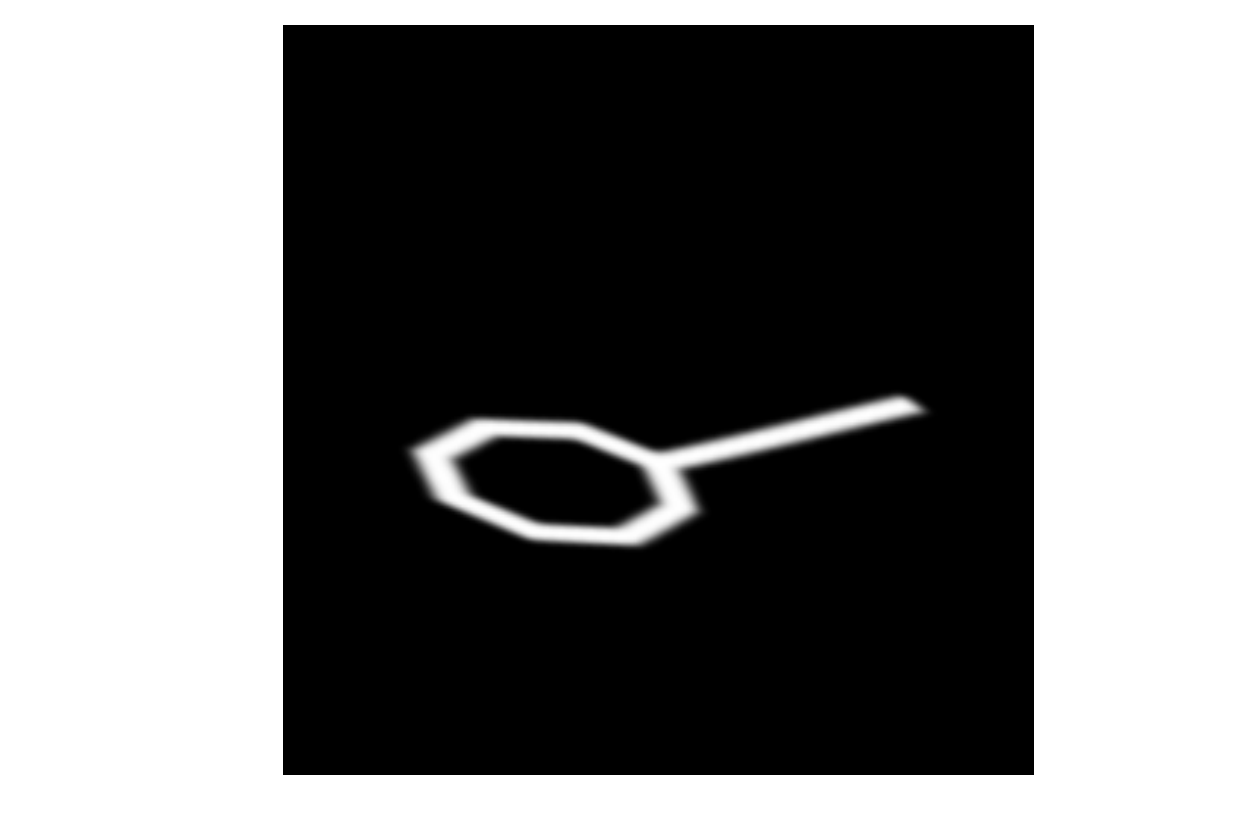}%
        & \includegraphics[width=0.05\linewidth, clip=true, trim=132pt 27pt 113pt 12pt]{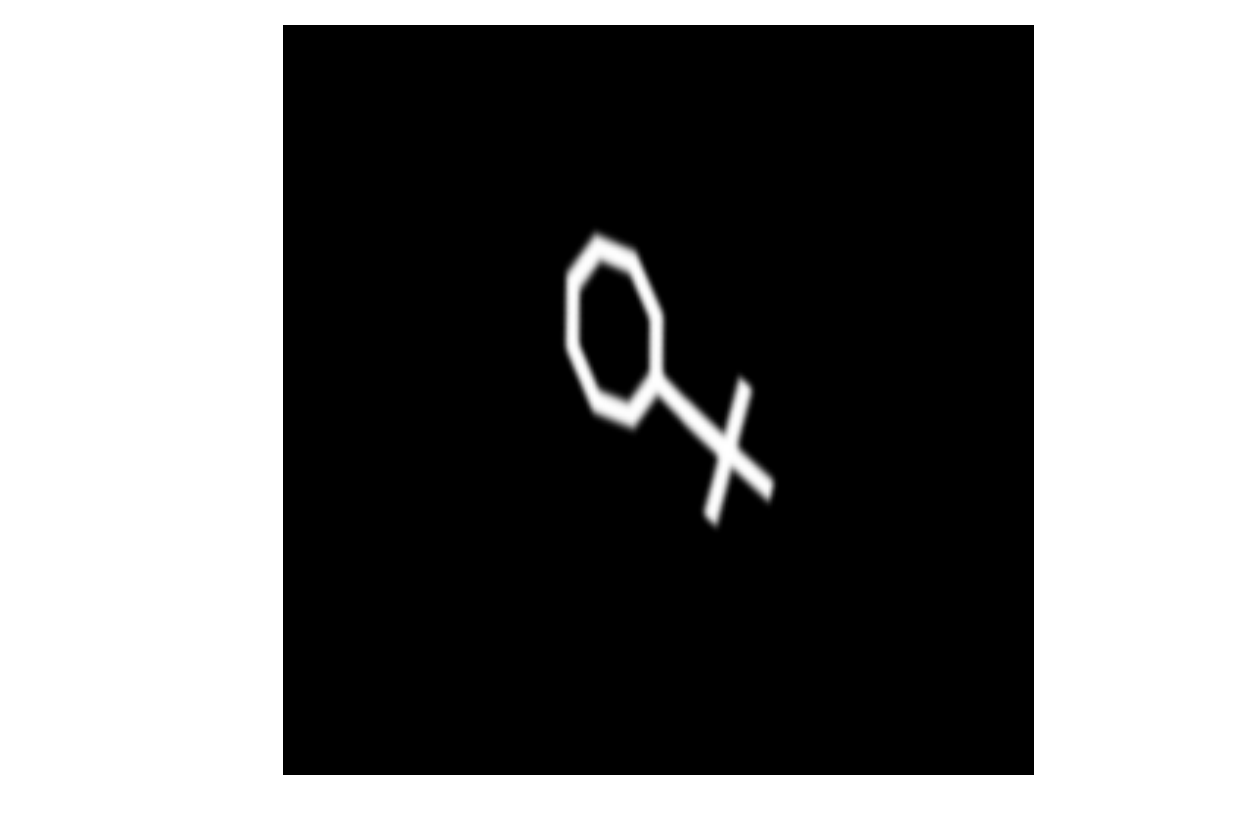}%
        \hspace{1pt}%
        \includegraphics[width=0.05\linewidth, clip=true, trim=132pt 27pt 113pt 12pt]{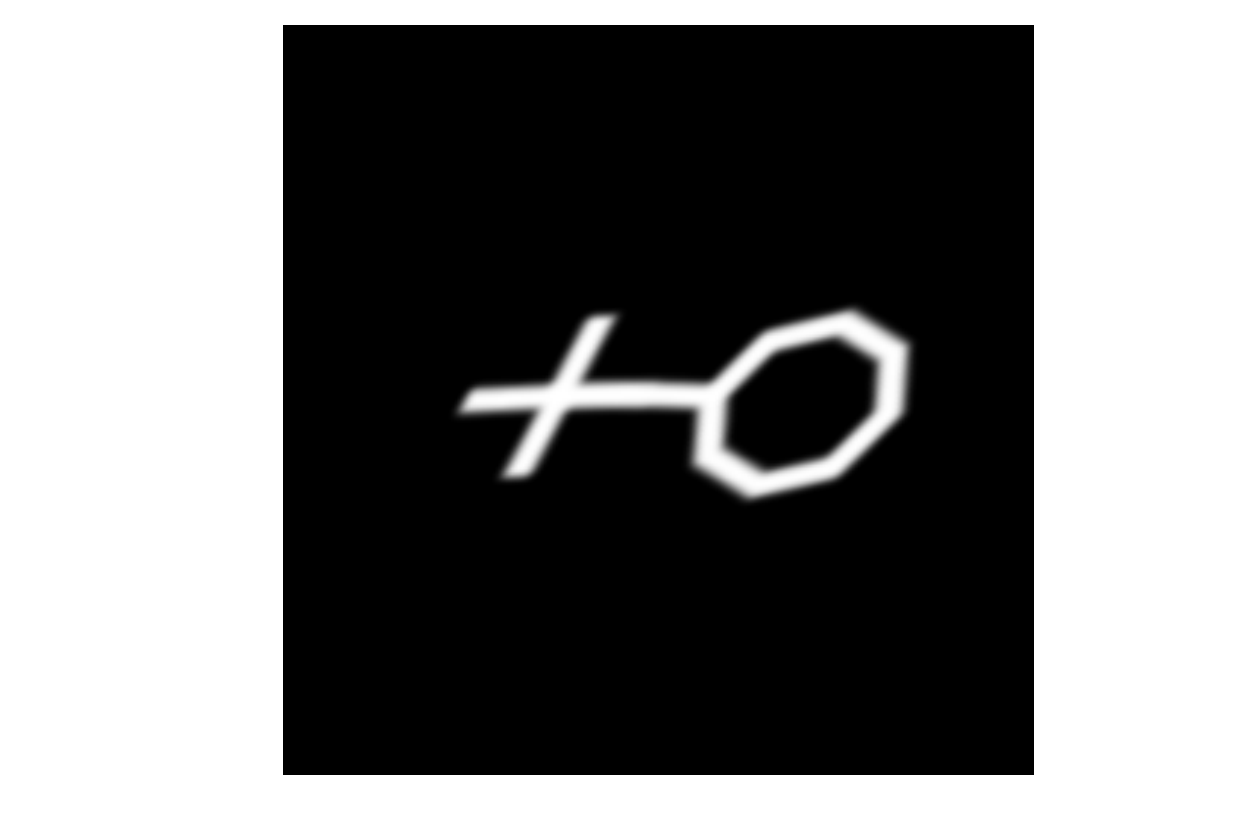}%
        & \includegraphics[width=0.05\linewidth, clip=true, trim=132pt 27pt 113pt 12pt]{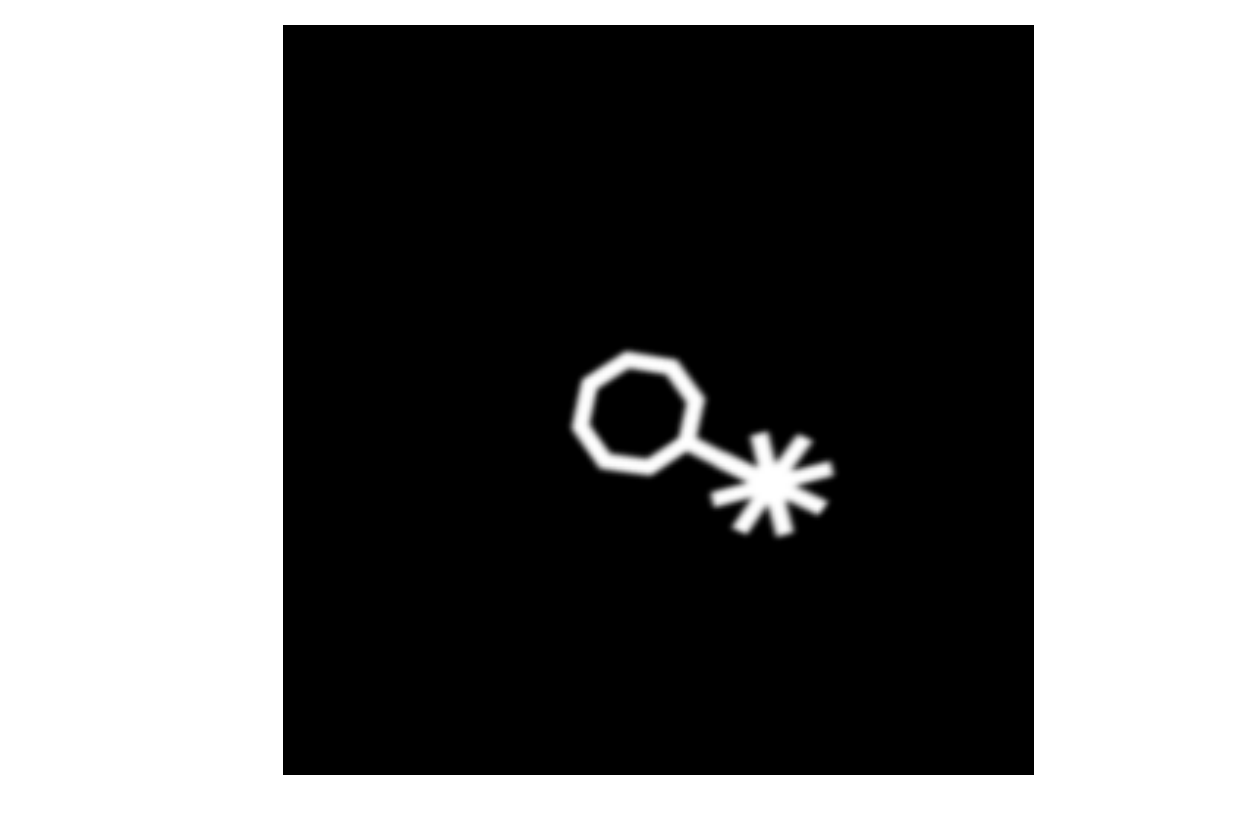}%
        \hspace{1pt}%
        \includegraphics[width=0.05\linewidth, clip=true, trim=132pt 27pt 113pt 12pt]{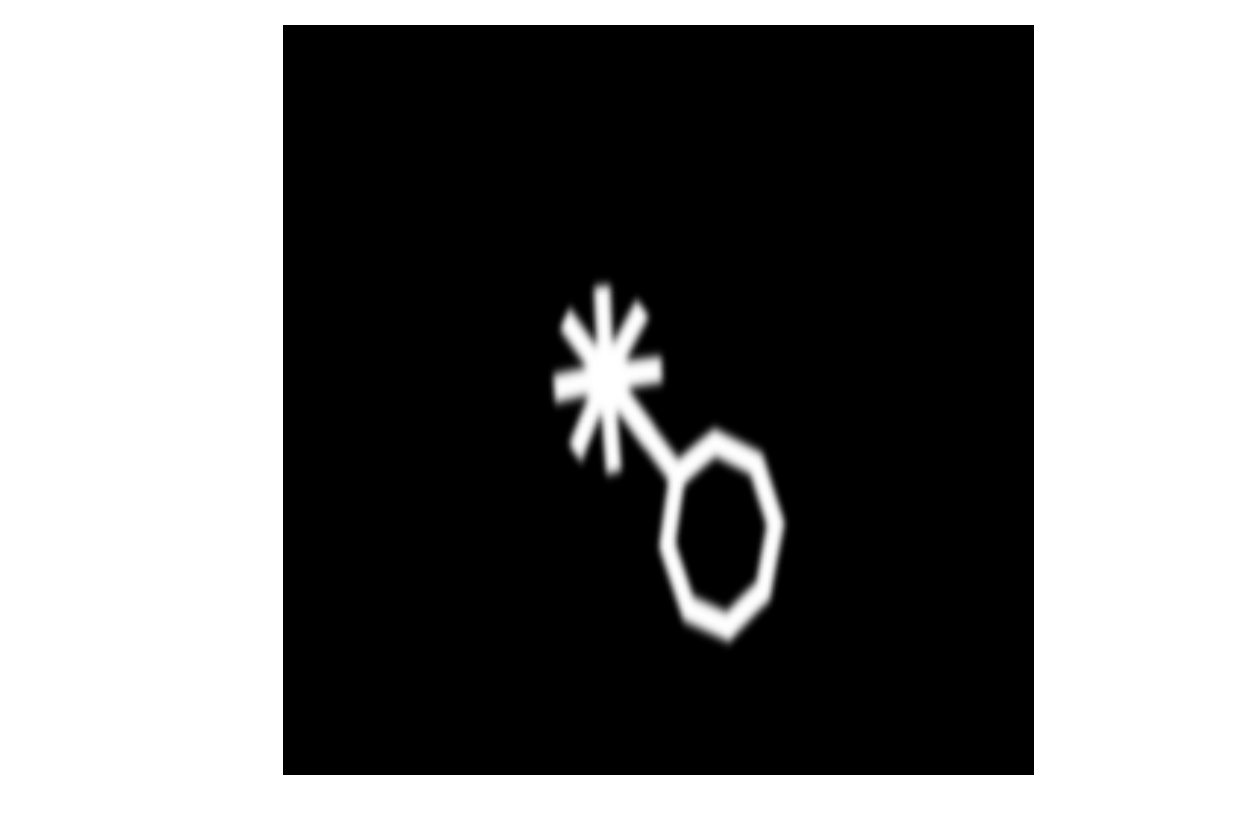}%
        & \includegraphics[width=0.05\linewidth, clip=true, trim=132pt 27pt 113pt 12pt]{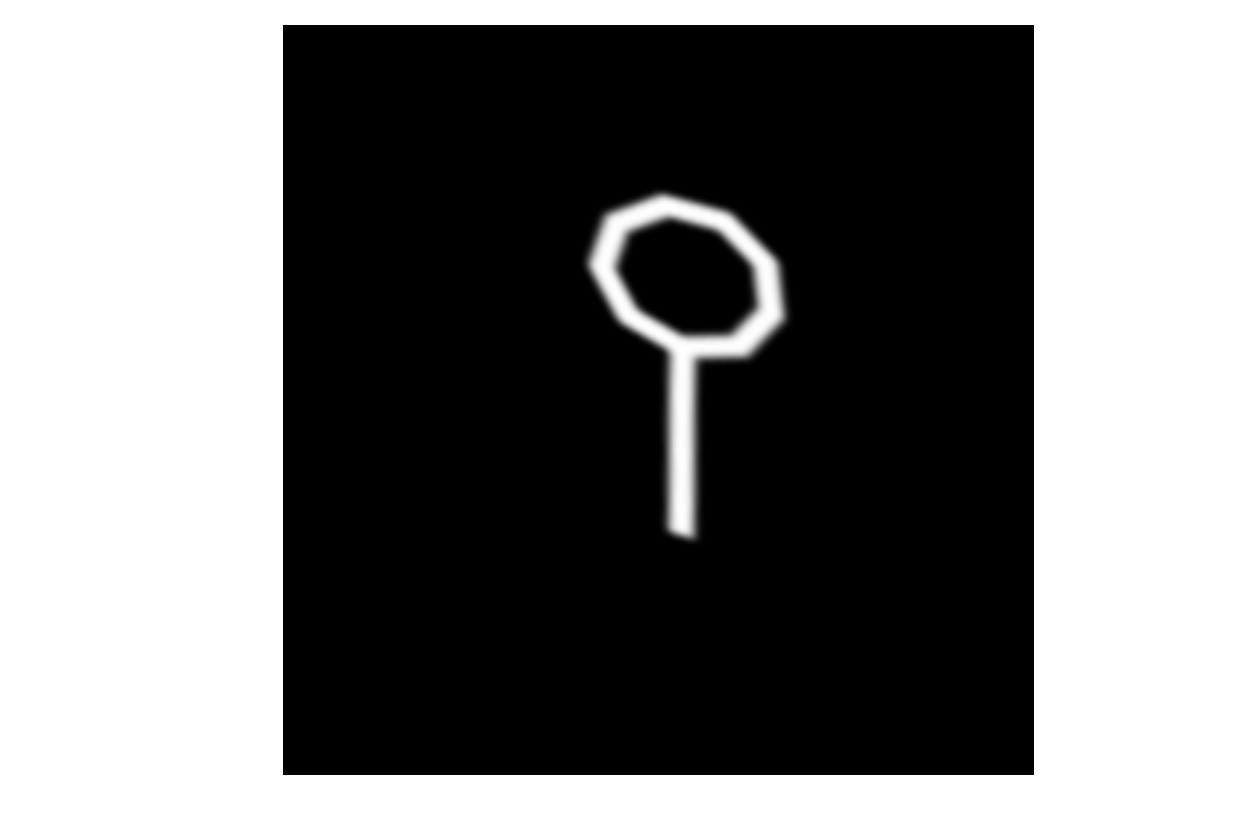}%
        \hspace{1pt}%
        \includegraphics[width=0.05\linewidth, clip=true, trim=132pt 27pt 113pt 12pt]{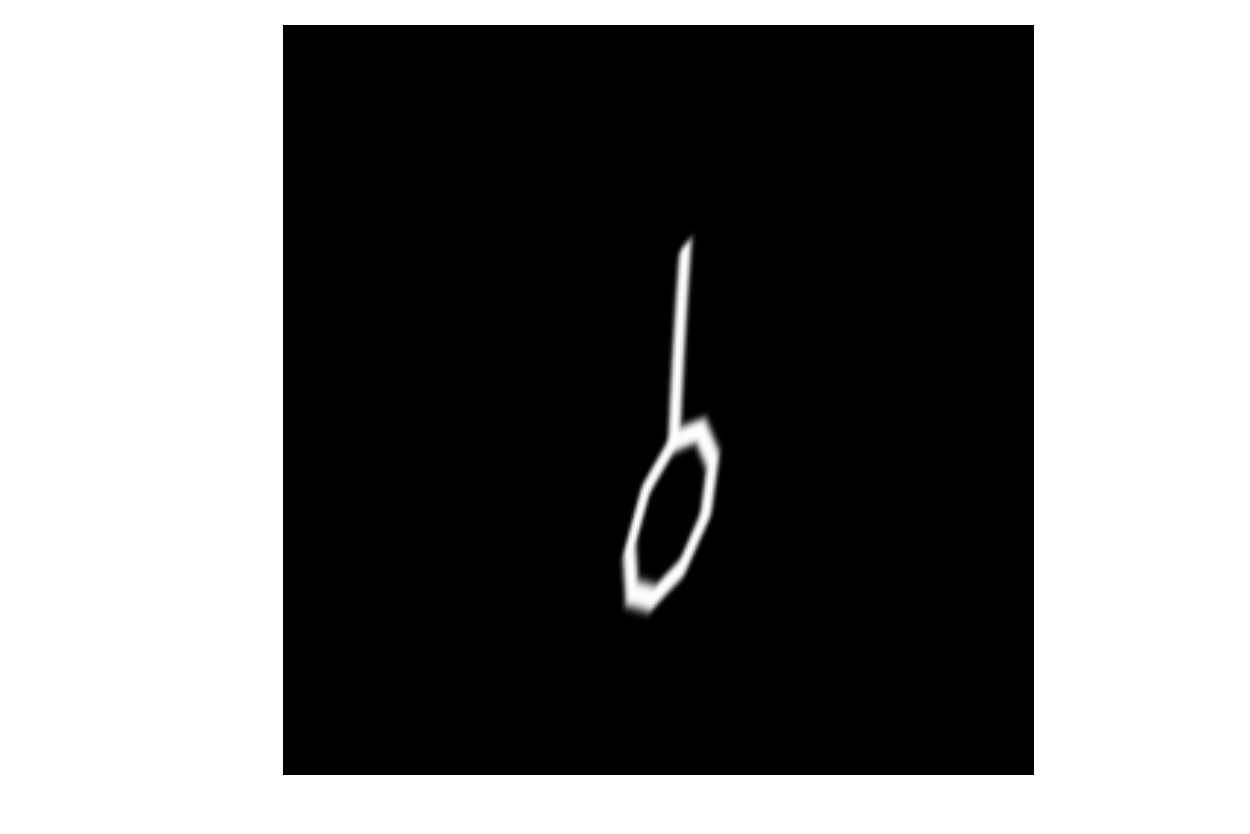}%
        & \includegraphics[width=0.05\linewidth, clip=true, trim=132pt 27pt 113pt 12pt]{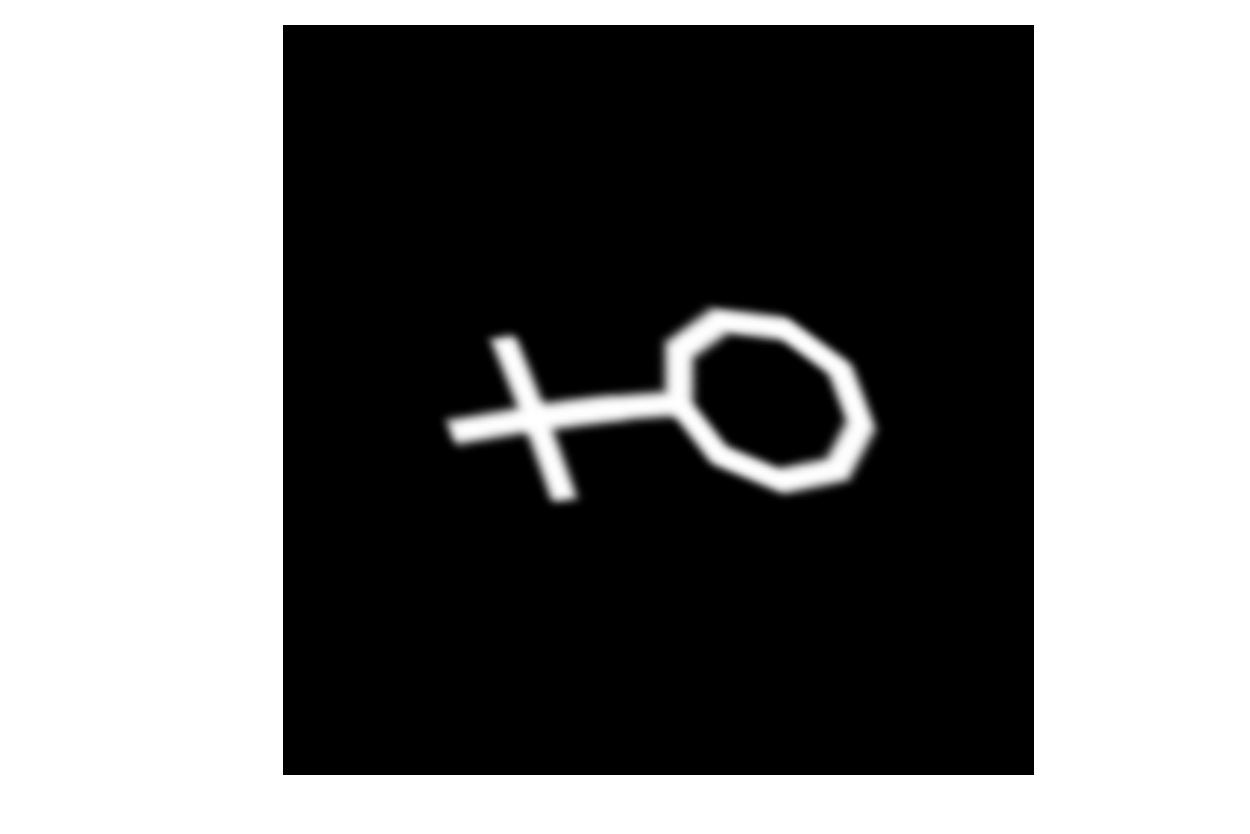}%
        \hspace{1pt}%
        \includegraphics[width=0.05\linewidth, clip=true, trim=132pt 27pt 113pt 12pt]{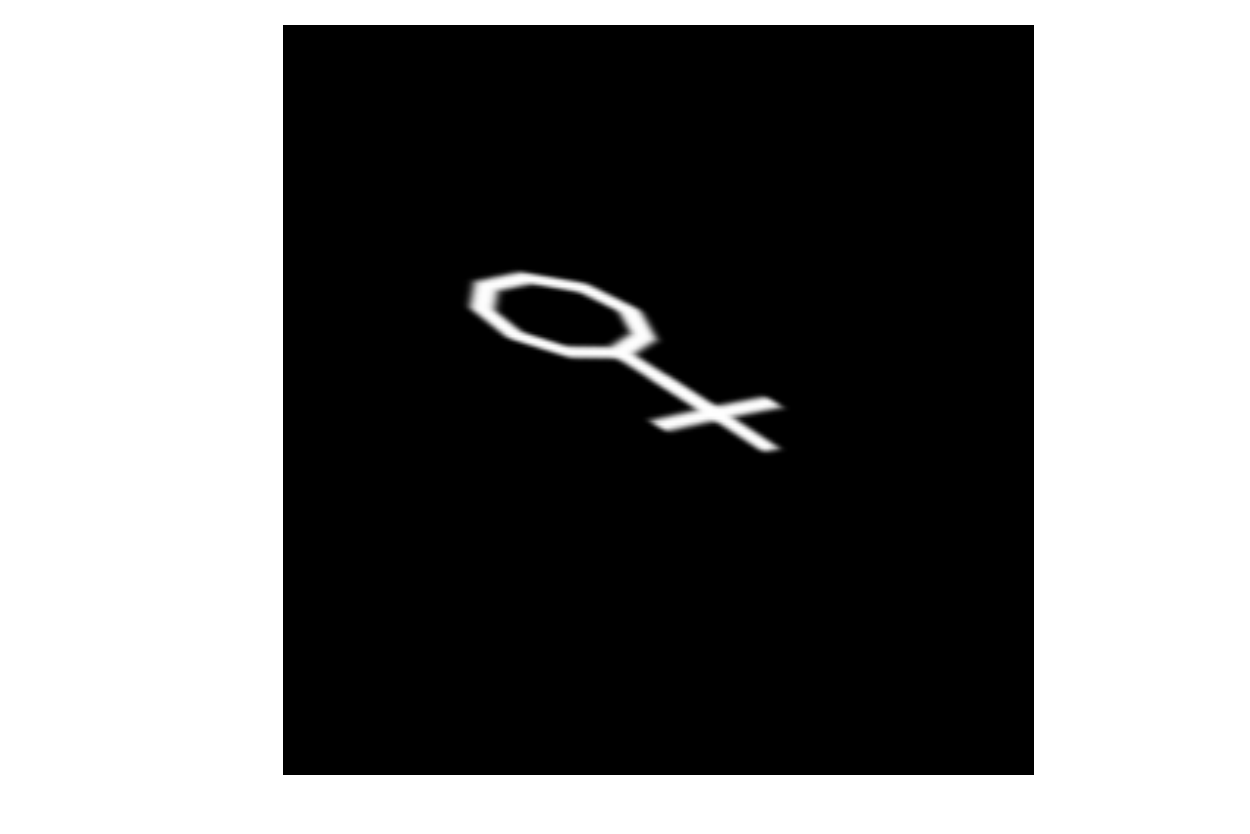}%
        & \includegraphics[width=0.05\linewidth, clip=true, trim=132pt 27pt 113pt 12pt]{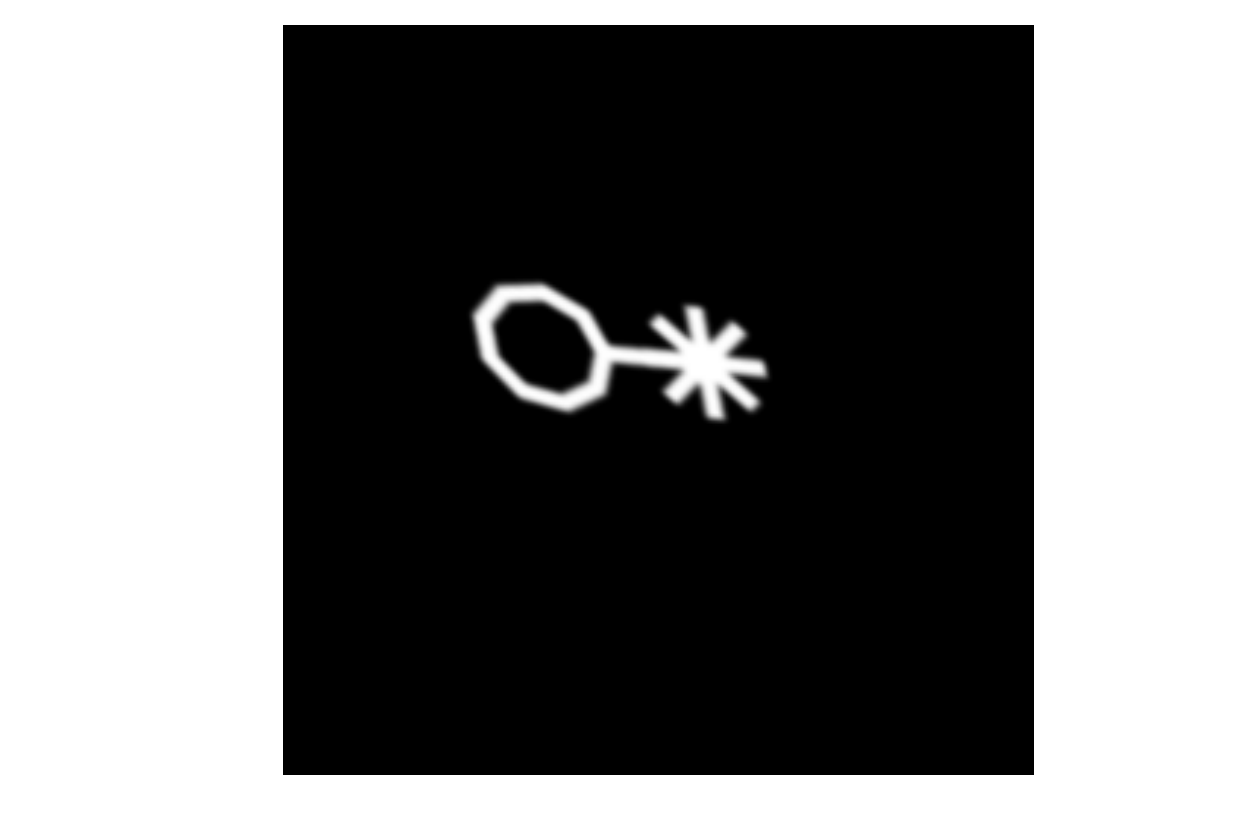}%
        \hspace{1pt}%
        \includegraphics[width=0.05\linewidth, clip=true, trim=132pt 27pt 113pt 12pt]{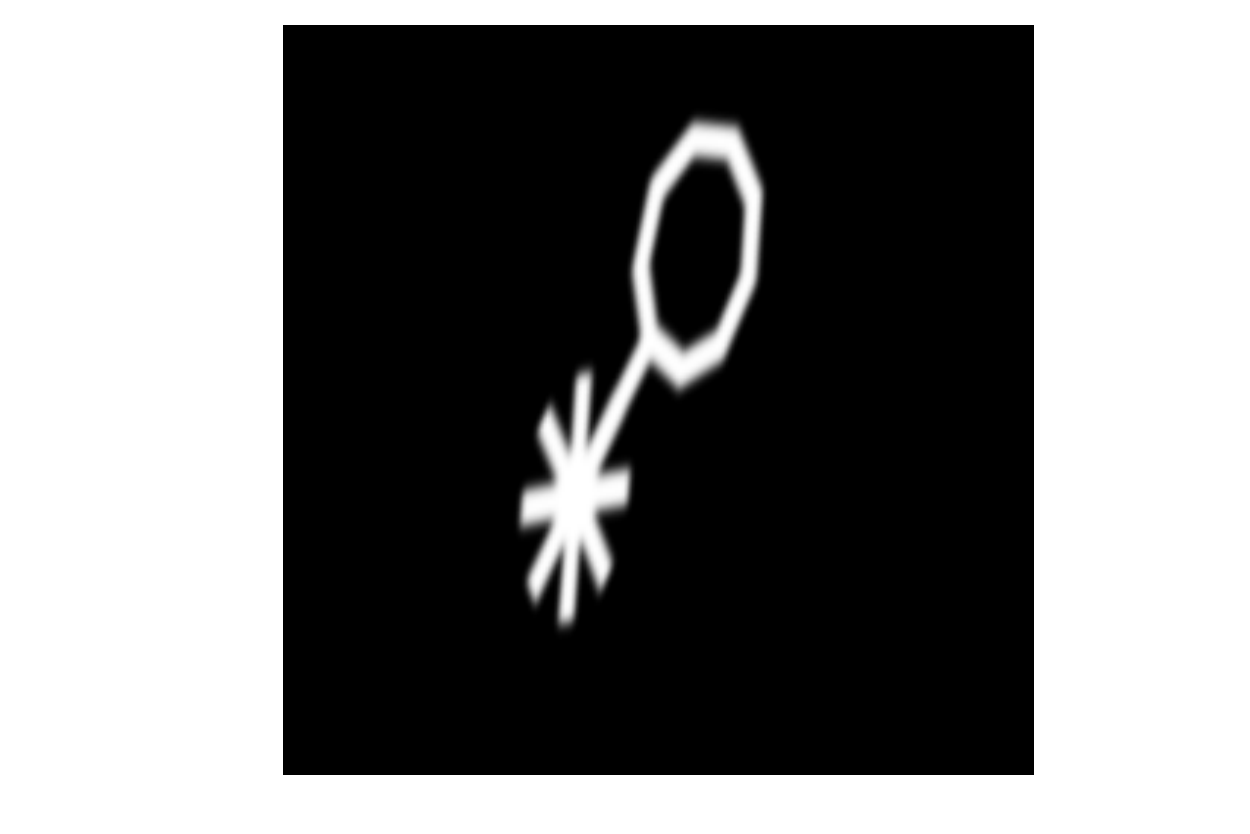}%
    \end{tabular}
    \caption{%
    Synthetic templates of the polygon dataset (top)
    and random affinely transformed samples 
    including slight non-affine perturbations 
    of the templates (bottom).}
    \label{fig:polygon_dataset_2d}
\end{figure*}

\begin{figure}[t]
    \centering%
    \footnotesize%
    \begin{tabular}{c @{\hspace{5pt}} c @{\hspace{5pt}} c}
        class~1 
        & class~2 
        & class~3
        \\%
        \includegraphics[width=0.3\linewidth, clip=true, trim=112pt 12pt 113pt 12pt]{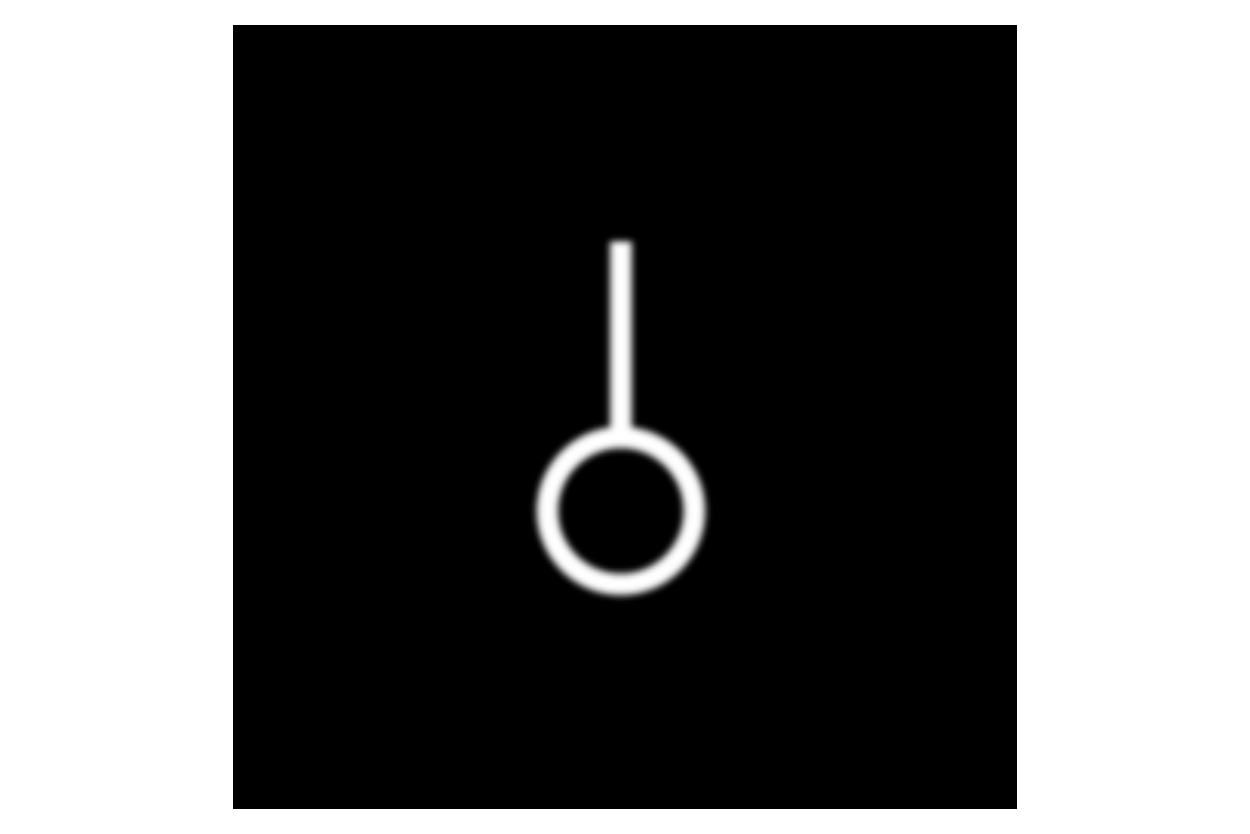}%
        & \includegraphics[width=0.3\linewidth, clip=true, trim=112pt 12pt 113pt 12pt]{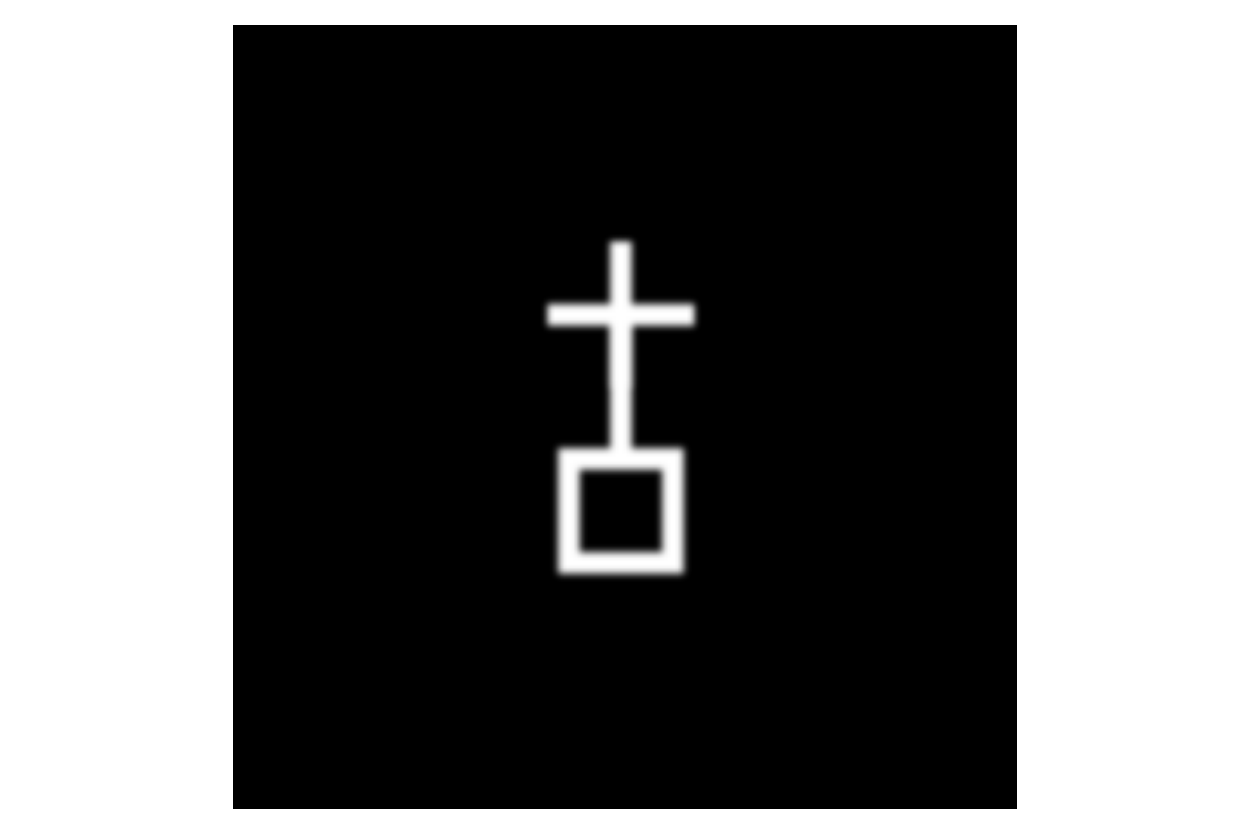}%
        & \includegraphics[width=0.3\linewidth, clip=true, trim=112pt 12pt 113pt 12pt]{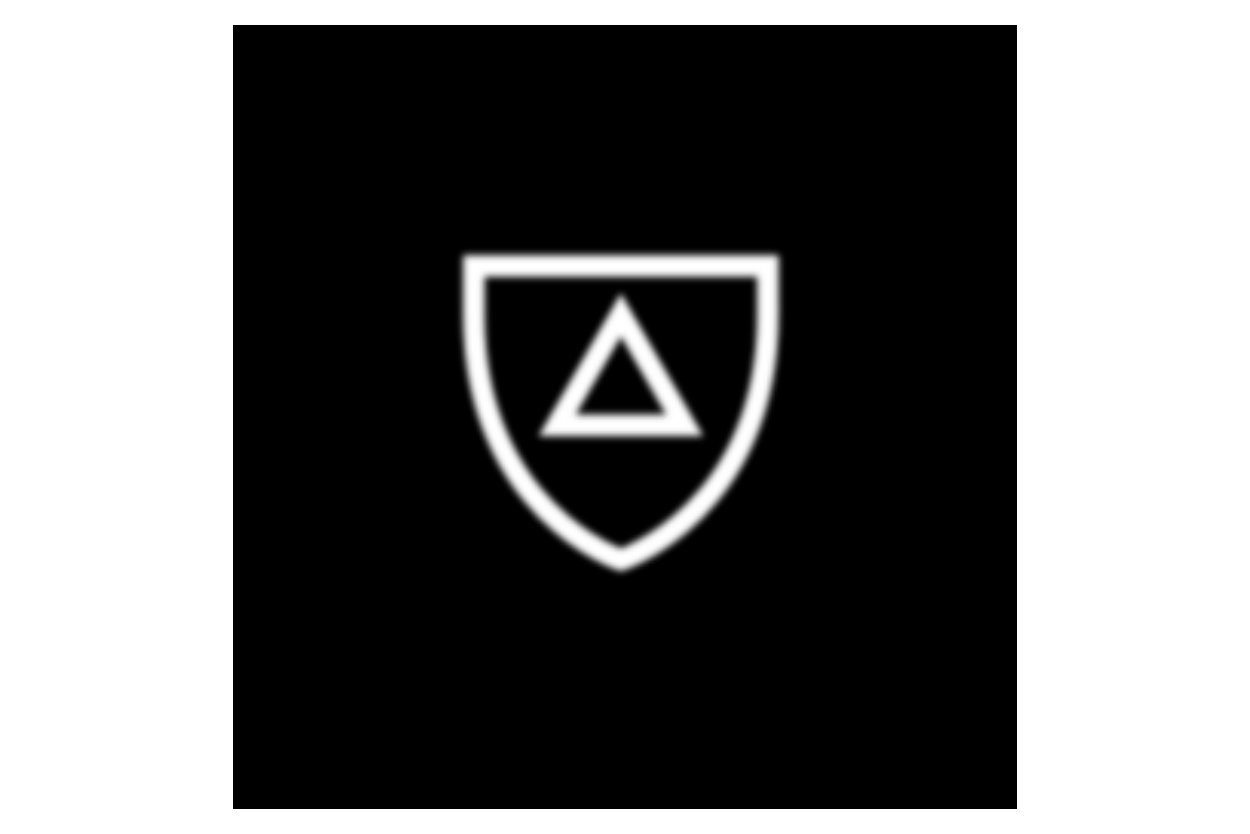}%
        \\%
        \includegraphics[width=0.143\linewidth, clip=true, trim=112pt 12pt 113pt 12pt]{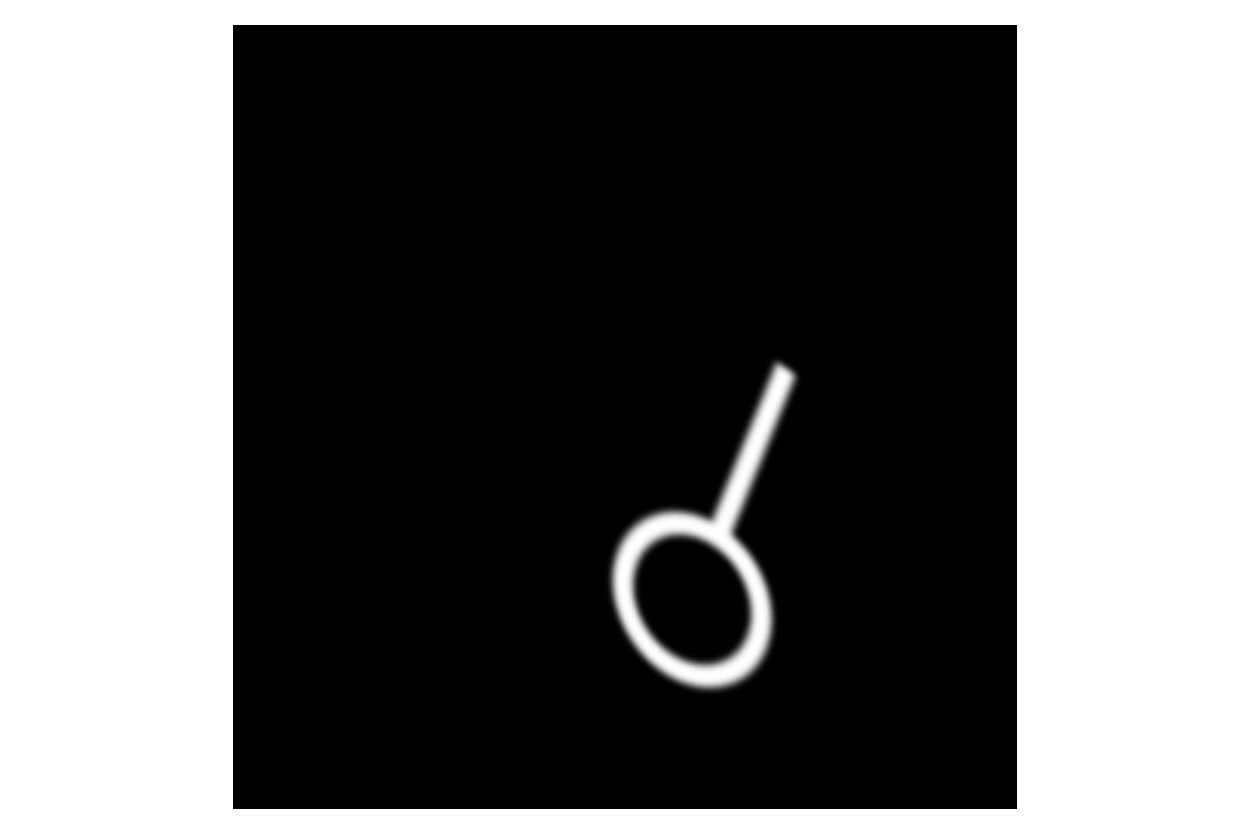}%
        \hspace{2pt}%
        \includegraphics[width=0.143\linewidth, clip=true, trim=112pt 12pt 113pt 12pt]{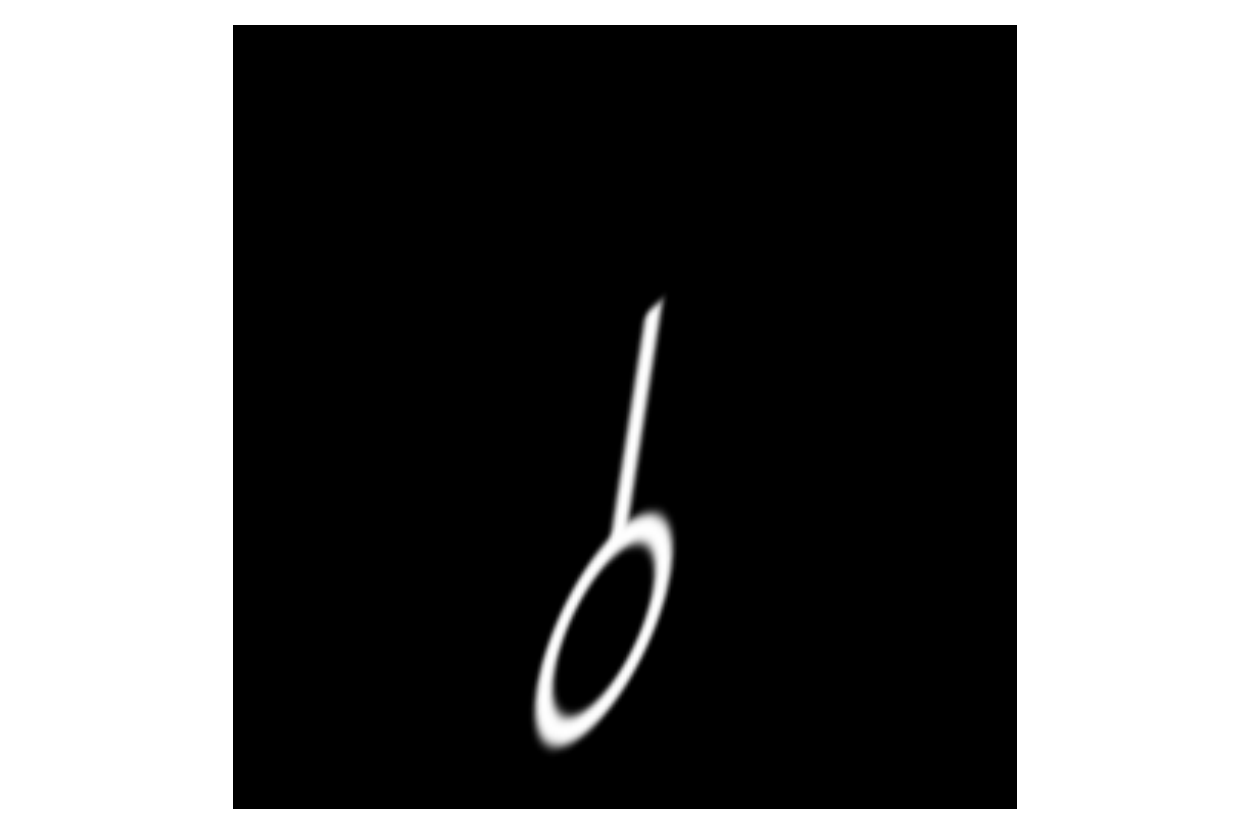}%
        & \includegraphics[width=0.143\linewidth, clip=true, trim=112pt 12pt 113pt 12pt]{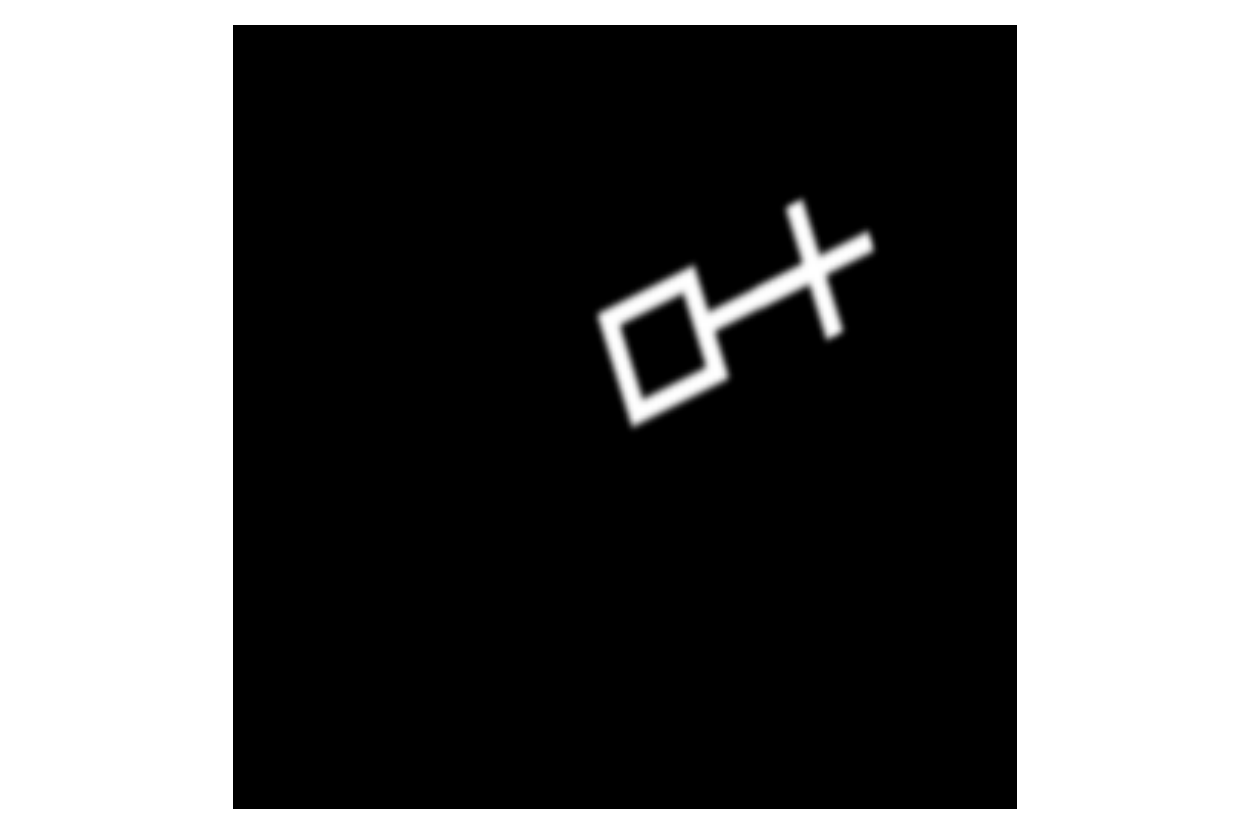}%
        \hspace{2pt}%
        \includegraphics[width=0.143\linewidth, clip=true, trim=112pt 12pt 113pt 12pt]{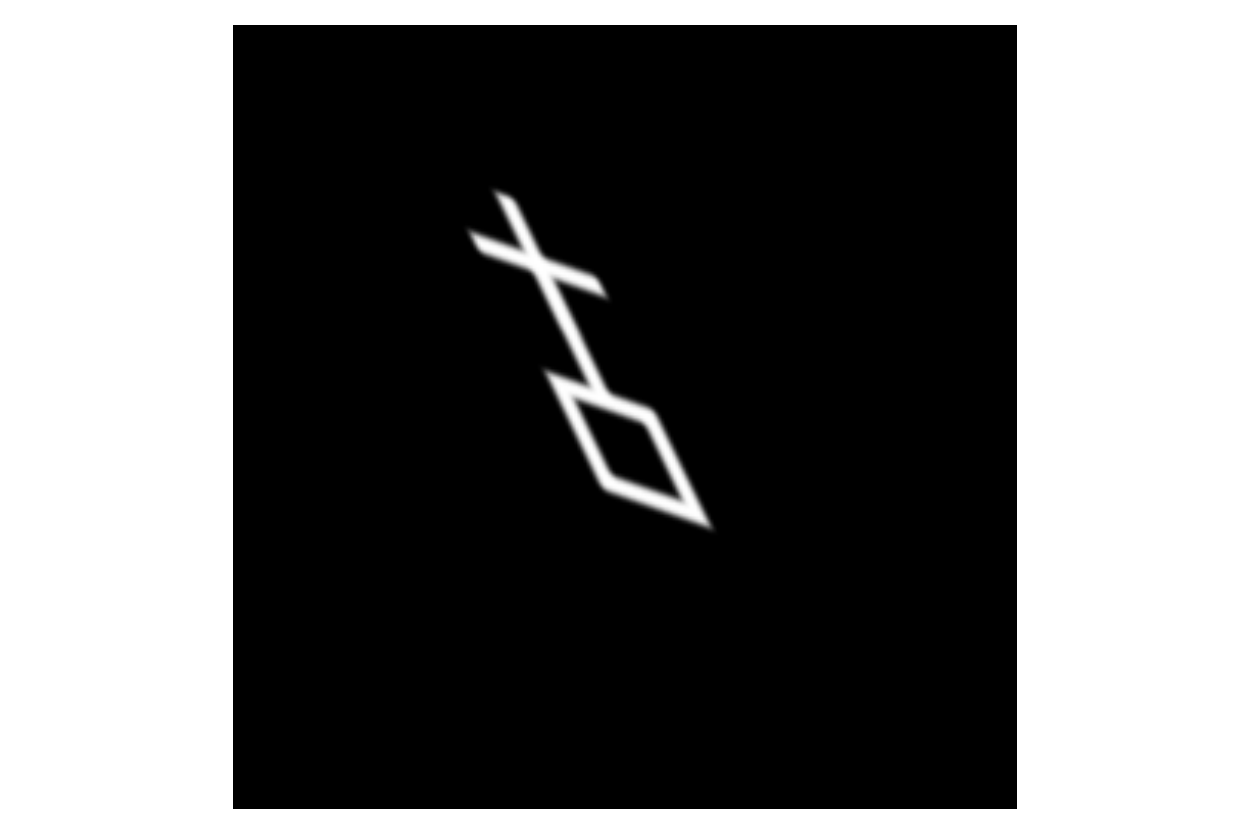}%
        & \includegraphics[width=0.143\linewidth, clip=true, trim=112pt 12pt 113pt 12pt]{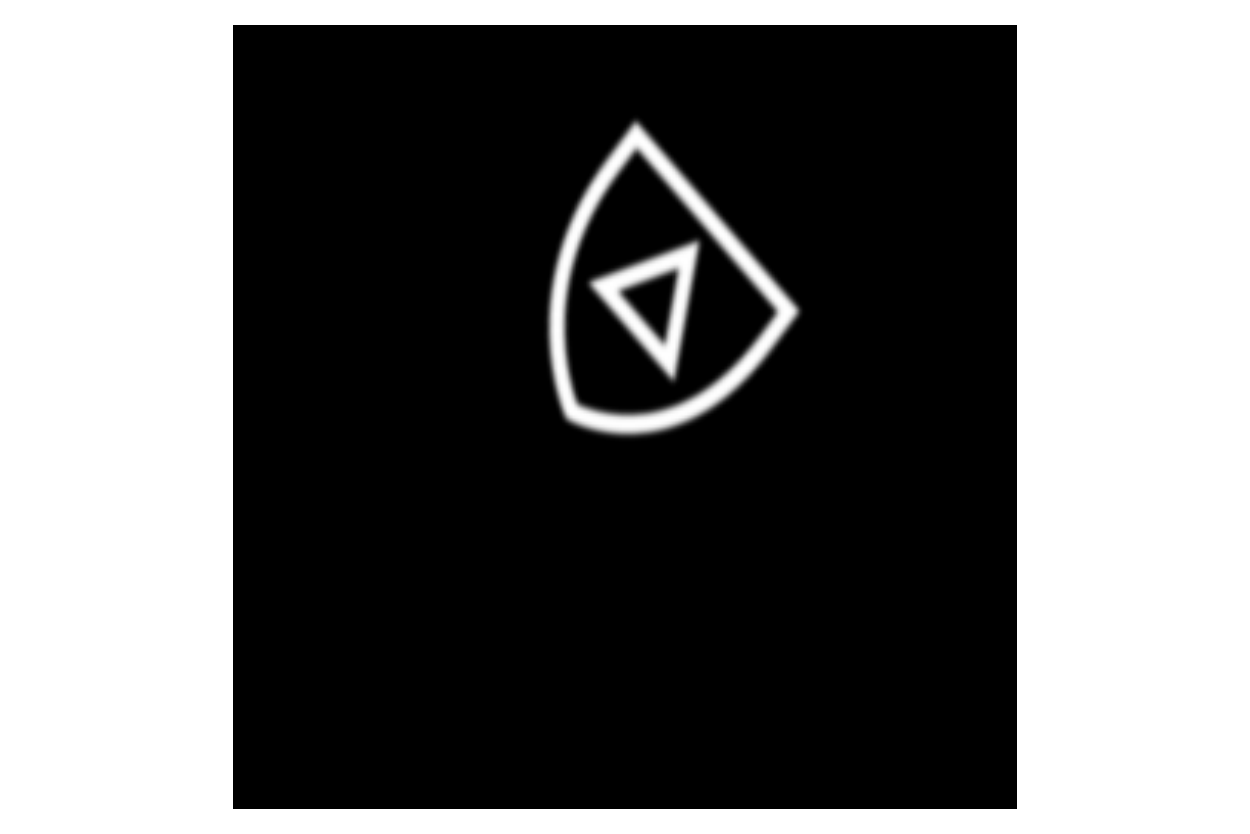}%
        \hspace{2pt}%
        \includegraphics[width=0.143\linewidth, clip=true, trim=112pt 12pt 113pt 12pt]{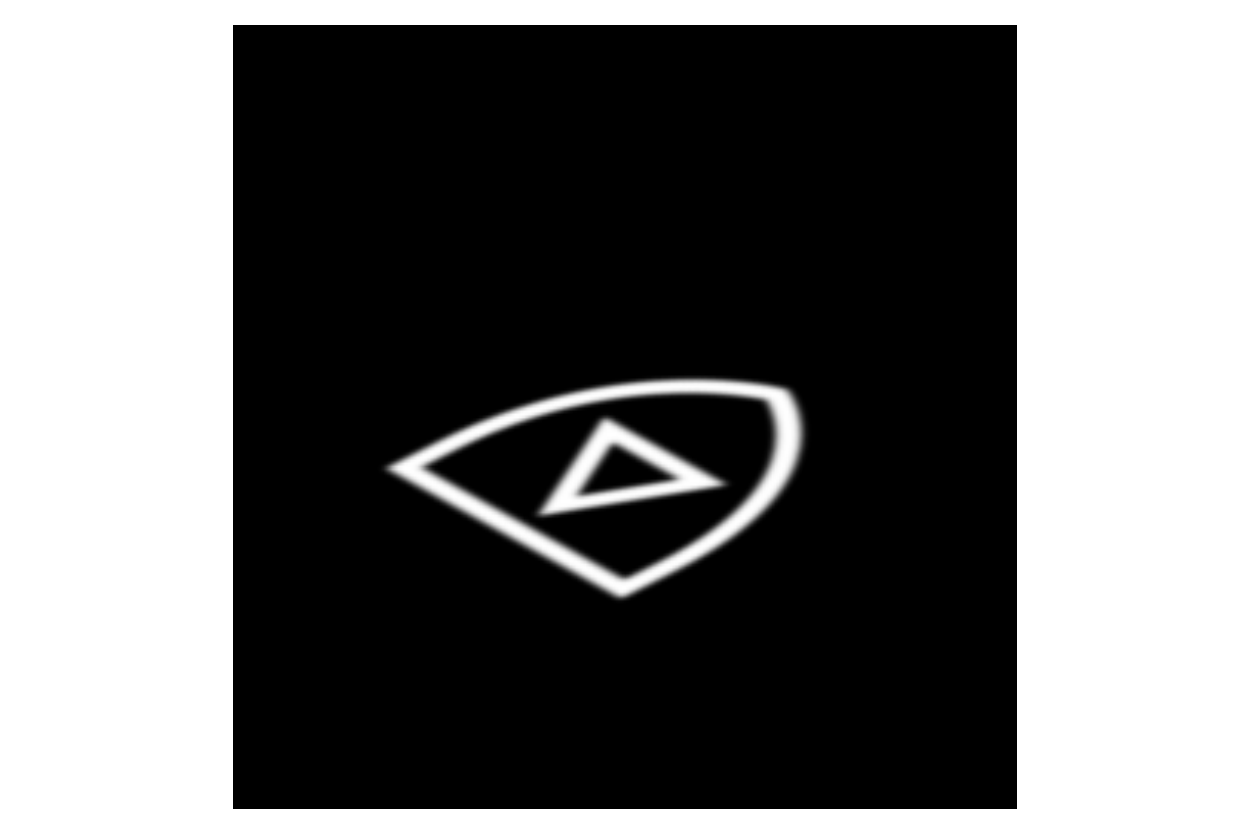}%
    \end{tabular}
    \caption{%
    Synthetic templates of the academic dataset (top)
    and random affinely transformed samples (bottom).}
    \label{fig:templates_syn}
\end{figure}

\begin{figure}[t]
    \centering%
    \footnotesize%
    \begin{tabular}{c@{\hspace{3pt}}c@{\hspace{5pt}}c@{\hspace{5pt}}c@{\hspace{5pt}}c}
        \raisebox{12pt}{\rotatebox{90}{class~1}}
        & \includegraphics[width=0.21\linewidth, clip=true, trim=136pt 29pt 106pt 12pt]{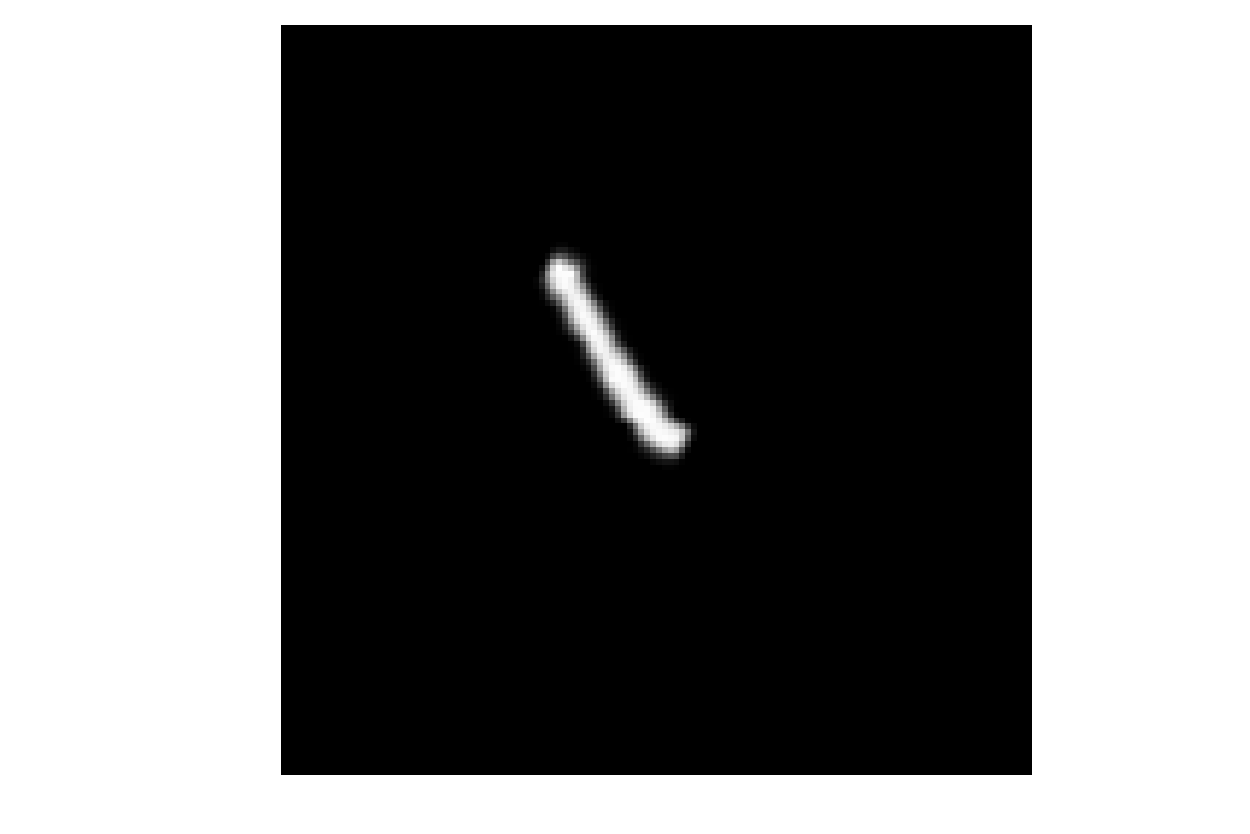}
        & \includegraphics[width=0.21\linewidth, clip=true, trim=136pt 29pt 106pt 12pt]{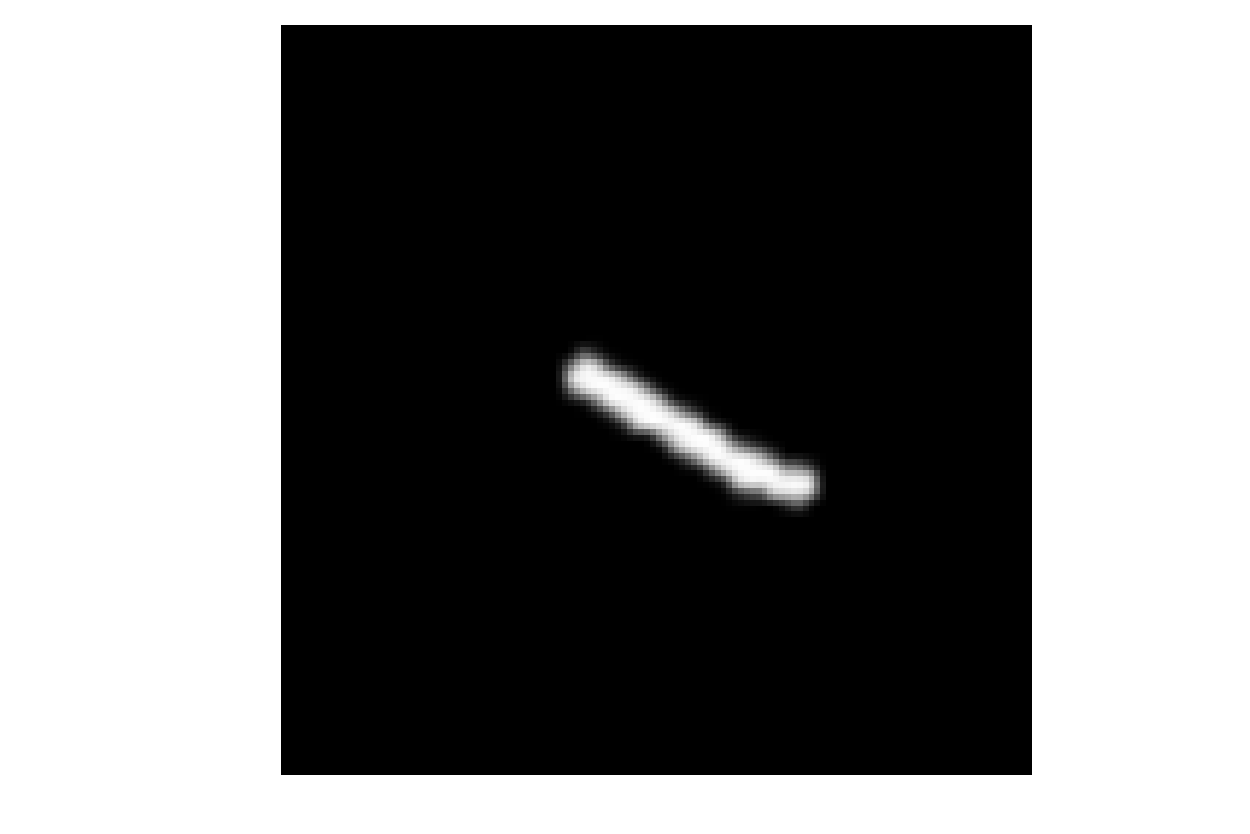}
        & \includegraphics[width=0.21\linewidth, clip=true, trim=136pt 29pt 106pt 12pt]{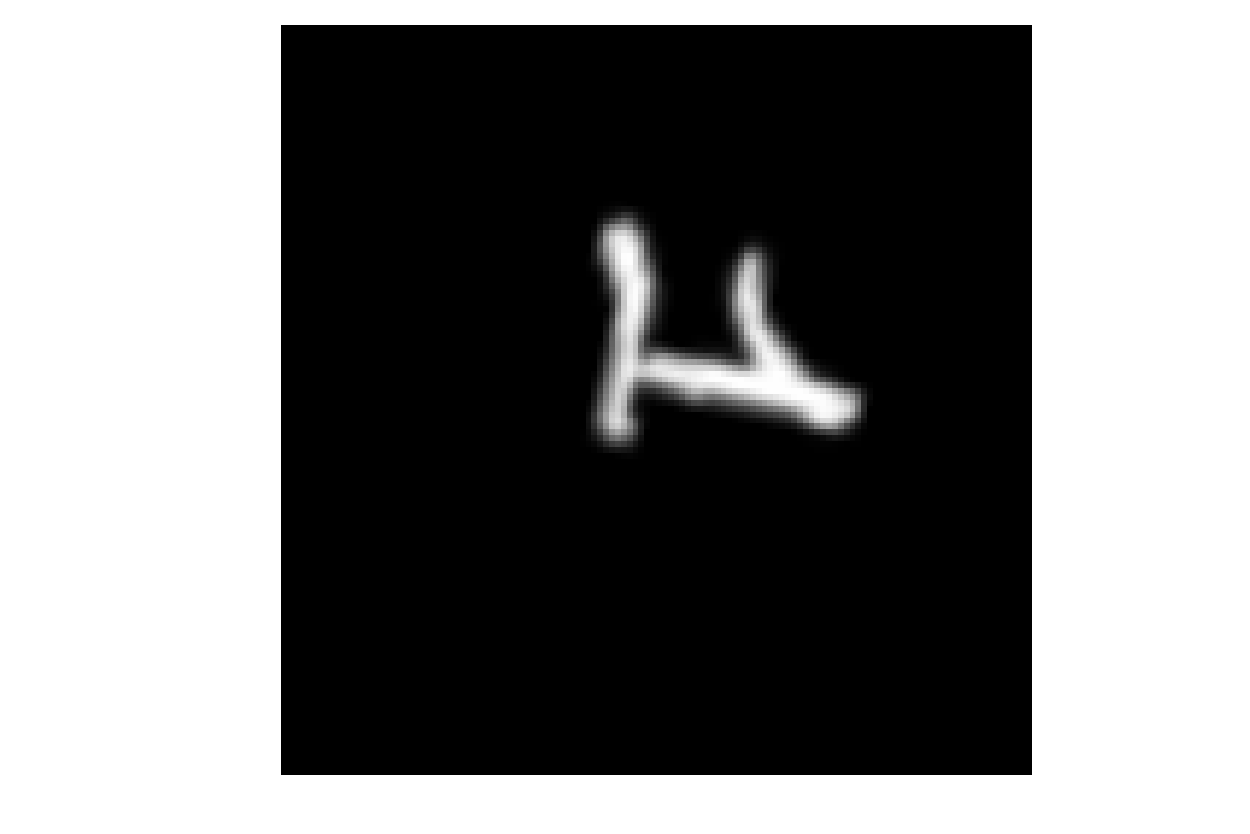}
        & \includegraphics[width=0.21\linewidth, clip=true, trim=136pt 29pt 106pt 12pt]{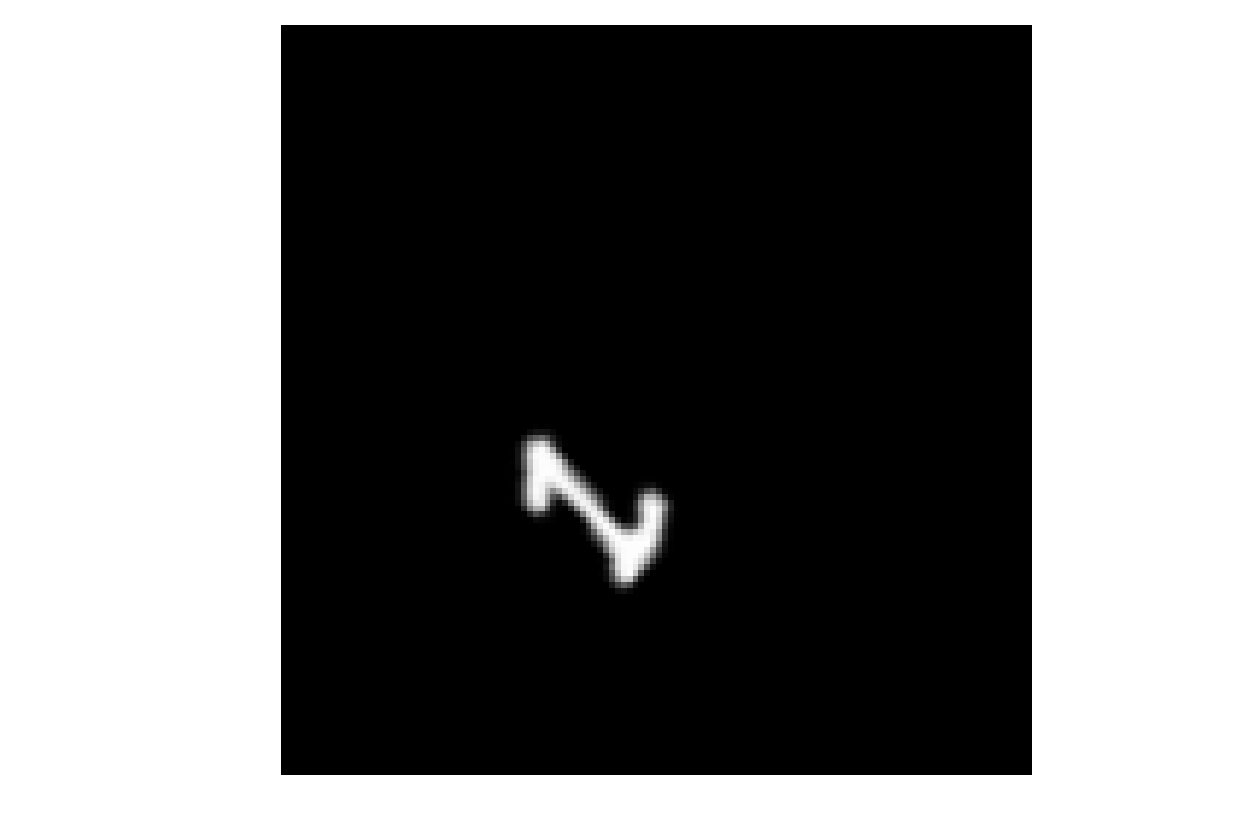}
        \\
        \raisebox{12pt}{\rotatebox{90}{class~5}}
        & \includegraphics[width=0.21\linewidth, clip=true, trim=136pt 29pt 106pt 12pt]{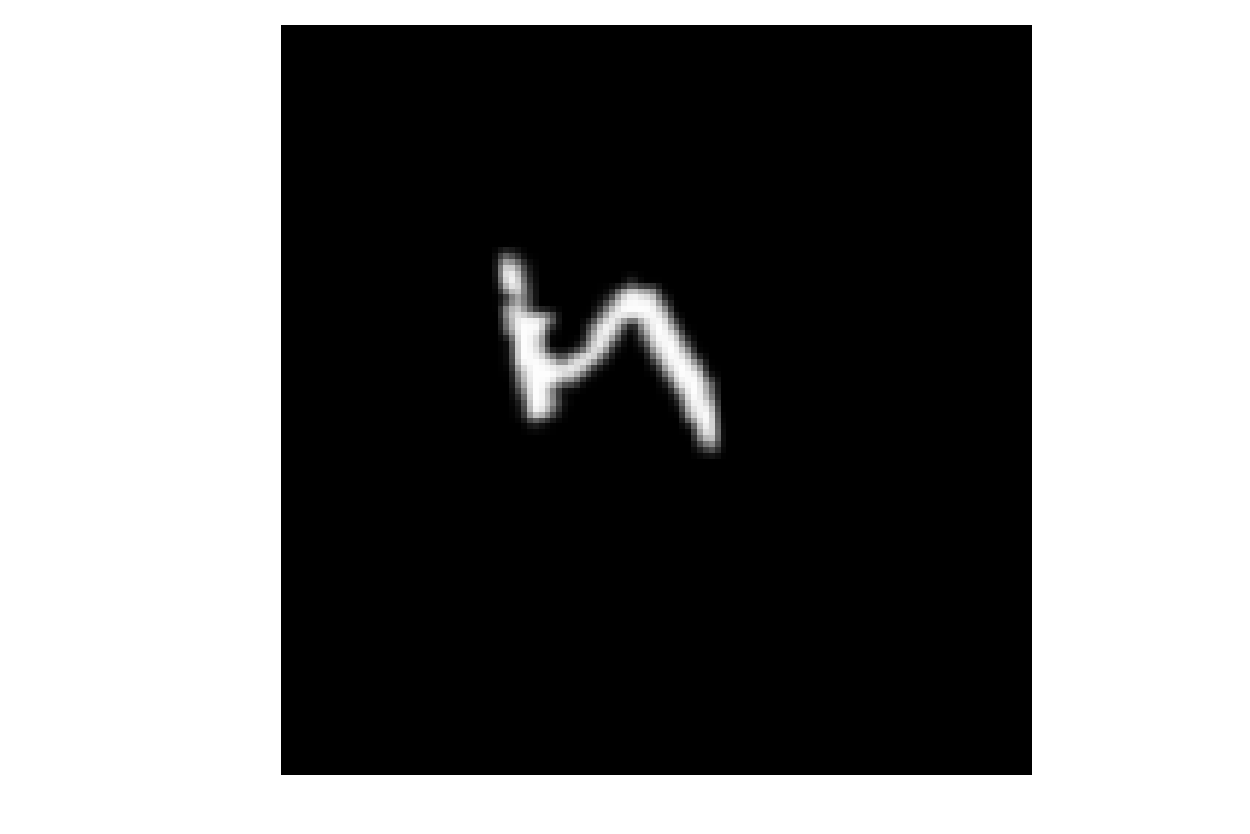}
        & \includegraphics[width=0.21\linewidth, clip=true, trim=136pt 29pt 106pt 12pt]{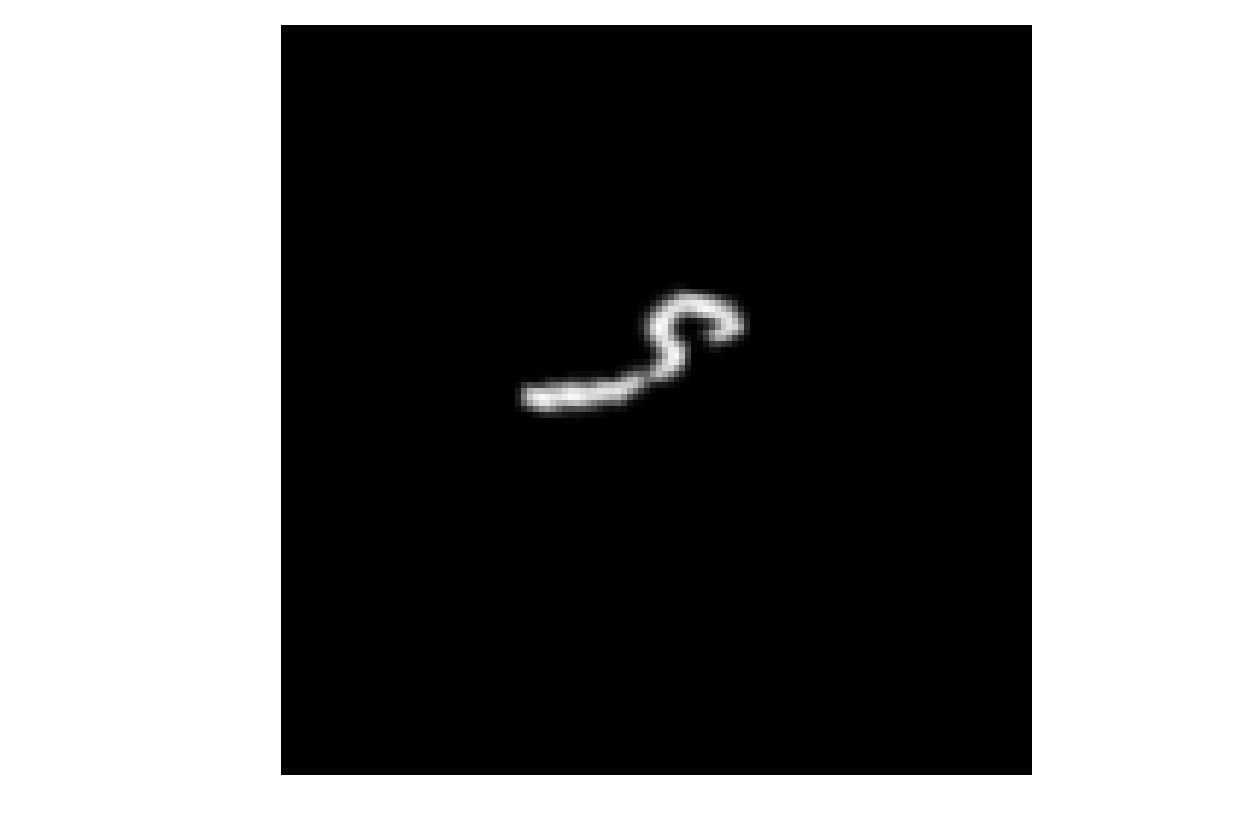}
        & \includegraphics[width=0.21\linewidth, clip=true, trim=136pt 29pt 106pt 12pt]{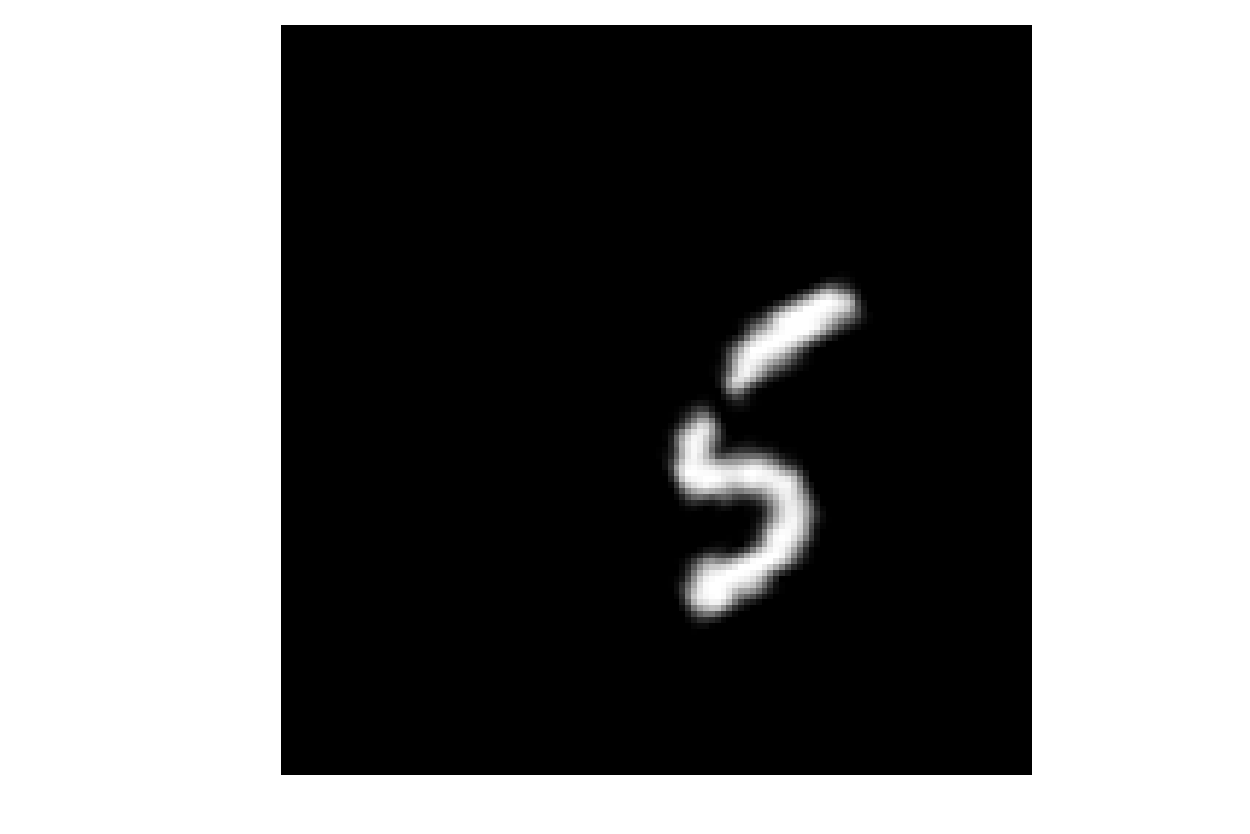}
        & \includegraphics[width=0.21\linewidth, clip=true, trim=136pt 29pt 106pt 12pt]{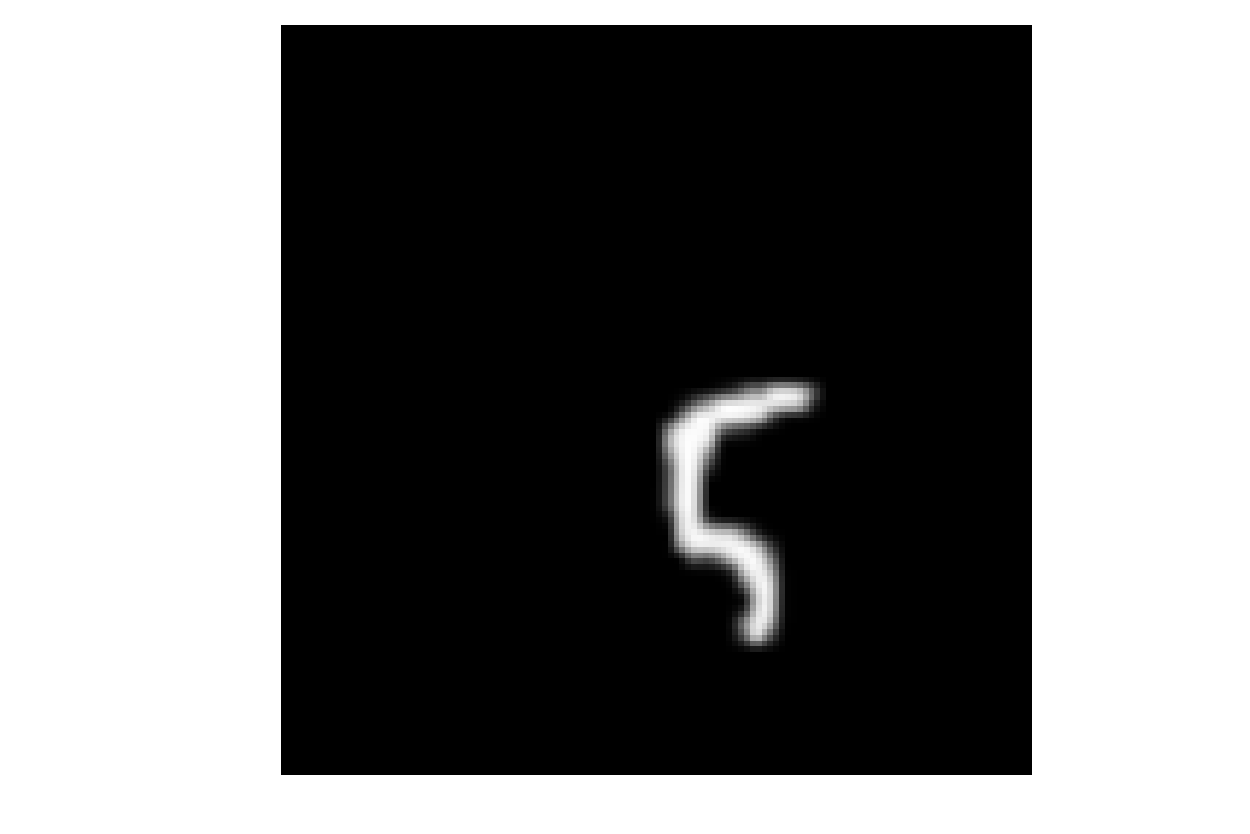}
        \\
        \raisebox{12pt}{\rotatebox{90}{class~7}}
        & \includegraphics[width=0.21\linewidth, clip=true, trim=136pt 29pt 106pt 12pt]{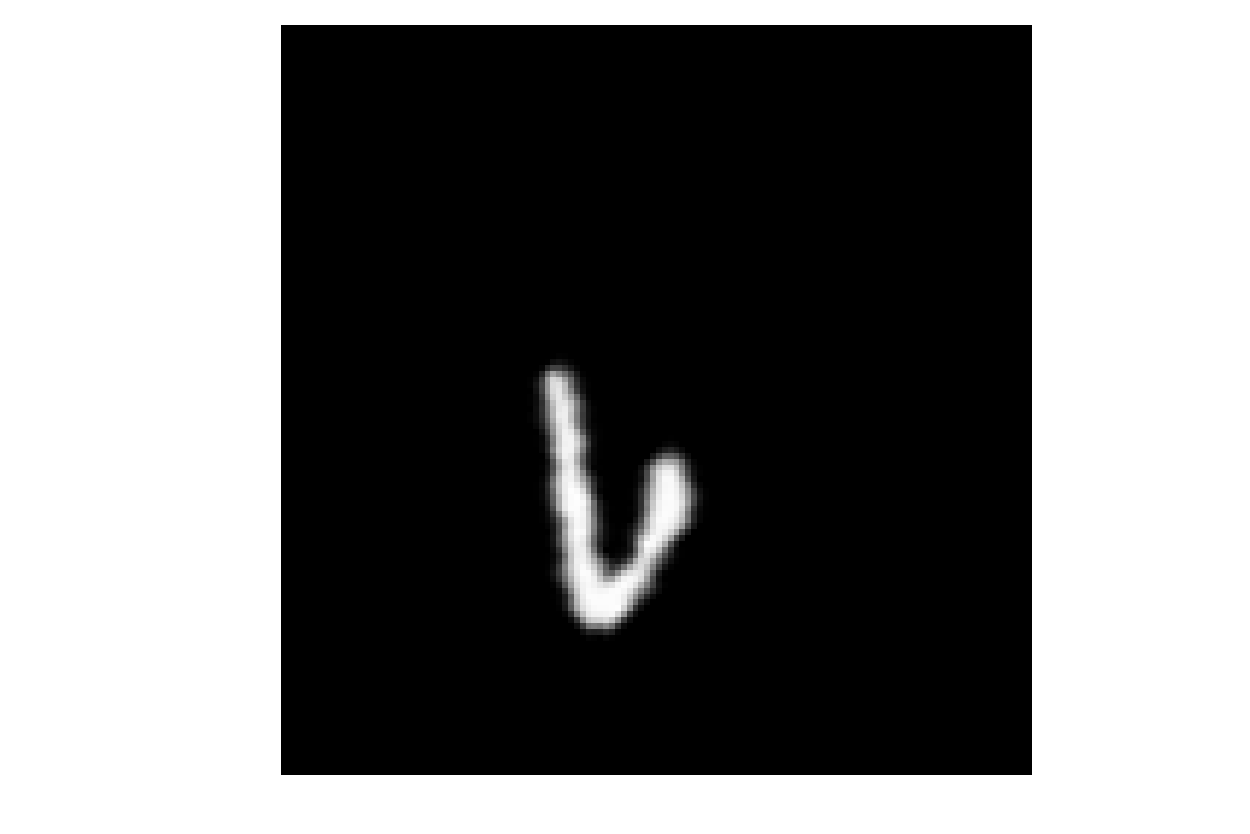}
        & \includegraphics[width=0.21\linewidth, clip=true, trim=136pt 29pt 106pt 12pt]{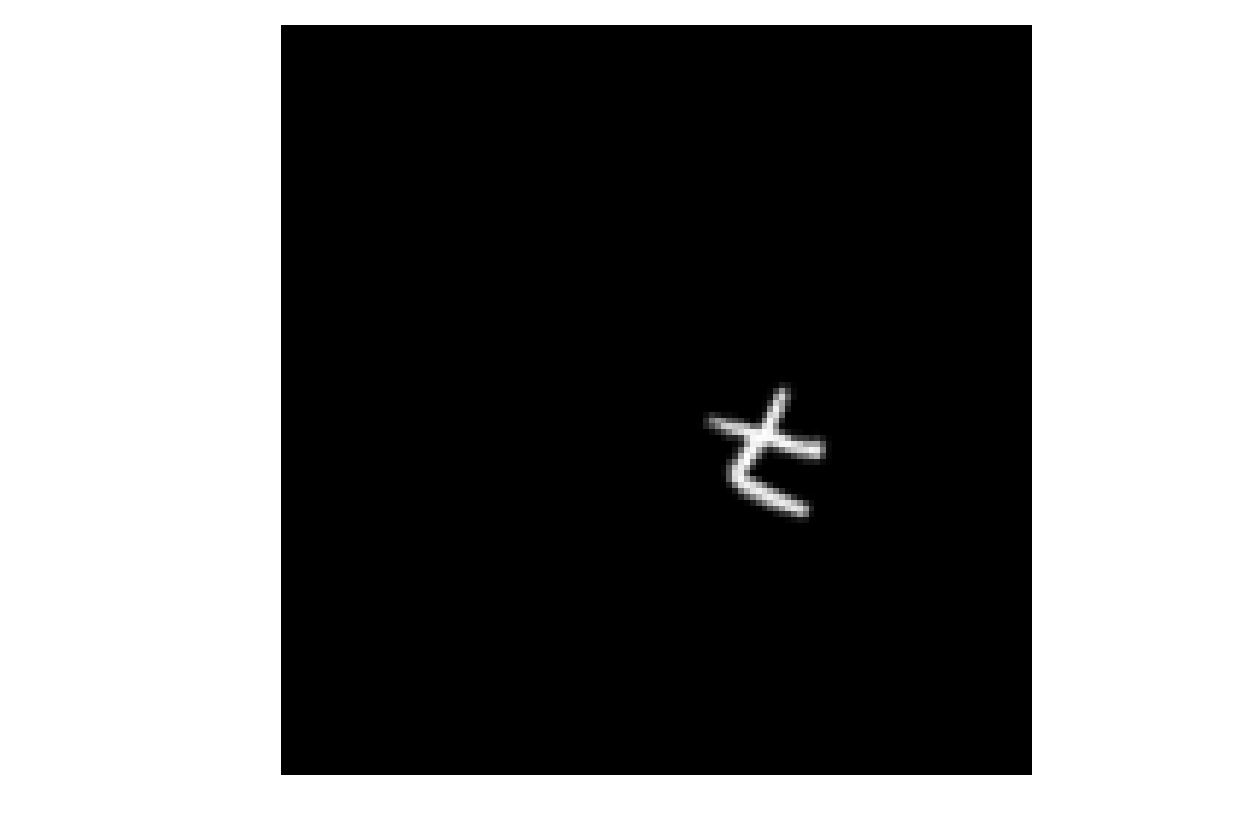}
        & \includegraphics[width=0.21\linewidth, clip=true, trim=136pt 29pt 106pt 12pt]{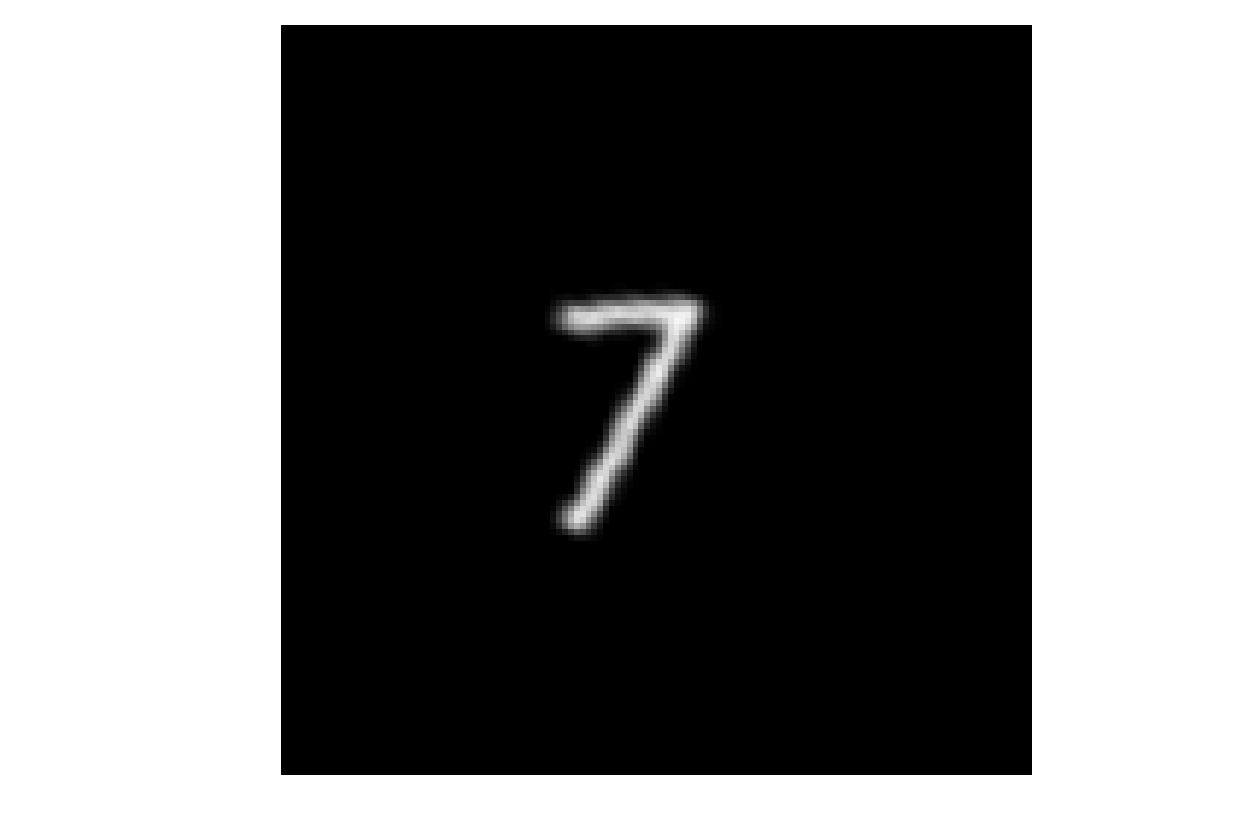}
        & \includegraphics[width=0.21\linewidth, clip=true, trim=136pt 29pt 106pt 12pt]{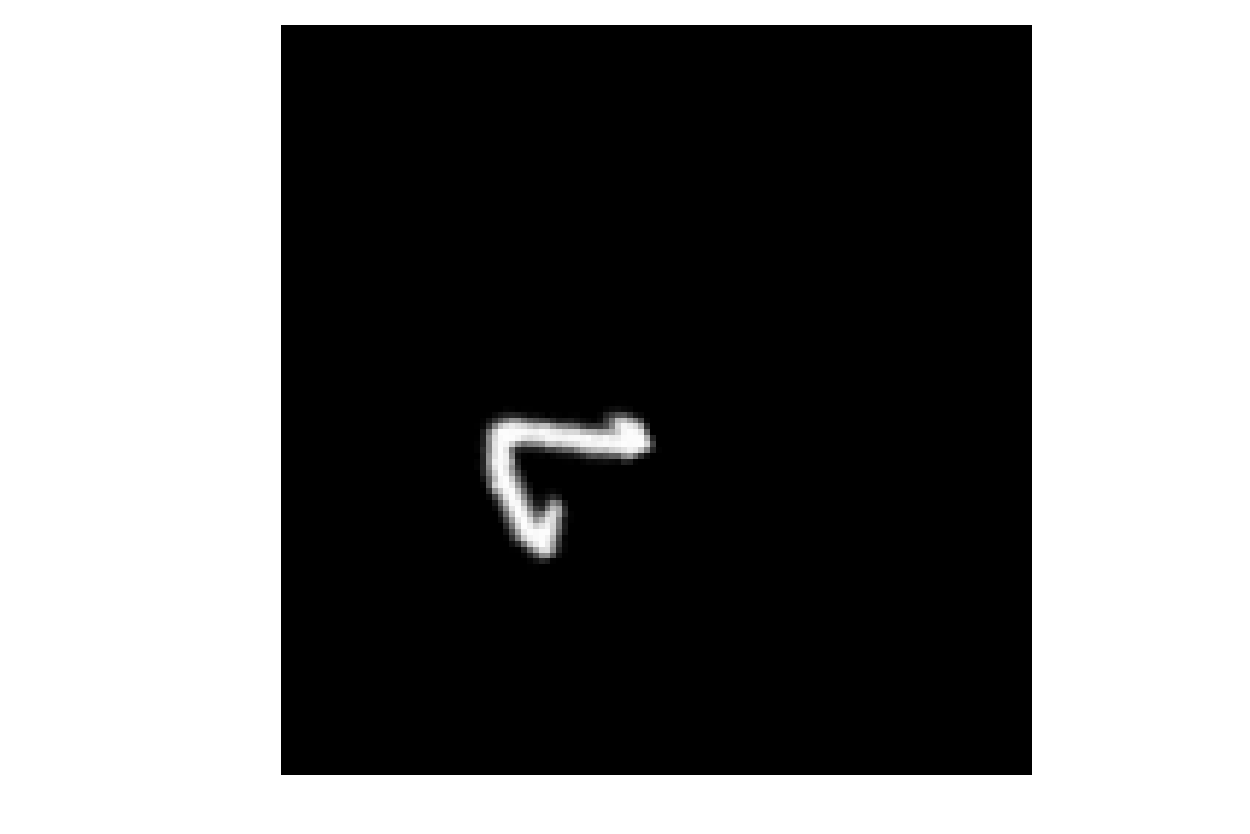}
    \end{tabular}
    \caption{%
    Samples of the LinMNIST dataset 
    (ones, fives and sevens) 
    based on affinely transformed MNIST digits.}
    \label{fig:templates_mnist}
\end{figure}

Our two-dimensional numerical experiments
make use of the following four datasets:
\begin{enumerate}[label=\alph*)]
    \item \textbf{Academic dataset.}
    Based on (up to) three synthetic template symbols,
    cf.~Figure~\ref{fig:templates_syn} (top row),
    represented by $256\times256$ pixels,
    we construct our classes by
    affinely transforming the templates with 
    random pixel shifts in $[-40,40]$, 
    rotation angles in $[0^\circ,360^\circ)$, 
    anisotropic scaling in $[0.75,1.0]$, 
    and shearing in $[-30^\circ,30^\circ]$,
    cf.~Figure~\ref{fig:templates_syn} (bottom row).
    Note that this dataset was first introduced
    in \cite{Beckmann2024a}
    for proof-of-concept experiments 
    showing linear separability in \mNRCDT{} space.
    
    \item \textbf{Polygon dataset.}
    Nine synthetic template symbols, 
    cf.~Figure~\ref{fig:polygon_dataset_2d} (top row),
    each of size $256\times256$ pixels,
    are slightly non-affinely deformed
    and, thereon, affinely transformed
    with random pixel shifts in $[-20,20]$, 
    rotation angles in $[0^\circ,360^\circ)$, 
    anisotropic scaling in $[0.5,1.25]$, 
    and shearing in $[-30^\circ,30^\circ]$,
    cf.~Figure~\ref{fig:polygon_dataset_2d} (bottom row).
    For the non-affine deformations,
    we assign to the $(j,k)$th pixel
    the bi-quadratically interpolated gray values
    at the perturbed location
    \begin{equation*}
        \bigl(j + a_1 \sin\bigl(\tfrac{2\pi f_1}{256} \, k\bigr), k + a_2 \cos\bigl(\tfrac{2\pi f_2}{256} \, j\bigr)\bigr)
    \end{equation*}
    with random frequencies $f_1, f_2$ in $[1.5,2.5]$
    and amplitudes $a_1, a_2$ in $[0.5, 2.0]$.
    
    \item \textbf{LinMNIST dataset.}
    This dataset was first introduced in~\cite{Beckmann2024} and
    consists of affinely transformed MNIST digits~\cite{Deng2012}.
    Here, we restrict ourselves to the three classes $\{1,5,7\}$,
    rescale the data to $128\times128$ pixels
    and use random pixel shifts in $[-20,20]$,
    rotation angles in $[0^\circ, 360^\circ)$,
    and anisotropic scaling in $[0.5,1.25]$,
    see Figure~\ref{fig:templates_mnist}
    for a random selection of samples
    of the generated dataset.

    \begin{figure*}[t]
        \centering%
        \resizebox{\linewidth}{!}{%
        \begin{tabular}{c @{\hspace{3pt}} c @{\hspace{3pt}} c @{\hspace{3pt}} c @{\hspace{3pt}} c @{\hspace{3pt}} c @{\hspace{3pt}} c @{\hspace{3pt}} c @{\hspace{3pt}} c @{\hspace{3pt}} c @{\hspace{3pt}} c}
            agnus dei 
            & bishop
            & crown
            & lilies (1st)
            & pinecone
            & keys
            & lilies (2nd)
            & moon
            & poppy 
            & orb
            & shield \\%
            \boxed{\includegraphics[width=0.1\linewidth, clip=true, trim=132pt 27pt 113pt 12pt]{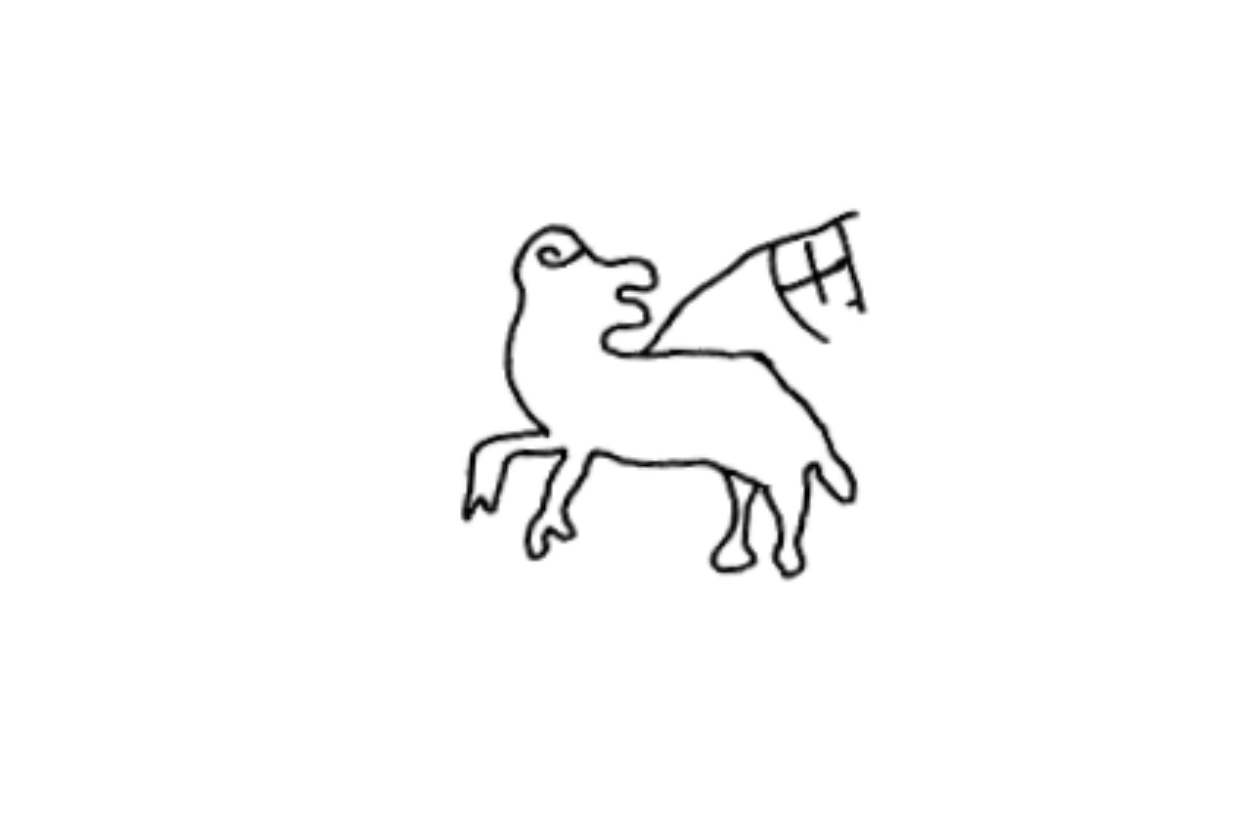}}%
            & \boxed{\includegraphics[width=0.1\linewidth, clip=true, trim=132pt 27pt 113pt 12pt]{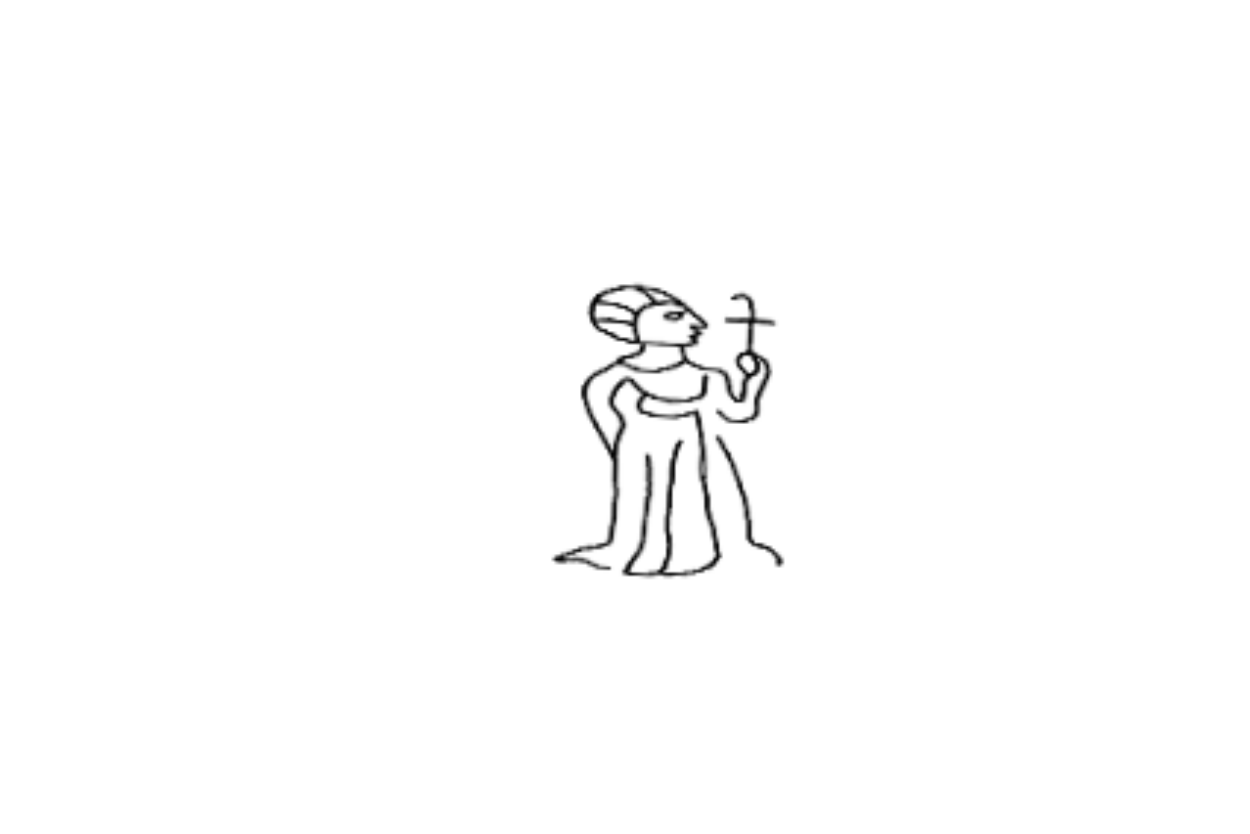}}%
            & \boxed{\includegraphics[width=0.1\linewidth, clip=true, trim=132pt 27pt 113pt 12pt]{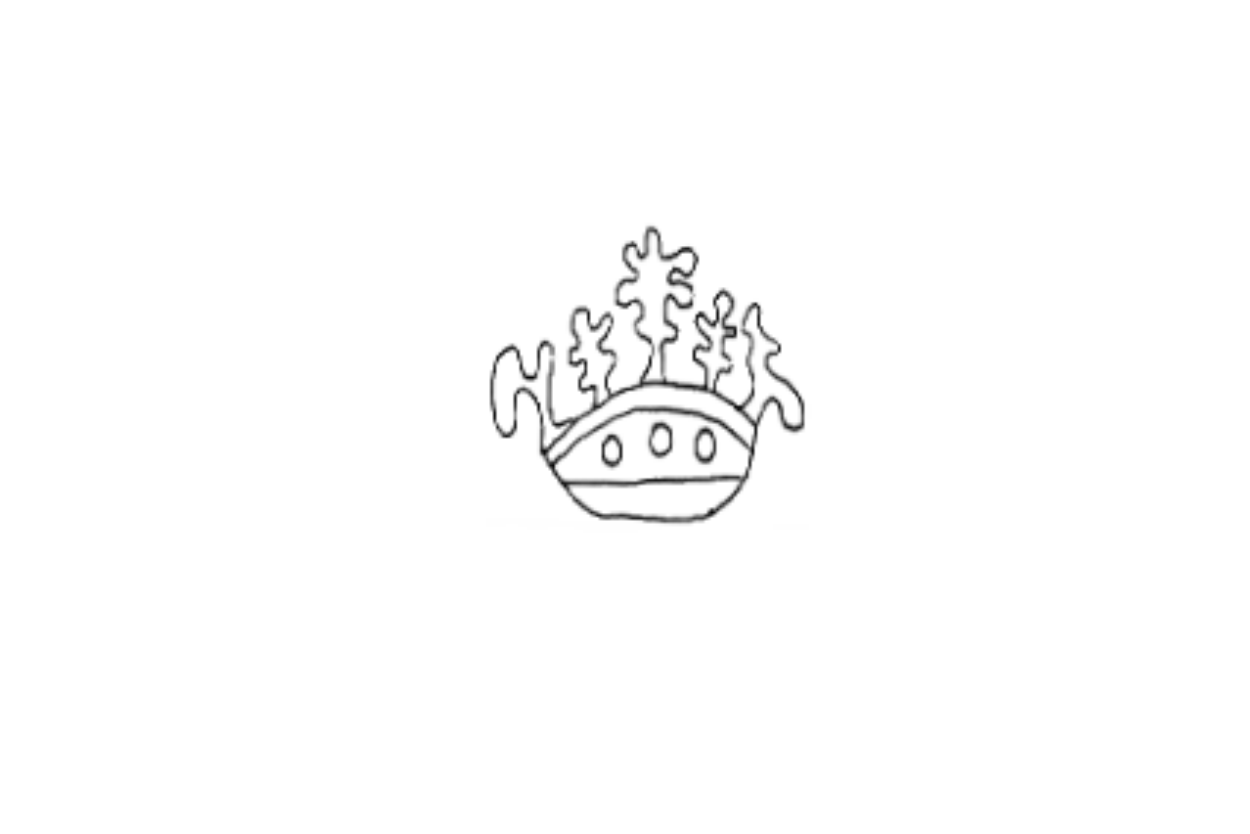}}%
            & \boxed{\includegraphics[width=0.1\linewidth, clip=true, trim=132pt 27pt 113pt 12pt]{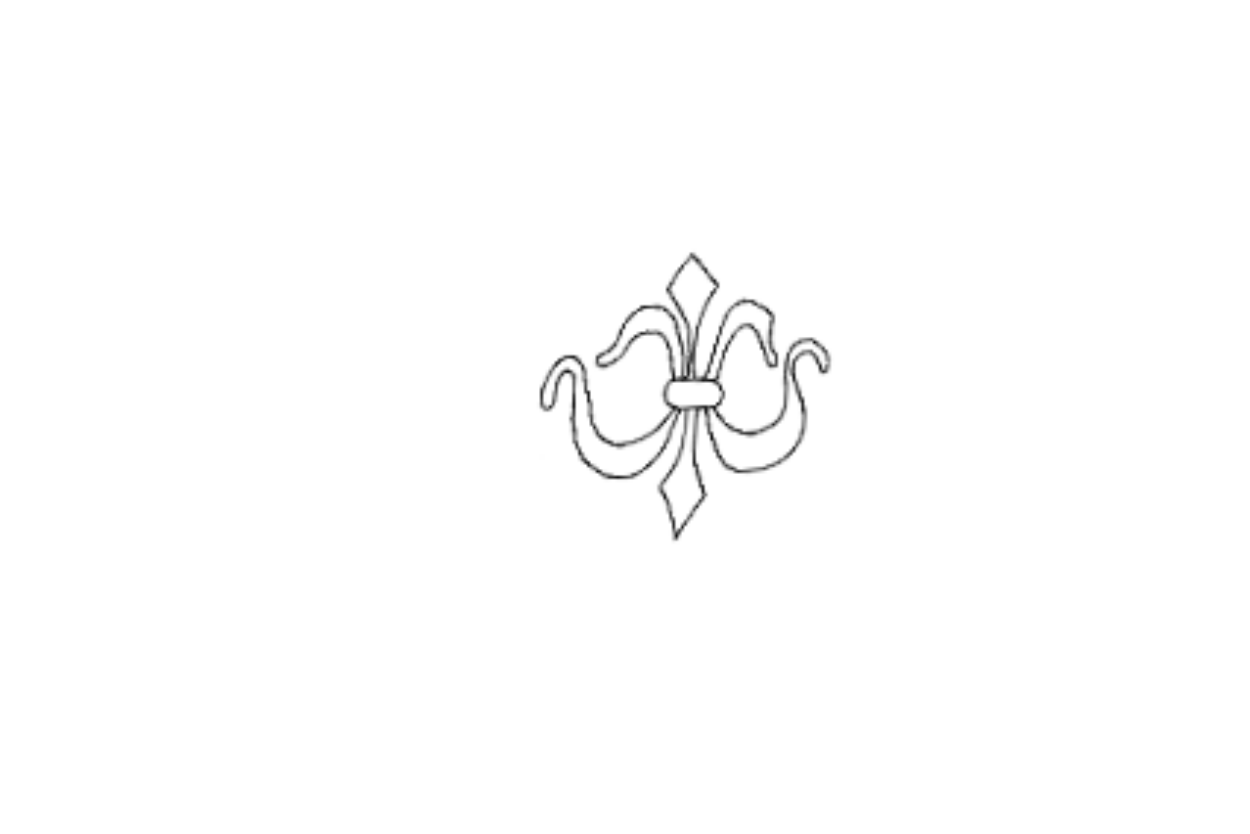}}%
            & \boxed{\includegraphics[width=0.1\linewidth, clip=true, trim=132pt 27pt 113pt 12pt]{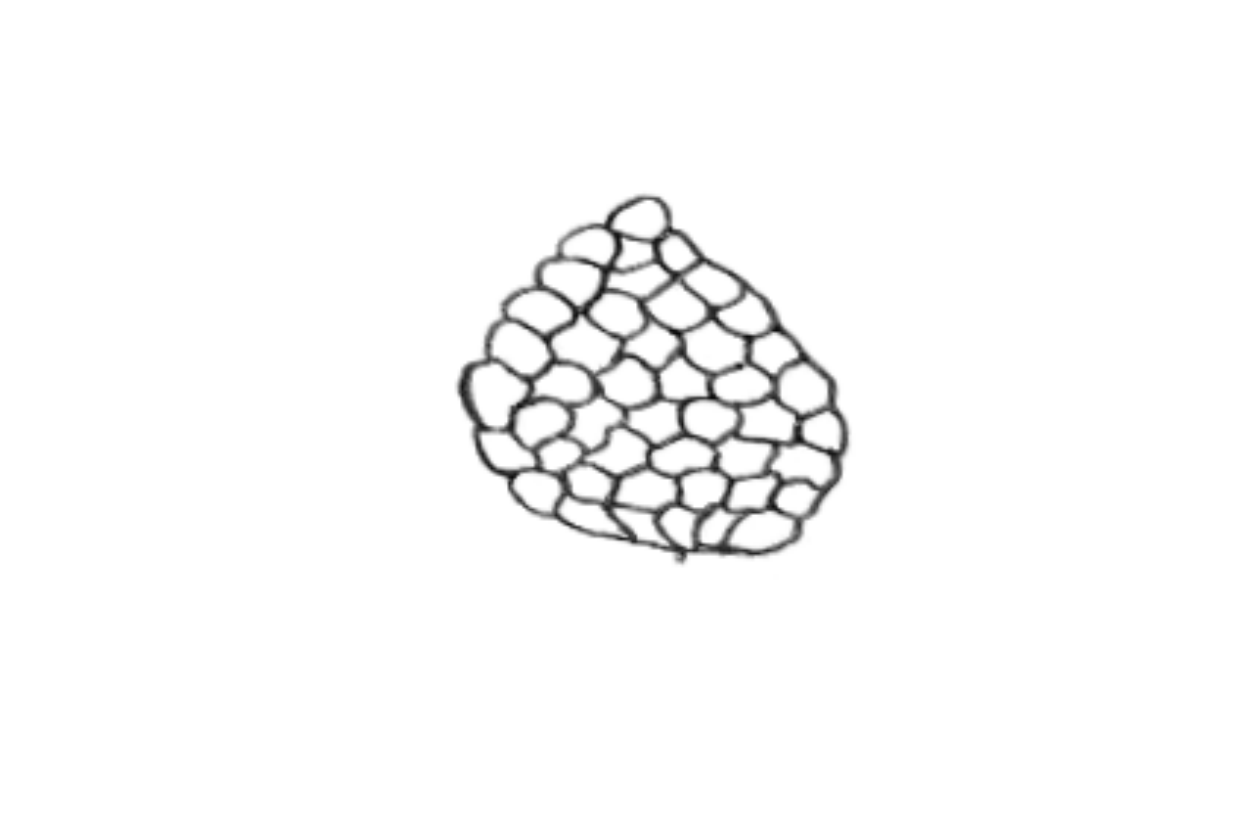}}%
            & \boxed{\includegraphics[width=0.1\linewidth, clip=true, trim=132pt 27pt 113pt 12pt]{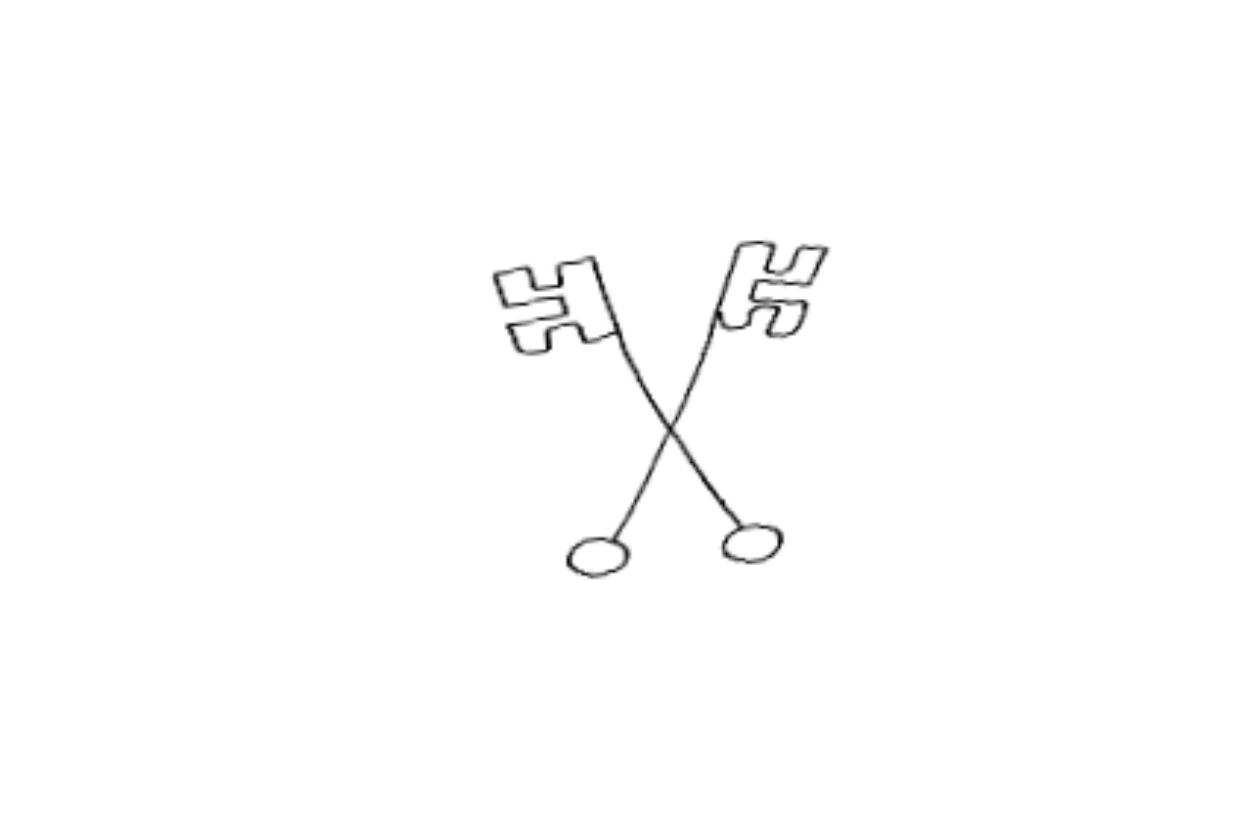}}%
            & \boxed{\includegraphics[width=0.1\linewidth, clip=true, trim=132pt 27pt 113pt 12pt]{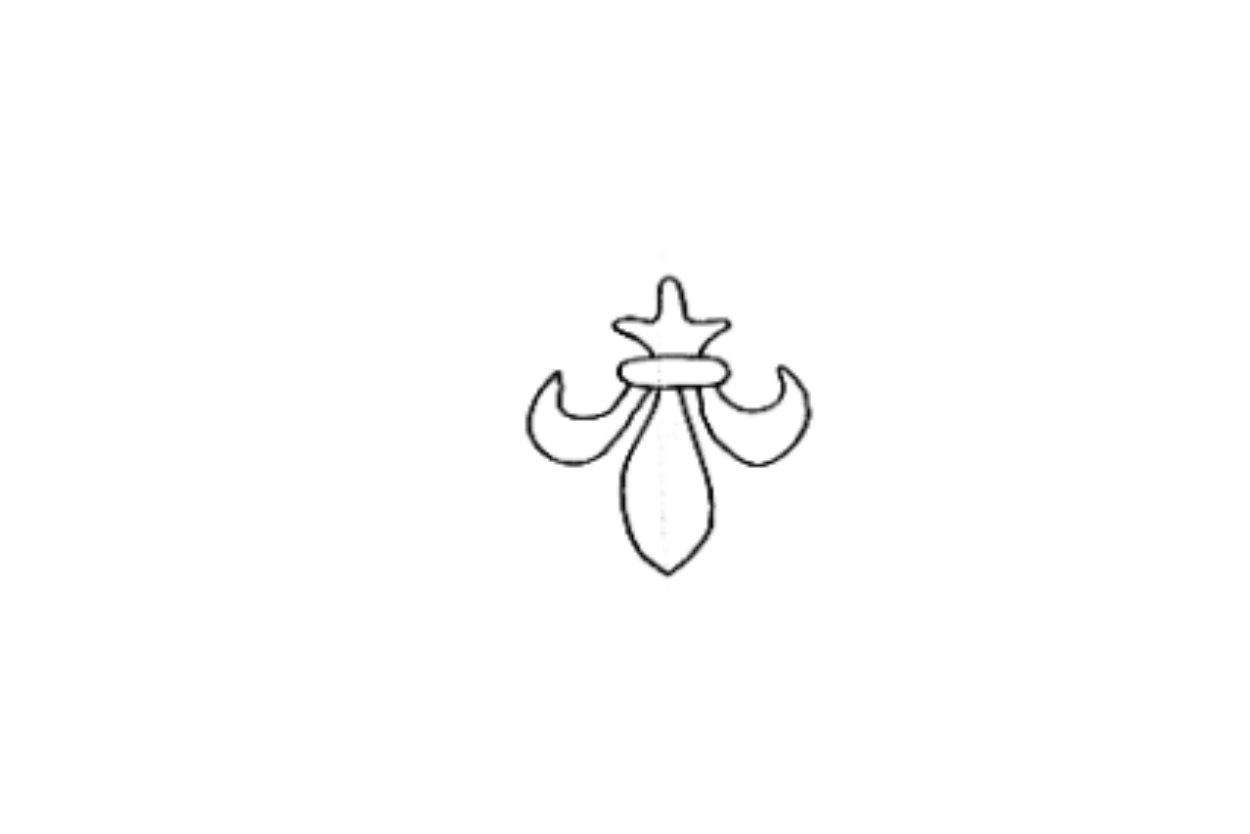}}%
            & \boxed{\includegraphics[width=0.1\linewidth, clip=true, trim=132pt 27pt 113pt 12pt]{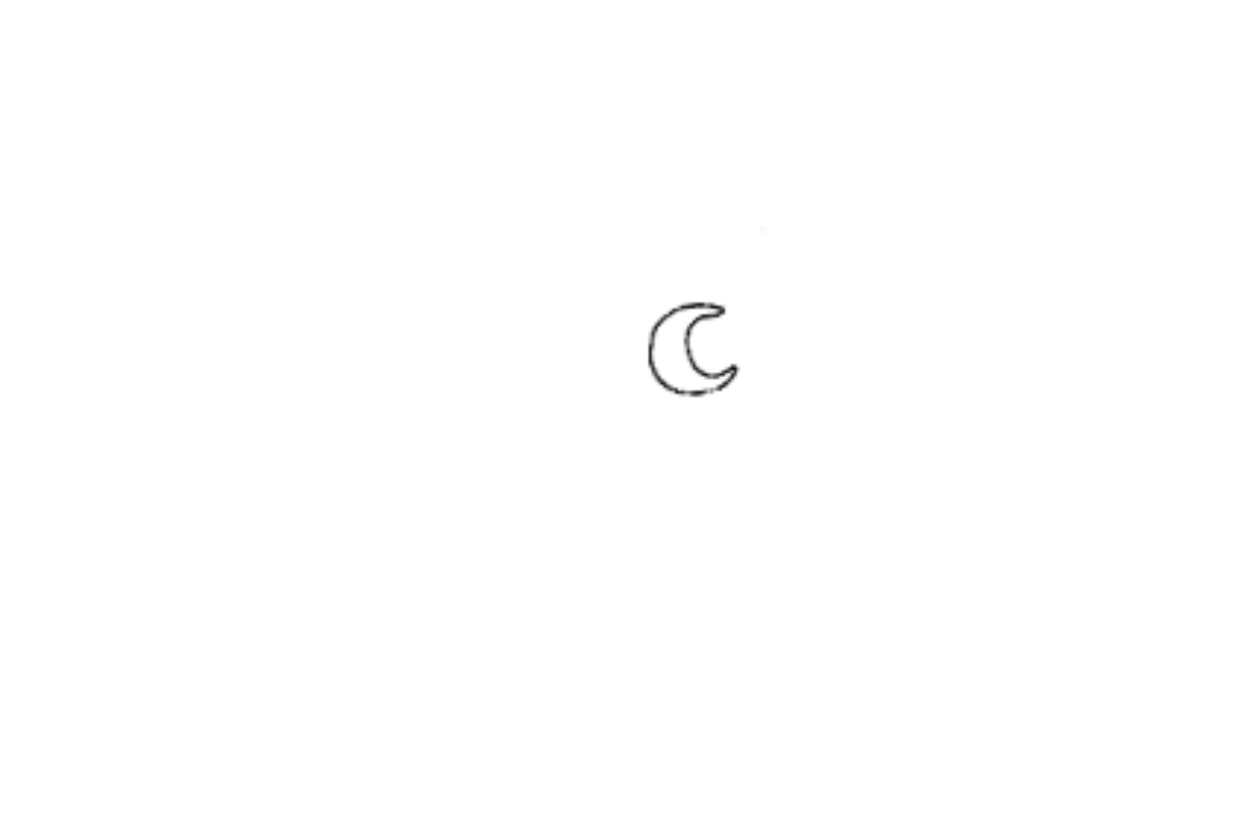}}%
            & \boxed{\includegraphics[width=0.1\linewidth, clip=true, trim=132pt 27pt 113pt 12pt]{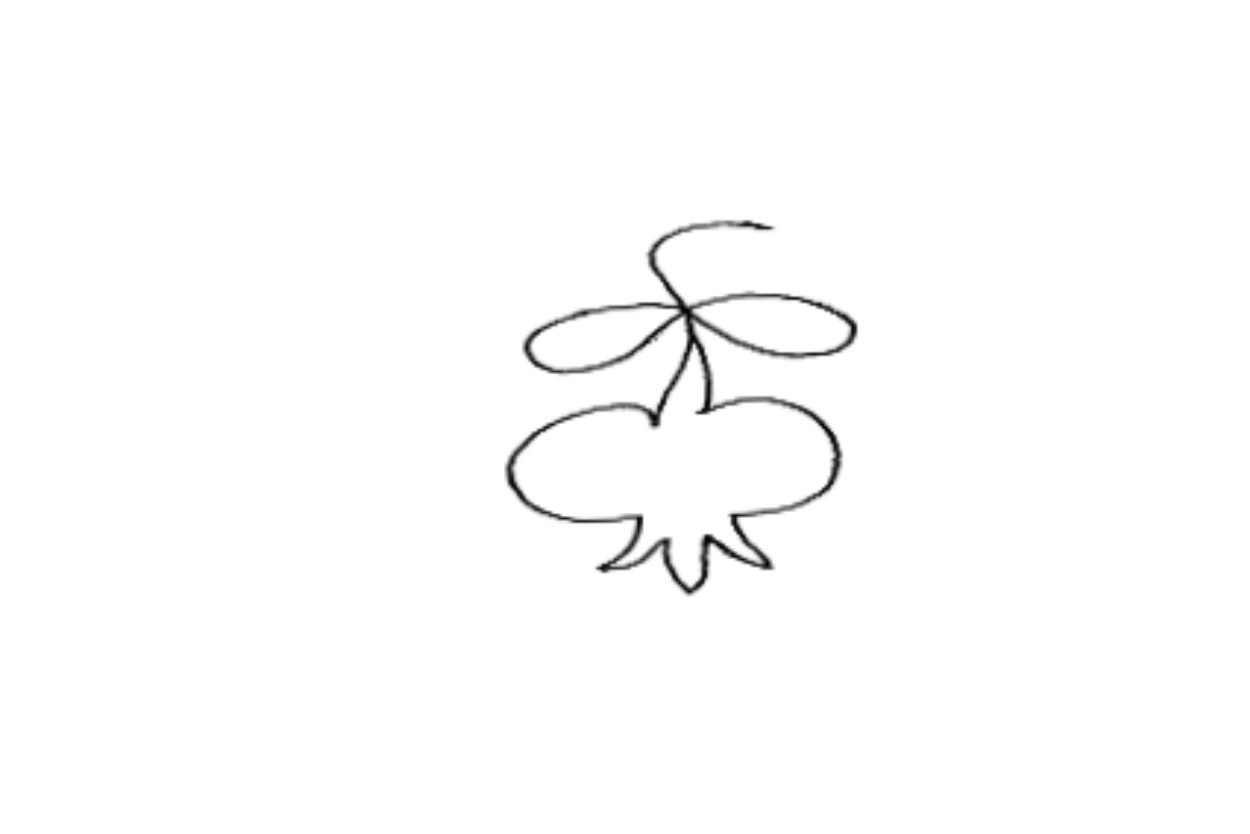}}%
            & \boxed{\includegraphics[width=0.1\linewidth, clip=true, trim=132pt 27pt 113pt 12pt]{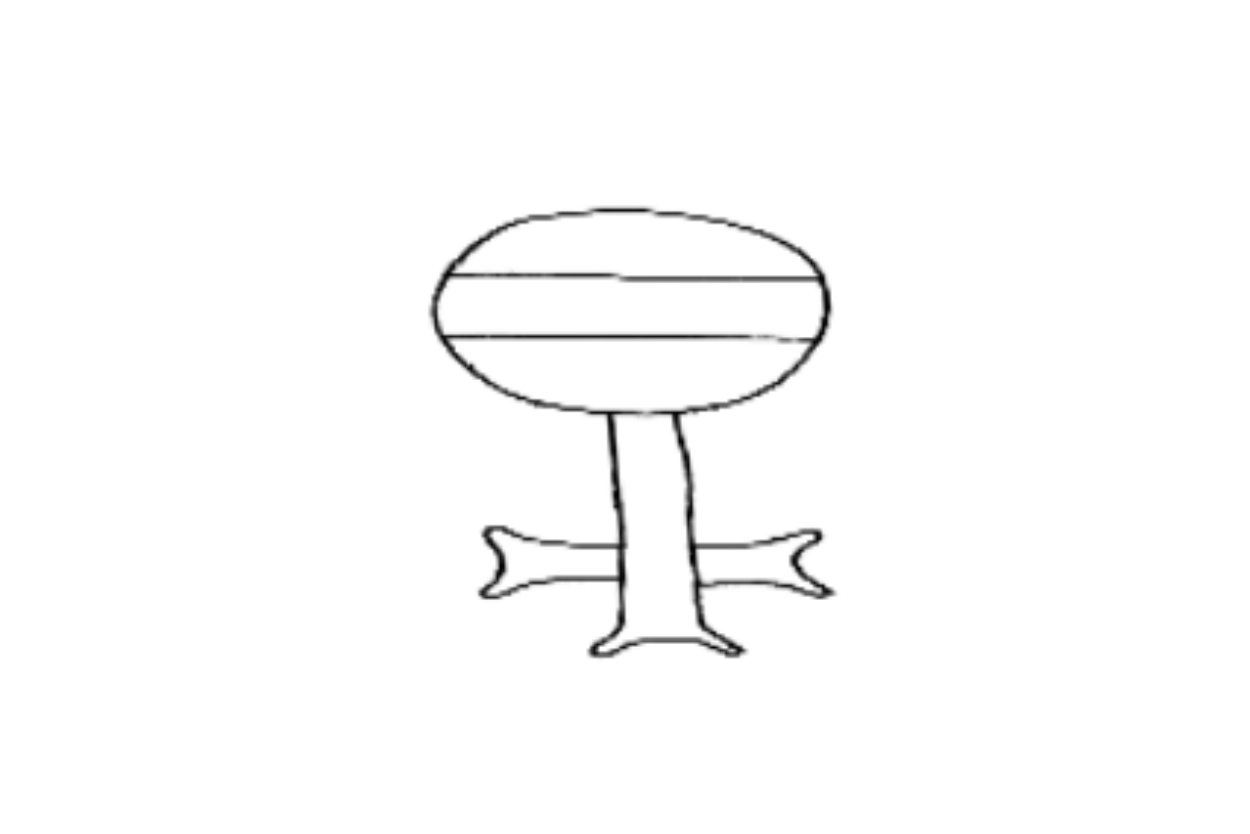}}%
            & \boxed{\includegraphics[width=0.1\linewidth, clip=true, trim=132pt 27pt 113pt 12pt]{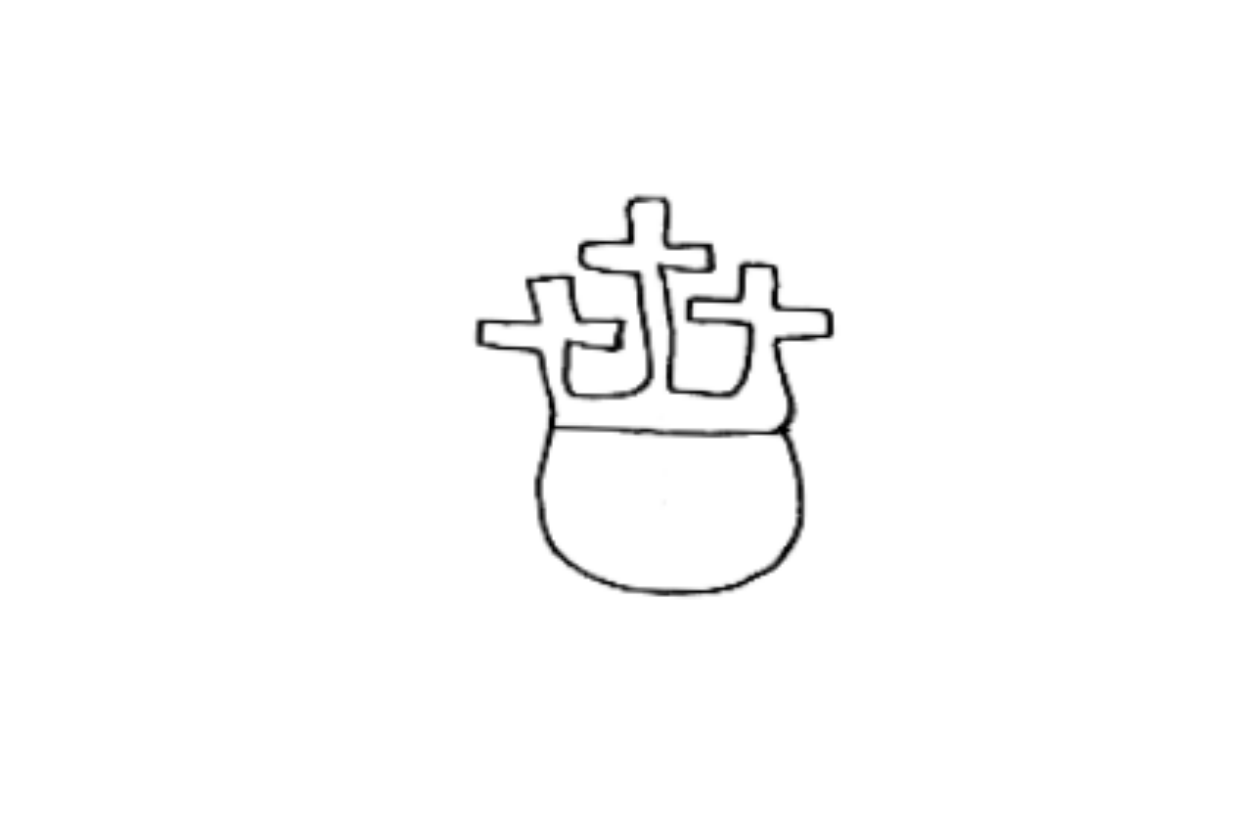}}%
            \\[1pt]%
            \boxed{\includegraphics[width=0.041\linewidth, clip=true, trim=132pt 27pt 113pt 12pt]{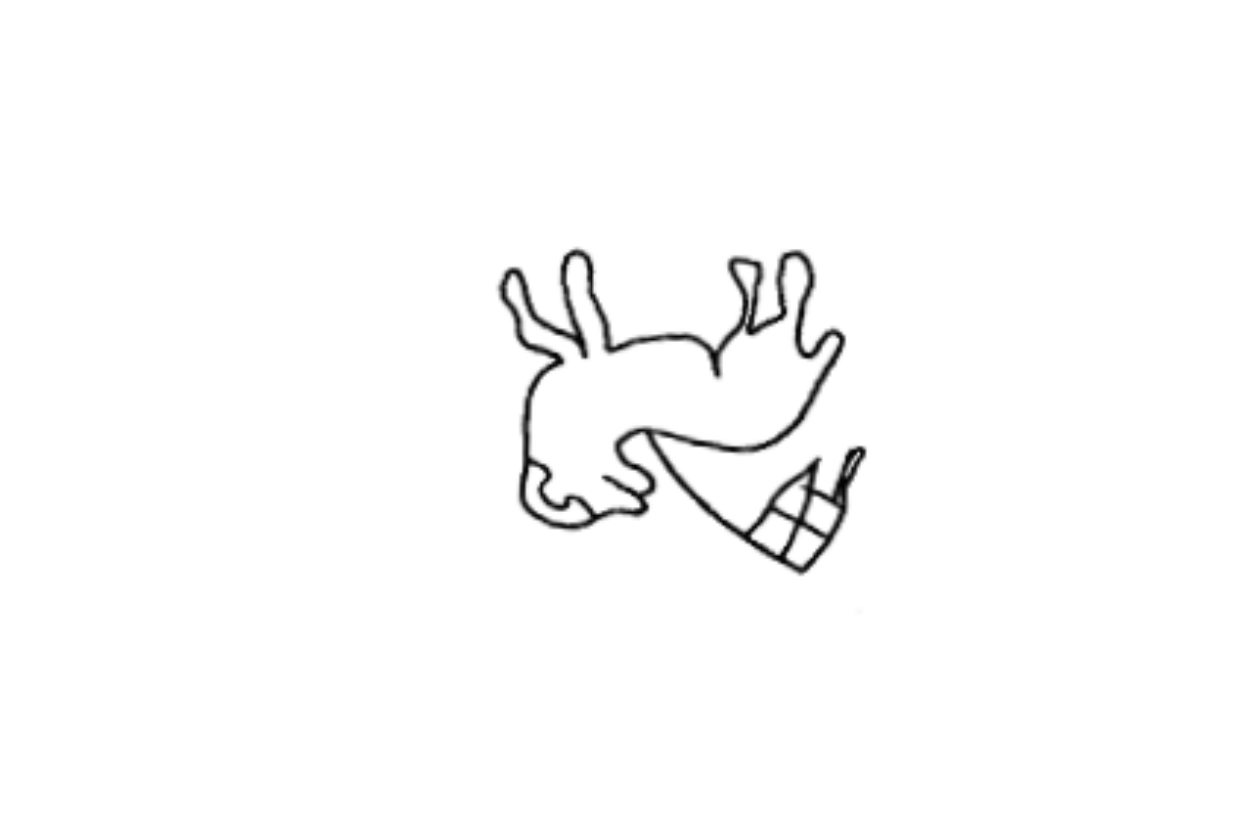}}%
            \hspace{1pt}%
            \boxed{\includegraphics[width=0.041\linewidth, clip=true, trim=132pt 27pt 113pt 12pt]{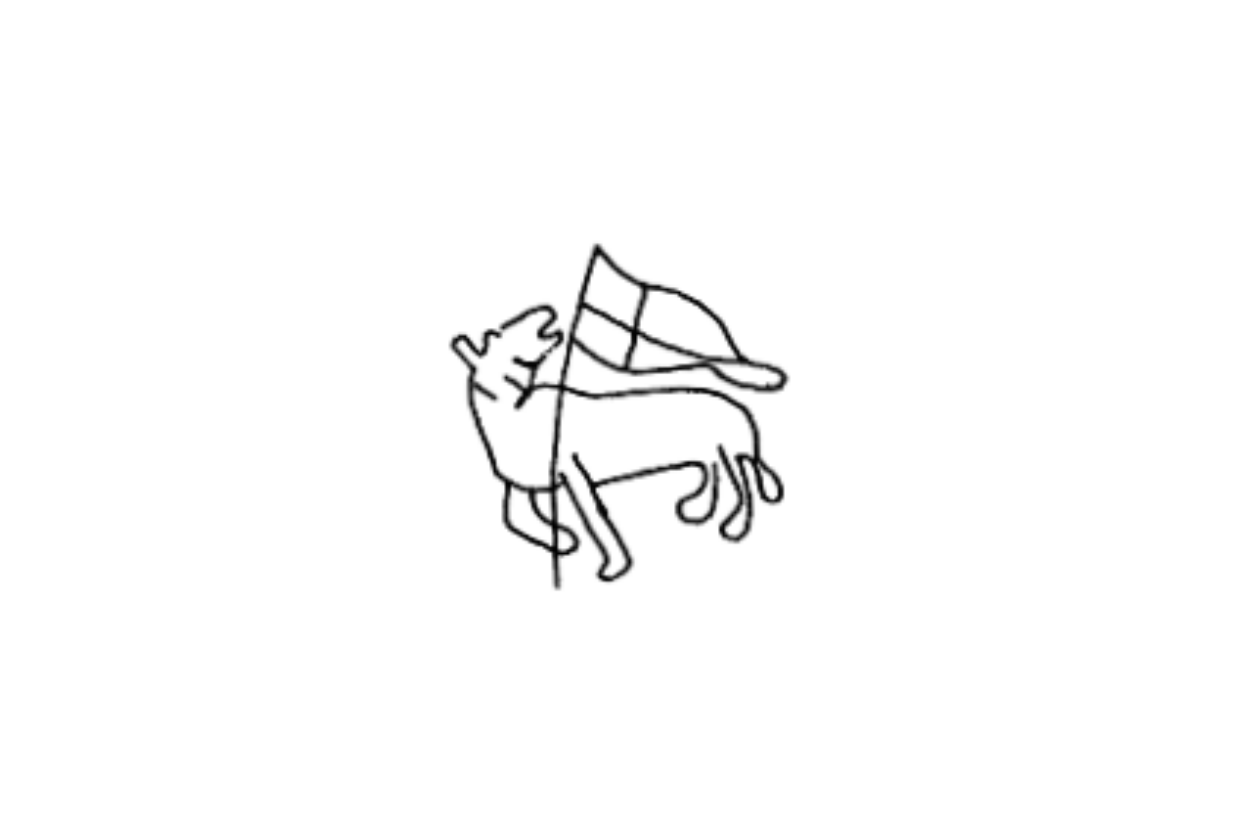}}%
            & \boxed{\includegraphics[width=0.041\linewidth, clip=true, trim=132pt 27pt 113pt 12pt]{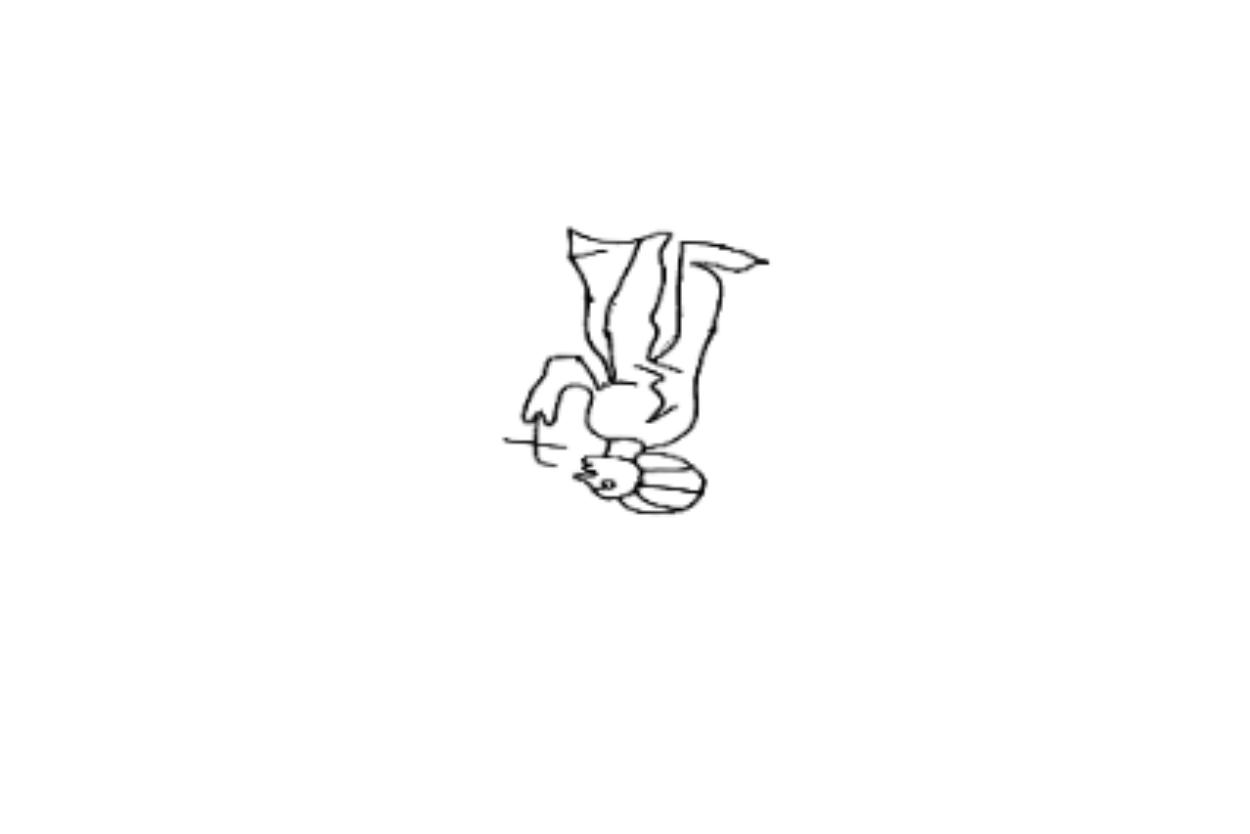}}%
            \hspace{1pt}%
            \boxed{\includegraphics[width=0.041\linewidth, clip=true, trim=132pt 27pt 113pt 12pt]{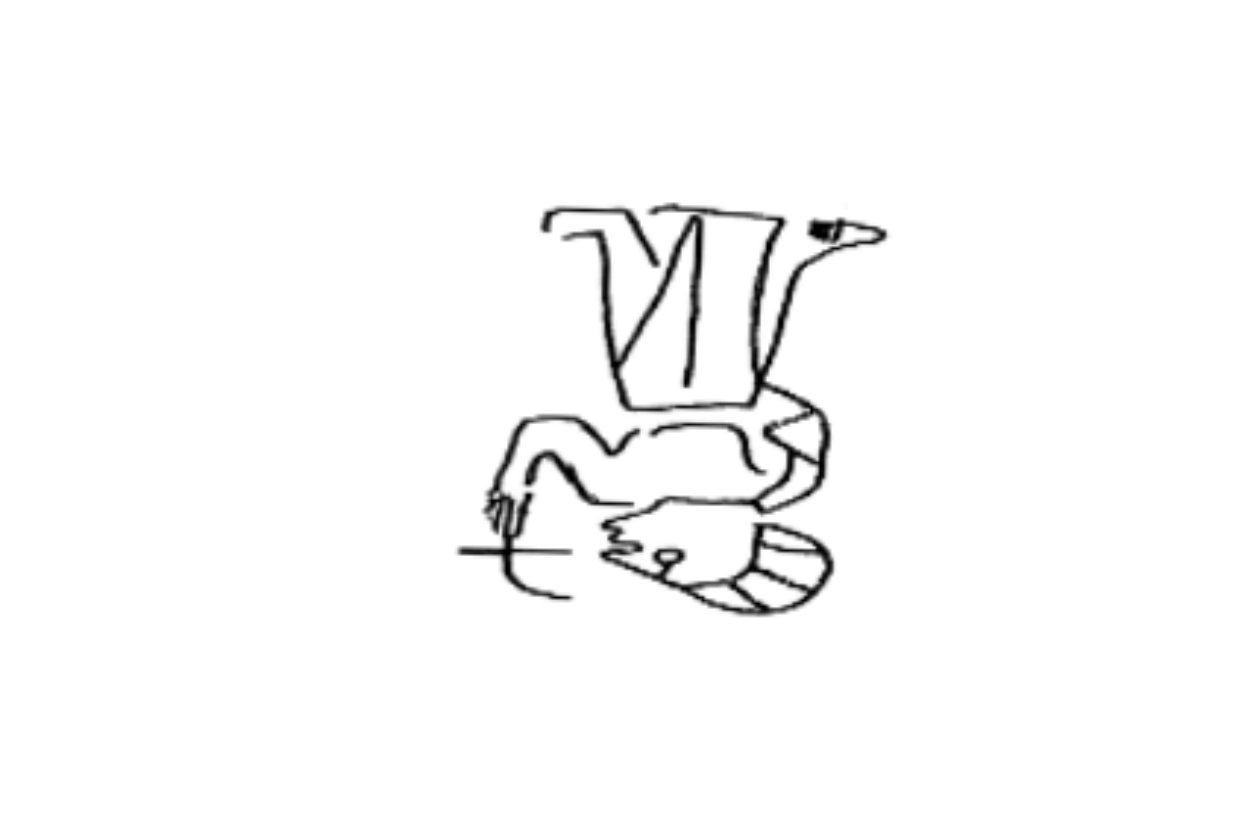}}%
            & \boxed{\includegraphics[width=0.041\linewidth, clip=true, trim=132pt 27pt 113pt 12pt]{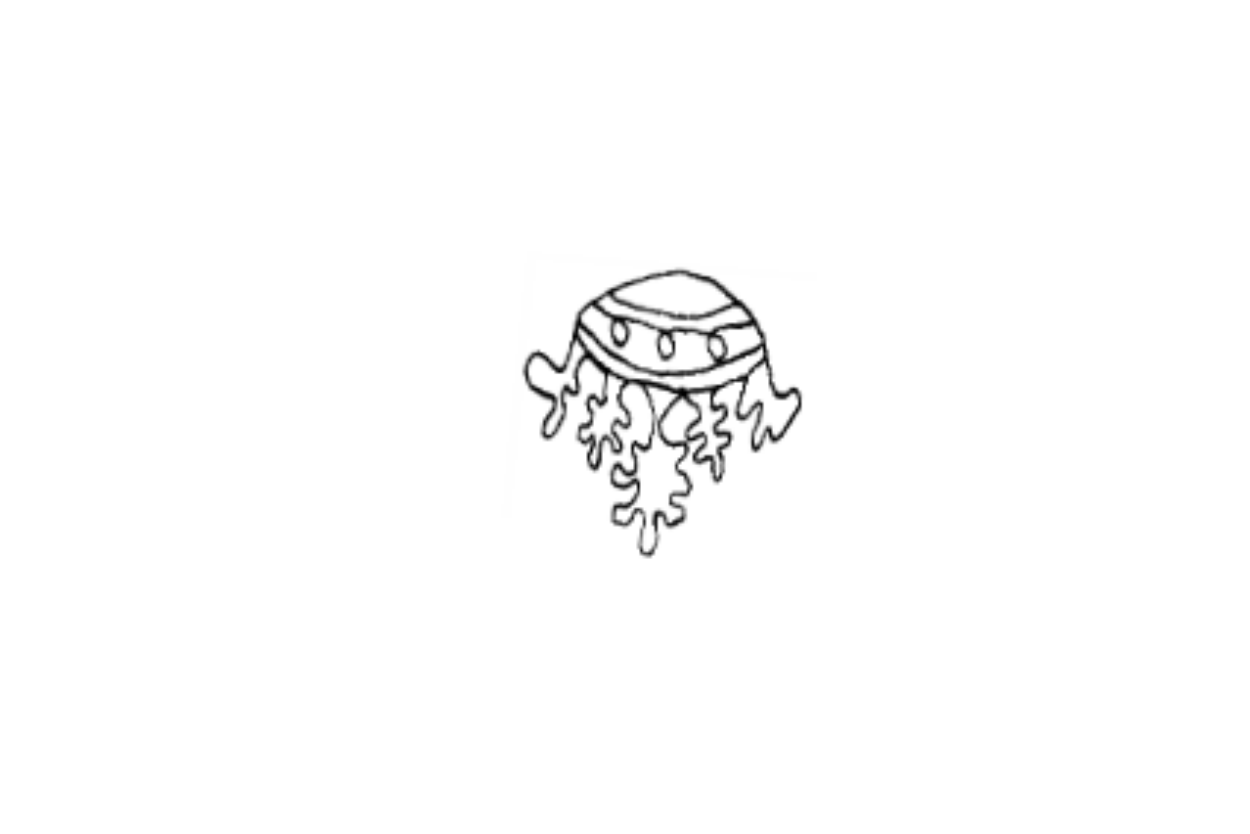}}%
            \hspace{1pt}%
            \boxed{\includegraphics[width=0.041\linewidth, clip=true, trim=132pt 27pt 113pt 12pt]{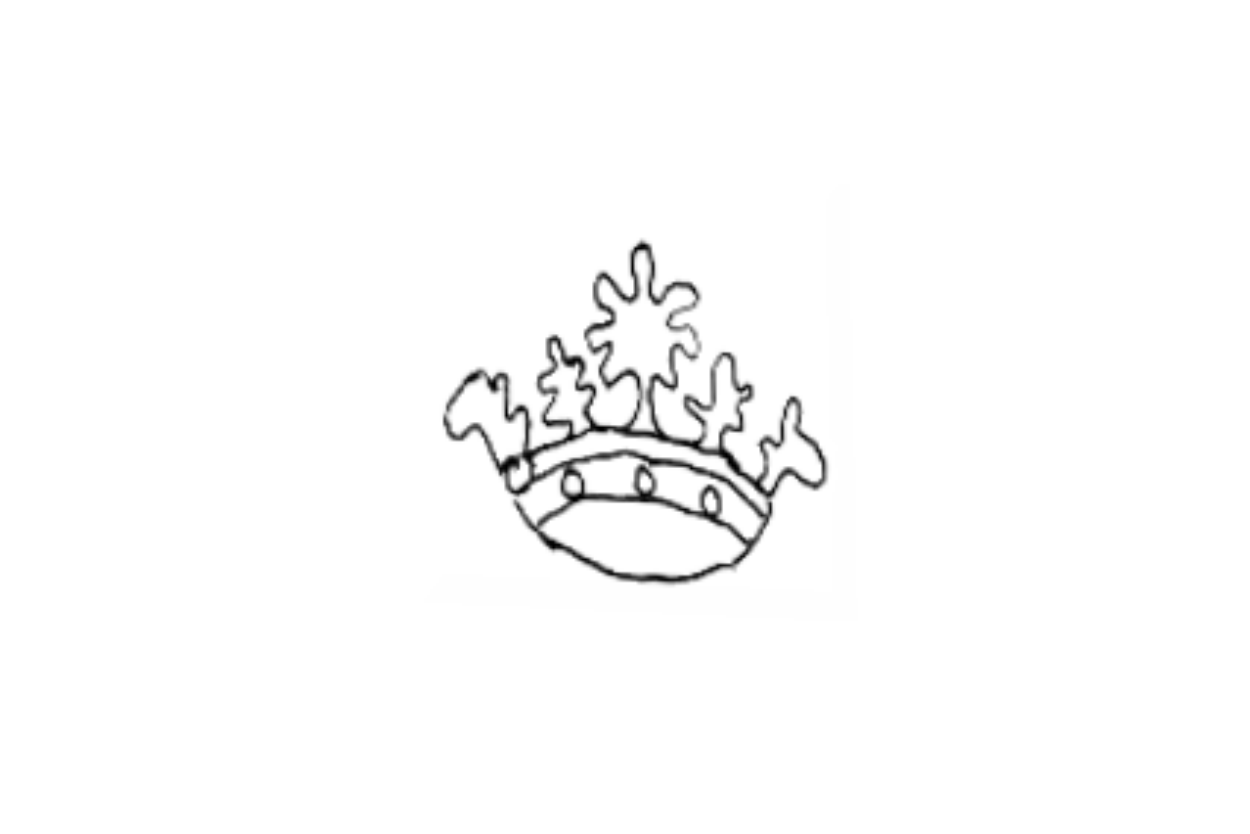}}%
            & \boxed{\includegraphics[width=0.041\linewidth, clip=true, trim=132pt 27pt 113pt 12pt]{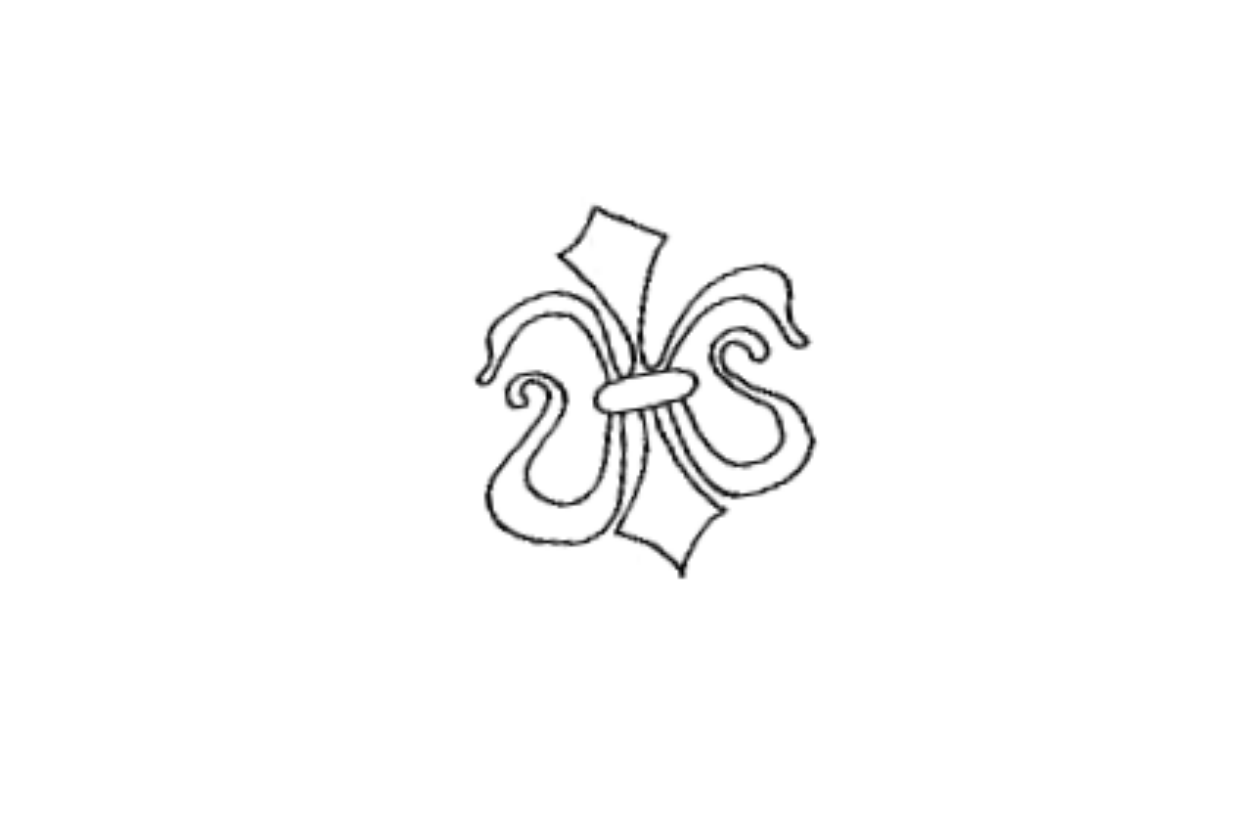}}%
            \hspace{1pt}%
            \boxed{\includegraphics[width=0.041\linewidth, clip=true, trim=132pt 27pt 113pt 12pt]{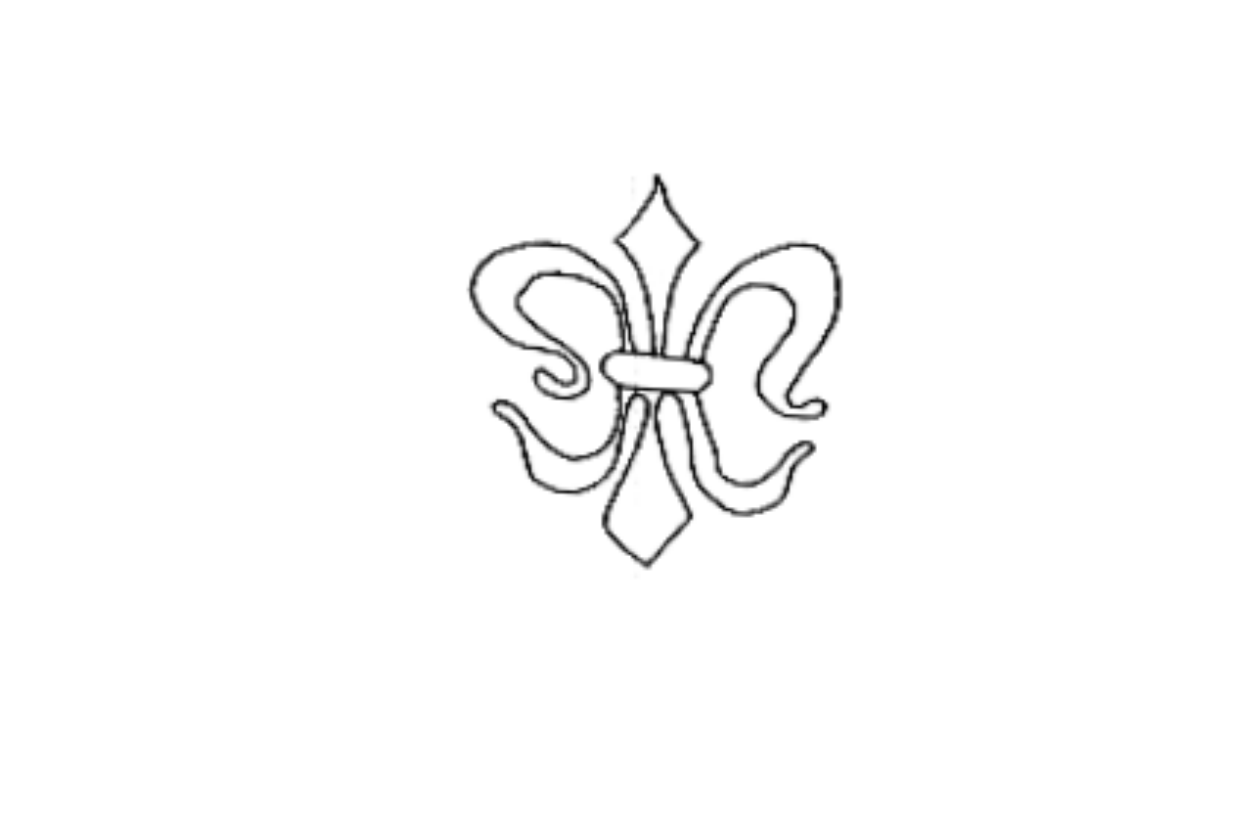}}%
            & \boxed{\includegraphics[width=0.041\linewidth, clip=true, trim=132pt 27pt 113pt 12pt]{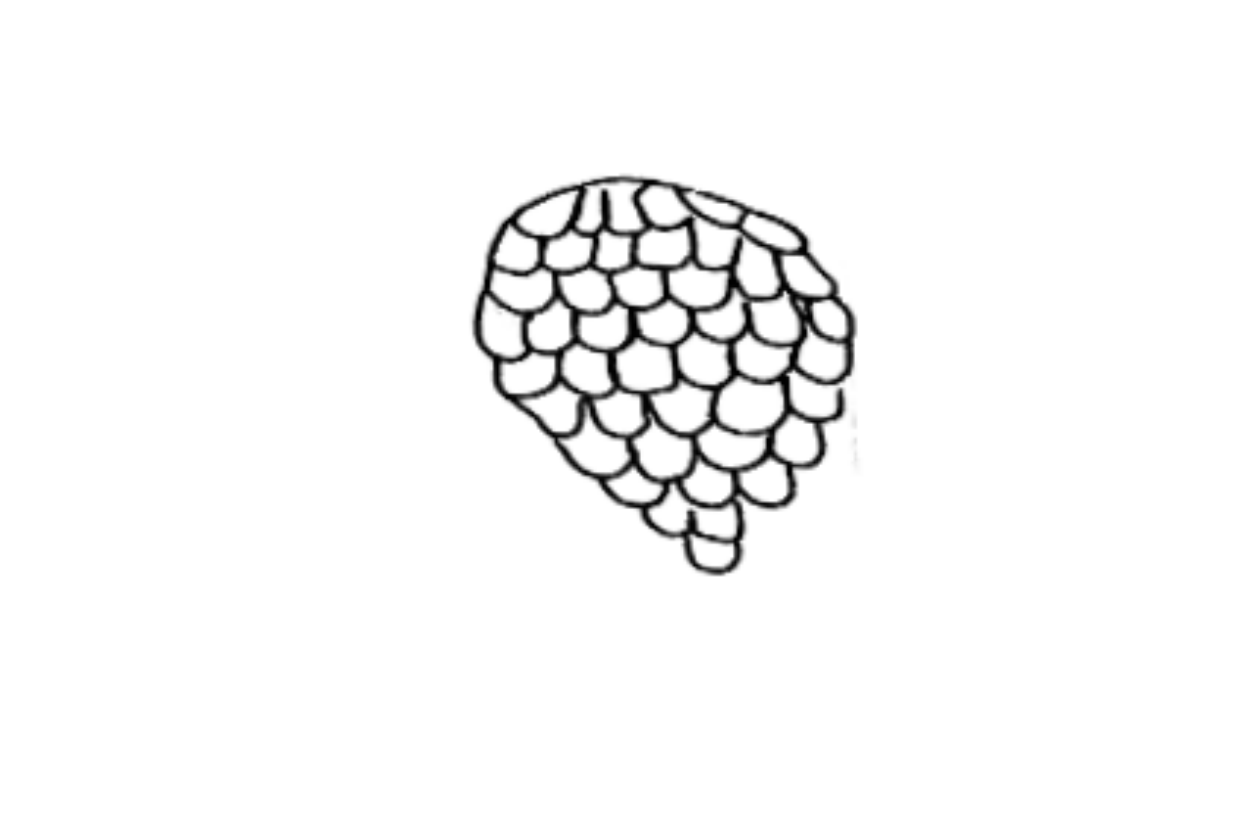}}%
            \hspace{1pt}%
            \boxed{\includegraphics[width=0.041\linewidth, clip=true, trim=132pt 27pt 113pt 12pt]{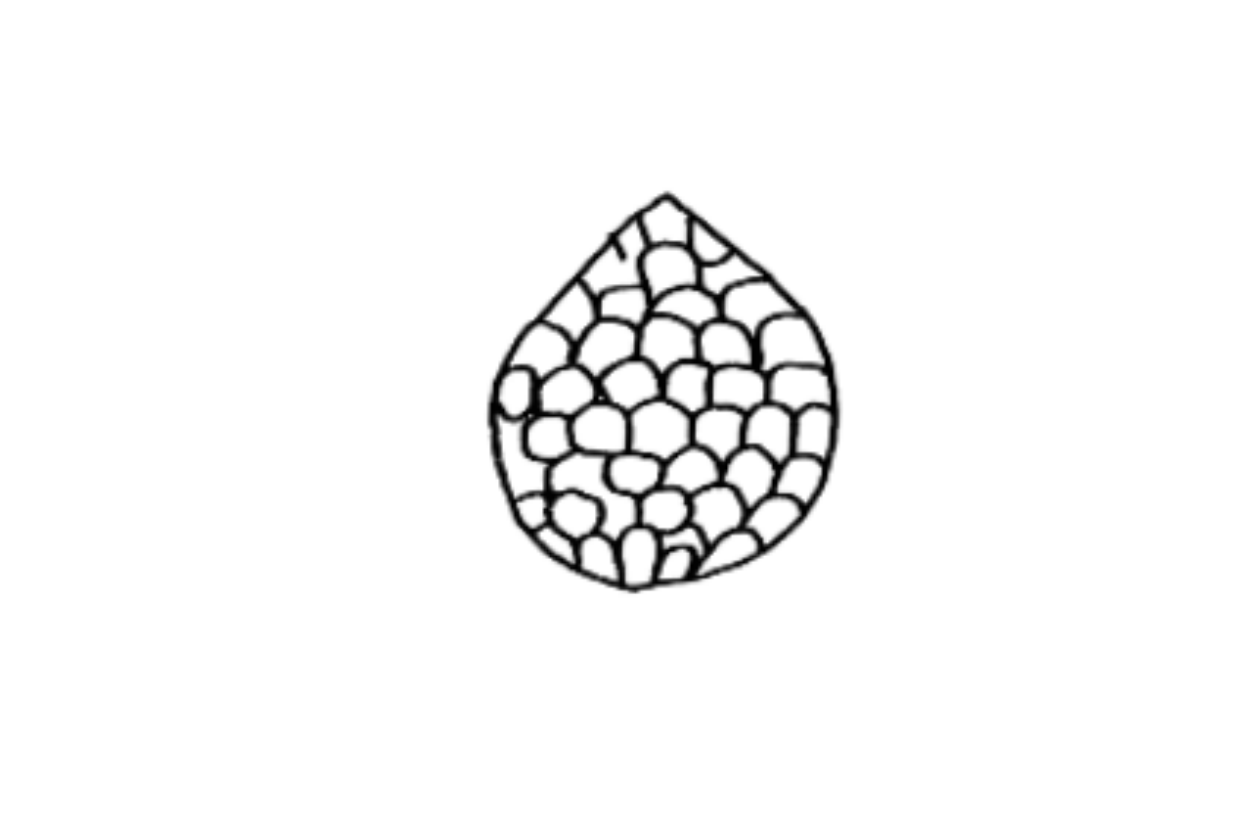}}%
            & \boxed{\includegraphics[width=0.041\linewidth, clip=true, trim=132pt 27pt 113pt 12pt]{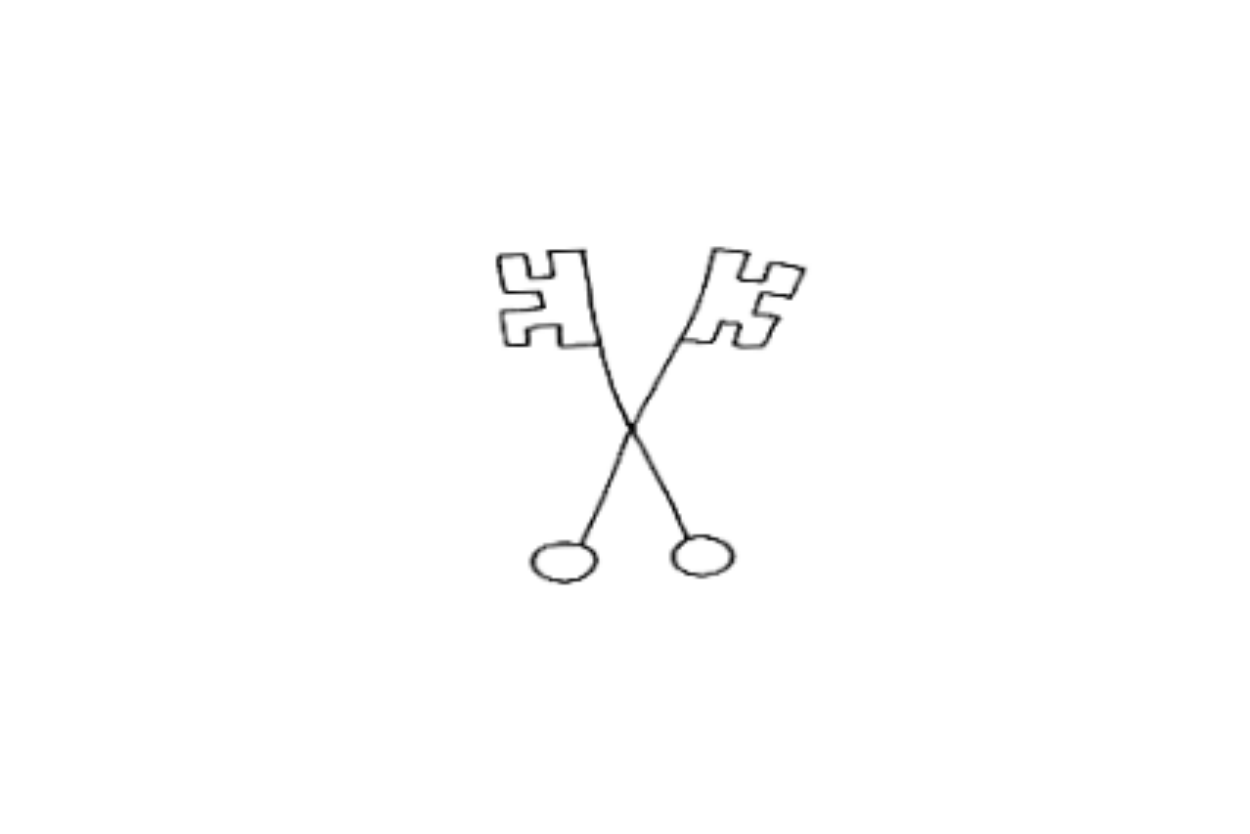}}%
            \hspace{1pt}%
            \boxed{\includegraphics[width=0.041\linewidth, clip=true, trim=132pt 27pt 113pt 12pt]{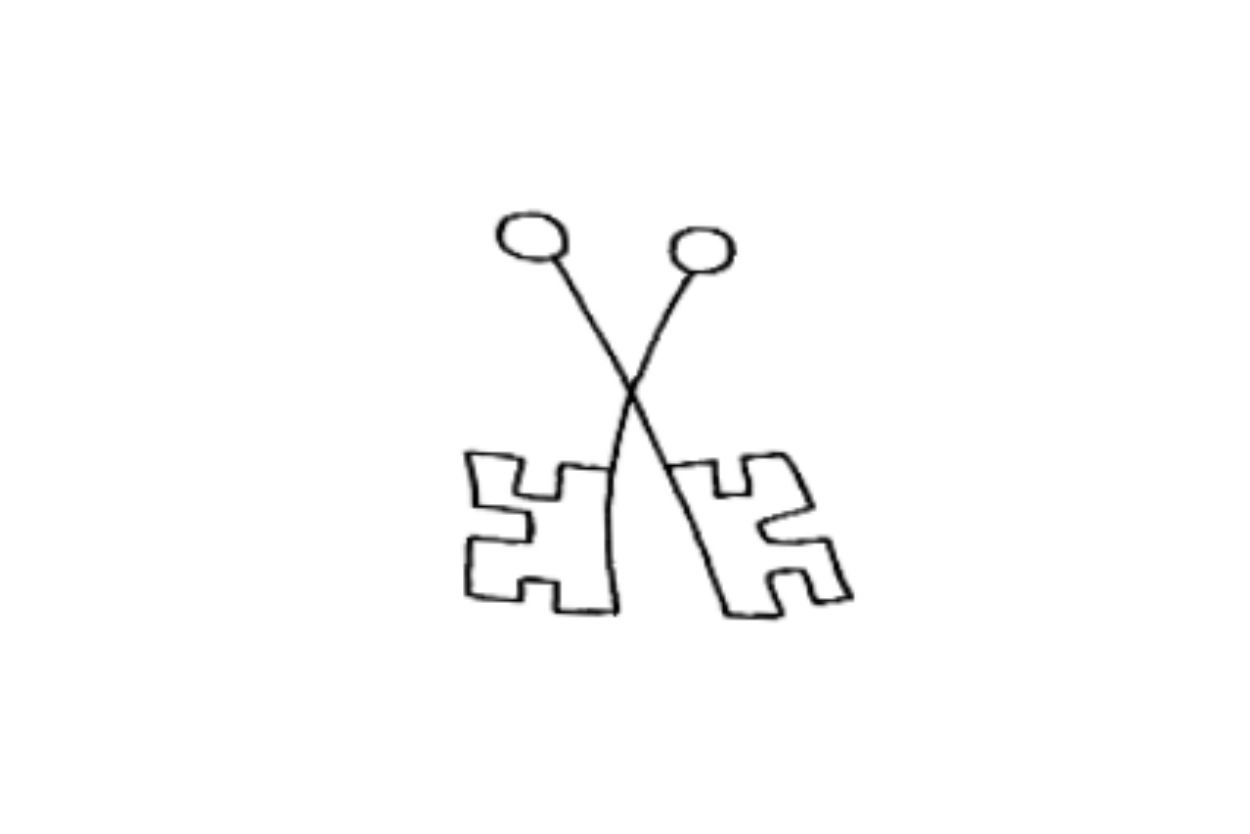}}%
            & \boxed{\includegraphics[width=0.041\linewidth, clip=true, trim=132pt 27pt 113pt 12pt]{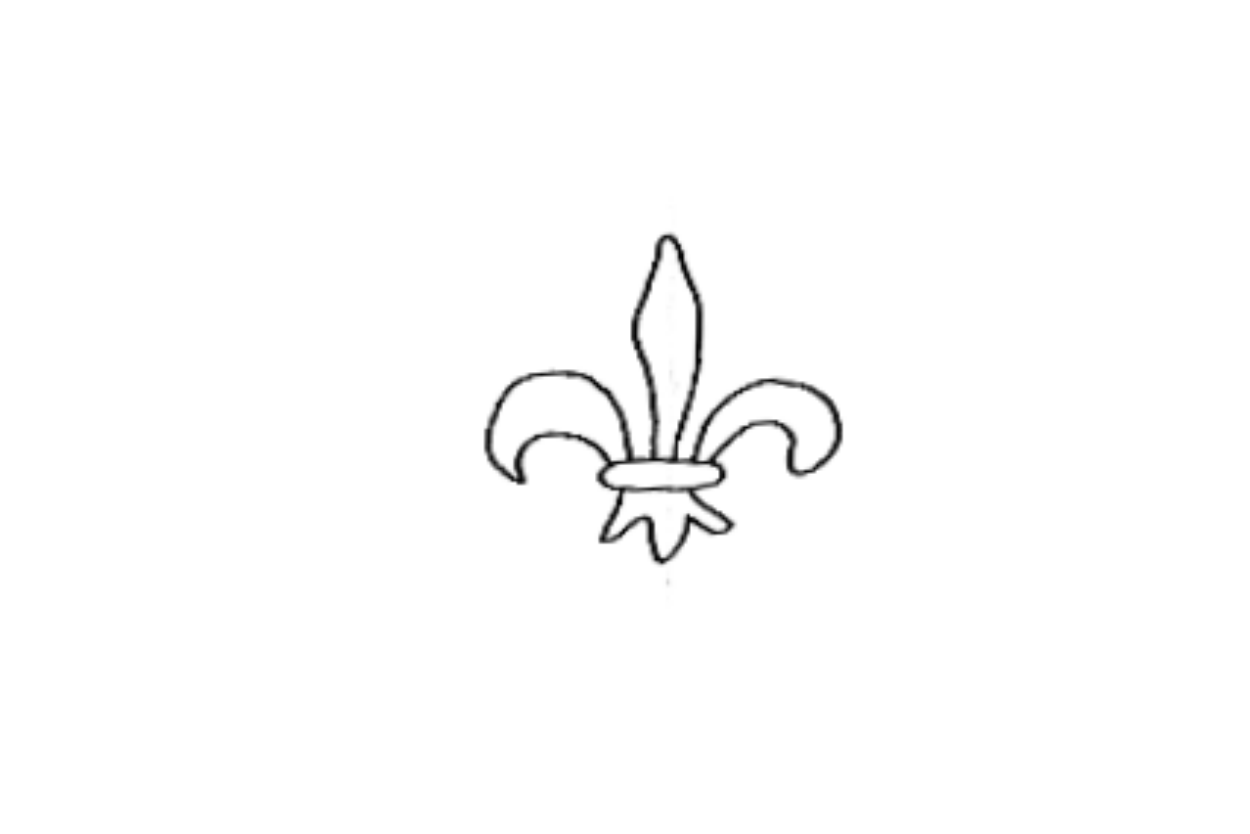}}%
            \hspace{1pt}%
            \boxed{\includegraphics[width=0.041\linewidth, clip=true, trim=132pt 27pt 113pt 12pt]{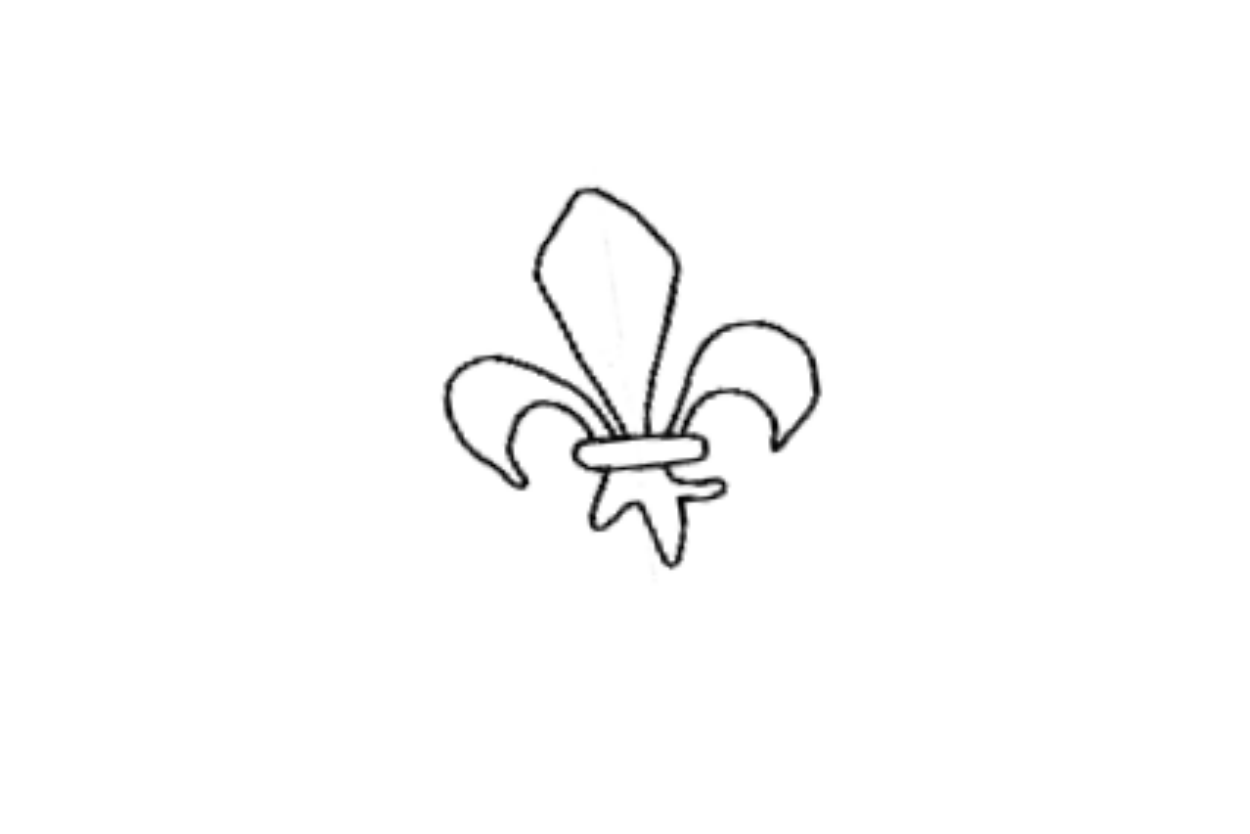}}%
            & \boxed{\includegraphics[width=0.041\linewidth, clip=true, trim=132pt 27pt 113pt 12pt]{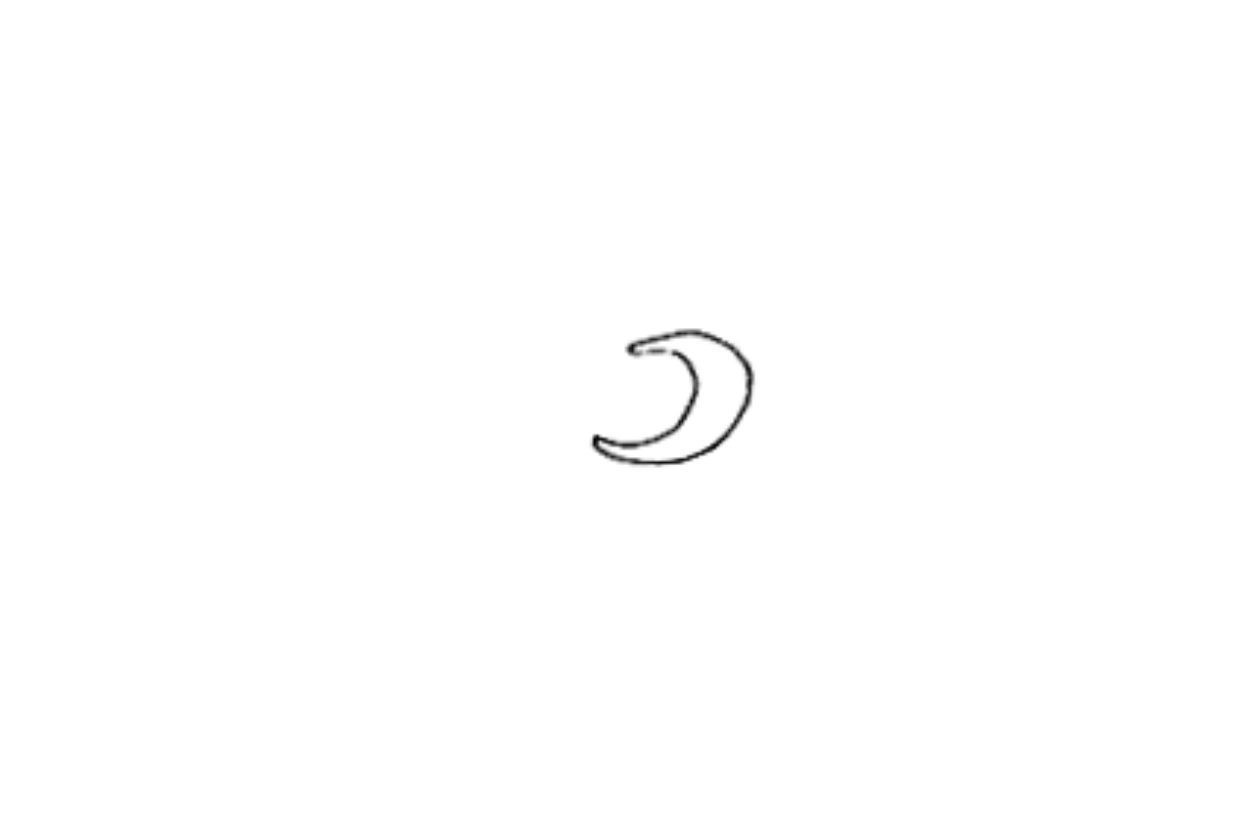}}%
            \hspace{1pt}%
            \boxed{\includegraphics[width=0.041\linewidth, clip=true, trim=132pt 27pt 113pt 12pt]{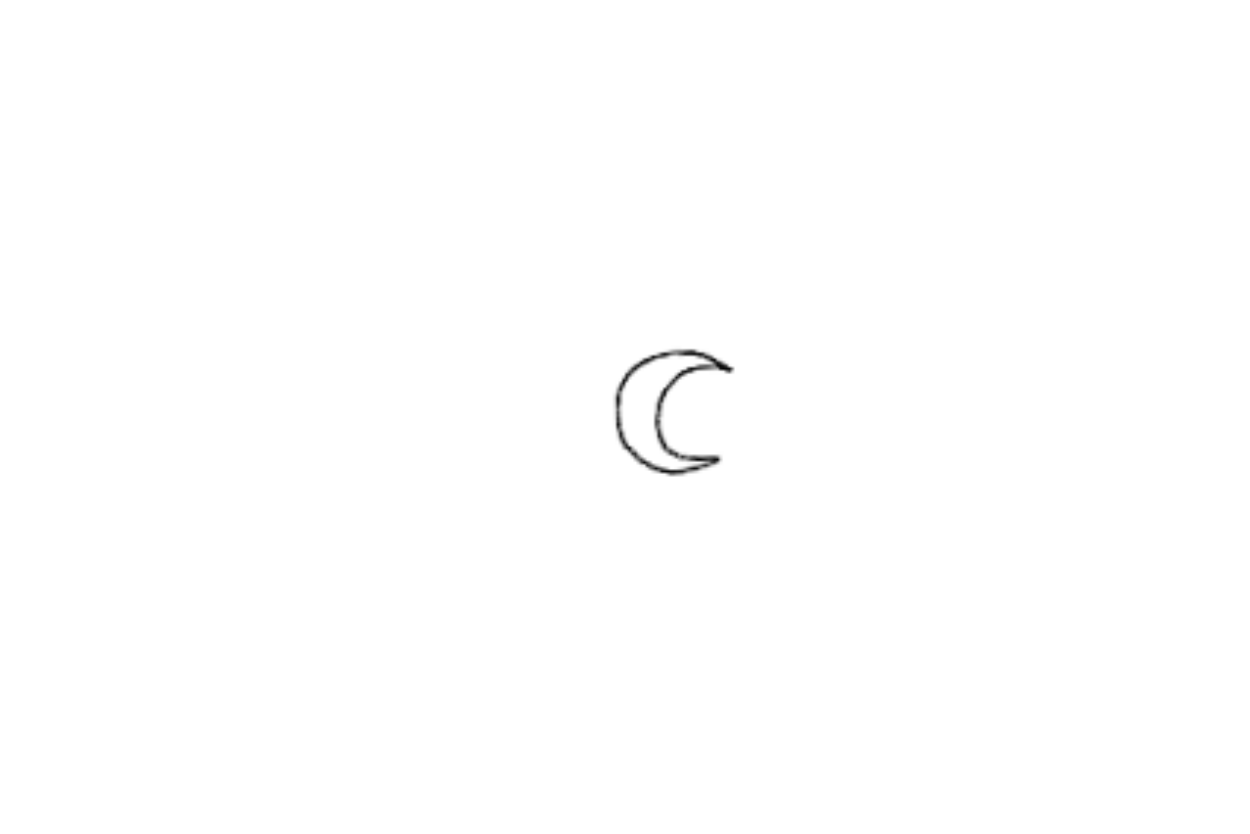}}%
            & \boxed{\includegraphics[width=0.041\linewidth, clip=true, trim=132pt 27pt 113pt 12pt]{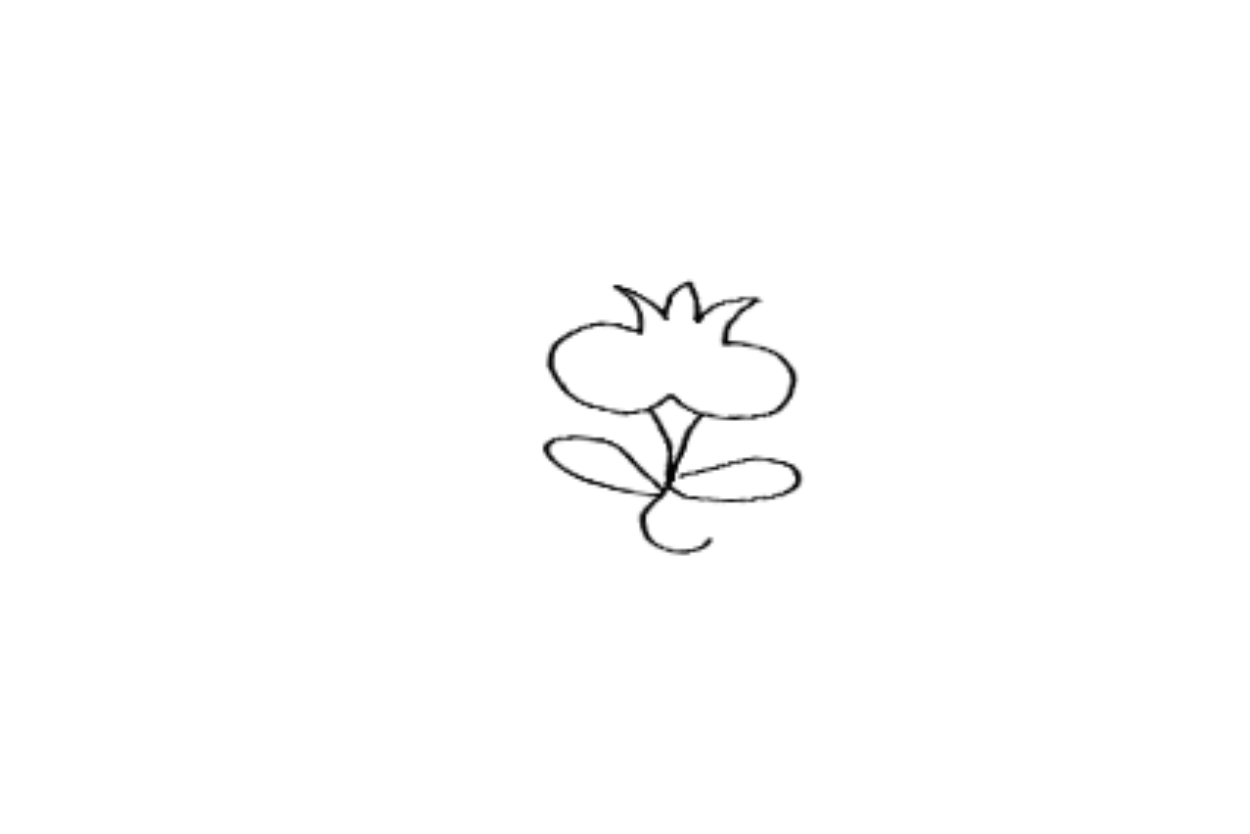}}%
            \hspace{1pt}%
            \boxed{\includegraphics[width=0.041\linewidth, clip=true, trim=132pt 27pt 113pt 12pt]{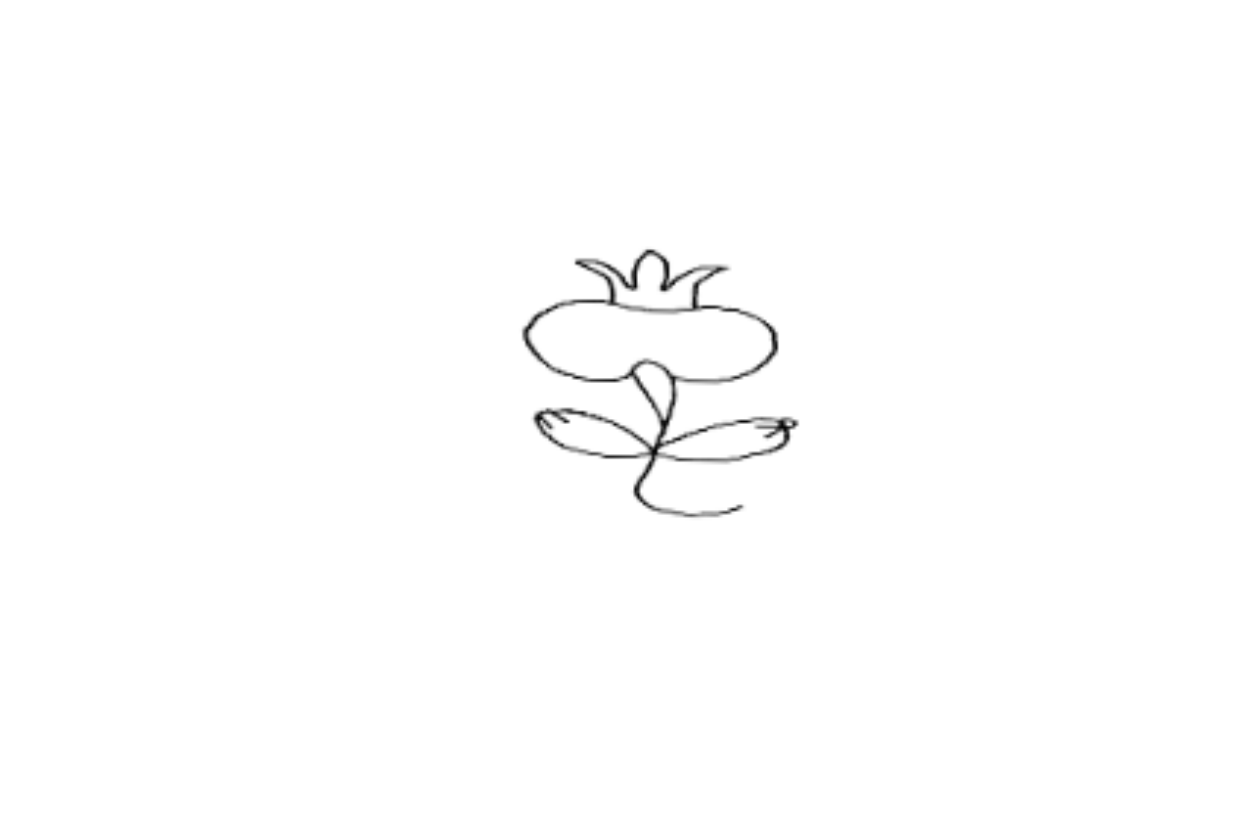}}%
            & \boxed{\includegraphics[width=0.041\linewidth, clip=true, trim=132pt 27pt 113pt 12pt]{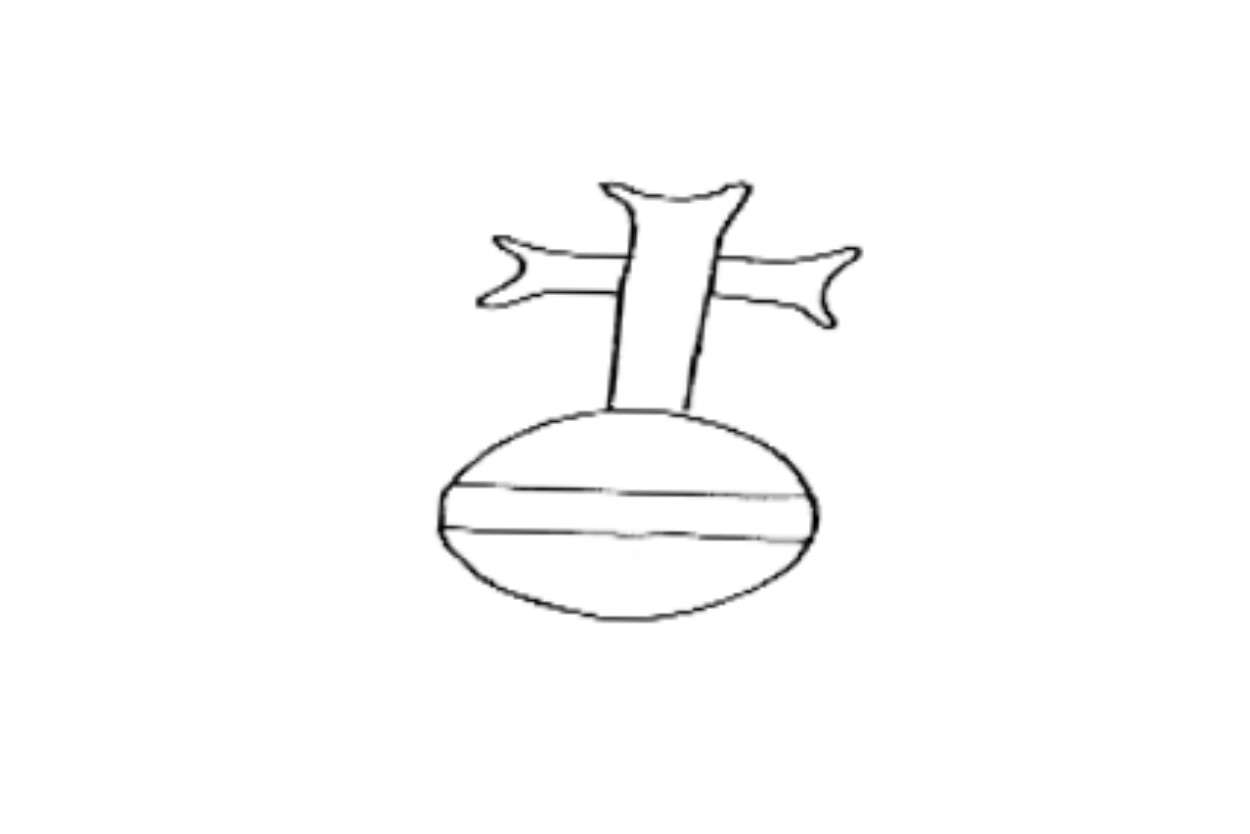}}%
            \hspace{1pt}%
            \boxed{\includegraphics[width=0.041\linewidth, clip=true, trim=132pt 27pt 113pt 12pt]{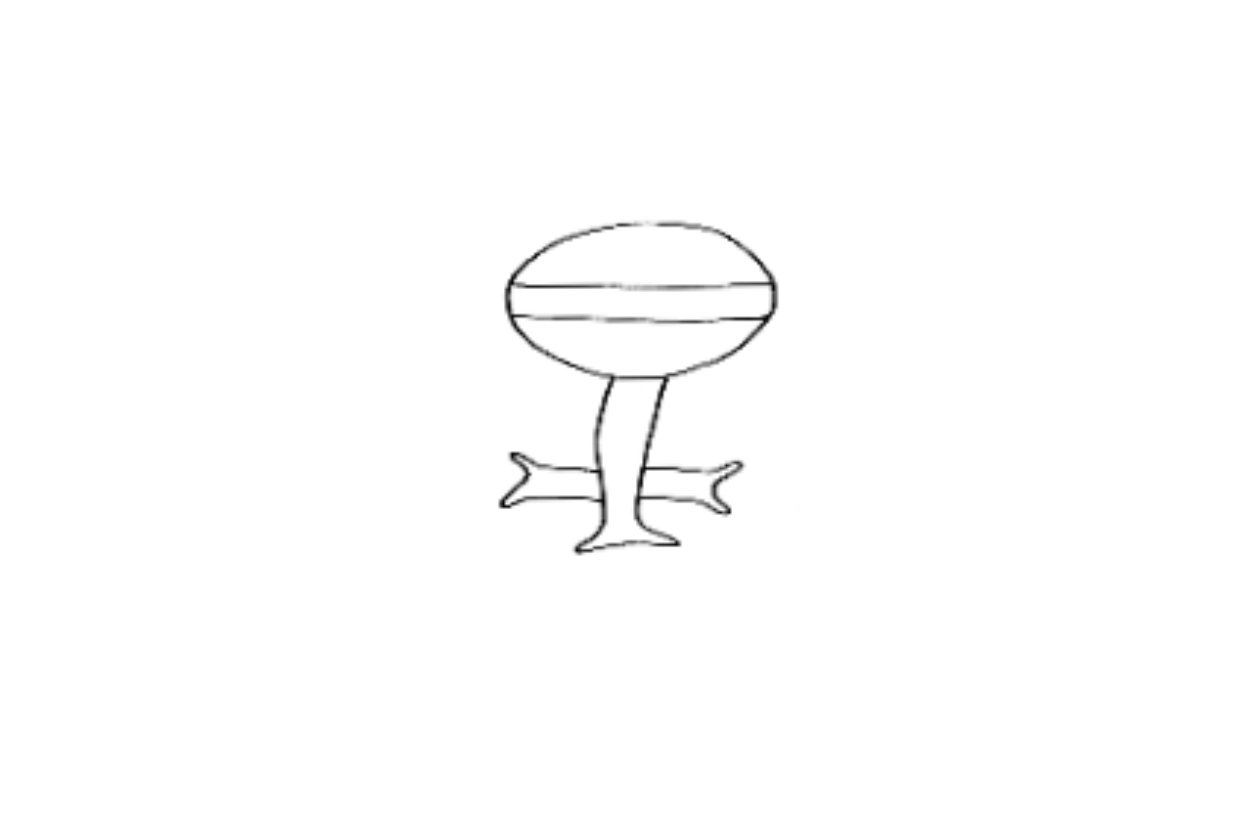}}%
            & \boxed{\includegraphics[width=0.041\linewidth, clip=true, trim=132pt 27pt 113pt 12pt]{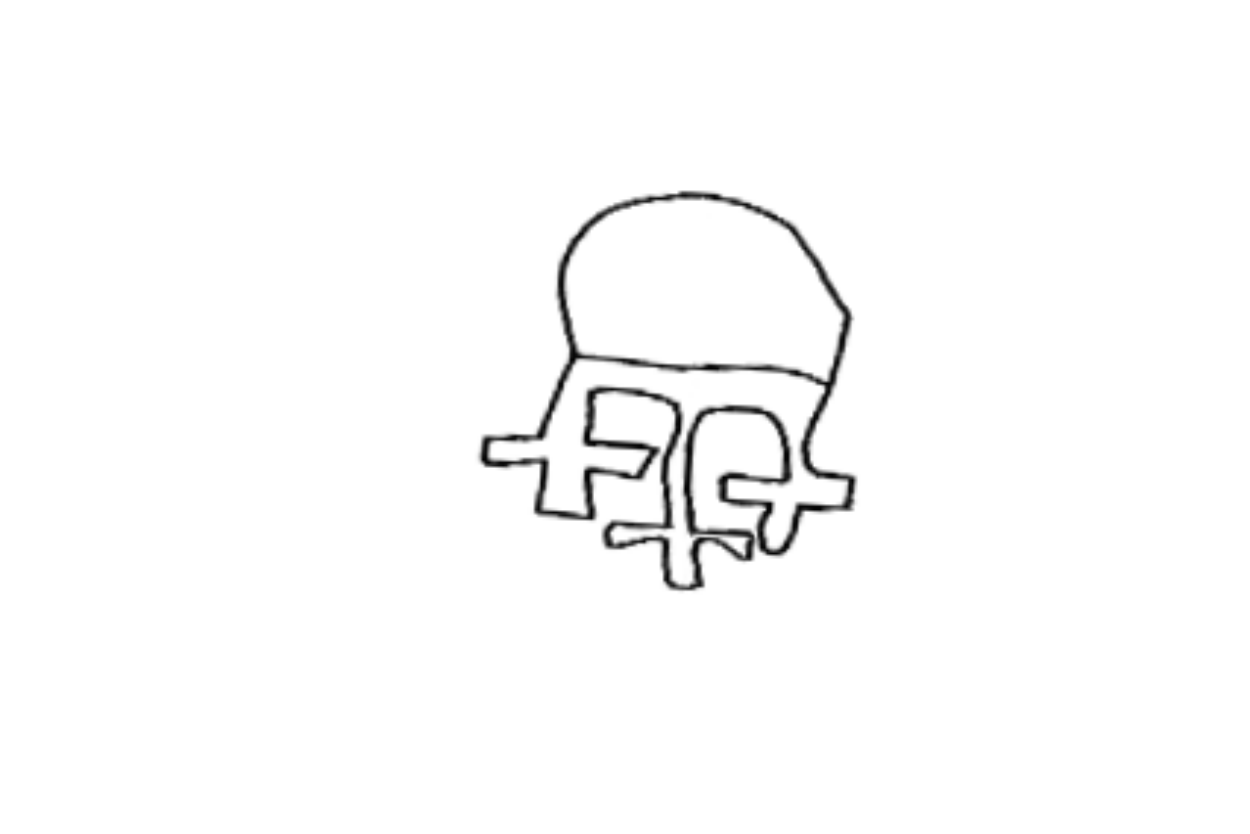}}%
            \hspace{1pt}%
            \boxed{\includegraphics[width=0.041\linewidth, clip=true, trim=132pt 27pt 113pt 12pt]{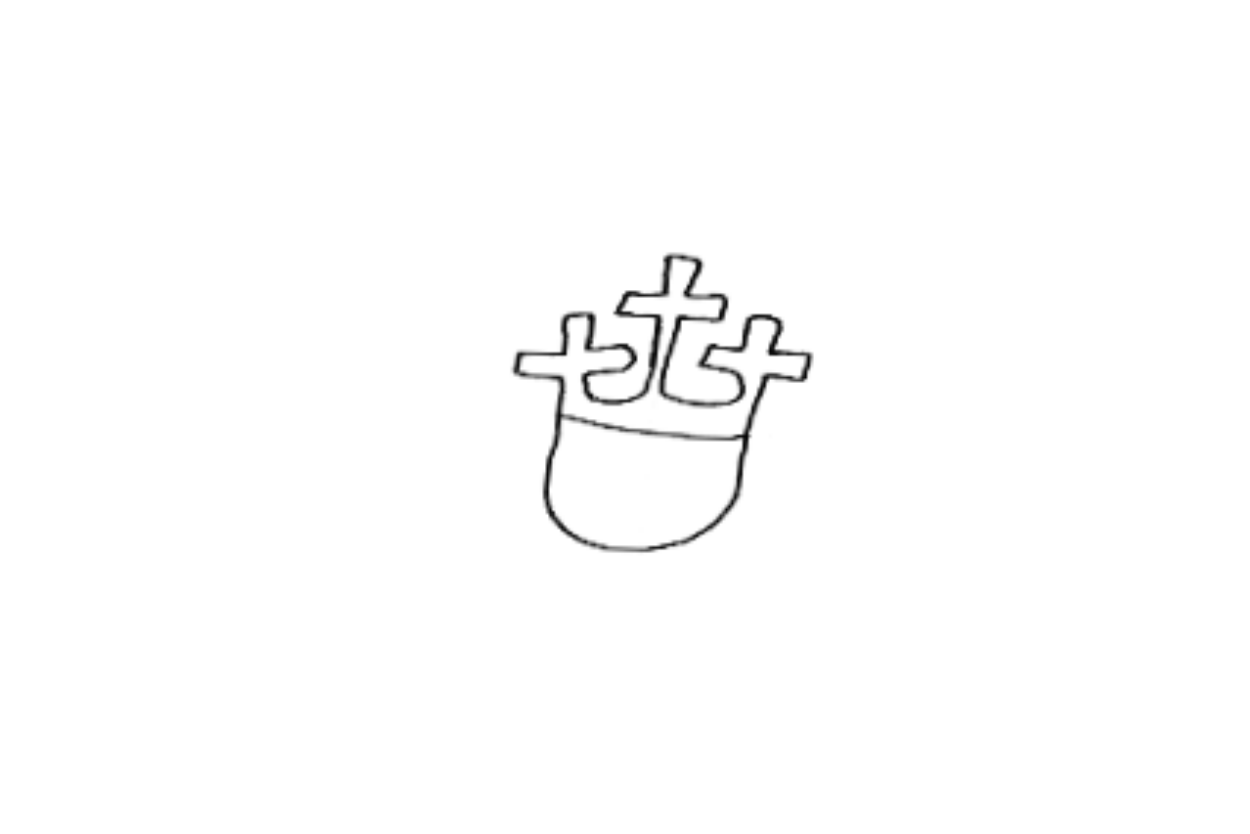}}%
        \end{tabular}}%
        \caption{%
        Selection of three independent samples per class from the historic watermark dataset.
        Here,
        black color corresponds to
        positive mass
        as opposed to Fig.~\ref{fig:polygon_dataset_2d}--\ref{fig:templates_mnist}.}
        \label{fig:watermark}
    \end{figure*}

    \item \textbf{Historic watermark dataset}.
    Inspired by our motivating application in filigranology,
    we extend the historic watermark experiment from \cite[§~5.5]{Beckmann2025}.
    More precisely, 
    we select eleven different classes 
    from the watermark information system (WZIS)%
    \footnote{\url{https://www.wasserzeichen-online.de}},
    each consisting of 4--13 samples.
    The resulting dataset comprises 75 images
    with $256 \times 256$ pixels,
    which are manually pre-processed 
    by removing artificial vertical lines 
    coming from the imaging technique.
    For illustration, three samples of each class
    are visualized in Figure~\ref{fig:watermark}. 
\end{enumerate}

\subsubsection{Nearest Template Classification}
\label{sssec:NT}

\begin{figure*}
    \centering%
    \footnotesize%
    \begin{tabular}{c c c}
    \multicolumn{3}{c}{%
    \LineClassOne
    \; class~1 \;
    \LineClassTwo
    \; class~2 \;
    \LineClassThree
    \; class~3} \\[1ex]
    \mNRCDT & \maNRCDT & \iaNRCDT \\
    \includegraphics[width=0.3\linewidth, clip=true, trim=12pt 10pt 12pt 10pt]{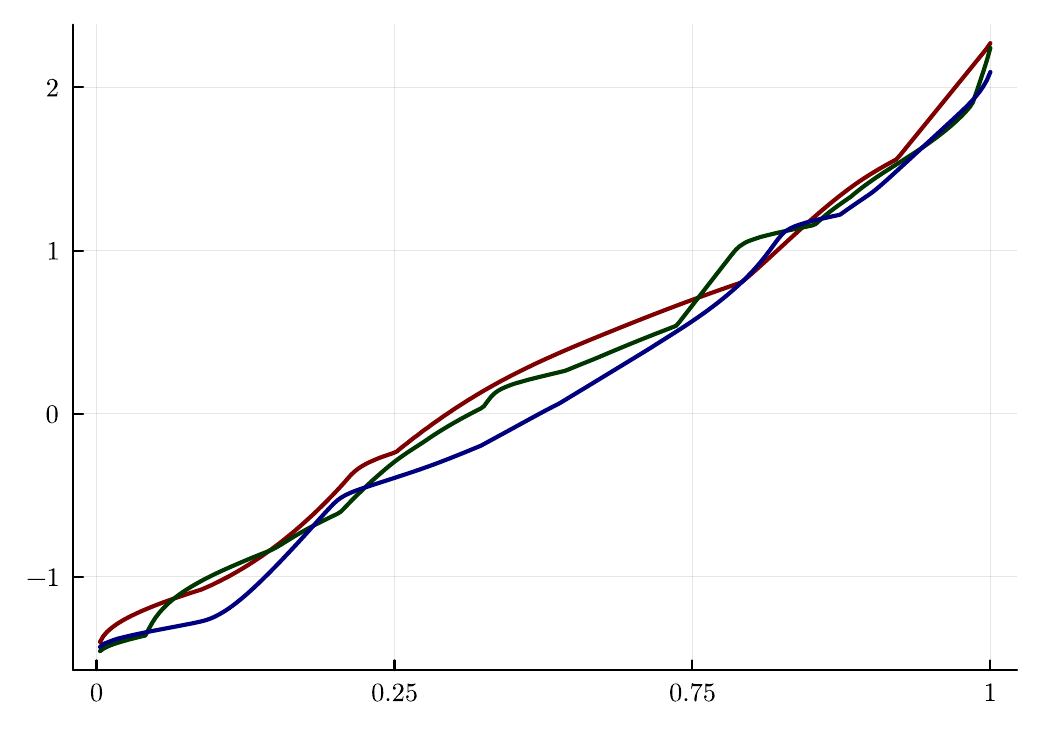}
    &\includegraphics[width=0.3\linewidth, clip=true, trim=12pt 10pt 12pt 10pt]{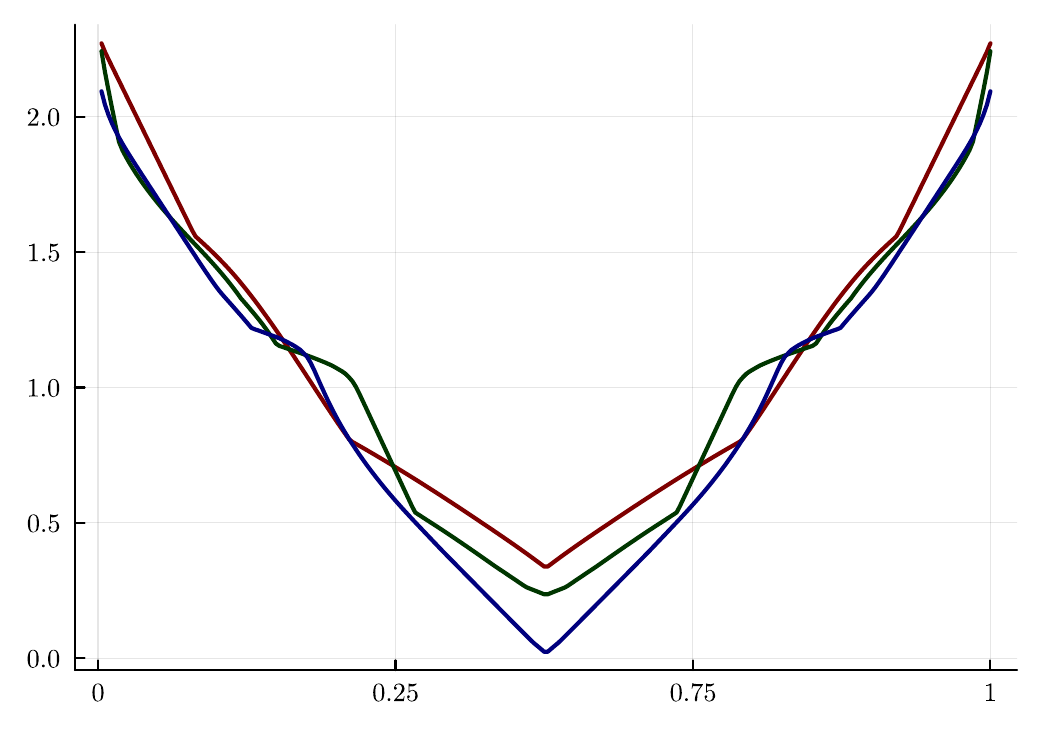}
    &\includegraphics[width=0.3\linewidth, clip=true, trim=12pt 10pt 12pt 10pt]{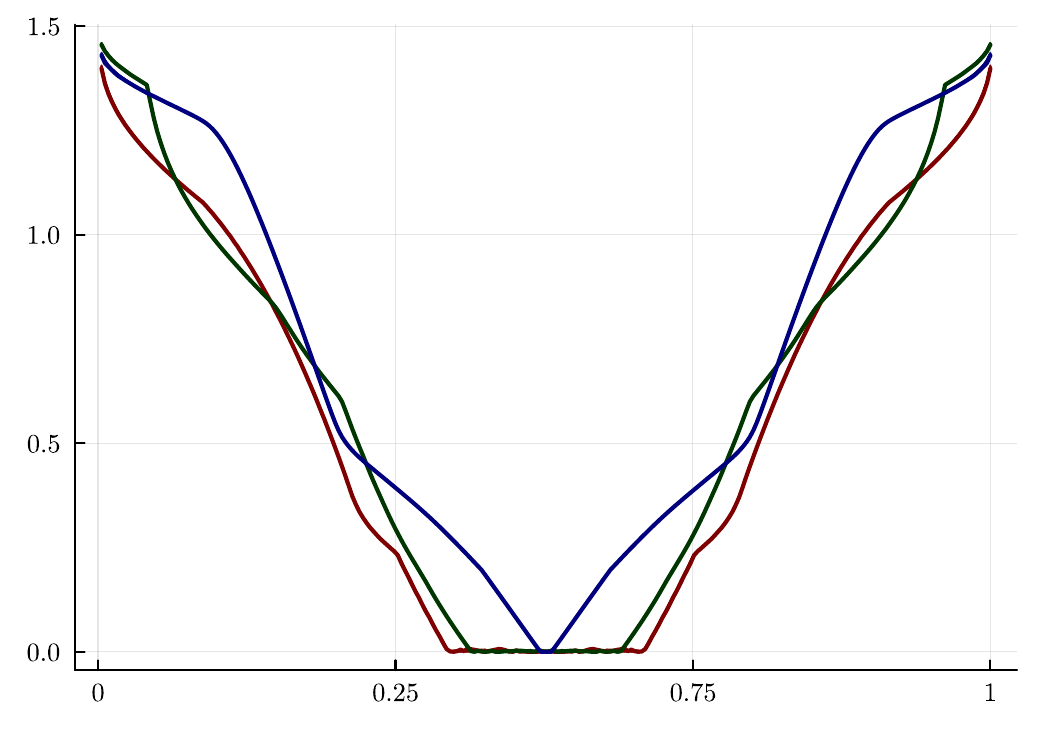} \\
    \miaNRCDT & \miNRCDT & \tvNRCDT \\
    \includegraphics[width=0.3\linewidth, clip=true, trim=12pt 10pt 12pt 10pt]{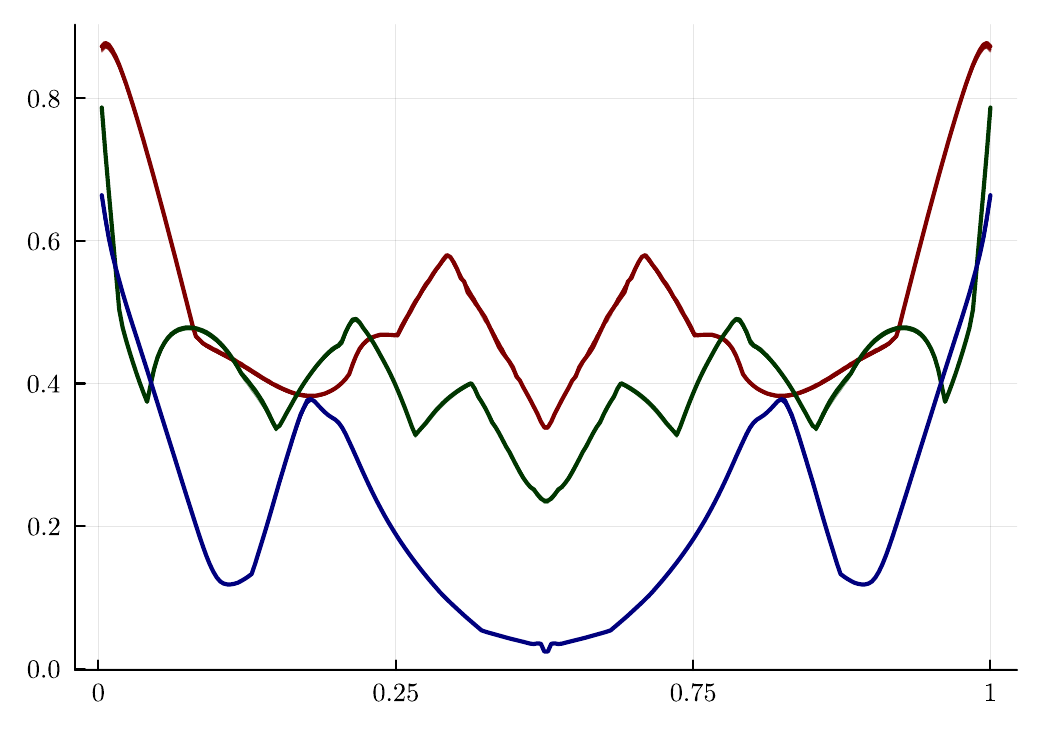}
    &\includegraphics[width=0.3\linewidth, clip=true, trim=12pt 10pt 12pt 10pt]{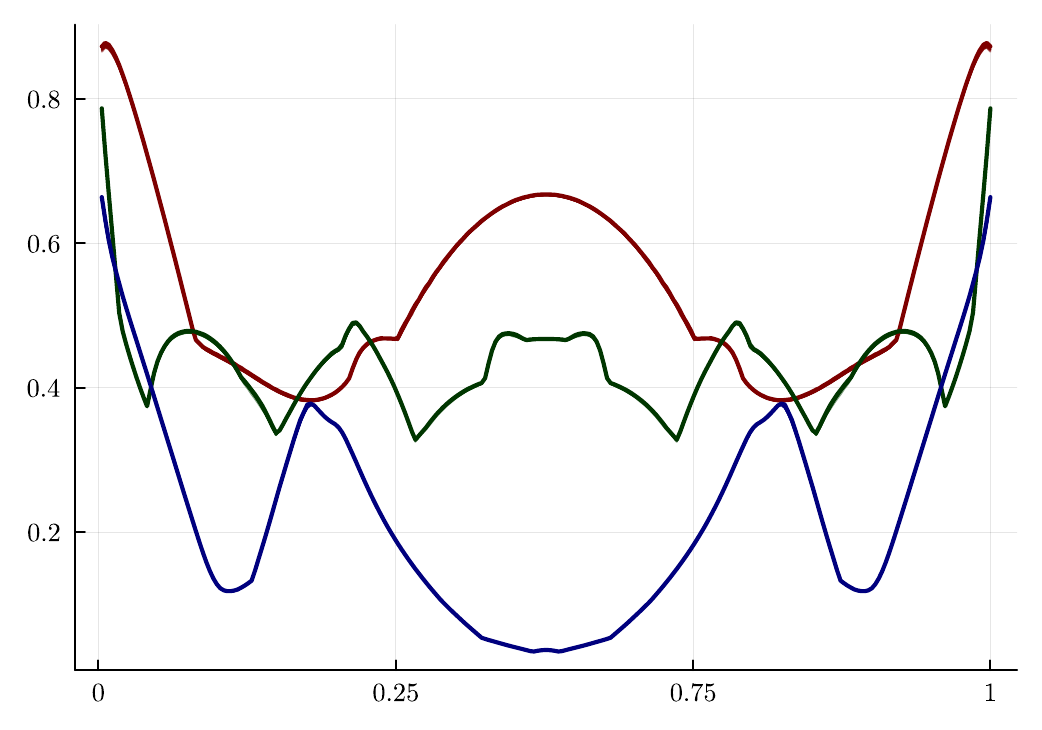}
    &\includegraphics[width=0.3\linewidth, clip=true, trim=12pt 10pt 12pt 10pt]{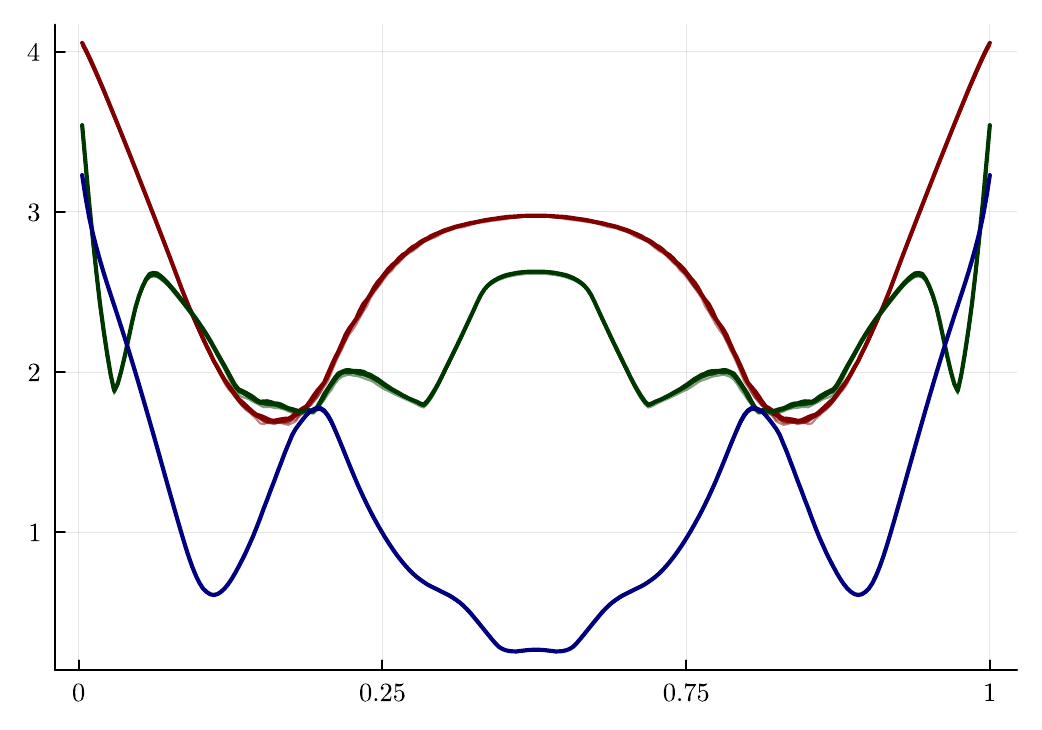}
    \end{tabular}
    \caption{%
    Visualizations of various \hNRCDT{}s for the academic dataset with $10$ samples per class
    and $256$ Radon angles.}
    \label{fig:hNRCDT_academic}
\end{figure*}

In the following first experiments,
we aim to validate the theoretical results from
Theorems~\ref{thm:sep-max-nrcdt} and~\ref{thm:sep-h-nrcdt}.
To this end,
we focus on the first two datasets,
which allow us to fully control the occurring affine and non-affine deformations.
Looking at the proofs,
we observe that the \hNRCDT{}s map each entire affine class
to a single point in the corresponding feature space.
As the datasets originate from known templates,
the easiest way for classification is the nearest template (NT) method,
which assigns the label of the closest template in the considered
feature spaces.

In order to observe the predicted behaviour numerically,
the underlying Radon transform and CDT
have to be discretized fine enough.
Therefore, we choose
different numbers of equispaced angles
and $850$ equispaced radii
for the Radon transform as well as
$256$ equispaced interpolation points
for the CDT.
In all our examples,
the grid for the Radon radii is at least three times finer
than the pixel size.
Consequently,
the considered images cannot be concentrated on a single Radon line,
i.e., they are contained in $\P_\c^*(\R^2)$.

We start with reporting our NT results
for the academic dataset with $10$ samples per class,
where we compare the Chebychev norm $\|\cdot\|_\infty$
and Euclidean norm $\|\cdot\|_2$ in \hNRCDT{} space
for determining the nearest template.
We compare our proposed feature extractors 
with the Euclidean baseline,
a moment-based (MB) feature extractor \cite{Flusser1993},
and linearized optimal transport (LOT) \cite{Moosmueller2023,Wang2013}.

\begin{figure*}
    \centering%
    \footnotesize%
    \begin{tabular}{c c c}
        template~1 & template~2 & template~3 \\
        \includegraphics[width=0.3\linewidth, clip=true, trim=12pt 10pt 0pt 0pt]{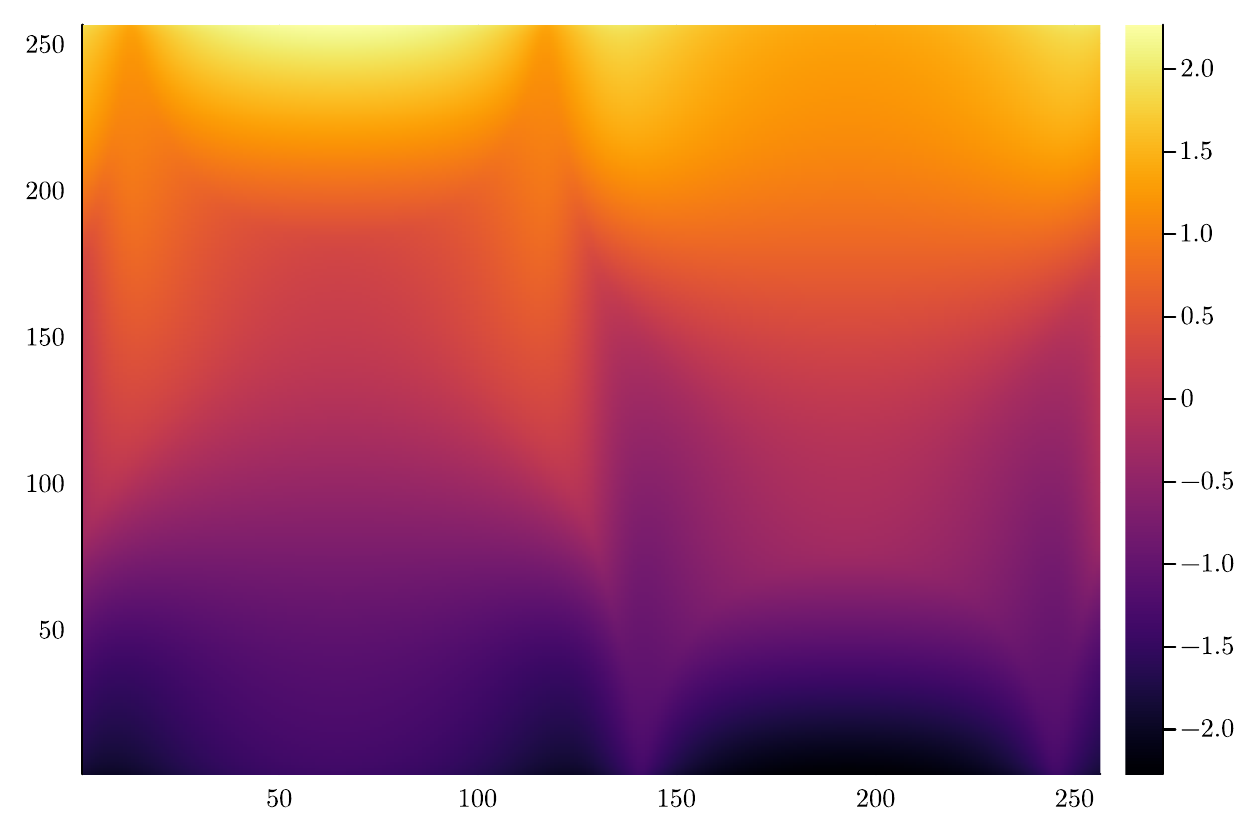}
        &\includegraphics[width=0.3\linewidth, clip=true, trim=12pt 10pt 0pt 0pt]{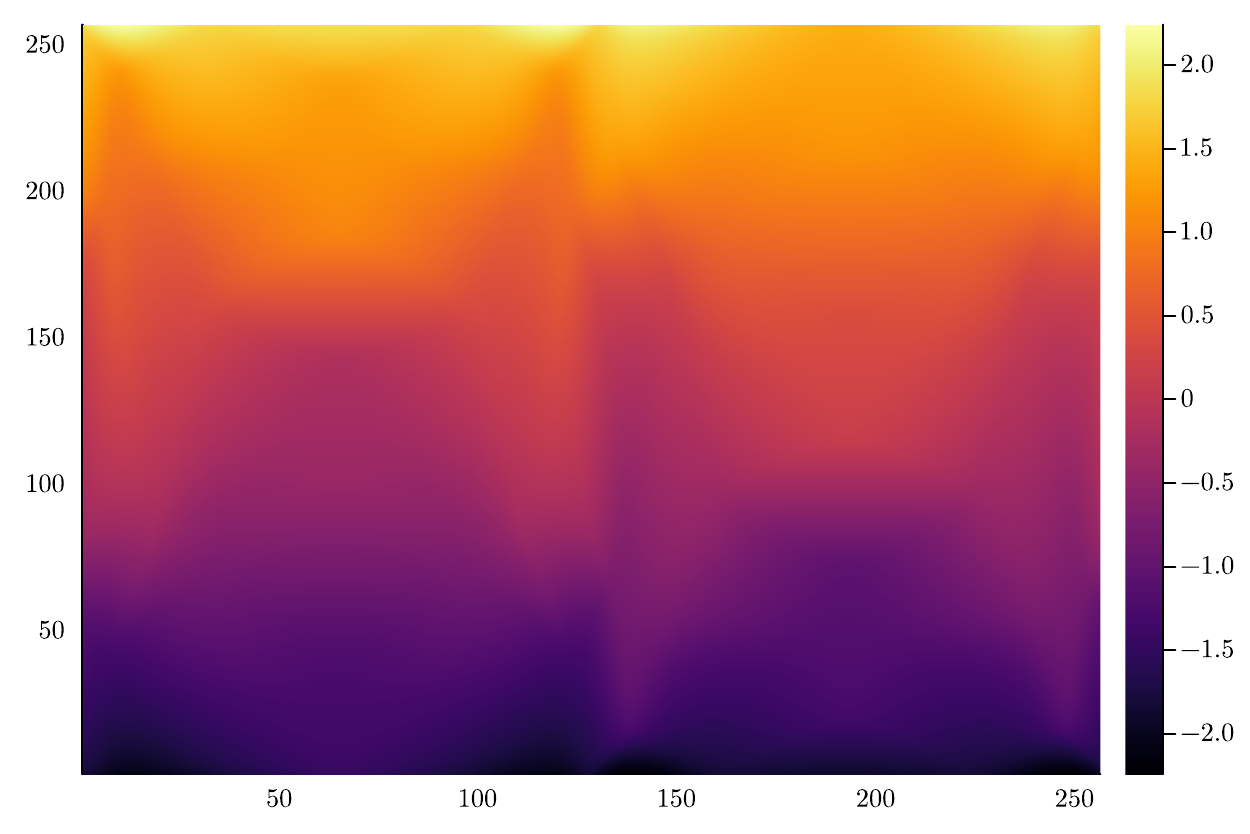}
        &\includegraphics[width=0.3\linewidth, clip=true, trim=12pt 10pt 0pt 0pt]{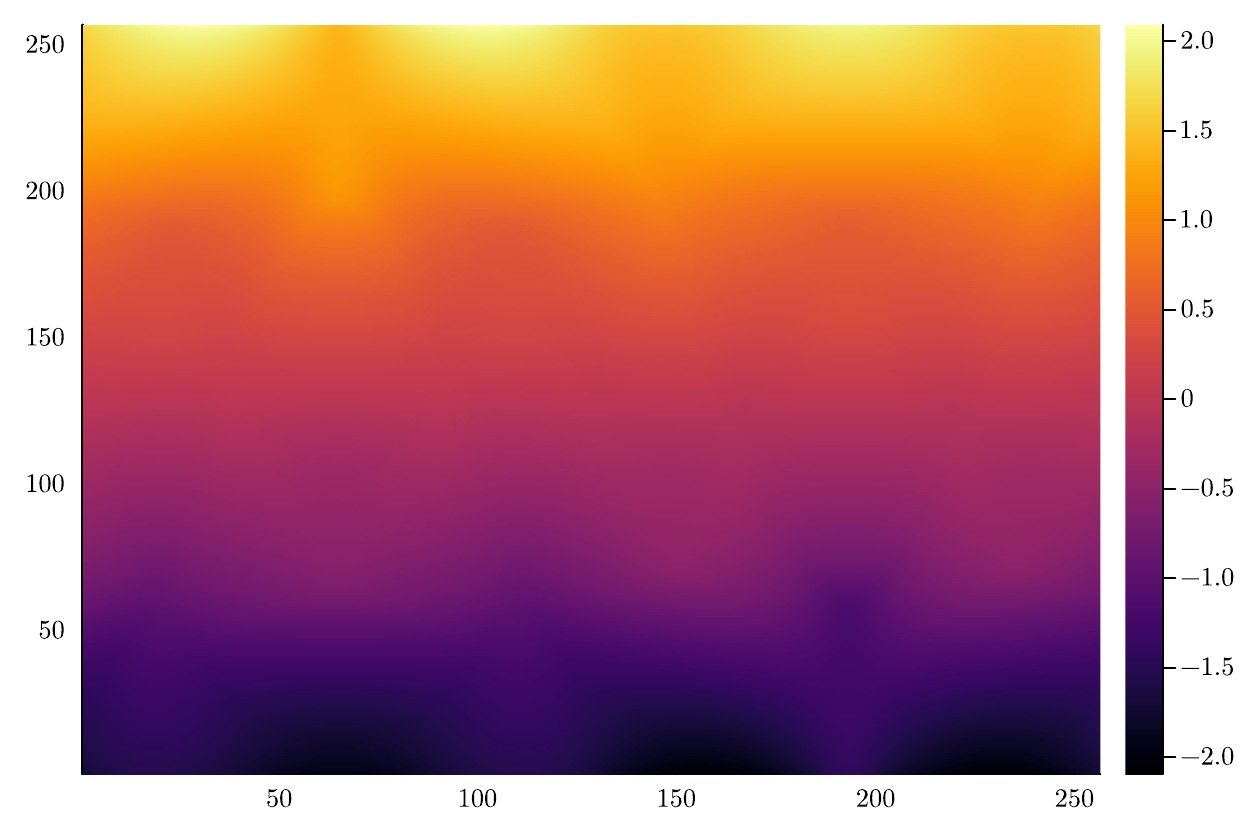}
    \end{tabular}
    \caption{%
    Visualization of the NR-CDTs
    for the three template measures
    of the academic dataset
    with $256$ Radon angles.}
    \label{fig:academic-NRCDT}
\end{figure*}

\begin{table*}
    \centering%
    \resizebox{\linewidth}{!}{%
    \begin{tabular}{l @{\quad} l @{\quad} l @{\quad} l @{\quad} l @{\;} l @{\quad} l @{\;} l @{\quad} l @{\;} l @{\quad} l @{\;} l @{\quad} l @{\;} l @{\quad} l @{\;} l @{\quad} l @{\;} l}
        \toprule
        \multirow{2}{*}{angles} 
        & \multicolumn{1}{@{}l}{Eucl.} 
        & \multicolumn{1}{@{}l}{MB}
        & \multicolumn{1}{@{}l}{LOT} 
        & \multicolumn{2}{@{}l}{R-CDT} 
        & \multicolumn{2}{@{}l}{\mNRCDT} 
        & \multicolumn{2}{@{}l}{\maNRCDT}
        & \multicolumn{2}{@{}l}{\iaNRCDT}
        & \multicolumn{2}{@{}l}{\miaNRCDT}
        & \multicolumn{2}{@{}l}{\miNRCDT}
        & \multicolumn{2}{@{}l}{\tvNRCDT}\\[0.5ex] 
        & $\nicefrac{\|\cdot\|_\infty}{\|\cdot\|_2}$ 
        & $\nicefrac{\|\cdot\|_\infty}{\|\cdot\|_2}$
        & 
        & $\|\cdot\|_\infty$ & $\|\cdot\|_2$
        & $\|\cdot\|_\infty$ & $\|\cdot\|_2$
        & $\|\cdot\|_\infty$ & $\|\cdot\|_2$ 
        & $\|\cdot\|_\infty$ & $\|\cdot\|_2$
        & $\|\cdot\|_\infty$ & $\|\cdot\|_2$
        & $\|\cdot\|_\infty$ & $\|\cdot\|_2$
        & $\|\cdot\|_\infty$ & $\|\cdot\|_2$\\
        \midrule
        2 & \multirow{8}{*}{\rotatebox{90}{0.333 / 0.366}}
        & \multirow{8}{*}{\rotatebox{90}{\textbf{1.000} / \textbf{1.000}}}
        & \multirow{8}{*}{\rotatebox{90}{0.333}}
        & 0.466 & 0.433 & 0.400 & 0.466 & 0.333 & 0.366 & 0.400 & \textbf{0.500} & 0.200 & 0.333 & 0.133 & 0.333 & 0.133 & 0.333 \\
        4 & & &  
        & 0.500 & 0.333 & 0.733 & 0.733 & 0.600 & 0.733 & 0.833 & 0.633 & \textbf{0.900} & 0.866 & 0.766 & 0.833 & 0.700 & 0.800 \\
        8 & & &
        & 0.400 & 0.366 & 0.666 & \textbf{0.933} & 0.666 & \textbf{0.933} & 0.733 & 0.733 & \textbf{0.933} & 0.800 & \textbf{0.933} & 0.866 & 0.666 & 0.733 \\
        16 & & &
        & 0.500 & 0.366 & \textbf{1.000} & \textbf{1.000} & \textbf{1.000} & \textbf{1.000} & 0.966 & \textbf{1.000} & \textbf{1.000} & \textbf{1.000} & \textbf{1.000} & \textbf{1.000} & 0.933 & 0.933 \\
        32 & & & 
        & 0.433 & 0.366 & \textbf{1.000} & \textbf{1.000} & \textbf{1.000} & \textbf{1.000} & \textbf{1.000} & \textbf{1.000} & \textbf{1.000} & \textbf{1.000} & \textbf{1.000} & \textbf{1.000} & \textbf{1.000} & \textbf{1.000} \\
        64 & & &
        & 0.433 & 0.366 & \textbf{1.000} & \textbf{1.000} & \textbf{1.000} & \textbf{1.000} & \textbf{1.000} & \textbf{1.000} & \textbf{1.000} & \textbf{1.000} & \textbf{1.000} & \textbf{1.000} & \textbf{1.000} & \textbf{1.000} \\
        128 & & &
        & 0.400 & 0.366 & \textbf{1.000} & \textbf{1.000} & \textbf{1.000} & \textbf{1.000} & \textbf{1.000} & \textbf{1.000} & \textbf{1.000} & \textbf{1.000} & \textbf{1.000} & \textbf{1.000} & \textbf{1.000} & \textbf{1.000} \\
        256 & & & 
        & 0.400 & 0.366 & \textbf{1.000} & \textbf{1.000} & \textbf{1.000} & \textbf{1.000} & \textbf{1.000} & \textbf{1.000} & \textbf{1.000} & \textbf{1.000} & \textbf{1.000} & \textbf{1.000} & \textbf{1.000} & \textbf{1.000} \\
        \bottomrule
    \end{tabular}}
    \caption{%
    NT classification accuracies
    for academic dataset
    with $10$ samples per class
    and different numbers of angles.}
    \label{tab:nearest_temp_academic}
\end{table*}

The MB feature extractor is based on the central moments 
\begin{equation*}
    \mu_{pq} \coloneqq \int_{\R^2} (x_1 - \bar x_1)^p (x_2 - \bar x_2)^q \d \mu(\bfx),
    \quad p,q \in \N,
\end{equation*}
where $(\bar x_1, \bar x_2)^\top \coloneqq \int_{\R^2} \bfx \d \mu(\bfx)$
denotes the mean.
As proposed in \cite{Flusser1993},
we use the four features from \cite[Eqs.~(23),~(26),~(28),~(31)]{Flusser1993}
that are invariant under affine transformations
and well-defined for $\mu \in \P_\c^*(\R^2)$.

For an absolutely continuous
reference $\sigma \in \P_2(\R^2)$,
LOT maps $\mu \in \P_2(\R^2)$ to
the optimal transport map $T_\sigma^\mu$
realizing the Wasserstein-2 distance
\begin{equation}
    \label{eq:LOT}
    W_2^2(\sigma,\mu)
    \coloneqq 
    \min_{T_{\#}\sigma = \mu}
    \int_{\R^2} \|T(\bfx) - \bfx\|^2 \d \sigma(\bfx),
\end{equation}
where we optimize over all measurable $T \colon \R^2 \to \R^2$.
For numerical realization,
instead of solving \eqref{eq:LOT},
we compute the optimal transport plan $\pi_\sigma^\mu$
from the Kantorovich problem 
\begin{equation}
    \label{eq:OT-plan}
    \pi_\sigma^\mu \in \argmin_{\pi \in \Pi(\sigma,\mu)}
    \int_{\R^2\times\R^2}\|\bfx - \bfy\|^2 \d \pi(\bfx,\bfy),
\end{equation}
where
$\Pi(\sigma,\mu)$ is the set of
measures in $\P(\R^2 \times \R^2)$
with marginals $\sigma$ and $\mu$.
Afterwards,
we approximate $T_\sigma^\mu$
by the barycentric projection 
\begin{equation*}
    T_{\pi_{\sigma}^\mu}(\bfx) 
    \coloneqq \int_{\R^2}\bfy \d \pi^\mu_{\bfx}(\bfy),
    \quad \bfx \in \R^2,
\end{equation*}
with $\pi_{\sigma}^\mu(A\times B) = \int_A \pi^\mu_{\bfx}(B) \d\sigma(\bfx)$
for any Borel sets $A, B \in \B(\R^2)$.

For illustration, 
the \hNRCDT{}s of the entire dataset
are depicted in Figure~\ref{fig:hNRCDT_academic},
where the thick lines correspond to the templates 
and the thin lines to the transformed images.
As predicted by our theory,
the lines per class are nearly perfectly aligned.
This also holds for the \tvNRCDT{}, which in theory
requires more regularity of the measures.
An inspection of Figure~\ref{fig:academic-NRCDT},
however,
reveals that the corresponding NR-CDTs of the templates
are rather smooth.

Our classification results are shown in
Table~\ref{tab:nearest_temp_academic}
for varying numbers of equispaced angles
for the underlying Radon transform.
We observe that the \hNRCDT{}
feature representations clearly outperform
the R-CDT, LOT, and Euclidean baseline
and, remarkably,
the classification is already perfect 
for a small number of chosen angles.
This is also the case for MB.
However, since only four features are employed,
we will see that
MB performs worse for more complex datasets.
As the different \hNRCDT{} representations perform
nearly on par, we henceforth consider only \mNRCDT{},
\hNRCDT{d} and \tvNRCDT{}.
Moreover, in most cases, $\|\cdot\|_2$ outperforms $\|\cdot\|_\infty$ so that from now on
we restrict ourselves to the Euclidean norm $\|\cdot\|_2$.

The computation times 
for this first academic experiment 
are reported in Table~\ref{tab:runtime_academic_dataset}.
While the four-dimensional MB feature extractor
can be computed in less than one second,
LOT is not tractable and we have to downscale
the images to $64\times64$ pixels.
However,
even in this case,
the computation of the optimal transport plans 
in~\eqref{eq:OT-plan}
takes around 2000 seconds.
For our proposed NR-CDT approach,
computing the R-CDT
is the dominating component.
Thereafter,
NT classification based on NR-CDT features
is up to 100-times faster than for R-CDT.
The computational cost of the
normalization procedure
is negligible.

\begin{table}[t]
    \centering%
    \resizebox{\linewidth}{!}{%
    \begin{tabular}{l @{\qquad} l @{\qquad} l @{\qquad} l @{\qquad} l }
    \toprule
    \multirow{2}{*}{angles}
    & \multirow{2}{*}{\parbox{1.8cm}{computation \\ of R-CDT}} & \multirow{2}{*}{\parbox{1.8cm}{ subsequent \\ normalization}} & \multicolumn{2}{@{}l}{NT classification} \\
    \cmidrule(r){4-5} 
    & & & R-CDT & \hNRCDT{} \\
    \midrule
    2 & 2.0687 & 0.0015 & 0.0012 & 0.0019 \\
    4 & 4.5088 & 0.0014 & 0.0041 & 0.0016 \\
    8 & 11.9253 & 0.0016 & 0.0025 & 0.0012 \\
    16 & 22.0266 & 0.0021 & 0.0065 & 0.0013 \\
    32 & 41.9814 & 0.0031 & 0.0109 & 0.0021 \\
    64 & 84.8437 & 0.0048 & 0.0261 & 0.0023 \\
    128 & 172.2234 & 0.0065 & 0.0458 & 0.0021 \\
    256 & 322.4234 & 0.0133 & 0.1372 & 0.0019 \\
    \bottomrule
    \end{tabular}}
    \caption{%
    Runtime study (in sec.)
    for academic dataset.}
    \label{tab:runtime_academic_dataset}
\end{table}

To study the robustness of the NT classification
we consider the polygon dataset with $10$ samples per class,
which includes small non-affine perturbations.
Hence, a perfect classification according to
Theorems~\ref{thm:sep-max-nrcdt} and~\ref{thm:sep-h-nrcdt}
cannot be expected.
The accuracy results 
are reported in Table~\ref{tab:nearest_temp_polygons}.
While the Euclidean baseline and R-CDT
are as bad as random guessing,
\mNRCDT{} and \hNRCDT{d} perform significantly better
and, for sufficiently many Radon angles,
\tvNRCDT{} even reaches perfect results.
This indicates its potential 
in real-world applications
and may be attributed to the fact that
\tvNRCDT{} uses the whole information
over the entire range of angles in $[0, 2\pi)$,
while \mNRCDT{} and \hNRCDT{d} only consider
the largest and smallest angle information.
The performance of MB deteriorates
compared to the academic dataset
due to the slight perturbations in the templates
and lags behind \tvNRCDT.

\begin{table}[t]
        \centering%
        \footnotesize%
        \begin{tabular}{l @{\quad} l @{\quad} l @{\quad} l @{\hspace{5pt}} l @{\hspace{5pt}} l @{\hspace{5pt}} l }
        \toprule
        \rotatebox{90}{angles} 
        & \rotatebox{90}{Eucl.}
        & \rotatebox{90}{MB}
        & R-CDT
        & \mNRCDT
        & \miNRCDT
        & \tvNRCDT \\
        \midrule
        2 
        & \multirow{8}{*}{\rotatebox{90}{0.144}}
        & \multirow{8}{*}{\rotatebox{90}{0.933}}
        & 0.122 & 0.089 & \textbf{0.122} & \textbf{0.122} \\
        4 
        & & & 0.133 & 0.278 & 0.289 & \textbf{0.311} \\
        8
        & & & 0.133 & \textbf{0.289} & 0.233 & 0.244 \\
        16
        & & & 0.144 & \textbf{0.567} & 0.389 & 0.389 \\
        32
        & & & 0.144 & 0.478 & 0.422 & \textbf{0.611} \\
        64
        & & & 0.144 & 0.700 & 0.711 & \textbf{0.944} \\
        128
        & & & 0.144 & 0.678 & 0.678 & \textbf{1.000} \\
        256
        & & & 0.144 & 0.678 & 0.667 & \textbf{1.000} \\
        \bottomrule
    \end{tabular}
    \caption{%
    NT classification accuracies
    for polygon dataset
    with $10$ samples per class
    and varying numbers of angles.}
    \label{tab:nearest_temp_polygons}
\end{table}

Towards a real-world application, 
we consider the historic watermark dataset.
Different from the academic and the polygon dataset,
the samples in each class of this dataset
are visually similar but do not originate
from one template measure 
under affine transformations.
Therefore, the assumptions 
of Theorems~\ref{thm:sep-max-nrcdt} and~\ref{thm:sep-h-nrcdt}
are violated and perfect separability is not expected.
For classification,
we select one sample from each class
as the template measure
and perform NT classification, as before.
The accuracy results are reported in Table~\ref{tab:watermark_NT}.
We observe that all our NR-CDT variants
clearly outperform
the Euclidean baseline, R-CDT and MB,
with \tvNRCDT{} reaching 
an accuracy of $92\%$
when using at least 512 Radon angles.

\begin{table}[t]
    \resizebox{\linewidth}{!}{
        \footnotesize
        \begin{tabular}{@{} l @{\quad} l @{\quad} l @{\quad} l @{\enspace} l @{\enspace} l @{\enspace} l @{}}
        \toprule
        \rotatebox{90}{angles}
        & \rotatebox{90}{Eucl.}
        & \rotatebox{90}{MB}
        & R-CDT
        & \mNRCDT 
        & \miNRCDT
        & \tvNRCDT \\
        \midrule
        2
        & \multirow{10}{*}{\rotatebox{90}{$0.078$}} 
        & \multirow{10}{*}{\rotatebox{90}{$0.406$}} 
        & $0.140$ & $0.516$ & $0.203$ & $0.200$ \\
        4
        & & & $0.219$ & $\bf{0.781}$ & $0.703$ & $0.719$ \\
        8
        & & & $0.188$ & $\bf{0.828}$ & $0.782$ & $0.766$ \\
        16
        & & & $0.203$ & $\bf{0.859}$ & $0.781$ & $0.766$ \\
        32
        & & & $0.203$ & $\bf{0.859}$ & $0.800$ & $0.822$ \\
        64  
        & & & $0.203$ & $0.859$ & $0.813$ & $\bf{0.891}$ \\
        128  
        & & & $0.203$ & $0.859$ & $0.813$ & $\bf{0.891}$ \\
        256 
        & & & $0.203$ & $0.859$ & $0.813$ & $\bf{0.891}$ \\
        512 
        & & & $0.203$ & $0.859$ & $0.813$ & $\bf{0.922}$ \\
        1024 
        & & & $0.203$ & $0.859$ & $0.796$ & $\bf{0.938}$ \\
        \bottomrule
    \end{tabular}}
    \caption{NT classification accuracies 
    for the historic watermark dataset.}
    \label{tab:watermark_NT}
\end{table}

\subsubsection{Nearest Neighbour Classification}
\label{sssec:NN}

\begin{table}[t]
    \resizebox{\linewidth}{!}{
        \footnotesize
        \begin{tabular}{@{} l @{\enspace} l @{\enspace} l @{\enspace} l @{\enspace} l @{\enspace} l @{\enspace} l @{}}
        \toprule
        \rotatebox{90}{train}
        & \rotatebox{90}{angle}
        & \rotatebox{90}{Eucl.}
        & R-CDT
        & \mNRCDT 
        & \miNRCDT
        & \tvNRCDT \\
        \midrule
        1 
        & 2
        & \multirow{8}{*}{\rotatebox{90}{$0.344{\pm}0.024$}} 
        & $0.331{\pm}0.016$ & $\mathbf{0.477{\pm}0.083}$ & $0.430{\pm}0.084$ & $0.429{\pm}0.099$ \\
        & 4
        & & $0.336{\pm}0.028$ & $\mathbf{0.589{\pm}0.089}$ & $\mathbf{0.589{\pm}0.108}$ & $0.571{\pm}0.135$ \\
        & 8
        & & $0.340{\pm}0.027$ & $0.763{\pm}0.080$ & $0.718{\pm}0.126$ & $\mathbf{0.769{\pm}0.130}$ \\
        & 16
        & & $0.340{\pm}0.027$ & $0.796{\pm}0.122$ & $0.768{\pm}0.114$ & $\mathbf{0.828{\pm}0.122}$ \\
        & 32
        & & $0.340{\pm}0.027$ & $0.818{\pm}0.093$ & $0.768{\pm}0.120$ & $\mathbf{0.830{\pm}0.119}$ \\
        & 64  
        & & $0.340{\pm}0.027$ & $0.820{\pm}0.090$ & $0.760{\pm}0.127$ & $\mathbf{0.824{\pm}0.119}$ \\
        & 128  
        & & $0.340{\pm}0.027$ & $0.820{\pm}0.091$ & $0.752{\pm}0.127$ & $\mathbf{0.823{\pm}0.119}$ \\
        & 256 
        & & $0.340{\pm}0.027$ & $0.820{\pm}0.091$ & $0.752{\pm}0.127$ & $\mathbf{0.822{\pm}0.119}$ \\
        \midrule
        5 
        & 2
        & \multirow{8}{*}{\rotatebox{90}{$0.340{\pm}0.038$}} 
        & $0.335{\pm}0.022$ & $0.541{\pm}0.047$ & $\mathbf{0.548{\pm}0.058}$ & $0.530{\pm}0.047$ \\
        & 4
        & & $0.343{\pm}0.025$ & $\mathbf{0.717{\pm}0.051}$ & $0.706{\pm}0.049$ & $0.696{\pm}0.071$ \\
        & 8
        & & $0.343{\pm}0.032$ & $0.832{\pm}0.037$ & $\mathbf{0.837{\pm}0.046}$ & $0.829{\pm}0.041$ \\
        & 16
        & & $0.343{\pm}0.031$ & $0.863{\pm}0.023$ & $0.849{\pm}0.033$ & $\mathbf{0.879{\pm}0.017}$ \\
        & 32
        & & $0.344{\pm}0.031$ & $0.885{\pm}0.019$ & $0.865{\pm}0.026$ & $\mathbf{0.891{\pm}0.019}$ \\
        & 64 
        & & $0.344{\pm}0.031$ & $0.884{\pm}0.021$ & $0.864{\pm}0.025$ & $\mathbf{0.889{\pm}0.021}$ \\
        & 128
        & & $0.344{\pm}0.031$ & $0.886{\pm}0.019$ & $0.864{\pm}0.024$ & $\mathbf{0.890{\pm}0.022}$ \\
        & 256
        & & $0.344{\pm}0.031$ & $0.886{\pm}0.019$ & $0.864{\pm}0.024$ & $\mathbf{0.890{\pm}0.023}$ \\
        \midrule
        10 
        & 2
        & \multirow{8}{*}{\rotatebox{90}{$0.348{\pm}0.022$}} 
        & $0.335{\pm}0.021$ & $\mathbf{0.586{\pm}0.021}$ & $0.523{\pm}0.061$ & $0.541{\pm}0.035$ \\
        & 4
        & & $0.350{\pm}0.028$ & $\mathbf{0.734{\pm}0.048}$ & $0.713{\pm}0.038$ & $0.729{\pm}0.027$ \\
        & 8
        & & $0.351{\pm}0.032$ & $0.846{\pm}0.025$ & $0.845{\pm}0.028$ & $\mathbf{0.856{\pm}0.017}$ \\
        & 16
        & & $0.352{\pm}0.032$ & $0.870{\pm}0.019$ & $0.866{\pm}0.026$ & $\mathbf{0.891{\pm}0.016}$ \\
        & 32
        & & $0.352{\pm}0.033$ & $\mathbf{0.899{\pm}0.021}$ & $0.878{\pm}0.022$ & $0.896{\pm}0.016$ \\
        & 64
        & & $0.352{\pm}0.033$ & $0.899{\pm}0.020$ & $0.878{\pm}0.018$ & $\mathbf{0.901{\pm}0.015}$ \\
        & 128
        & & $0.352{\pm}0.033$ & $\mathbf{0.902{\pm}0.020}$ & $0.878{\pm}0.018$ & $0.901{\pm}0.013$ \\
        & 256
        & & $0.352{\pm}0.033$ & $\mathbf{0.901{\pm}0.020}$ & $0.880{\pm}0.018$ & $\mathbf{0.901{\pm}0.013}$ \\
        \bottomrule
    \end{tabular}}
    \caption{%
    $1$-NN classification accuracies (mean$\pm$std)
    for LinMNIST dataset with $100$ samples of classes $\{1,5,7\}$
    and different numbers of training samples
    and angles.}
    \label{tab:nearest_neighbor_LinMNIST}
\end{table}

For a more realistic scenario,
we consider the LinMNIST dataset
with $100$ samples per class.
As in this case
no templates exist,
we replace NT classification
by the nearest neighbour ($1$-NN) method,
which assigns the label of the nearest member
of a given training set
with respect to the Euclidean distance in feature space.
To this end,
a subset of varying size is randomly selected
to serve as training data
and the remaining samples of the dataset
are used for testing.
This process is repeated $20$ times
and the mean and standard deviation (std)
of the resulting classification accuracies
are reported in Table~\ref{tab:nearest_neighbor_LinMNIST},
where we use (up to) $256$ angles, $500$ radii
and $256$ interpolation points.

While the Euclidean baseline and R-CDT
perform at the level of random guessing,
we observe that all our \hNRCDT{} variants
yield remarkable results in this limited data setting,
that improve with increasing numbers of angles
and selected training samples.
In particular,
when using only {\em one} training datum per class,
\mNRCDT{} and \tvNRCDT{}
already reach an accuracy of $82\%$,
which improves to $90\%$
for $10$ training samples.

\subsubsection{Support Vector Machines}
\label{sssec:LinSVM}

\begin{table}[t]
    \resizebox{\linewidth}{!}{
        \footnotesize
        \begin{tabular}{@{} l @{\enspace} l @{\enspace} l @{\enspace} l @{\enspace} l @{\enspace} l @{\enspace} l @{}}
        \toprule
        \rotatebox{90}{train}
        & \rotatebox{90}{angle}
        & \rotatebox{90}{Eucl.}
        & R-CDT 
        & \mNRCDT
        & \miNRCDT
        & \tvNRCDT \\
        \midrule
        1
        & 2 
        & \multirow{6}{*}{\rotatebox{90}{$0.500{\pm}0.015$}} 
        & $0.499{\pm}0.033$ & $0.617{\pm}0.126$ & $0.663{\pm}0.119$ & $\mathbf{0.676{\pm}0.131}$ \\
        & 4 
        & & $0.499{\pm}0.023$ & $0.777{\pm}0.084$ & $0.725{\pm}0.112$ & $\mathbf{0.783{\pm}0.102}$ \\
        & 8 
        & & $0.520{\pm}0.018$ & $\mathbf{0.976{\pm}0.059}$ & $0.913{\pm}0.094$ & $0.763{\pm}0.139$ \\
        & 16 
        & & $0.514{\pm}0.023$ & $\mathbf{1.000{\pm}0.000}$ & $0.988{\pm}0.023$ & $0.937{\pm}0.055$ \\
        & 32 
        & & $0.578{\pm}0.020$ & $\mathbf{1.000{\pm}0.000}$ & $\mathbf{1.000{\pm}0.000}$ & $\mathbf{1.000{\pm}0.000}$ \\
        & 64
        & & $0.505{\pm}0.025$ & $\mathbf{1.000{\pm}0.000}$ & $\mathbf{1.000{\pm}0.000}$ & $\mathbf{1.000{\pm}0.000}$ \\
        \midrule
        5
        & 2 
        & \multirow{6}{*}{\rotatebox{90}{$0.508{\pm}0.019$}} 
        & $0.515{\pm}0.039$ & $\mathbf{0.863{\pm}0.052}$ & $0.744{\pm}0.051$ & $0.760{\pm}0.084$ \\
        & 4 
        & & $0.522{\pm}0.029$ & $\mathbf{0.971{\pm}0.040}$ & $0.947{\pm}0.041$ & $0.964{\pm}0.031$ \\
        & 8 
        & & $0.528{\pm}0.045$ & $\mathbf{1.000{\pm}0.000}$ & $\mathbf{1.000{\pm}0.000}$ & $0.984{\pm}0.017$ \\
        & 16 
        & & $0.556{\pm}0.058$ & $\mathbf{1.000{\pm}0.000}$ & $\mathbf{1.000{\pm}0.000}$ & $0.986{\pm}0.011$ \\
        & 32 
        & & $0.596{\pm}0.076$ & $\mathbf{1.000{\pm}0.000}$ & $\mathbf{1.000{\pm}0.000}$ & $\mathbf{1.000{\pm}0.000}$ \\
        & 64
        & & $0.599{\pm}0.086$ & $\mathbf{1.000{\pm}0.000}$ & $\mathbf{1.000{\pm}0.000}$ & $\mathbf{1.000{\pm}0.000}$ \\
        \midrule
        10
        & 2 
        & \multirow{6}{*}{\rotatebox{90}{$0.518{\pm}0.023$}} 
        & $0.549{\pm}0.031$ & $\mathbf{0.936{\pm}0.049}$ & $0.797{\pm}0.047$ & $0.847{\pm}0.049$ \\
        & 4 
        & & $0.572{\pm}0.048$ & $\mathbf{0.989{\pm}0.022}$ & $0.978{\pm}0.009$ & $0.972{\pm}0.045$ \\
        & 8 
        & & $0.630{\pm}0.050$ & $\mathbf{1.000{\pm}0.000}$ & $\mathbf{1.000{\pm}0.000}$ & $0.991{\pm}0.013$ \\
        & 16 
        & & $0.703{\pm}0.062$ & $\mathbf{1.000{\pm}0.000}$ & $\mathbf{1.000{\pm}0.000}$ & $0.996{\pm}0.006$ \\
        & 32 
        & & $0.767{\pm}0.092$ & $\mathbf{1.000{\pm}0.000}$ & $\mathbf{1.000{\pm}0.000}$ & $\mathbf{1.000{\pm}0.000}$ \\
        & 64
        & & $0.768{\pm}0.093$ & $\mathbf{1.000{\pm}0.000}$ & $\mathbf{1.000{\pm}0.000}$ & $\mathbf{1.000{\pm}0.000}$ \\
        \midrule
        25
        & 2 
        & \multirow{6}{*}{\rotatebox{90}{$0.527{\pm}0.014$}} 
        & $0.593{\pm}0.085$ & $\mathbf{0.975{\pm}0.025}$ & $0.881{\pm}0.041$ & $0.940{\pm}0.022$ \\
        & 4 
        & & $0.730{\pm}0.068$ & $\mathbf{1.000{\pm}0.000}$ & $0.986{\pm}0.007$ & $0.990{\pm}0.007$ \\
        & 8 
        & & $0.868{\pm}0.064$ & $\mathbf{1.000{\pm}0.000}$ & $\mathbf{1.000{\pm}0.000}$ & $0.998{\pm}0.005$ \\
        & 16 
        & & $0.975{\pm}0.037$ & $\mathbf{1.000{\pm}0.000}$ & $\mathbf{1.000{\pm}0.000}$ & $\mathbf{1.000{\pm}0.000}$ \\
        & 32 
        & & $0.991{\pm}0.014$ & $\mathbf{1.000{\pm}0.000}$ & $\mathbf{1.000{\pm}0.000}$ & $\mathbf{1.000{\pm}0.000}$ \\
        & 64
        & & $0.985{\pm}0.022$ & $\mathbf{1.000{\pm}0.000}$ & $\mathbf{1.000{\pm}0.000}$ & $\mathbf{1.000{\pm}0.000}$ \\
        \midrule
        50
        & 2 
        & \multirow{6}{*}{\rotatebox{90}{$0.531{\pm}0.017$}} 
        & $0.632{\pm}0.031$ & $0.983{\pm}0.010$ & $0.933{\pm}0.022$ & $\mathbf{0.979{\pm}0.016}$ \\
        & 4 
        & & $0.863{\pm}0.035$ & $\mathbf{1.000{\pm}0.000}$ & $0.989{\pm}0.004$ & $0.993{\pm}0.007$ \\
        & 8 
        & & $0.971{\pm}0.023$ & $\mathbf{1.000{\pm}0.000}$ & $\mathbf{1.000{\pm}0.000}$ & $\mathbf{1.000{\pm}0.000}$ \\
        & 16 
        & & $\mathbf{1.000{\pm}0.000}$ & $\mathbf{1.000{\pm}0.000}$ & $\mathbf{1.000{\pm}0.000}$ & $\mathbf{1.000{\pm}0.000}$ \\
        & 32 
        & & $\mathbf{1.000{\pm}0.000}$ & $\mathbf{1.000{\pm}0.000}$ & $\mathbf{1.000{\pm}0.000}$ & $\mathbf{1.000{\pm}0.000}$ \\
        & 64
        & & $0.991{\pm}0.018$ & $\mathbf{1.000{\pm}0.000}$ & $\mathbf{1.000{\pm}0.000}$ & $\mathbf{1.000{\pm}0.000}$ \\
        \bottomrule
    \end{tabular}}
    \caption{%
    SVM classification accuracies (mean$\pm$std)
    for academic dataset with 500 samples of classes $\{1,2\}$
    and different numbers of training samples
    and Radon angles.}
    \label{tab:svm-academic}
\end{table}

\begin{table}[t]
    \resizebox{\linewidth}{!}{
        \footnotesize
        \begin{tabular}{@{} l @{\enspace} l @{\enspace} l @{\enspace} l @{\enspace} l @{\enspace} l @{\enspace} l @{}}
        \toprule
        \rotatebox{90}{train}
        & \rotatebox{90}{angle}
        & \rotatebox{90}{Eucl.}
        & R-CDT 
        & \mNRCDT
        & \miNRCDT
        & \tvNRCDT \\
        \midrule
        1
        & 2 
        & \multirow{6}{*}{\rotatebox{90}{$0.504{\pm}0.023$}}  
        & $0.505{\pm}0.030$ & $\mathbf{0.678{\pm}0.102}$ & $0.520{\pm}0.024$ & $0.518{\pm}0.046$ \\
        & 4 
        & & $0.496{\pm}0.027$ & $\mathbf{0.783{\pm}0.083}$ & $0.584{\pm}0.109$ & $0.657{\pm}0.099$ \\
        & 8 
        & & $0.505{\pm}0.023$ & $\mathbf{0.877{\pm}0.120}$ & $0.805{\pm}0.091$ & $0.784{\pm}0.145$ \\
        & 16 
        & & $0.502{\pm}0.025$ & $\mathbf{0.931{\pm}0.036}$ & $0.827{\pm}0.181$ & $0.837{\pm}0.099$ \\
        & 32 
        & & $0.510{\pm}0.019$ & $\mathbf{0.957{\pm}0.010}$ & $0.805{\pm}0.129$ & $0.836{\pm}0.185$ \\
        & 64
        & & $0.519{\pm}0.025$ & $\mathbf{0.952{\pm}0.011}$ & $0.872{\pm}0.120$ & $0.903{\pm}0.058$ \\
        \midrule
        5
        & 2 
        & \multirow{6}{*}{\rotatebox{90}{$0.504{\pm}0.018$}}  
        & $0.515{\pm}0.028$ & $\mathbf{0.717{\pm}0.082}$ & $0.528{\pm}0.047$ & $0.535{\pm}0.031$ \\
        & 4 
        & & $0.518{\pm}0.021$ & $\mathbf{0.821{\pm}0.064}$ & $0.709{\pm}0.084$ & $0.744{\pm}0.036$ \\
        & 8 
        & & $0.517{\pm}0.042$ & $\mathbf{0.937{\pm}0.031}$ & $0.873{\pm}0.042$ & $0.869{\pm}0.042$ \\
        & 16 
        & & $0.515{\pm}0.021$ & $\mathbf{0.949{\pm}0.021}$ & $0.911{\pm}0.045$ & $0.933{\pm}0.020$ \\
        & 32 
        & & $0.514{\pm}0.044$ & $\mathbf{0.954{\pm}0.011}$ & $0.944{\pm}0.018$ & $0.931{\pm}0.033$ \\
        & 64
        & & $0.507{\pm}0.028$ & $\mathbf{0.961{\pm}0.007}$ & $0.947{\pm}0.016$ & $0.926{\pm}0.085$ \\
        \midrule
        10
        & 2 
        & \multirow{6}{*}{\rotatebox{90}{$0.507{\pm}0.023$}}  
        & $0.511{\pm}0.026$ & $\mathbf{0.784{\pm}0.032}$ & $0.537{\pm}0.037$ & $0.562{\pm}0.034$ \\
        & 4 
        & & $0.514{\pm}0.026$ & $\mathbf{0.847{\pm}0.035}$ & $0.753{\pm}0.022$ & $0.745{\pm}0.045$ \\
        & 8 
        & & $0.510{\pm}0.031$ & $\mathbf{0.944{\pm}0.013}$ & $0.907{\pm}0.018$ & $0.871{\pm}0.083$ \\
        & 16 
        & & $0.528{\pm}0.032$ & $\mathbf{0.953{\pm}0.010}$ & $0.935{\pm}0.016$ & $0.937{\pm}0.029$ \\
        & 32 
        & & $0.534{\pm}0.030$ & $\mathbf{0.960{\pm}0.010}$ & $0.946{\pm}0.027$ & $0.948{\pm}0.027$ \\
        & 64
        & & $0.541{\pm}0.041$ & $\mathbf{0.956{\pm}0.019}$ & $0.955{\pm}0.020$ & $0.949{\pm}0.033$ \\
        \midrule
        25
        & 2 
        & \multirow{6}{*}{\rotatebox{90}{$0.517{\pm}0.025$}}  
        & $0.521{\pm}0.024$ & $\mathbf{0.799{\pm}0.017}$ & $0.569{\pm}0.034$ & $0.581{\pm}0.027$ \\
        & 4 
        & & $0.513{\pm}0.023$ & $\mathbf{0.860{\pm}0.028}$ & $0.768{\pm}0.013$ & $0.782{\pm}0.018$ \\
        & 8 
        & & $0.533{\pm}0.027$ & $\mathbf{0.946{\pm}0.015}$ & $0.909{\pm}0.010$ & $0.923{\pm}0.030$ \\
        & 16 
        & & $0.541{\pm}0.024$ & $\mathbf{0.959{\pm}0.011}$ & $0.950{\pm}0.009$ & $0.946{\pm}0.024$ \\
        & 32 
        & & $0.550{\pm}0.037$ & $\mathbf{0.962{\pm}0.010}$ & $0.960{\pm}0.012$ & $0.955{\pm}0.022$ \\
        & 64
        & & $0.556{\pm}0.027$ & $\mathbf{0.962{\pm}0.013}$ & $0.962{\pm}0.012$ & $0.961{\pm}0.021$ \\
        \midrule
        50
        & 2 
        & \multirow{6}{*}{\rotatebox{90}{$0.518{\pm}0.015$}}  
        & $0.535{\pm}0.025$ & $\mathbf{0.812{\pm}0.021}$ & $0.554{\pm}0.033$ & $0.591{\pm}0.028$ \\
        & 4 
        & & $0.528{\pm}0.022$ & $\mathbf{0.881{\pm}0.016}$ & $0.780{\pm}0.008$ & $0.793{\pm}0.011$ \\
        & 8 
        & & $0.552{\pm}0.027$ & $\mathbf{0.957{\pm}0.008}$ & $0.923{\pm}0.006$ & $0.937{\pm}0.016$ \\
        & 16 
        & & $0.567{\pm}0.032$ & $\mathbf{0.962{\pm}0.009}$ & $0.855{\pm}0.005$ & $0.956{\pm}0.016$ \\
        & 32 
        & & $0.571{\pm}0.029$ & $\mathbf{0.968{\pm}0.007}$ & $0.967{\pm}0.007$ & $0.967{\pm}0.007$ \\
        & 64
        & & $0.583{\pm}0.032$ & $\mathbf{0.968{\pm}0.009}$ & $0.967{\pm}0.009$ & $\mathbf{0.968{\pm}0.008}$ \\
        \bottomrule
    \end{tabular}}
    \caption{%
    SVM classification accuracies (mean$\pm$std)
    for LinMNIST dataset with 500 samples of classes $\{1,7\}$
    and different numbers of training samples
    and Radon angles.}
    \label{tab:svm-LinMNIST}
\end{table}

In this last set of numerical experiments
regarding image classification,
we investigate the performance of
our feature representations combined with
linear support vector machines (SVMs)~\cite{Fan2008}.
To this end,
we utilize the academic dataset
with $500$ samples of classes $\{1,2\}$
and the LinMNIST dataset
with $500$ samples of classes $\{1,7\}$.
We again randomly select
a subset of varying size
to train the SVMs
and use the remaining samples
of the dataset for testing.

The SVM classification accuracies
(mean $\pm$ std based on $20$ repetitions)
for the academic dataset
are reported in Table~\ref{tab:svm-academic},
where we use (up to) $256$ angles, $850$ radii
and $256$ interpolation points.
We observe that
the Euclidean embedding performs on the level
of random guessing regardless of the increasing
number of training samples.
R-CDT reaches perfect results,
but needs a lot of training data.
In contrast to this,
all our \hNRCDT{} representations
yield perfect classification results
already for {\em one} training sample per class,
provided the number of Radon angles
is sufficiently large.

The classification results
for the LinMNIST dataset
are reported in Table~\ref{tab:svm-LinMNIST},
where we use (up to) $256$ angles, $500$ radii
and $256$ interpolation points.
Now, both the Euclidean baseline and R-CDT
yield accuracies comparable to random guessing.
As opposed to this,
our new \hNRCDT{} representations
yield near perfect results,
even in the very challenging case of only
{\em one} training sample per class,
with \mNRCDT{} performing best.

\subsection{Image Clustering}
\label{sec:image_clustering}

While our experiments in the previous section
are all based on a supervised learning setting,
either by making use of known templates or
choosing some part of the dataset as training data,
we now transition to an unsupervised setting.
We consider the first three datasets from before,
now with $100$ samples per class,
and study unsupervised classification
in the sense that we first perform a cluster analysis
on a subset of the data
and, thereon, classify the remaining part
based on the corresponding cluster centres
in feature space.

As our theory predicts
linear separability
in \hNRCDT{} space,
we apply the classical
$k$-means algorithm
to cluster the first $50$ samples
of each class
in the respective feature space,
where $k$ is the expected number of classes.
Thereon, we use the corresponding cluster centres
to assign the remaining data to the closest cluster
with respect to the Euclidean norm in feature space.
The first step is referred to as training phase,
the second as test phase.
We evaluate the quality of
the assignments by the
so-called Rand index (RI)~\cite{Hubert1985}
and the variational information (VI)~\cite{Meila2003},
which we now introduce in more details.

\begin{figure*}[t]
    \resizebox{\linewidth}{!}{
    \begin{tabular}{c c c}
        \multicolumn{3}{c}{%
        class:\;
        \Circle~1 \;
        \Star~2 \;
        \Triangle~3 
        \qquad
        centre:\;
        \Boxblue~1 \;
        \Boxred~2 \;
        \Boxgreen~3 \;
        \qquad
        train/test:\;
        \UnBoxblue/\UnsBoxblue~1 \;
        \UnBoxred/\UnsBoxred~2 \;
        \UnBoxgreen/\UnsBoxgreen~3
        } \\[1ex]
        Eucl. & R-CDT & \mNRCDT \\
        \multirow[t]{4}{*}{%
        \raisebox{-85pt}[0pt][0pt]{%
        \includegraphics[width=0.5\linewidth, clip=true, trim=10pt 0pt 7pt 0pt]{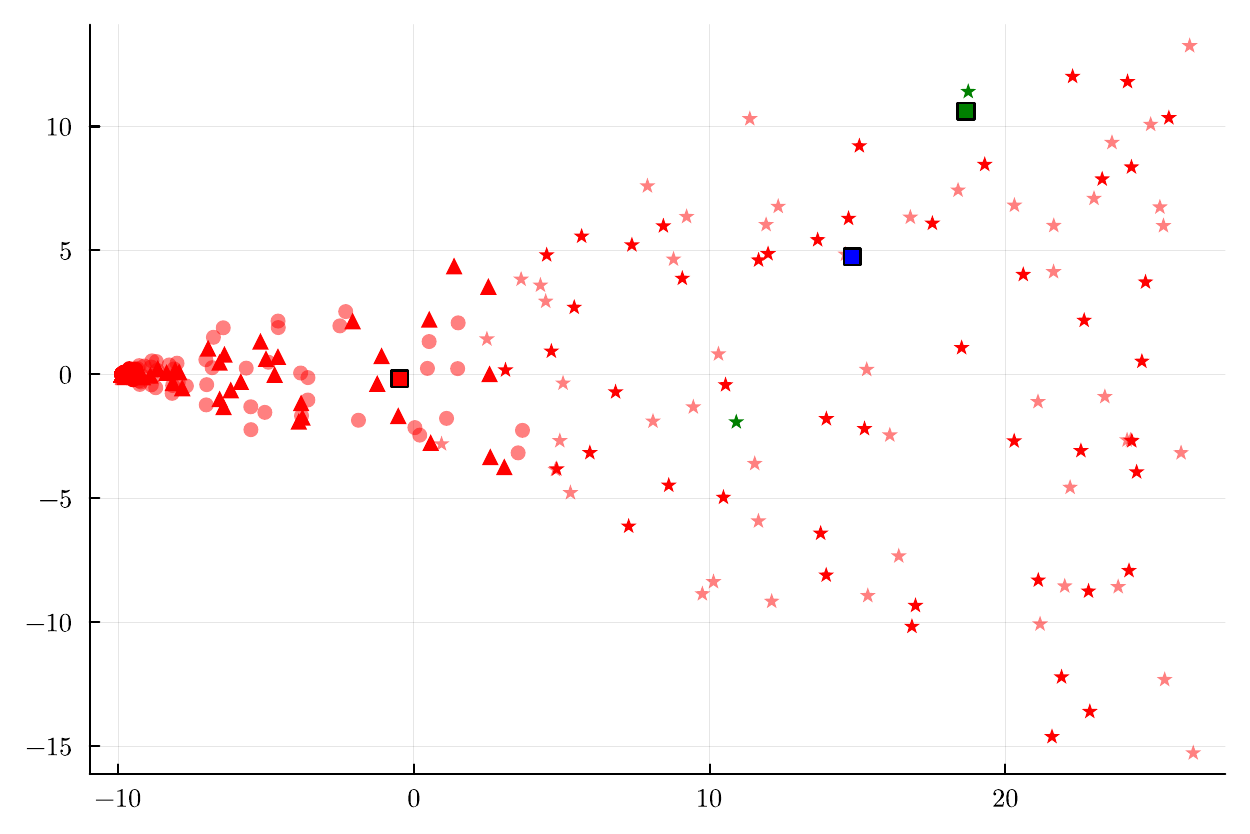}}} 
        & 
        \multirow[t]{4}{*}{%
        \raisebox{-85pt}[0pt][0pt]{
        \includegraphics[width=0.5\linewidth, clip=true, trim=10pt 0pt 7pt 0pt]{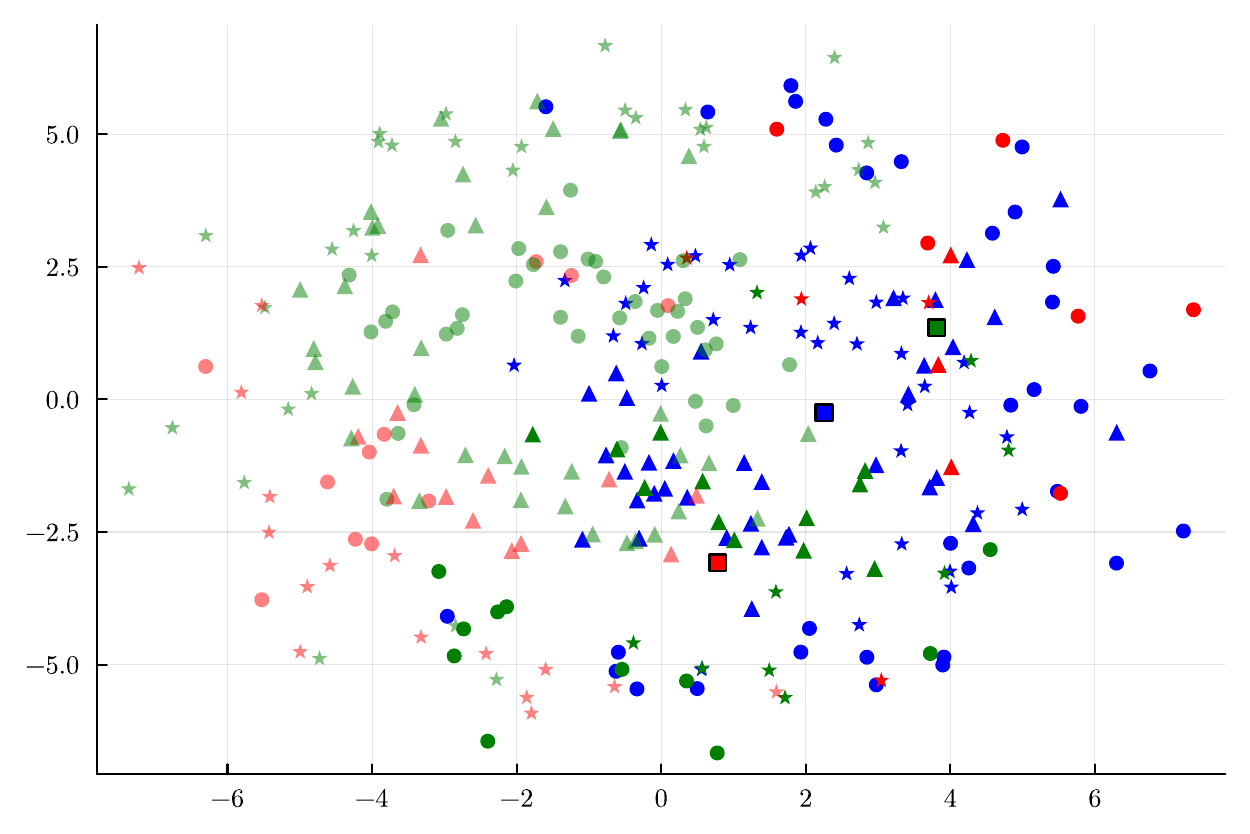}}}
        & \includegraphics[width=0.22\linewidth, clip=true, trim=17pt 0pt 7pt 0pt]{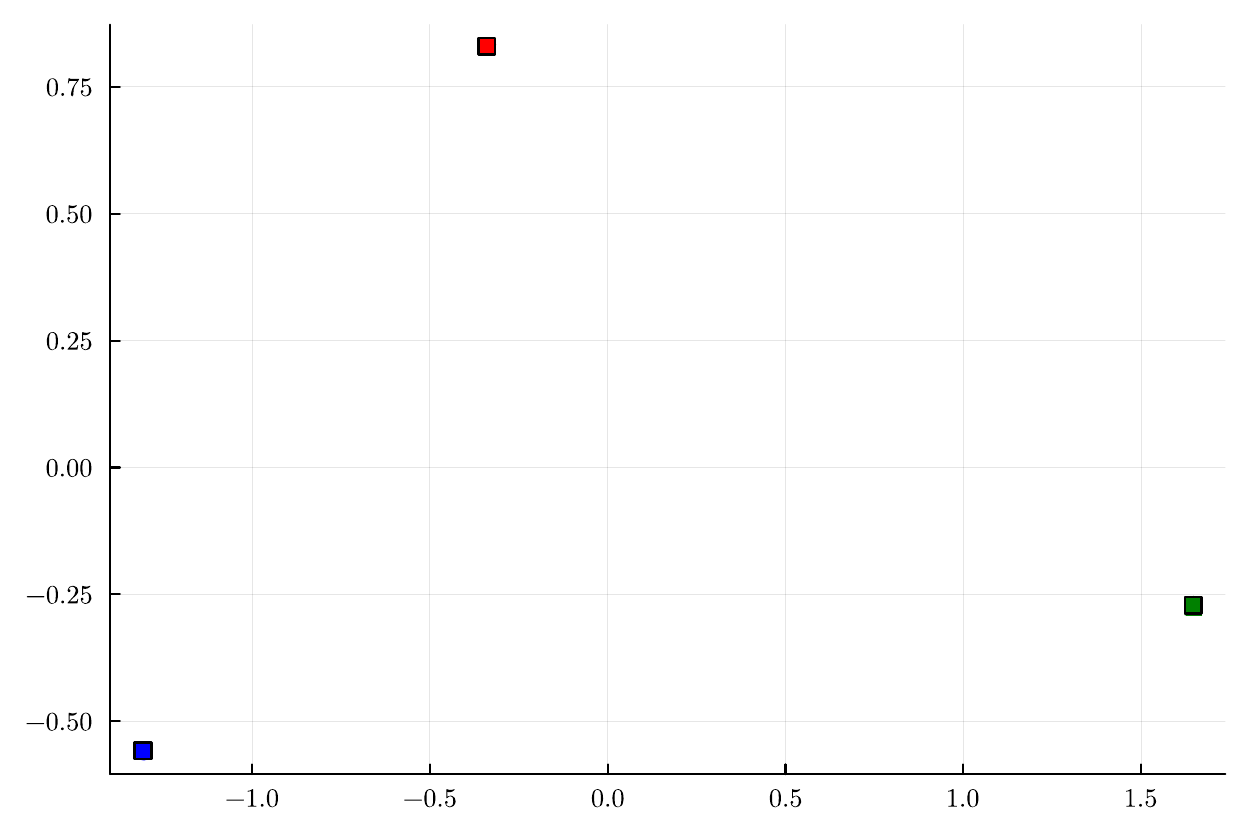} \\
        & & \tvNRCDT \\
        & & \includegraphics[width=0.22\linewidth, clip=true, trim=-15pt 0pt 10pt 0pt]{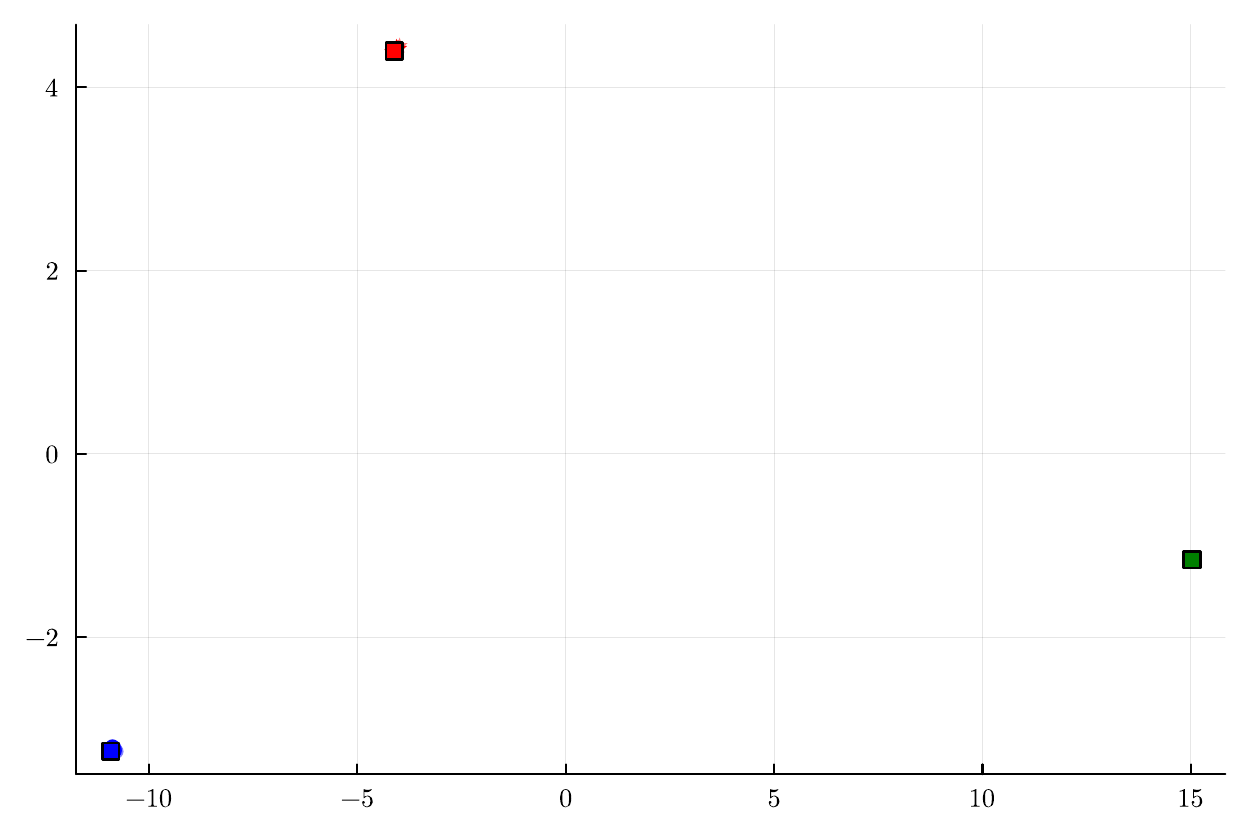} \\
    \end{tabular}
    }
    \caption{$3$-means cluster visualization 
    for the academic dataset
    using a 2d PCA in
    the respective feature space.}
    \label{fig:kmeansAcademic}
\end{figure*}

Let $[n] \coloneqq \{1,2,...,n\}$ and consider
two partitions $\U = \{U_1,...,U_k\}$
and $\V = \{V_1, ...., V_k\}$ of $[n]$.
Then, the Rand index is defined as
\begin{equation*}
    \RI(\U, \V)
    \coloneqq \frac{a+b}{\binom{n}{2}}
\end{equation*}
with the number of accordance
\begin{equation*}
    a = |\{(x,y) \in [n]^2 \mid \exists \, s,t \in [k] \colon x,y \in U_s \cap V_t\}|
\end{equation*}
and the number of distinction
\begin{equation*}
    b = |\{(x,y) \in [n]^2 \mid \nexists \, s,t \in [k] \colon x,y \in U_s \cup V_t\}|.
\end{equation*}
Intuitively,
RI measures the similarity between two clusterings
and can be seen as prediction accuracy.
To introduce the variational information,
let
\begin{equation*}
    p_{st} \coloneqq \frac{|U_s \cap V_t|}{n},
    \quad p_{s\cdot} \coloneqq \frac{|U_s|}{n},
    \quad p_{\cdot t} \coloneqq \frac{|V_t|}{n}
\end{equation*}
for $(s,t) \in [k]^2$ and
consider the entropies
\begin{equation*}
    \begin{aligned}
    H(\U)
    &\coloneqq - \sum_{s = 1}^k p_{s\cdot} \log(p_{s\cdot}),
    \\
    H(\V)
    &\coloneqq - \sum_{t = 1}^k p_{\cdot t} \log(p_{\cdot t})
    \end{aligned}
\end{equation*}
as well as the mutual information
\begin{equation*}
    I(\mathcal U, \mathcal V)
    \coloneqq \sum_{s,t = 1}^k p_{st} \log\bigl(\tfrac{p_{st}}{p_{s\cdot}p_{\cdot t}}\bigr).
\end{equation*}
Then, the variational information is defined as
\begin{equation*}
    \VI(\U, \V)
    \coloneqq H(\U) + H(\V) - 2 I(\U, \V)
\end{equation*}
and quantifies the amount of information
that is lost or gained
when changing from $\U$ to $\V$.
In our experiments,
we compare the clustering induced by the true labels
with the $k$-means clustering results.
For visualization,
we make use of 2d or 3d principal component analysis (PCA),
where the true classes are visualized by different symbols 
(e.g. \Circle, \Star, \Triangle, ...)
and the $k$-means cluster centres
by coloured boxes  
(e.g. \Boxblue, \Boxred, \Boxgreen, ...).
The cluster members
are coloured opaquely
(e.g. \UnBoxblue, \UnBoxred, \UnBoxgreen, ...)
and the test data transparently
(e.g. \UnsBoxblue, \UnsBoxred, \UnsBoxgreen, ...)
according to the nearest cluster centre.

\begin{table}[t]
    \centering%
    \footnotesize%
    \begin{tabular}{l @{\enspace} l @{\quad} c @{\quad} c @{\quad} c @{\quad} c}
        \toprule 
        & & Eucl.
        & R-CDT
        & \mNRCDT 
        & \tvNRCDT  \\
        \midrule
        RI$_{\textrm{train}}$ & ($\uparrow$) & $0.3378$ & $0.5486$ & $\mathbf{1.0000}$ & $\mathbf{1.0000}$\\
        VI$_{\textrm{train}}$ & ($\downarrow$) & $1.1492$ & $2.1676$ & $\mathbf{0.0000}$ & $\mathbf{0.0000}$\\
        RI$_{\textrm{test}}$ & ($\uparrow$) & $0.3288$ & $0.5537$ & $\mathbf{1.0000}$ & $\mathbf{1.0000}$\\
        VI$_{\textrm{test}}$ & ($\downarrow$) & $1.0986$ & $2.1207$ & $\mathbf{0.0000}$ & $\mathbf{0.0000}$\\
        \bottomrule
    \end{tabular}
    \caption{Quality measures
    for $3$-means clustering 
    of the academic dataset with $100$ images per class,
    where $50$ images are used for training 
    and the rest for testing.}
    \label{tab:kmeansAcademic}
\end{table}

In all our numerical experiments,
we use $850$ radii and $256$ angles
for the Radon transform
and $256$ interpolation points for the CDT.
The results for the academic dataset
are reported in Table~\ref{tab:kmeansAcademic}.
We observe that our \hNRCDT{} feature representations
clearly outperform the Euclidean baseline and
R-CDT embedding,
yielding perfect RI and VI values
on both the train and test set.
Moreover,
the cluster visualization in Figure~\ref{fig:kmeansAcademic}
indicates that the classes are indeed
mapped to single points in \hNRCDT{} space,
as predicted by our theory.

\begin{table}[t]
    \centering%
    \footnotesize%
    \begin{tabular}{l @{\enspace} l @{\quad} c @{\quad} c @{\quad} c @{\quad} c}
        \toprule 
        & & Eucl.
        & R-CDT
        & \mNRCDT 
        & \tvNRCDT  \\
        \midrule
        RI$_{\textrm{train}}$ & $(\uparrow)$ & $0.1367$ & $0.7406$ & $0.8602$ & $\mathbf{0.9752}$\\
        VI$_{\textrm{train}}$ & $(\downarrow)$ & $2.2449$ & $3.8193$ & $1.6014$ & $\mathbf{0.2853}$\\
        RI$_{\textrm{test}}$ & $(\uparrow)$ & $0.1091$ & $0.7130$ & $0.8579$ & $\mathbf{0.9752}$\\
        VI$_{\textrm{test}}$ & $(\downarrow)$ & $2.1972$ & $3.6659$ & $1.6275$ & $\mathbf{0.2706}$\\
        \bottomrule
    \end{tabular}
    \caption{Quality measures
    for $9$-means clustering 
    of the polygon dataset with $100$ images per class,
    where $50$ images are used for training 
    and the rest for testing.}
    \label{tab:kmeansPolygons}
\end{table}

In case of the polygon dataset
we observe that \tvNRCDT{} outperforms
even the \mNRCDT{} with a
near perfect Rand index of $97\%$,
see Table~\ref{tab:kmeansPolygons}.
In contrast to this,
\mNRCDT{} yields a Rand index of $86\%$,
which is still significantly better than
the Euclidean baseline and R-CDT.
An inspection of the cluster visualization
in Figure~\ref{fig:kmeansPolygons} reveals that
now the classes are not mapped to
singletons in \hNRCDT{} space.
This, however,
is to be expected
as the polygon dataset includes
small non-affine perturbations.
Nevertheless,
the classes are mapped to
well-separated sets,
which is more pronounced for
\tvNRCDT{} than \mNRCDT{}.
This again indicates the potential
of including the whole angular information
instead of only considering the maximum.

\begin{table}[t]
    \centering%
    \footnotesize%
    \begin{tabular}{l @{\enspace} l @{\quad} c @{\quad} c @{\quad} c @{\quad} c}
        \toprule 
        & & Eucl.
        & R-CDT
        & \mNRCDT 
        & \tvNRCDT  \\
        \midrule
        RI$_{\textrm{train}}$ & $(\uparrow)$ & $0.4832$ & $0.5517$ & $0.8385$ & $\mathbf{0.8526}$\\
        VI$_{\textrm{train}}$ & $(\downarrow)$ & $1.2223$ & $2.1840$ & $0.8303$ & $\mathbf{0.7978}$\\
        RI$_{\textrm{test}}$ & $(\uparrow)$ & $0.4826$ & $0.5489$ & $0.8488$ & $\mathbf{0.8798}$\\
        VI$_{\textrm{test}}$ & $(\downarrow)$ & $1.2255$ & $2.1577$ & $0.8293$ & $\mathbf{0.6713}$\\
        \bottomrule
    \end{tabular}
    \caption{Quality measures
    for $3$-means clustering 
    of the LinMNIST dataset 
    with $100$ images per class,
    where $50$ images are used for training
    and the rest for testing.}
    \label{tab:kmeansLinMNIST}
\end{table}

Finally, for the LinMNIST dataset
\tvNRCDT{} and \mNRCDT{} yield comparable results
and perform significantly better than
the Euclidean and R-CDT embedding,
which are nearly on the same level,
cf.\ Table~\ref{tab:kmeansLinMNIST}.
The cluster visualizations in Figure~\ref{fig:kmeansLinMNIST}
show that \mNRCDT{} and \tvNRCDT{}
yield nicely separated clusters
with larger distances in \tvNRCDT{} space,
which might explain the slightly better results
for \tvNRCDT{}.

\begin{figure*}
    \centering%
    \footnotesize%
    \begin{tabular}{c c}
        \multicolumn{2}{c}{%
        class:\;
        \Circle~1 \;
        \Star~2 \;
        \Triangle~3 
        \qquad
        centre:\;
        \Boxblue~1 \;
        \Boxred~2 \;
        \Boxgreen~3 \;
        \qquad
        train/test:\;
        \UnBoxblue/\UnsBoxblue~1 \;
        \UnBoxred/\UnsBoxred~2 \;
        \UnBoxgreen/\UnsBoxgreen~3
        } \\[0.5ex]
        \multicolumn{2}{c}{%
        \textcolor{white}{class:}\;
        \Ubox~4 \;
        \Utriangle~5 \;
        \Pentagon~6 
        \qquad
        \textcolor{white}{centre:}\;
        \Boxyellow~4 \;
        \Boxpurple~5 \;
        \Boxbrown~6 \;
        \qquad
        \textcolor{white}{train/test:}\;
        \UnBoxyellow/\UnsBoxyellow~4 \;
        \UnBoxpurple/\UnsBoxpurple~5 \;
        \UnBoxbrown/\UnsBoxbrown~6
        } \\[0.5ex]
        \multicolumn{2}{c}{%
        \textcolor{white}{class:}\;
        \hspace{2pt}\Ltriangle\hspace{1pt}~7 \;
        \hspace{2pt}\Rtriangle~8 \;
        \hspace{-1pt}\Hexagon~9
        \qquad
        \textcolor{white}{centre:}\;
        \Boxorange~7 \;
        \Boxpink~8 \;
        \Boxcyan~9 \;
        \qquad
        \textcolor{white}{train/test:}\;
        \UnBoxorange/\UnsBoxorange~7 \;
        \UnBoxpink/\UnsBoxpink~8 \;
        \UnBoxcyan/\UnsBoxcyan~9
        } \\[1em]
        Eucl. & R-CDT \\
        \includegraphics[width=0.45\linewidth, clip=true, trim=130pt 10pt 120pt 10pt]{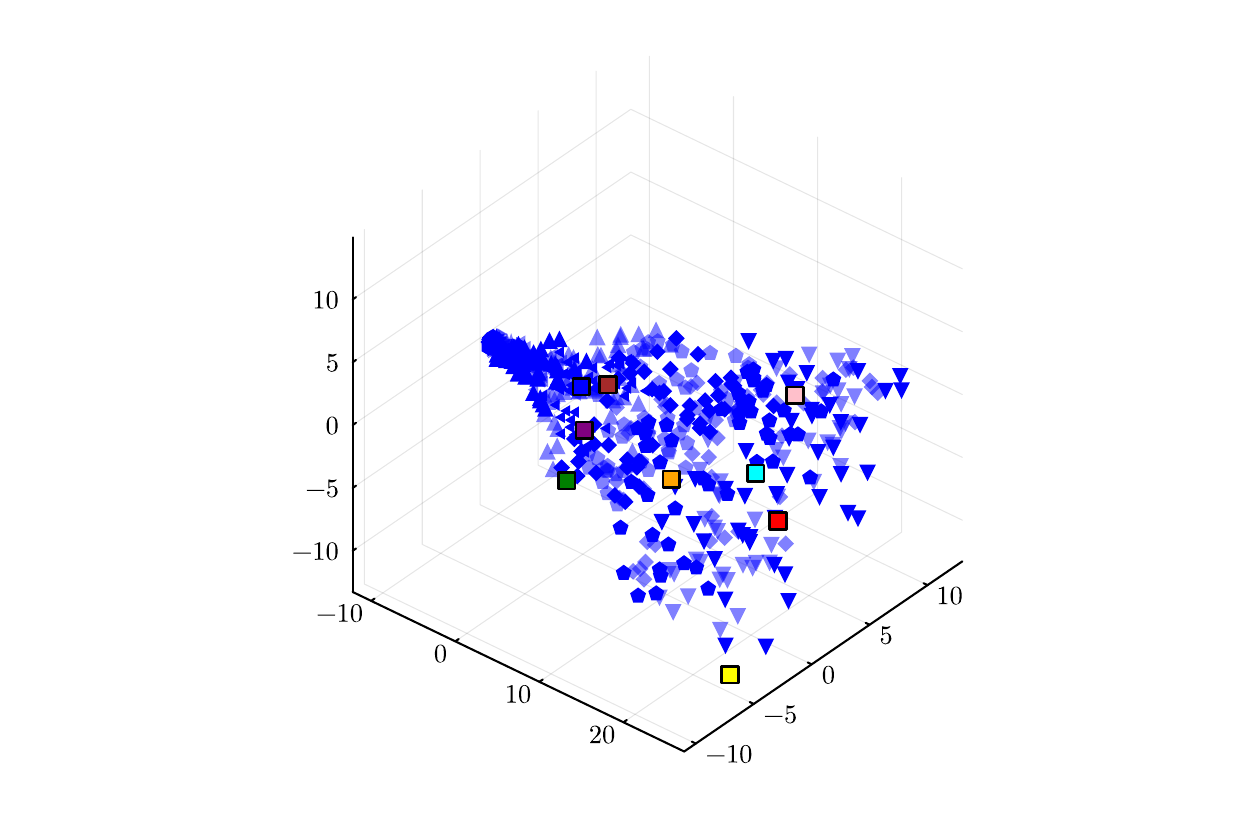} 
        & \includegraphics[width=0.45\linewidth, clip=true, trim=130pt 10pt 120pt 10pt]{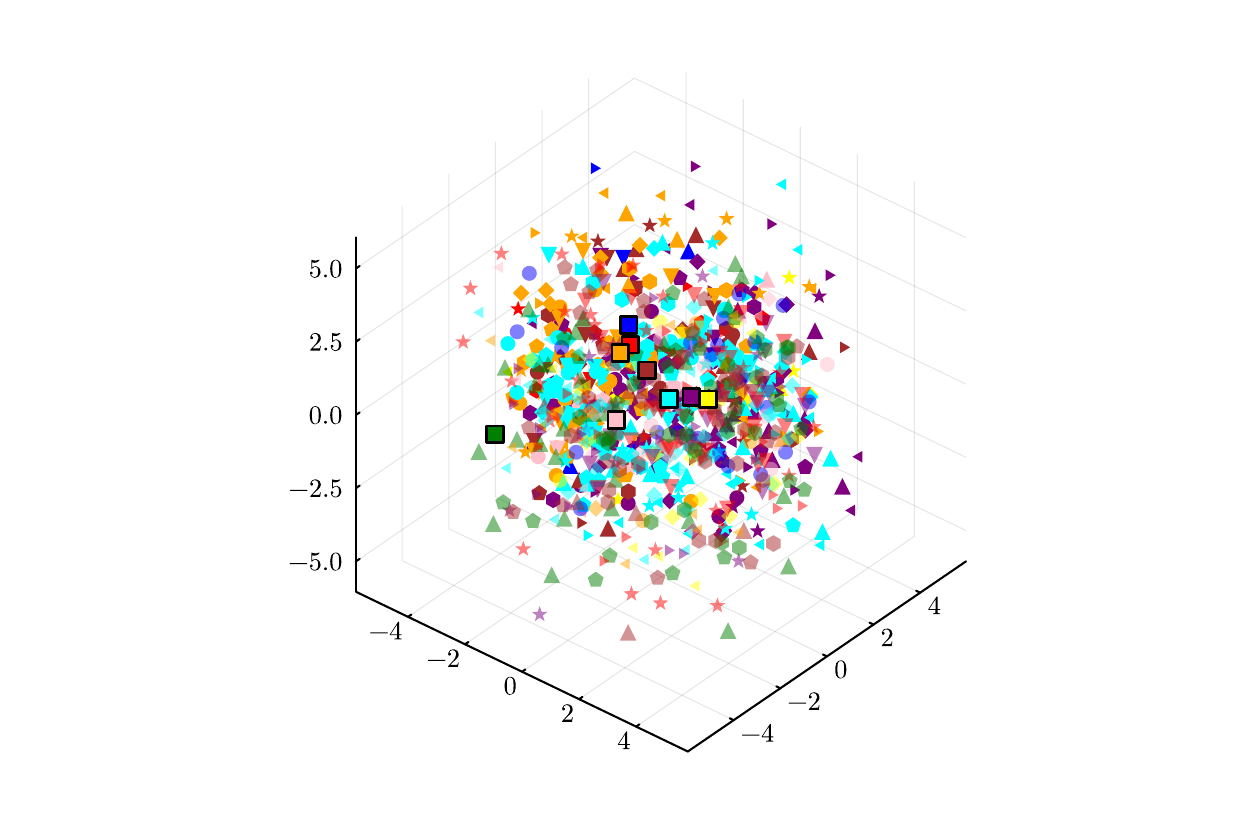}\\
        \mNRCDT  & \tvNRCDT \\
        \includegraphics[width=0.45\linewidth, clip=true, trim=130pt 10pt 120pt 10pt]{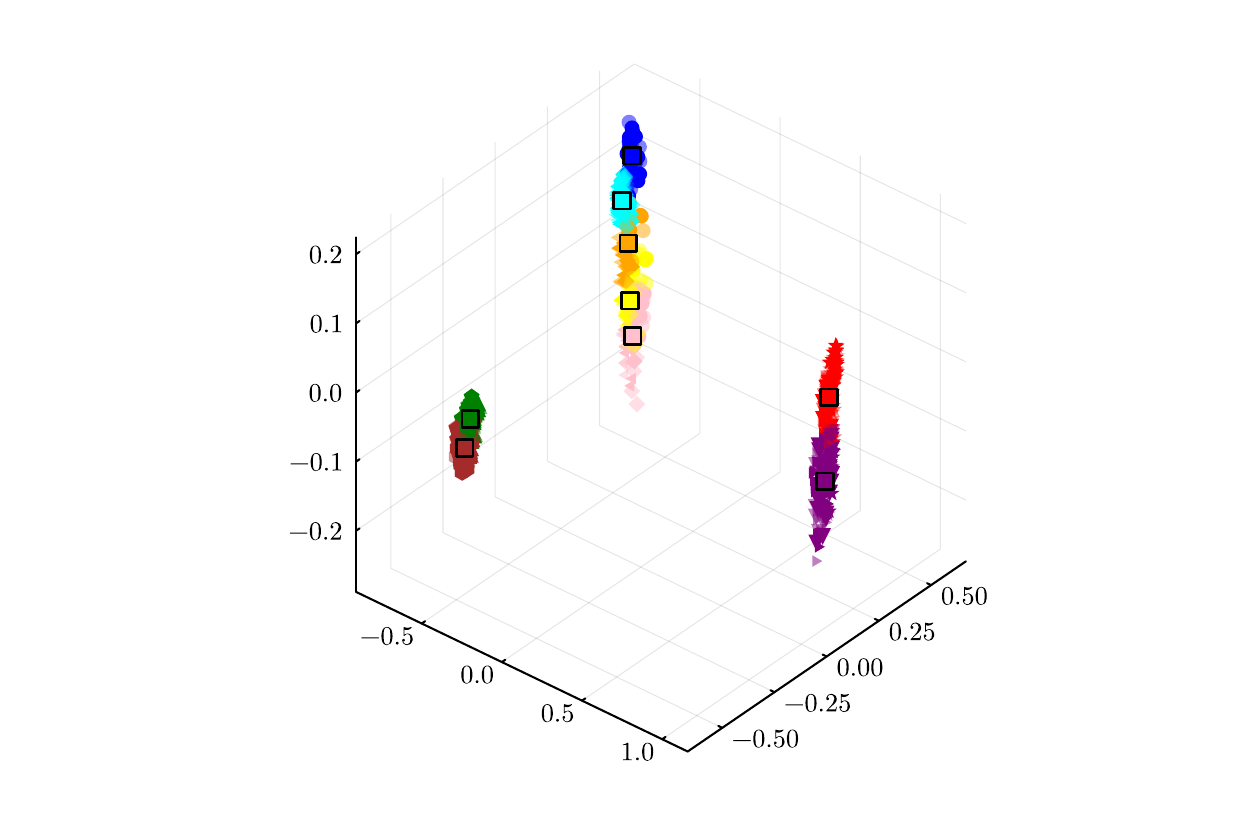}  
        & \includegraphics[width=0.45\linewidth, clip=true, trim=130pt 10pt 120pt 10pt]{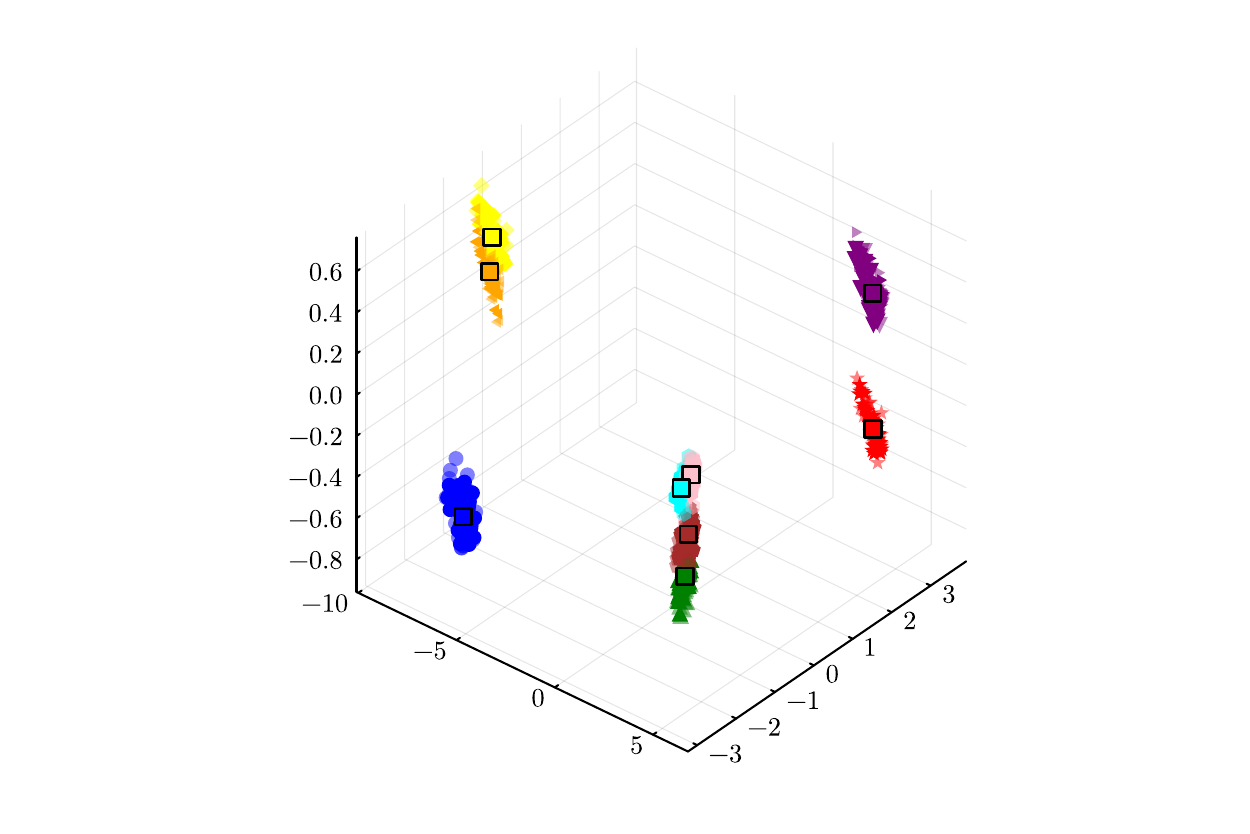} 
    \end{tabular}
    \caption{$9$-means cluster visualization 
    for the polygon dataset
    using a 3d PCA in
    the respective feature space.}
    \label{fig:kmeansPolygons}
    \vspace*{3cm}
\end{figure*}

\begin{figure*}
    \centering%
    \footnotesize%
    \begin{tabular}{c c}
        \multicolumn{2}{c}{%
        class:\;
        \Circle~1 \;
        \Star~5 \;
        \Triangle~7 
        \qquad
        centre:\;
        \Boxblue~1 \;
        \Boxred~5 \;
        \Boxgreen~7 \;
        \qquad
        train/test:\;
        \UnBoxblue/\UnsBoxblue~1 \;
        \UnBoxred/\UnsBoxred~5 \;
        \UnBoxgreen/\UnsBoxgreen~7
        } \\[1ex]
        Eucl. & R-CDT \\
        \includegraphics[width=0.45\linewidth, clip=true, trim=7pt 0pt 7pt 0pt]{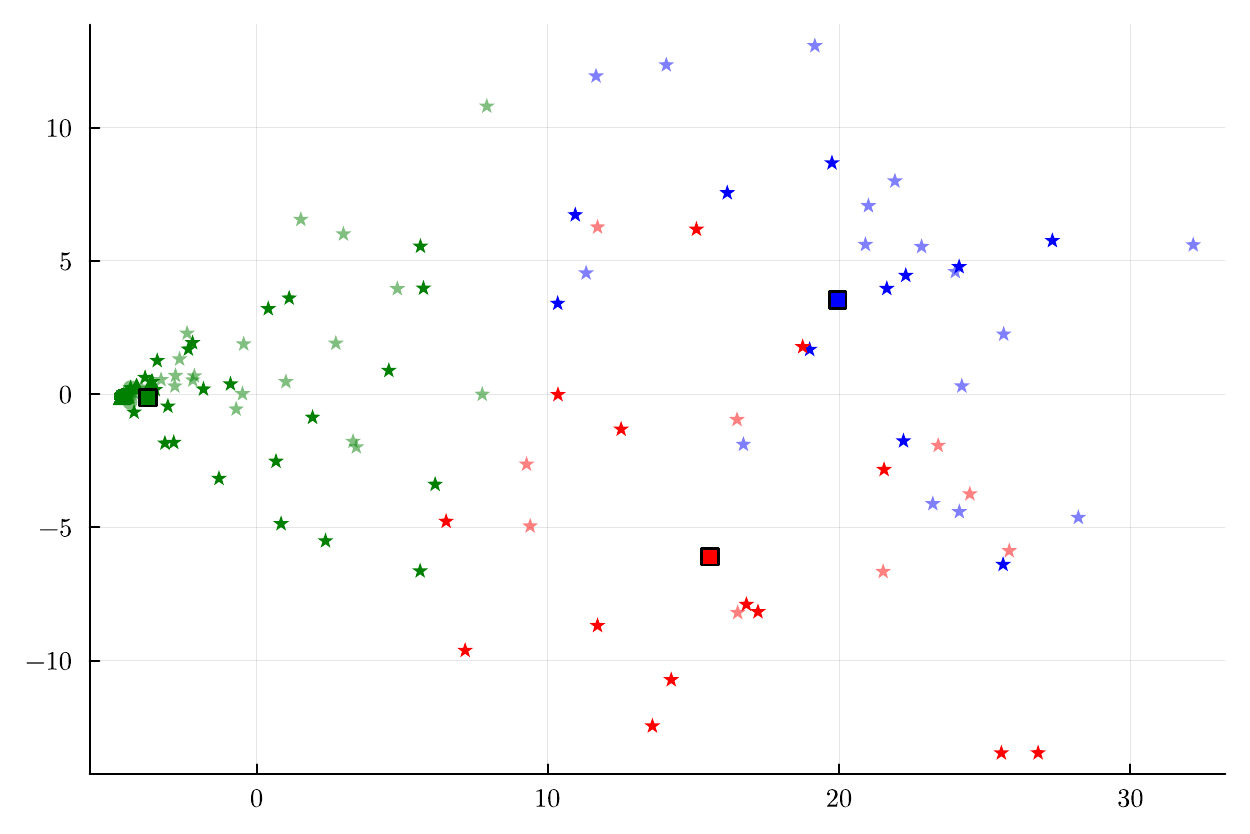} 
        & \includegraphics[width=0.45\linewidth, clip=true, trim=10pt 0pt 7pt 0pt]{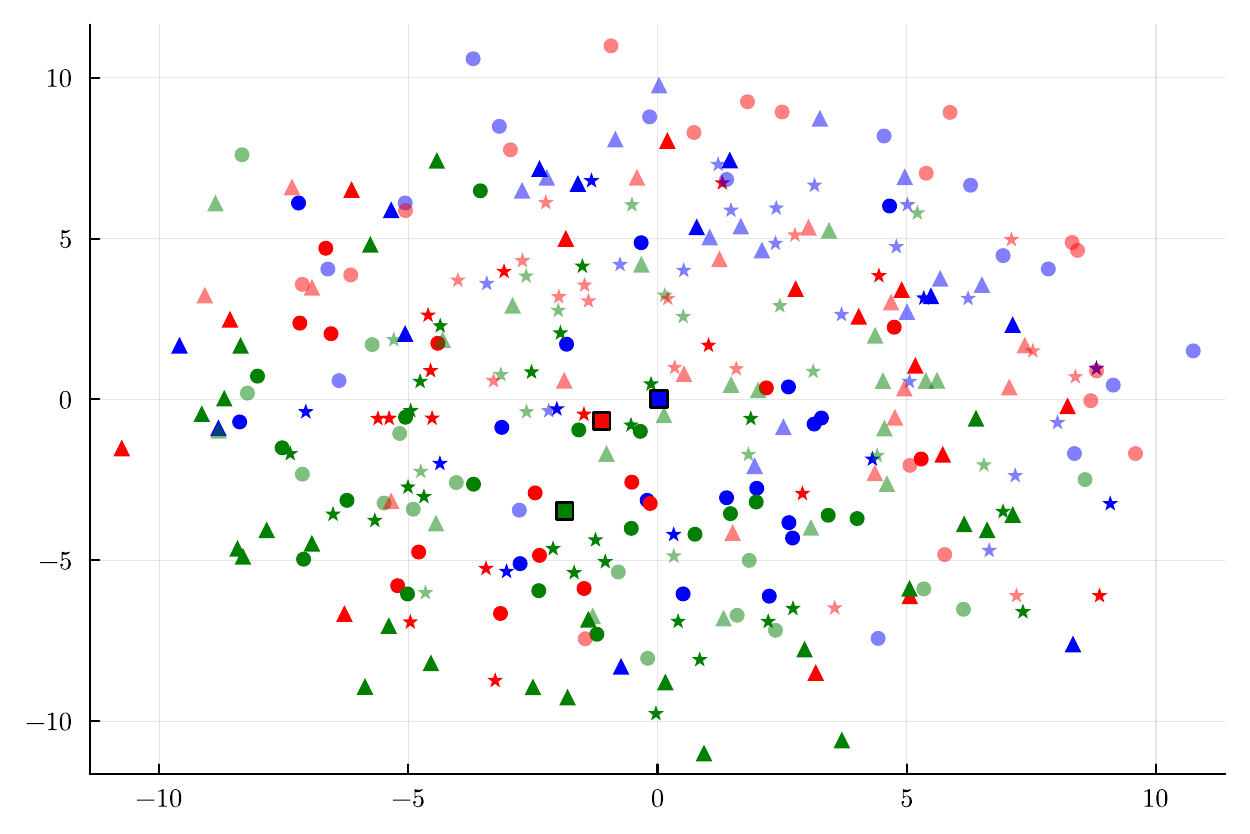} \\
        \mNRCDT & \tvNRCDT \\
        \includegraphics[width=0.45\linewidth, clip=true, trim=10pt 0pt 7pt 0pt]{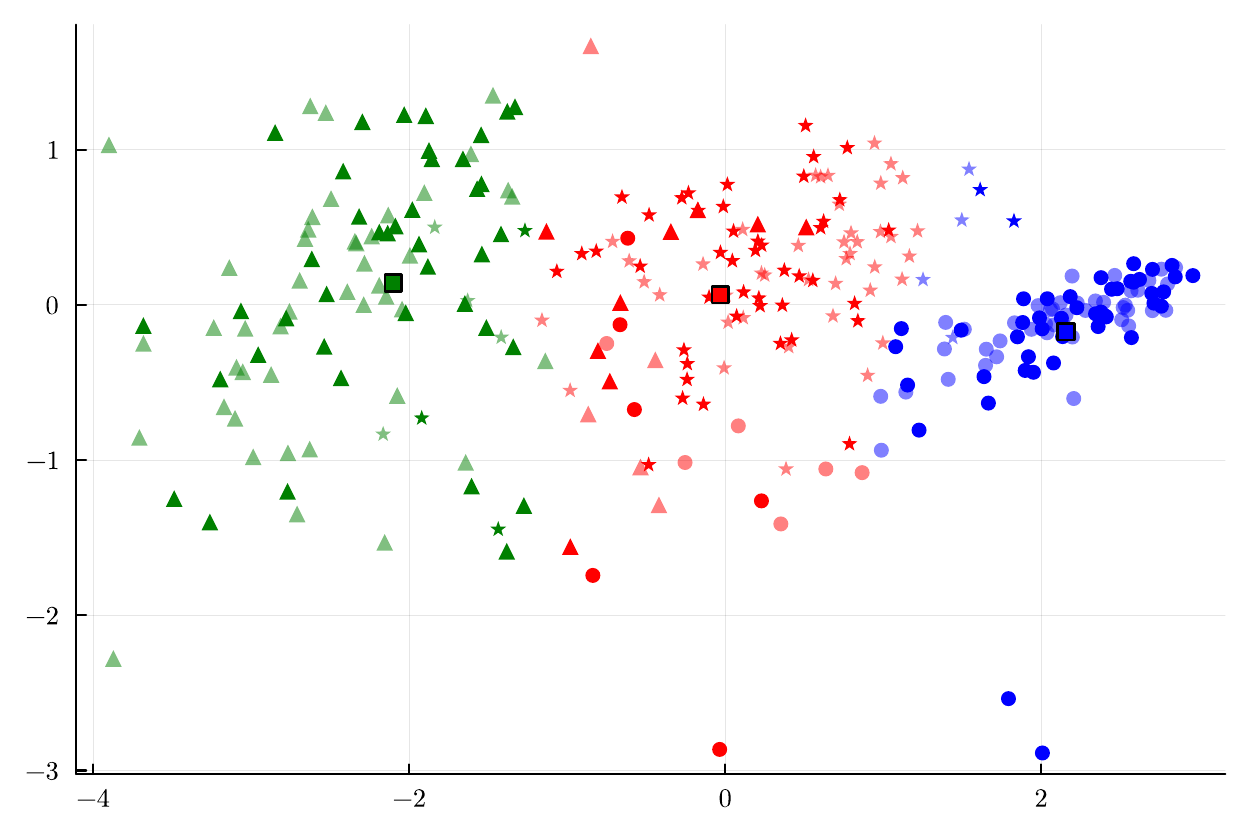}
        & \includegraphics[width=0.45\linewidth, clip=true, trim=10pt 0pt 7pt 0pt]{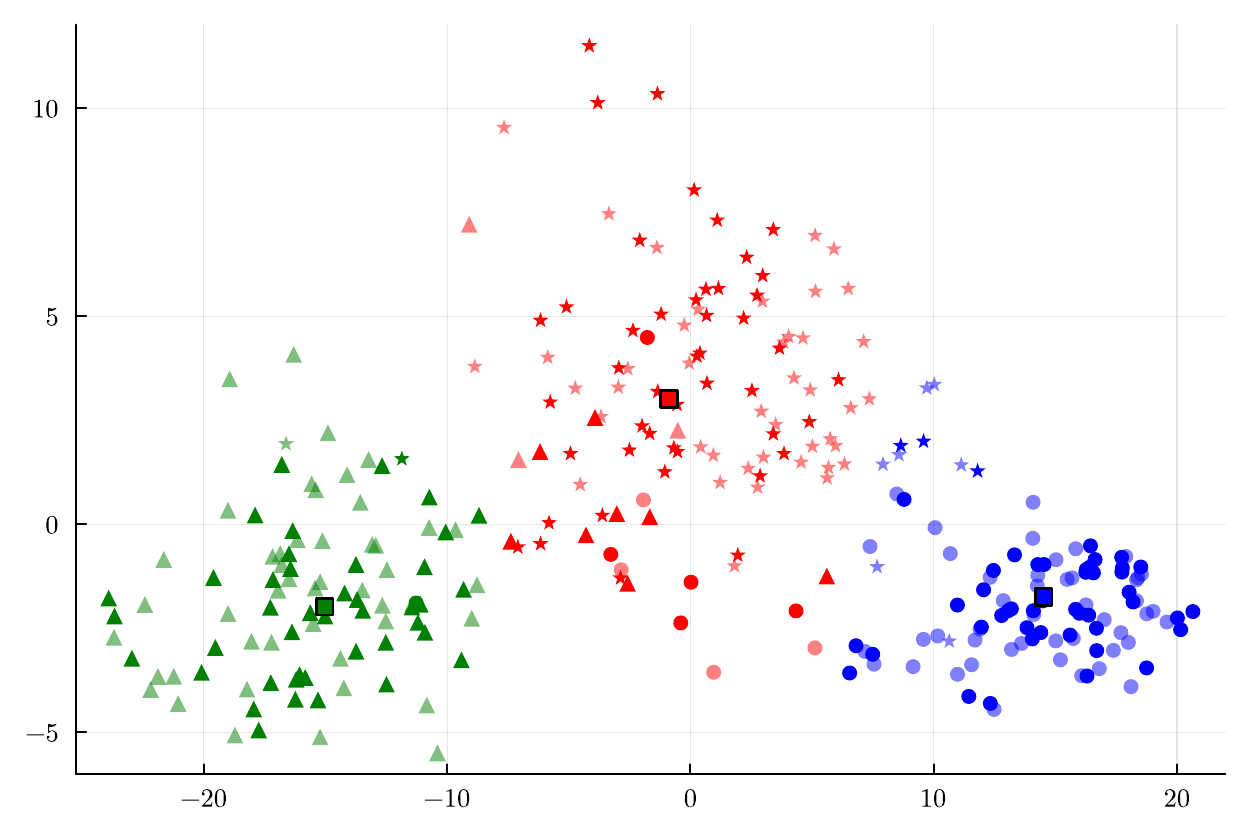}
    \end{tabular}
    \caption{$3$-means cluster visualization 
    for the LinMNIST dataset
    using a 2d PCA in
    the respective feature space.}
    \label{fig:kmeansLinMNIST}
\end{figure*}

\subsection{Multi-dimensional Classification}
\label{sec:gen_hNRCDTs}

In the final set of numerical experiments,
we leave the 2d Euclidean setting
and explore the aptitude of our generalized \hNRCDT{}s
regarding classification tasks 
on $\XX = \R^3$ and $\XX = \SO(3)$.
For this,
we consider classes of empirical measures 
\begin{equation*}
    \mu \coloneqq \frac{1}{K} \sum_{k=1}^K \delta_{\bfx_k},
    \quad
    \bfx_k \in \XX,
\end{equation*}
whose generalized restricted Radon transforms are analytically given by
\begin{equation*}
    \Radon_{\phi,\bftheta}[\mu]
    =
    \frac{1}{K} \sum_{k=1}^K \delta_{\phi(\bfx_k, \bftheta)}.
\end{equation*}
Then, the corresponding quantile functions are piecewise constant
and these are sampled on an equispaced grid
to calculate the final \hNRCDT{}s.

\subsubsection{3D Object and Pose Recognition}
\label{sssec:3d_object_recognition}

The first generalized \hNRCDT{}
we study numerically
is the multi-dimensional extension 
in §~\ref{sec:multi-NRCDT},
which can be applied to the recognition of 3d objects.
For first proof-of-concept experiments,
we consider the following three datasets:
\begin{enumerate}[label=\alph*)]
    \item \textbf{Animal dataset}.
    We use the 3d triangular mesh information of six animals%
    ---lion, cat, camel, horse, flamingo, elephant---%
    provided in \cite{SumnerWeb,Sumner2004},
    where each mesh consists of 5000 vertices.
    More precisely,
    we equip these vertices 
    with a uniform probability distribution
    to obtain empirical measures on $\R^3$.
    On the basis of these six templates,
    the dataset itself is generated by
    applying 10 random affine transformations
    per template,
    consisting of anisotropic scaling factors in $[0.5, 1.0]$
    and shearing in $[-15^\circ, 15 ^\circ]$ 
    for each coordinate direction
    as well as
    random 3d rotations
    and shifts in $[-25,25]^3$.
    The animal dataset is illustrated 
    in  Figure~\ref{fig:animal_dataset_3d}.       
    \begin{figure}[t]
        \footnotesize%
        \centering%
        \begin{tabular}{c @{\hspace{5pt}} c @{\hspace{5pt}} c}
        lion & cat & camel 
        \\
        {\includegraphics[width=0.3\linewidth, clip=true, trim=280pt 150pt 300pt 150pt]{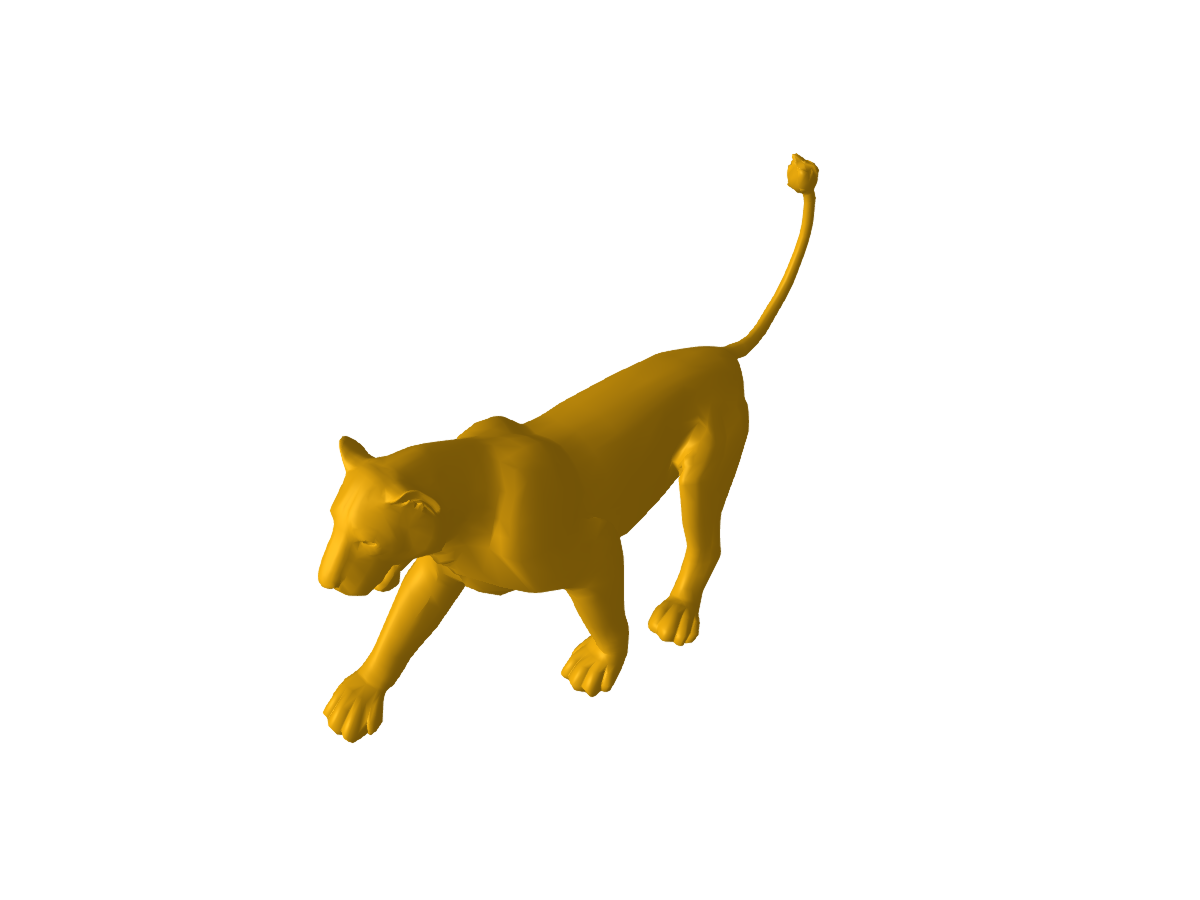}}%
        & {\includegraphics[width=0.3\linewidth, clip=true, trim=280pt 150pt 300pt 150pt]{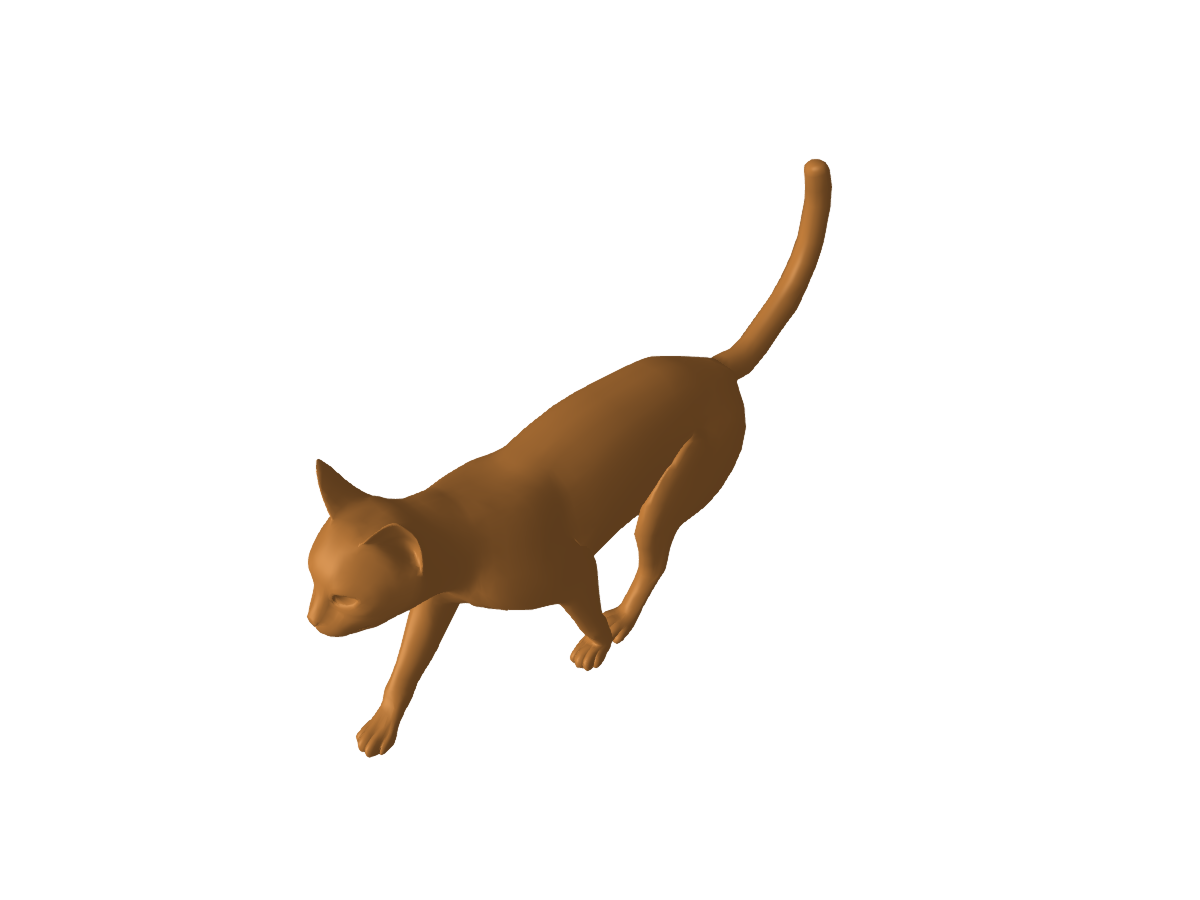}}%
        & {\includegraphics[width=0.3\linewidth, clip=true, trim=280pt 150pt 300pt 150pt]{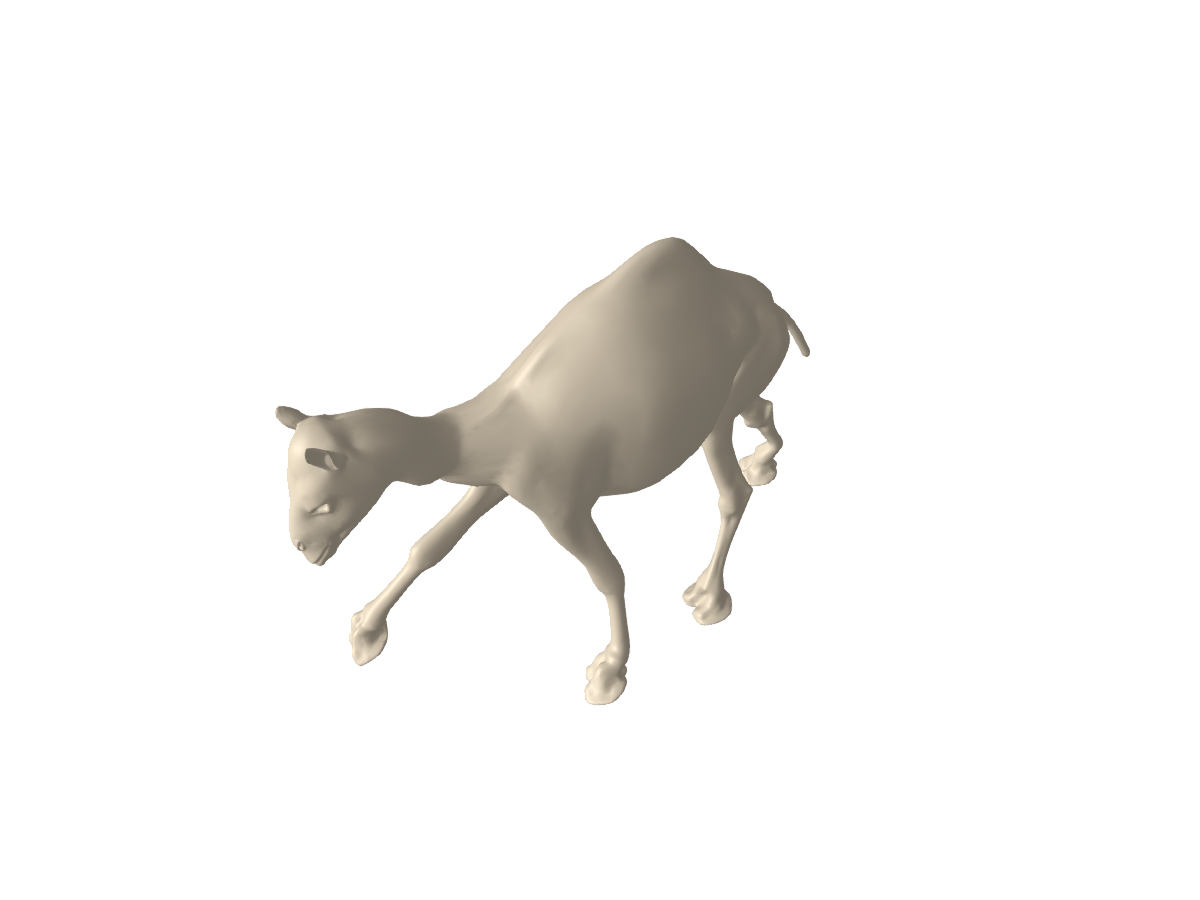}}%
        \\
        {\includegraphics[width=0.15\linewidth, clip=true, trim=280pt 150pt 280pt 150pt]{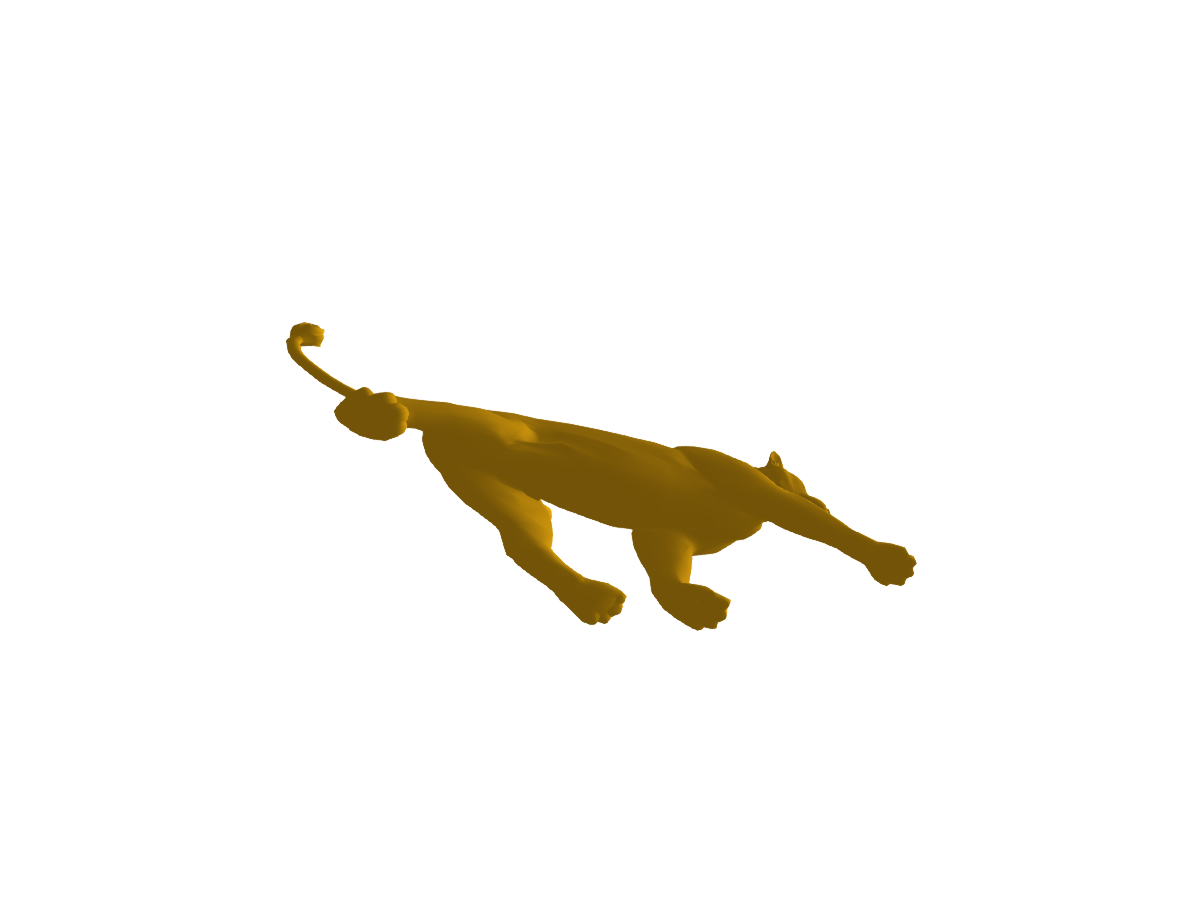}}%
        \hfill%
        {\includegraphics[width=0.15\linewidth, clip=true, trim=280pt 200pt 240pt 200pt]{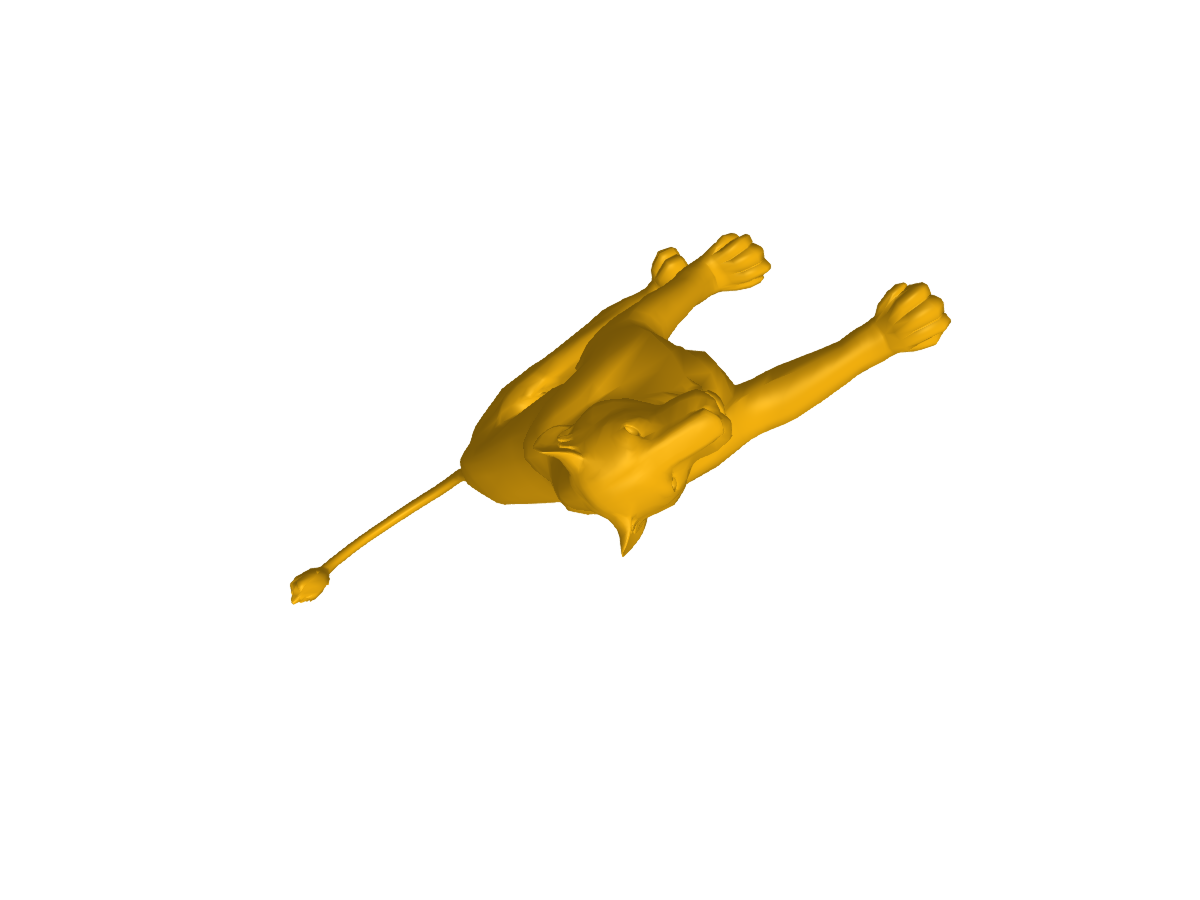}}%
        & {\includegraphics[width=0.15\linewidth, clip=true, trim=280pt 80pt 280pt 50pt]{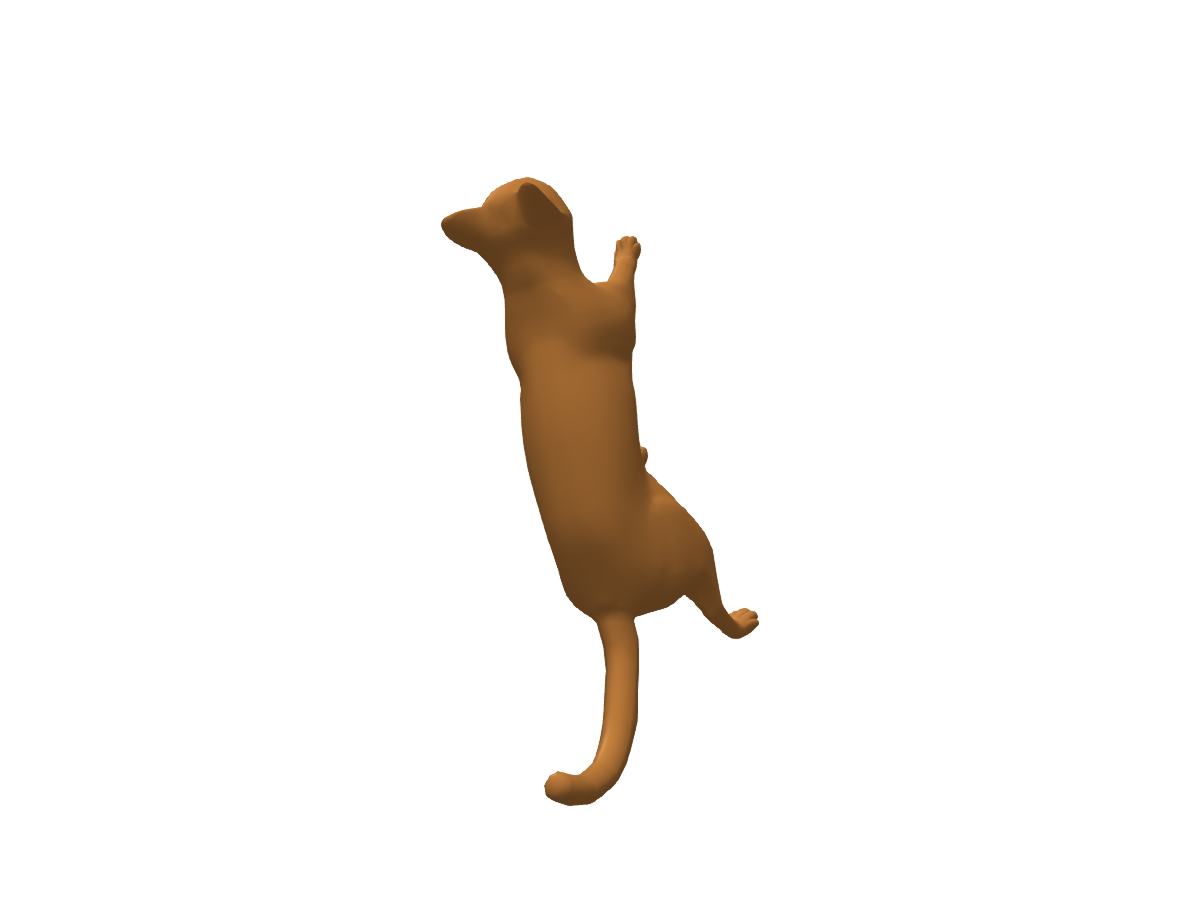}}%
        \hfill%
        {\includegraphics[width=0.15\linewidth, clip=true, trim=280pt 50pt 280pt 50pt]{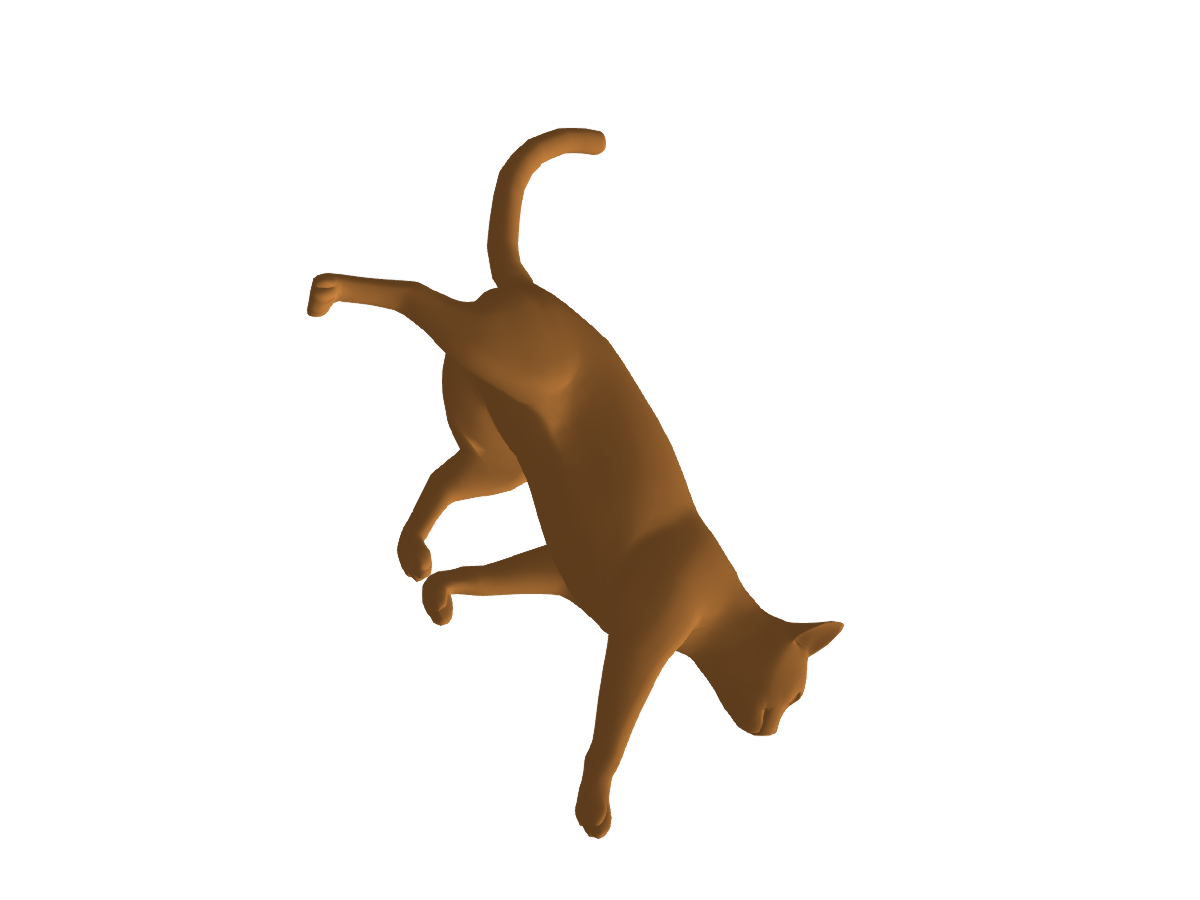}}%
        & {\includegraphics[width=0.15\linewidth, clip=true, trim=250pt 0pt 250pt 0pt]{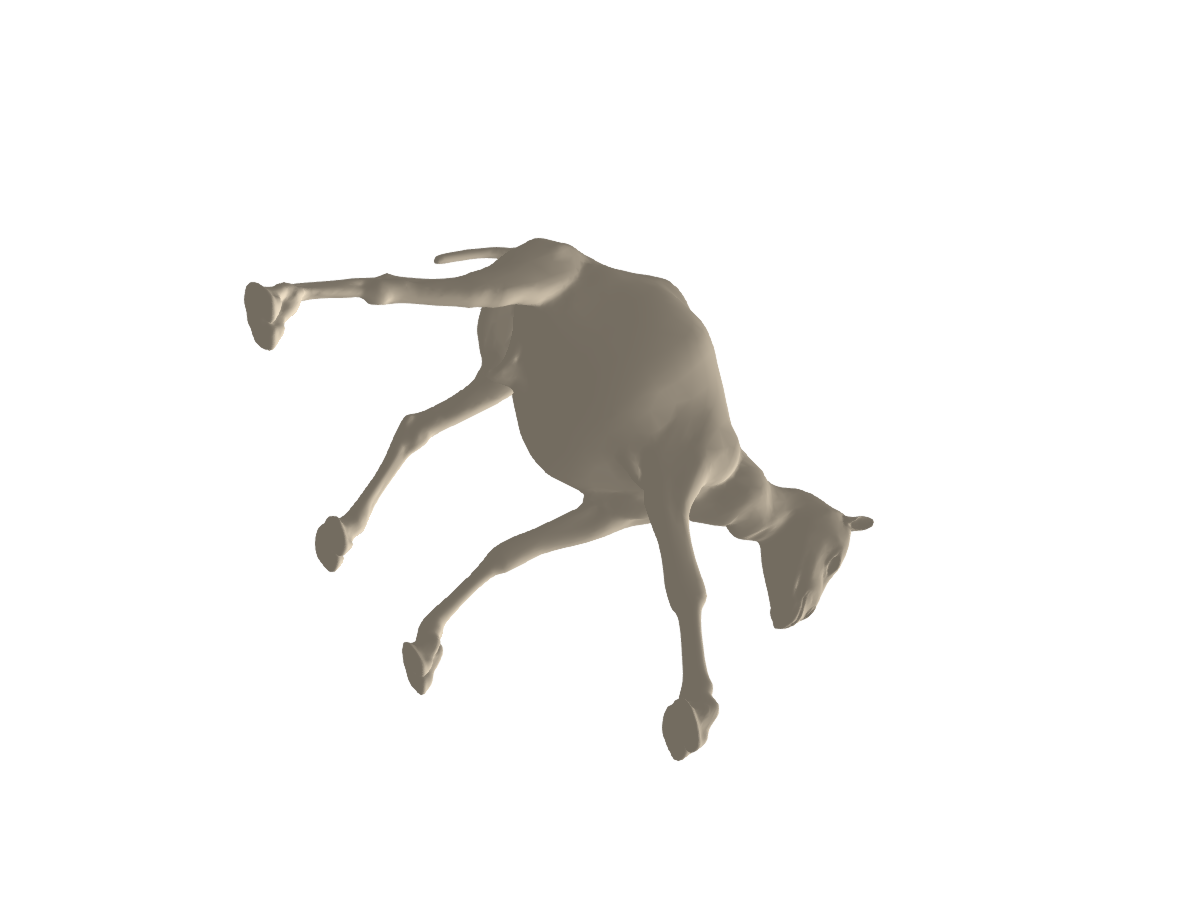}}%
        \hfill%
        {\includegraphics[width=0.15\linewidth, clip=true, trim=350pt 150pt 300pt 50pt]{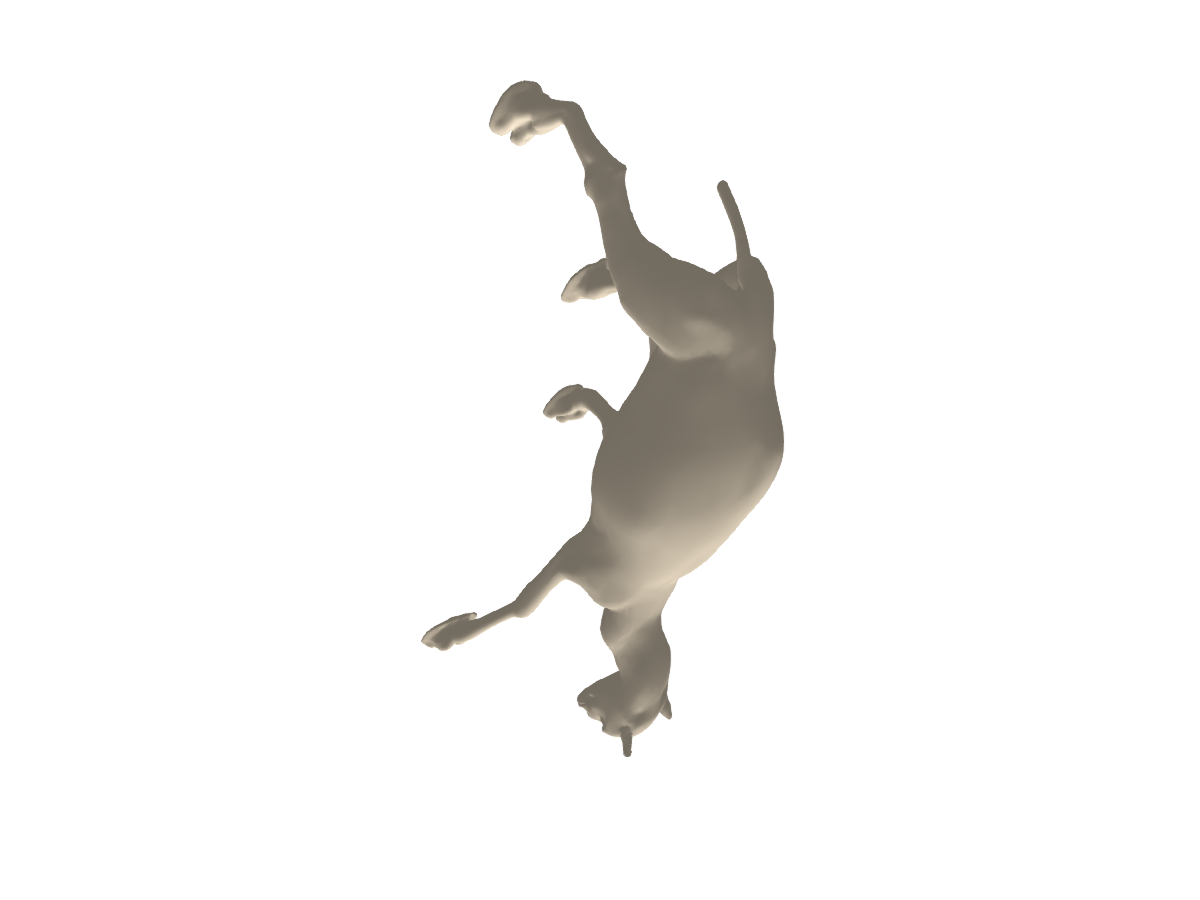}}%
        \\[10pt]
        horse & flamingo & elephant\\
        {\includegraphics[width=0.3\linewidth, clip=true, trim=220pt 150pt 360pt 150pt]{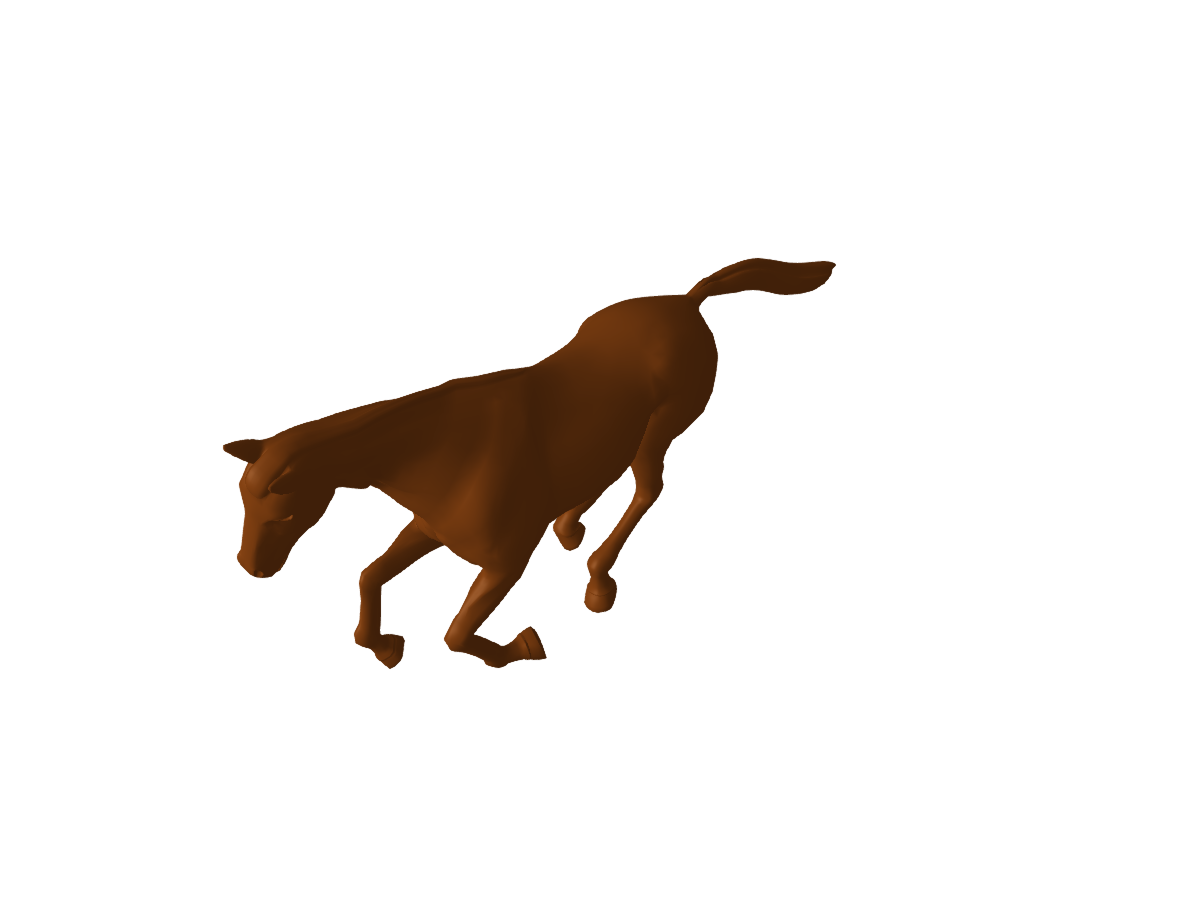}}%
        & {\includegraphics[width=0.3\linewidth, clip=true, trim=280pt 150pt 300pt 150pt]{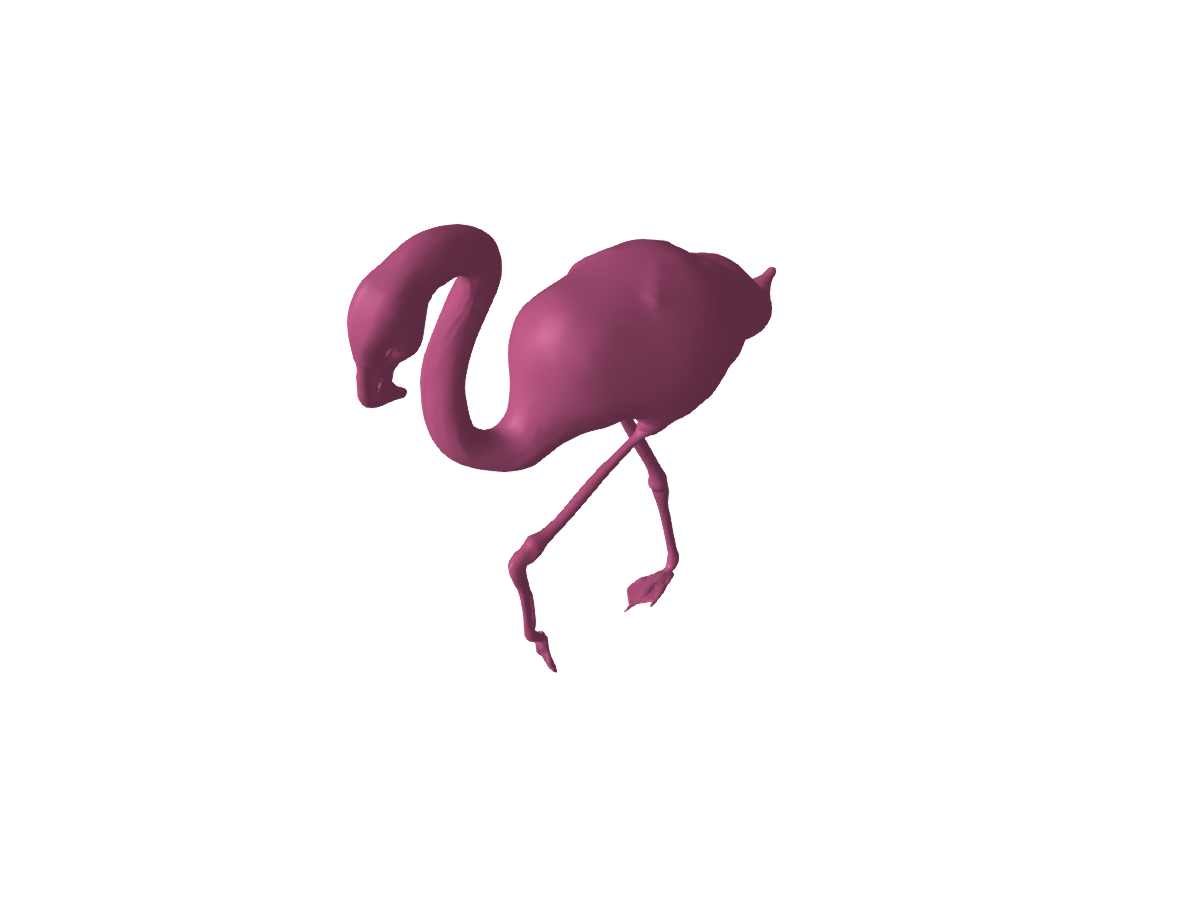}}%
        & {\includegraphics[width=0.3\linewidth, clip=true, trim=280pt 150pt 300pt 150pt]{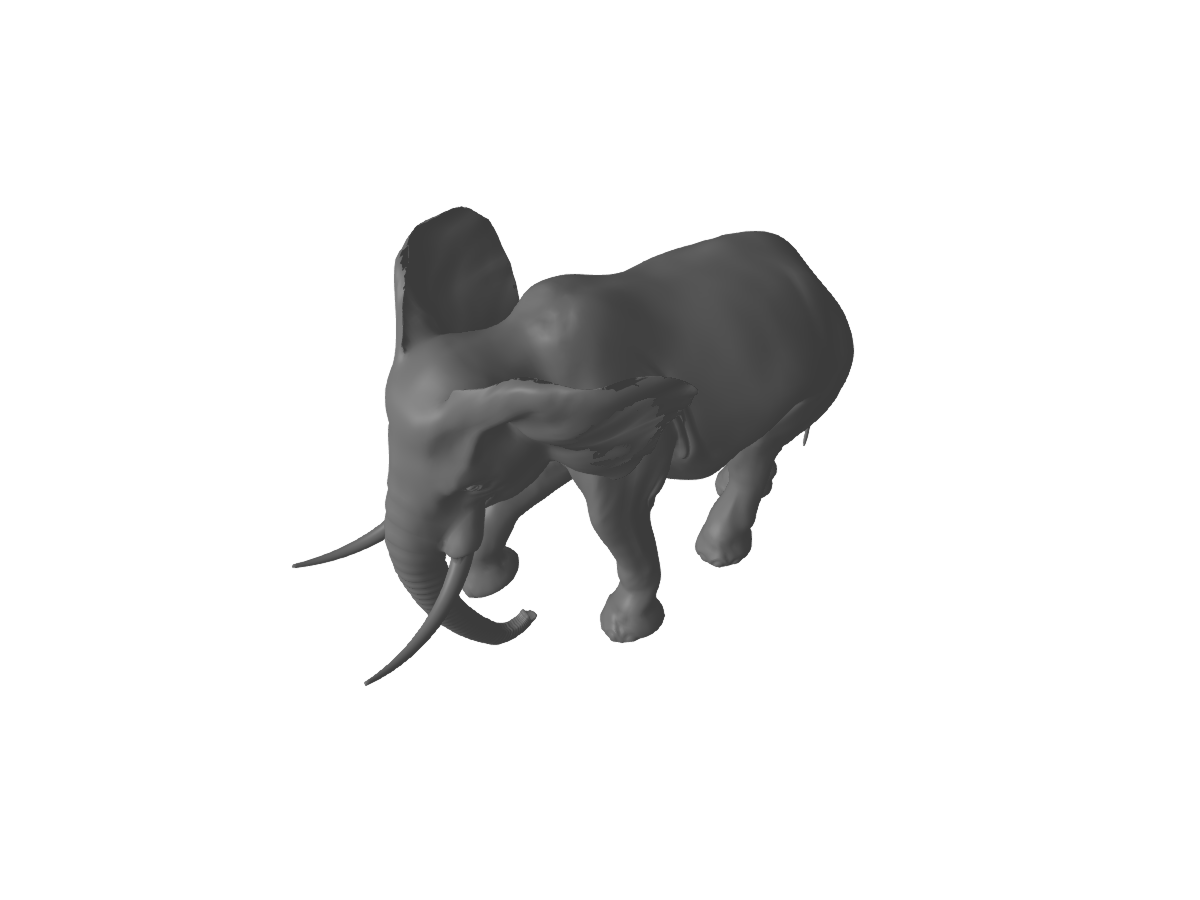}}%
        \\
        {\includegraphics[width=0.15\linewidth, clip=true, trim=250pt 200pt 250pt 200pt]{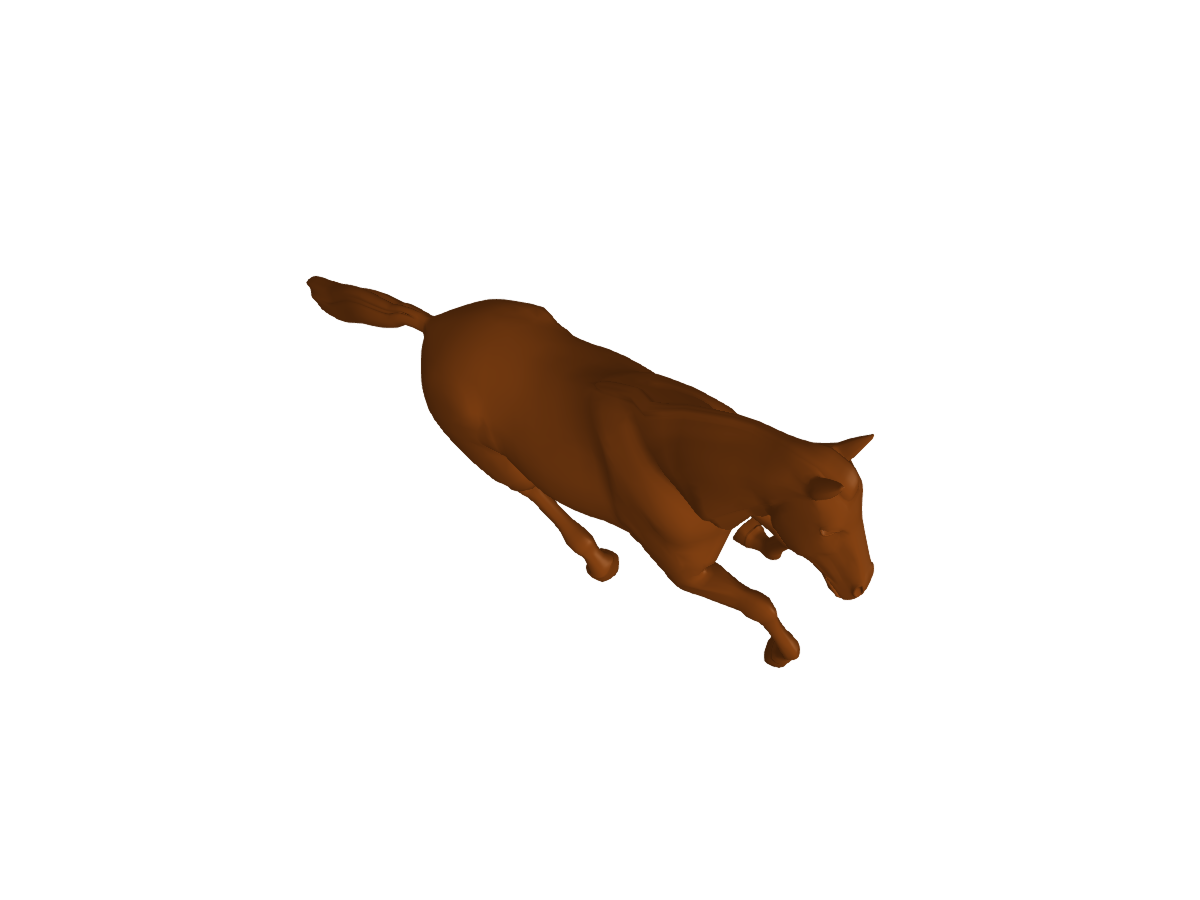}}%
        \hfill%
        {\includegraphics[width=0.15\linewidth, clip=true, trim=300pt 150pt 300pt 100pt]{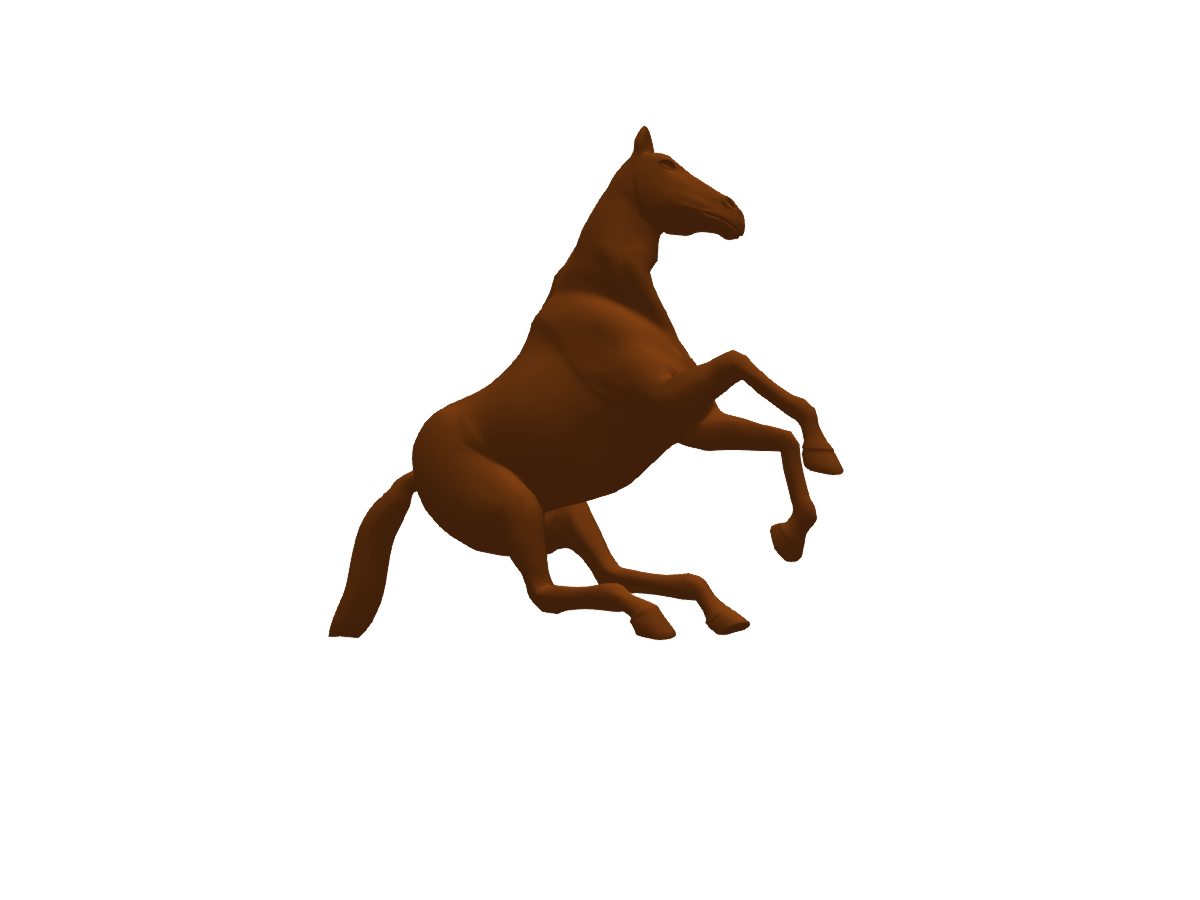}}%
        & {\includegraphics[width=0.15\linewidth, clip=true, trim=300pt 100pt 250pt 100pt]{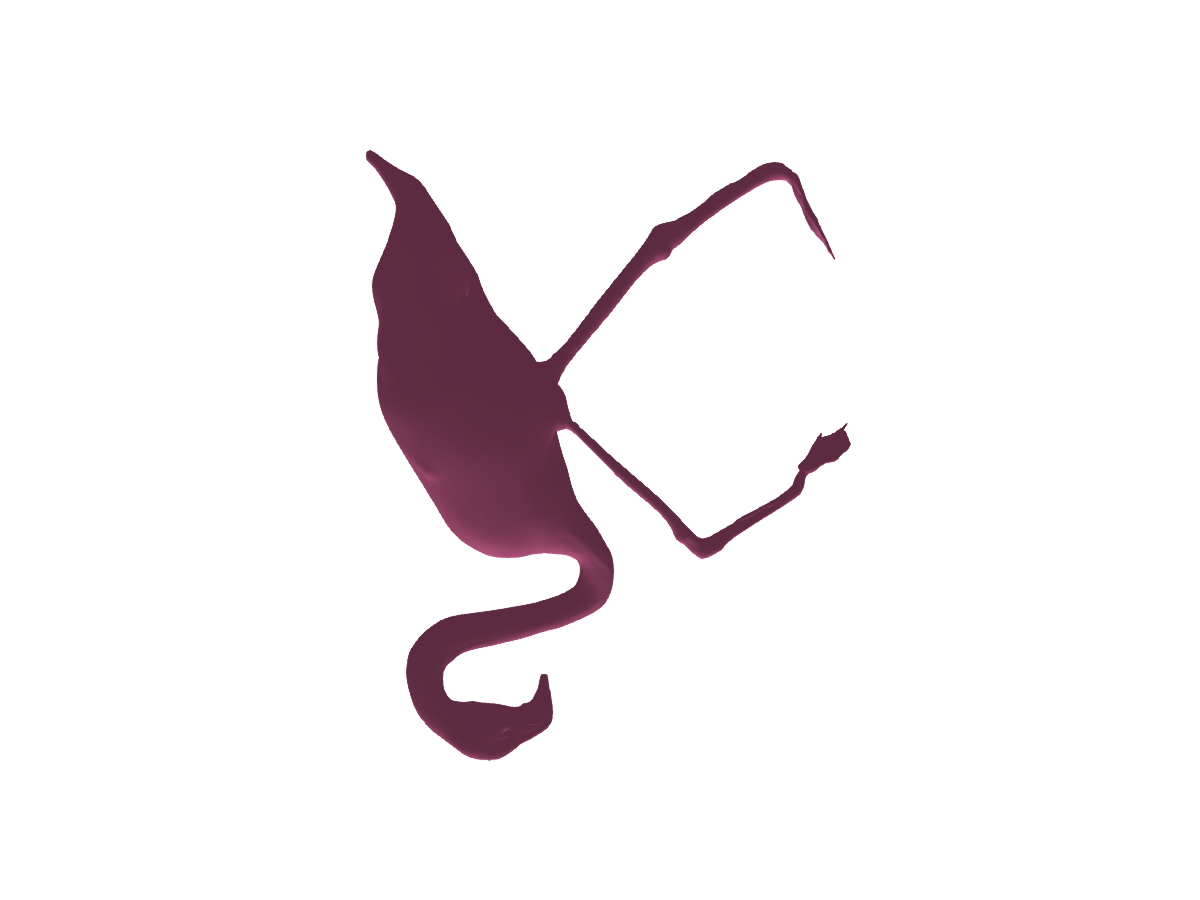}}%
        \hfill%
        {\includegraphics[width=0.15\linewidth, clip=true, trim=300pt 200pt 280pt 200pt]{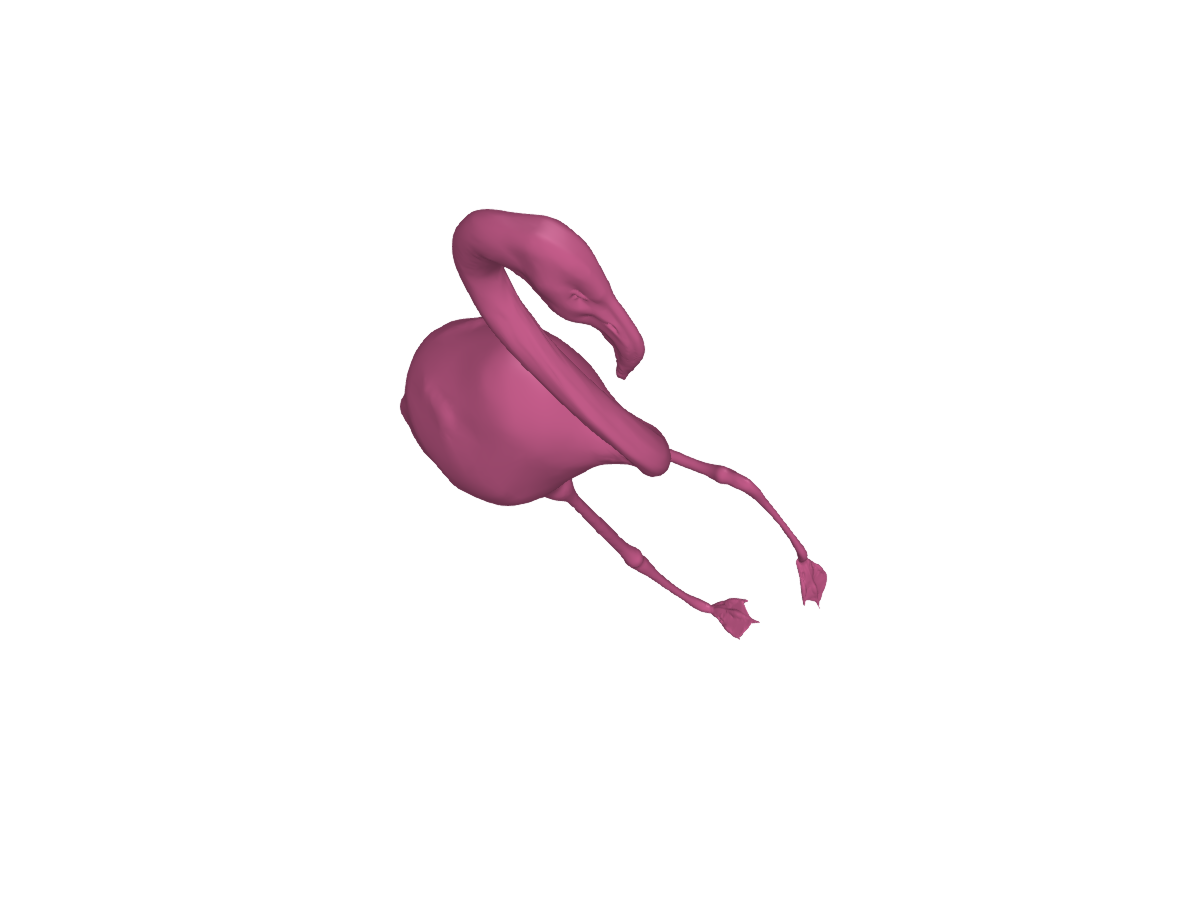}}%
        & {\includegraphics[width=0.15\linewidth, clip=true, trim=250pt 100pt 300pt 100pt]{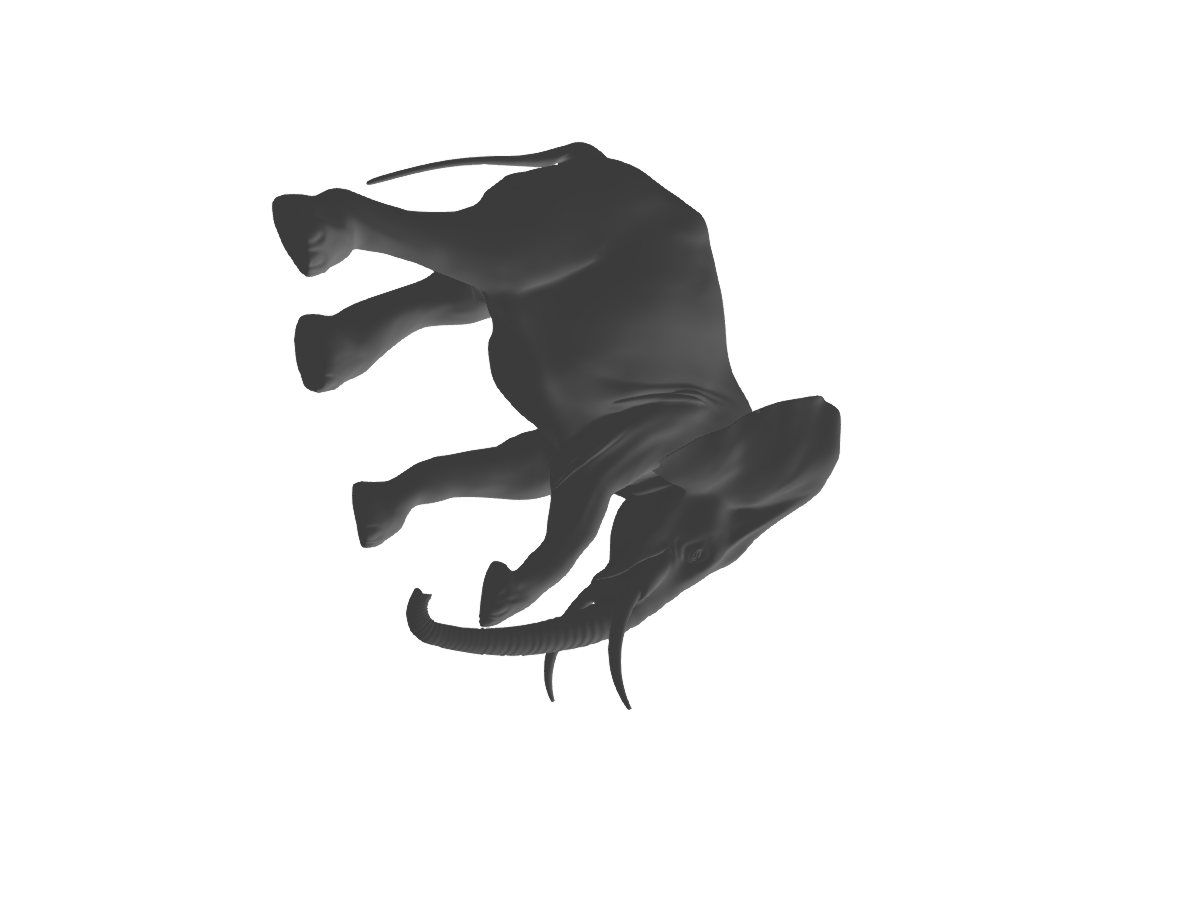}}%
        \hfill%
        {\includegraphics[width=0.15\linewidth, clip=true, trim=350pt 180pt 350pt 200pt]{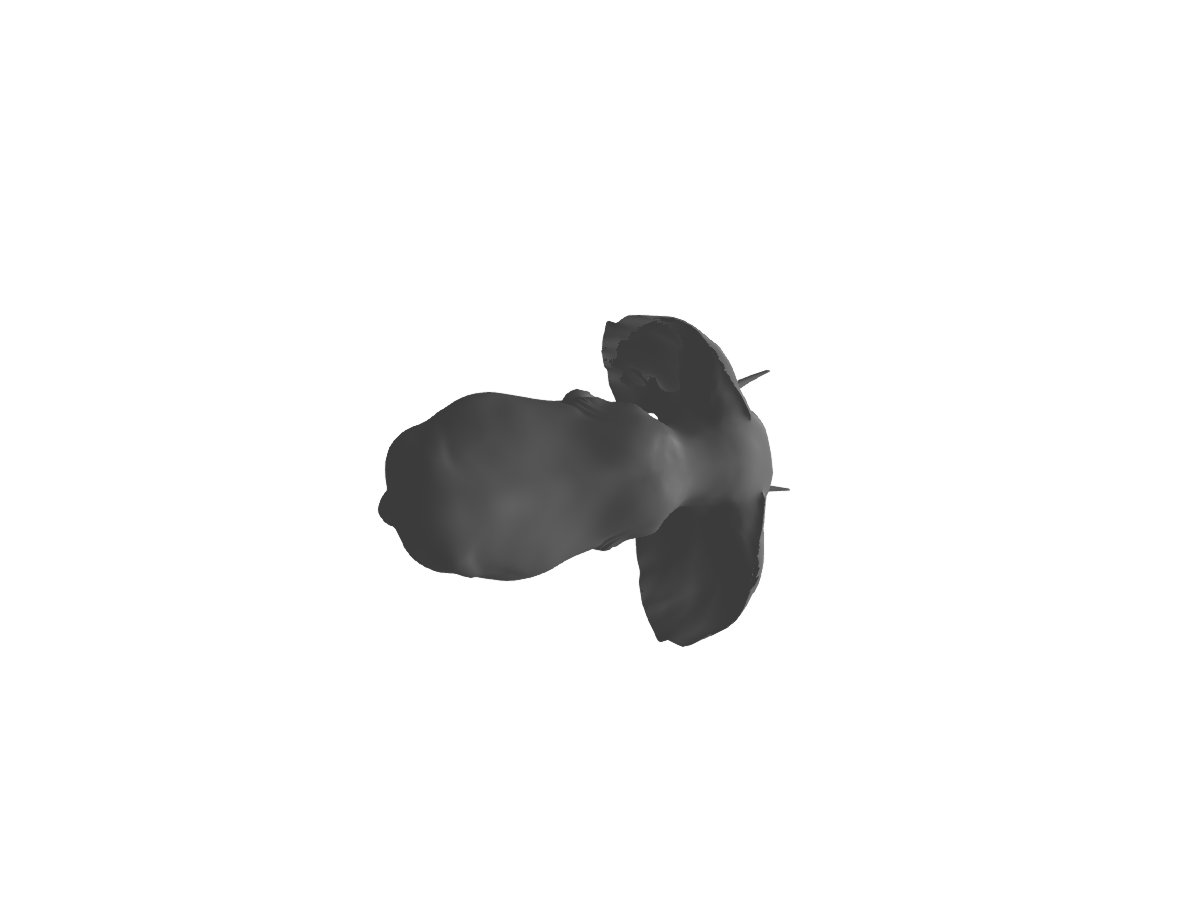}}
        \end{tabular}
        \caption{Visualization of the animal dataset:
        template meshes (top) and
        affinely transformed samples (bottom).}
        \label{fig:animal_dataset_3d}
    \end{figure}
    
    \item \textbf{ABC dataset}.
    The original `A Big CAD Model Dataset' \cite{Koch2019}
    consists of over 1 million geometrical models.
    To build our affine dataset,
    we use a subset of 80 shapes
    that are visualized in Figure~\ref{fig:abc_dataset_3d}.
    Each object mesh consists of 2028 vertices,
    which we equip with the uniform probability distribution
    to obtain an empirical measure in $\R^3$.
    Similar to the animal dataset,
    each template is used to construct a class of
    10 random affinely transformed versions.
    For the construction,
    we again employ 
    anisotropic scaling factors in $[0.5, 1.0]$
    and shearing in $[-15^\circ, 15 ^\circ]$
    as well as
    random 3d rotations
    and shifts in $[-25,25]^3$.
        
    \begin{figure}[t]
        \resizebox{\linewidth}{!}{%
        \begin{tabular}{c @{\hspace{0pt}} c @{\hspace{0pt}} c @{\hspace{0pt}} c @{\hspace{0pt}} c @{\hspace{0pt}} c @{\hspace{0pt}} c @{\hspace{0pt}} c @{\hspace{0pt}} c @{\hspace{0pt}} c }
        \includegraphics[width=0.1\linewidth, clip=true, trim=140pt 50pt 140pt 50pt]{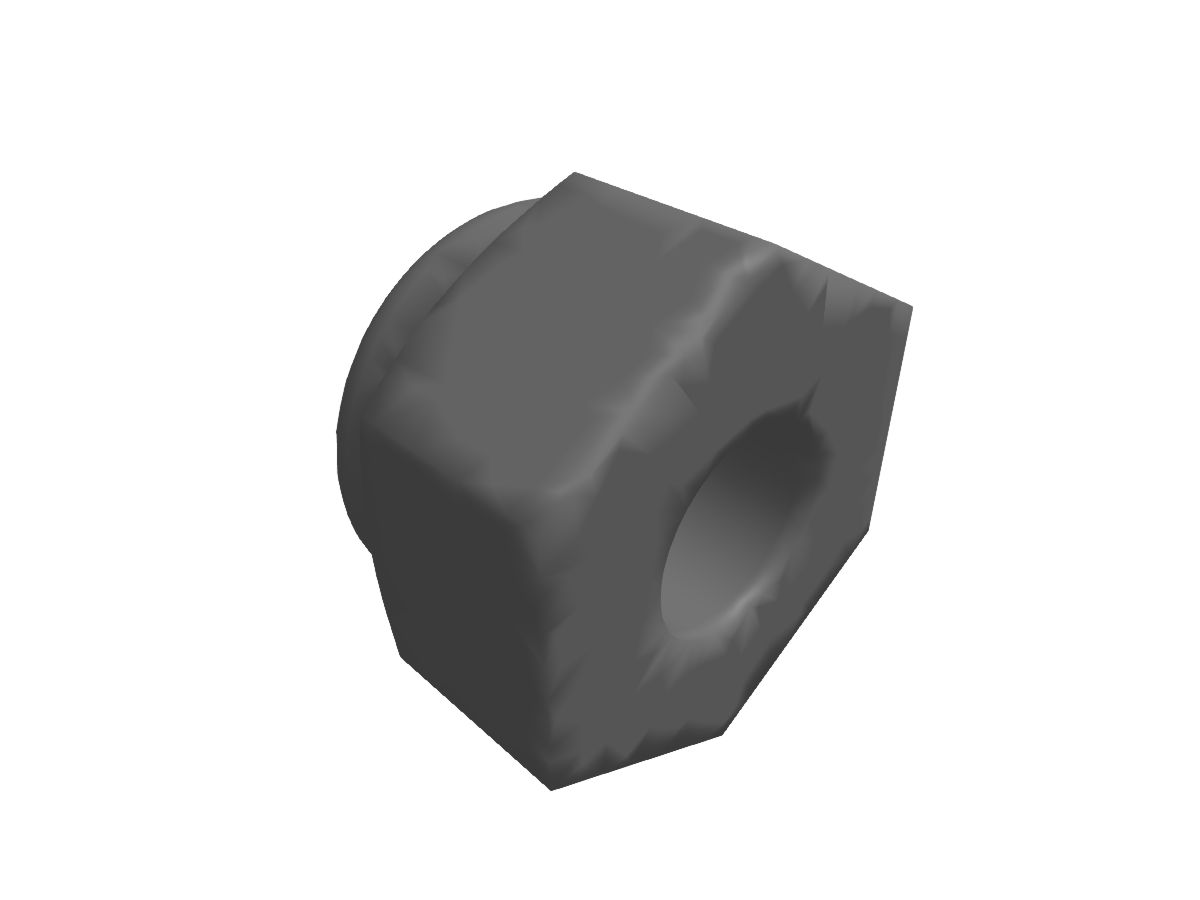}
        &\includegraphics[width=0.1\linewidth, clip=true, trim=140pt 50pt 140pt 50pt]{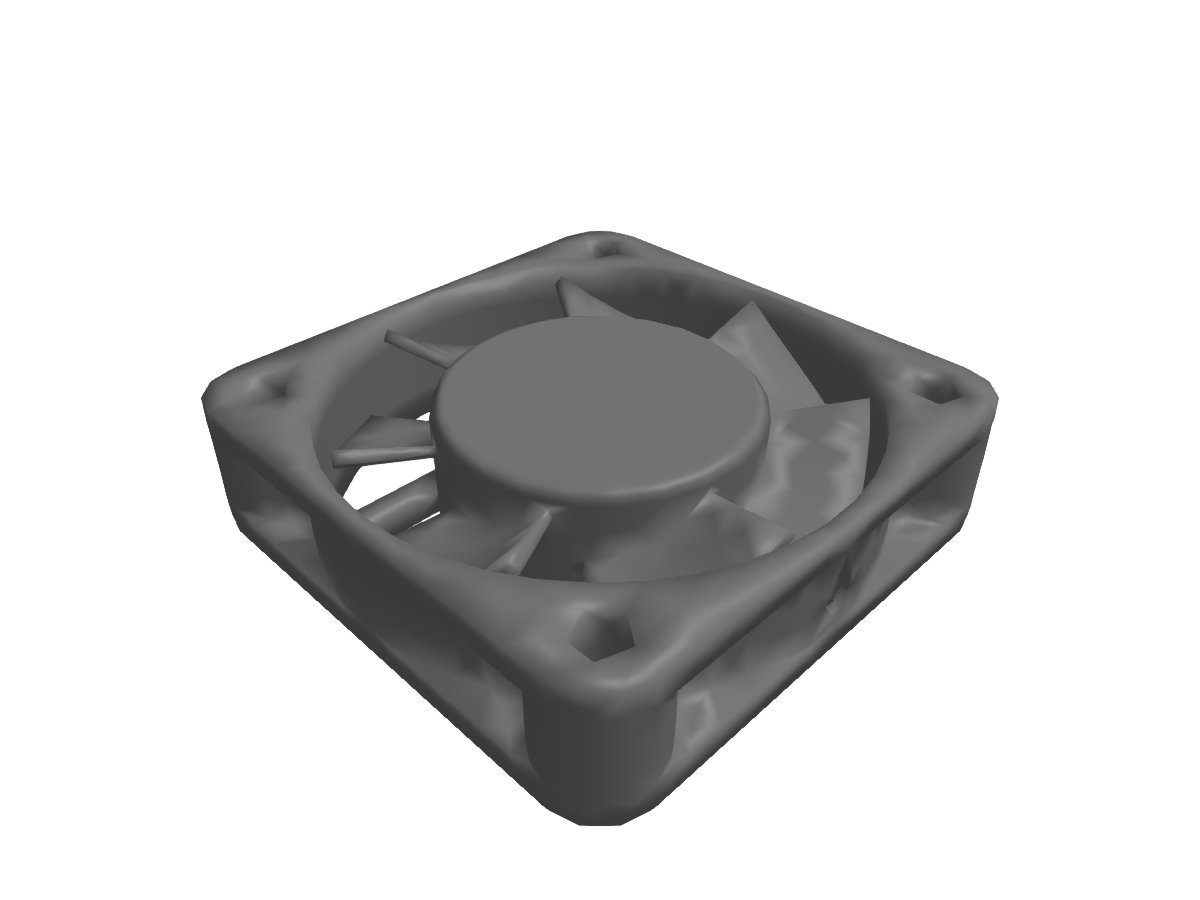}
        &\includegraphics[width=0.1\linewidth, clip=true, trim=140pt 50pt 140pt 50pt]{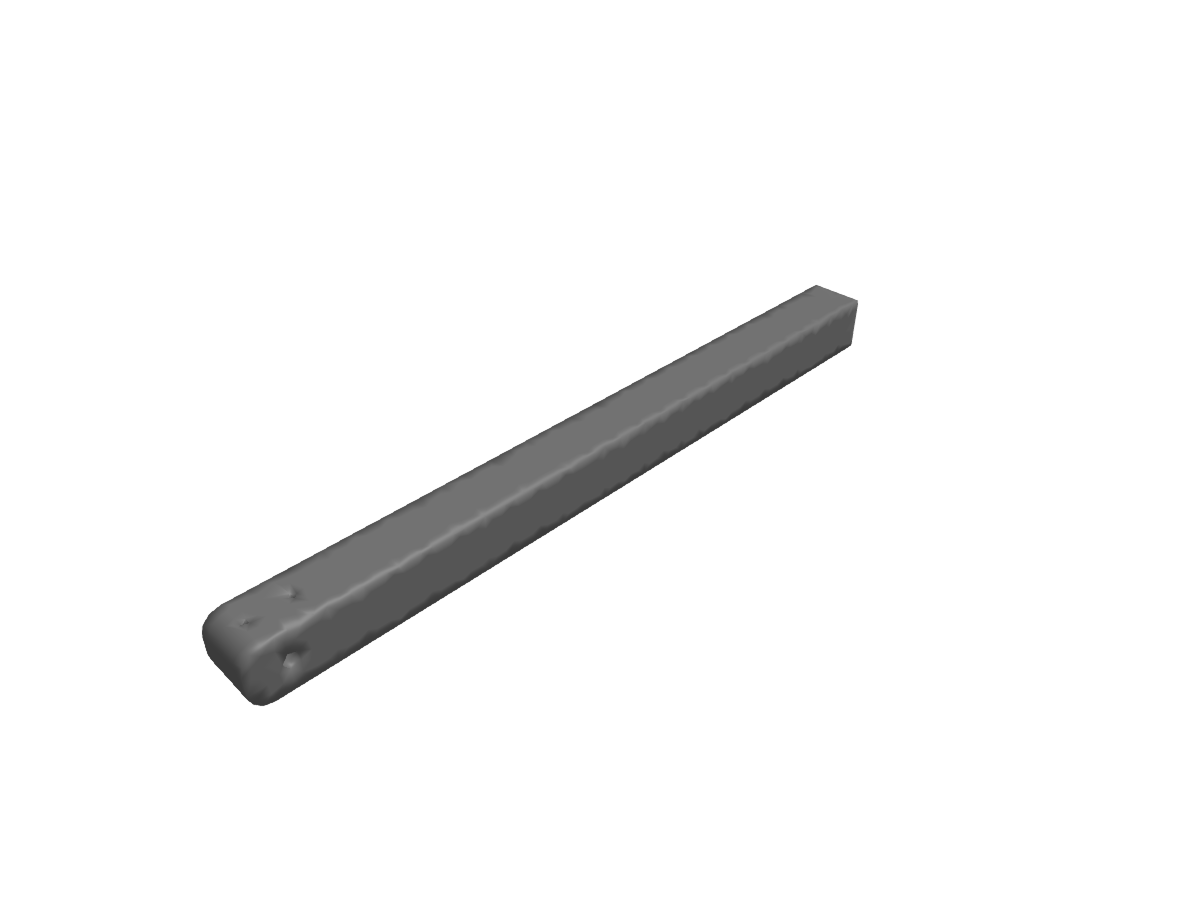}
        &\includegraphics[width=0.1\linewidth, clip=true, trim=140pt 50pt 140pt 50pt]{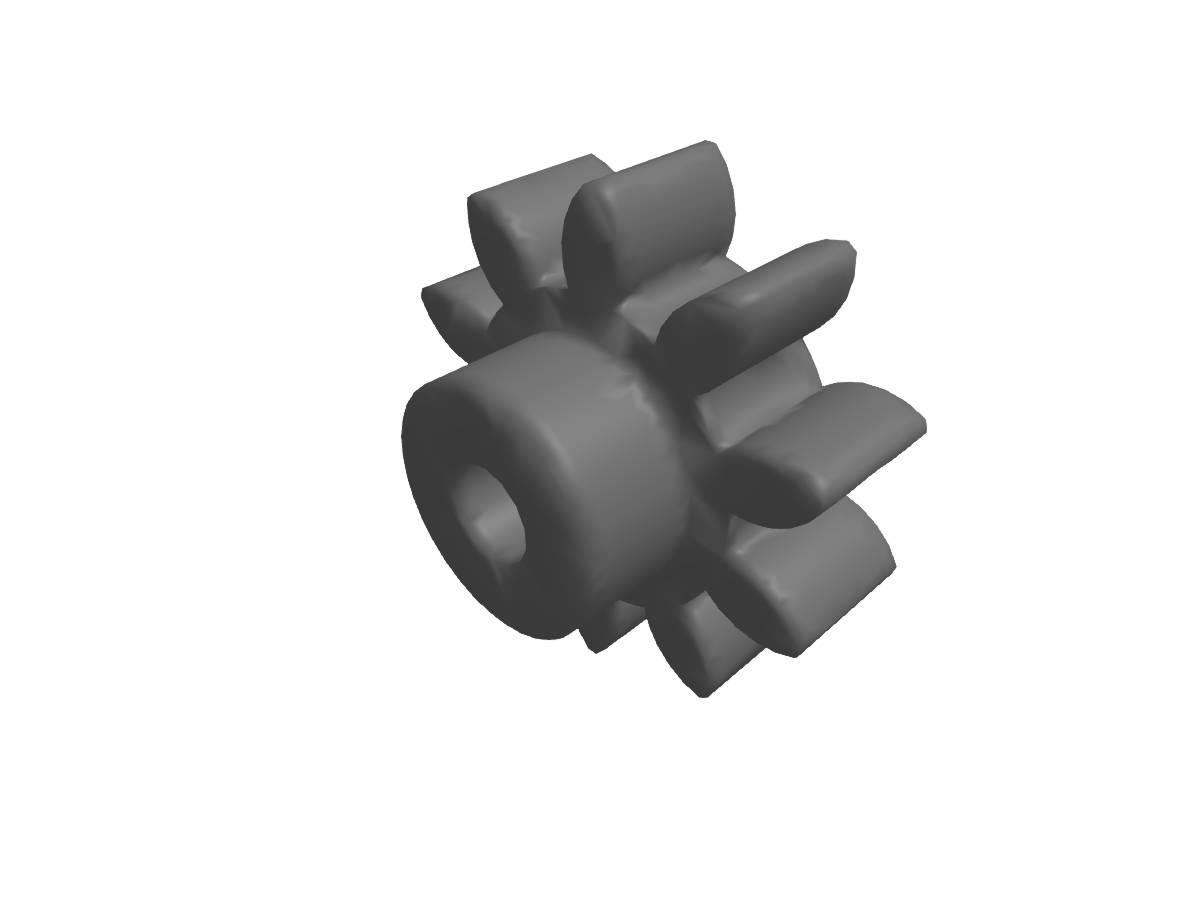}
        &\includegraphics[width=0.1\linewidth, clip=true, trim=140pt 50pt 140pt 50pt]{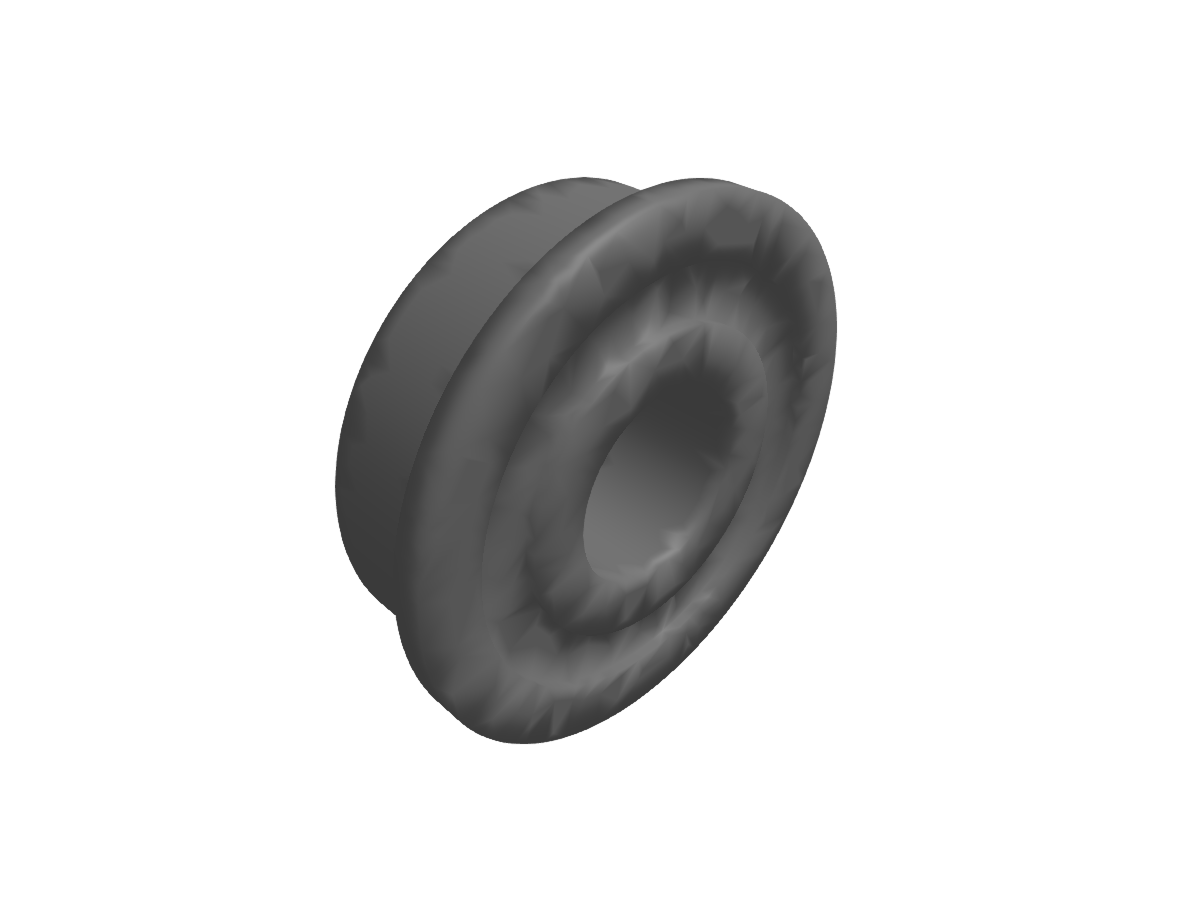}
        &\includegraphics[width=0.1\linewidth, clip=true, trim=140pt 50pt 140pt 50pt]{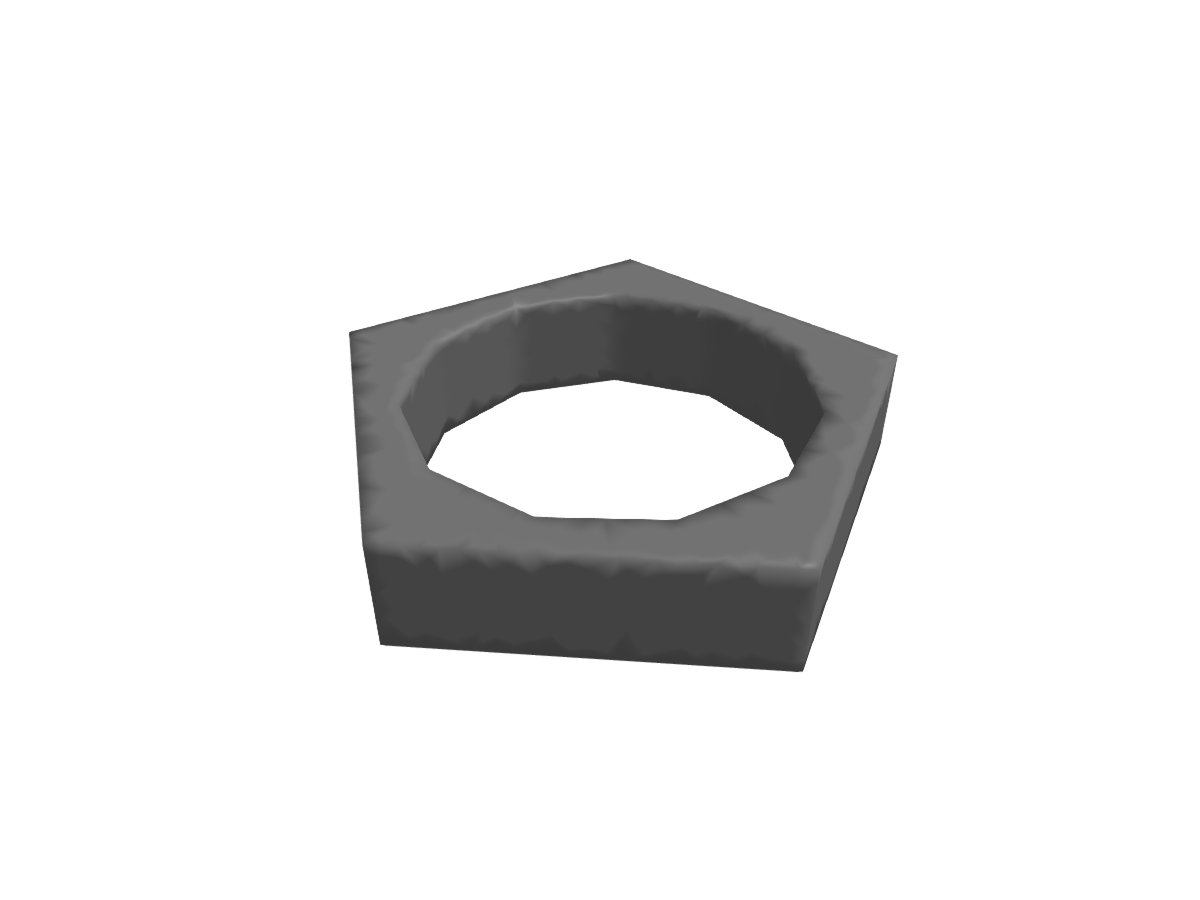}
        &\includegraphics[width=0.1\linewidth, clip=true, trim=140pt 50pt 140pt 50pt]{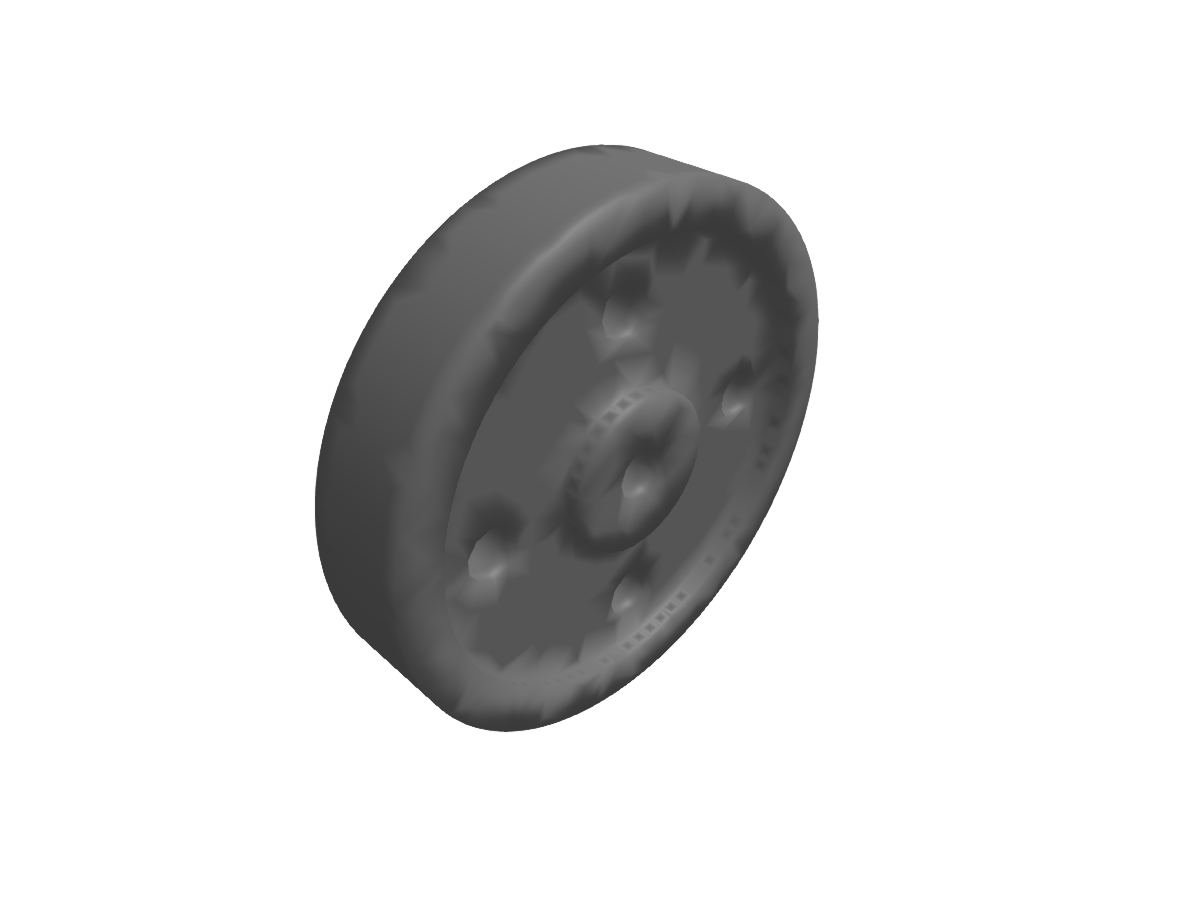}
        &\includegraphics[width=0.1\linewidth, clip=true, trim=140pt 50pt 140pt 50pt]{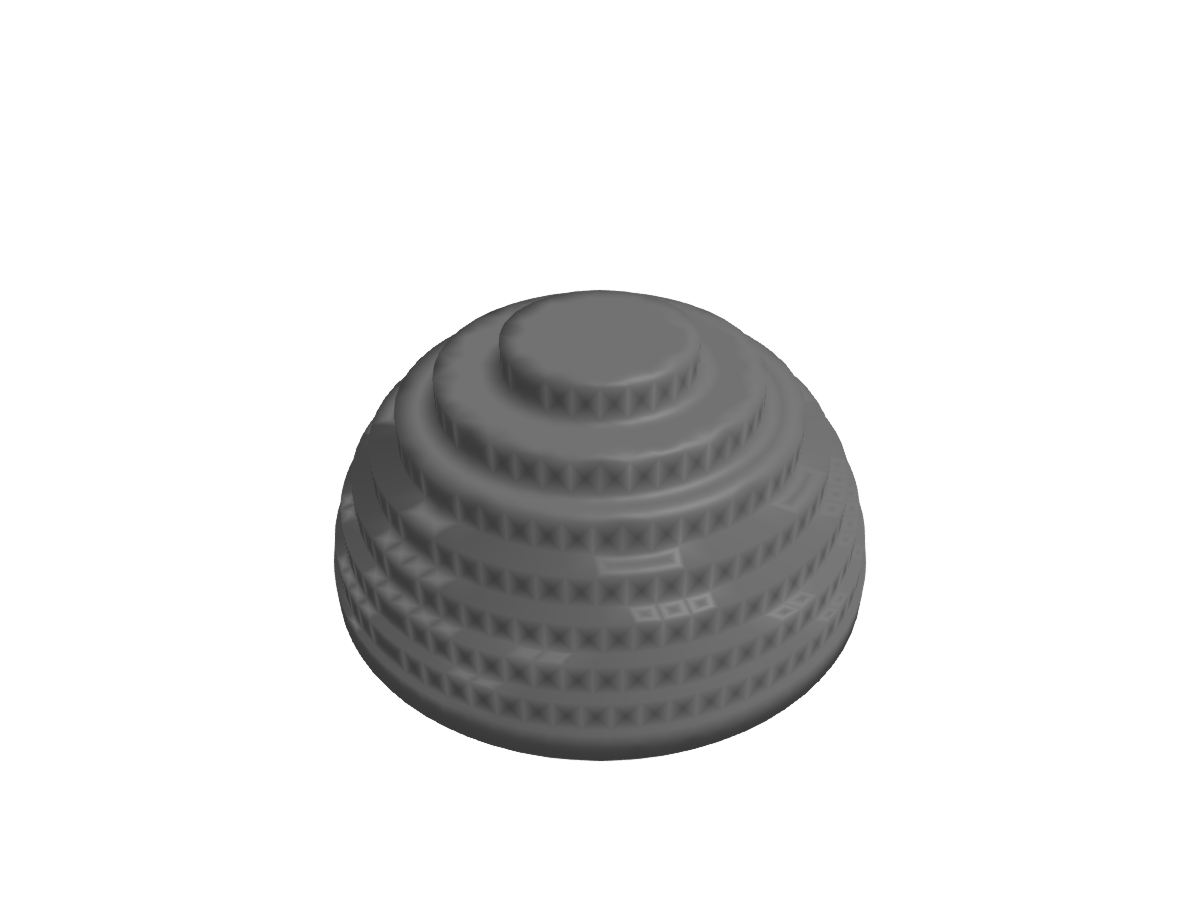}
        &\includegraphics[width=0.1\linewidth, clip=true, trim=140pt 50pt 140pt 50pt]{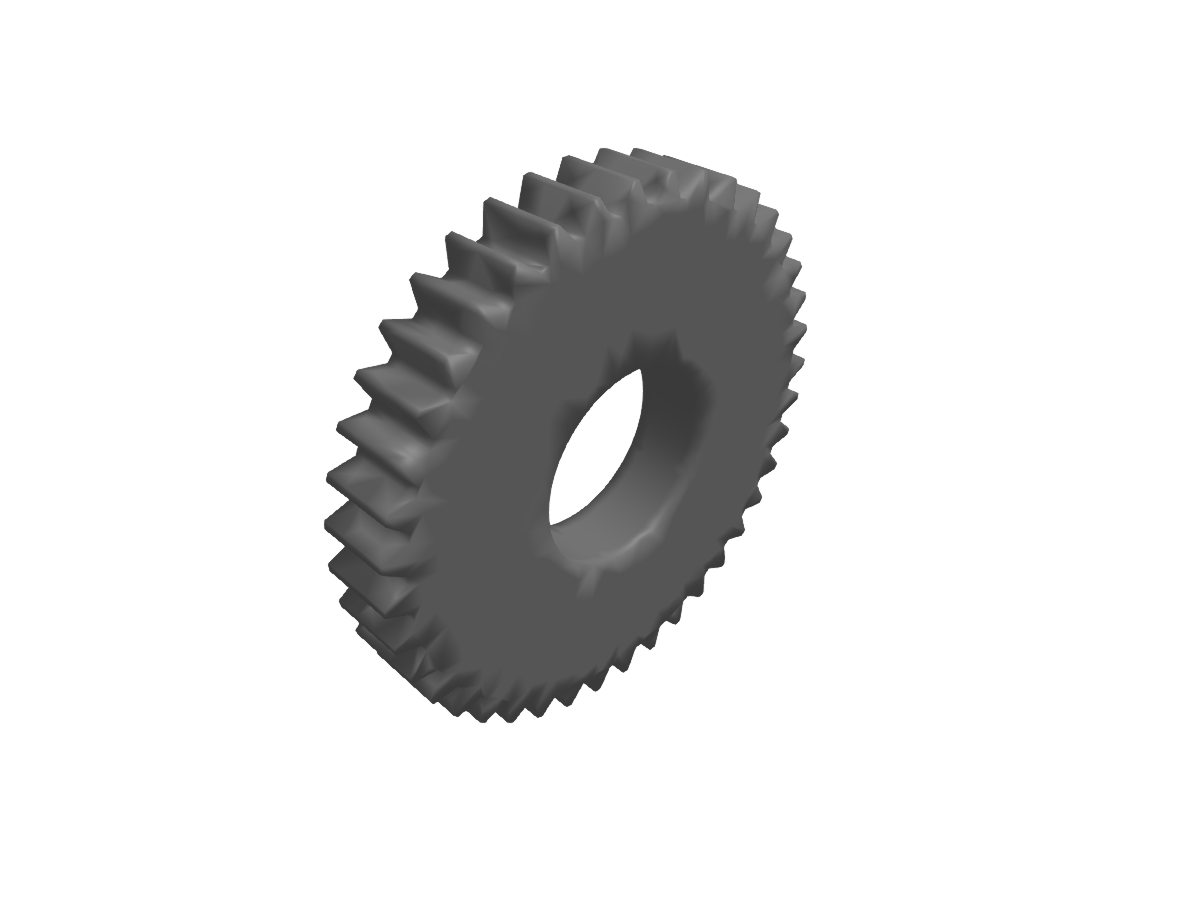}
        &\includegraphics[width=0.1\linewidth, clip=true, trim=140pt 50pt 140pt 50pt]{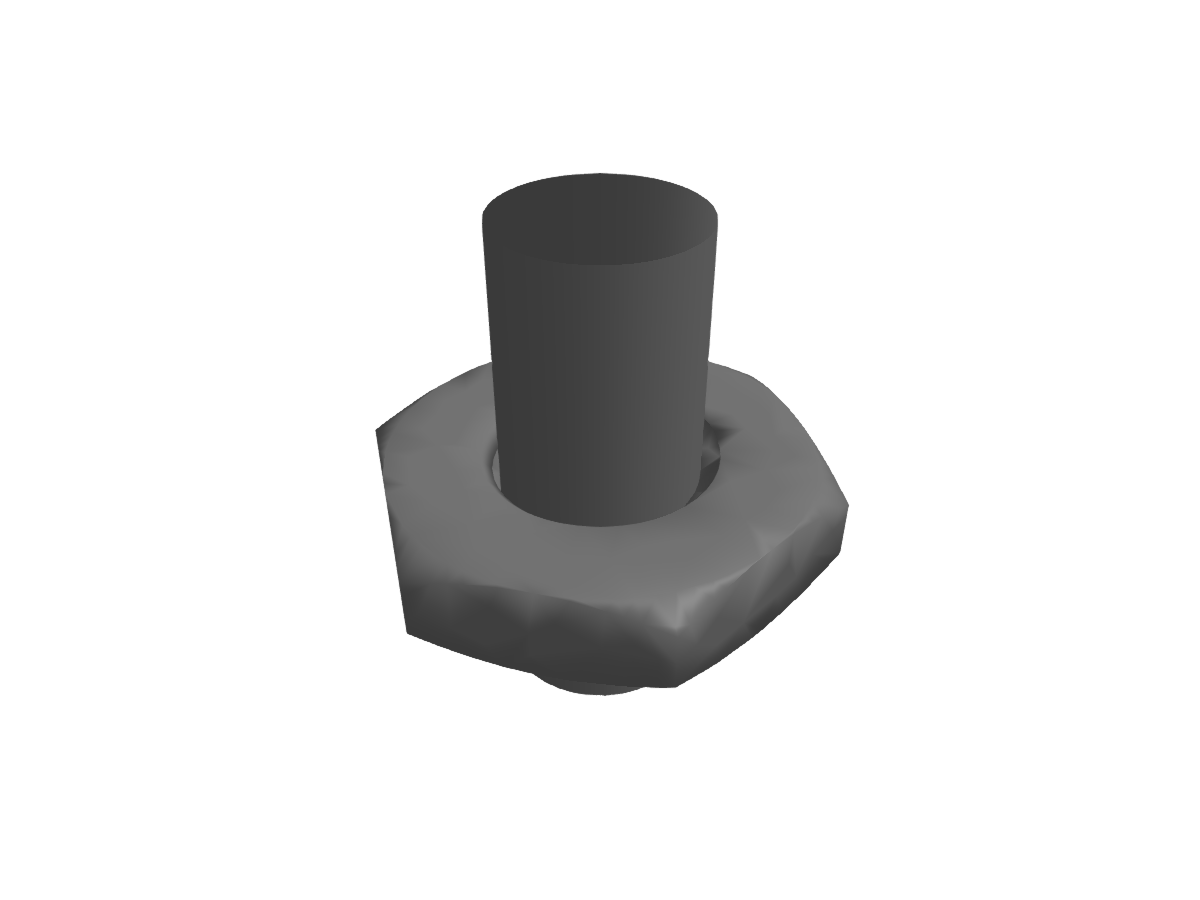}
        \\
        \includegraphics[width=0.1\linewidth, clip=true, trim=140pt 50pt 140pt 50pt]{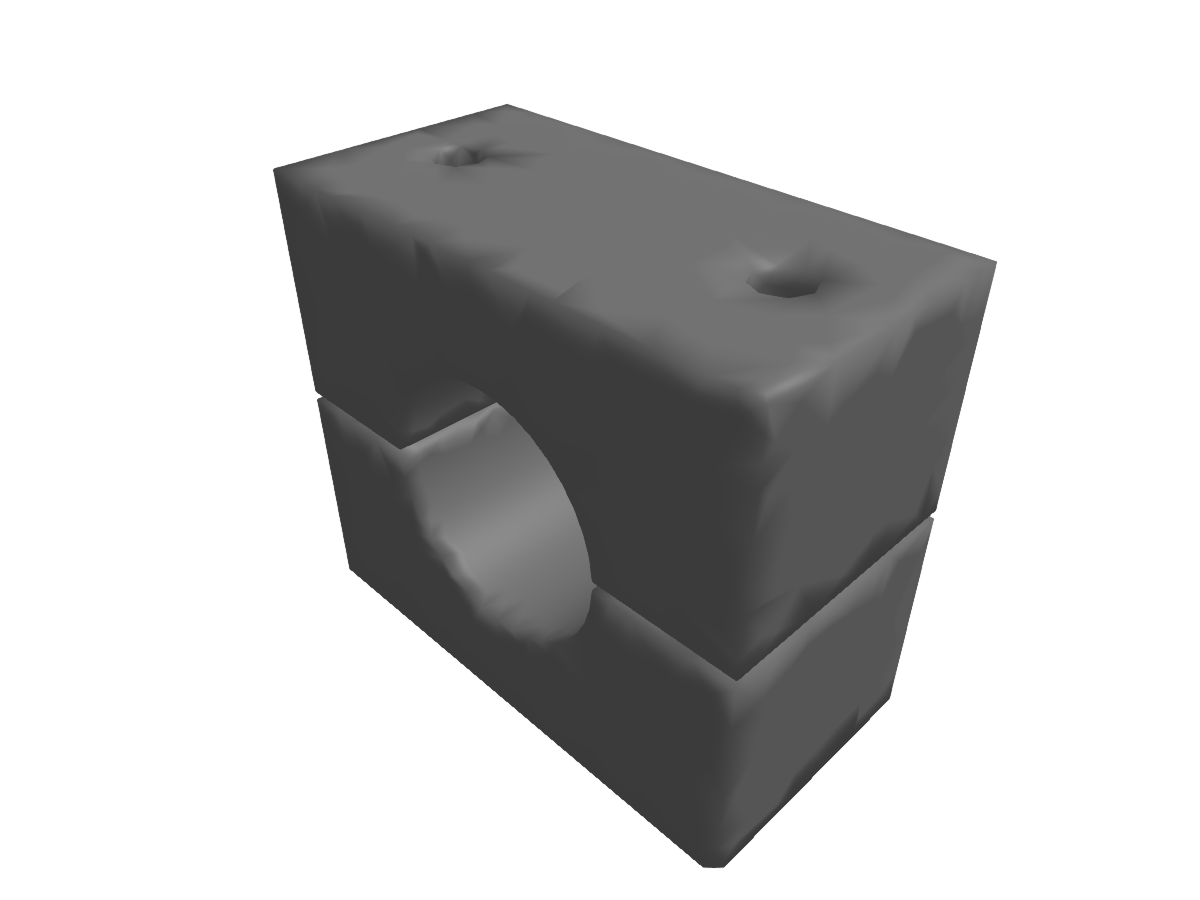}
        &\includegraphics[width=0.1\linewidth, clip=true, trim=140pt 50pt 140pt 50pt]{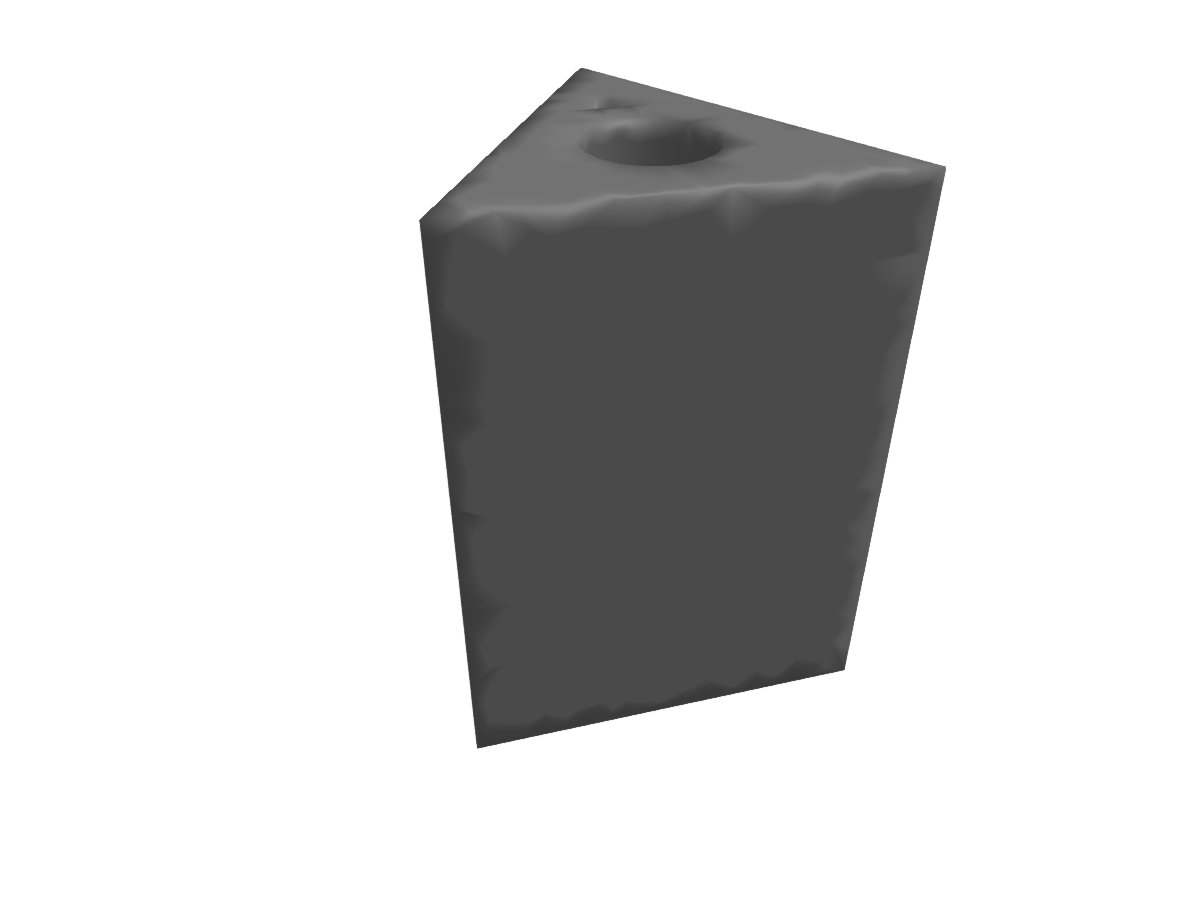} 
        &\includegraphics[width=0.1\linewidth, clip=true, trim=140pt 50pt 140pt 50pt]{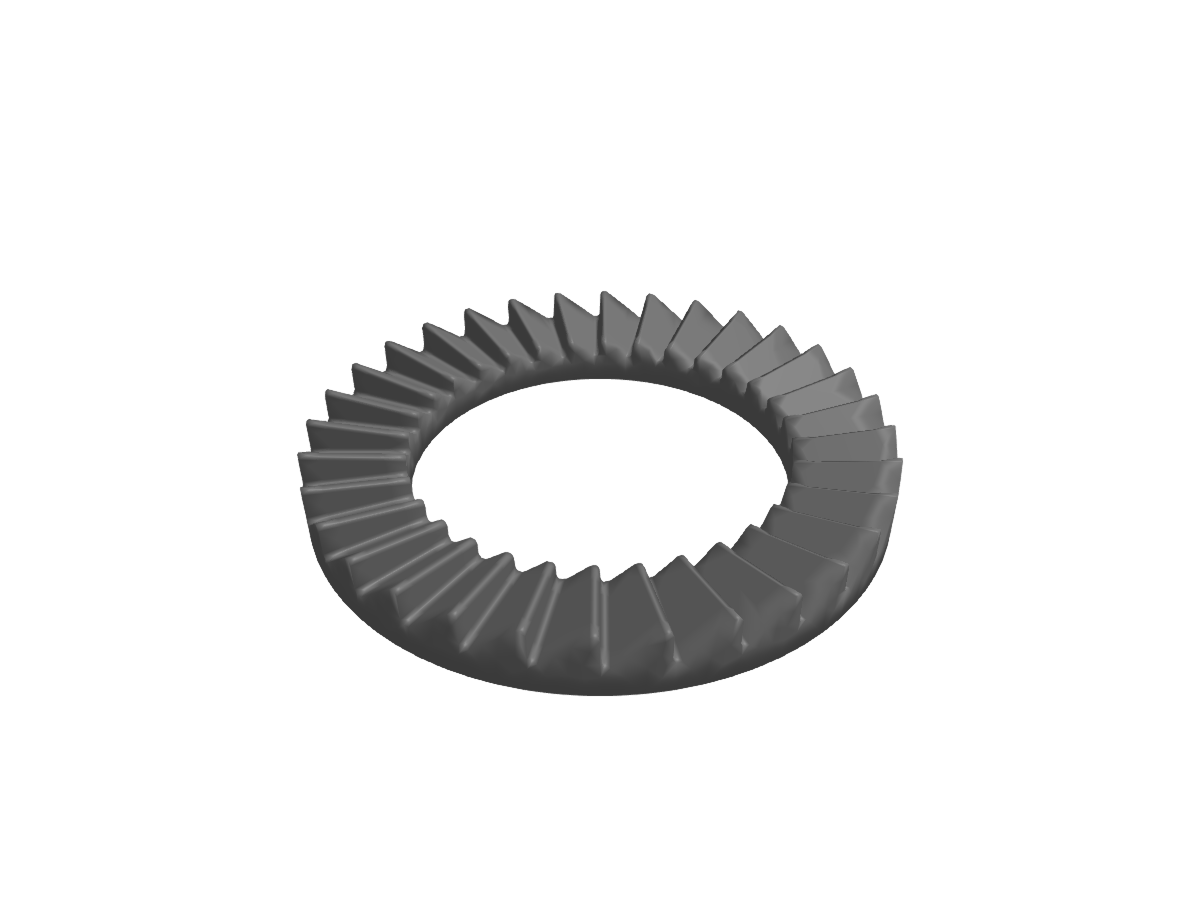}
        &\includegraphics[width=0.1\linewidth, clip=true, trim=140pt 50pt 140pt 50pt]{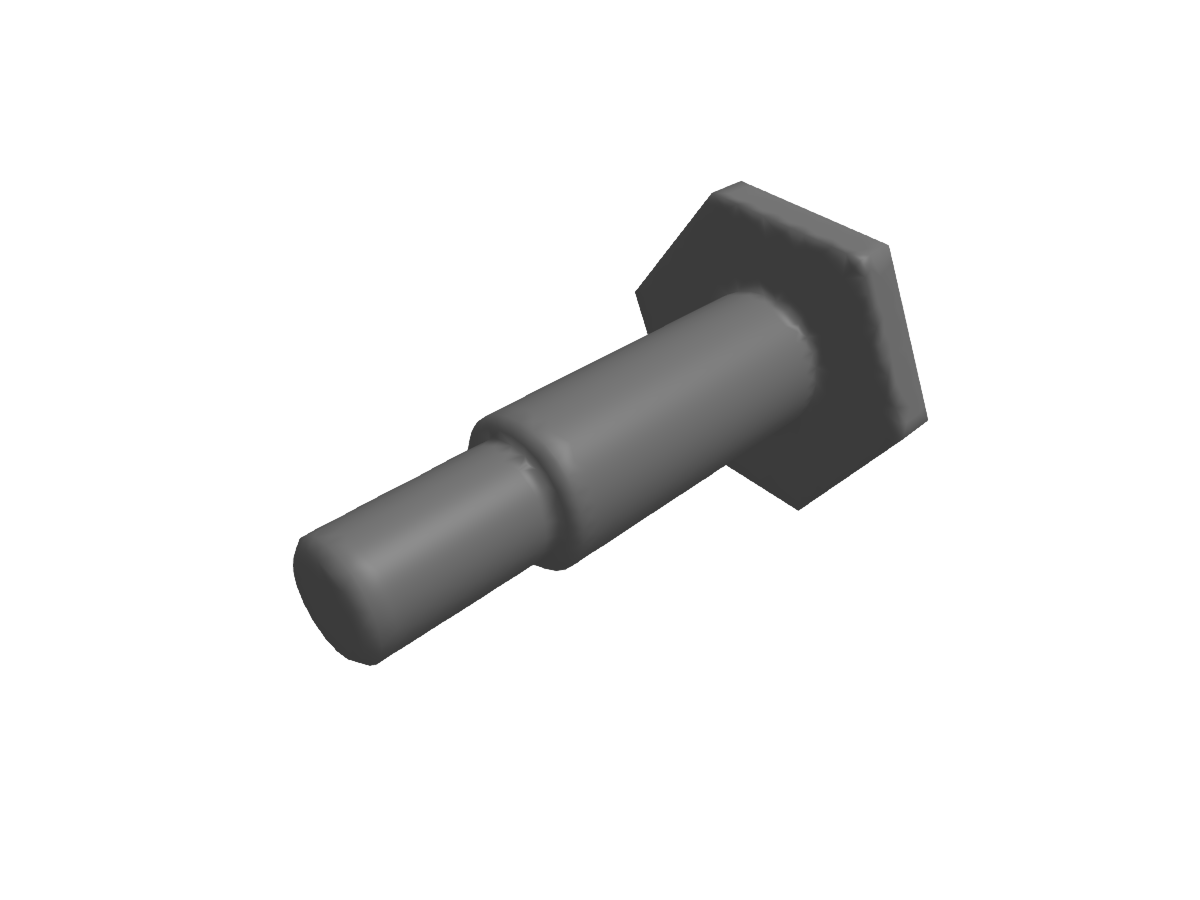}
        &\includegraphics[width=0.1\linewidth, clip=true, trim=140pt 50pt 140pt 50pt]{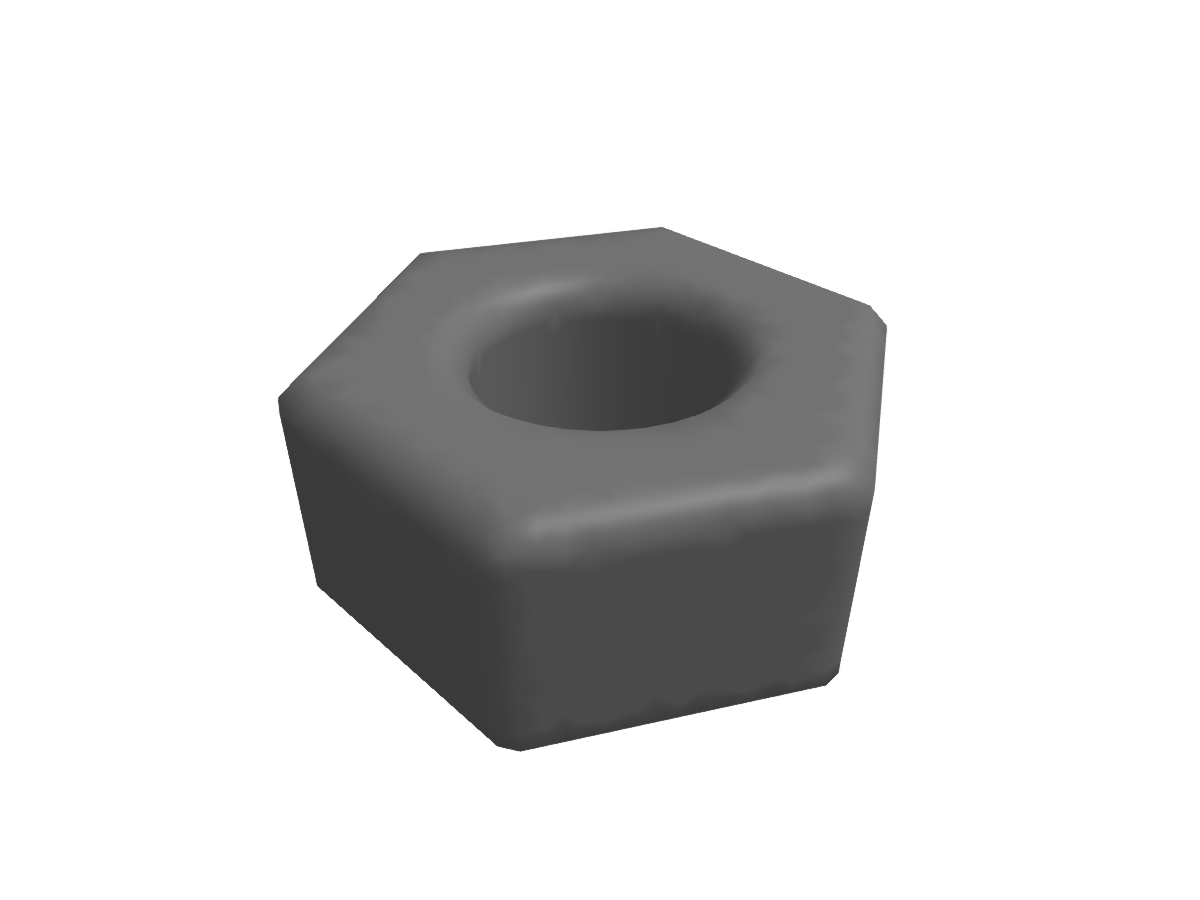}
        &\includegraphics[width=0.1\linewidth, clip=true, trim=140pt 50pt 140pt 50pt]{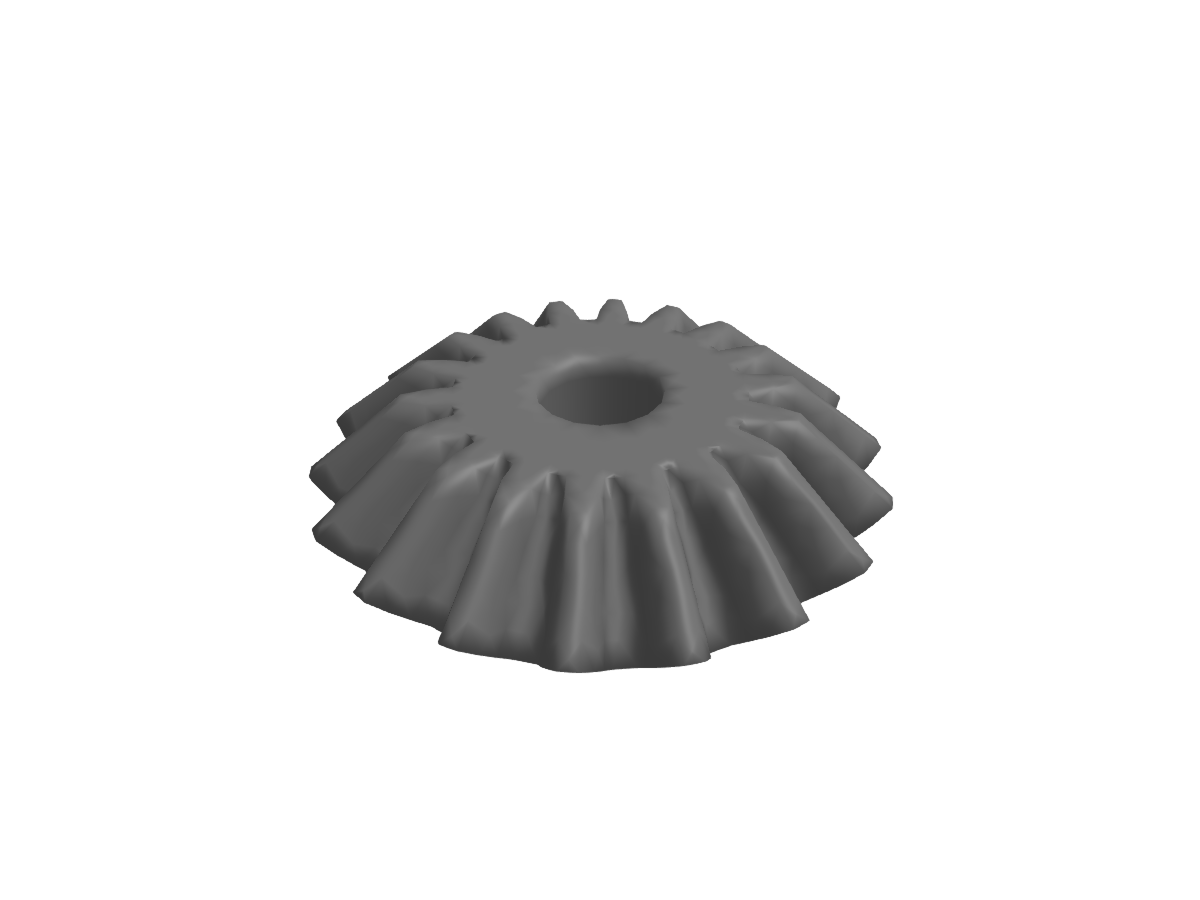}
        &\includegraphics[width=0.1\linewidth, clip=true, trim=140pt 50pt 140pt 50pt]{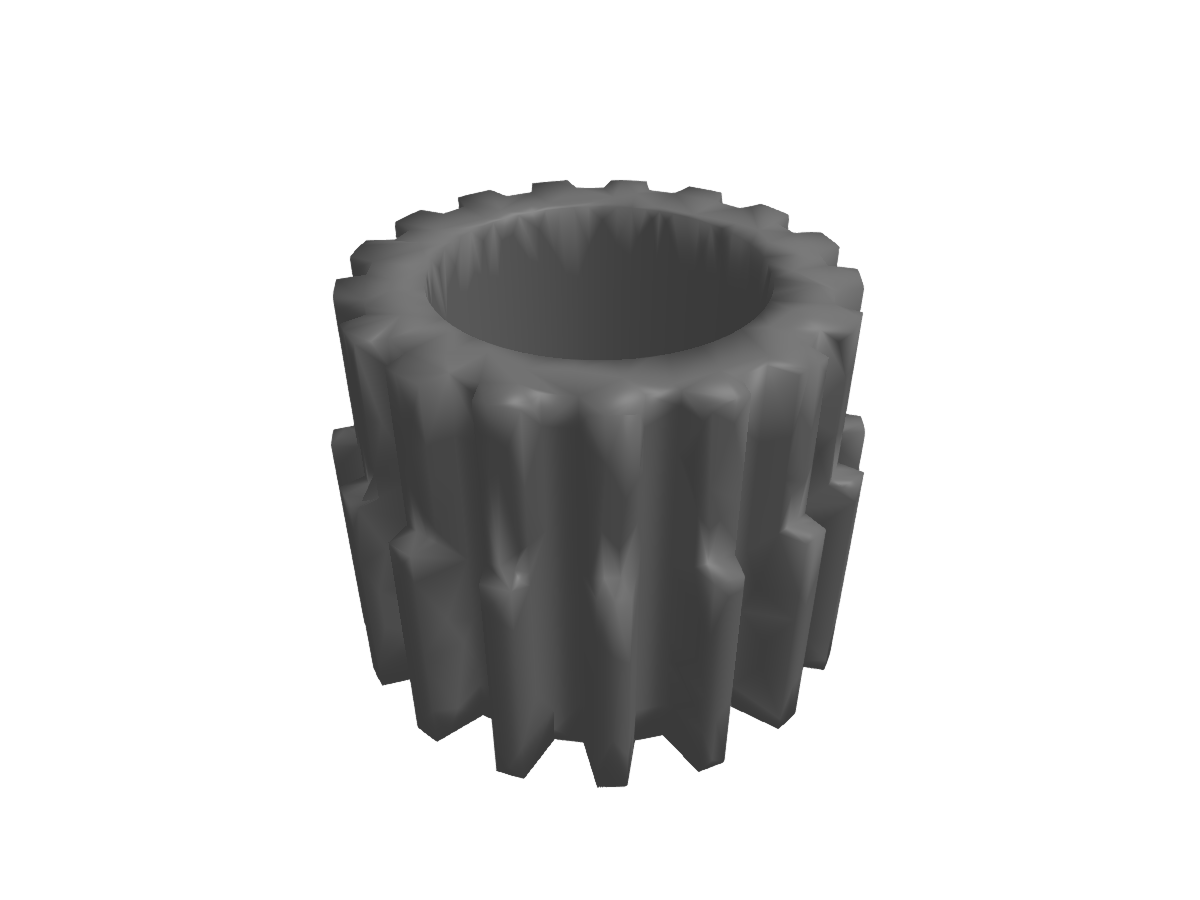}
        &\includegraphics[width=0.1\linewidth, clip=true, trim=140pt 50pt 140pt 50pt]{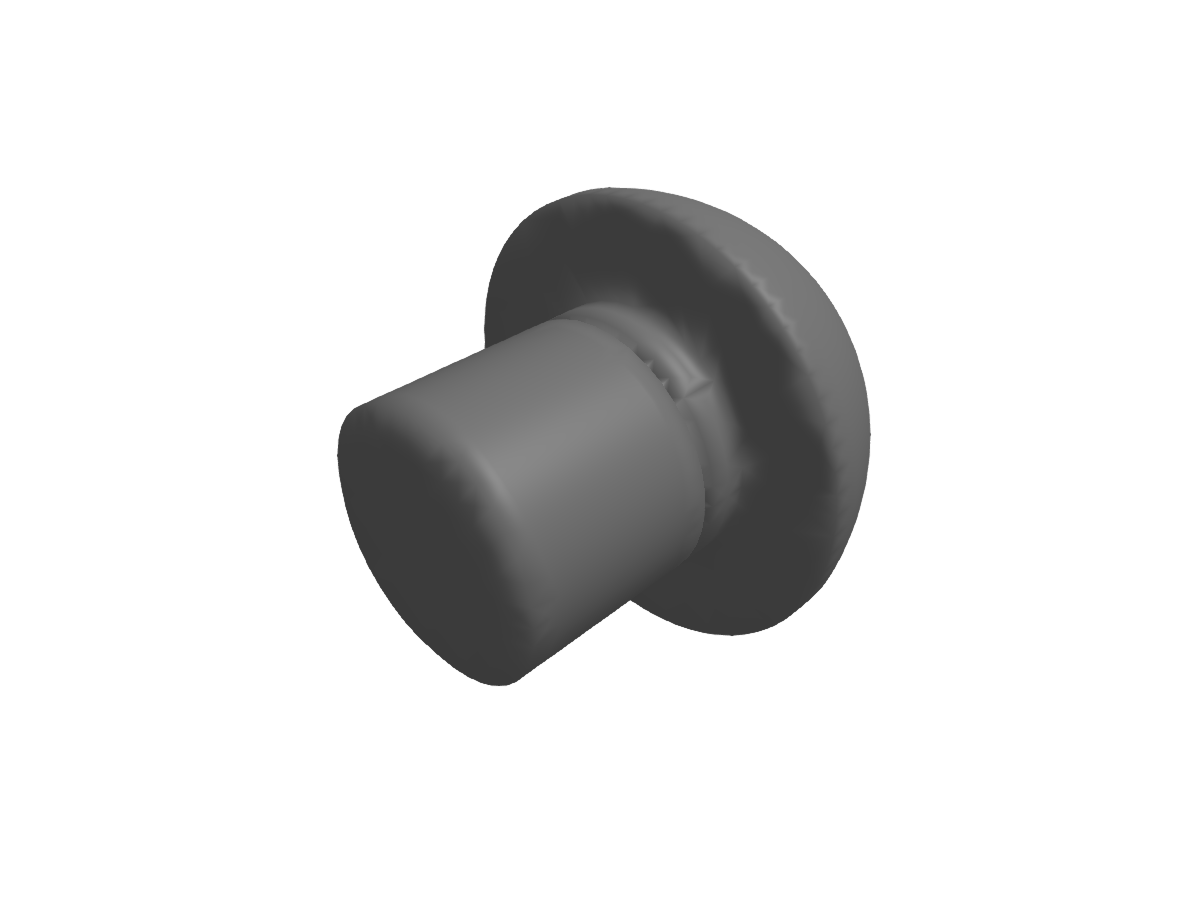}
        &\includegraphics[width=0.1\linewidth, clip=true, trim=140pt 50pt 140pt 50pt]{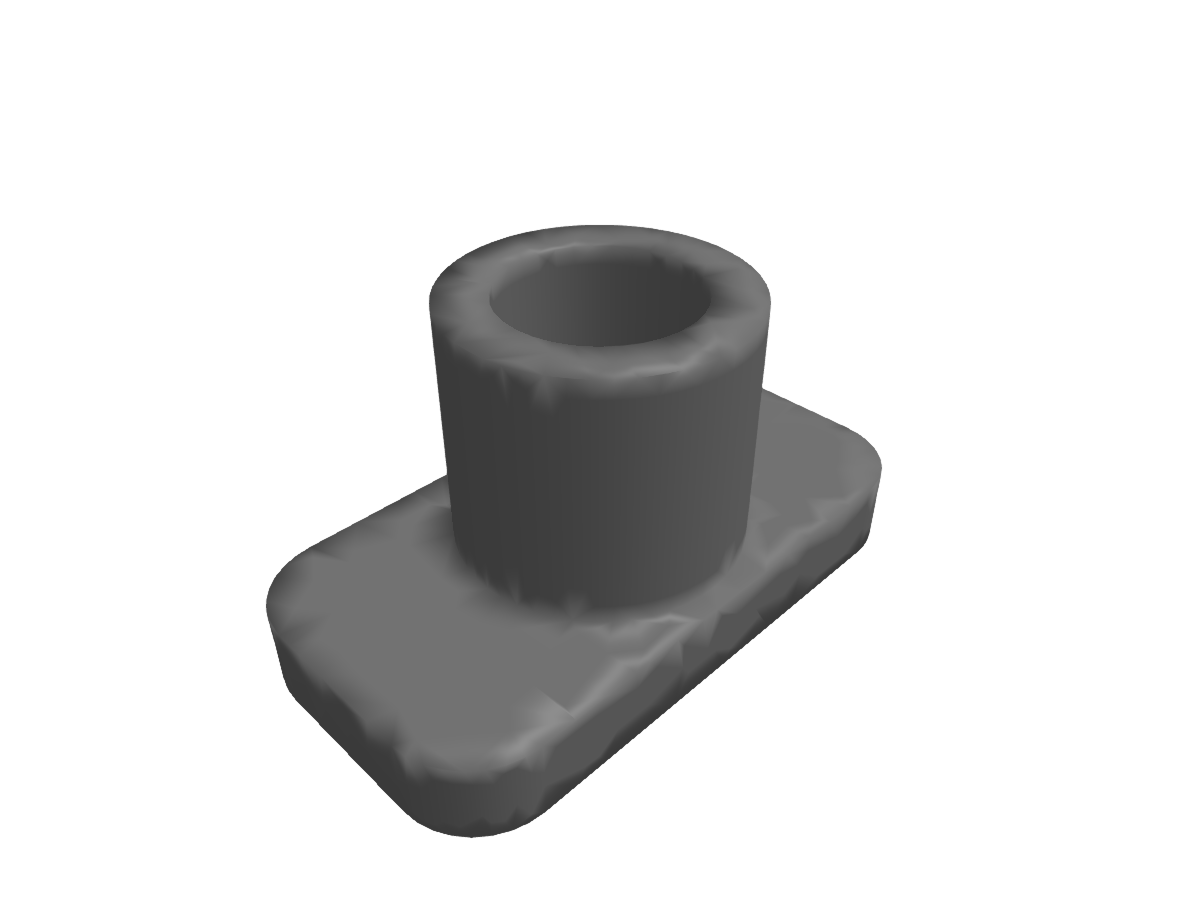}
        &\includegraphics[width=0.1\linewidth, clip=true, trim=140pt 50pt 140pt 50pt]{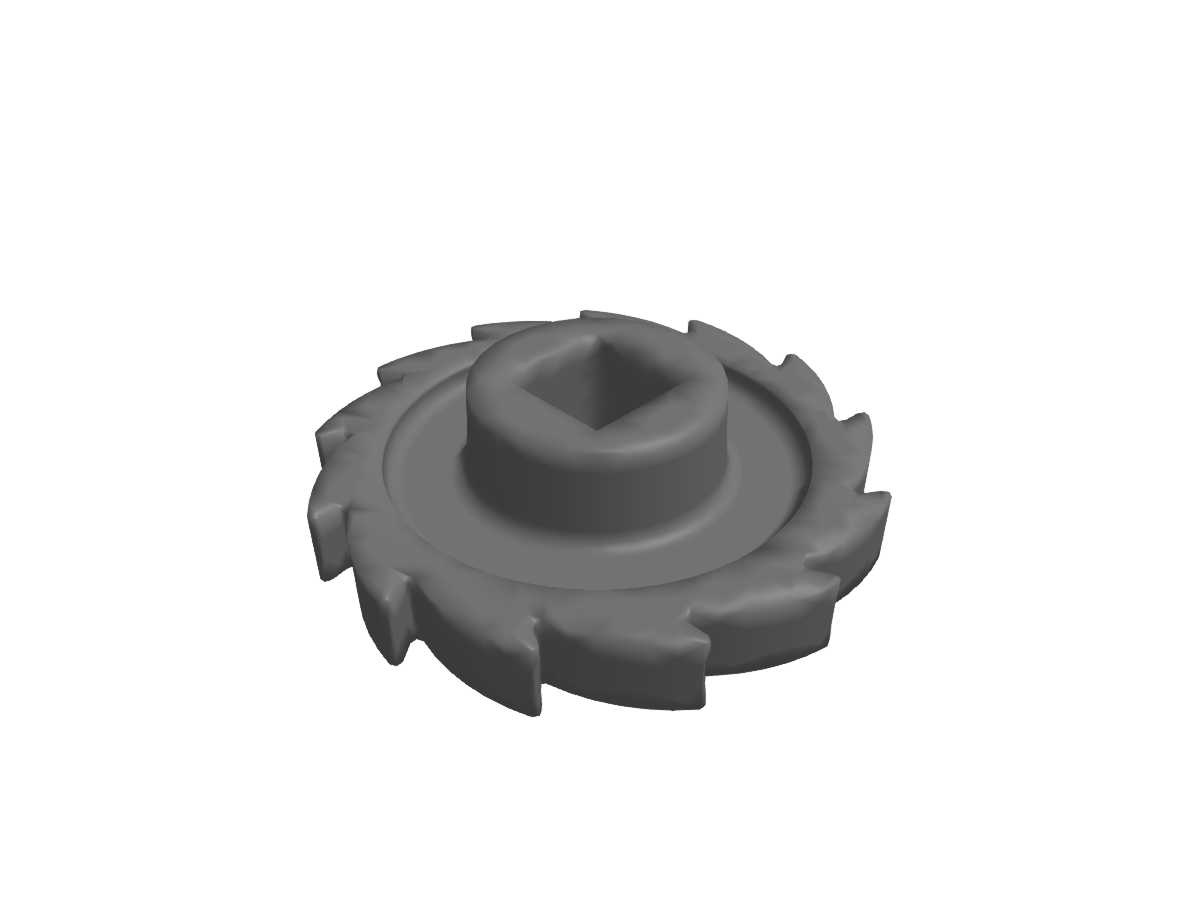}
        \\
        \includegraphics[width=0.1\linewidth, clip=true, trim=140pt 50pt 140pt 50pt]{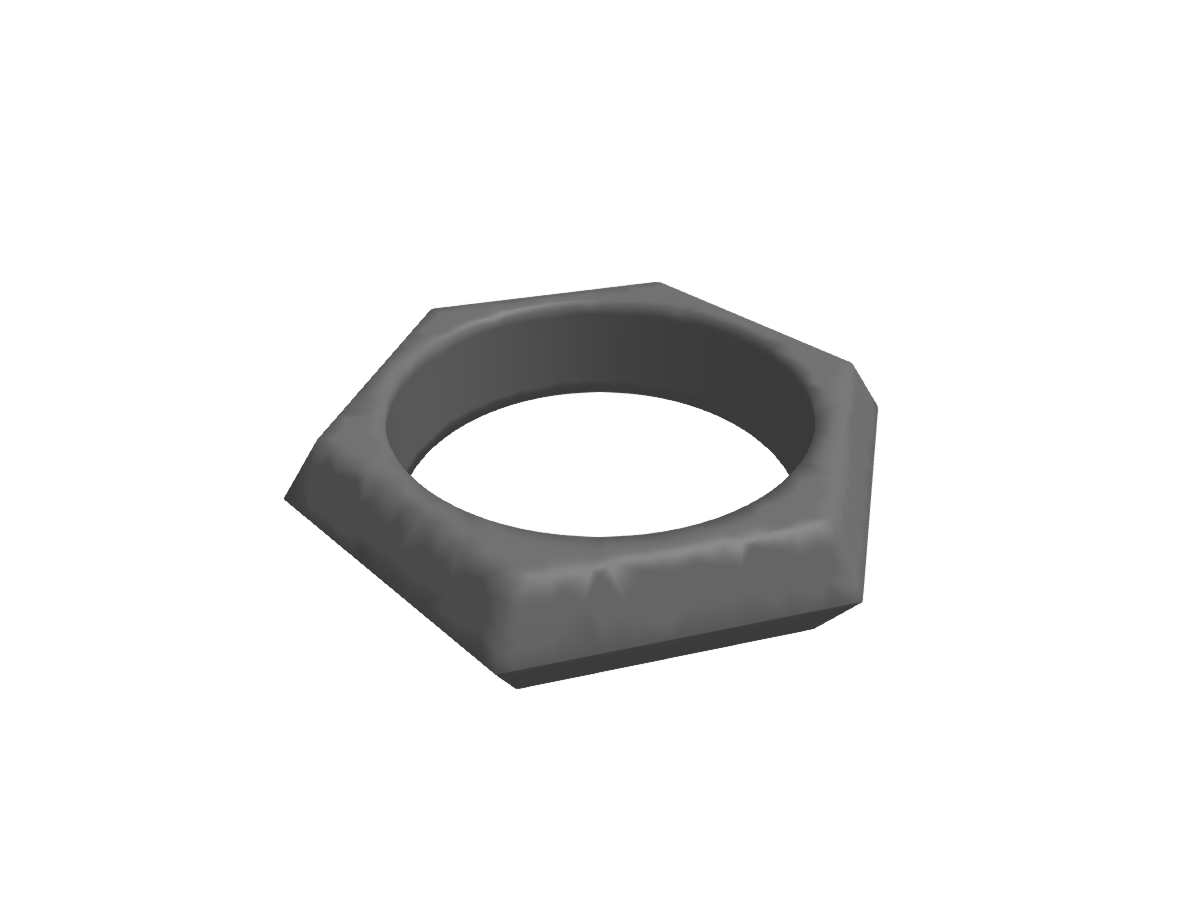}
        &\includegraphics[width=0.1\linewidth, clip=true, trim=140pt 50pt 140pt 50pt]{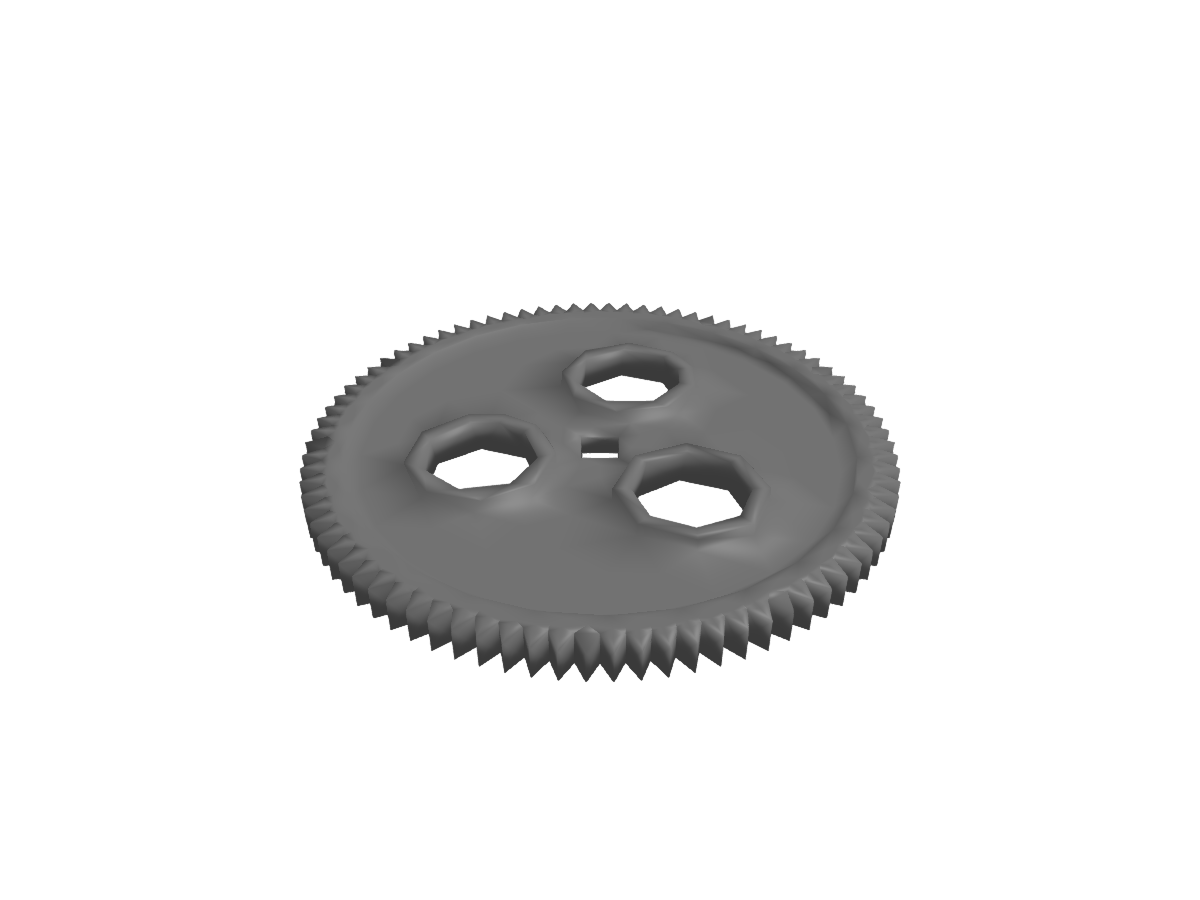}
        &\includegraphics[width=0.1\linewidth, clip=true, trim=140pt 50pt 140pt 50pt]{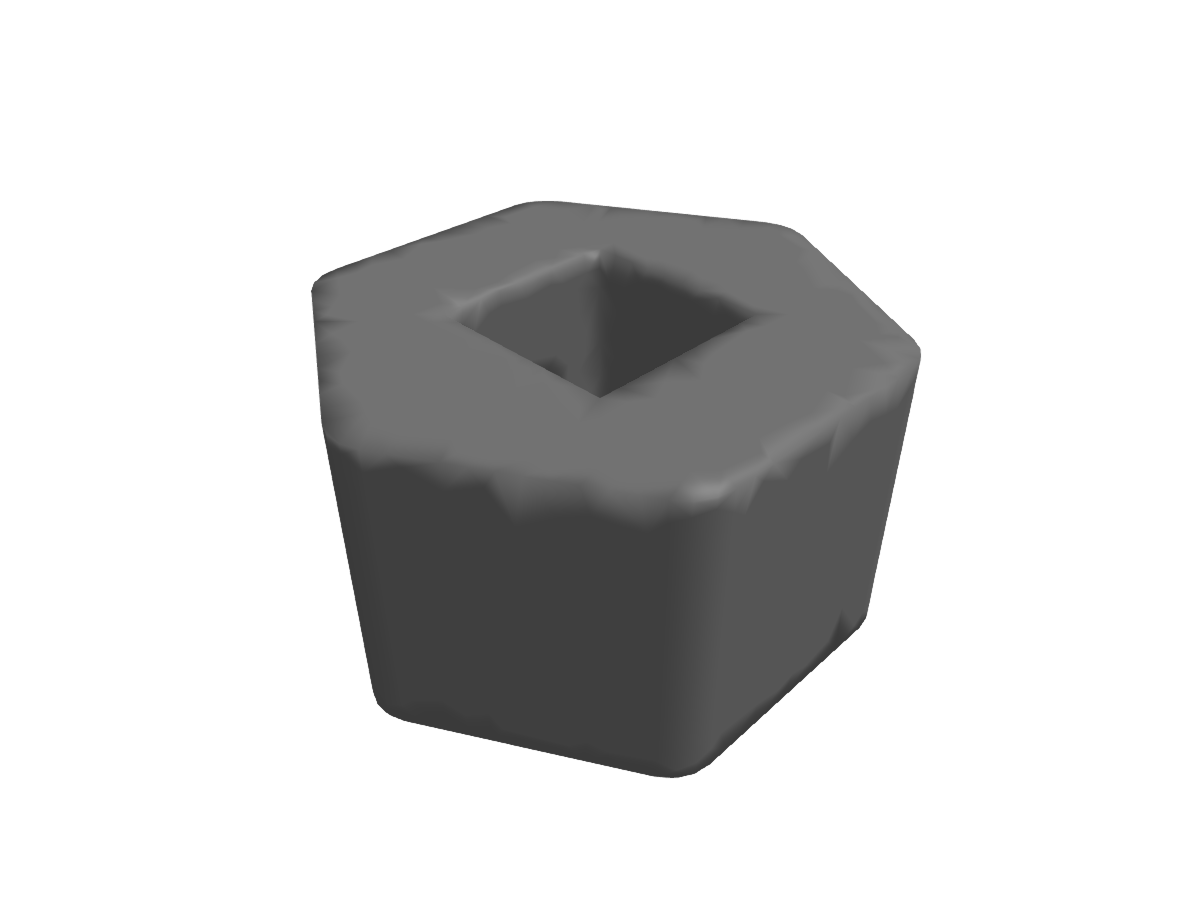}
        &\includegraphics[width=0.1\linewidth, clip=true, trim=140pt 50pt 140pt 50pt]{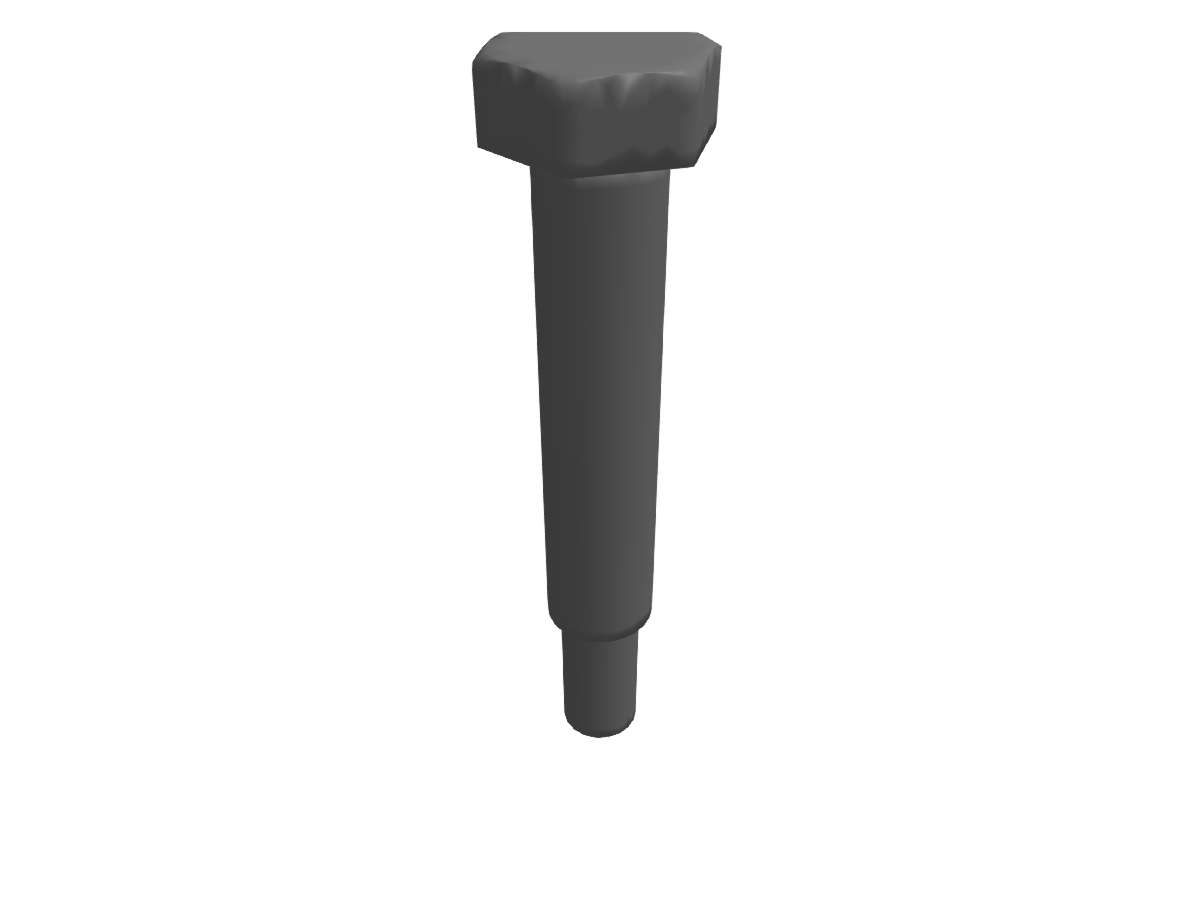}
        &\includegraphics[width=0.1\linewidth, clip=true, trim=140pt 50pt 140pt 50pt]{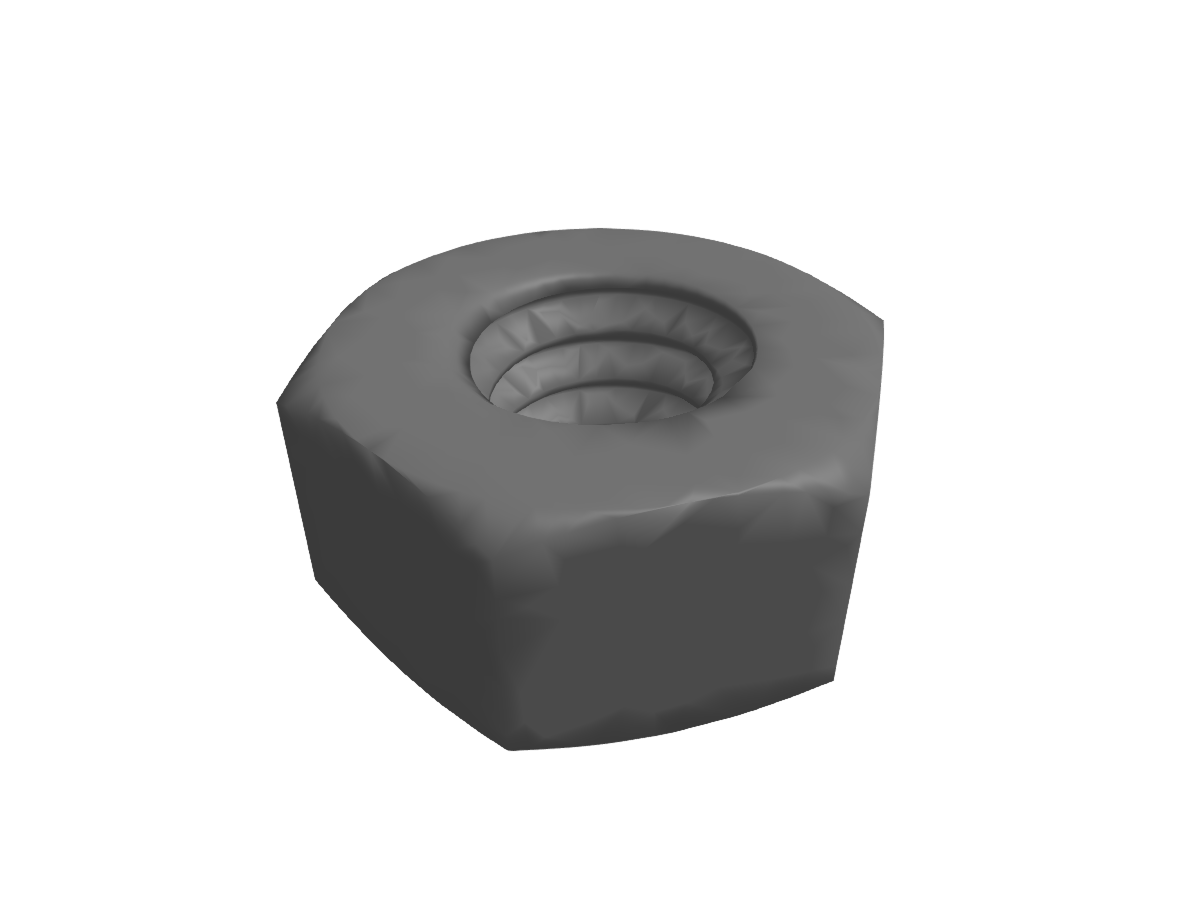}
        &\includegraphics[width=0.1\linewidth, clip=true, trim=140pt 50pt 140pt 50pt]{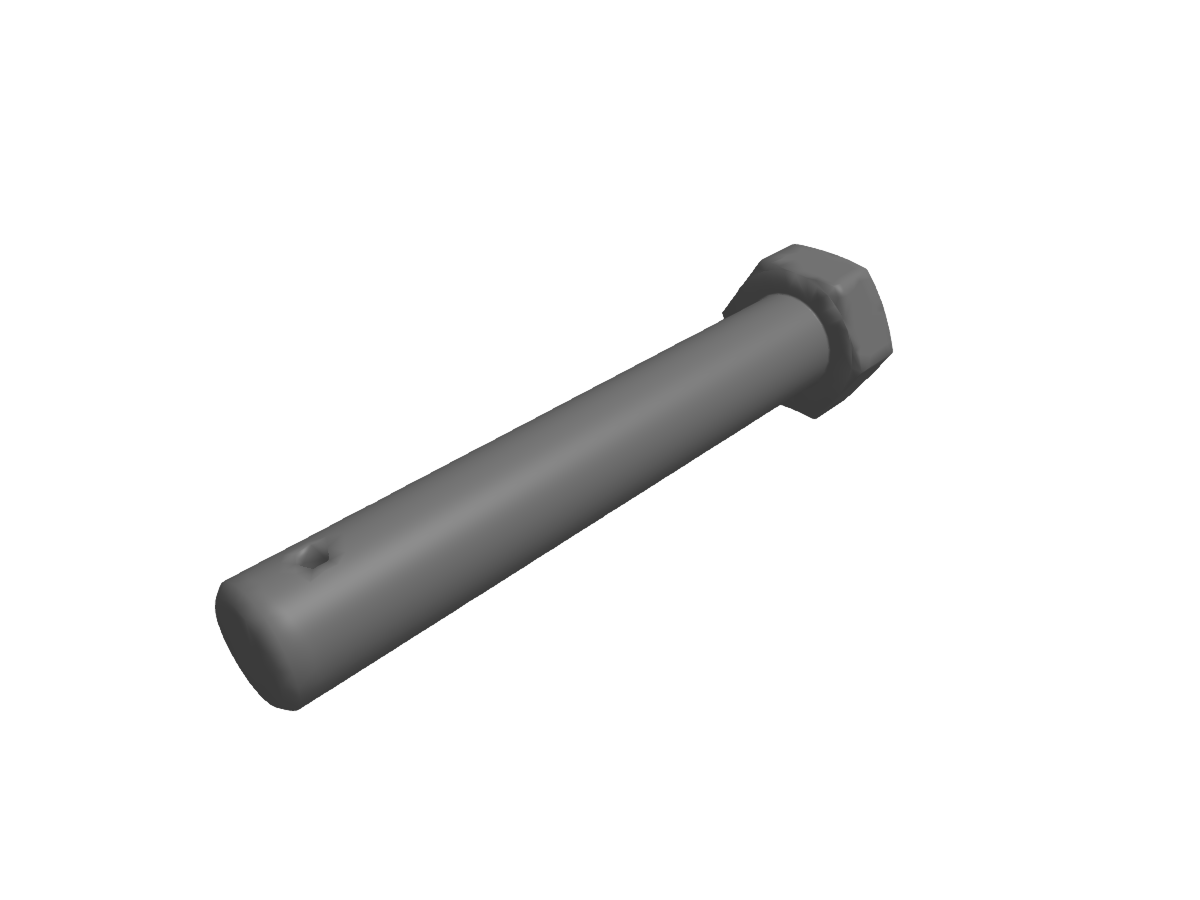}
        &\includegraphics[width=0.1\linewidth, clip=true, trim=140pt 50pt 140pt 50pt]{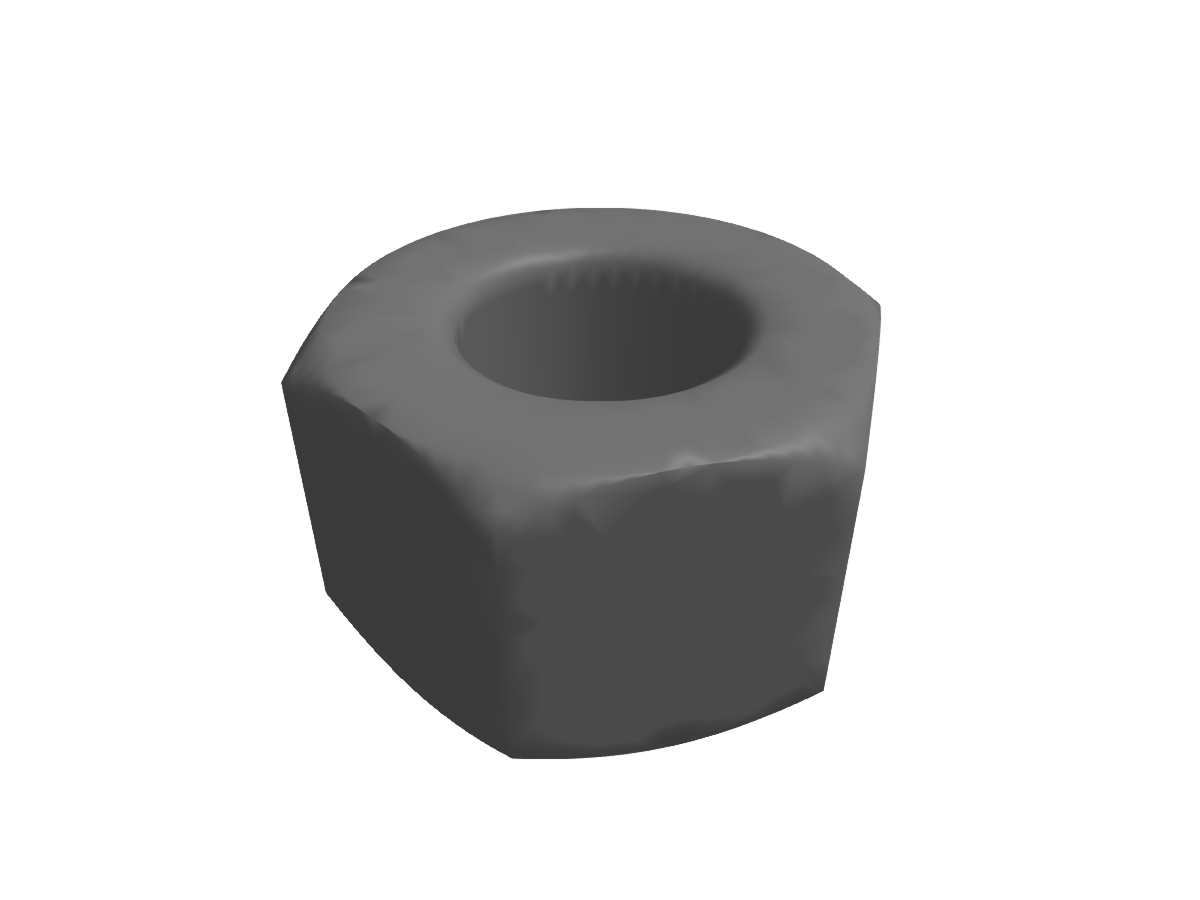}
        &\includegraphics[width=0.1\linewidth, clip=true, trim=140pt 50pt 140pt 50pt]{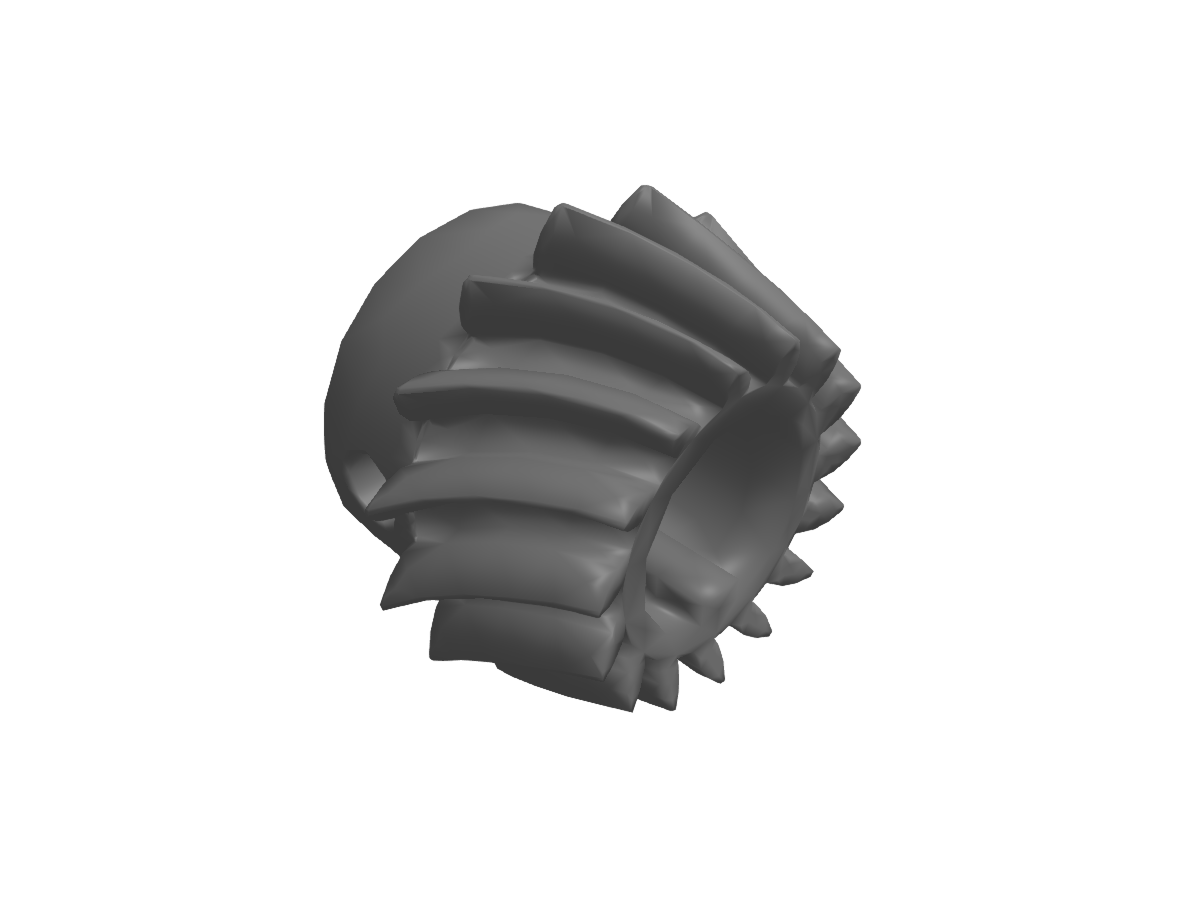}
        &\includegraphics[width=0.1\linewidth, clip=true, trim=140pt 50pt 140pt 50pt]{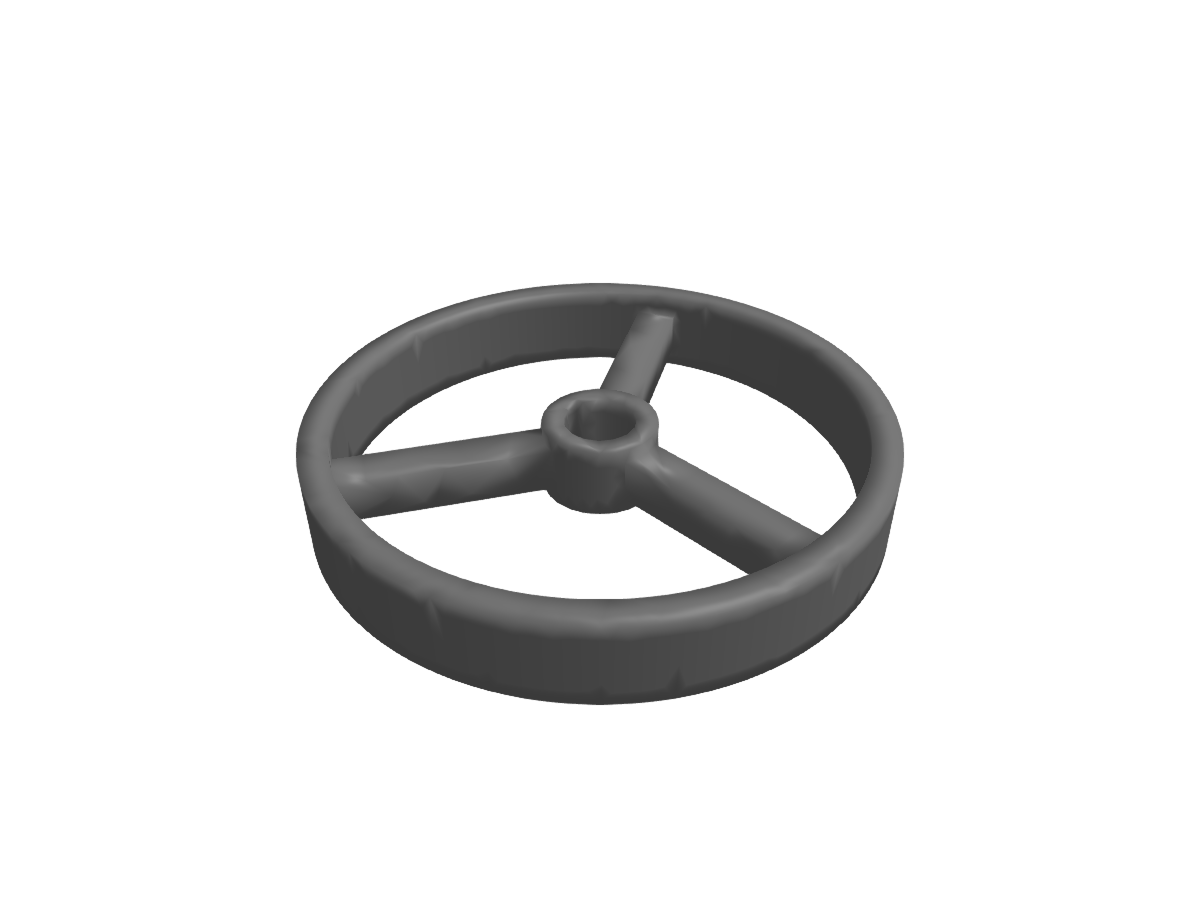}
        &\includegraphics[width=0.1\linewidth, clip=true, trim=140pt 50pt 140pt 50pt]{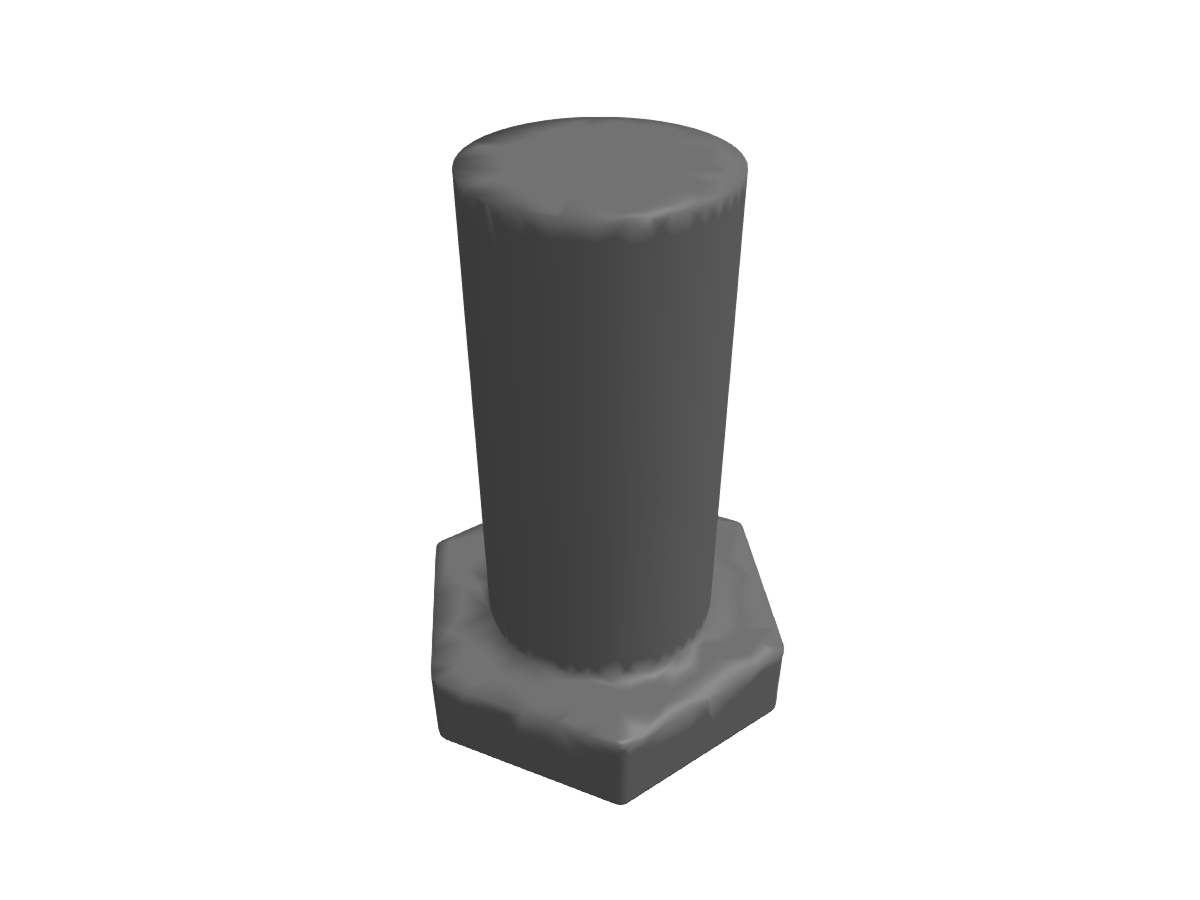}
        \\
        \includegraphics[width=0.1\linewidth, clip=true, trim=140pt 50pt 140pt 50pt]{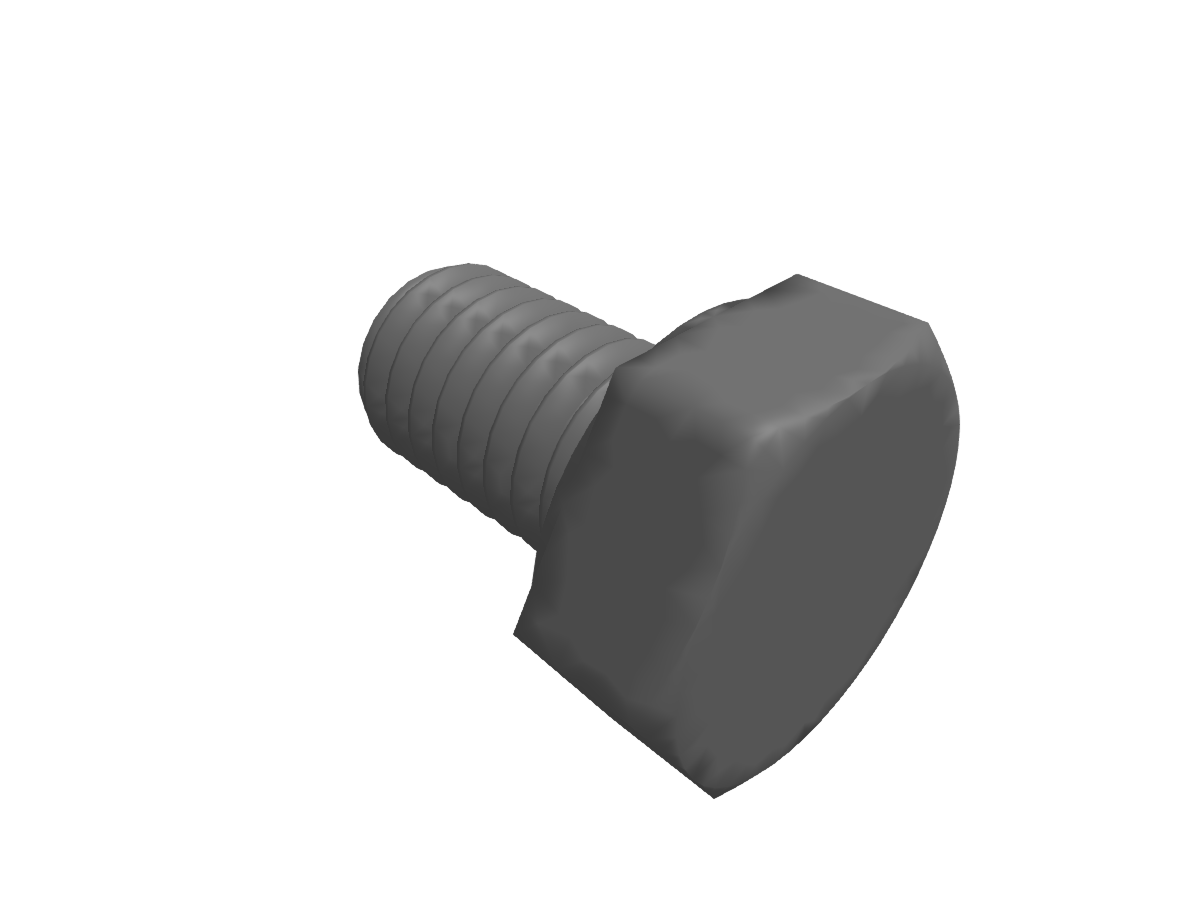}
        &\includegraphics[width=0.1\linewidth, clip=true, trim=140pt 50pt 140pt 50pt]{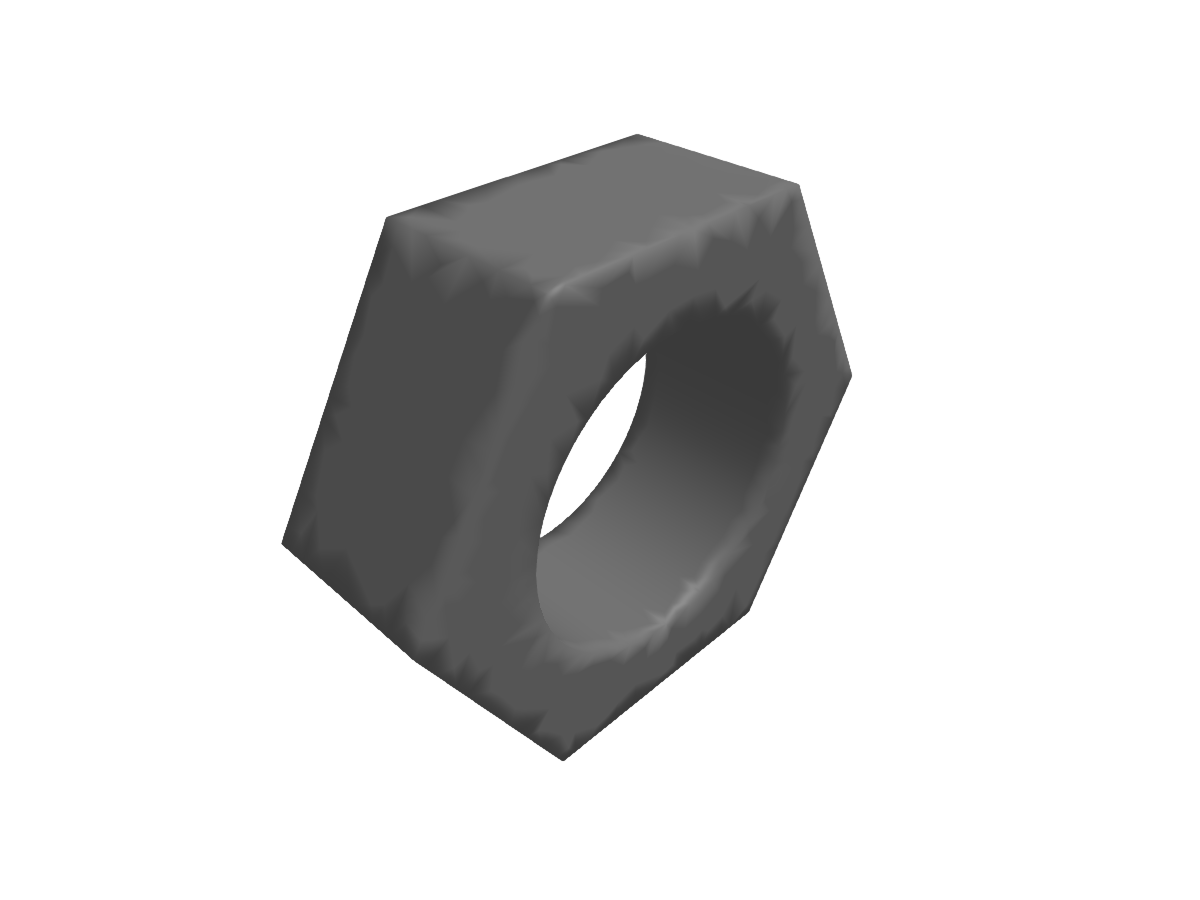}
        &\includegraphics[width=0.1\linewidth, clip=true, trim=140pt 50pt 140pt 50pt]{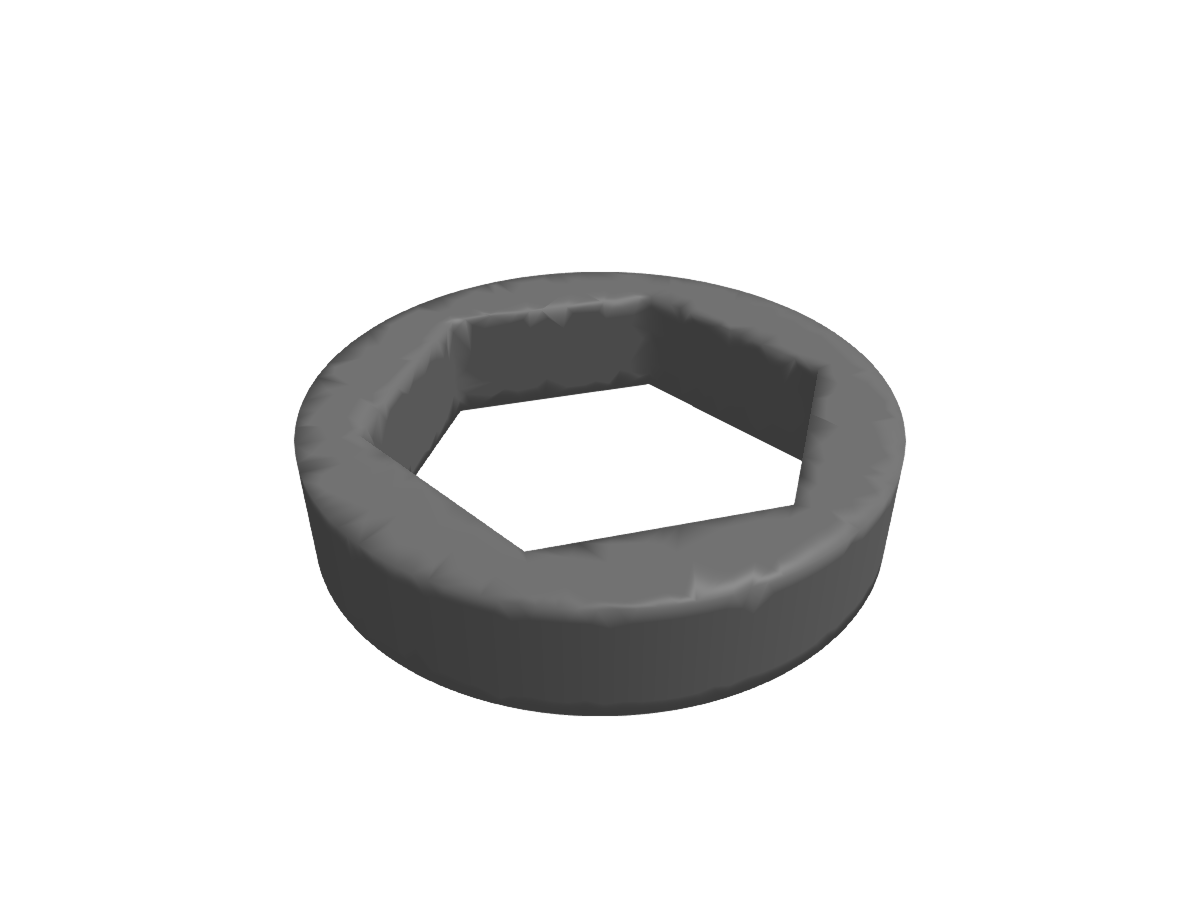}
        &\includegraphics[width=0.1\linewidth, clip=true, trim=140pt 50pt 140pt 50pt]{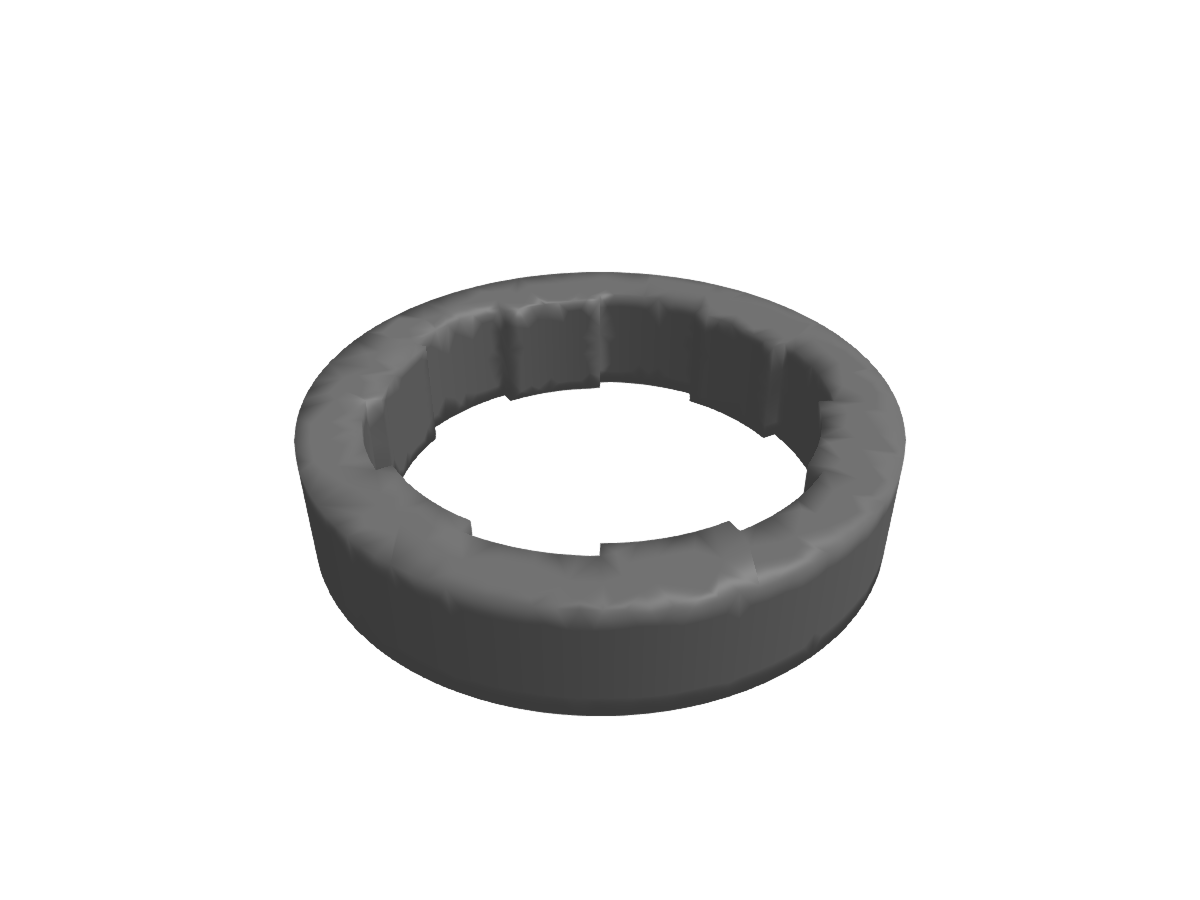}
        &\includegraphics[width=0.1\linewidth, clip=true, trim=140pt 50pt 140pt 50pt]{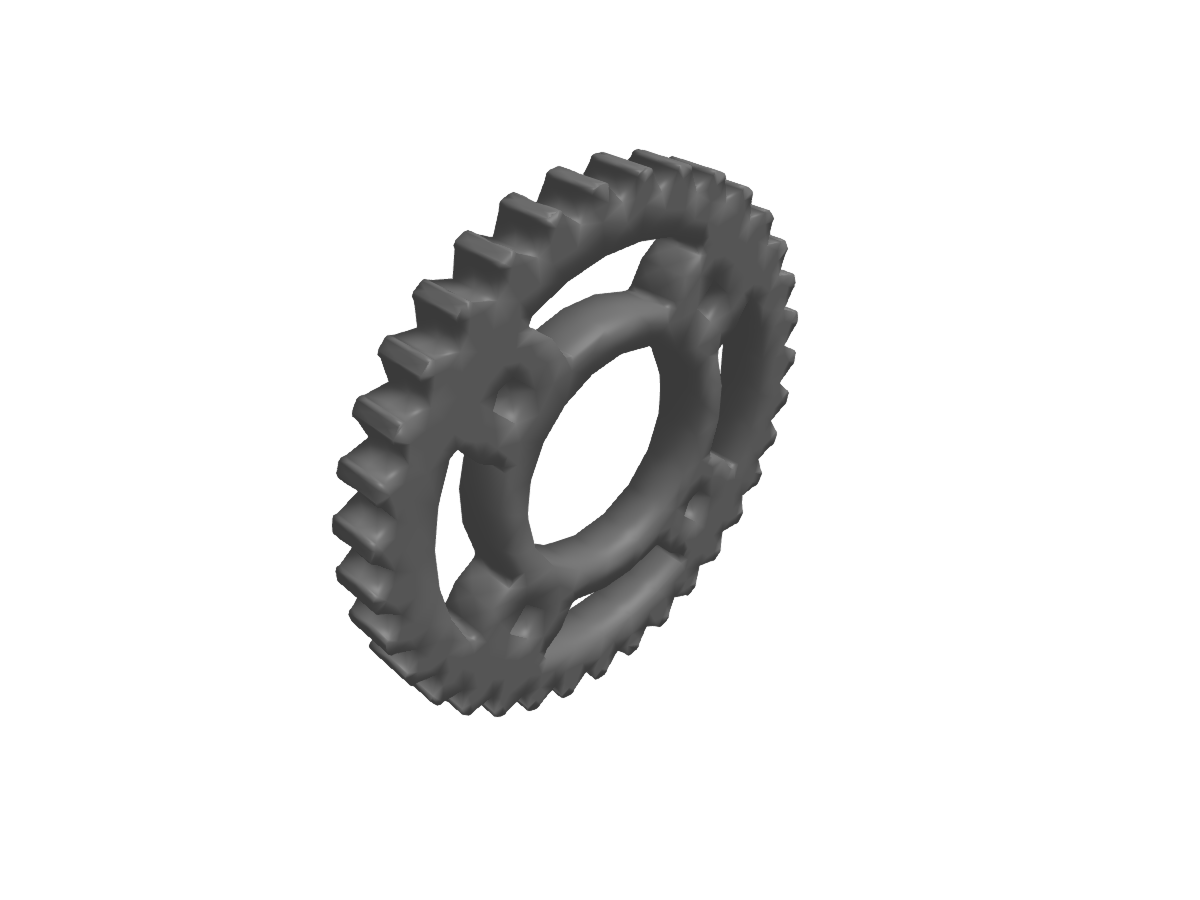}
        &\includegraphics[width=0.1\linewidth, clip=true, trim=140pt 50pt 140pt 50pt]{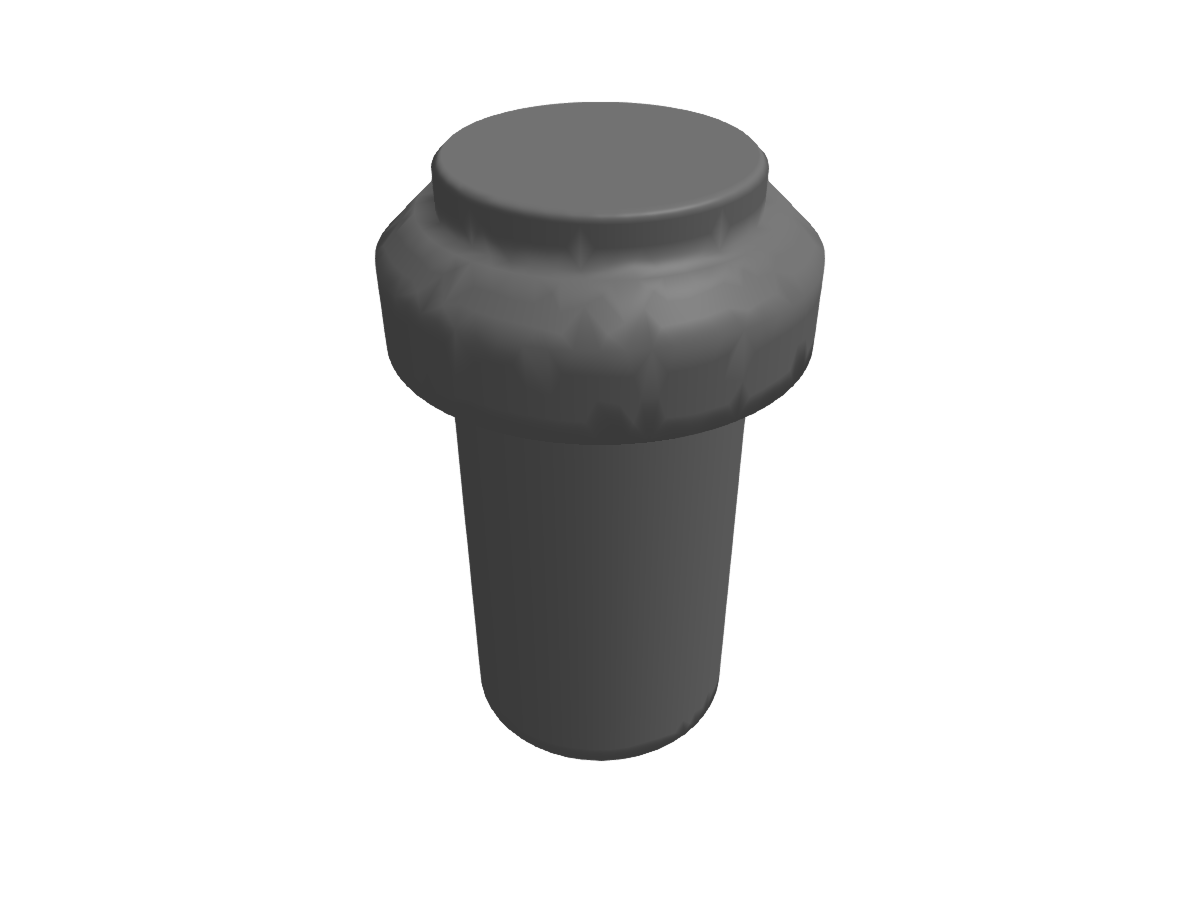}
        &\includegraphics[width=0.1\linewidth, clip=true, trim=140pt 50pt 140pt 50pt]{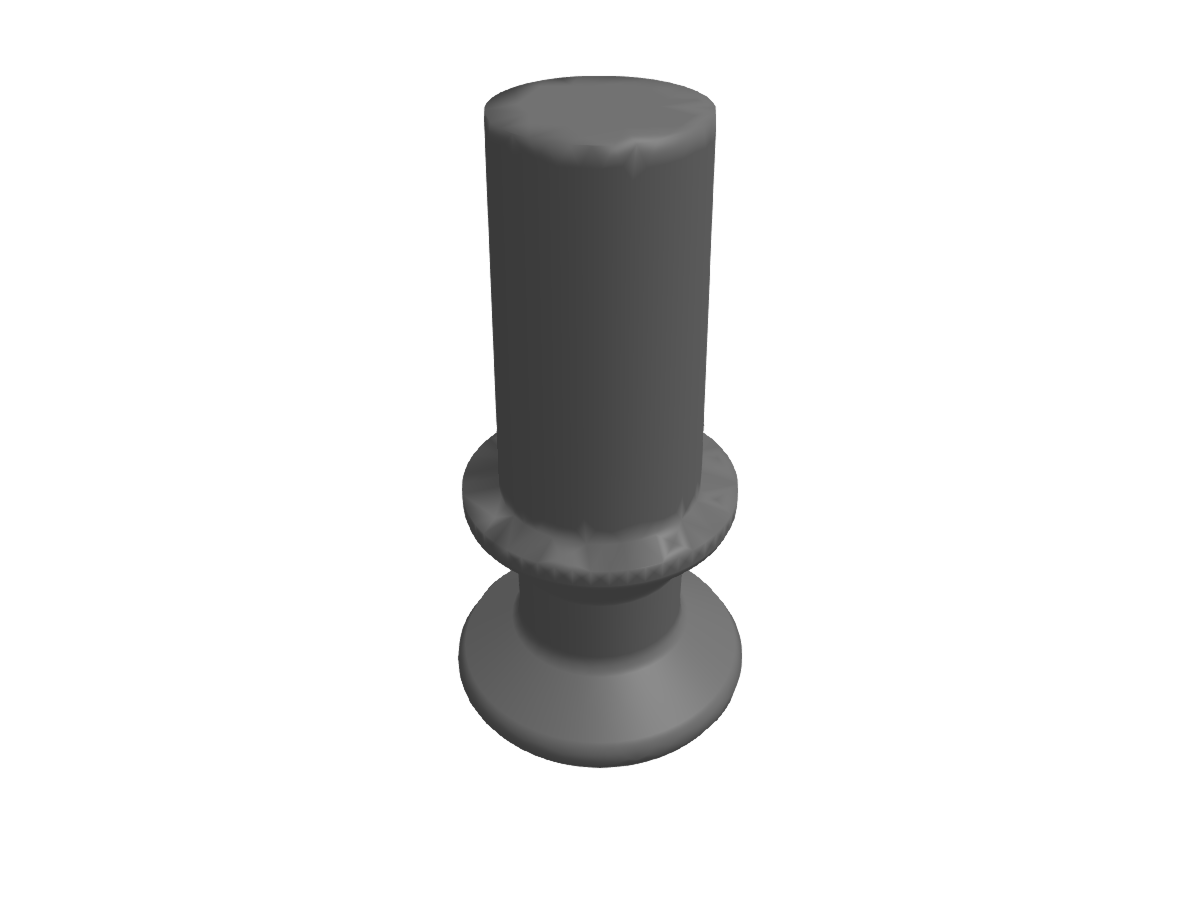}
        &\includegraphics[width=0.1\linewidth, clip=true, trim=140pt 50pt 140pt 50pt]{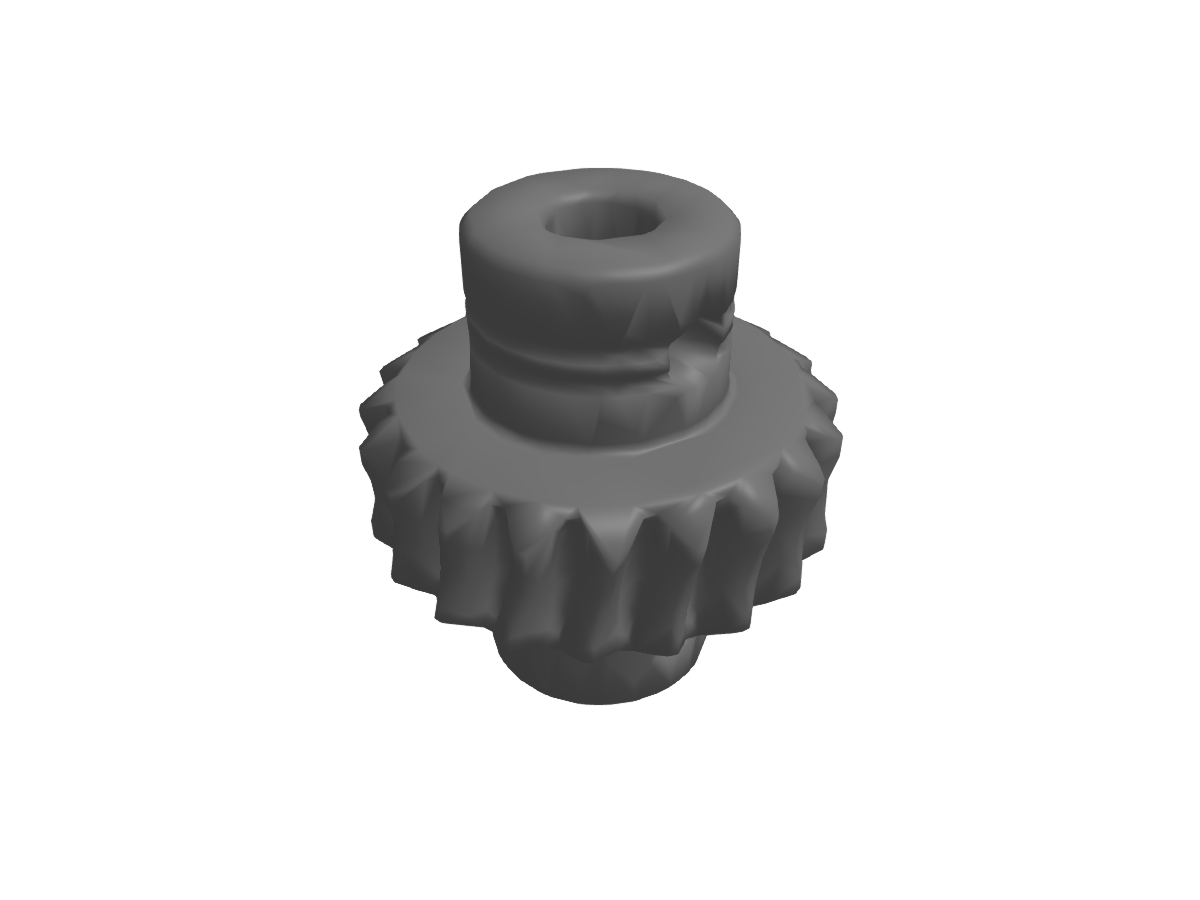}
        &\includegraphics[width=0.1\linewidth, clip=true, trim=140pt 50pt 140pt 50pt]{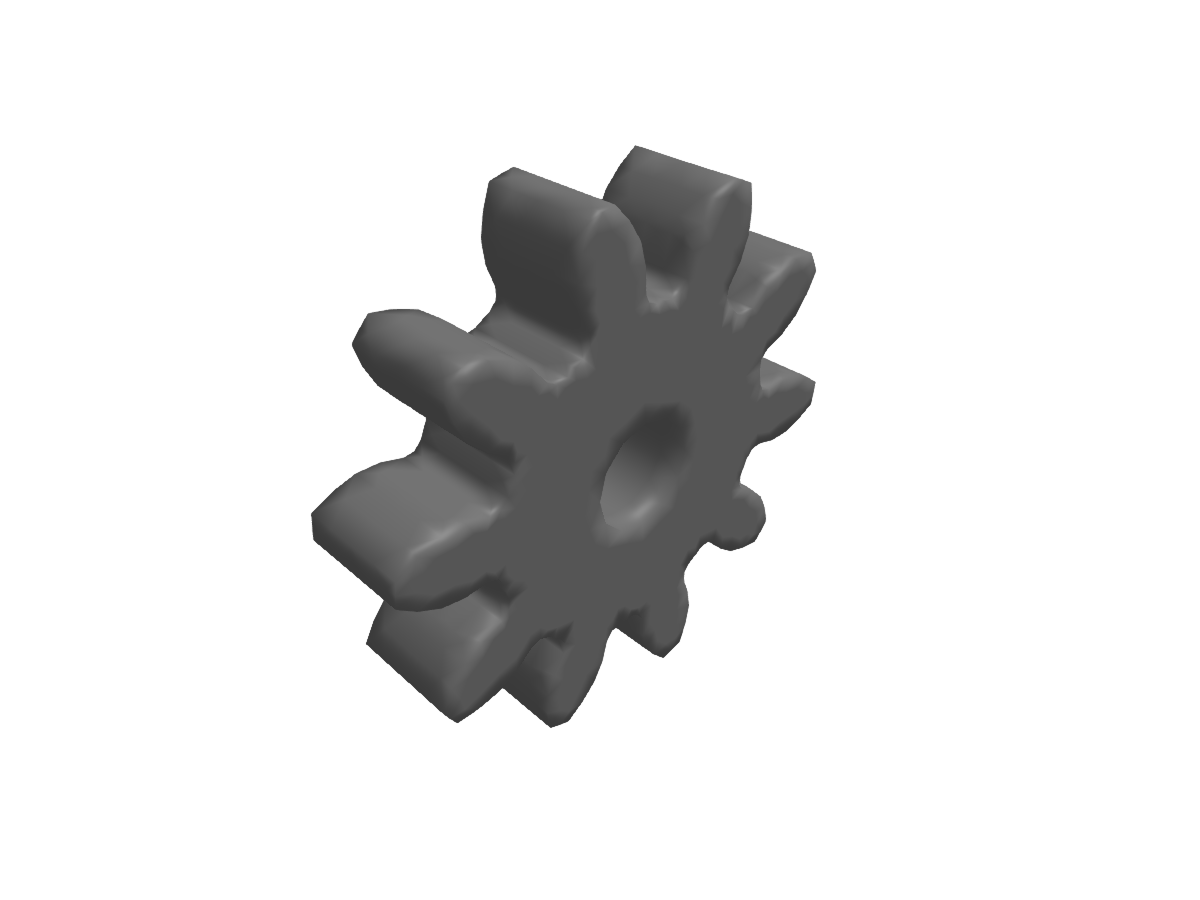}
        &\includegraphics[width=0.1\linewidth, clip=true, trim=140pt 50pt 140pt 50pt]{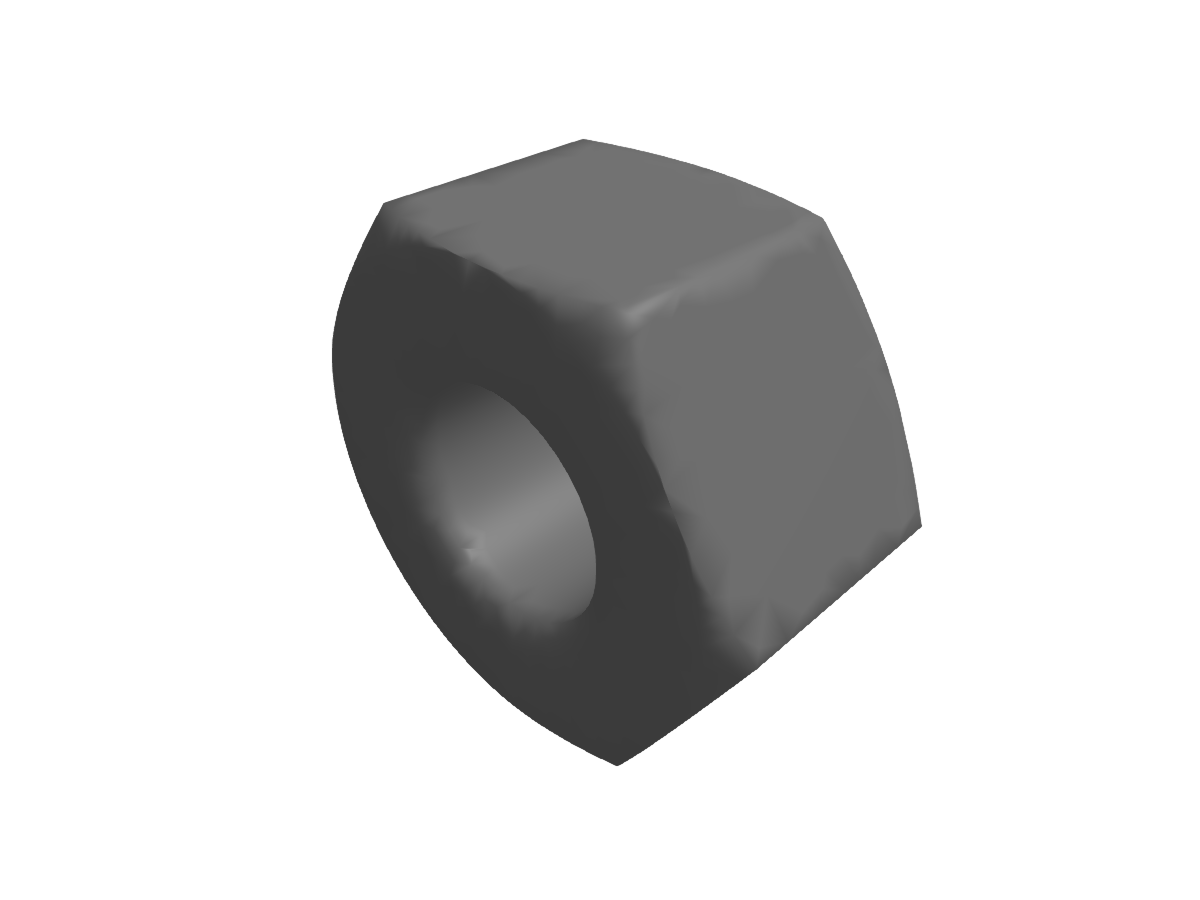}
        \\
        \includegraphics[width=0.1\linewidth, clip=true, trim=140pt 50pt 140pt 50pt]{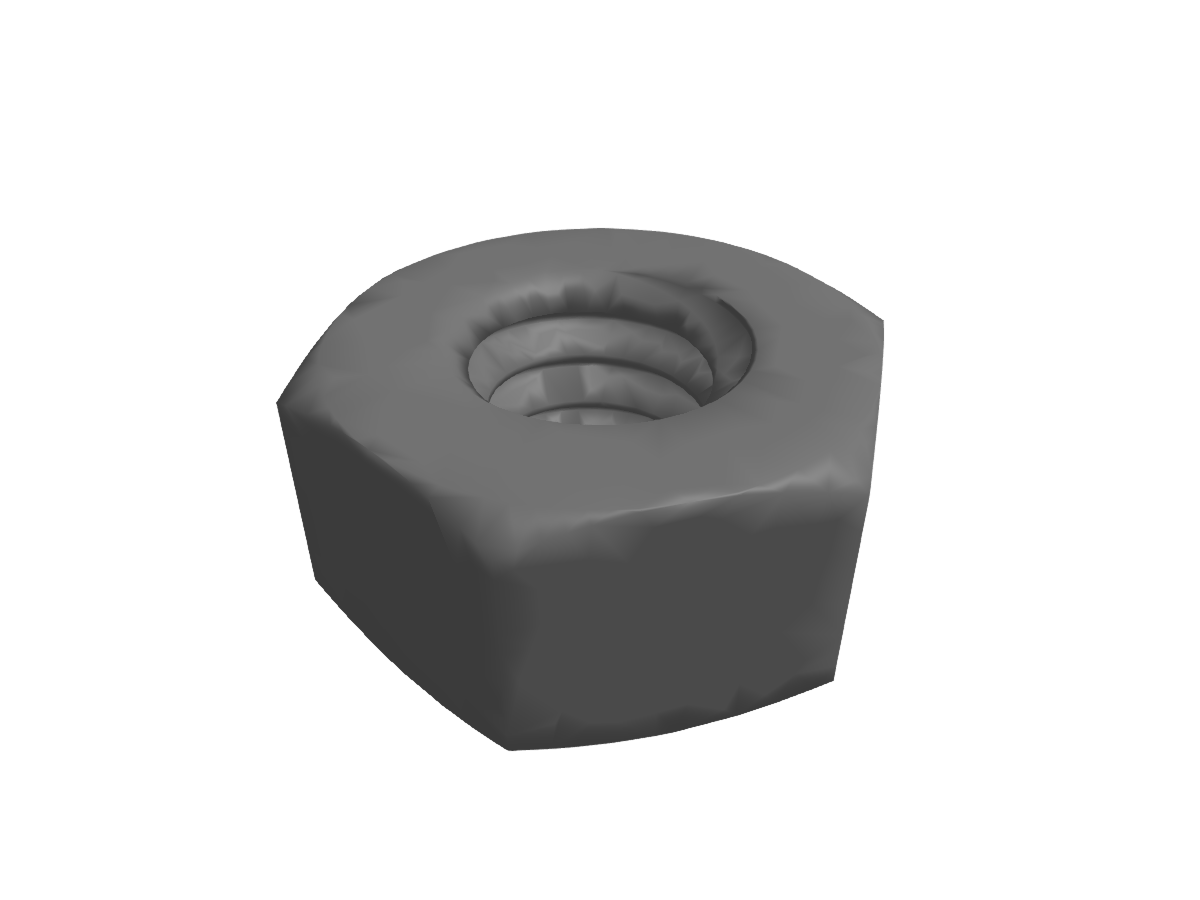}
        &\includegraphics[width=0.1\linewidth, clip=true, trim=140pt 50pt 140pt 50pt]{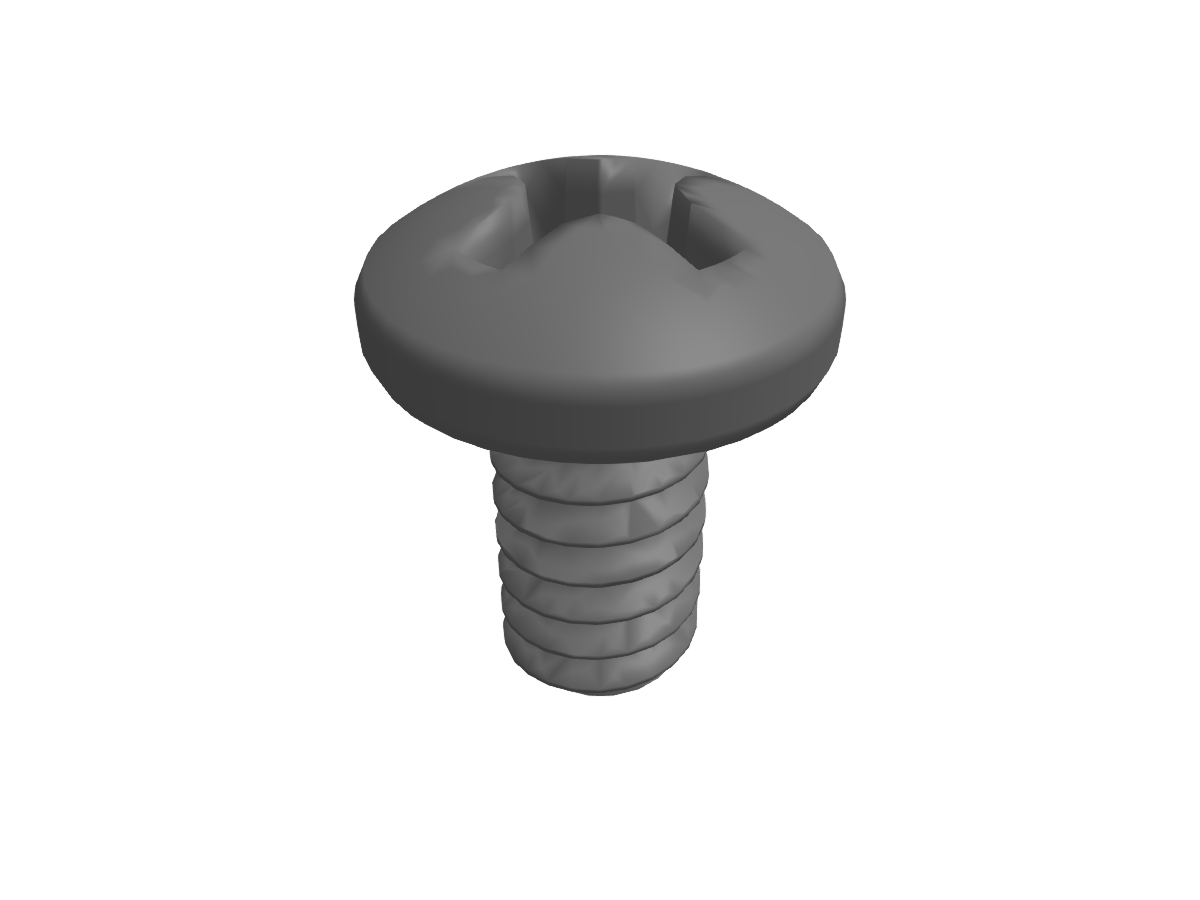}
        &\includegraphics[width=0.1\linewidth, clip=true, trim=140pt 50pt 140pt 50pt]{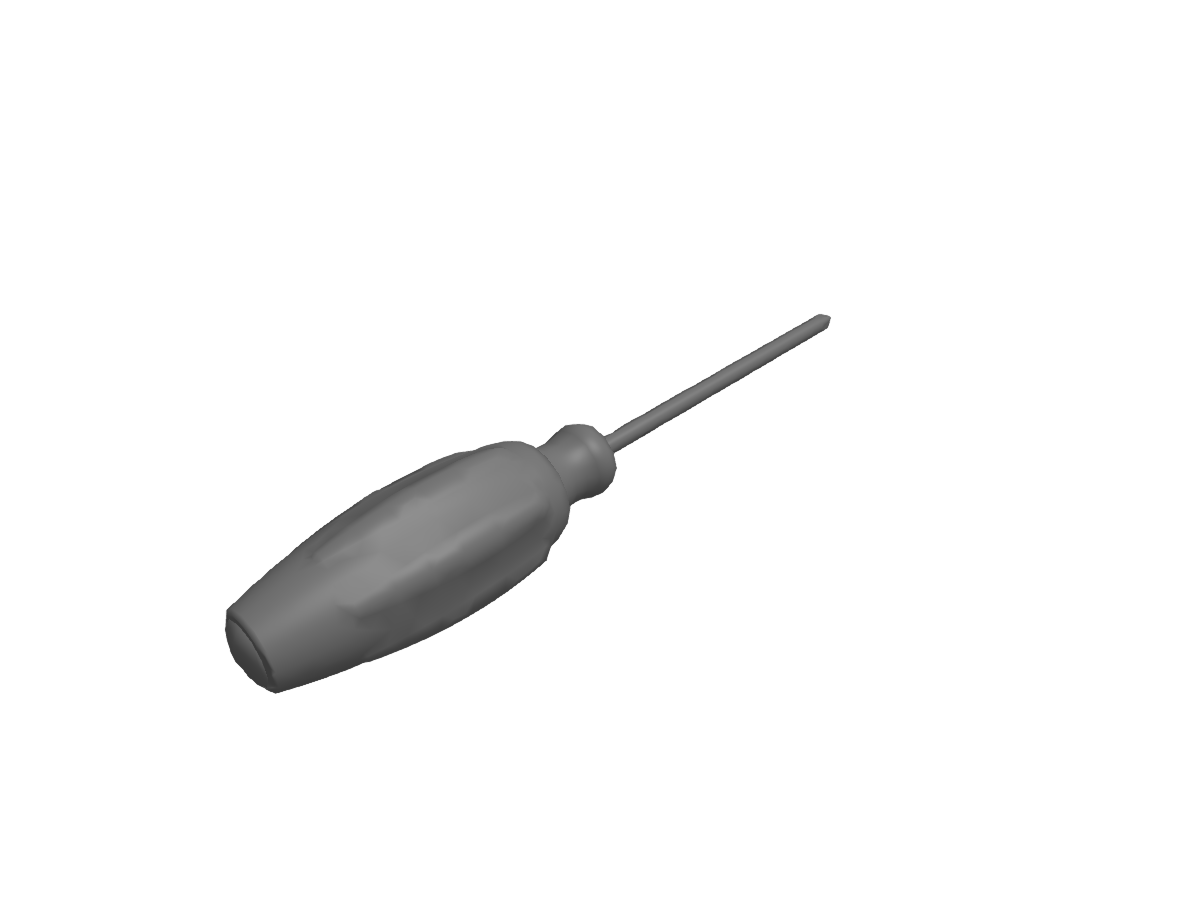}
        &\includegraphics[width=0.1\linewidth, clip=true, trim=140pt 50pt 140pt 50pt]{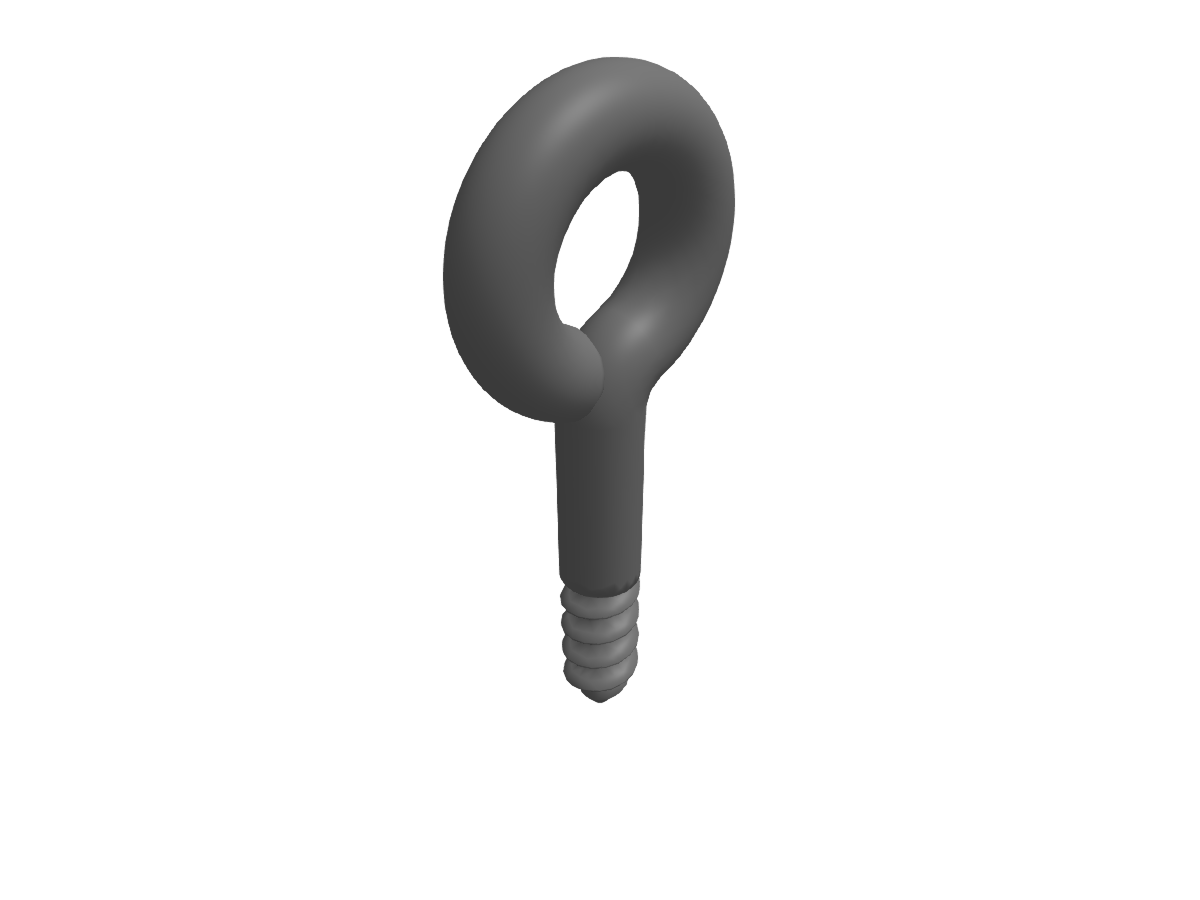}
        &\includegraphics[width=0.1\linewidth, clip=true, trim=140pt 50pt 140pt 50pt]{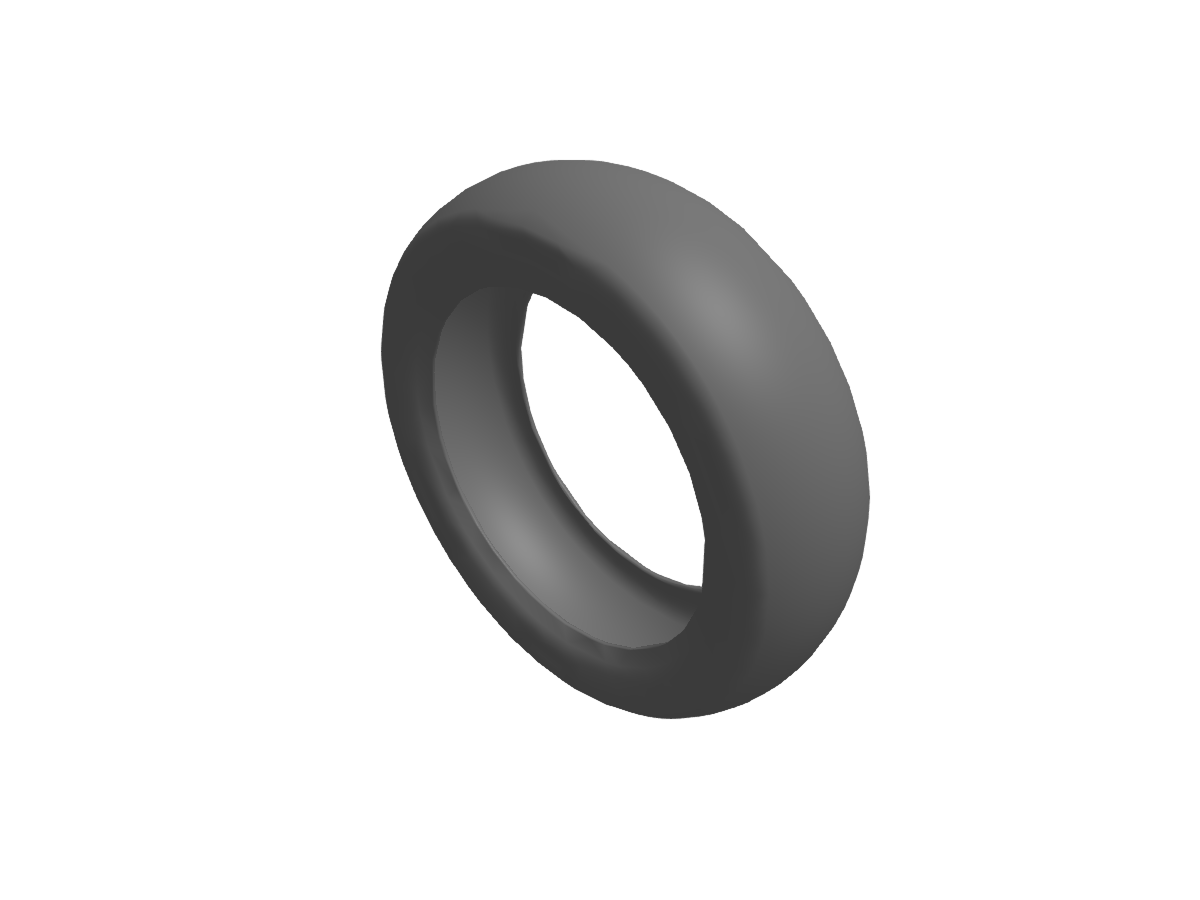}
        &\includegraphics[width=0.1\linewidth, clip=true, trim=140pt 50pt 140pt 50pt]{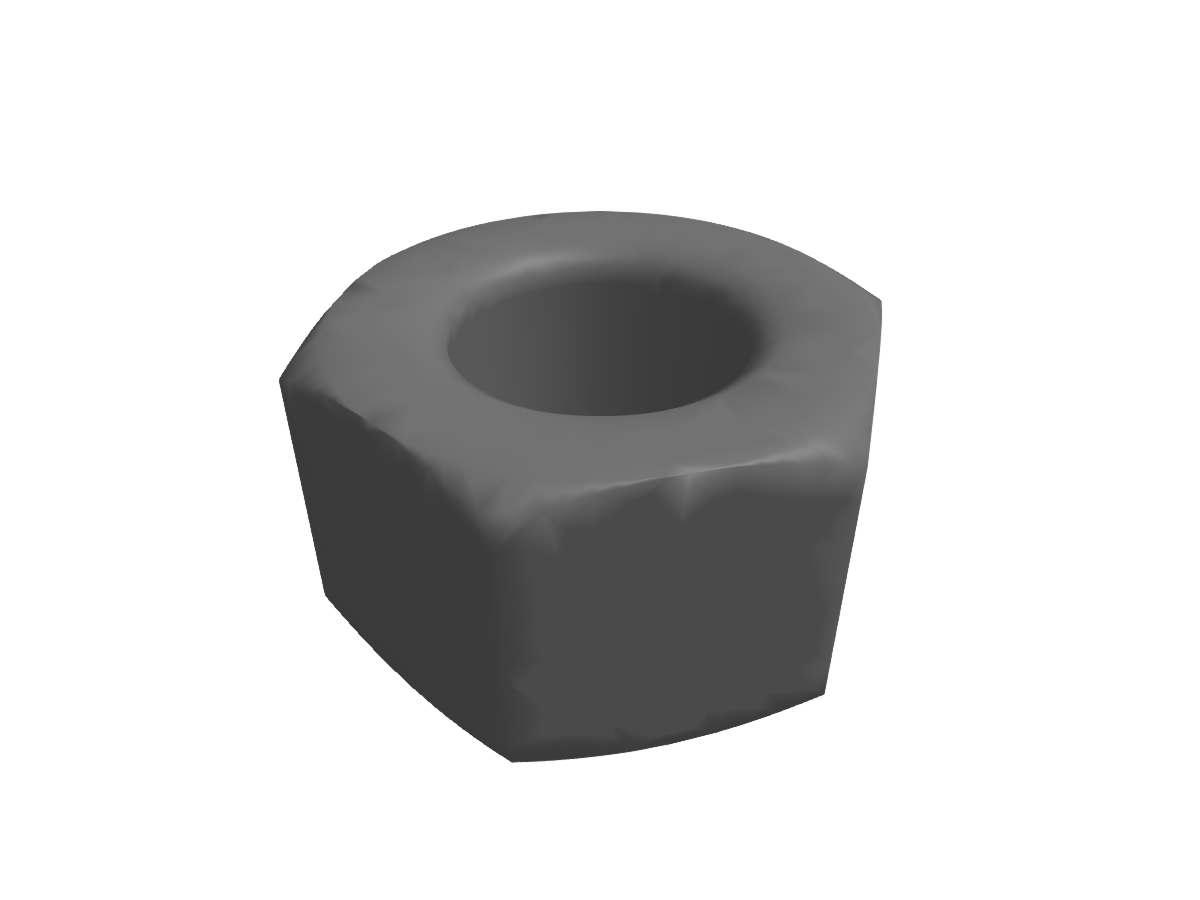}
        &\includegraphics[width=0.1\linewidth, clip=true, trim=140pt 50pt 140pt 50pt]{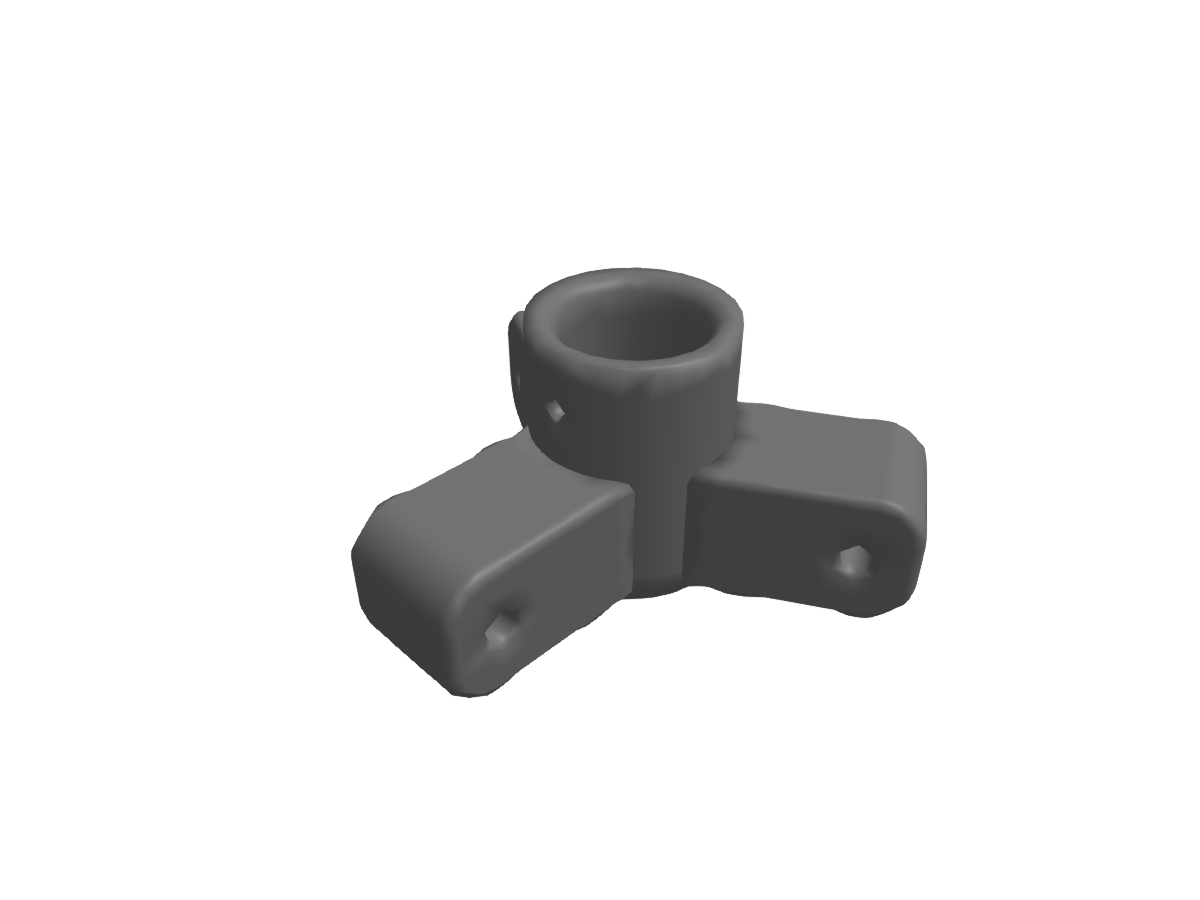}
        &\includegraphics[width=0.1\linewidth, clip=true, trim=140pt 50pt 140pt 50pt]{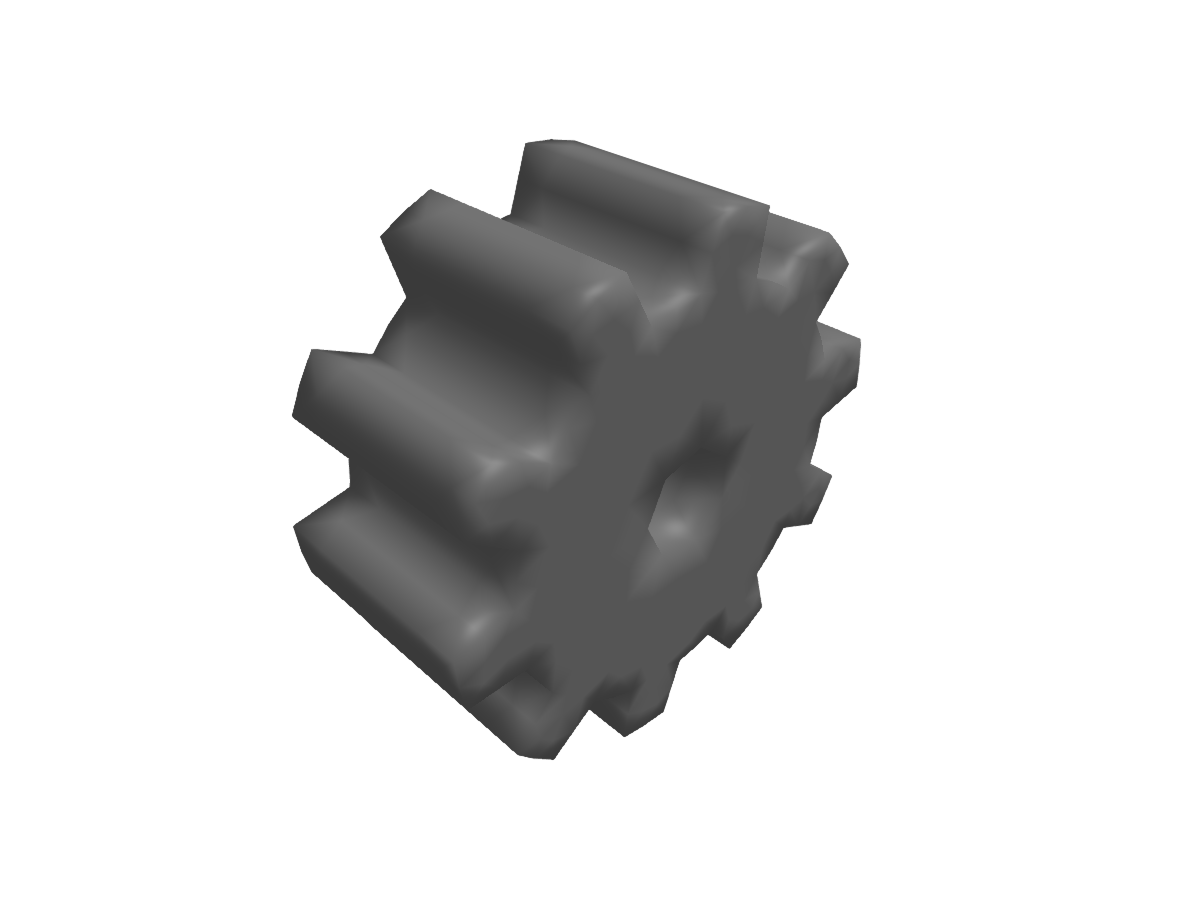}
        &\includegraphics[width=0.1\linewidth, clip=true, trim=140pt 50pt 140pt 50pt]{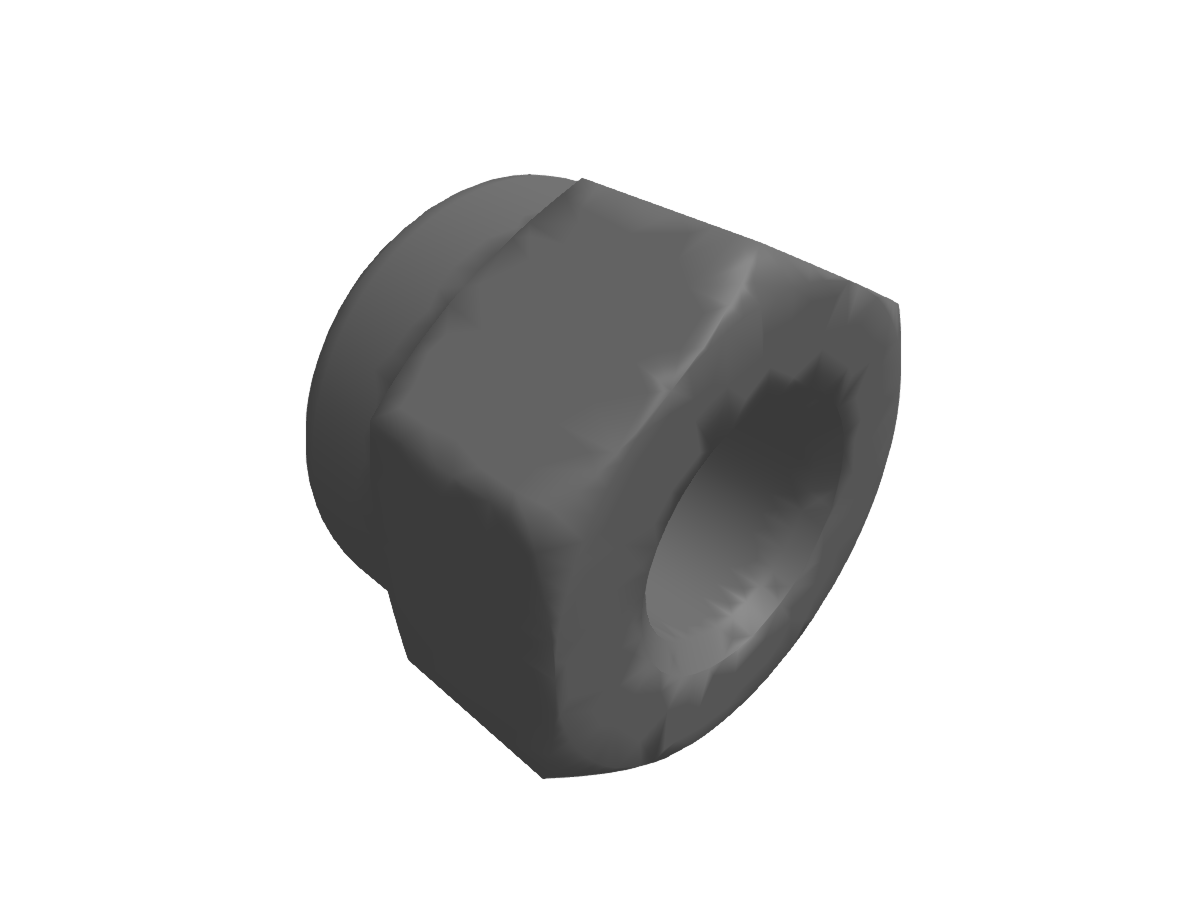}
        &\includegraphics[width=0.1\linewidth, clip=true, trim=140pt 50pt 140pt 50pt]{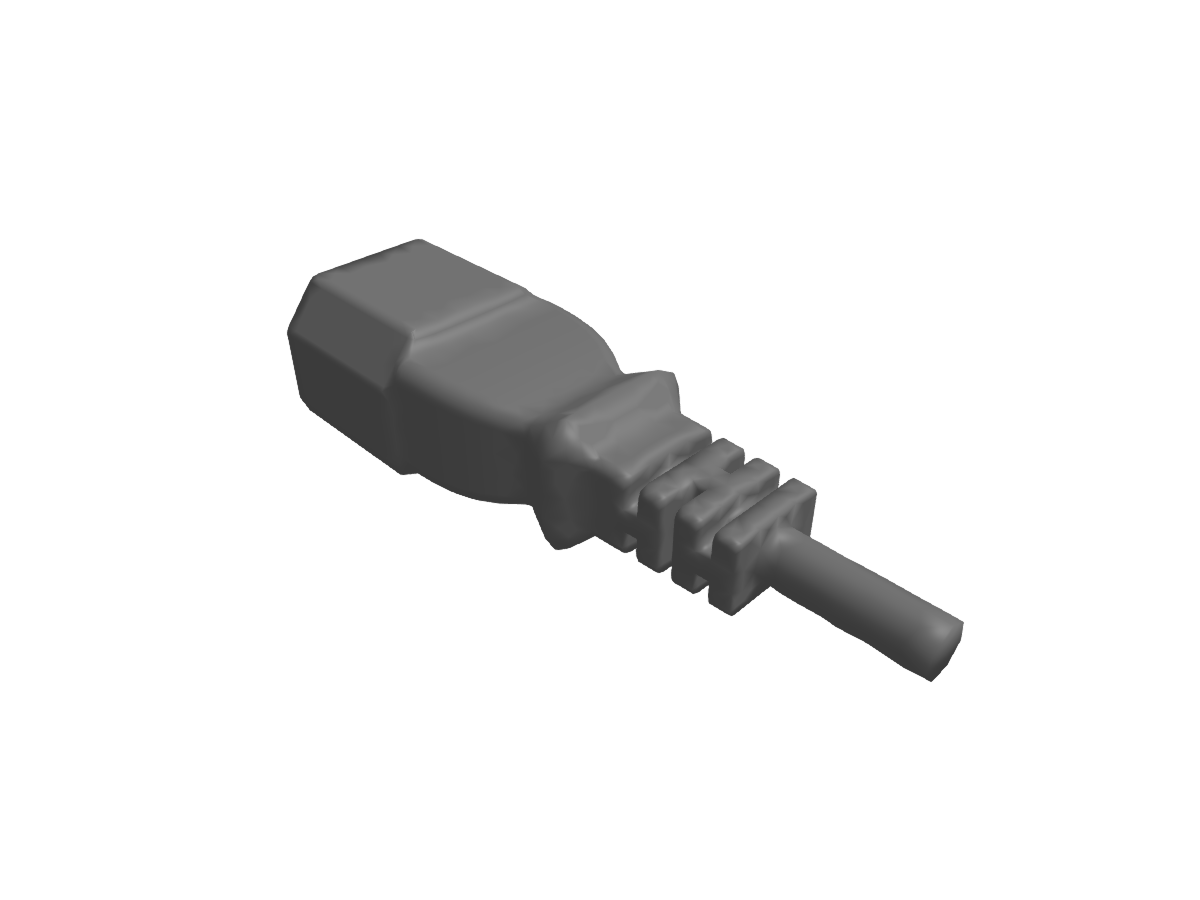}
        \\
        \includegraphics[width=0.1\linewidth, clip=true, trim=140pt 50pt 140pt 50pt]{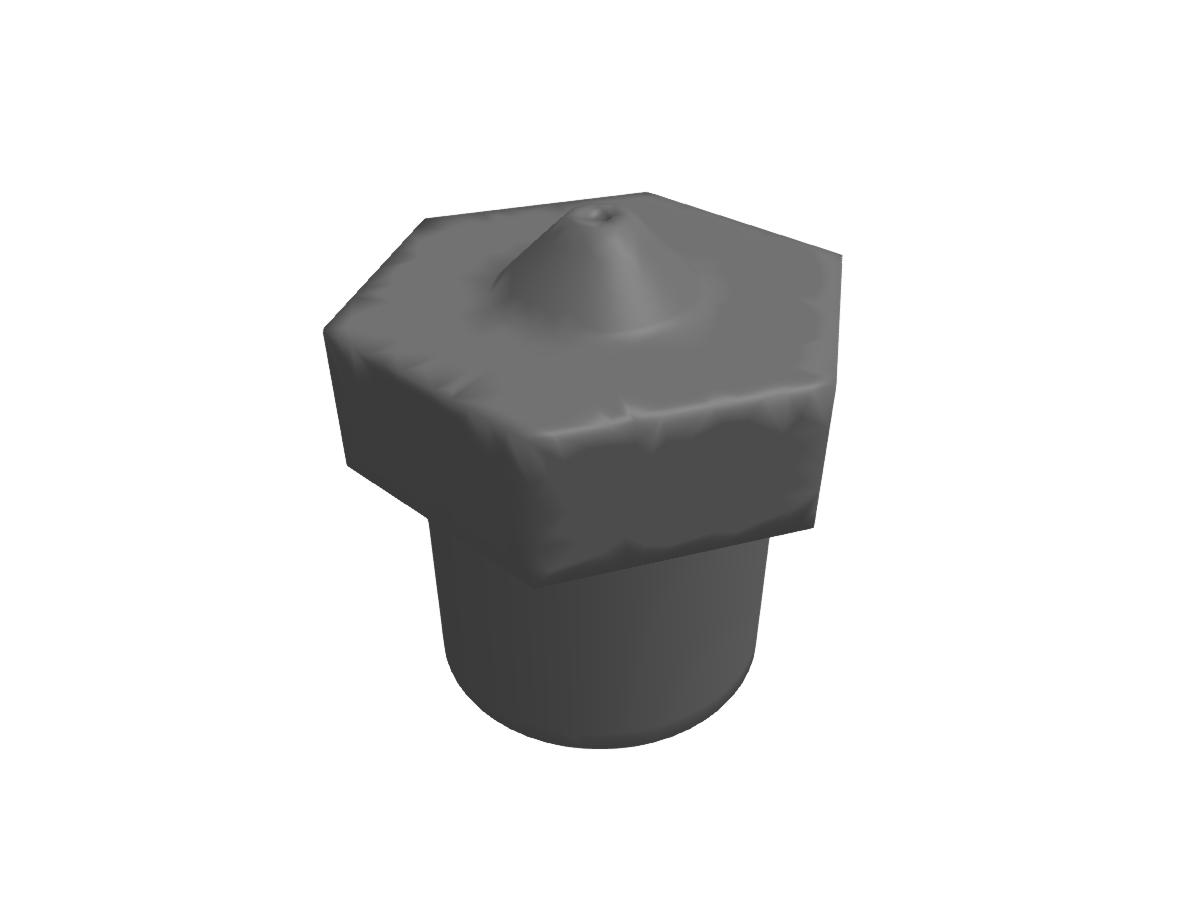}
        &\includegraphics[width=0.1\linewidth, clip=true, trim=140pt 50pt 140pt 50pt]{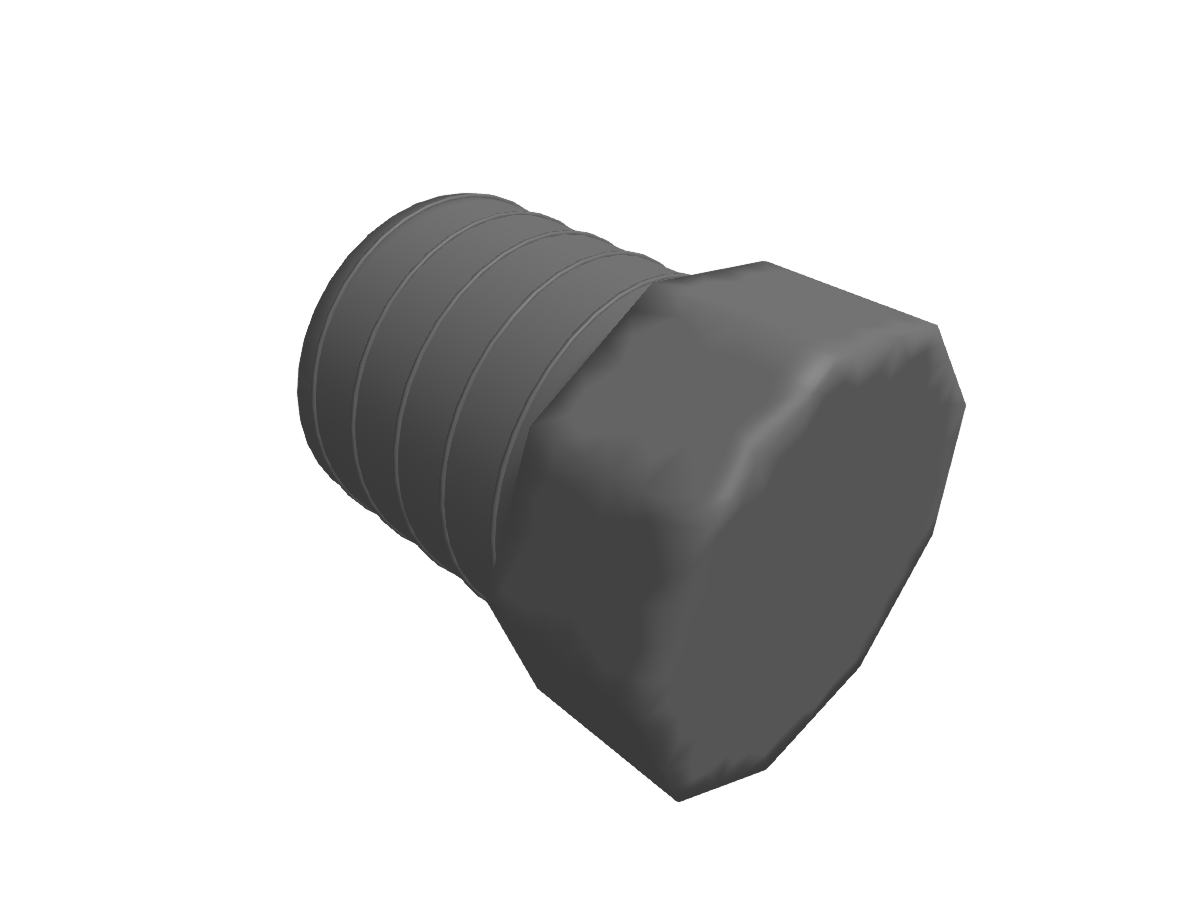}
        &\includegraphics[width=0.1\linewidth, clip=true, trim=140pt 50pt 140pt 50pt]{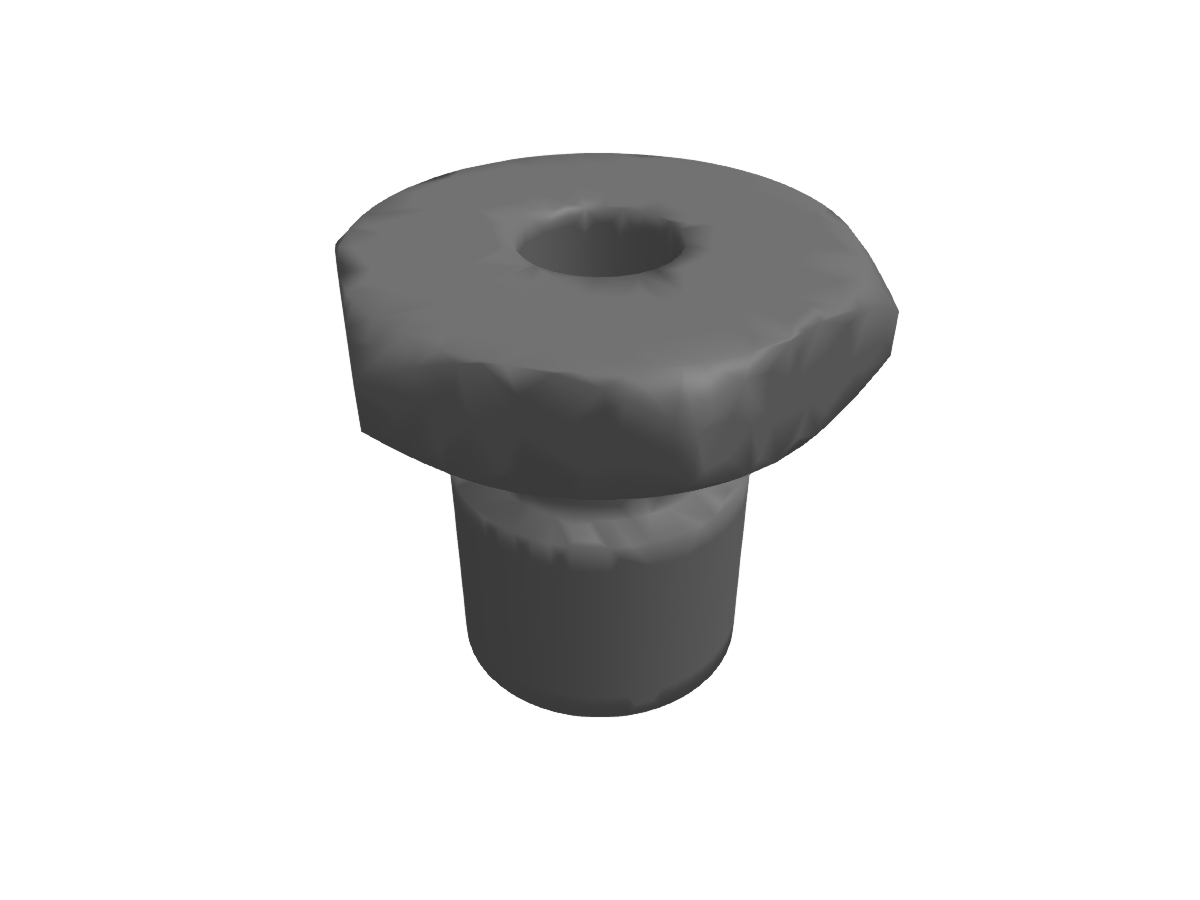}
        &\includegraphics[width=0.1\linewidth, clip=true, trim=140pt 50pt 140pt 50pt]{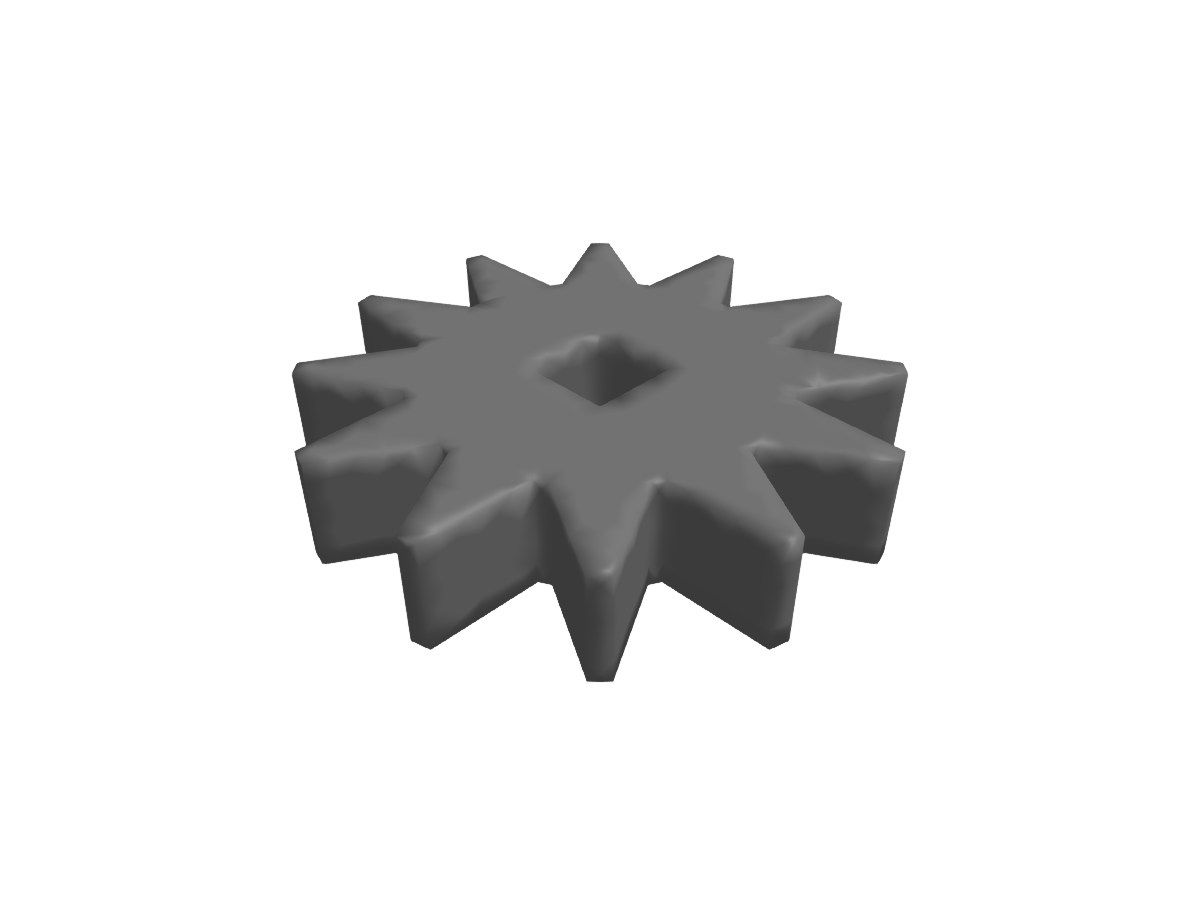}
        &\includegraphics[width=0.1\linewidth, clip=true, trim=140pt 50pt 140pt 50pt]{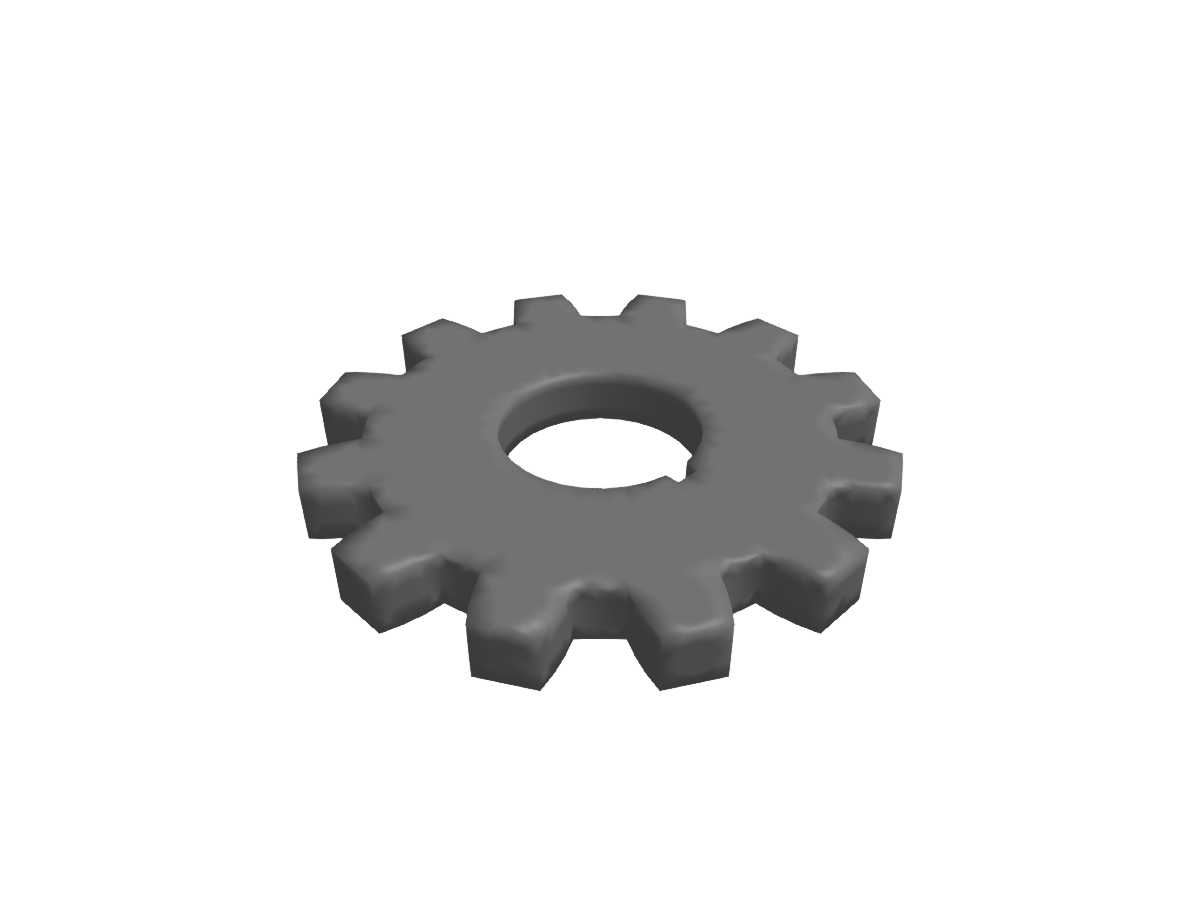}
        &\includegraphics[width=0.1\linewidth, clip=true, trim=140pt 50pt 140pt 50pt]{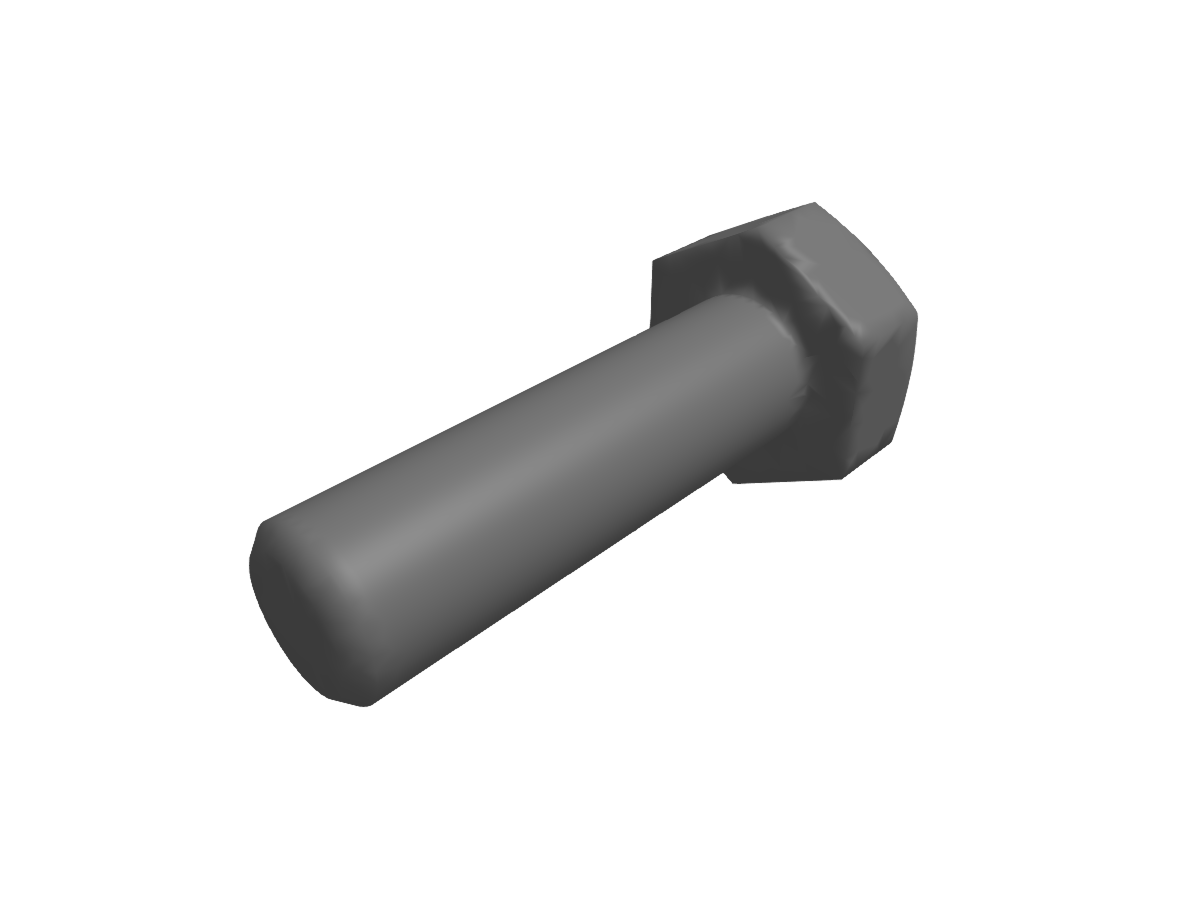}
        &\includegraphics[width=0.1\linewidth, clip=true, trim=140pt 50pt 140pt 50pt]{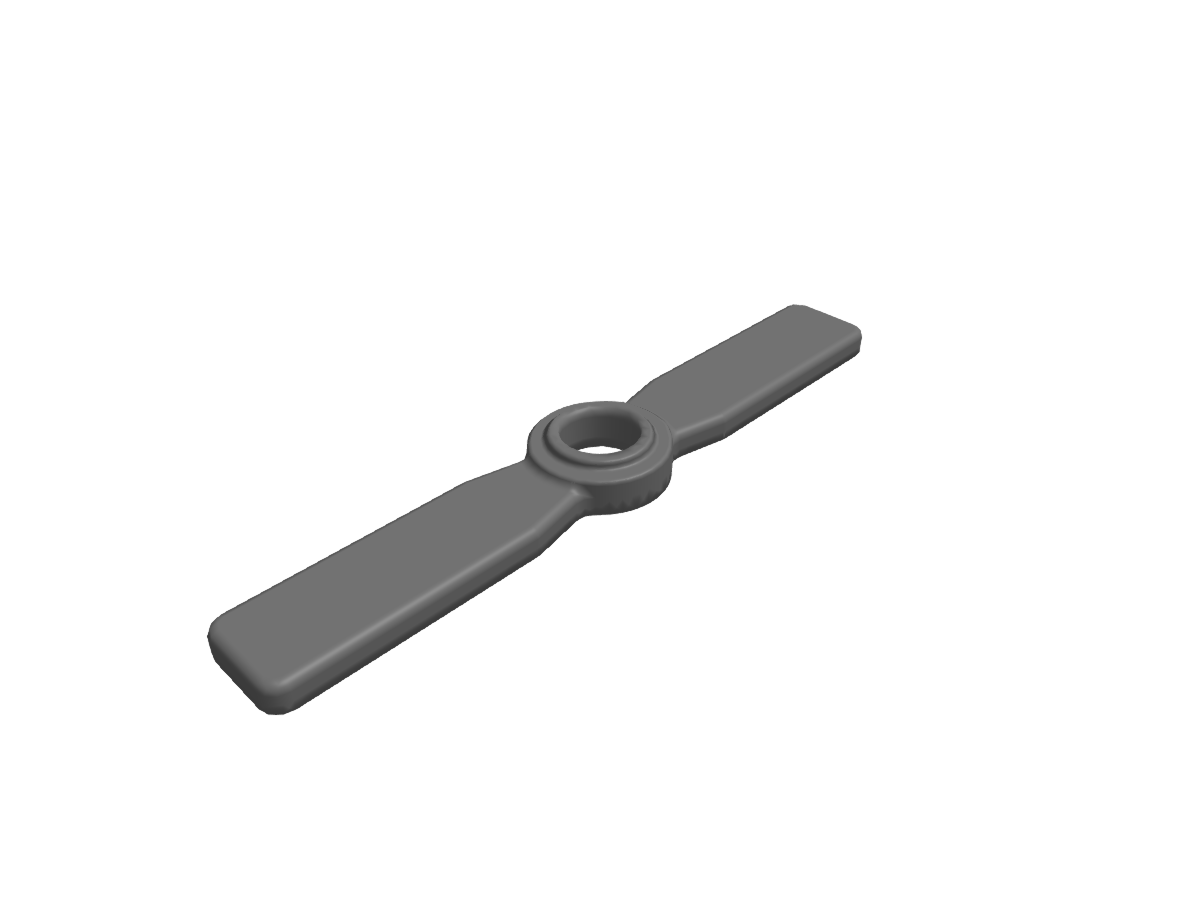}
        &\includegraphics[width=0.1\linewidth, clip=true, trim=140pt 50pt 140pt 50pt]{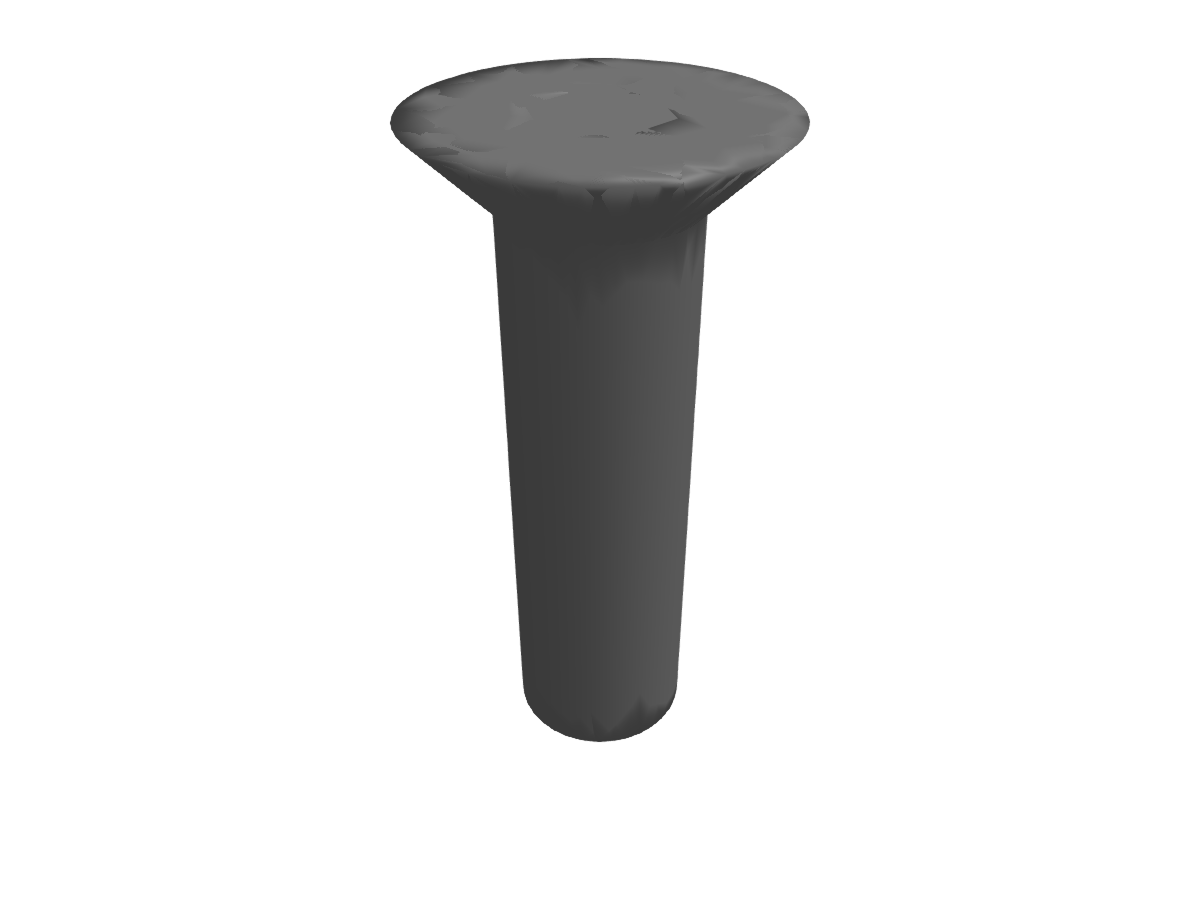}
        &\includegraphics[width=0.1\linewidth, clip=true, trim=140pt 50pt 140pt 50pt]{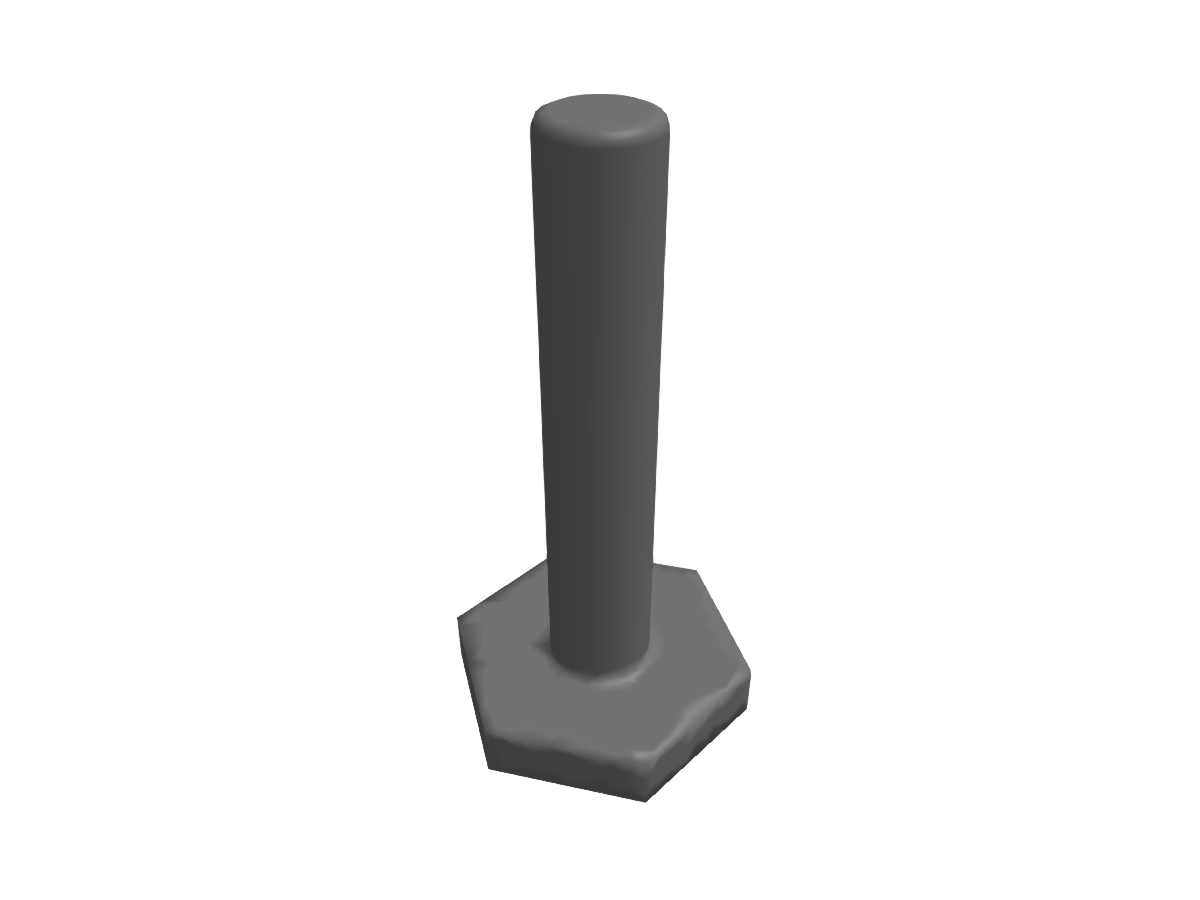}
        &\includegraphics[width=0.1\linewidth, clip=true, trim=140pt 50pt 140pt 50pt]{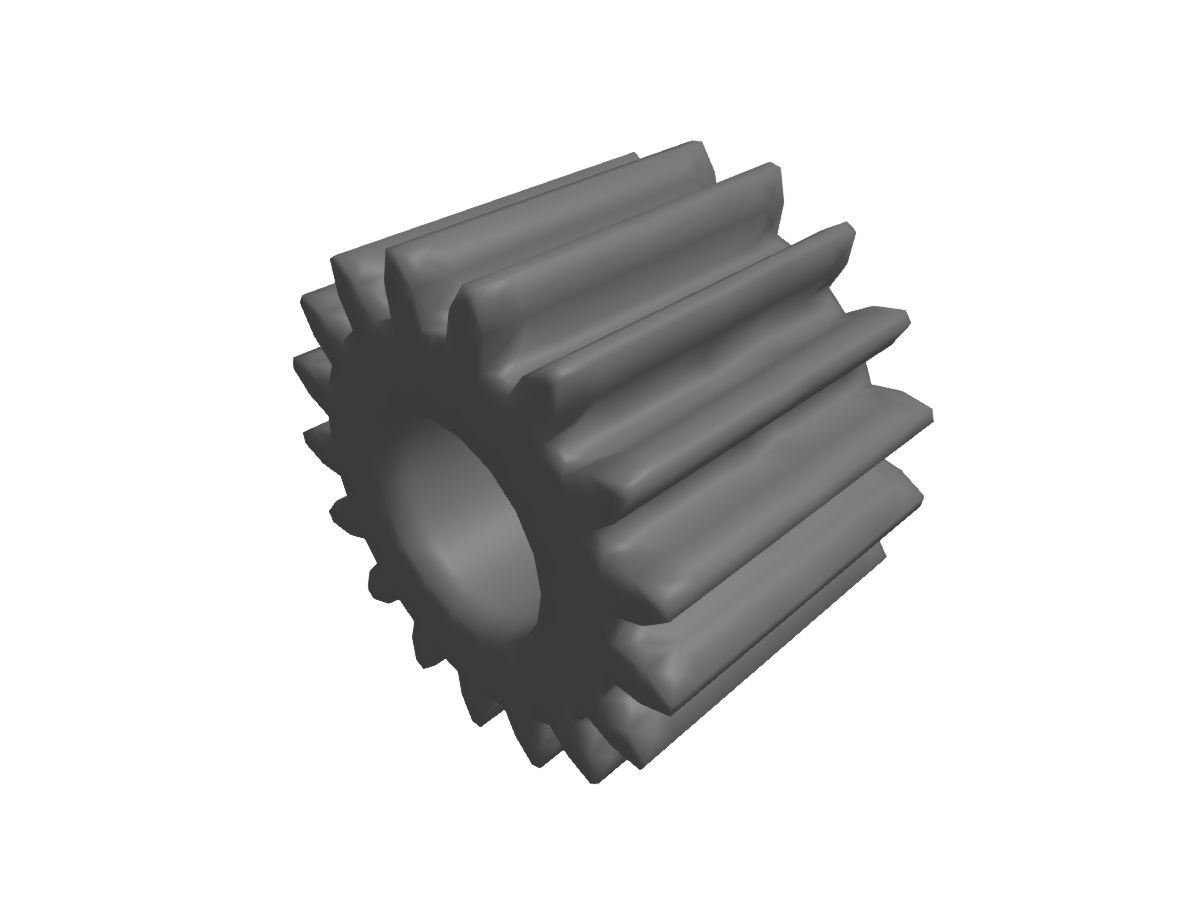}
        \\
        \includegraphics[width=0.1\linewidth, clip=true, trim=140pt 50pt 140pt 50pt]{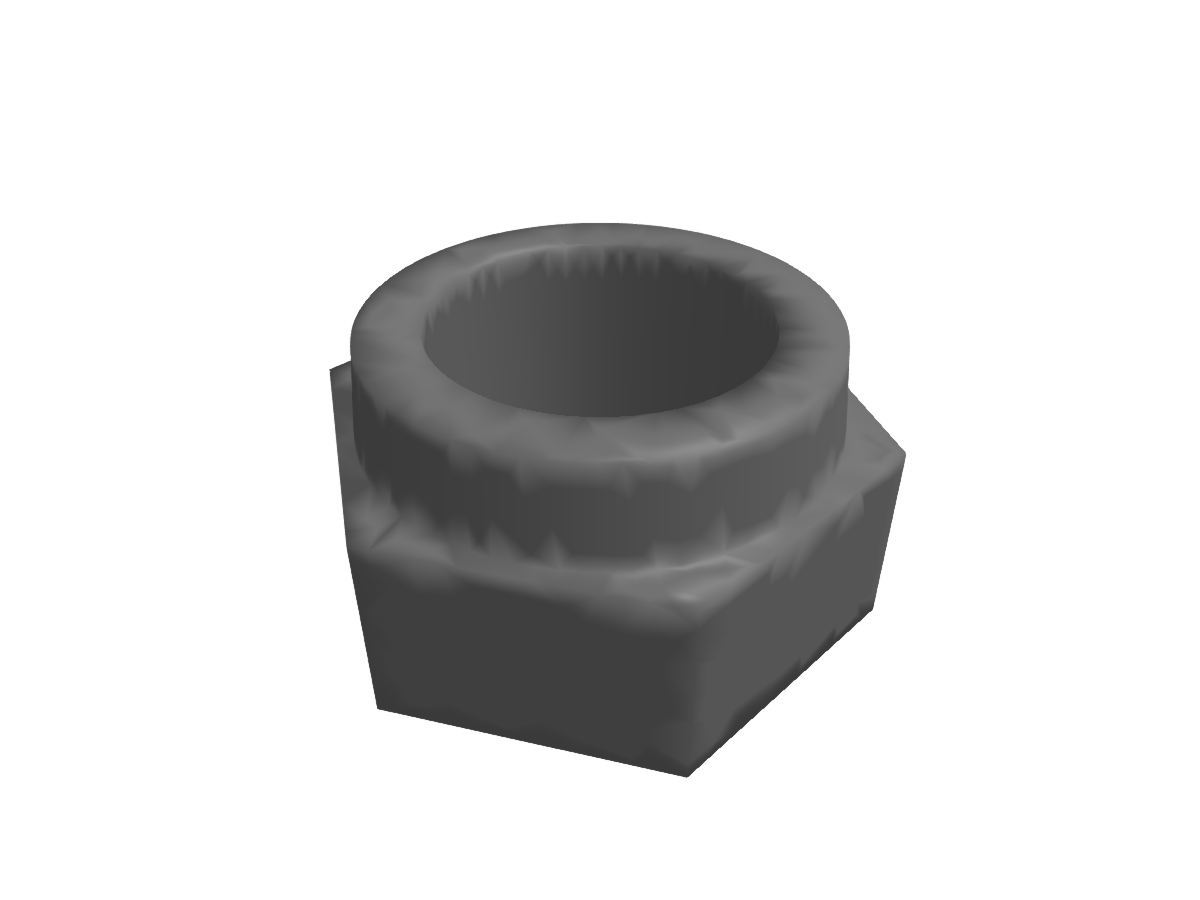}
        &\includegraphics[width=0.1\linewidth, clip=true, trim=140pt 50pt 140pt 50pt]{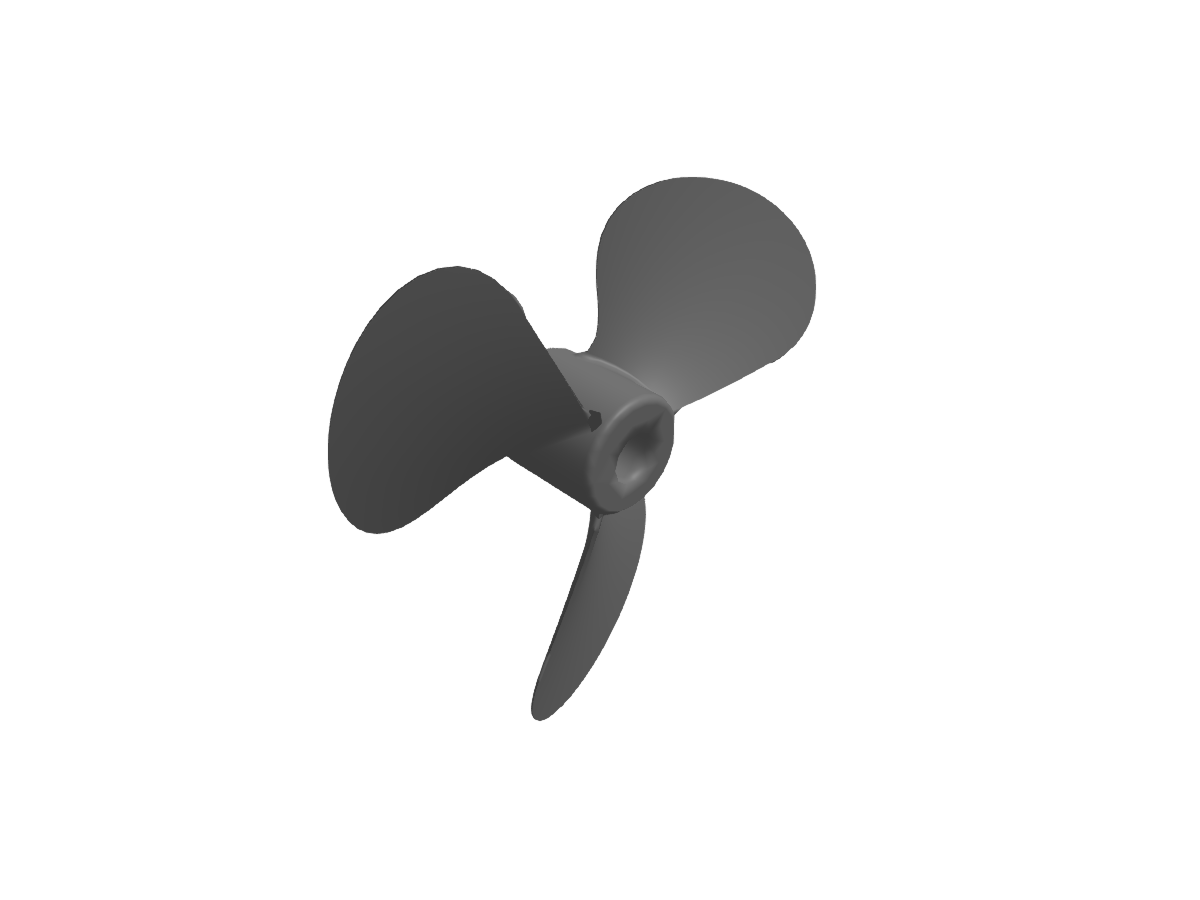}
        &\includegraphics[width=0.1\linewidth, clip=true, trim=140pt 50pt 140pt 50pt]{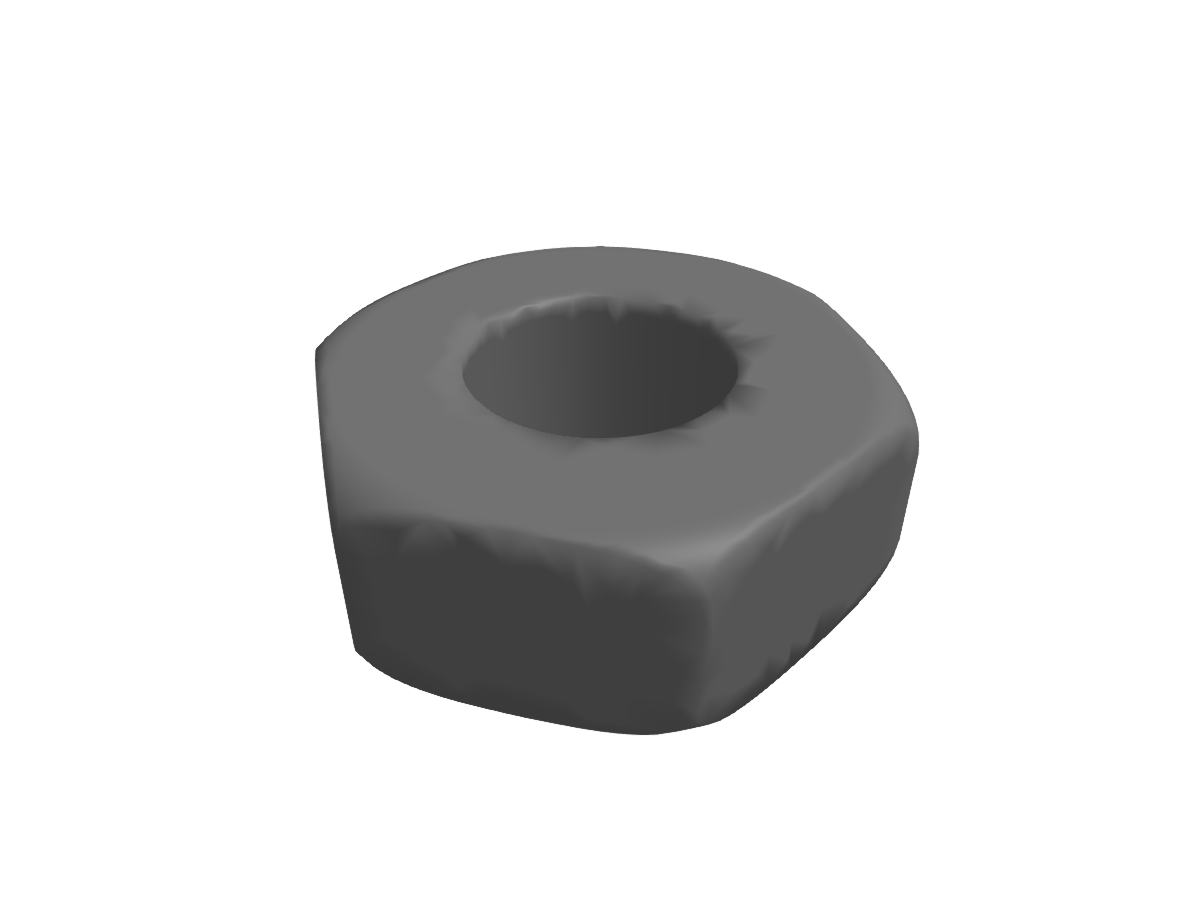}
        &\includegraphics[width=0.1\linewidth, clip=true, trim=140pt 50pt 140pt 50pt]{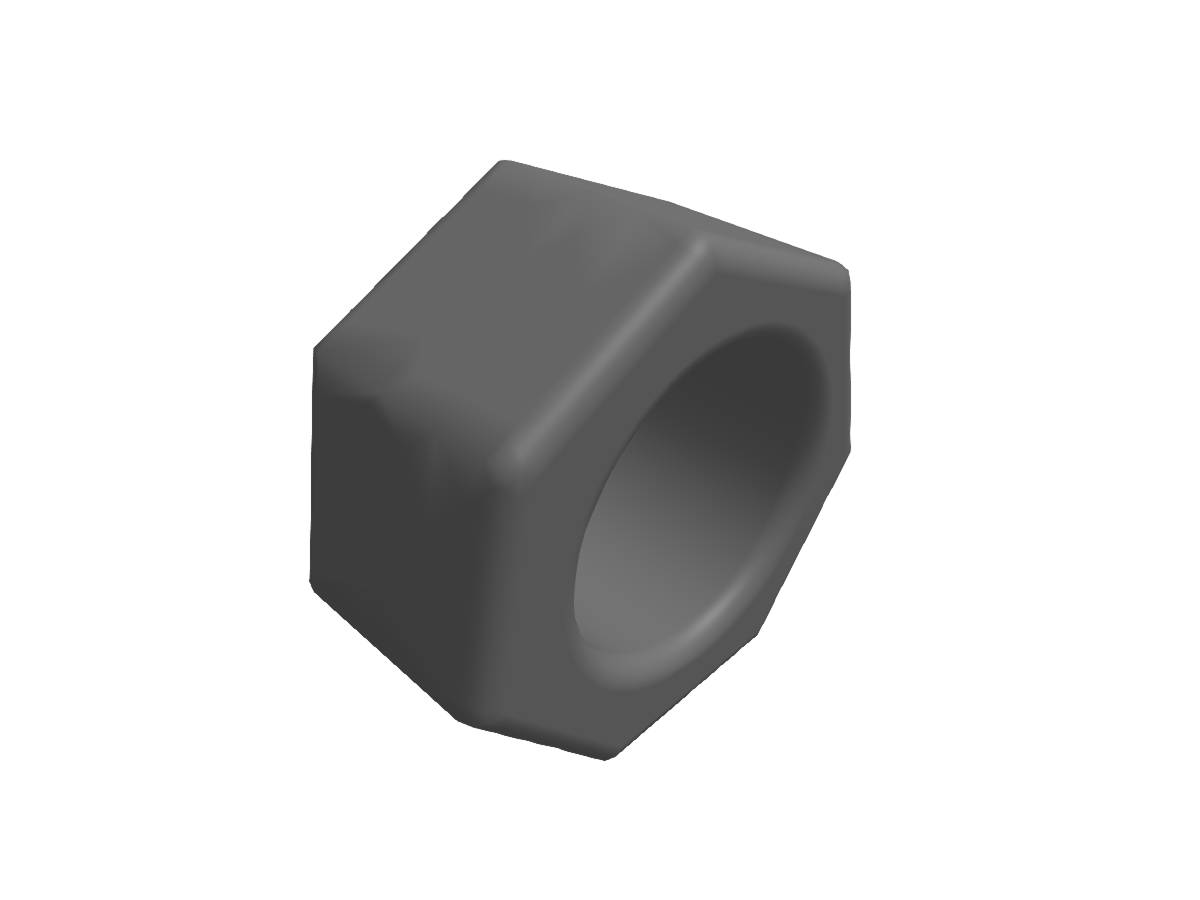}
        &\includegraphics[width=0.1\linewidth, clip=true, trim=140pt 50pt 140pt 50pt]{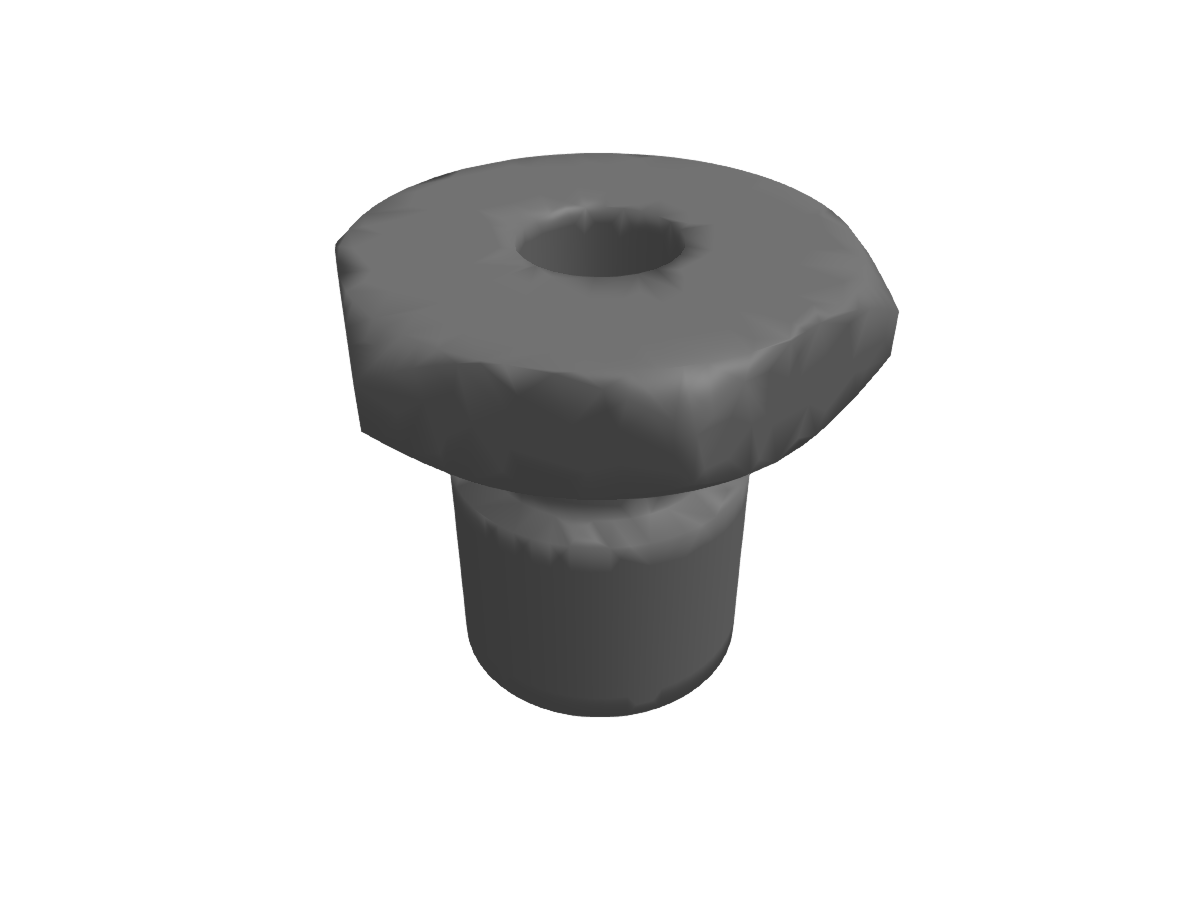}
        &\includegraphics[width=0.1\linewidth, clip=true, trim=140pt 50pt 140pt 50pt]{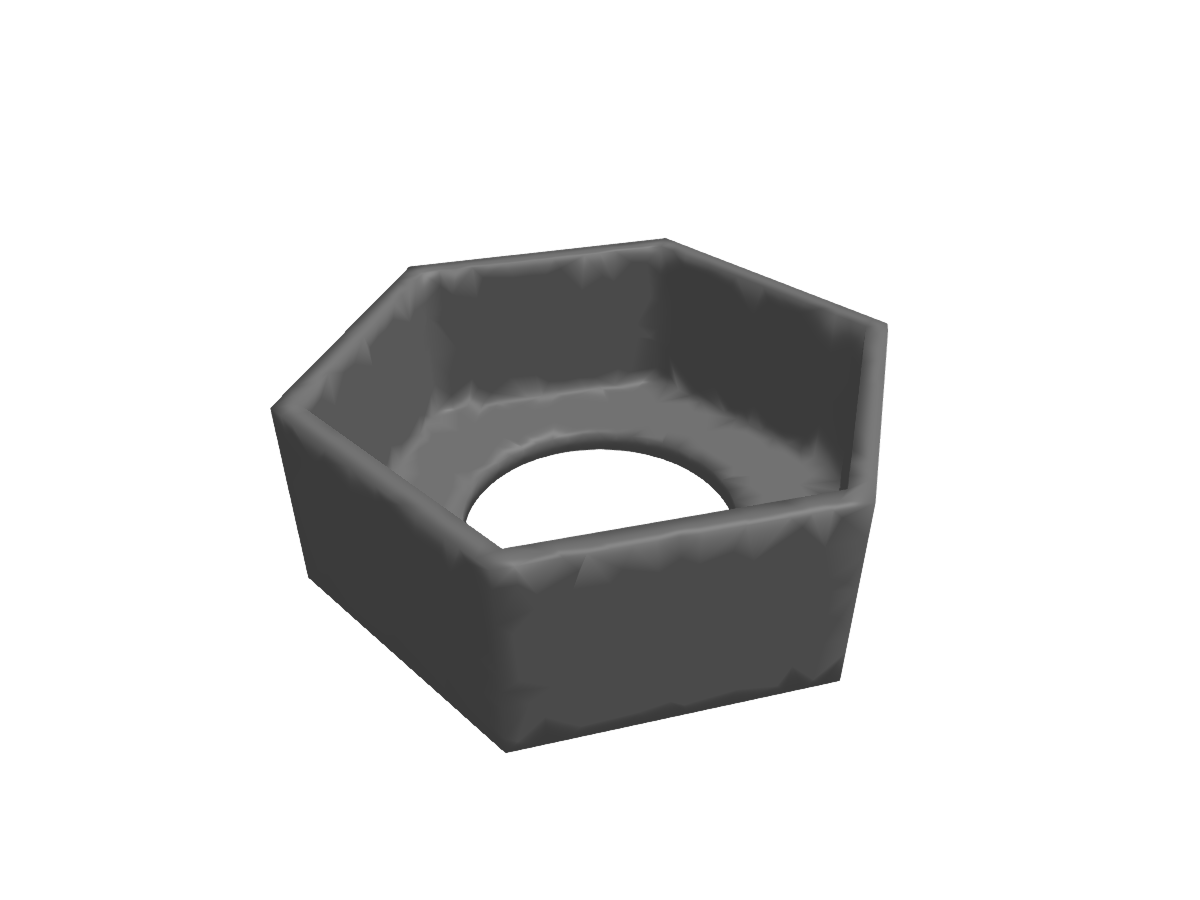}
        &\includegraphics[width=0.1\linewidth, clip=true, trim=140pt 50pt 140pt 50pt]{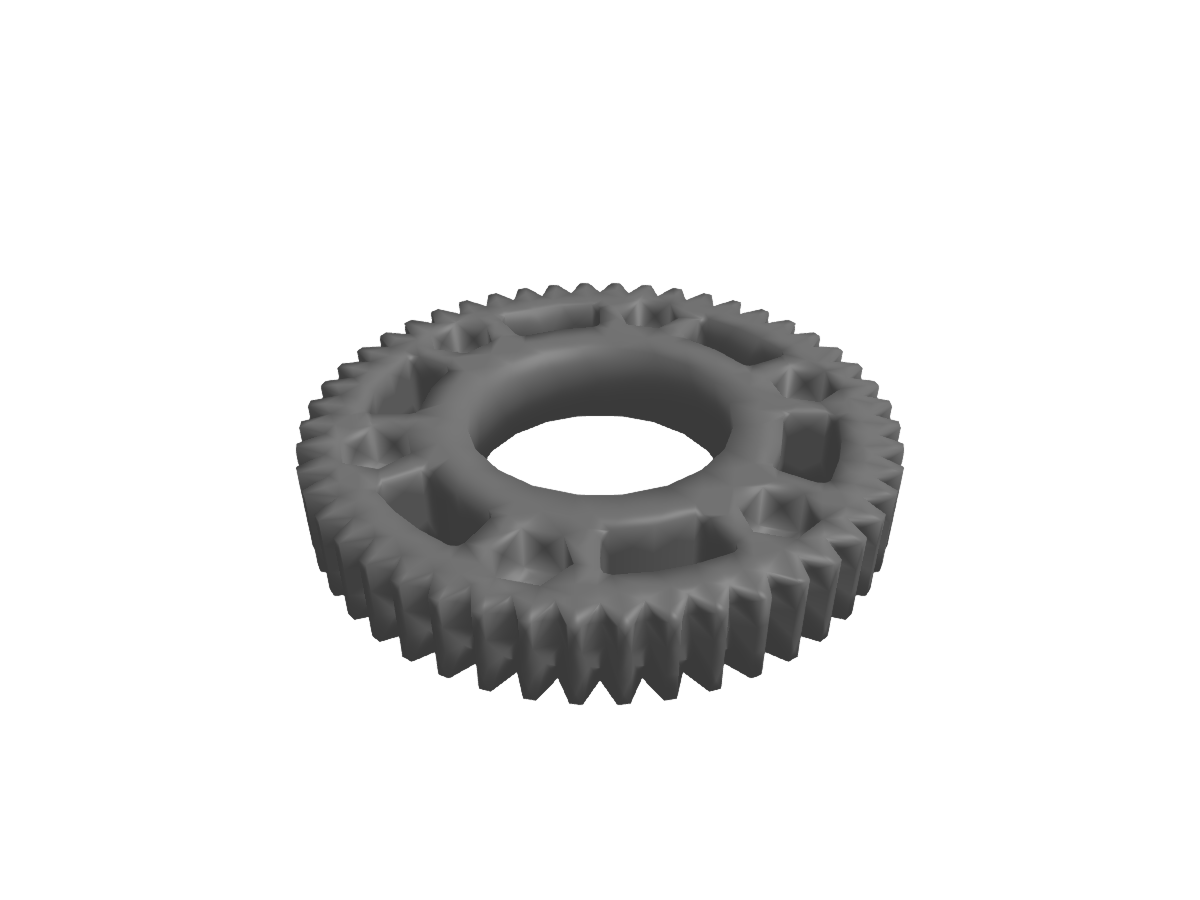}
        &\includegraphics[width=0.1\linewidth, clip=true, trim=140pt 50pt 140pt 50pt]{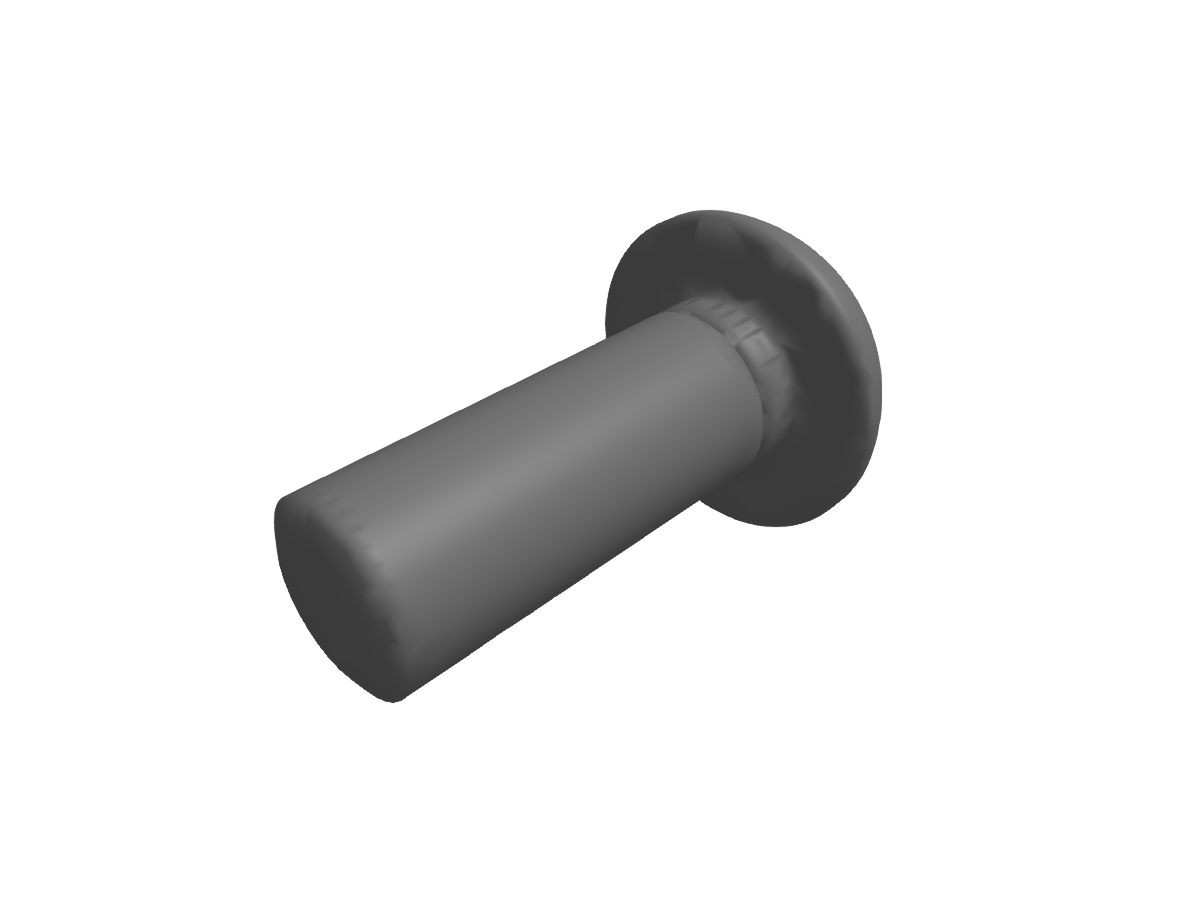}
        &\includegraphics[width=0.1\linewidth, clip=true, trim=140pt 50pt 140pt 50pt]{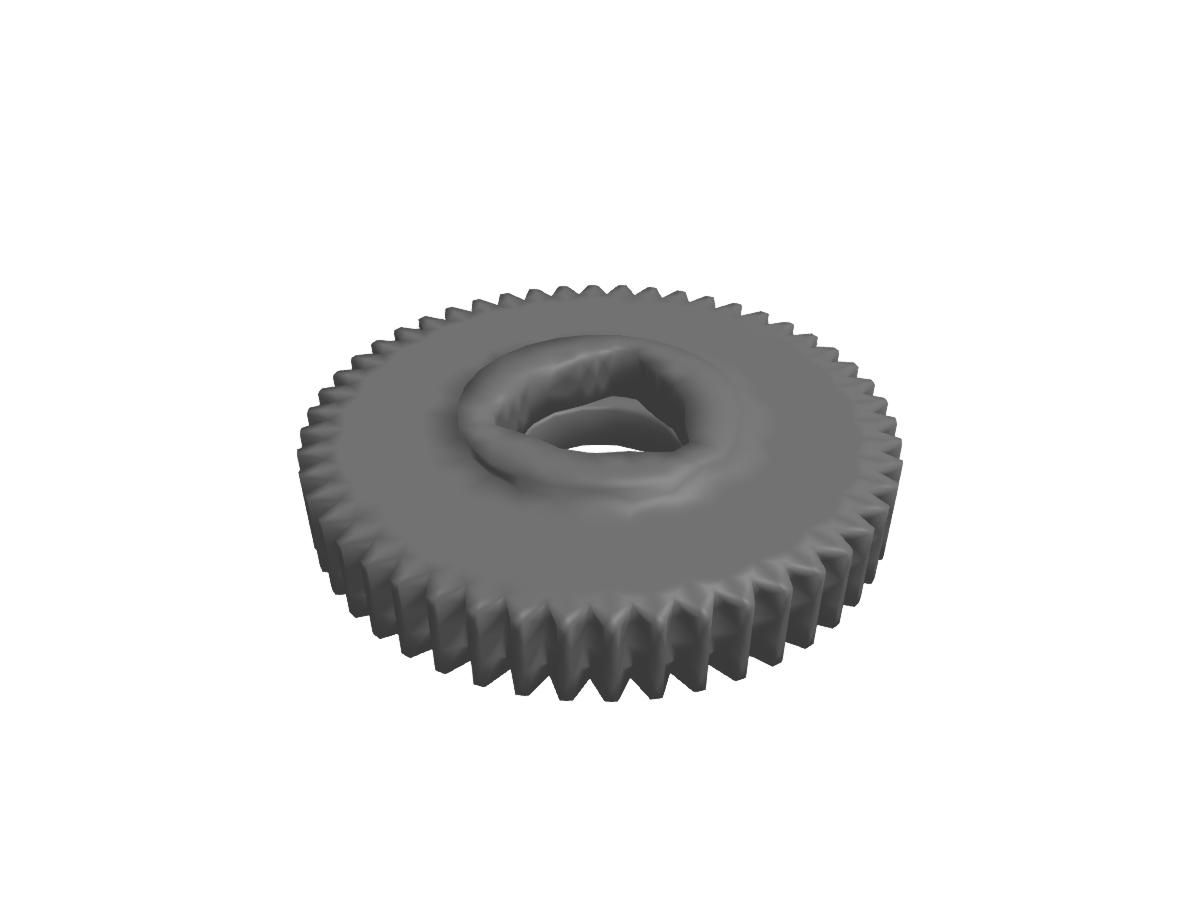}
        &\includegraphics[width=0.1\linewidth, clip=true, trim=140pt 50pt 140pt 50pt]{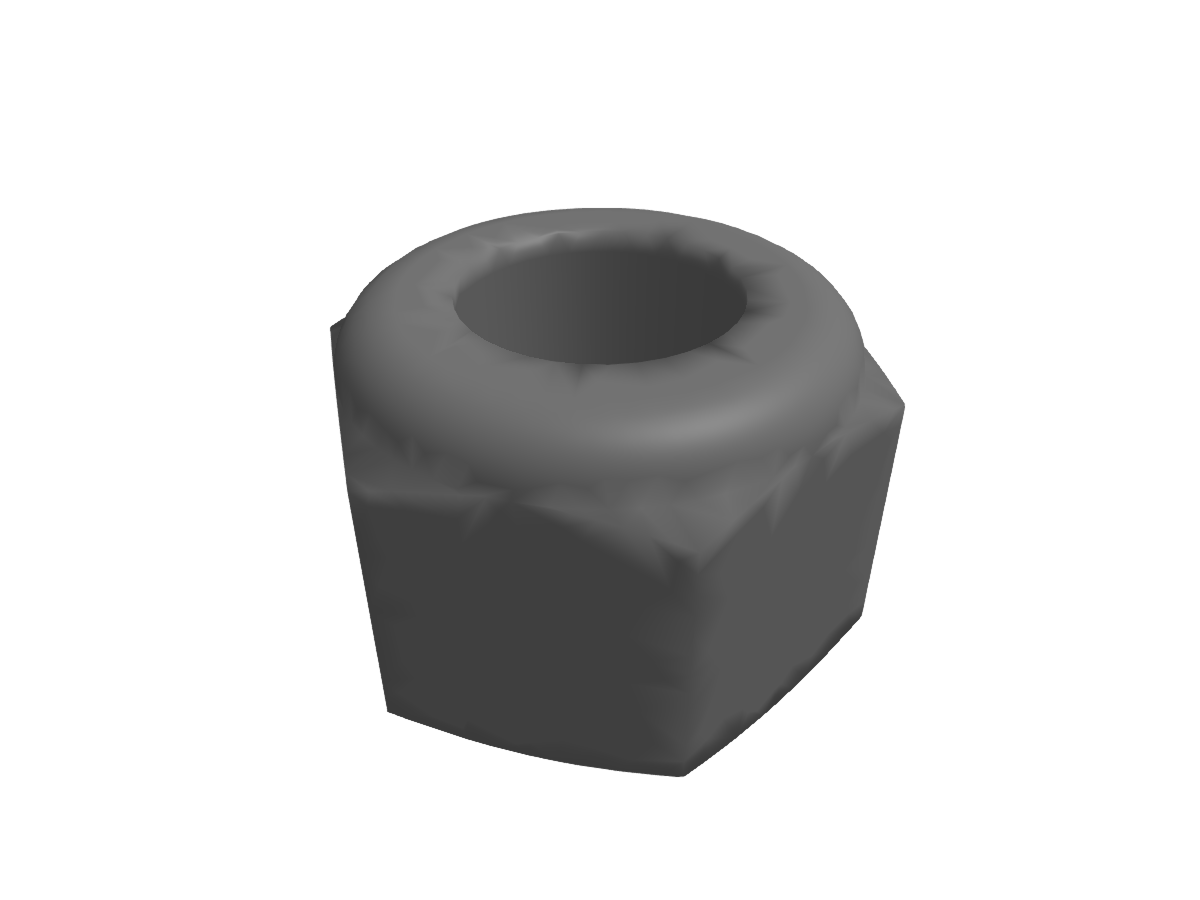}
        \\
        \includegraphics[width=0.1\linewidth, clip=true, trim=140pt 50pt 140pt 50pt]{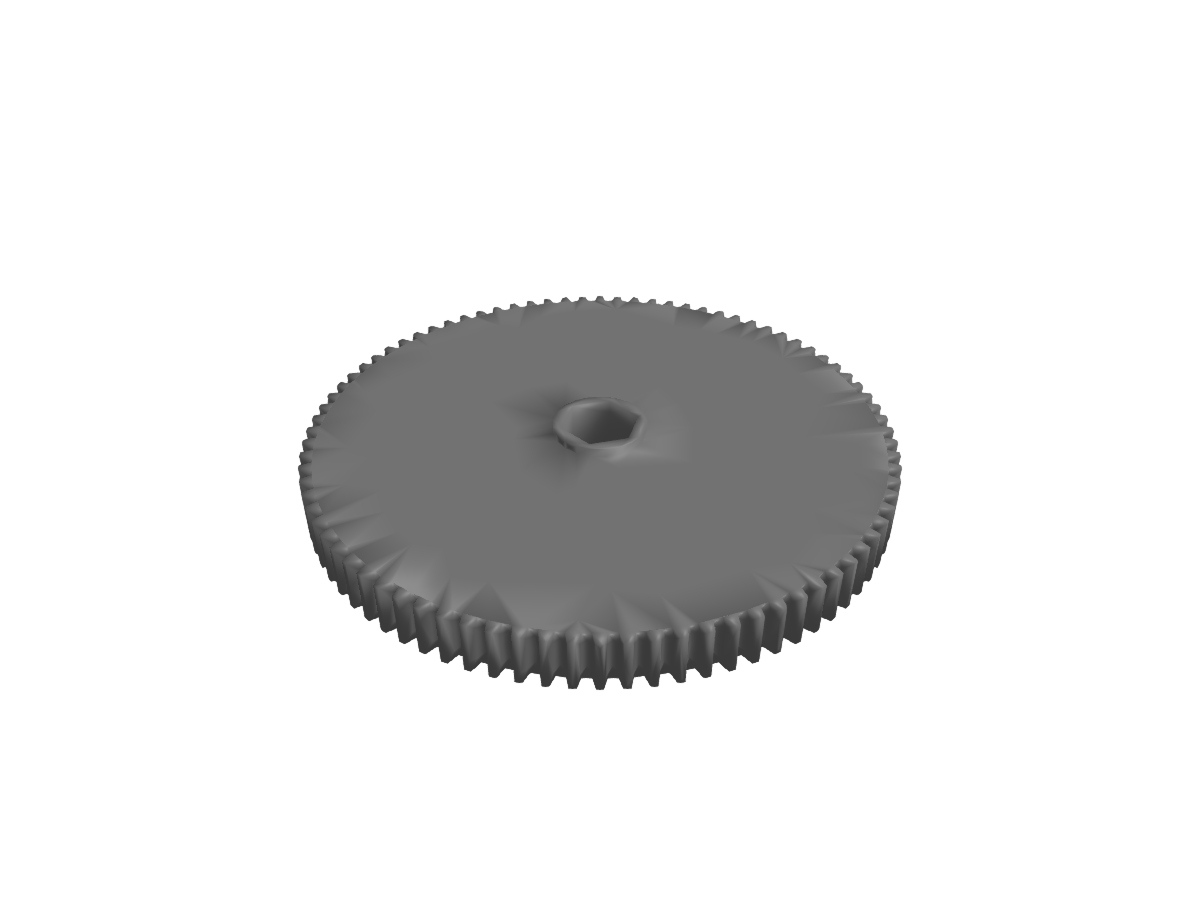}
        &\includegraphics[width=0.1\linewidth, clip=true, trim=140pt 50pt 140pt 50pt]{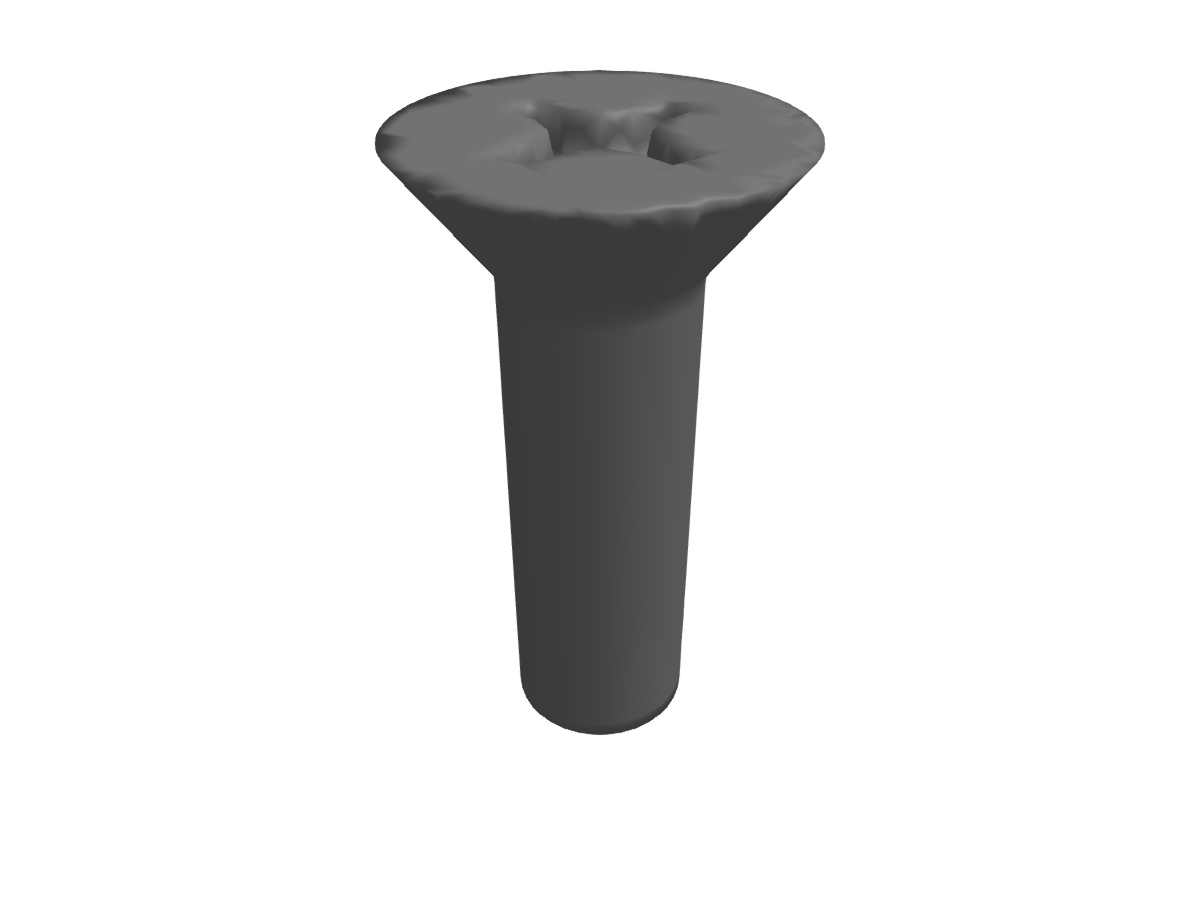}
        &\includegraphics[width=0.1\linewidth, clip=true, trim=140pt 50pt 140pt 50pt]{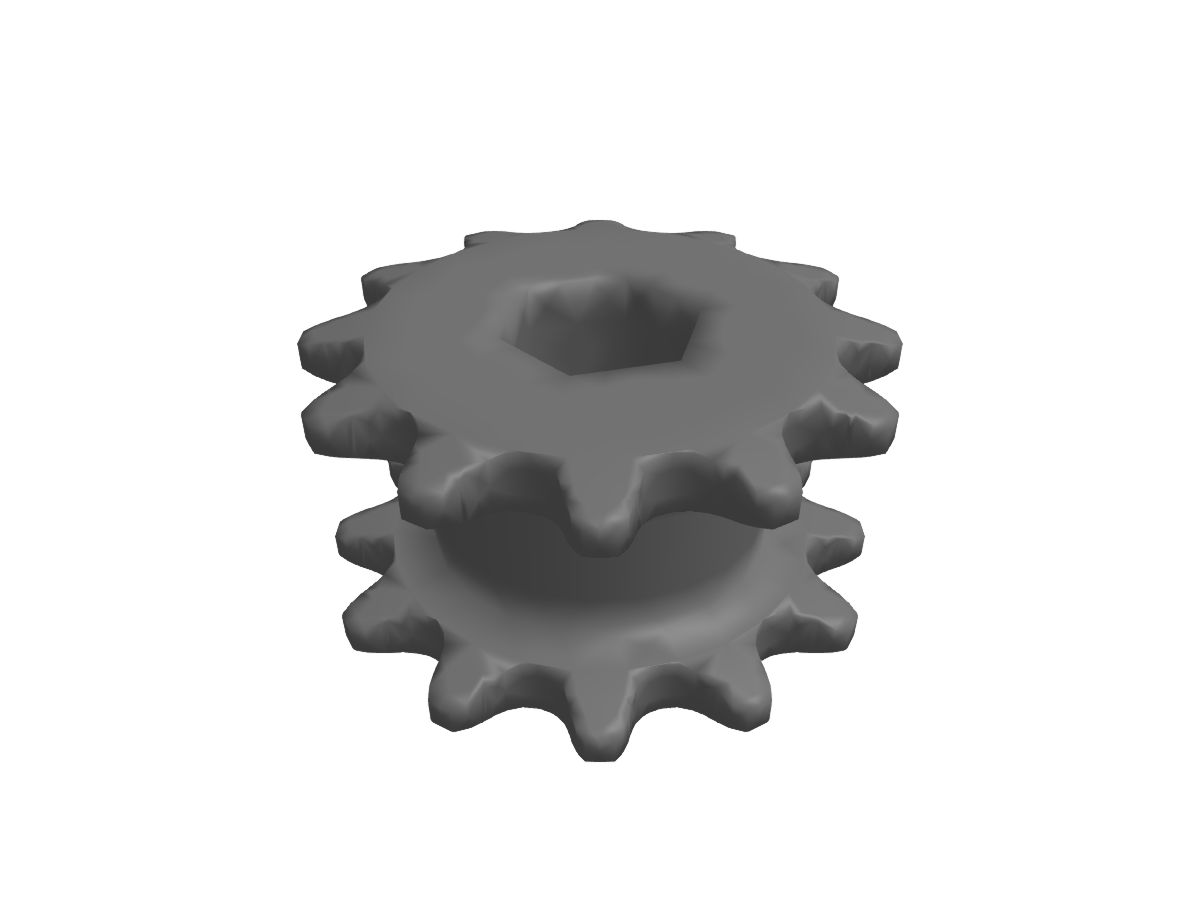}
        &\includegraphics[width=0.1\linewidth, clip=true, trim=140pt 50pt 140pt 50pt]{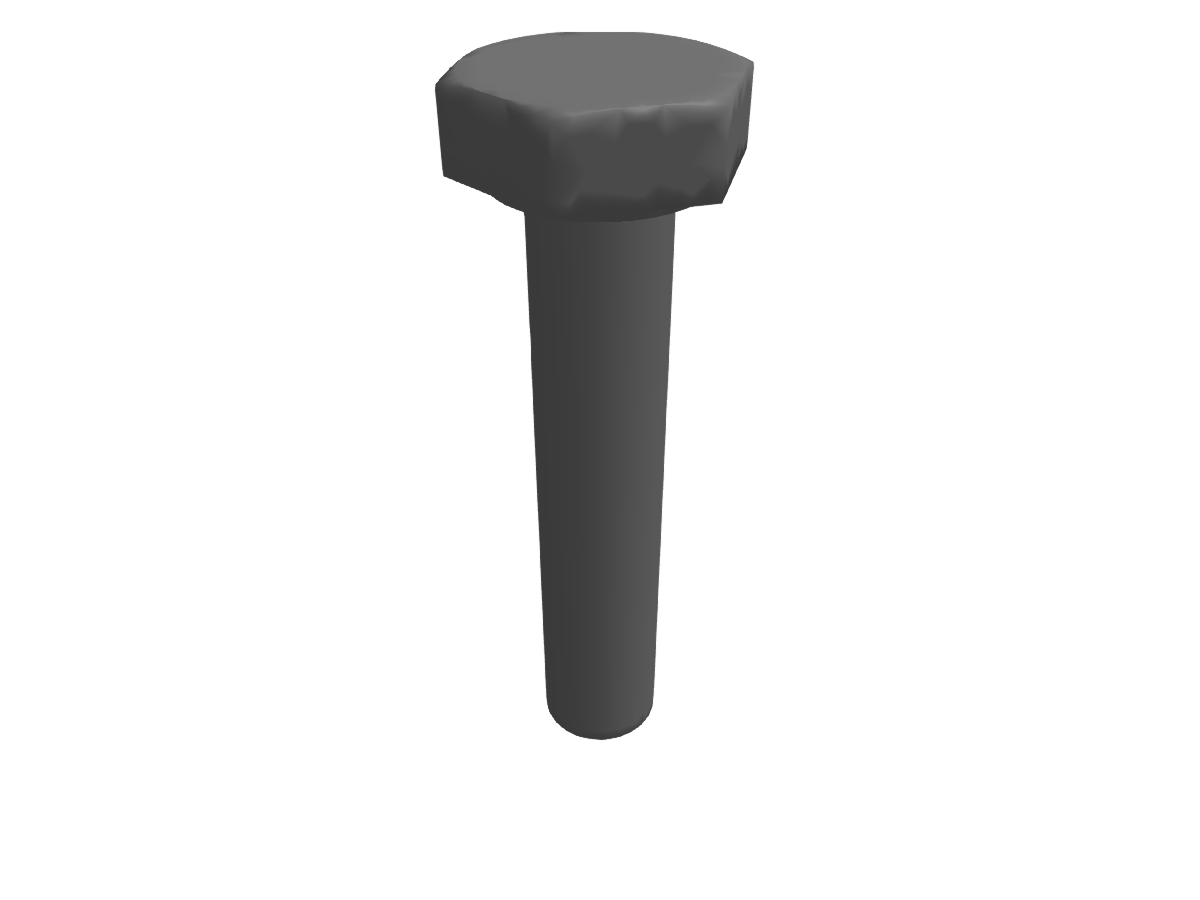}
        &\includegraphics[width=0.1\linewidth, clip=true, trim=140pt 50pt 140pt 50pt]{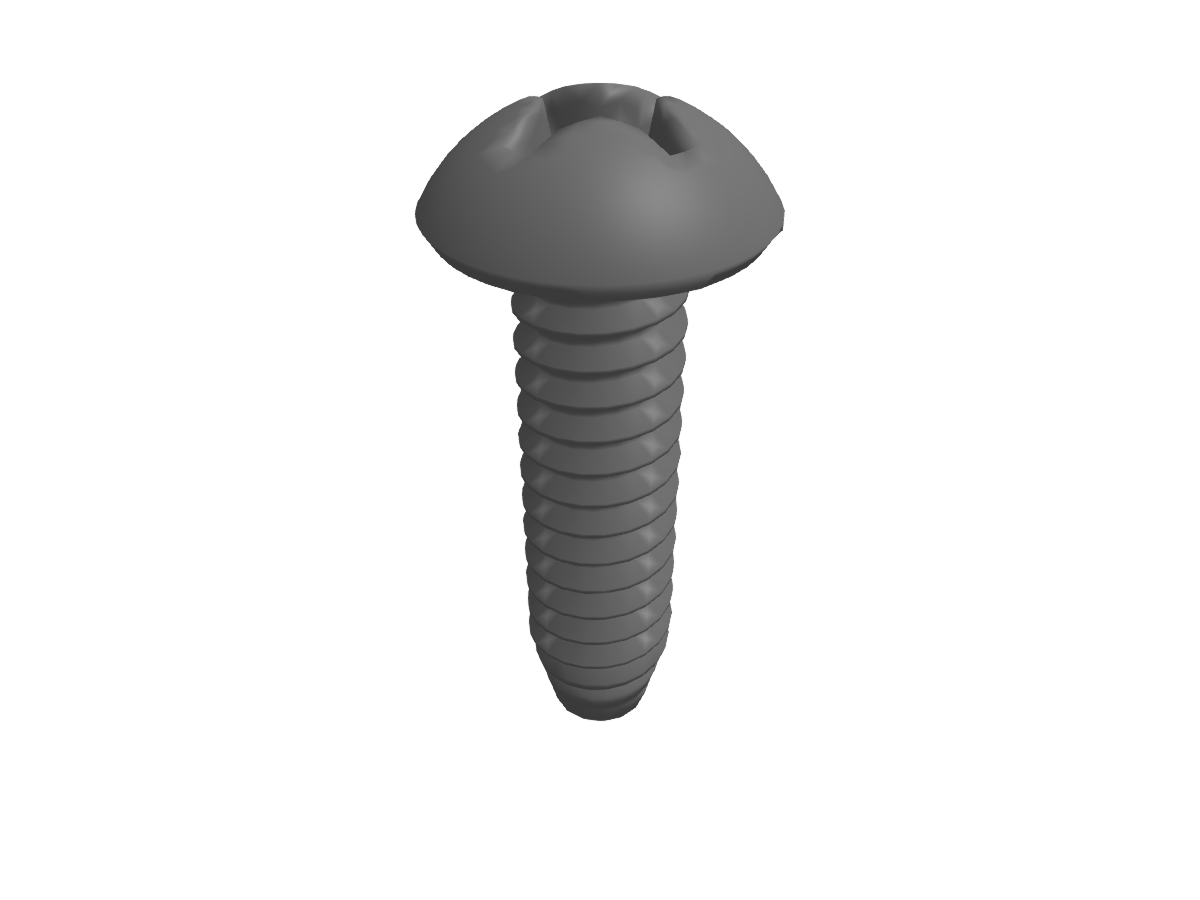}
        &\includegraphics[width=0.1\linewidth, clip=true, trim=140pt 50pt 140pt 50pt]{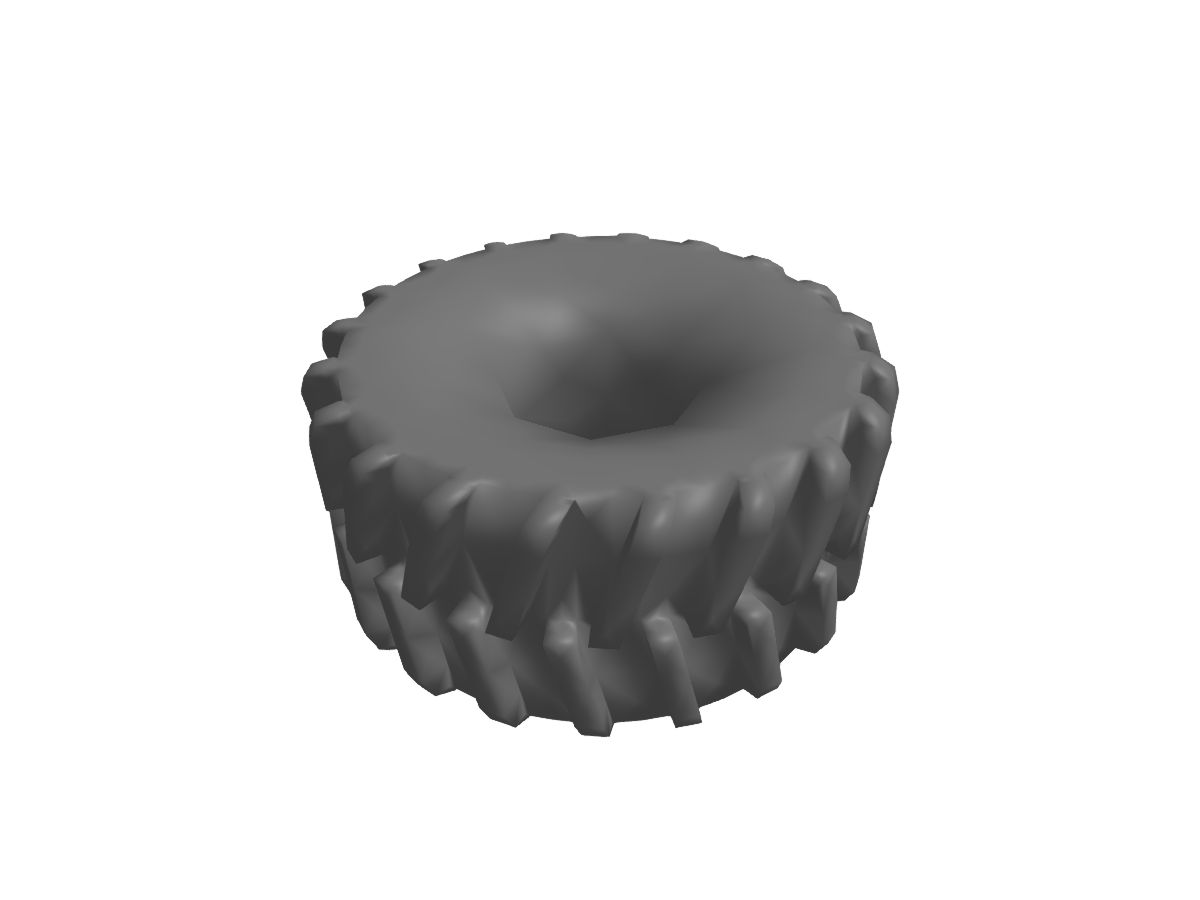}
        &\includegraphics[width=0.1\linewidth, clip=true, trim=140pt 50pt 140pt 50pt]{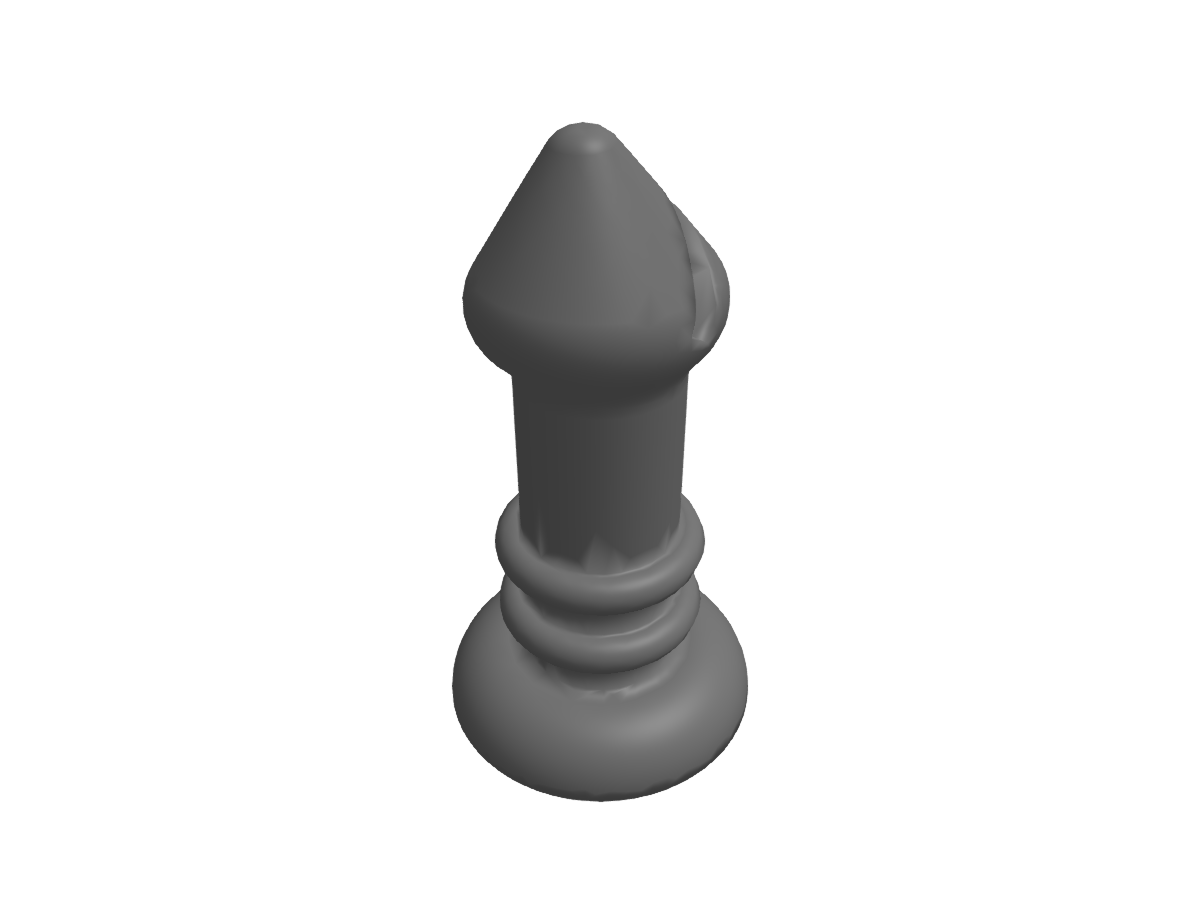}
        &\includegraphics[width=0.1\linewidth, clip=true, trim=140pt 50pt 140pt 50pt]{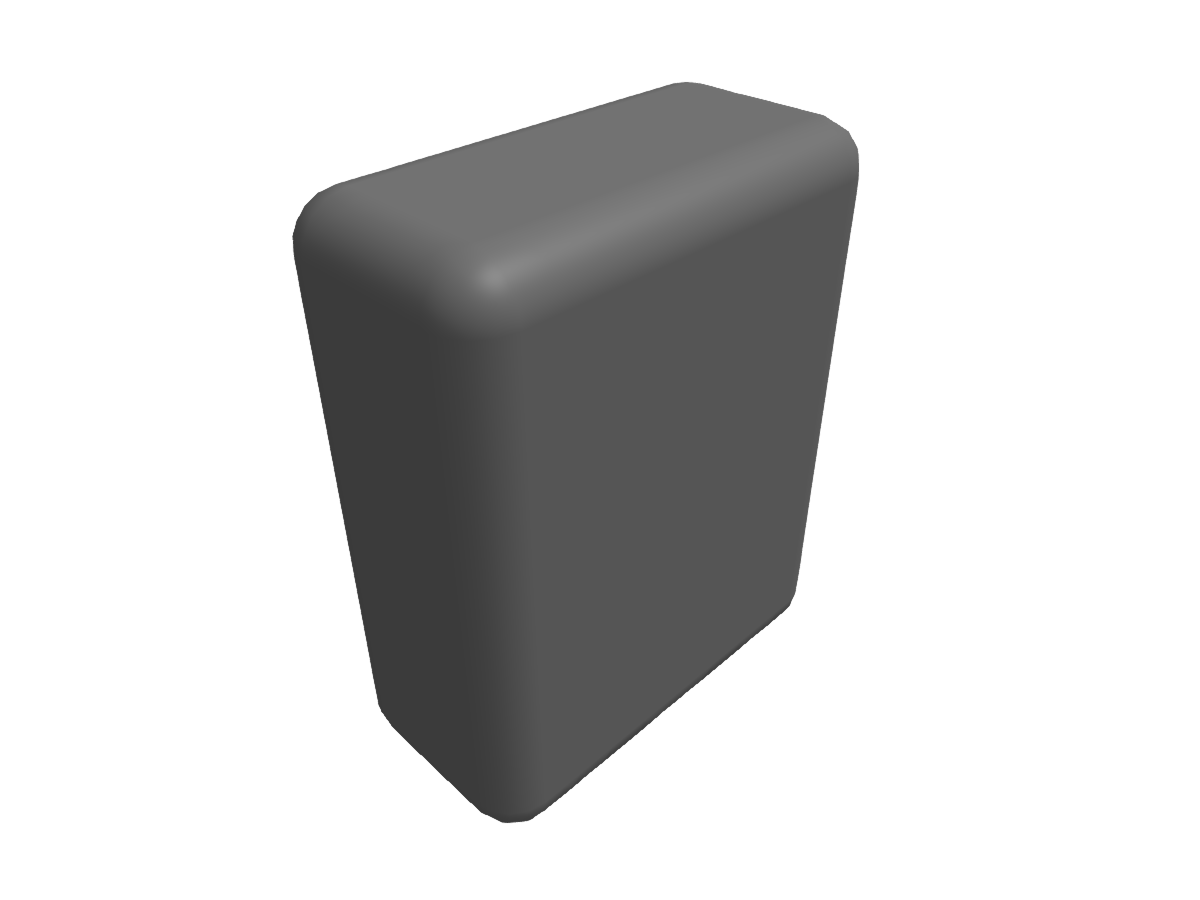}
        &\includegraphics[width=0.1\linewidth, clip=true, trim=140pt 50pt 140pt 50pt]{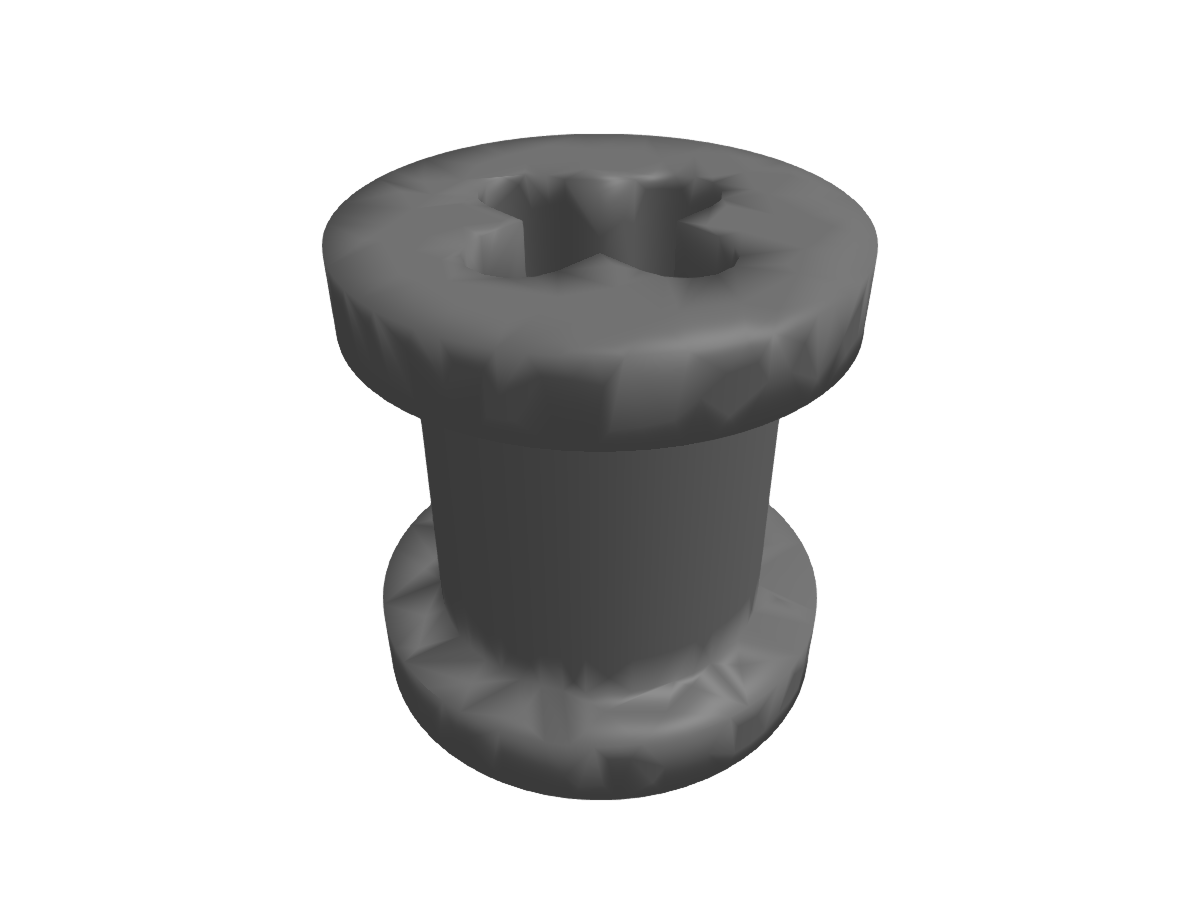}
        &\includegraphics[width=0.1\linewidth, clip=true, trim=140pt 50pt 140pt 50pt]{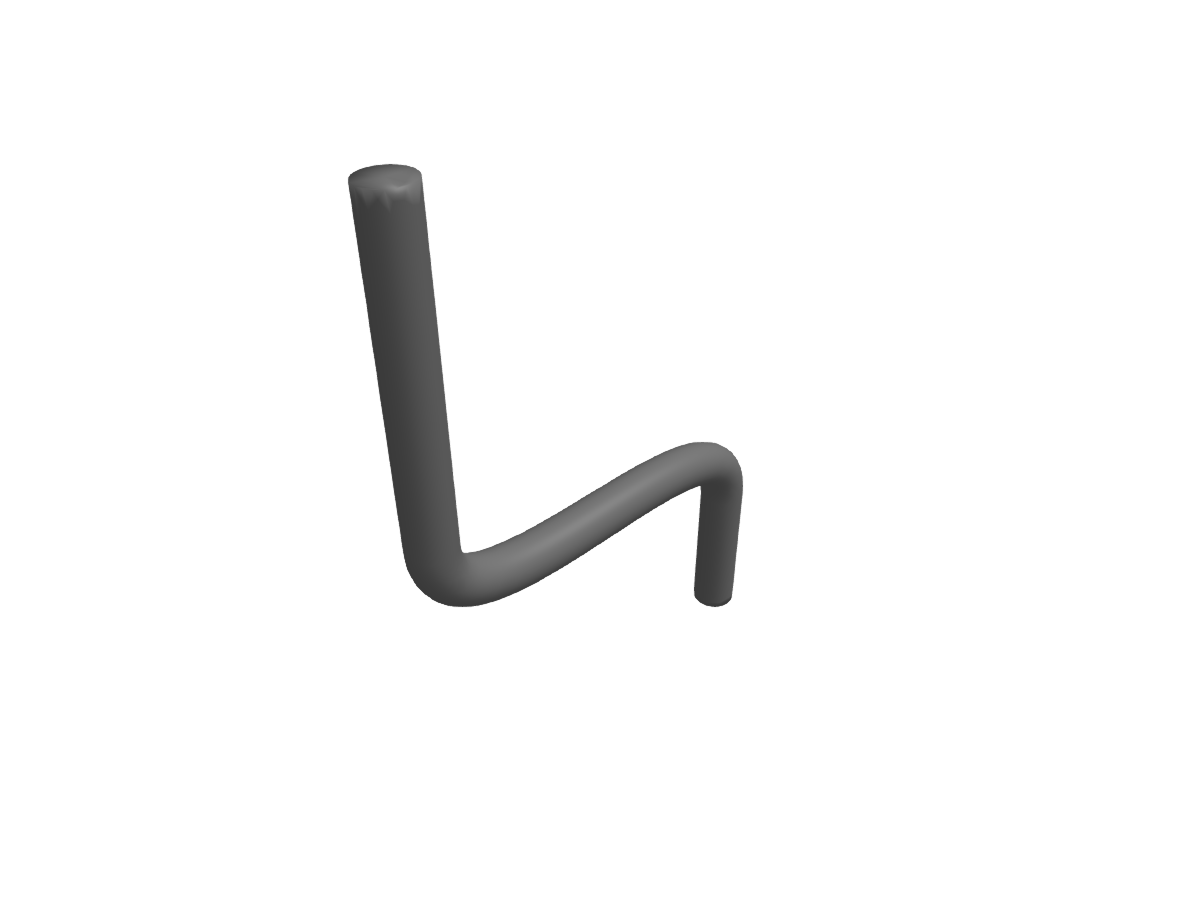}
        \end{tabular}}
        \caption{Visualization of the templates
        of ABC dataset.}
        \label{fig:abc_dataset_3d}
    \end{figure}

    \item \textbf{FAUST dataset}.
    The original `MPI FAUST Dataset' \cite{Bogo2014}
    comprises real-world human scans
    of 10 persons in 30 poses.
    We use a subset of 10 poses
    from the registration data.
    Each sample is represented by
    a mesh with 6890 vertices,
    which we equip with the uniform distribution 
    to obtain an empirical measure on $\R^3$.
    For illustration,
    four humans 
    in four poses 
    are visualized in Figure~\ref{fig:faust}.

    \begin{figure}
    \centering%
    \resizebox{\linewidth}{!}{%
    \begin{tabular}{c @{\hspace{10pt}} c @{\hspace{0pt}} c @{\hspace{0pt}} c @{\hspace{0pt}} c}
    & human~1 & human~3 & human~5 & human~7 \\[10pt]
    \rotatebox{90}{\hspace{50pt}pose~1}
    &{\includegraphics[width=0.31\linewidth,trim=0pt 0pt 0pt 0pt]{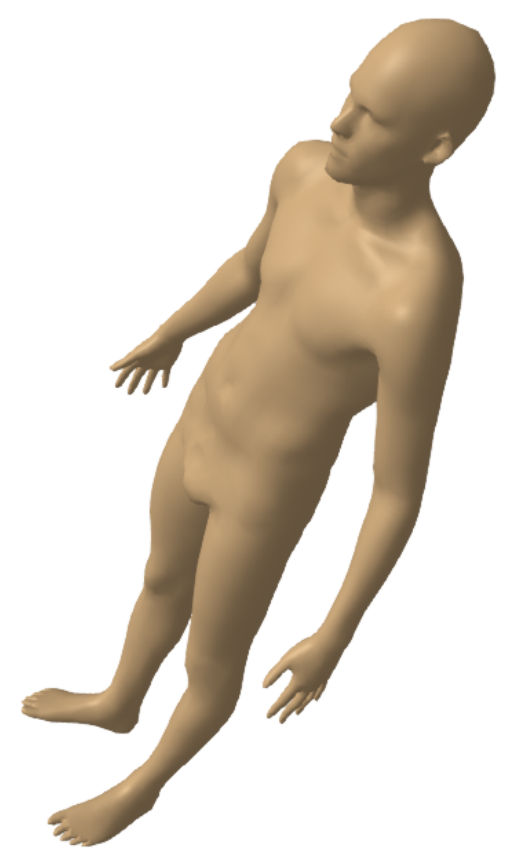}}%
    &{\includegraphics[width=0.31\linewidth,trim=0pt 0pt 0pt 0pt]{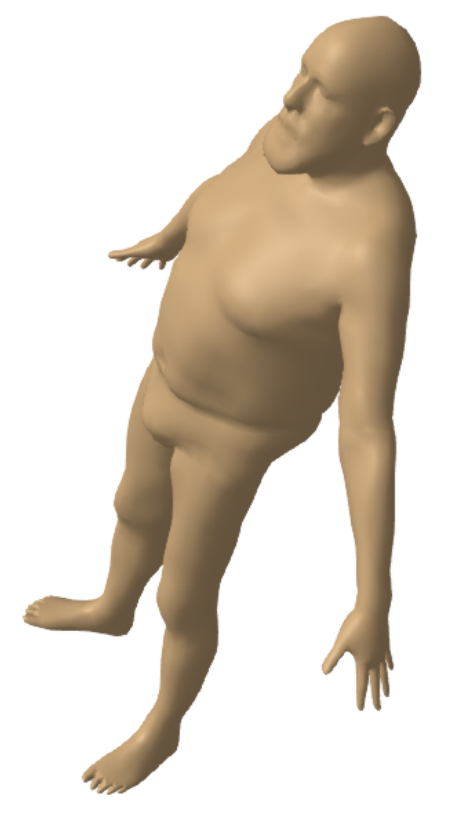}}%
    &{\includegraphics[width=0.31\linewidth,trim=0pt 0pt 0pt 0pt]{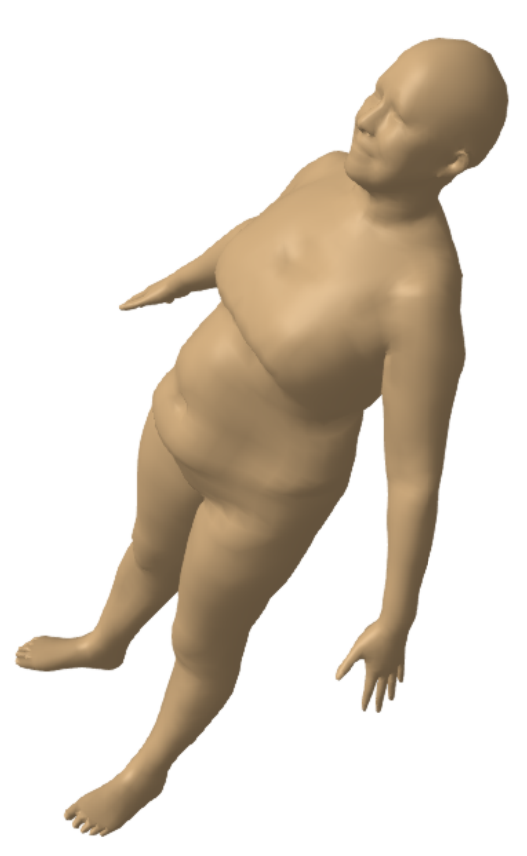}}%
    &{\includegraphics[width=0.31\linewidth,trim=0pt 0pt 0pt 0pt]{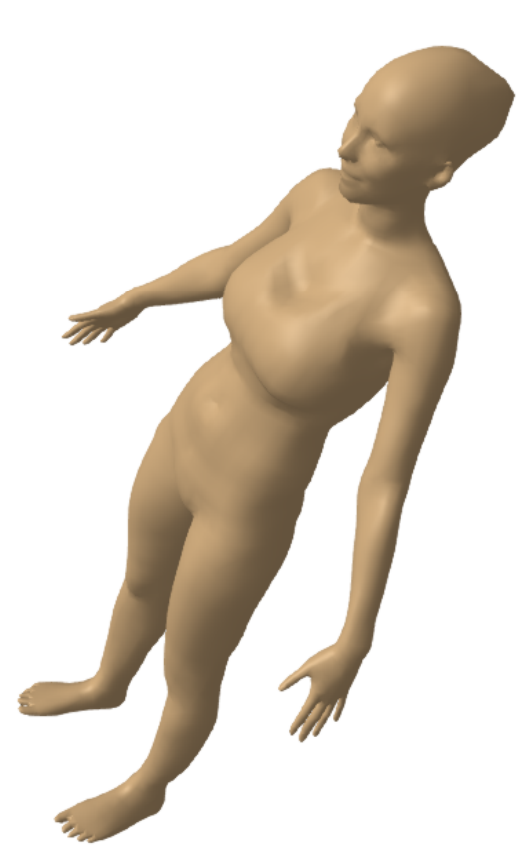}}%
    \\
    \rotatebox{90}{\hspace{50pt}pose~2}
    &{\includegraphics[width=0.31\linewidth,trim=0pt 0pt 0pt 0pt]{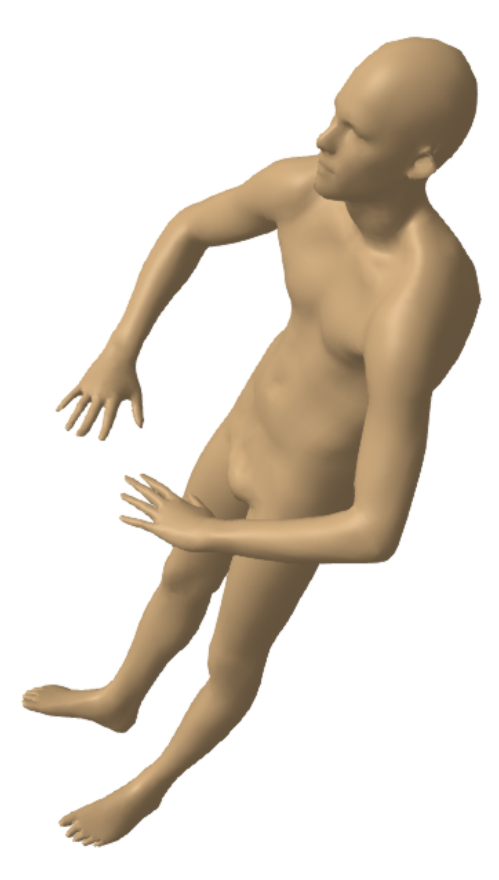}}%
    &{\includegraphics[width=0.31\linewidth,trim=0pt 0pt 0pt 0pt]{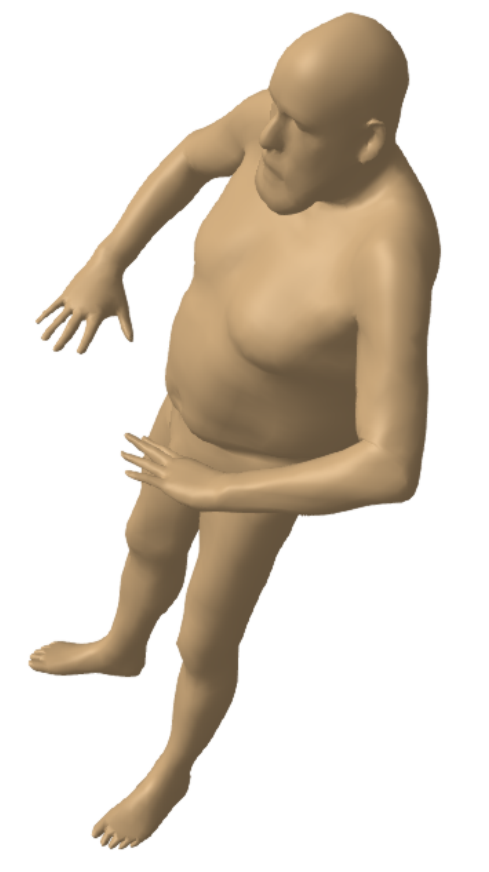}}%
    &{\includegraphics[width=0.31\linewidth,trim=-10pt -10pt -10pt -10pt]{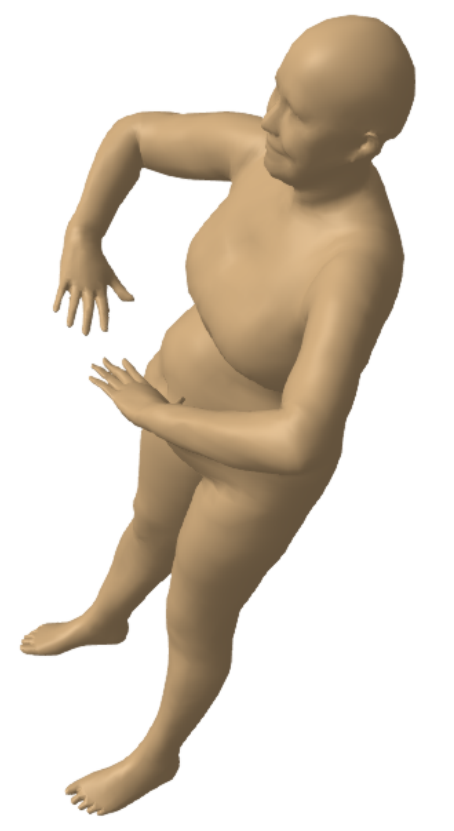}}%
    &{\includegraphics[width=0.31\linewidth,trim=0pt 0pt 0pt 0pt]{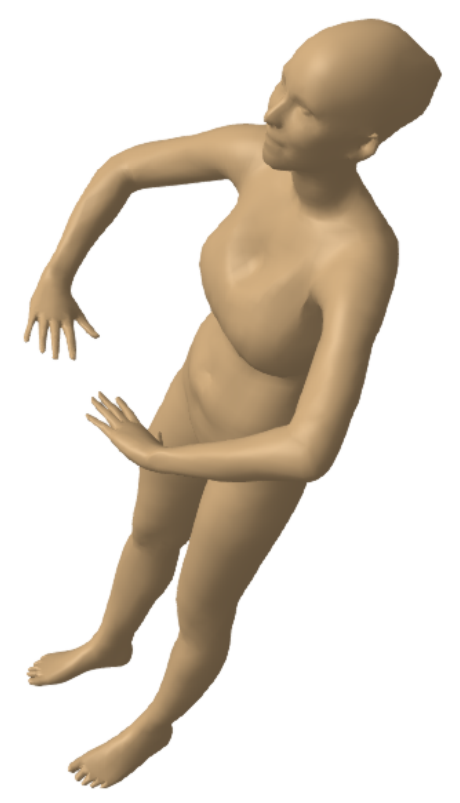}}%
    \\
    \rotatebox{90}{\hspace{50pt}pose~3}
    &{\includegraphics[width=0.31\linewidth,trim=0pt 0pt 0pt 0pt]{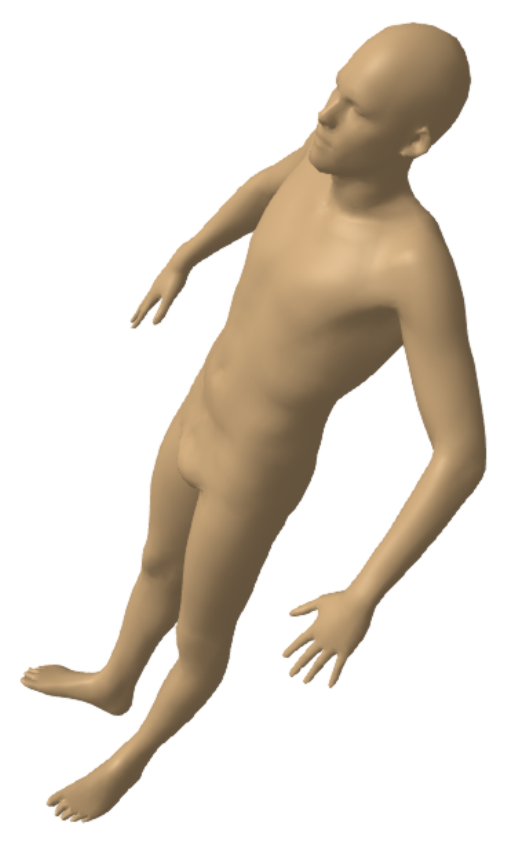}}%
    &{\includegraphics[width=0.31\linewidth,trim=-10pt -10pt -10pt -10pt]{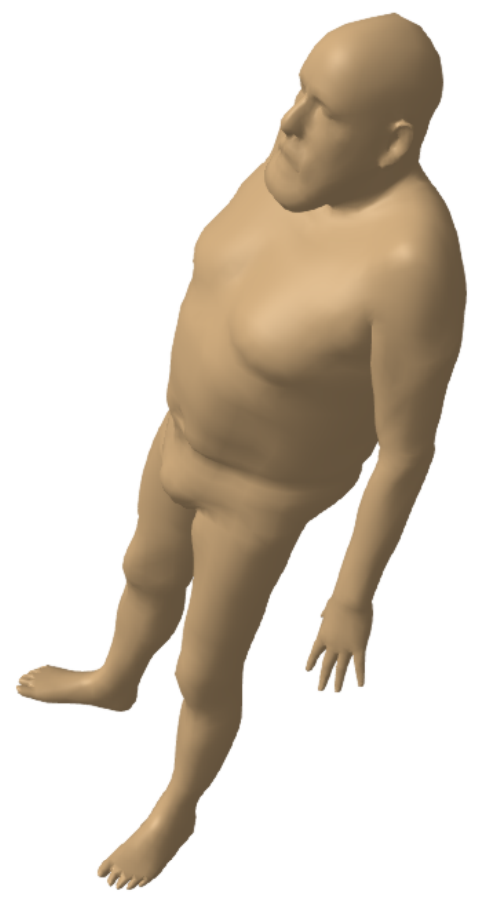}}%
    &{\includegraphics[width=0.31\linewidth,trim=-10pt -10pt -10pt -10pt]{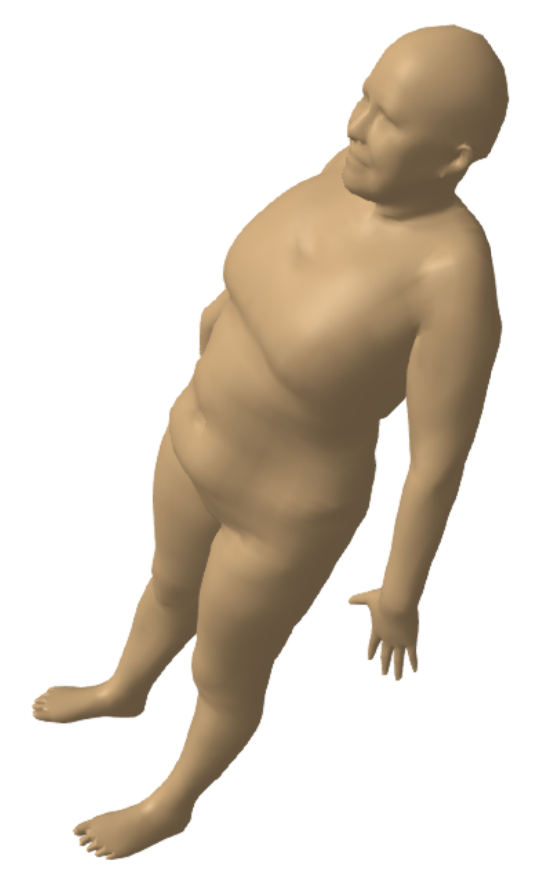}}%
    &{\includegraphics[width=0.31\linewidth,trim=0pt 0pt 0pt 0pt]{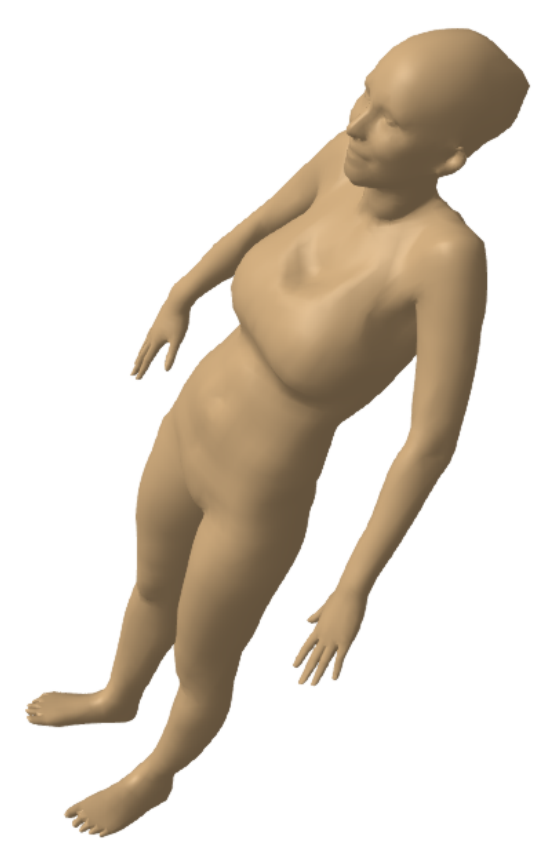}}%
    \\
    \rotatebox{90}{\hspace{50pt}pose~4}
    &{\includegraphics[width=0.31\linewidth,trim=0pt 0pt 0pt 0pt]{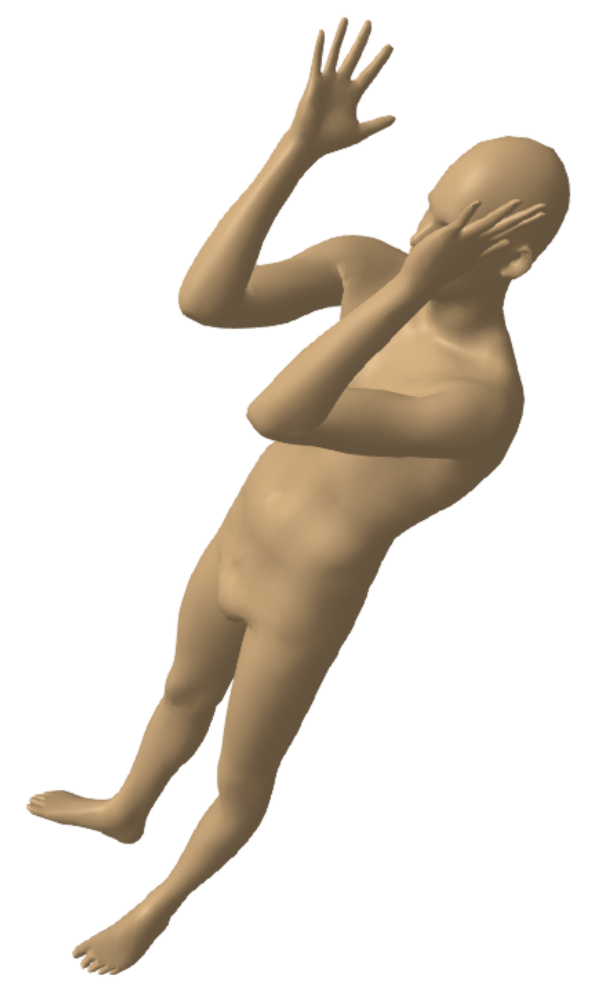}}%
    &{\includegraphics[width=0.31\linewidth,trim=-5pt -5pt -5pt -5pt]{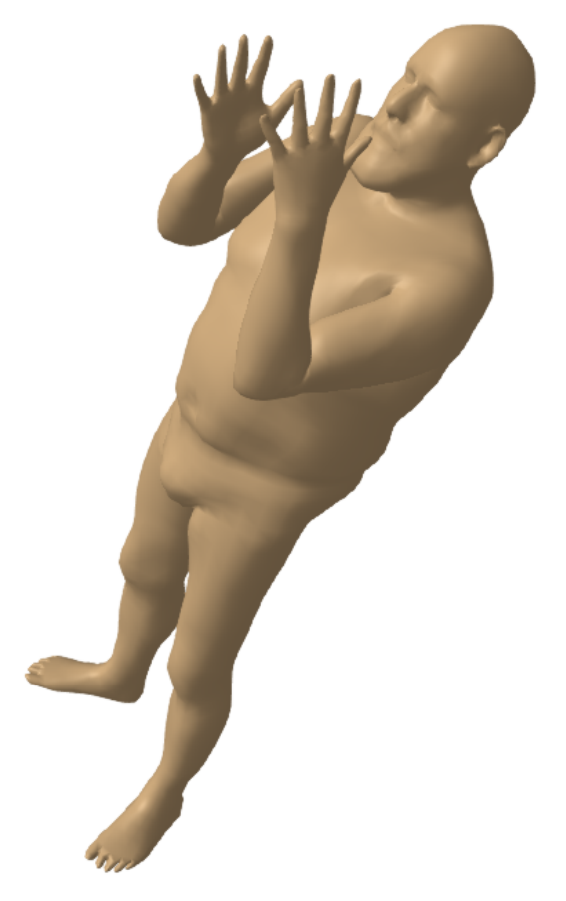}}%
    &{\includegraphics[width=0.31\linewidth,trim=-5pt -5pt -5pt -5pt]{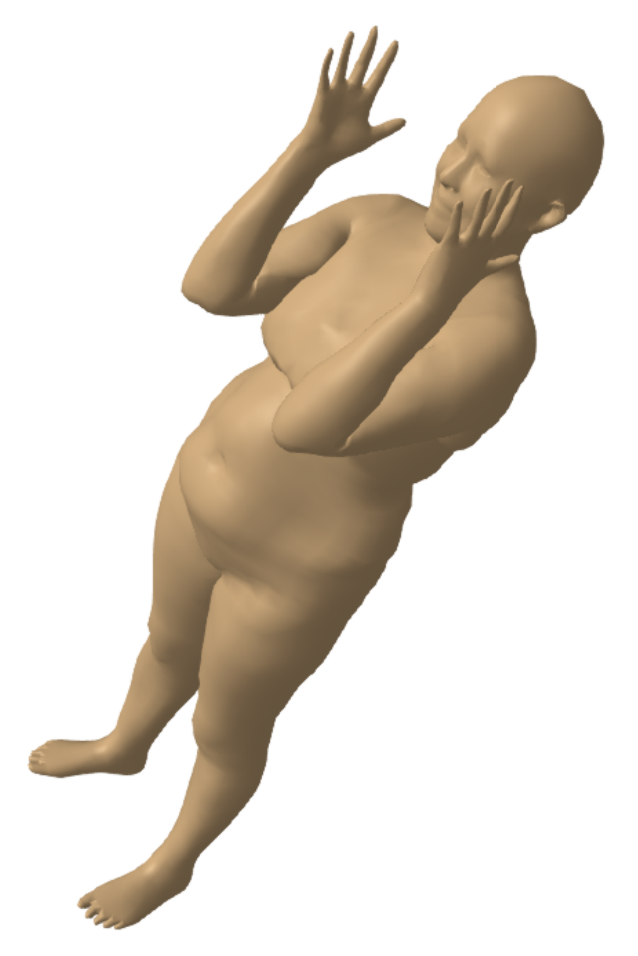}}%
    &{\includegraphics[width=0.31\linewidth,trim=0pt 0pt 0pt 0pt]{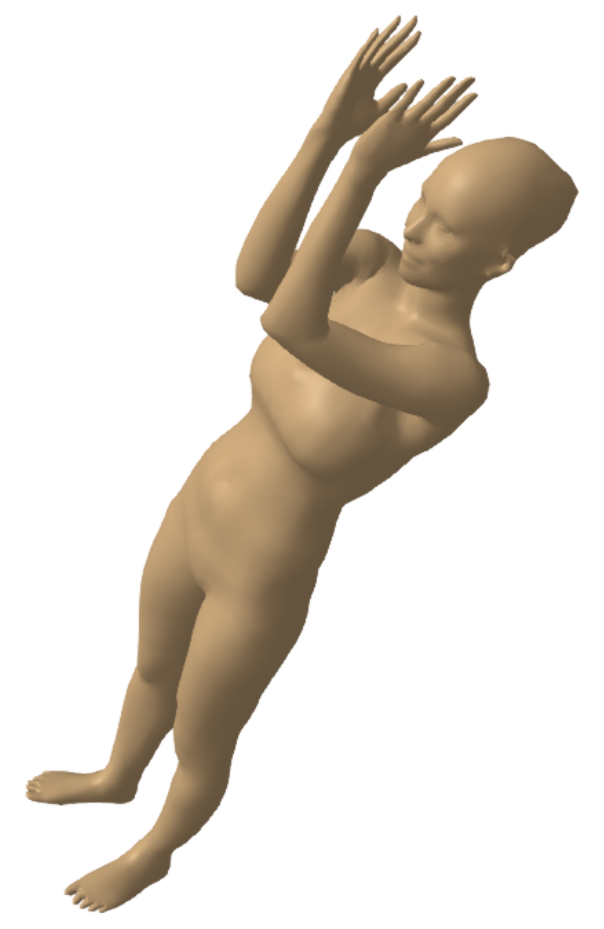}}%
    \end{tabular}}
    \caption{Visualization of a subset of the FAUST dataset:
    four humans in four poses.}
    \label{fig:faust}
\end{figure}
\end{enumerate}
In all cases, the data represents
surfaces of 3d objects with positive volume
so that the induced empirical measures
are not supported on planes
and, hence,
are contained in $\P_\c^*(\R^3)$.

\begin{figure*}[t]
    \resizebox{\linewidth}{!}{%
    \begin{tabular}{c c c}
    \multicolumn{3}{c}{%
    \ALineClassblue
    \; lions \;
    \ALineClassred
    \; cats \;
    \ALineClassgreen
    \; camels \;
    \ALineClassyellow
    \; horses \;
    \ALineClasspurple
    \; flamingos \;
    \ALineClassbrown
    \; elephants} \\[1ex]
    \mNRCDT & \maNRCDT & \miNRCDT \\
    \includegraphics[width=0.4\linewidth, clip=true, trim=12pt 10pt 12pt 10pt]{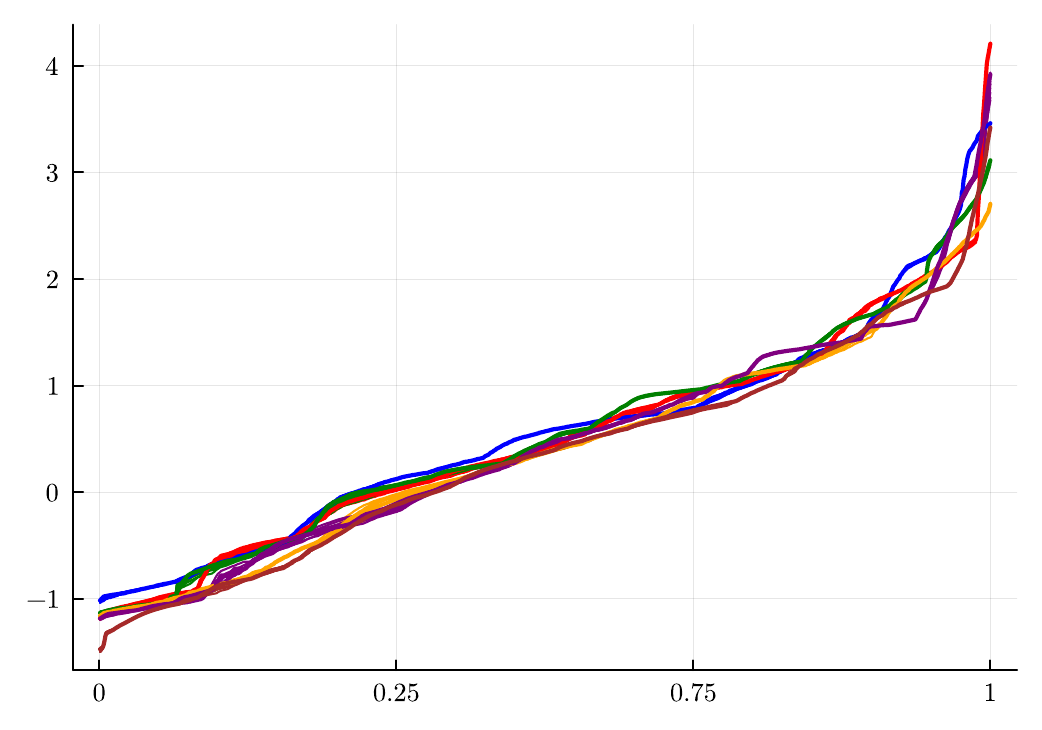}
    &\includegraphics[width=0.4\linewidth, clip=true, trim=12pt 10pt 12pt 10pt]{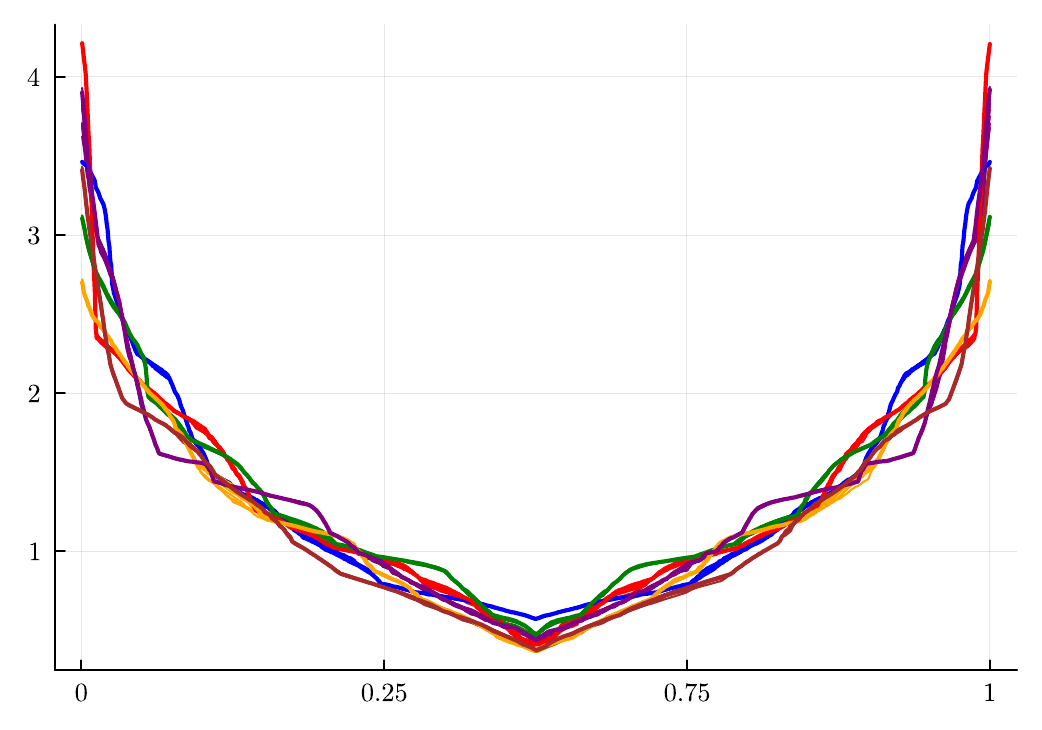}
    &\includegraphics[width=0.4\linewidth, clip=true, trim=12pt 10pt 12pt 10pt]{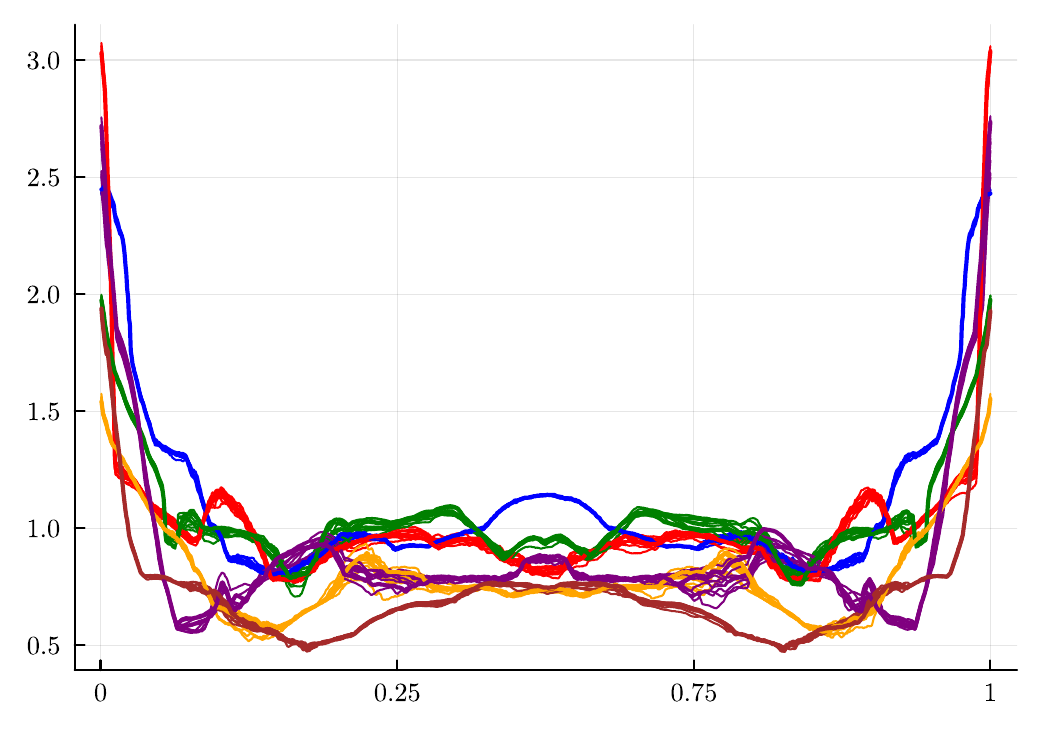}
    \end{tabular}}
    \caption{Visualization of various \hNRCDT{}s of
    the animal dataset with $10$ samples per class
    and $2048$ Fibonacci-like points.}
    \label{fig:hNRCDT_animals_3d}
\end{figure*}

\begin{table}[t]
    \centering%
    \footnotesize%
    \begin{tabular}{l @{\enspace} c @{\enspace} c @{\enspace} c @{\enspace} c}
        \toprule
        directions
        & R-CDT
        & \mNRCDT
        & \maNRCDT
        & \miNRCDT\\
        \midrule
        2 
        & 0.233 & 0.450 & \textbf{0.450} & 0.267 \\
        4 
        & 0.233 & 0.283 & \textbf{0.417} & 0.367 \\
        8
        & 0.233 & 0.733 & \textbf{0.767} & 0.650 \\
        16
        & 0.250 & 0.883 & 0.950 & \textbf{0.967} \\
        32
        & 0.250 & 0.933 & \textbf{1.000} & 0.967 \\
        64
        & 0.250 & 0.983 & \textbf{1.000} & \textbf{1.000} \\
        128
        & 0.250 & \textbf{1.000} & \textbf{1.000} & \textbf{1.000} \\
        256
        & 0.250 & \textbf{1.000} & \textbf{1.000} & \textbf{1.000} \\
        \bottomrule
    \end{tabular}
    \caption{NT classification accuracies
    for the animal dataset with $10$ samples per class
    and varying numbers of directions
    (in form of Fibonacci-like points on $\Sphere_2$).}
    \label{tab:nearest_temp_animal_3d}
\end{table}

Note that the main crucial point 
in the multi-dimensional setting
compared to the 2d setting
is that
$h$ 
in the definition of the generalized \hNRCDT{} 
operates on a larger set of directions,
here $\Theta = \Sphere_2$.
Figuratively,
this means that
more information is condensed,
potentially lowering the feature extraction
capabilities of the \hNRCDT{}.
Considering the \hNRCDT{} of the six animal classes
in Figure~\ref{fig:hNRCDT_animals_3d},
where the underlying CDTs are sampled on
an equispaced grid of $1000$ points in $(0,1)$,
we notice that
this apprehension does not materialize.
Again,
as suggested by the theory,
each animal class shrinks to a single point in \hNRCDT{} space,
and the different \hNRCDT{}s are well distinguishable.
For the numerical computations,
we rely on the Fibonacci-like points
to discretize the direction set $\Sphere_2$
in style of \cite{Gonzalez2010},
given by
\begin{equation*}
    \bftheta_k
    \coloneqq
    \biggl[\begin{smallmatrix}
        \sqrt{1 - z_k}\cos(k \varphi) \\
        \sqrt{1 - z_k}\sin(k \varphi) \\
        z_k
    \end{smallmatrix}\biggr],
    \quad k \in \{0,...,n-1\},
\end{equation*}
with $z_k = 1 - \frac{2k+1}{n}$ and golden angle
$\varphi = \pi(3 - \sqrt{5})$.
Note that 
we do not consider the \tvNRCDT{}
in this setting,
since there is no straightforward generalization
to functions on the sphere $\Sphere_2$.

\begin{table}[t]
    \centering%
    \footnotesize%
    \begin{tabular}{l @{\enspace} c @{\enspace} c @{\enspace} c @{\enspace} c}
        \toprule
        directions
        & R-CDT
        & \mNRCDT
        & \maNRCDT
        & \miNRCDT \\
        \midrule
        2 
        & 0.030 & 0.228 & \textbf{0.239} & 0.089 \\
        4 
        & 0.028 & 0.346 & \textbf{0.351} & 0.218 \\
        8
        & 0.023 & 0.505 & \textbf{0.538} & 0.389 \\
        16
        & 0.025 & 0.710 & \textbf{0.745} & 0.604 \\
        32
        & 0.026 & 0.746 & \textbf{0.754} & 0.676 \\
        64
        & 0.026 & 0.829 & \textbf{0.829} & 0.775 \\
        128
        & 0.026 & 0.850 & \textbf{0.859} & 0.785 \\
        256
        & 0.026 & \textbf{0.975} & 0.894 & 0.846 \\
        512
        & 0.026 & 0.913 & \textbf{0.925} & 0.888 \\
        1024
        & 0.026 & \textbf{0.919} & 0.915 & 0.909 \\
        2048
        & 0.026 & \textbf{0.945} & 0.936 & 0.928 \\
        \bottomrule
    \end{tabular}
    \caption{NT classification accuracies
    for the ABC dataset with $10$ samples per class
    and varying numbers of directions 
    (in form of Fibonacci-like points on $\Sphere_2$).}
    \label{tab:nearest_temp_abc_3d}
\end{table}

\begin{table}[t]
    \centering%
    \resizebox{\linewidth}{!}{%
    \begin{tabular}{l @{\quad} c @{\quad} c @{\quad} c @{\quad}
    c}
        \toprule
        train
        & Eucl.
        & R-CDT
        & \mNRCDT
        & \miaNRCDT \\
        \midrule
        1 
        & $0.658\pm0.054$ & $0.475\pm0.060$ & $\bf{0.706\pm0.044}$ 
        & $0.672\pm0.030$ \\
        3 
        & $0.727\pm0.035$ & $0.575\pm0.041$ & $\bf{0.801\pm0.050}$ 
        & $0.784\pm0.042$ \\
        5
        & $0.757\pm0.037$ & $0.591\pm0.040$ & $\bf{0.838\pm0.044}$ 
        & $0.828\pm0.038$ \\
        \bottomrule
    \end{tabular}}
    \caption{1-NN classification accuracies
    for the FAUST dataset with $10$ samples per class
    and using 256 directions 
    (in form of Fibonacci-like points on $\Sphere_2$).}
    \label{tab:nearest_neigbor_faust_3d}
\end{table}

We adapt the NT classification experiments 
from §~\ref{sssec:NT},
i.e.,
we label the affinely transformed data
based on the nearest template 
in \hNRCDT{} space
with respect to the Euclidean norm.
The achieved classification accuracies 
for the animal dataset
are reported in Table~\ref{tab:nearest_temp_animal_3d}.
As expected, 
all employed \hNRCDT{} variants yield perfect classifications,
whereas the multi-dimensional R-CDT cannot classify the different animals.
Worth mentioning,
we obtain already convincing results
for small numbers of Fibonacci points.

The NT classification accuracies
for the ABC dataset are
shown in Table~\ref{tab:nearest_temp_abc_3d}.
Here, the main challenge is the
large number of 80 different classes,
which are partly based on very similar templates.
Nevertheless,
for a sufficient fine discretization of $\Sphere_2$,
most samples are correctly classified
when using our \hNRCDT{} feature representations.
Note again that
R-CDT is not designed to distinguish
affine classes based on anisotropic scaling and shearing
used in this experiment.

\begin{figure*}
    \centering%
    \resizebox{\linewidth}{!}{%
    \begin{tabular}{c c c c}
        \multicolumn{4}{c}{%
        class:\;
        \Circle~1 \;
        \Star~2 \;
        \Triangle~3 \;
        \Ubox~4 \;
        \Utriangle~5 
        \quad
        centre:\;
        \Boxblue~1 \;
        \Boxred~2 \;
        \Boxgreen~3 \;
        \Boxyellow~4 \;
        \Boxpurple~5 
        \quad
        train/test:\;
        \UnBoxblue/\UnsBoxblue~1 \;
        \UnBoxred/\UnsBoxred~2 \;
        \UnBoxgreen/\UnsBoxgreen~3 \;
        \UnBoxyellow/\UnsBoxyellow~4 \;
        \UnBoxpurple/\UnsBoxpurple~5 
        } \\[0.5ex]%
        \multicolumn{4}{c}{%
        \textcolor{white}{class:}\hspace{8pt}
        \Pentagon~6 \;
        \Ltriangle~7 \;\;
        \Rtriangle~8 \;
        \Hexagon~9 \;
        \Heptagon~10
        \hspace{4pt}
        \textcolor{white}{centre:}\;
        \Boxbrown~6 \;
        \Boxorange~7 \;
        \Boxpink~8 \;
        \Boxcyan~9 \;
        \Boxolive~10 
        \hspace{2pt}
        \textcolor{white}{train/test:}\;
        \UnBoxbrown/\UnsBoxbrown~6 \;
        \UnBoxorange/\UnsBoxorange~7 \;
        \UnBoxpink/\UnsBoxpink~8 \;
        \UnBoxcyan/\UnsBoxcyan~9 \;
        \UnBoxolive/\UnsBoxolive~10
        } \\[1em]
        Eucl. & R-CDT & \mNRCDT{} & \miaNRCDT \\
        \includegraphics[width=0.32\linewidth, clip=true, trim=150pt 30pt 130pt 30pt]{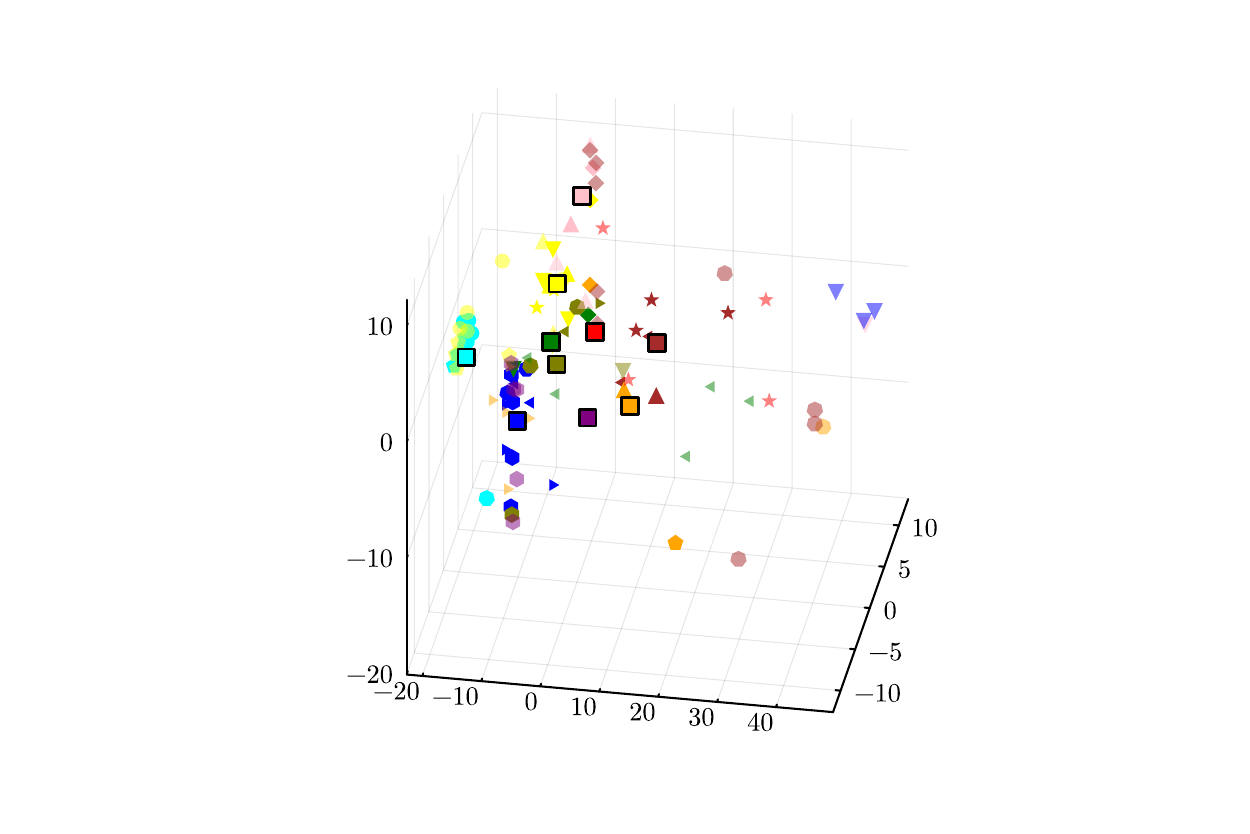}
        & \includegraphics[width=0.32\linewidth, clip=true, trim=150pt 30pt 130pt 30pt]{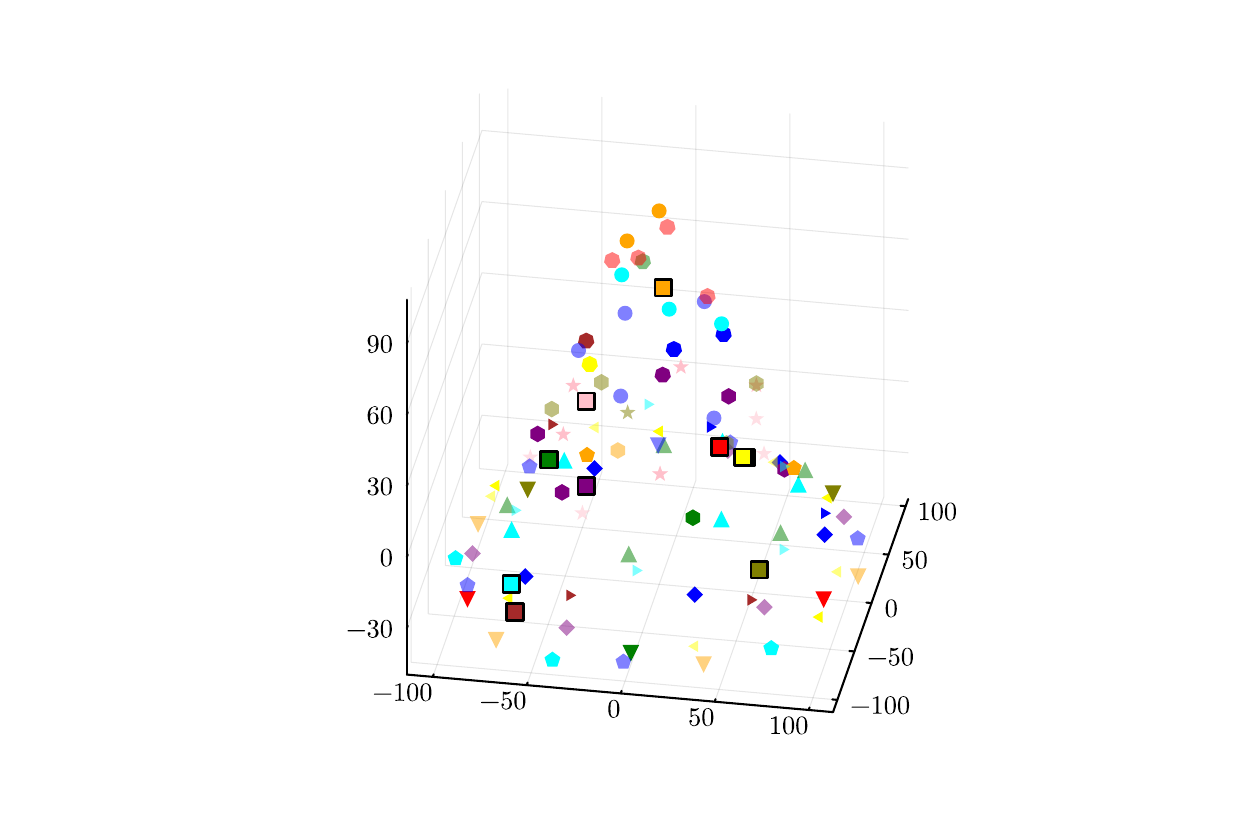}
        & \includegraphics[width=0.32\linewidth, clip=true, trim=150pt 30pt 130pt 30pt]{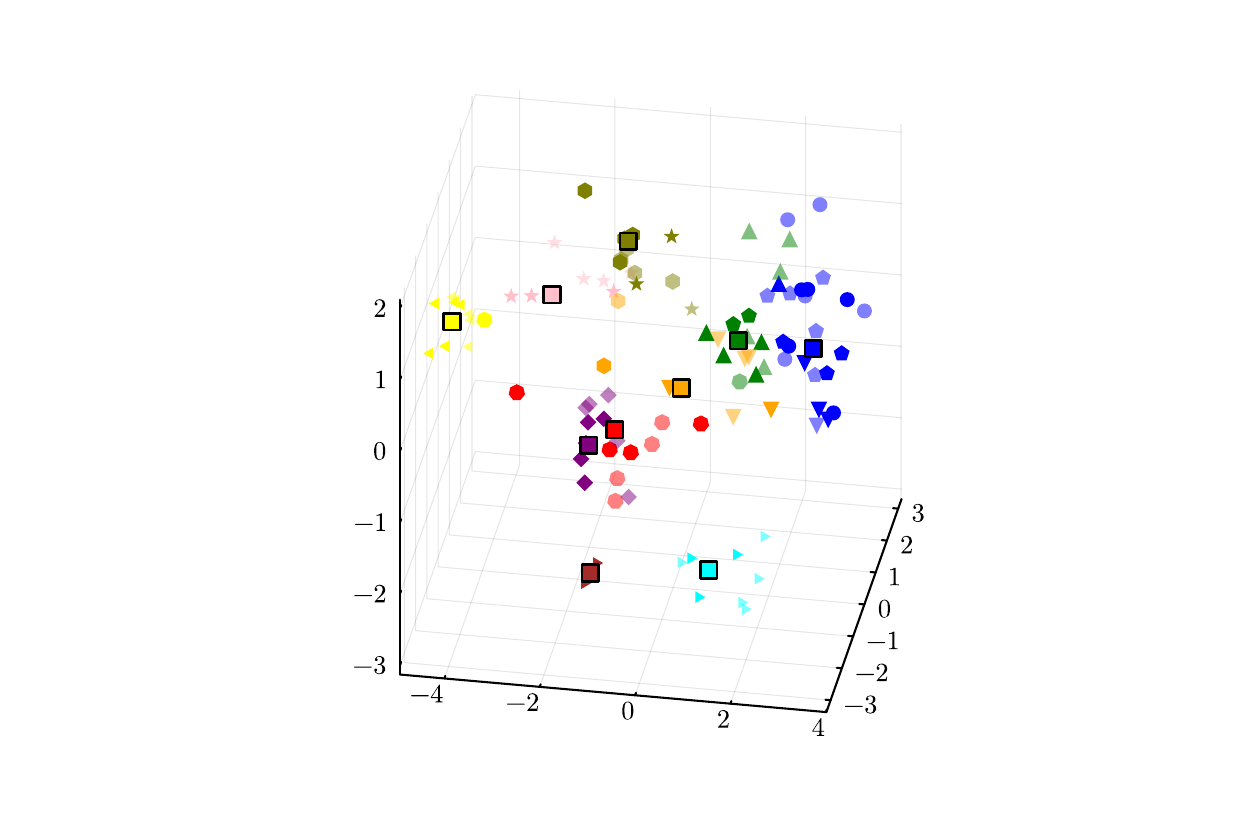}
        & \includegraphics[width=0.32\linewidth, clip=true, trim=150pt 30pt 130pt 30pt]{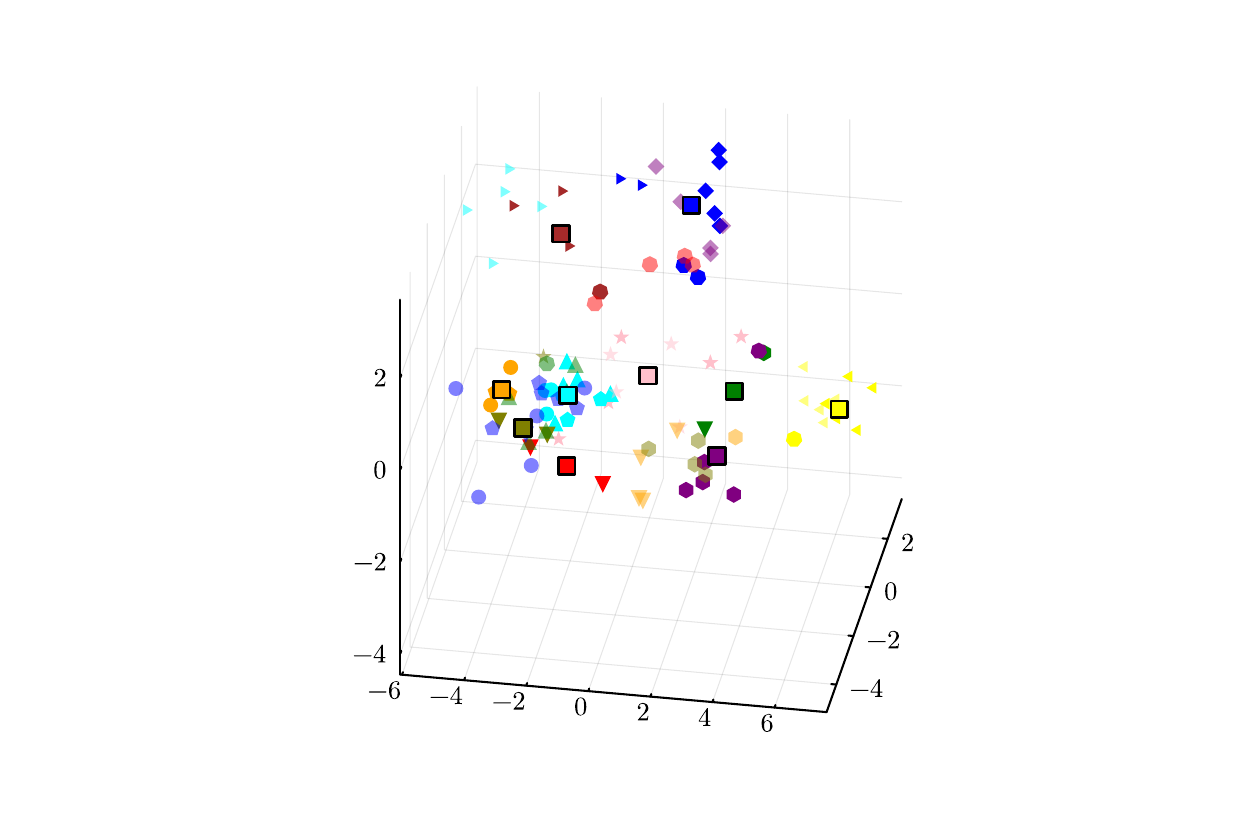}
        \end{tabular}}
    \caption{$10$-means cluster visualization for the FAUST dataset
    using a 3d PCA in the respective feature space.}
    \label{fig:kmeansFaust}
\end{figure*}

Finally,
we consider the FAUST dataset
as benchmark for
pose estimation.
Since in this case no templates exist,
we perform 1-NN classification,
similar to the LinMNIST example
in~§~\ref{sec:image_classification}.
The resulting accuracies 
are shown in Table~\ref{tab:nearest_neigbor_faust_3d},
where we use 256 directions on $\Sphere^2$.
We observe that \mNRCDT{} performs best,
followed by \hNRCDT{c}.
Additionally,
we perform a cluster analysis
using $k$-means
as in §~\ref{sec:image_clustering}.
To this end, each class is split into a train and test subset,
each consisting of five samples.
The results are reported in
Table~\ref{tab:kmeansFaust}
and the clusters are visualized
in Figure~\ref{fig:kmeansFaust}.
While \hNRCDT{c} shows
the best Rand index 
on both the train and test sets,
the clusters using \mNRCDT{}
visually appear more separated.
We wish to stress that
some poses are very similar
so that a perfect classification/clustering
cannot be expected.

In conclusion,
these first proof-of-concept simulations
indicate the usefulness of our multi-dimensional
\hNRCDT{} feature representations
in 3d point cloud recognition tasks.
Further experiments
on corresponding real-world applications 
and sensitivity studies
are left for future research.

\begin{table}[t]
    \centering%
    \footnotesize%
    \begin{tabular}{l @{\enspace} l @{\quad} c @{\qquad} c @{\quad} c @{\quad} c}
        \toprule 
        & & Eucl.
        & R-CDT
        & \mNRCDT 
        & \miaNRCDT  \\
        \midrule
        RI$_{\textrm{train}}$ & $(\uparrow)$ & $0.8261$ & $0.8424$ & $0.9208$ & $\mathbf{0.9208}$\\
        VI$_{\textrm{train}}$ & $(\downarrow)$ & $1.7299$ & $2.7545$ & $\mathbf{0.8316}$ & $0.9300$\\
        RI$_{\textrm{test}}$ & $(\uparrow)$ & $0.7910$ & $0.8424$ & $0.9346$ & $\mathbf{0.9444}$\\
        VI$_{\textrm{test}}$ & $(\downarrow)$ & $1.6587$ & $2.7964$ & $\mathbf{0.6975}$ & $0.7074$\\
        \bottomrule
    \end{tabular}
    \caption{Quality measures
    for $10$-means clustering 
    of the FAUST dataset 
    with $10$ samples per class,
    where $5$ samples are used for training
    and the rest for testing.}
    \label{tab:kmeansFaust}
\end{table}

\subsubsection{Classification on the 3D Rotation Group}
\label{sssec:class_hNRCDT_so3}

In our last set of numerical experiments
we leave the Euclidean setting
and study a classification task on
the 3d rotation group $\SO(3)$
as in §~\ref{sec:SO3-NRCDT}.
To this end,
we consider a synthetic dataset
based on three empirical template measures,
each supported on $1000$ rotation matrices.
More precisely,
the matrices for the first template
are generated by the
Matrix-Fisher distribution \cite{Downs1972}
around $I_3$ with concentration $\kappa = 10$,
the matrices for the second template
are uniformly distributed on
the subset of rotations around
the $x$-$y$-equator,
and the matrices for the third template
are constructed by computing the
QR decomposition
of random Gaussian matrices in $\R^{3 \times 3}$.
Thereon,
these template measures are transformed
by applying random rotations
to form the dataset,
we refer to as {\bf rotation dataset},
see Figure~\ref{fig:academic_dataset_SO3}
for visualization.

\begin{figure*}[t]
    \resizebox{\linewidth}{!}{%
    \begin{tabular}{l l l}
    \raisebox{7pt}[]{Template~1} 
    & \raisebox{7pt}[]{\scalebox{1.2}{\{}}
    \includegraphics[width=1.0\linewidth, clip=true, trim=55pt 10pt 12pt 10pt]{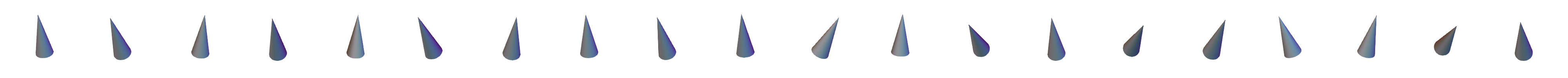} 
    \raisebox{7pt}[]{\scalebox{1.2}{... \}}} \\
    \raisebox{7pt}[]{Class~1} 
    & \raisebox{7pt}[]{\scalebox{1.2}{\{}}
    \includegraphics[width=1.0\linewidth, clip=true, trim=55pt 10pt 12pt 10pt]{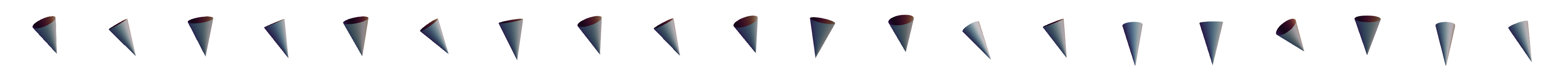} 
    \raisebox{7pt}[]{\scalebox{1.2}{... \}}} \\
    & \raisebox{7pt}[]{\scalebox{1.2}{\{}}
    \includegraphics[width=1.0\linewidth, clip=true, trim=55pt 10pt 12pt 10pt]{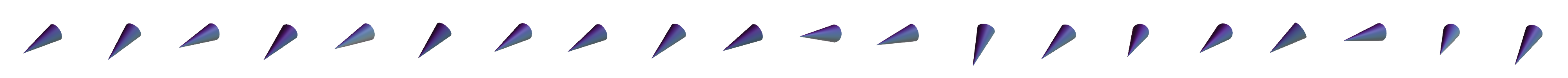} 
    \raisebox{7pt}[]{\scalebox{1.2}{... \}}} \\
    & \raisebox{7pt}[]{\scalebox{1.2}{\{}}
    \includegraphics[width=1.0\linewidth, clip=true, trim=55pt 10pt 12pt 10pt]{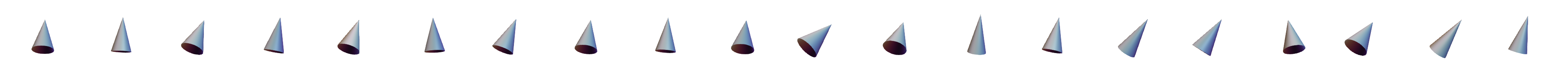} 
    \raisebox{7pt}[]{\scalebox{1.2}{... \}}} \\
    \raisebox{7pt}[]{Template~2} 
    & \raisebox{7pt}[]{\scalebox{1.2}{\{}}
    \includegraphics[width=1.0\linewidth, clip=true, trim=55pt 10pt 12pt 10pt]{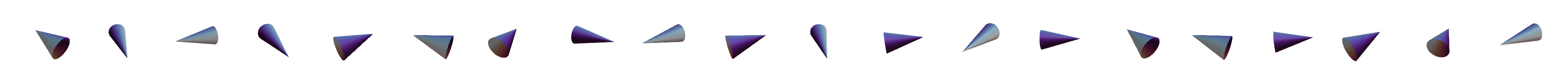} 
    \raisebox{7pt}[]{\scalebox{1.2}{... \}}} \\
    \raisebox{7pt}[]{Class~2} 
    & \raisebox{7pt}[]{\scalebox{1.2}{\{}}
    \includegraphics[width=1.0\linewidth, clip=true, trim=55pt 10pt 12pt 10pt]{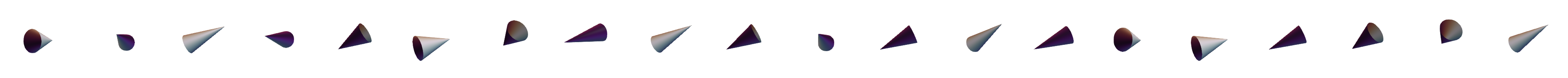} 
    \raisebox{7pt}[]{\scalebox{1.2}{... \}}} \\
    & \raisebox{7pt}[]{\scalebox{1.2}{\{}}
    \includegraphics[width=1.0\linewidth, clip=true, trim=55pt 10pt 12pt 10pt]{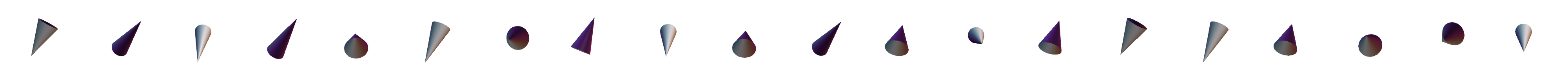} 
    \raisebox{7pt}[]{\scalebox{1.2}{... \}}} \\
    & \raisebox{7pt}[]{\scalebox{1.2}{\{}}
    \includegraphics[width=1.0\linewidth, clip=true, trim=55pt 10pt 12pt 10pt]{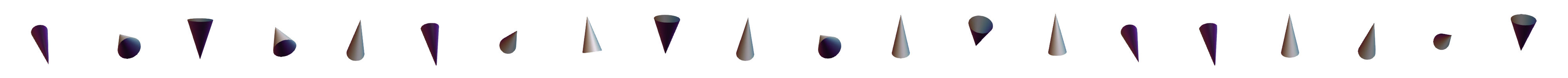} 
    \raisebox{7pt}[]{\scalebox{1.2}{... \}}} \\
    \raisebox{7pt}[]{Template~3} 
    & \raisebox{7pt}[]{\scalebox{1.2}{\{}}
    \includegraphics[width=1.0\linewidth, clip=true, trim=55pt 10pt 12pt 10pt]{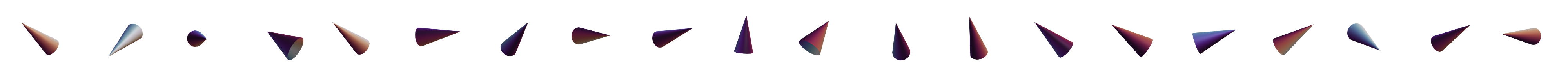}
    \raisebox{7pt}[]{\scalebox{1.2}{... \}}} \\
    \raisebox{7pt}[]{Class~3} 
    & \raisebox{7pt}[]{\scalebox{1.2}{\{}}
    \includegraphics[width=1.0\linewidth, clip=true, trim=55pt 10pt 12pt 10pt]{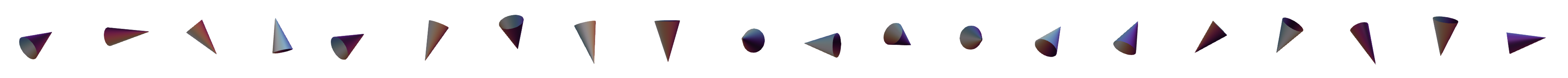}
    \raisebox{7pt}[]{\scalebox{1.2}{... \}}} \\
    & \raisebox{7pt}[]{\scalebox{1.2}{\{}}
    \includegraphics[width=1.0\linewidth, clip=true, trim=55pt 10pt 12pt 10pt]{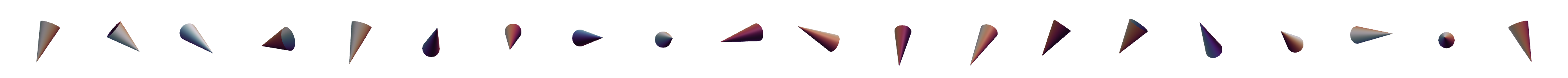}
    \raisebox{7pt}[]{\scalebox{1.2}{... \}}} \\
    & \raisebox{7pt}[]{\scalebox{1.2}{\{}}
    \includegraphics[width=1.0\linewidth, clip=true, trim=55pt 10pt 12pt 10pt]{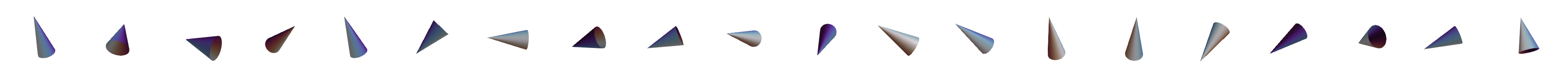}
    \raisebox{7pt}[]{\scalebox{1.2}{... \}}}
    \end{tabular}}
    \caption{%
    Illustration of the rotation dataset
    by visualizing the action of the underlying rotation matrices
    on a cone pointing towards the north pole,
    whose surface is coloured 
    by a periodic colour bar
    to depict rotations
    around its z-axis.
    For readability,
    we only show the first 20
    out of 1000 components
    for each empirical measure.%
    }
    \label{fig:academic_dataset_SO3}
\end{figure*}

\begin{figure*}[t]
	\centering%
    \footnotesize%
    \begin{tabular}{c c c}
    \multicolumn{3}{c}{%
    \LineClassOne
    \; class~1 \;
    \LineClassTwo
    \; class~2 \;
    \LineClassThree
    \; class~3} \\[1ex]
    \mNRCDT & \maNRCDT & \miNRCDT \\
    \includegraphics[width=0.3\linewidth, clip=true, trim=12pt 10pt 12pt 10pt]{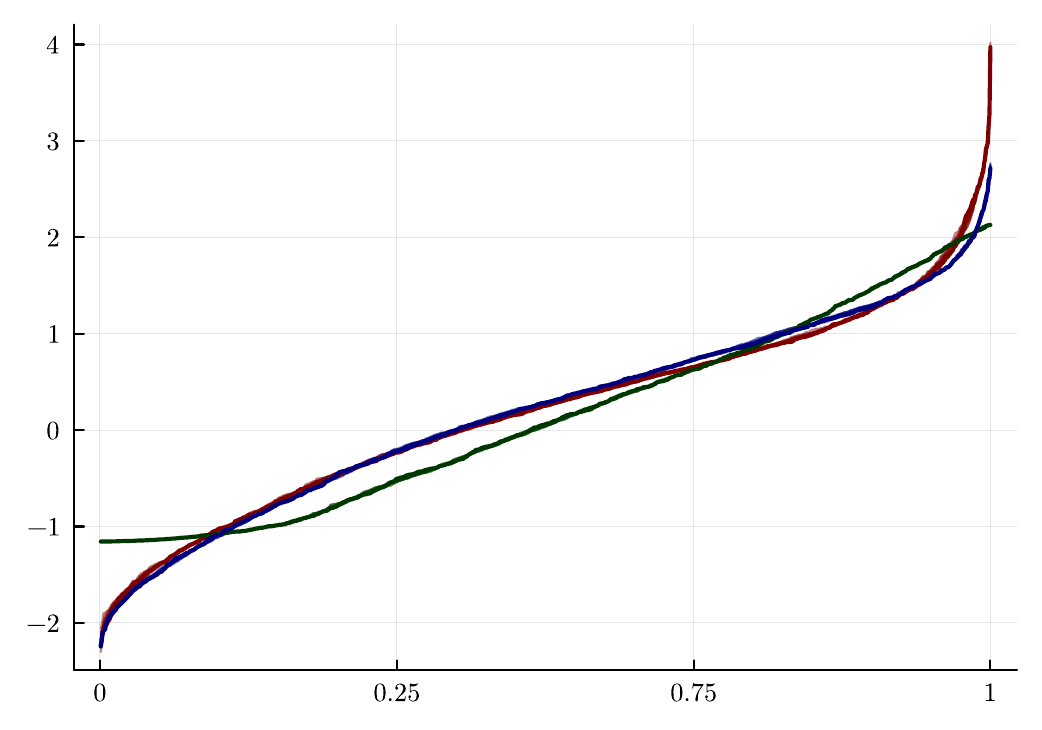}
    &\includegraphics[width=0.3\linewidth, clip=true, trim=12pt 10pt 12pt 10pt]{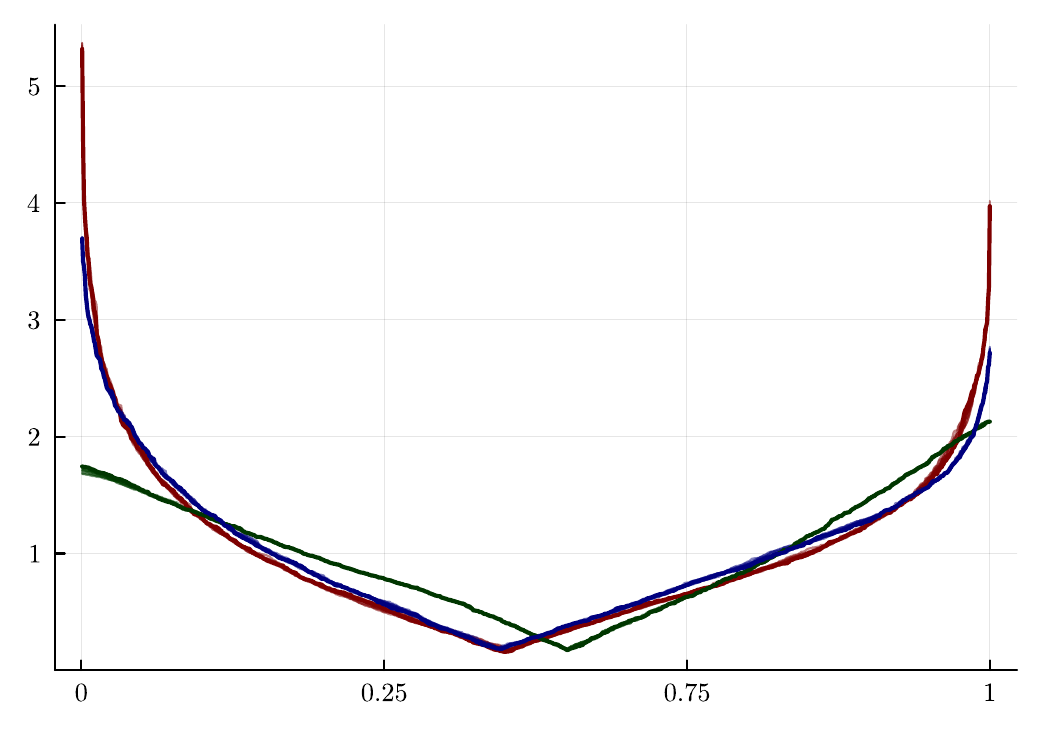}
    &\includegraphics[width=0.3\linewidth, clip=true, trim=12pt 10pt 12pt 10pt]{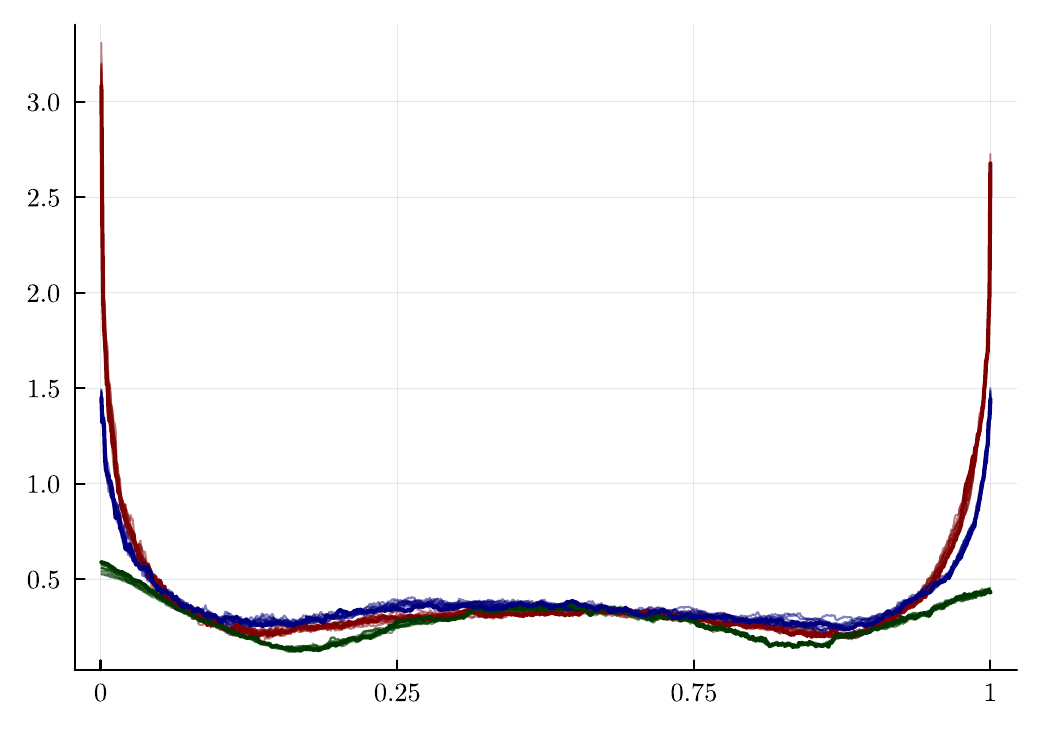}
    \end{tabular}
    \caption{Visualization of
    \hNRCDT{}s of the rotation dataset
    with $10$ samples per class
    and $2048$ Super-Fibonacci points.}
    \label{fig:hNRCDT_rot_so3}
\end{figure*}

For the numerical computations,
we need to discretize the direction set,
which is now given by $\Theta = \SO(3)$.
For this, we make use of the so-called
Super-Fibonacci points~\cite{Alexa2022},
defined by the Euler parameters
\begin{align*}
\begin{smallmatrix}
a_k = \sqrt{\tfrac{2k+1}{2n}}\sin\bigl(\tfrac{\pi(2k+1)}{\varphi_1}\bigr),
& c_k = \sqrt{1 - \tfrac{2k+1}{2n}}\cos\bigl(\tfrac{\pi(2k+1)}{\varphi_2}\bigr), \\
b_k = \sqrt{\tfrac{2k+1}{2n}}\cos\bigl(\tfrac{\pi(2k+1)}{\varphi_1}\bigr),
& d_k = \sqrt{1 - \tfrac{2k+1}{2n}}\sin\bigl(\tfrac{\pi(2k+1)}{\varphi_2}\bigr),
\end{smallmatrix}
\end{align*}
for $k \in \{0,\ldots,n-1\}$,
where $\varphi_1 = \sqrt{2}$
and $\varphi_2$ is the positive real solution to
$\varphi_2^4 = \varphi_2 + 4$.
One can show that $\varphi_2 = \tfrac{1}{2} (\sqrt{m} + \sqrt{\nicefrac{2}{\sqrt{m}} - m})$
with
$m = \sqrt[3]{m_+} + \sqrt[3]{m_-}$
and
$m_{\pm} = \tfrac{1}{18}(9 \pm \sqrt{49233})$.
The rotation matrices $\bftheta_k \in \SO(3)$
are then given by the Euler–Rodrigues formula
\begin{equation*}
    \bftheta_k
    \coloneqq 
    \biggl[\begin{smallmatrix}
        1 - 2(c_k^2 + d_k^2)
        & 2(b_k c_k - a_k d_k)
        & 2(b_k d_k + a_k c_k) \\
        2(b_k c_k + a_k d_k)
        & 1 - 2(b_k^2 + d_k^2)
        & 2(c_k d_k - a_k b_k) \\
        2(b_k d_k - a_k c_k)
        & 2(c_k d_k + a_k b_k)
        & 1 - 2(b_k^2 + c_k^2)
    \end{smallmatrix}\biggr].
\end{equation*}

We again adapt the NT classification experiments 
from §~\ref{sssec:NT} and
label the transformed data
based on the nearest template 
in \hNRCDT{} space
with respect to the Euclidean norm.
The \hNRCDT{}s of the entire dataset
are depicted in Figure~\ref{fig:hNRCDT_rot_so3},
where the underlying CDTs are sampled on an equispaced grid 
consisting of $1000$ points in $(0,1)$,
and show an almost perfect alignment within each class.
The achieved classification accuracies 
are reported in Table~\ref{tab:nearest_temp_rot_so3}.
As expected, 
all employed \hNRCDT{} variants yield perfect results,
whereas R-CDT cannot distinguish the different classes.

\begin{table}[t]
    \centering%
    \footnotesize%
    \begin{tabular}{l @{\enspace} c @{\enspace} c @{\enspace} c @{\enspace} c}
        \toprule
        rotations 
        & R-CDT
        & \mNRCDT
        & \maNRCDT
        & \miNRCDT \\
        \midrule
        2 
        & 0.633 & 0.933 & 0.933 & 0.433 \\
        4 
        & 0.400 & \textbf{1.000} & \textbf{1.000} & 0.733 \\
        8
        & 0.333 & \textbf{1.000} & \textbf{1.000} & 0.900 \\
        16
        & 0.367 & \textbf{1.000} & \textbf{1.000} & 0.967 \\
        32
        & 0.367 & \textbf{1.000} & \textbf{1.000} & \textbf{1.000} \\
        64
        & 0.400 & \textbf{1.000} & \textbf{1.000} & \textbf{1.000} \\
        \bottomrule
    \end{tabular}
    \caption{NT classification accuracies
    for rotation dataset
    with $10$ samples per class
    and varying numbers of rotations
    (Super-Fibonacci points on $\SO(3)$).}
    \label{tab:nearest_temp_rot_so3}
\end{table}

\section{Conclusion}
In this paper, 
we introduced the novel
generalized \hNRCDT{} 
for feature representation 
and proved linear separability of classes 
generated by affine transforms of given templates.
This was validated by numerical experiments
in two- and three-dimensional
as well as non-Euclidean settings
showing a significant increase in classification accuracy
over the original R-CDT.

Potential future directions include
the investigation of \hNRCDT{} performance
in combination with more advanced
classification and clustering methods,
like in \cite{Shifat-E-Rabbi2021,Auffenberg2025,Xing2002},
as well as
a more in-depth numerical study
in various real-world applications.
Moreover, the robustness of \hNRCDT{}
to noise and outliers needs to be analyzed.
For \mNRCDT{}, this has already been done
in our previous work~\cite{Beckmann2025}
and we wish to adapt our techniques
to the generalizations
introduced here.
Our theoretical results require
the distinguishability of
the embedded measure classes
in feature space.
While no collisions arose
in our numerical experiments,
a more descriptive characterization
of the non-injectivity
would still greatly enhance the
application of our theory
in practice.

At the current stage,
our \hNRCDT{} is limited to probability measures.
It may be extended to signed measures
by adapting the signed CDT \cite{Aldroubi2022}
and handling the positive and negative parts separately.
Moreover, in our experiments,
we always choose the uniform reference measure $\rho = u_{[0,1]}$.
While our theory is independent of its choice,
we expect that an application-dependent choice can even boost the \hNRCDT{} performance
by, e.g., introducing a denser sampling of certain regions.

%
%

\end{document}